\documentclass[10pt,oldfontcommands,oneside]{memoir}

\setstocksize{240mm}{170mm}
\settrimmedsize{240mm}{170mm}{*}

\settrims{0mm}{0mm}

\settypeblocksize{168.96mm}{112mm}{*}

\setlrmargins{22mm}{*}{*}

\setulmargins{27mm}{*}{*}

\setheaderspaces{*}{7mm}{*}

\checkandfixthelayout
\fixpdflayout

\setsecnumdepth{subsection}

\numberwithin{equation}{chapter}
\counterwithin{figure}{chapter}

\maxtocdepth{section}

\setpnumwidth{3em}
\setrmarg{4em}

\usepackage{graphicx}
\chapterstyle{bianchi}

\nouppercaseheads

\aliaspagestyle{title}{empty}

\aliaspagestyle{part}{empty}

\usepackage{microtype}

\usepackage{dhara}
\usepackage{setspace}

\author{\vspace{.7cm}\\ter verkrijging van de graad van doctor aan de\\
  Technische Universiteit Eindhoven, op gezag van de\\
  rector magnificus, prof.dr.ir. F.P.T. Baaijens, voor een\\
  commissie aangewezen door het College voor\\
   Promoties in het openbaar te verdedigen\\
  op dinsdag 28 augustus 2018 om 13.30 uur\\
  \vspace{1cm}\\
door\\
\vspace{1cm}\\ Souvik Dhara\\
\vspace{.7cm}\\
geboren te Kolkata, India}
\title{\textbf{Critical Percolation on Random Networks with Prescribed Degrees}\\
\vspace{1.5cm}
{\LARGE\scshape proefschrift}}
\date{}
\begin{document} 
\frontmatter

\thispagestyle{empty}
\begin{vplace}[0.5]
\begin{center}
\Huge Critical Percolation on Random Networks with Prescribed Degrees
\end{center}
\end{vplace}
\newpage

\ 
\vfill
\thispagestyle{empty}%
\noindent This work was financially supported by The Netherlands Organization for Scientific Research (NWO) through the Gravitation Networks grant 024.002.003.\\\\
\includegraphics[width = 0.55\linewidth]{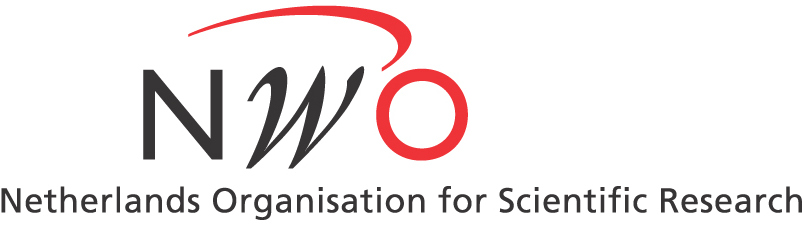}\hfill
\includegraphics[width = 0.22\linewidth]{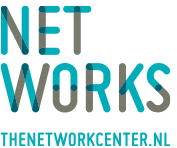}

\vspace{1cm}
\noindent \copyright~Souvik Dhara, 2018
\vspace{1em}

\noindent Critical Percolation on Random Networks with Prescribed Degrees
\vspace{1em}

\noindent A catalogue record is available from the Eindhoven University of Technology Library\\
\noindent ISBN: 978-90-386-4559-9
\vspace{1em}


\noindent Printed by Gildeprint Drukkerijen, Enschede

\newpage
\maketitle
\aliaspagestyle{title}{empty} 
\newpage
\noindent Dit proefschrift is goedgekeurd door de promotoren en de samenstelling\\van de promotiecommissie is als volgt:\\ \thispagestyle{empty}
\begin{flushleft}
\begin{tabular}{ l l }
voorzitter:   & prof.dr. M.A. Peletier\\
$1^\text{e}$ promotor: &prof.dr. R.W. van der Hofstad\\
$2^\text{e}$ promotor: &prof.dr. J.S.H. van Leeuwaarden\\
leden: 
&prof.dr.ir. S.C. Borst\\
&prof.dr. N. Broutin (Sorbonne Universit\'e)\\
&prof.dr. W.Th.F. den Hollander (UL)\\
&prof.dr. N.V. Litvak\\
&dr. L. Warnke (Georgia Institute of Technology)
\end{tabular}
\end{flushleft}
\vfill
Het onderzoek dat in dit proefschrift wordt beschreven is uitgevoerd in\\ overeenstemming met de TU/e Gedragscode Wetenschapsbeoefening.

\chapter*{Acknowledgment}
The journey from a student to a researcher requires development of several aspects: sharpening knowledge, becoming aware of contemporary research, developing communication skills, and identifying sources of inspiration. In the four years of my PhD, I have received enormous support from many people in nourishing each of the above aspects. Before going into the contents of this thesis, I would like to express my gratitude to them for their contribution to my academic career.

I am deeply indebted to my supervisors Remco van der Hofstad and Johan van Leeuwaarden for guiding me through this journey.  It is an honor to be your successor in the math genealogy-tree. Remco, your unending energy and enthusiasm for research are and will always be my inspiration. You have taught me everything from simple to deep mathematical facts. Our 17:30 meetings will always be in my cherished memories. You have always led with a brilliant example and showed me the highest standard for research and teaching.  Johan, your thoughts always provided me with a unique perspective. Your questions would always bring out the best from our research projects. You have always given me the best career advices, and shared your own invaluable experiences. 

I sincerely thank the committee members Sem Borst, Nicolas Broutin, Frank den Hollander, Nelly Litvak, and Lutz Warnke for a careful reading of my thesis and giving me feedback for improvement.

I have had wonderful experiences during research visits outside TU/e. I would like to thank Shankar Bhamidi for welcoming me heartily at the University of North Carolina Chapel Hill, and for sharing his enthusiasm and passion for research. Shankar, our trip to Nantahala is memorable to me. Special thanks to Michel Mandjes for sharing ideas from a vibrant research field which was orthogonal to my research project, and for adding new perspectives to my research. 

Working with Sem Borst was a fantastic experience. I learned a lot about how real-world systems work in our joint project. Sem, I was fortunate to witness your exquisite style of writing a research paper, and I hope to reflect some of the things that I have learned from you in my future writings. Thanks to Sanchayan Sen for sharing his deep technical insights, which has enriched my understanding throughout my PhD.

The summer internship at Microsoft Research Lab New England was a very special experience for me. Many thanks to Christian Borgs and Jennifer Chayes for mentoring me during my internship, for introducing me to a wonderful line of research, and teaching me many fundamental qualities of a researcher. 
Working with Subhabrata Sen in the internship project was an absolute pleasure. 
Subhabrata, your unique ability to ask fundamental questions and simplifying problems taught me a lot.

The NETWORKS group has been an integral part of my academic career during my PhD. I highly appreciate the exposure that I received about contemporary research in diverse fields of probability, combinatorics and computer science.

I would like to thank all the members in the department for a fantastic work-culture and environment. I thank Onno Boxma for giving me a personalized introductory course on queueing theory. Thanks to Julia Komjathy for showing me how to organize an advanced course such as Random graphs. I would like to thank Alberto Brini and Clara Stegehuis for creating a convivial and fantastic work-environment in our office. Thanks to Gianmarco Bet, Kay Bogerd, Lorenzo Federico, Jaron Sanders, Clara Stegehuis, Alessandro Zocca for many interesting research discussions. It was  very much enjoyable to share teaching duties with Angelos Aveklouris, Youri Raaijmakers, Rik Timmerman, and Viktoria Vadon. Special thanks to Enrico Baroni for a memorable trip to Niagara Falls, and for being our guide during our wonderful trip to Italy. Many thanks to Nikhil Bansal, Mark de Berg, Sem Borst, Onno Boxma, Robert Fitzner, Bart Jansen, Tim Hulshof, Julia Komjathy, Nelly Litvak, Maria Vlasiou, and Bert Zwart for giving me extensive feedback during my interview. I would also like to thank Chantal Reemers and Petra Rozema-Hoekerd for helping me out relentlessly through the administrative procedures.

Coming to the Netherlands from India was a major change of culture in my life. 
In the past four years, I have heavily relied on Soma Ray for her advices about the lifestyle in Eindhoven as an Indian. 

Moreover, I am thankful to my teachers at Indian Statistical Institute during my Masters degree. I extend my special thanks to Antar Bandyopadhyay, Sreela Gangopadhyay, and Arup Kumar Pal for their inspiring courses which formed the basis of my background in probability. Thanks to Krishanu Maulik for informing me about this PhD position in  Netherlands, and for sharing his experiences.

This work will not have been possible without the blessings of my parents. 
Thanks to my father for being my first mathematics teacher and inducing in me the passion for mathematics. 
My mother and sister have always been my source of love and affection. My wife Sukanya, thanks for loving me unconditionally through the hard times of life. Finally, thanks to my friend Debankur for being a caring friend, a fantastic teacher, and a great collaborator. Thanks for being there through all the unique experiences of my life for the last nine years.

\newpage 
\tableofcontents

\mainmatter

\chapter[Introduction]{Introduction}
\label{chap:introduction}

The study of large networks, arising from applications in social, physical and life sciences, has witnessed meteoric growth over the past two decades. 
Recent technological innovations allow practitioners to access and study network data of ever-increasing size. 
A thorough understanding of large networks can often provide deep insights into the workings of complex systems.
These networks are usually composed of two key components: 
\begin{enumerate}[(1)]
\item \emph{Structure:} 
The units/individuals are represented by nodes/vertices in the network, whereas their connectivity represents some sort of interaction.
The connectivity structure or topology of these large networks are often  either unknown or highly complex.  
Therefore, probabilistic models of graphs or random graphs have been used extensively to model real-world networks, along with their diverse structural characteristics.
\item \emph{Functionality:} 
Each of these networks come with certain functionality in the real world, such as information exchange or the spread of rumors/diseases. 
The functionality of networks is often modeled as a stochastic process acting upon the network. 
\end{enumerate}
From a mathematical perspective, the understanding of processes on random networks is interesting  due to the inherent double randomness: the random structure produced by the random graph model, and the stochastic process living on this random structure.
The interplay between the random topology and the stochastic process gives rise to novel behavior in the large network regime in terms of scaling limits
and their analysis demands new tools from probability theory.

\subsection*{Network properties and universality.}
Discovering fundamental principles that can describe large complex networks has been a celebrated theme of research in network science.
Below we discuss the main conceptual strands related to this research area: 
\paragraph{Sparsity and power-law degree distributions.}
Empirical analysis of a number of real-world systems such as the internet, citation networks, and  protein-protein interaction networks, seem to suggest that these networks are inherently sparse in the sense that the number of edges scales linearly with the number of vertices.
Moreover, the empirical degree distributions follow approximately power-law distributions. 
Write $N_k$ for the number of vertices with degree $k$ and $p_k = N_k/n$ for the degree distribution with $n$ being the number of vertices in the network.
Then for large $k$, 
\begin{equation}
p_k \approx Ck^{-\tau}, \text{ for } C>0,\text{ and }\tau >2.
\end{equation} 
The constant $\tau$ is called the \emph{degree exponent} of the corresponding network.
This feature makes it impossible to model real-world networks by classical homogeneous random graph models such as the Erd\H{o}s-R\'enyi random graph or the random regular graph. 
Extensive discussions of the veracity as well as limitations of these findings, and the multitude of network models proposed to understand real-world data can be found in \cite{RGCN1,RTD-book,Newman-book,pastor2007evolution,AB02,BC18} and the references therein.

\paragraph{Universality in the large network limit.}
The second major thread that has emerged in network science, especially in the probability community, is the notion of universality: asymptotics in the large network limit,  for a wide range of functionals, often depend mainly on the degree exponent $\tau$.
In turn, this suggests that the qualitative behavior, across a plethora of models, can be largely insensitive to the details of the network model. 

\subsection*{Phase transition.}
Many random graph models for real-world networks are observed to exhibit a \emph{phase transition}. 
For a communication network, where the link between any two servers can fail with a certain probability, one may expect to observe a transition in the connectivity structure of networks depending on the link failure probability. 
One may also consider a computer virus spreading over the internet, which becomes an epidemic over a short window of time. 
The study of phase transitions is often interesting in numerous applications in statistical physics as well.
From the perspective of network science, the objective is to provide a framework for describing  universality laws governing the phase transition in terms of basic characteristics like the degree distribution.
However, even defining rigorously the meaning of a phase transition is a non-trivial task for the simplest models. 
Identifying the point and nature of phase transitions has been a  fundamental question in the development of the current existing theory of random graphs and complex network models.

\subsection*{Aim of this chapter.}
We discuss how the phase transition occurs in random graphs when percolation acts on them.
This is a fundamental model for analyzing the effect of link failure or spread of epidemics on a given network. 
A detailed analysis is provided for the critical behavior in random graphs that generate networks with arbitrary degree distributions.
The idea is to establish key relations between the network statistics such as the power-law exponent, and the nature of the critical behavior of this phase transition. 
%
%
%
%
In particular, we investigate different universality classes for the critical behavior of percolation based on the degree exponent $\tau$. 
In particular, for $\tau>4$, the behavior lies in the same universality class as classical homogeneous random graph models. 
This shows that the inhomogeneity in the degree distribution does not influence the percolation critical behavior as long as the degree distribution has a finite third-moment.
The behavior is more intricate in the other regimes with $\tau\in (3,4)$ and $\tau \in (2,3)$.
In Section~\ref{sec:intro-models}, we give an introduction to some of the most studied random graph models that will be pivotal to the discussions in this thesis. 
In Section~\ref{sec:intro-perc}, we define the percolation phase transition on finite graphs.
In Section~\ref{sec:intro-perc-crit}, we provide an overview of the rich history of critical behavior of percolation on finite graphs,
and formulate the key questions associated to this literature. 
Different universality classes are also described in this section.
In the following sections, we then provide answers to the key questions, and present our contributions to this literature.
The overall goal of this introductory chapter  is to describe the results at a high level, and discuss the central ideas behind our methods. 
The main results, along with all the associated technical details, will be discussed in full detail in subsequent chapters.

\section{Random graph preliminaries}
\label{sec:intro-models}
A graph $G = (V,E)$ consists of a vertex set $V$, and a set of edges $E\subset \{\{i,j\}: i,j\in V\}$ specifying the connections between different vertices.
For a multigraph, $E$ is a multi-set possibly consisting of multiple-edges between vertices, as well as self-loops. 
Throughout, we will assume that $|V|,|E|<\infty$. 
A random graph model specifies a probability distribution over the space of graphs.
We will consider $n$ vertices labeled by $[n]:=\{1,2,...,n\}$, which will serve as the vertex set of the random graph.
We now discuss some classical random graph models, and some properties related to their connectivity structure.
 

\subsection{Random graph models}\label{sec:models-2}

\paragraph*{Erd\H{o}s-R\'enyi random graph.} 
The Erd\H{o}s-R\'enyi random graph is the simplest and most widely studied random graph model, where any two vertices share an edge with some fixed probability~$p$, independently across edges. 
We denote the graph generated by the above procedure on $n$ vertices  by $\mathrm{ERRG}_n(p)$. 
This model was studied in the earliest work on random graphs by  \cite{ER59,ER60,G59}. 
The model in \cite{ER59,ER60} chooses $M$ edges uniformly at random from all possible ${n \choose 2}$ edges, and thus is slightly different. 
However, for $M\approx np$, the two models are asymptotically equivalent \cite{JLR00}. 
Note that the degree of each vertex is distributed as a $\mathrm{Bin}(n-1,p)$ random variable. 
Thus for $p = c/n$, the asymptotic degree of each vertex is $\mathrm{Poisson}(c)$, with fixed average degree~$c$. 

\paragraph*{The configuration model.} 
 Consider a non-increasing sequence of degrees $\boldsymbol{d} = ( d_i )_{i \in [n]}$ such that $\ell_n = \sum_{i \in [n]}d_i$ is even. 
 The configuration model on $n$ vertices having degree sequence $\boldsymbol{d}$ is constructed as follows \cite{B80,BC78}:
 \begin{itemize}
 \item[] Equip vertex $j$ with $d_{j}$ stubs, or \emph{half-edges}. Two half-edges create an edge once they are paired. Therefore, initially we have $\ell_n=\sum_{i \in [n]}d_i$ half-edges. Pick any one half-edge and pair it with a uniformly chosen half-edge from the remaining unpaired half-edges and keep repeating the above procedure until all the unpaired half-edges are exhausted. 
 \end{itemize}
  Let $\CM$ denote the graph constructed by the above procedure.
  Note that $\CM$ may contain self-loops or multiple edges. 
  Given any degree sequence, let $\mathrm{UM}_n(\bld{d})$ denote the graph chosen uniformly at random from the collection of all simple graphs with degree sequence $\boldsymbol{d}$.
  It can be shown that the law of $\mathrm{CM}_{n}(\boldsymbol{d})$, conditioned on the graph being simple, is the same as that of $\mathrm{UM}_n(\bld{d})$ (see \cite[Proposition 7.13]{RGCN1}). 
  Thus, in order to sample a graph uniformly from the space of all simple graphs with a given degree sequence $\bld{d}$, we can keep on generating the configuration model until we obtain a simple graph.
 It was shown in \cite{J09c,AHH16} that, under very general assumptions, the asymptotic probability of the graph being simple is positive, so that with high probability we need to repeat the above algorithm only a finite number of times to generate $\mathrm{UM}_n(\bld{d})$.
The graph $\mathrm{UM}_n(\bld{d})$, in the special case $d_i = d$ for some fixed $d$, is in the literature also known as the random $d$-regular graph.
 
\paragraph*{Inhomogeneous random graph.} 
An inhomogeneous random graph is generated by equipping each vertex $i$  with weight a $w_i$, and creating an edge between vertices $i$ and $j$ with probability $p_{ij} = \kappa(w_i,w_j)$ independently, for some function $\kappa$.
Thus the special case where $\kappa$ is a constant, i.e., $p_{ij}=p$ gives rise to $\mathrm{ERRG}_n(p)$ defined above.
A detailed analysis of the properties of this random graph model has been provided in \cite{BJR07} under a very general setup. 
Some choices of $p_{ij}$ have been popular for their special properties: 
\begin{enumerate}
\item[$\rhd$] \emph{Norros-Reittu model} \cite{NR06}. $p_{ij} = 1-\exp(-w_iw_j/L_n)$, where we define $L_n := \sum_{i\in [n]}w_i$. 
This model is also referred to as  Poissonian graph model, or the Norros-Reittu random graph.  
We denote this model by $\mathrm{NR}_n(\bld{w})$. 
\item[$\rhd$] \emph{Chung-Lu model} \cite{CL02,CL02b}. $p_{ij} = \min \{w_iw_j/L_n,1\}$. This model will be denoted by $\mathrm{CL}_n(\bld{w})$.

\item[$\rhd$] \emph{Generalized random graph} \cite{BDM06}. $p_{ij} = w_iw_j/(L_n+w_iw_j)$. 
We denote this model by $\GRG$.
This model has the property that the distribution of this random graph, conditionally on the degree sequence $\bld{d}$, is the same as for $\mathrm{UM}_n(\bld{d})$.
\end{enumerate}
The weight $w_i$ plays a similar role as degree $d_i$ for $\CM$.
In fact, the expected degree of vertex $i$ is asymptotically $w_i$ in all the random graphs $\mathrm{NR}_n(\bld{w},1)$, $\mathrm{CL}_n(\bld{w})$, and $\GRG$ under some regularity conditions. 
These models are often \emph{asymptotically equivalent}. 
We refer the reader to \cite[Chapter 6]{RGCN1} for a detailed account of these properties.

In the subsequent sections, we will consider a sequence of degrees sequences $(\bld{d}_n)_{n\geq 1}$ and weight sequences $(\bld{w}_n)_{n\geq 1}$ while generating sequences of graphs. 
For notational convenience, we suppress the dependence of the degree and the weight sequences on $n$.

\subsection{Existence of a giant component}\label{sec:intro-existence-giant}

A giant component exists if, in the large network limit, the proportion of vertices in the largest connected component stays bounded away from zero.
In the sparse regime, where the number of edges scales linearly as the number of vertices, a unique giant component exists with high probability for most random graph models. 
Interestingly, global properties such as the existence of the giant component can be approximated by local properties, owing to the well-behaved topology of random graphs.
To understand this more precisely, let $\sC(v)$ denote the component containing vertex $v$ and let $V_n$ denote a vertex chosen uniformly at random. 
Then the existence of the giant component can be characterized by the following two fundamental properties:

\paragraph*{Branching process approximation.} 
 For all random graph models described above, the proportion of vertices involved in a cycle of length at most $2k$ is negligible, so that for any fixed $k\geq 1$, the $k$-neighborhood of $V_n$ is a tree with high probability. 
Therefore the finite neighborhoods of a randomly chosen vertex can be \emph{approximated} by the neighborhoods of the root of an infinite random rooted tree.
In all the random graph models mentioned above, the random tree is a branching process with a suitable progeny distribution. 
However, there are examples of random graph models where the approximating tree is not a simple branching process \cite{BJR07,BBCS14}.
The above notion of approximation can be formalized in terms of \emph{local weak convergence}, introduced by Benjamini and Schramm \cite{BS01} (see also the survey \cite{AS04}). 
We refer the reader to \cite{RGCN2} for local weak convergence results for random graphs. 
Thus, the so-called \emph{local events} (events depending on finite neighborhoods of $V_n$) can be described by functionals of a branching process that are often tractable.

\paragraph*{Approximating global events by local events.}
The components having size $[\omega(n), n/\omega(n)]$, for any $\omega(n) \to\infty$, span an asymptotically negligible proportion of vertices, so that with high probability, either $|\sC(V_n)|$ must have finite size, or $\sC(V_n)$ is the giant component.
 Therefore, when a long path exists from $V_n$, then $V_n$ must be in the giant component, and $\sC(V_n)$ is the giant. 
This is the reason why for random graph models in this section a unique giant component exists with high probability if and only if the mean of the progeny distribution of the approximating branching process is larger than one.

\vspace{.4cm}
For $\ERRG$, it is not difficult to see that the progeny distribution of the approximating branching process is Poisson$(c)$, and indeed a giant component exists precisely when $c>1$ \cite{G59,ER60}.
For $\CM$, note that while pairing the $k$-th half-edge, the probability of pairing to a vertex of degree $r$ is approximately $rn_r/\ell_n$. 
Thus the degree of a neighbor of a vertex is approximately given by the \emph{size-biased distribution} 
\begin{eq}\label{eq:SB-dist-degree}
\PR(D_n^* = r) = \frac{rn_r}{\ell_n}.
\end{eq} 
Therefore, the approximating branching process has progeny $D_n^*-1$, since one edge is connected to the parent of a vertex. 
The expectation of $D_n^*-1$ is
\begin{eq}
\nu_n: = \frac{\sum_{r=1 }^\infty r(r-1)n_r}{\ell_n} = \frac{\sum_{i\in [n]}d_i(d_i-1)}{\ell_n}.
\end{eq} 
It was established in \cite{MR95,JL09} (see also the recent results \cite{JPRR18}) that the giant component exists precisely when 
\begin{equation}\label{eq:conv-nu-local-weak}
\lim_{n\to\infty}\nu_n = \nu >1,
\end{equation}
again confirming the local weak-limit heuristics.
It can also be shown that for the inhomogeneous random graph models, the mean of the approximating branching process turns out to be $\sum_{i\in [n]}w_i^2/\sum_{i\in [n]}w_i$, and the giant component exists \cite{BJR07,RGCN1} when  
\begin{eq}
\lim_{n\to\infty}\frac{\sum_{i\in [n]}w_i^2}{\sum_{i\in [n]}w_i}>1.
\end{eq}


\section{Percolation on finite graphs}
\label{sec:intro-perc} 

\paragraph*{Percolation process and Harris coupling.}
Given a graph $G$, bond (site) percolation refers to deleting each edge (vertex)  independently with probability $1-p$. 
Throughout, we will be interested in bond percolation; thus we simply write percolation which refers to bond percolation, and the obtained graph is denoted by $G(p)$.
In case of percolation on random graphs, the deletion of edges is also independent from the underlying random graph. 
With the \emph{percolation process}, we refer to the graph-valued stochastic process $(G(p))_{p\in [0,1]}$ coupled through the so-called Harris coupling. More precisely: 
\begin{itemize}
\item[] Associate an independent uniform $[0,1]$ random variable $U_e$ to each edge $e$ of the graph $G$. $G(p)$ can be generated by keeping edge $e$ iff $U_e\leq p$. Keeping the uniform random variables fixed while varying $p$, gives a coupling between the graphs $(G(p))_{p\in [0,1]}$.
\end{itemize}

Classically, percolation has been extensively studied on infinite connected graphs such as the hypercubic lattice. 
This is the simplest known model that exhibits a \emph{phase transition}. 
If $p$ is small, then $G(p)$ consists of connected components of finite size only. 
On the other hand, if $p$ is close to~1,  then $G(p)$ contains an infinite cluster. 
Since for $p_1<p_2$, $G(p_1)$ is a subgraph of  $G(p_2)$ under the Harris coupling, there exists a \emph{unique} value $p_c$ such that $G(p)$ contains an infinite cluster if and only if $p>p_c$. 
Thus $p_c$ can also be defined as the unique point of discontinuity of the function $\PR(\exists \text{ an infinite component in }G(p))$.
The quantity $p_c$ is called the \emph{critical value} for the phase transition of the percolation process. 
Several questions like finding the value of $p_c$, establishing the uniqueness of the infinite components and behavior of different functionals close to $p_c$ for infinite transitive graphs have been discussed extensively in \cite{Gri99,BR06,Kes82,HvdH17}.

\paragraph*{Phase transition on finite graphs.}
It is not evident how to define the phase transition for a fixed finite graph $G_n$ with $n$ vertices. 
For any $f:G_n\mapsto \R$, the expectation $\E[f(G_n(p))]$ is a smooth function in $p$, and therefore none of the functionals of the graph experience a  transition. 
The phase transition can only arise when the graphs become large, i.e., it should be  related to the large-network limit.
%
For this reason, the phase transition is defined for a sequence of graphs $(G_n)_{n\geq 1}$ rather than a given fixed graph, and the transition is captured in terms of the limit  $\lim_{n\to\infty}\E[f(G_n(p))]$.   
Let $\sC_{\sss(k)}(p)$ denote the $k$-th largest connected component of $G_n(p)$.  
The \emph{critical value} $p_c = p_c(n)$ is defined such that the following holds given any $\varepsilon>0$: 
\begin{eq} \label{eq:phase-transition-defn} 
\lim_{n\to\infty}\PR(|\sC_{\sss (1)}(p)| > \delta n )=
\begin{cases}
 0 \quad &\text{for }p< p_c(1-\varepsilon), \quad \text{({subcritical})}\\
 1  \quad &\text{for }p> p_c(1+\varepsilon), \quad \text{({supercritical})}
\end{cases}
\end{eq}
where the first limit should hold for all $\delta>0$, and the second one for some $\delta = \delta(p)>0$ sufficiently small. 
In most cases, $|\sC_{\sss (1)}(p)|/n$ converges in probability to some positive constant $\eta$ that depends on the graph sequence $(G_n)_{n\geq 1}$ and $p$, and more importantly $|\sC_{\sss (1)}(p)|\gg|\sC_{\sss (2)}(p)|$, so that the giant component is \emph{unique}. 
We will stick to the above definition of the phase transition and critical value throughout this thesis.
It is worthwhile mentioning that there is a substantial literature on how to define the critical value, and the phase transition. 
See \cite{NP08,JW18,BCHSS05,HvdH17,Hof17} for different definitions of critical probability and related discussions. 
One could also note that $p_c$ is not unique for finite graphs. This is due to the fact that $p_c$ is allowed to depend on $n$ and the phase transition is only an asymptotic notion.  

The study of random graphs practically started with the question of identifying the critical value of $p_c$. 
Note that percolation on a complete graph yields an Erd\H{o}s-R\'enyi random graph, and in the early works \cite{G59,ER60}, it was shown that $p_c=1/(n-1)$.
Over the past 60 years of development of the random graph literature, identifying the critical value and the asymptotics of the giant component has been one of the guiding questions, not only for the percolation process, but for any sequence of dynamically growing graph processes.
The threshold has been identified under fairly general conditions, for example when the underlying graph is an expander \cite{BBLR12},
converges in a cut metric to an irreducible graphon \cite{BBCR10}, or even general sequences of graphs \cite{CHL09,Ald16}.
For more details we refer to \cite{J09,F07,JLR00,ABS04,RGCN1,Bol01,BJR07,BBCR10,RW12,BBLR12,CHL09,Ald16} and the references therein.

\paragraph*{Percolation on random graphs.}
For a sequence of random graphs, the phase-transition of the percolation process occurs when \eqref{eq:phase-transition-defn} holds with high probability with respect to the joint distribution of the random graph and the percolation process. 
The percolation process is viewed as a dynamic process living on a disordered medium, i.e., the random graph. 
In a sense, this may appear paradoxical, since percolation itself often serves as a model for generating a disordered medium on which stochastic processes like random walks act. 
However, when viewing random graphs as models for real-world networks, 
percolation serves as a model for robustness of internet or communication networks when the nodes/edges of the underlying network experience random damage. 
Percolation has also been used to model the vaccination on a network to prevent the growth of an epidemic. 
A detailed account of these applications can be found in \cite{Newman-book,Bar16}.

%
%

Of particular interest to this thesis is the phase transition result for uniformly chosen graphs with given degree $\mathrm{UM}_n(\bld{d})$, and the configuration model $\CM$.
For random $d$-regular graphs \cite{ABS04}, $p_c = 1/(d-1)$.
The percolation phase transition on $\mathrm{UM}_n(\bld{d})$ and on $\CM$ was studied in \cite{F07,J09} when the empirical degree distribution has a finite second moment in the large network limit, as $n\to\infty$. 
It was shown that $p_c = 1/\nu$, where the parameter $\nu$ is defined by \eqref{eq:conv-nu-local-weak}.
See also \cite{FJP16,JP18} for some recent results on more general degree sequences.

\paragraph*{Relation to branching process approximation.}
The critical probability $p_c$ is intimately related to the branching process approximation. 
Indeed, when the number of edges in $G_n(p)$ scales linearly with $n$, the typical local neighborhoods of $G_n(p)$ can be approximated by a branching process in the sense of Section~\ref{sec:intro-existence-giant} for most sequences of (random) graphs $(G_n)_{n\geq 1}$.
The value $p_c$ is then such that the mean of this approximating branching process is 1.
Indeed, that turns out to be the case for $\ERRG$, random regular graphs, and $\CM$. 
Heuristically, these results complement Schramm's conjecture about infinite transitive graphs stating that the local weak limit determines the percolation threshold. 
Without going into further details, we refer the interested reader to \cite{BNP11,BPT17,DcT17} and the references therein for a beautiful line of work initiated with Schramm's conjecture.

\paragraph*{Formation of a complex structure.}
Around the critical value, the phase transition happens not only with respect to the size of $\sC_{\sss (1)}$, but also with respect to the \emph{complexity} of its connectivity structure. 
To measure complexity, let us define for a connected graph $G$ the number of surplus edges as the number of edges to be deleted to turn $G$ into a tree. 
Thus $\#$ surplus edges of $G$ = $\#$ edges $-$ $\#$ vertices + 1, denoted by $\SP(G)$. 
Note that $\SP(G) = 0$ means that $G$ is a tree, and a large value of $\SP(G)$ means that $G$ has many (possibly overlapping) cycles with a more complex structure. 
In the subcritical regime, any component has at most one surplus edge \cite{JLR00,HM12,DMS17,Bol01}, and there are finitely many surplus edges in the whole graph, so that the subcritical components are mostly trees. 
On the other hand, the giant component in the supercritical regime satisfies $\SP(\sC_{\sss (1)}(p)) \to\infty$ \cite{JL09,JLR00} with high probability, so that the structure of the giant is highly complex. 
See \cite{DKLP10} for a detailed result about the giant component of $\ERRG$. 
Thus the percolation process starts adding cycles and the complex structure of the giant component begins to form precisely around the critical value $p_c$.
This explains the interest in the percolation critical behavior.

\subsection{Some definitions and notation.} \label{sec:notation-intro} In the next section, we discuss the critical window of phase transition. 
We now define some basic notation used throughout this thesis.
We will use the standard notation  $\xrightarrow{\sss \mathbb{P}}$, $\xrightarrow{\sss d}$ to denote convergence in probability and in distribution or law, respectively. 
We often use the Bachmann Landau notation $O(\cdot)$, $o(\cdot)$, $\Theta(\cdot)$ for large $n$ asymptotics of real numbers.
The topology needed for the distributional convergence will always be specified unless it is  clear from the context.  A sequence of events $(\mathcal{E}_n)_{n\geq 1}$ is said to occur with high probability with respect to probability measures $(\mathbb{P}_n)_{n\geq 1}$  if $\mathbb{P}_n\big( \mathcal{E}_n \big) \to 1$. Denote $f_n = O_{\sss \mathbb{P}}(g_n)$ if $ ( |f_n|/|g_n| )_{n \geq 1} $ is tight; $f_n =o_{\sss \mathbb{P}}(g_n)$ if $(|f_n|/|g_n|)_{n\geq 1}$ converges in probability to zero; $f_n =\Theta_{\sss \mathbb{P}}(g_n)$ if $f_n=O_{\sss \mathbb{P}}(g_n)$ and $g_n=O_{\sss \mathbb{P}}(f_n)$. 
Denote by 
\begin{equation}\ell^p_{\shortarrow}:= \big\{ \mathbf{x}= (x_1, x_2, x_3, ...): x_1 \geq x_2 \geq x_3 \geq ... \text{ and } \sum_{i=1}^{\infty} x_{i}^p < \infty \big\},
\end{equation}the subspace of non-negative, non-increasing  sequences of real numbers with square norm metric $d(\mathbf{x}, \mathbf{y})=( \sum_{i=1}^{\infty} (x_i-y_i)^p )^{1/p}$.
Let $(\ell^2_{\shortarrow})^k$ denote the $k$-fold product space of $\ell^2_{\shortarrow}$. 
With $\ell^2_{\shortarrow} \times \mathbb{N}^{\infty}$, we denote the product topology of $\ell^2_{\shortarrow}$ and $\mathbb{N}^{\infty}$, where $\mathbb{N}^{\infty}$ denotes the collection of sequences on~$\mathbb{N}$, endowed with the product topology. Define also
\begin{equation}
\mathbb{U}_{\shortarrow}:= \Big\{ ((x_i,y_i))_{i=1}^{\infty}\in  \ell^2_{\shortarrow} \times \mathbb{N}^{\infty}: \sum_{i=1}^{\infty} x_iy_i < \infty \text{ and } y_i=0 \text{ whenever } x_i=0   \Big\}
\end{equation} with the metric \begin{equation} \label{c0:defn_U_metric}d_{\mathbb{U}}((\mathbf{x}_1, \mathbf{y}_1), (\mathbf{x}_2, \mathbf{y}_2)):= \bigg( \sum_{i=1}^{\infty} (x_{1i}-x_{2i})^2 \bigg)^{1/2}+ \sum_{i=1}^{\infty} \big| x_{1i} y_{1i} - x_{2i}y_{2i}\big|.
\end{equation} Further, we introduce $\mathbb{U}^0_{\shortarrow} \subset \mathbb{U}_{\shortarrow}$ as \begin{equation}\mathbb{U}^0_{\shortarrow}:= \big\{((x_i,y_i))_{i=1}^{\infty}\in\mathbb{U}_{\shortarrow} : \text{ if } x_k = x_m, k \leq m,\text{ then }y_k \geq y_m\big\}.
\end{equation}
 We usually use the boldface notation $\mathbf{X}$ for a time-dependent stochastic process $( X(s))_{s \geq 0}$, unless stated otherwise, $\mathbb{C}[0,t]$  denotes the set of all continuous functions  from $[0,t]$ to $\mathbb{R}$ equipped with the topology induced by sup-norm $||\cdot||_{t}$. 
Similarly, $\mathbb{D}[0,t]$ (resp.~$\mathbb{D}[0,\infty)$) denotes the set of all c\`adl\`ag functions from $[0,t]$ (resp.~$[0,\infty)$) to $\mathbb{R}$ equipped with the Skorohod $J_1$ topology~\cite{Bil99}. 

\section{Critical window and emergence of the giant}
\label{sec:intro-perc-crit}
The critical regime lies on the boundary between the subcritical and supercritical regimes, where the system exhibits an intermediate behavior.
From a statistical physics perspective, this is the interesting regime to study because the properties in the critical regime help to answer the question ``How did the phase transition happen?". 
Here, one tries to identify  principles that govern the phase transition, which not only depend on the specifics of the system, but hold  universally for a large class of systems.
From a mathematical perspective, critical behavior often gives rise to novel scaling limit results.
%
In this section, we first discuss the importance and relevance of studying the critical behavior for percolation processes. 
Then we state some key questions about the critical behavior of percolation in Section~\ref{sec:key-question}.
In Section~\ref{sec:universality-classes}, we describe three fundamental types of critical behavior, i.e., universality classes, that will be crucial throughout this thesis.
We finish this section with a review of the related literature and the relevance of our work; see Section~\ref{sec:literature-review}.

\paragraph*{Critical window.}
To observe the critical behavior, one must take $p=p_c(1\mp \varepsilon_n)$ in \eqref{eq:phase-transition-defn}, for some $\varepsilon_n\to 0$ as $n\to\infty$.
Interestingly, the critical behavior is not observed for any $\varepsilon_n$; there is a range of $\varepsilon_n$ where the graph shows qualitatively similar features as the sub/supercritical regimes and the critical behavior is observed only when $\varepsilon_n$ is chosen appropriately. 
In most situations, this means that $\varepsilon_n = \Theta (n^{-\eta})$, where $\eta>0$ is a model-dependent constant. 
To be more precise, recall that $\sC_{\sss (i)}(p)$ denotes the $i$-th largest component of $G_n(p)$. 
The following are classical results \cite{Bol01,JLR00} for $G_n(p) = $ Erd\H{o}s-R\'enyi random graph (i.e., $G_n$ is the complete graph), where $\eta =1/3$: 
\begin{enumerate}[(a)]
\item \emph{Barely subcritical regime:} $p=p_c(1-\varepsilon_n)$ with $\varepsilon_n n^{\eta} \to \infty$.
Then for each fixed $i\geq 1$, as $n\to\infty$,
\begin{equation}\label{eq:barely-subcrit-feature}
\frac{|\sC_{\sss (i)}(p)|}{2\varepsilon_n^{-2} \log(n\varepsilon_n^3)} \pto 1, \quad\text{and}\quad\PR(\exists i: \SP(\sC_{\sss (i)}(p)) >1) \to 0.
\end{equation}
Thus, $G_n(p)$ shows the two characteristic features of the subcritical regime: $|\sC_{\sss (1)}(p)|$ is not distinctively larger than $|\sC_{\sss (2)}(p)|$, and $G_n(p)$ is essentially a collection of trees. 
Thus, even if $p \approx p_c$, $G_n(p)$ is subcritical in this regime.
This regime is often referred to as the barely subcritical regime in the literature.  
\item \emph{Barely supercritical regime:} $p=p_c(1+\varepsilon_n)$ with $\varepsilon_n n^{\eta} \to \infty$. 
Then, as $n\to\infty$, 
\begin{gather}\label{eq:barely-supercrit-feature}
\frac{|\sC_{\sss (1)}(p)|}{2n\varepsilon_n} \pto 1, \quad \frac{\SP(\sC_{\sss (1)}(p))}{n\varepsilon_n^3} \pto \frac{2}{3}, \quad \frac{|\sC_{\sss (2)}(p)|}{ |\sC_{\sss (1)}(p)|}\pto 0, \\
\text{and}\quad \PR(\exists i\geq 2: \SP(\sC_{\sss (i)}(p)) >1) \to 0. \nonumber
\end{gather}
See \cite[Section 23]{JKLP93}, \cite{JLR00,Bol01}. 
Thus, $G_n(p)$ exhibits two characteristic features in the supercritical regime: $|\sC_{\sss (1)}(p)|$ is considerable larger than all other components, and  $\sC_{\sss (1)}(p)$ is complex in the sense that there is a growing number of surplus edges, while all other components are trees. 
\end{enumerate}
\noindent 
Although the above formulations are stated for the Erd\H{o}s-R\'enyi
 random graph, the recent literature has provided many interesting results about the barely subcritical regimes \cite{J08,BDHS17,BHL12} and supercritical regimes \cite{JL09,HJL16} for graphs with general degree sequence (see also \cite{KS08,HM12,R12} for results in both regimes). 
Now, the phase transition takes place between the barely subcritical and supercritical regimes when $\varepsilon_n \sim n^{-\eta}$.  
This regime is known as the \emph{critical window} for the phase transition.  
More precisely, the critical window is defined to be the values of $p$ given by
\begin{equation}\label{ths:defn-crit-value}
p_c(\lambda) = p_c(1+\lambda n^{-\eta}), \quad -\infty<\lambda<\infty.
\end{equation}
In this regime, the largest components exhibit features that are completely different than  the subcritical, or the supercritical regime: 
 There exists a model-dependent exponent $\rho>0$ such that 
\begin{eq} \label{eq:non-concentrating-comp-size}
n^{-\rho} (|\sC_{\sss (i)}(p)|)_{i\geq 1} \text{ converges to a non-degenerate random vector.}
\end{eq}
 Further, for any $i\geq 1$, $\liminf_{n\to\infty}\PR(\SP(\sC_{\sss (i)}(p)) > 1) >0 $, but $\SP(\sC_{\sss (i)})$ is tight,  so that the surplus-edge count for large components starts to grow in the critical window.  
The above two properties hold for all values of $\lambda$ in \eqref{ths:defn-crit-value}; in this sense there is not a single critical value, but a whole ``window'' of critical values over which the phase transition happens.
This is due to finite-size effects and the joint scaling of $\varepsilon_n$ and $n$, a feature that is typically absent in the case of the phase transition on infinite graphs.
The exponent $\eta$ in \eqref{ths:defn-crit-value} is chosen as largest value such that the limit of \eqref{eq:non-concentrating-comp-size} depends on $\lambda$, so that $\eta$ is uniquely defined.

Paul Erd\H{o}s described the percolation process as the race between the components to become the giant \cite{ABS17}. 
\begin{figure} 
\begin{center}
\begin{tikzpicture}[scale=.98]
\filldraw[ right color=blue!30!,left color=red!30!, minimum height = 2.5cm] (0,0) rectangle node  {
\begin{tabular}{ccccccc}
$\varepsilon>0$ &$\varepsilon_n \gg n^{-\eta}$ & $\qquad $ & $\varepsilon_n \sim n^{-\eta}$ &$\qquad$ &$\varepsilon_n \gg n^{-\eta}$ & $\varepsilon>0$\\
 \multicolumn{2}{c}{Subcritical}& $\qquad$ & Critical window& $\qquad$ & \multicolumn{2}{c}{Supercritical}
\end{tabular}} +(12.2,1.2);
\node (A) at (1.4,-.5) { Mostly trees};
\node (B) at (6.1,-.5) {  Components merge};
\node (C) at (10.9,-.5) { Birth of giant};
\draw[->,thick ] (A) -- (B);
\draw[->,thick] (B) -- (C);
\end{tikzpicture}
\end{center}
\end{figure} 
The mental picture is that the collection of trees in the barely subcritical regime are the participants of this race and the component that outnumbers the other components in terms of the number of vertices wins the race. 
As the percolation parameter transitions through the critical window with $\lambda$ increasing with the Harris coupling in place, components grow in size and complexity, and the race is on.
$\mathscr{C}_{\sss (1)}(p_c(\lambda_1))$ and $\mathscr{C}_{\sss (1)}(p_c(\lambda_2))$ can be completely disjoint sets of vertices for $\lambda_1\neq \lambda_2$.
However, at the end of the critical window, when $\lambda$ becomes  sufficiently large, the leader $\mathscr{C}_{\sss (1)}(p_c(\lambda))$ stops changing and this leader becomes the young giant component at the end of the critical window. 
At the barely supercritical phase, the race ends and the largest component stays the largest throughout the future of the percolation process.  
See \cite{ABS17} for a formalization of this picture under a general setup.

It is worthwhile to highlight the fact in \eqref{eq:non-concentrating-comp-size} that the component sizes, after proper rescaling, converge to non-degenerate random variables.
This is a special feature of the critical window that is never observed in the sub/supercritical regime. 
In fact, to the best of our knowledge, all dynamic graph processes that show phase transition with respect to its component sizes, exhibits this feature.
Thus, this property could be considered as a potential definition of the critical window. 


\subsection{Key questions} \label{sec:key-question}
We now describe the key questions about the percolation process in the critical window that we address in this thesis.
\paragraph*{(1) Component sizes and surplus edges.} 
The phase transition typically happens with respect to functionals such as the size of the largest components and their surplus edges.
Therefore, the most natural approach in this context is to find limit theorems for these functionals.
For each fixed $\lambda$, consider $p_c(\lambda)$ defined in \eqref{ths:defn-crit-value}, and define 
\begin{eq}\label{defn:com-surp-Zn}
\mathbf{Z}_n(\lambda):= (n^{-\rho}|\sC_{\sss (i)}(p_c(\lambda))|,\SP(\sC_{\sss (i)}(p_c(\lambda))))_{i\geq 1}
\end{eq} 
for some model-dependent constant $\rho>0$. 
As discussed in \eqref{eq:non-concentrating-comp-size}, $\mathbf{Z}_n(\lambda)$ is expected to  converge in distribution to some non-degenerate random vector. 
Since one deals with convergence of infinite-dimensional random vectors, the topology for the underlying distributional convergence turns out to be  important, because one gets convergence of more functionals under a stronger topology.
The results of this thesis will be discussed under the $\mathbb{U}^0_{\shortarrow}$ topology, defined in Section~\ref{sec:notation-intro}.

\paragraph*{(2) Evolution when passing through the critical window.}
As mentioned before, there is not a single critical value here, but a whole window of critical values. 
It is thus interesting to explore the relation between the relevant component functionals for different values of $\lambda$.
Now, under the Harris coupling, $\mathbf{Z}_n(\lambda)$ can be viewed as a stochastic process in $-\infty<\lambda<\infty$.
As $\lambda$ increases, more and more edges get added in $G_n(p_c(\lambda))$, and $\mathbf{Z}_n(\lambda)$ evolves. 
In the context of the ``race to become a giant", $(\mathbf{Z}_n(\lambda))_{\lambda\in\R}$ is the movie of this race. 
Therefore it is desirable to study the limit of the stochastic process $(\mathbf{Z}_n(\lambda))_{\lambda\in\R}$. 
This is a $\mathbb{U}^0_{\shortarrow}$-valued process, and we will consider the topology $(\mathbb{U}^0_{\shortarrow})^\N$ for convergence of this process.

\paragraph*{(3) Global metric structure.}
A recent direction in this literature aims to find the global structure, and characterize the distance-related functionals of the components. 
The motivation comes from understanding the minimal spanning tree on a random network, which is important in many contexts like the spread of epidemics. 
Of course, the term \emph{global structure} is a bit vague;
however, this can be formalized.  
Each component $\mathscr{C} \subset G_n((p_c(\lambda))$ can be viewed as a metric space, equipped with a measure on the associated Borel sigma algebra.
 The metric on $\sC$ is the graph-distance where (i) each edge has length one, (ii) the measure being proportional to the counting measure, i.e., for any $A\subset \mathscr{C}$, the measure of $A$ is given by $\mu_{\sss \mathrm{ct}}(A) = |A|/|\mathscr{C}|$. 
 Then, $\sC_{\sss (i)}(p_c(\lambda))$ can be viewed as a random element from $\mathscr{M}$, the space of metric spaces with an associated probability measure. 
For $M = (M,\mathrm{d},\mu) \in \mathscr{M}$ and $a>0$, define  $aM$ to be the measured metric space $(M,a\mathrm{d},\mu)$.  
Then the goal is to 
\begin{eq} \label{eq:metric-space-limit-goal}
\text{find the distributional limit of } \big( n^{-\delta}\mathscr{C}_{\sss (i)}(p_c(\lambda)) \big)_{i\geq 1}. 
\end{eq} 
Since the limit is obtained after rescaling of graph-distances by $n^{\delta}$, and the limit is usually a compact metric space, the distances in $\mathscr{C}_{\sss (i)}(p_c(\lambda)$ scale as $n^{\delta}$.
The above quantity is an $\mathscr{M}$-valued sequence.
Of course, the topology on $\mathscr{M}$ is important, and we will explore two different topologies, namely, the Gromov weak-topology and the Gromov-Hausdorff-Prokhov topology. 

Another key question, which should have been stated as question (0) but will not be a topic of our discussion, is finding the value of $\eta$. 
For graphs, and in particular those with an underlying geometric structure, finding $\eta$ is a highly non-trivial task. 
An interested reader is referred to \cite{FHHH17,HS05} 
and the references therein.
In the context of the models in this thesis, values of $\eta$ are  described below while discussing the different universality classes.

\subsection{Major universality classes} \label{sec:universality-classes}
In a seminal work, Aldous~\cite{A97} studied the first two questions above in the context of Erd\H{o}s-R\'enyi random graphs.
It turns out that $\eta =1/3$, and $\rho =2/3$.
Along with identifying the limiting object for the component sizes, Aldous observed that the evolution of the rescaled component sizes can be described by a process called the multiplicative coalescent; 
see Chapter~\ref{chap:thirdmoment} for a precise definition.
The first result about the convergence of the global structure was provided recently in \cite{ABG09} for the critical Erd\H{o}s-R\'enyi random graphs with $\delta = 1/3$. 
%
%
Subsequently, there has been a surge in the literature to understand the most general cases under which one can establish qualitatively similar behavior as the Erd\H{o}s-R\'enyi random graph, and identify the cases when the behavior is different.
Following the above discussion, two universality classes have emerged  in the literature. 
It turns out that, when the asymptotic degree distribution follows a power-law with exponent~$\tau$, there is a transition in the critical behavior with respect to the exponent~$\tau$.
\paragraph*{Erd\H{o}s-R\'enyi universality class.}
For $\tau > 4$, the asymptotic empirical degree distribution has a finite third moment. 
In this case, the critical window turns out to be $p = p_c(1+\lambda n^{-1/3})$, the maximal component sizes $|\sC_{\sss(i)} (p_c(\lambda))|$, for any fixed $i$, are of the order $n^{2/3}$ in the critical regime, whilst typical distances in these maximal connected components scale like $n^{1/3}$. 
Thus $\rho = 2/3$ and $\eta =\delta= 1/3$.
Moreover, the scaling limits are the same as for the Erd\H{o}s-R\'enyi random graphs 
up to constant factors of adjustment in the parameters.
\paragraph*{Heavy-tailed behavior. } 
For $\tau \in (3,4)$, the asymptotic degree distribution has an infinite third moment, but a finite second moment.
Here the critical window turns out to be $p = p_c(1+\lambda n^{-(\tau-3)/(\tau-1)})$, $|\sC_{\sss(i)}(p_c(\lambda))|$ is of the order $n^{(\tau-2)/(\tau-1)}$, whilst distances scale like $n^{(\tau-3)/(\tau - 1)}$. 
Thus $\rho = (\tau-2)/(\tau-1)$ and $\eta = \delta =(\tau-3)/(\tau-1)$.
The scaling limits turn out to be completely different in this regime.
For example, it turns out that the high-degree vertices play a crucial role in the connectivity structure of $\sC_{\sss(i)}(p_c(\lambda))$ in the sense that a deliberate deletion of the $k$-th highest degree vertex changes the scaling limit completely. 
This is in sharp contrast with the behavior for the $\tau >4$ regime. 

\paragraph{Universality in the evolution of the components.}
To intuitively understand the evolution of the component sizes and surplus edges, let us consider the Erd\H{o}s-R\'enyi case.
After increasing $p$ slightly, a new edge might appear in the graph, and due to the homogeneity in the connectivity structure of Erd\H{o}s-R\'enyi random graphs, this edge selects two end-points uniformly at random. 
For this reason, two components $\sC_{\sss (i)}(p_c(\lambda))$ and $\sC_{\sss (j)}(p_c(\lambda))$ merge at rate $|\sC_{\sss (i)}(p_c(\lambda))|\times |\sC_{\sss (j)}(p_c(\lambda))|$ and create a component of size $|\sC_{\sss (i)}(p_c(\lambda))| +|\sC_{\sss (j)}(p_c(\lambda))|$. 
Moreover, a surplus edge is created in  $\sC_{\sss (i)}(p_c(\lambda))$ at rate \linebreak $|\sC_{\sss (i)}(p_c(\lambda))|\times(|\sC_{\sss (i)}(p_c(\lambda))|-1)/2$.
This merging dynamics of a collection of particles according to the product of their weights is known as the \emph{multiplicative coalescent} \cite{A97,Ald99}.
The creation of surplus edges can also be augmented in the evolution of the component sizes~\cite{BBW12}.
In both the $\tau>4$ and $\tau\in (3,4)$ regimes, the above merging dynamics describes the evolution of the component sizes and surplus edges over the critical window for a wide array of models, in-spite of the dependence in the connectivity structure.
Thus, even if the scaling limits for $\tau>4$ and $\tau\in (3,4)$ are completely different for each fixed $\lambda$, the merging dynamics as $\lambda$ varies is guided by the same dynamics.

\subsection{Literature review and the relevance of our work}
\label{sec:literature-review} 
Each of the key questions (1)--(3) in Section~\ref{sec:key-question} have posed novel theoretical challenges in probability theory and combinatorics over the past decades. 
The study of critical random graphs began in the 1990's with the early works \cite{Bol84,L90,JKLP93,LPW94} on critical Erd\H{o}s-R\'enyi random graphs, where it was shown that the critical window is $p = n^{-1}(1+\lambda n^{-1/3})$, and the component sizes are of the order $n^{2/3}$, whereas the surplus edges are $O(1)$. 
In a seminal work \cite{A97}, Aldous derived the exact scaling limits of the rescaled component sizes and surplus edges, and showed that the evolution of the component sizes over the critical window can be described by the multiplicative coalescent process. 
This initiated the program for a large body of subsequent work \cite{NP10a,NP10b,R12,Jo10,BHL10,BBW12,AP00,H13,DLV14,BM14,HN17,HLC16,FHHH16}, showing that the behavior of a wide array of random graphs at criticality is universal in the sense that it does not depend on the precise description of the model. 
Of particular relevance to this thesis are the works on $\CM$ \cite{Jo10,NP10b,R12}. 
The question (1), for the $d$-regular case, was extensively analyzed in \cite{NP10b}, and the scaling limit for $\mathbf{Z}_n(\lambda)$ was derived for the critical $\CM$ with bounded maximum degree in \cite{R12} under the product topology. 
The results in \cite{Jo10} considered the special case that the degrees are an iid sample from a distribution having finite third-moment. 
Scaling limit results were derived for the component sizes; however there is no notion of ``critical window" in this set up. 
In Section~\ref{sec:finite-third-moment}, we discuss the joint convergence of the component sizes and the surplus edges  when the degree distribution satisfies a finite third-moment condition. 

%
%
%
%
%

The second major universality class emerged with the study of dynamically evolving random networks given by the Norros-Reittu random graph model, with a heavy-tailed empirical distribution of average degrees. 
In \cite{BHL12,H13}, the critical window was identified, along with scaling limit results for component sizes. 
In Section~\ref{sec:infinite-third-moment-intro}, we will show that the scaling limit of $\mathbf{Z}_n(\lambda)$ for $\CM$ under the heavy-tailed setup  lies in the universality class of \cite{BHL12}. 
In fact, the results are stronger than \cite{BHL12} in terms of the topology of convergence. Joseph~\cite{Jo10} studied the iid degree case, where the scaling limit turns out to be somewhat different than \cite{BHL12}.

In the context of the evolution of the component sizes, Aldous~\cite{A97} first studied the evolution of the component sizes. 
The evolution of the component sizes was also studied in the context of random graphs with immigrating vertices \cite{AP00}, and the Norros-Reittu random graph \cite{BHL12}. 
See also \cite{BM14} for a construction of the multiplicative coalescent.
A complete description of this process, along with its entrance boundary conditions, was provided in \cite{AL98}.
This was generalized to augmented multiplicative coalescent processes in \cite{BBW12} to capture the evolution of the surplus edges as well. 
In Section~\ref{sec:intro-Q2}, we describe the evolution of $\mathbf{Z}_n(\lambda)$ in both universality classes. 
%
%
%
%
%
%
%
%
%
%
%
%
%
%
%
%
%

%

The study of the global metric structure is a recently emerging direction in this field, which started with the pioneering work \cite{ABG09} on critical Erd\H{o}s-R\'enyi random graphs. 
The scaling limit identified in \cite{ABG09} was shown to be universal for the $\tau>4$ regime in a recent line of work \cite{BBSX14,BSW14,BS16}. 
In the context of critical random graphs with degree-exponent $\tau\in (3,4)$, candidate limit laws of maximal components with each edge rescaled to have length $1/n^{(\tau-3)/(\tau - 1)}$ were established in~\cite{BHS15}.
In Section~\ref{sec:intro-Q3}, we describe a ``universality principle'' for the $\tau\in (3,4)$ regime, which yields the scaling limits for $\mathrm{CM}_n(\bld{d},p_c(\lambda))$.

%
%
%
%
%
%
%

\subsection{A new universality class.}
All the above literature assumes a finite second-moment condition on the degree distribution, and thus does not include the $\tau\in (2,3)$ case, where the asymptotic degree distribution has an infinite second moment but a finite first moment.
These networks are known in the literature as \emph{scale-free networks} \cite{Bar16}. 
One of the popular features of scale-free networks is that these networks are \emph{robust} under random edge-deletion, i.e., for any sequence $(p_n)_{n\geq 1}$ with $\liminf_{n\to\infty} p_n > 0$, the graph obtained by applying percolation with probability $p_n$ is always supercritical. 
This feature has been studied experimentally in~\cite{AJB00}, using heuristic arguments in~\cite{CEbAH00,CNSW00,DGM08,CbAH02} (see also \cite{braunstein2003optimal,braunstein2007optimal,Halvin05} in the context of optimal paths in the strong disorder regime), and mathematically in \cite{BR03}. 
Thus, in order to observe the percolation critical behavior, one needs to take $p_c \to 0$ with the network size, even if the average degree of the network is finite.
It was predicted from the physics literature that the critical value should be $p_c \sim n^{-(3-\tau)/(\tau-1)}$: Detailed properties of the component sizes and structures remained as open question. 

In Section~\ref{sec:intro-infinite-second}, we discuss the first mathematically rigorous results in the $\tau \in (2,3)$ regime for  component sizes and their complexity.
The most striking thing about the results in the $\tau \in (2,3)$ regime is that the critical value changes depending on whether the underlying random graph has the so-called \emph{single-edge constraint}, i.e., the critical value when the underlying graph is a random multigraph generated by the configuration model is different than that under models like the erased configuration model $\ECM$ and the generalized random graph $\GRG$, where the underlying graph is simple.
This feature was never observed in the finite second-moment scenario.
For $\CM$ the critical value indeed turns out to be $p_c \sim n^{-(3-\tau)/(\tau-1)}$, whereas for $\ECM$ or $\GRG$, we find that $p_c \sim n^{-(3-\tau)/2}$.
The largest component sizes in both regimes are of the order $n^\alpha p_c$, and the scaling limits are in a completely different universality class than in the $\tau\in (3,4)$ and $\tau>4$ cases.

\subsection{Discussion}\label{sec:intro-discussion}

\paragraph*{Relation to branching process approximations.} The distinction between the universality classes $\tau>4$ and $\tau\in (3,4)$ can also be seen in terms of the branching process approximation. 
Recall that for $\CM$ or $\mathrm{UM}_n(\bld{d})$, the local neighborhoods can be approximated by a branching process.
The progeny distribution is $D_n^*-1$, where $D_n^*$ is given by \eqref{eq:SB-dist-degree}. 
This distribution has asymptotically finite variance if and only if the third moment of the asymptotic degree distribution is finite (i.e., $\tau > 4$). 
It is known that for critical branching processes the growth rate of the neighborhoods crucially depends on the variance \cite{A17,Kor17}. 
In fact the height scales as $|T|^{1/2}$ and $|T|^{(\tau-3)/(\tau-2)}$ in the $\tau >4$ and $\tau\in (3,4)$ regimes, respectively, where $|T|$ is the total progeny of the branching process.
Now, heuristically speaking, if $\sC_{\sss (1)}(p_c(\lambda))$  was a tree, then from the theory of branching processes, one would expect the following relations to be true:
\begin{eq} \label{eq:scaling-relation-regimes}
\eta = \frac{\rho}{2}, \text{ for }\tau>4, \quad \text{and}\quad \eta = \frac{\tau-3}{\tau-2} \rho, \text{ for } \tau \in (3,4).
\end{eq}
Following \cite{Gri99}, we refer to the identities in \eqref{eq:scaling-relation-regimes} as scaling relations. 
This aligns with the exponents suggested above for the two regimes.
 Intuitively, the above relations should hold since $\sC_{\sss (1)}(p_c(\lambda))$ is a tree, i.e., \linebreak $\SP(\sC_{\sss (1)}(p_c(\lambda))) = 0$, with probability bounded away from zero. 
A disclaimer to the reader is that the bounds in \cite{A17,Kor17} are proved for a fixed branching process rather than a sequence of those. 
For a more rigorous explanation \eqref{eq:scaling-relation-regimes}, an interested reader is referred to \cite{H13}. 
For $\tau\in (2,3)$, the branching process approximation does not work for $\CM$ due to the presence of multiple edges.

\paragraph*{About the $\Unot$-topology.} 
The distributional convergence of $\mathbf{Z}_n(\lambda)$ under the $\Unot$-topology implies convergence of many interesting functionals.  
Let $\mathscr{C}(v)$ denote the connected component containing vertex $v$ in $G_n(p)$, and let $V_n$ denote a vertex chosen uniformly at random from $[n]$, independently of $G_n(p)$.
One example is a quenched version of the susceptibility function defined as $\E[|\sC(V_n)||G_n(p)]$. 
Note that 
\begin{eq}\label{eq:suscep-l2-relation}
\E\big[\ |\sC(V_n)|\ \big\vert\ G_n(p_c(\lambda))\big] = \frac{1}{n} \sum_{v\in [n]} |\sC(v)| = \frac{1}{n} \sum_{i\geq 1} |\sC_{\sss (i)}(p_c(\lambda))|^2,
\end{eq}and therefore the convergence in $\Unot$-topology implies the convergence of $n^{-1/3}\E[|\sC(V_n)|\vert G_n(p_c(\lambda))]$. 
We also get the convergence of the quantity $n^{-2/3}\sum_{i\geq 1} |\sC_{\sss (i)}(p_c(\lambda))|\SP(\sC_{\sss (i)}(p_c(\lambda)))$, which in particular implies that the components of small size cannot contain too many surplus edges.
  The relevance of this topology is also discussed in \cite{BBW12} (see also \cite{A17}), because this turns up naturally  in defining the augmented multiplicative coalescent, and establishing a version of the Feller property.

\paragraph*{Scaling relation.}
The following scaling relation is true in both the regimes $\tau>4$ and $\tau\in (3,4)$:
\begin{equation}\label{eq:scaling-relation}
1+\eta = 2\rho.
\end{equation}
This can be understood intuitively. 
Since the component sizes converge in the $\ell^2_{\shortarrow}$-topology, one can expect that the expected value of $\sum_{i\geq 1}|\sC_{\sss (i)}(p_c(\lambda))|^2$ is of the order $n^{2\rho}$. 
One may also use \eqref{eq:suscep-l2-relation} to calculate this expectation. 
In fact, \eqref{eq:suscep-l2-relation} implies the scaling relation \eqref{eq:scaling-relation} if $\E[|\sC(V_n)|] = O(n^{\eta})$.
Now $|\sC(V_n)| \leq \sum_{l\geq 1}P_l$, where $P_l$ is the expected number of paths of length $l$ starting from vertex $v$.
Using the branching process approximation for $\CM$, $\E[P_l] \leq C(p_c \nu_n)^{l}$, for some constant $C>0$. 
Summing this estimate over $l$, one gets 
\begin{eq}
\E[\sC(V_n)] \leq C\sum_{l\geq 1} (p_c \nu_n)^l \leq C\sum_{l\geq 1}\Big(1+\frac{\lambda}{n^{\eta}}\Big)^l \leq Cn^{\eta},
\end{eq}for $\lambda<0$, which yields a heuristic derivation of \eqref{eq:scaling-relation}.
The above path counting technique has been formalized in \cite{J09b,janson2009susceptibility,BDHS17}.

%
%
%
%

\paragraph*{Effect of slowly-varying corrections.}
Suppose that the asymptotic degree distribution satisfies $\PR(D\geq x) \sim L(x)x^{-(\tau-1)}$ with $L(\cdot)$ some slowly-varying function.
For $\tau>4$, the scaling limits, as well as the exponents, are insensitive to $L(\cdot)$. 
On the contrary, the component size, or even the critical window, depends crucially on the slowly-varying function for $\tau\in (3,4)$.  
The critical window becomes $p = p_c(1+\lambda L^*(n)^2 n^{-\eta})$, and the component sizes turn out to be of the same order as $n^{\rho}/L^*(n)$, for some slowly varying $L^*$. 
However, the scaling limits lie in the same universality class; see Chapter~\ref{chap:secondmoment}.



\section{Component sizes and surplus edges}
\label{sec:intro-Q1} 
In this section, we provide an outline of the proofs for establishing scaling limits of $\mathbf{Z}_n(\lambda)$ for the random graph $\mathrm{CM}_n(\bld{d},p_c(\lambda))$. 
We provide the key ideas, and the strategy of the proof, 
leaving many details for the later chapters.
In Section~\ref{sec:Janson-construction-intro}, we start by describing a construction of $\mathrm{CM}_n(\bld{d},p)$ due to Janson~\cite{J09}, which is a key tool throughout this thesis.
This construction allows us to treat $\mathrm{CM}_n(\bld{d},p_c(\lambda))$ as a configuration model with a suitable degree distribution, which can be easier to work with due to the sequential construction provided in Section~\ref{sec:intro-models}.
In Section~\ref{sec:finite-third-moment}, we consider the scaling limit for $\mathbf{Z}_n(\lambda)$ for the finite third-moment case, and outline a detailed proof strategy. 
The infinite third-moment case is considered in Section~\ref{sec:infinite-third-moment-intro}.
\subsection{Janson's construction} \label{sec:Janson-construction-intro}
 Suppose that $\boldsymbol{d}'$ is the random degree sequence obtained after percolation. 
Fountoulakis~\cite{F07} showed that, conditionally on $\boldsymbol{d}'$, the law of $\mathrm{CM}_{n}(\boldsymbol{d},p)$  is same as the law of $\mathrm{CM}_{n}(\boldsymbol{d}')$.  
Often asymptotics of different functionals of $\bld{d}'$ can be calculated, which gives a powerful tool to deal with percolation on random graphs with general degree sequence. 
The following explicit construction of $\mathrm{CM}_{n}(\boldsymbol{d},p)$ is due to Janson~\cite{J09}, provided in the context of identifying the percolation phase transition on $\CM$.
This construction will be crucial in what follows.
\begin{algo} \label{algo:janson-construction-intro}
\normalfont \begin{itemize}
 \item[(S1)] For each half-edge $e$, let $v_e$ be the vertex to which $e$ is attached. With probability $1-\sqrt{p}$, one detaches $e$ from $v_e$ and associates $e$ to a new vertex $v'$ that we color red. 
 This is done independently for every half-edge. 
 Let $n_+$ be the number of red vertices created and $\tilde{n}=n+n_+$.  Suppose that $\Mtilde{\boldsymbol{d}} = ( \tilde{d}_i )_{i \in [\tilde{n}]}$ is the new degree sequence obtained by the above procedure, i.e., $\tilde{d}_i \sim \text{Bin} (d_i, \sqrt{p})$ for $i \in [n]$ and $\tilde{d}_i=1$ for $i \in [\tilde{n}] \setminus [n]$.
 \item[(S2)] Construct $\mathrm{CM}_{\tilde{n}}(\Mtilde{\boldsymbol{d}})$, independently of (S1).
 \item[(S3)] Delete all the red vertices. Alternatively, one can choose any $n_+$ degree-one vertices uniformly at random without replacement, independently of (S1) and (S2), and delete them.
 \end{itemize}
 \end{algo}
An edge is kept by Algorithm~\ref{algo:janson-construction-intro} if both its endpoints are not red, which happens with probability $p$. 
Also, conditionally on the choice of non-red half-edges, the pairing between these half-edges is a uniform perfect matching.
Algorithm~\ref{algo:janson-construction-intro} indeed produces $\mathrm{CM}_n(\bld{d},p)$ using Fountoulakis' result \cite{F07} mentioned above.
Due to the uniform matching, it does not matter whether we delete the red vertices, or $n_+$ degree-one vertices chosen uniformly at random. 
We end up with a sample from the same random graph distribution.

In what follows, directly setting up a technically tractable framework with  exploration processes on $\mathrm{CM}_n(\bld{d},p_c(\lambda))$ turns out to be difficult even for simple $d$-regular graphs \cite{NP10b}. 
On the other hand, due to the sequential construction, the configuration model is often easier to handle. 
The above construction allows us to study $\mathrm{CM}_n(\bld{d},p_c(\lambda))$ via a suitable configuration model.


\subsection{Finite third-moment case} \label{sec:finite-third-moment}
This section is based on \cite{DHLS15}, where the asymptotics of $\mathbf{Z}_n(\lambda)$, under the finite third-moment assumption, has been treated.  
To ensure that $\CM$ has a giant component (otherwise there will be no phase transition for the percolation process), we must assume that \eqref{eq:conv-nu-local-weak} holds. 
In this case, $\eta =1/3$, and $p_c = 1/\nu_n$, so that
\begin{equation}
p_c(\lambda) = \frac{1}{\nu_n}(1+\lambda n^{-1/3}), \quad -\infty<\lambda<\infty.
\end{equation}
Firstly, let us state the assumptions on the degree distribution, which includes the empirical degree distribution to obey a power law with exponent  $\tau>4$ as a special case. 
\begin{assumption} \label{assumption-finite-third-intro}  \normalfont
For each $n\geq 1$, let $\bld{d}=\bld{d}_n=(d_i)_{i\in [n]}$ be a degree sequence such that $\ell_n = \sum_{i\in [n]}d_i$ is even. 
We assume the following about $(\boldsymbol{d}_n)_{n\geq 1}$ as $n\to\infty$:
 Let $D_n$ denote the degree of a vertex chosen uniformly at random independently of the graph. Then,
 \begin{enumerate}[(i)]
  \item (\emph{Weak convergence of $D_n$})  $ D_n \dto D$, for some random variable $D$ such that $\mathbb{E}[D^3] < \infty $.
 \item \label{assumption1-2}(\emph{Uniform integrability of $D_n^3$})
  $\E[D_n^3]= \frac{1}{n}\sum_{i\in [n]}d_{i}^{3} \to \E[D^{3}].$
 \end{enumerate}
 \end{assumption}
\noindent We will use Algorithm~\ref{algo:janson-construction-intro} to reduce the analysis of $\mathrm{CM}_n(\bld{d},p_c(\lambda))$ to $\mathrm{CM}_{\tilde{n}}(\Mtilde{\bld{d}})$.
In fact, the following holds for $\Mtilde{\bld{d}}$ for $p=p_c(\lambda)$: Let  $\mathbb{P}_p^n$ denote the  probability measure induced on $\mathbb{N}^{\infty}$ by Algorithm~\ref{algo:janson-construction-intro}~(S1). 
Denote the product measure of $(\PR_p^n)_{n\geq 1}$ by $\mathbb{P}_p$.
\begin{lemma}\label{prop_percolation_condition} The statements below are true  $\mathbb{P}_p$ almost surely:
 \textrm{Assumption~\ref{assumption-finite-third-intro}} is satisfied by $\Mtilde{\bld{d}}$ and
  $$ \frac{\sum_{i\in [\tilde{n}]}\tilde{d}_i(\tilde{d}_i-1)}{\sum_{i\in [\tilde{n}]}\tilde{d}_i} = 1+\lambda n^{-1/3}+o(n^{-1/3}).$$
\end{lemma}
\noindent The proof involves computing functionals of binomial distributions and their concentration, see Section~\ref{c1:sec_percolation} for a proof. 
Further, while performing (S3), the number of surplus edges within each component does not change, while the component size changes by the amount of deleted degree-one vertices. 
The latter can be estimated from the number of vertices of degree-one in each of the connected components.
Thus without loss of generality, our study reduces to finding the scaling limit of $\mathbf{Z}_n(\lambda)$ on $\CM$ satisfying
\begin{equation}\label{eq:thm-intro-nu-n}
   \nu_{n}:= \frac{\sum_{i\in [n]}d_i(d_i-1)}{\sum_{i\in [n]}d_i} =1+\lambda n^{-1/3}+o(n^{-1/3}), \quad  \text{for some }\lambda\in \mathbb{R}.
  \end{equation}
Another technical assumption that we make is that $\PR(D=1) > 0$, which is required for the phase transition result in \cite{MR95,JL09}, as well as in some technical parts of our proof. 
Formally, we aim to prove the following theorem: 
\begin{theorem} \label{thm_surplus-intro}
  Consider $\CM$ satisfying \textrm{Assumption~\ref{assumption-finite-third-intro}}, and \eqref{eq:thm-intro-nu-n} for some $\lambda\in \mathbb{R}$. As $n\to\infty$,
  \begin{equation} \label{eqn_thm_surplus}
  \mathbf{Z}_n(\lambda) \dto \mathbf{Z}(\lambda)
  \end{equation} with respect to the $\mathbb{U}^0_{\shortarrow}$ topology.
  Here $\mathbf{Z}(\lambda)$ is some non-degenerate random vector which will be defined in the proof ideas and more formally in Theorem~\ref{c1:thm_surplus}.
 \end{theorem}
\noindent 
In the subsequent sections, we describe the proof idea for Theorem~\ref{thm_surplus-intro}.

\subsubsection{The exploration process}\label{sec:exploration-process-intro}
The central idea to prove scaling limits of critical component sizes was introduced by Aldous~\cite{A97} in the context of the Erd\H{o}s-R\'enyi random graph.
The idea is to explore the graph sequentially and encode the relevant information in terms of a walk called the \emph{exploration process}. 
Then the idea is to establish scaling limits of the exploration process and then try to read off, if possible, the relevant property from the limit of the exploration process. 
Let us explore $\CM$ sequentially using depth-first exploration. 
At each step $k$, we find a new vertex with degree $d_{\sss (k)}$.
This vertex may create $c_{\sss (k)}$ edges to the vertices which are already explored. 
Thus $d_{\sss (k)}-c_{\sss (k)}$ half-edges can give new vertices during the exploration. 
Once all the half-edges of a vertex are explored, the vertex is declared dead, meaning that the complete neighborhood of that vertex has been identified.
The precise description of the exploration algorithm is given in Section~\ref{c1:sec_proof}. 
Based on this exploration algorithm, define the exploration process by
 \begin{equation}\label{walk1-intro}
 S_{n}(0)=0, \quad S_{n}(i)=\sum_{j=1}^{i}(d_{\sss(j)} -2-2c_{\sss(j)}).
 \end{equation}
 The minus two is due to the fact that an edge (i.e. two half-edges) is explored at each step.
 The process $\mathbf{S}_{n}= ( S_{n}(i))_{i \in [n]}$ ``encodes the component sizes as lengths of path segments above past minima'' as discussed in \cite{A97}. 
 Suppose $\mathscr{C}_{i}$ is the $i$-th connected component explored by the above exploration process. Define
\begin{equation}
\tau_{k}=\inf \big\{ i:S_{n}(i)=-2k \big\}.
\end{equation}
Then  $\mathscr{C}_{k}$ is discovered between the times $\tau_{k-1}+1$ and $\tau_k$ and  $|\mathscr{C}_{k}|=\tau_{k}-\tau_{k-1}$.
\begin{figure}
\begin{center}
\begin{tikzpicture}[scale=.6]
\draw (1.5,3.)-- (1.5,-3.);
\draw (1.5,-3.)-- (19.5,-3.);
\draw (19.5,-3.)-- (19.5,3.);
\draw (19.5,3.)-- (1.5,3.);
\draw (1.5,0.)-- (19.5,0.);
\draw [->] (3.83275+5,-3.7) -- (6.96864+5,-3.7);
\draw (10.4,-3.35) node { Time};
\draw [->] (1.2,-1.2) -- (1.2,1.3);
\draw (0.8,1.78) node {$S_n(t)$};
\begin{scriptsize}
\draw (1.5,0)--(2,.7);
\draw [fill=black] (2,.7) circle (1pt);
\draw (2,.7)--(2.5,1.4);
\draw [fill=black] (2.5,1.4) circle (1pt);
\draw (2.5,1.4)--(3,2.8);
\draw [fill=black] (3,2.8) circle (1pt);
\draw (3,2.8)--(3.5,2.1);
\draw [fill=black] (3.5,2.1) circle (1pt);
\draw (3.5,2.1)--(4,1.4);
\draw [fill=black] (4,1.4) circle (1pt);
\draw (4,1.4)--(4.5,2.8);
\draw [fill=black] (4.5,2.8) circle (1pt);
\draw (4.5,2.8)--(5,2.1);
\draw [fill=black] (5,2.1) circle (1pt);
\draw (5,2.1)--(5.5,1.4);
\draw [fill=black] (5.5,1.4) circle (1pt);
\draw (5.5,1.4)--(6,1.4);
\draw [fill=black] (6,1.4) circle (1pt);
\draw (6,1.4)--(6.5,.7);
\draw [fill=black] (6.5,.7) circle (1pt);
\draw (6.5,.7)--(7,0);
\draw [fill=black] (7,0) circle (1pt);
\draw (7,0)--(7.5,-.7);
\draw [fill=black] (7.5,-.7) circle (1pt);
\draw (7.5,-.7)--(8,-1.4);
\draw [fill=black] (8,-1.4) circle (1pt);
\draw (8,-1.4)--(8.5,0);
\draw [fill=black] (8.5,0) circle (1pt);
\draw (8.5,0)--(9,-.7);
\draw [fill=black] (9,-.7) circle (1pt);
\draw (9,-.7)--(9.5,0);
\draw [fill=black] (9.5,0) circle (1pt);
\draw (9.5,0)--(10,.7);
\draw [fill=black] (10,.7) circle (1pt);
\draw (10,.7)--(10.5, .7);
\draw [fill=black] (10.5,.7) circle (1pt);
\draw (10.5,.7)--(11, 2.1);
\draw [fill=black] (11,2.1) circle (1pt);
\draw (11, 2.1)--(11.5, 1.4);
\draw [fill=black] (11.5,1.4) circle (1pt);
\draw (11.5,1.4)-- (12,.7);
\draw [fill=black] (12,.7) circle (1pt);
\draw (12,.7)-- (12.5,0);
\draw [fill=black] (12.5,0) circle (1pt);
\draw (12.5,0)-- (13,-.7);
\draw [fill=black] (13,-.7) circle (1pt);
\draw (13,-.7)-- (13.5,-1.4);
\draw [fill=black] (13.5,-1.4) circle (1pt);
\draw (13.5,-1.4) -- (14,-1.4);
\draw [fill=black] (14,-1.4) circle (1pt);
\draw (14.5,-2.1) -- (14,-1.4);
\draw [fill=black] (14.5,-2.1)  circle (1pt);
\draw (14.5,-2.1) -- (15,-2.8);
\draw [fill=black] (15,-2.8)  circle (1pt);
\draw (15.5,-2.1) -- (15,-2.8);
\draw [fill=black] (15.5,-2.1)  circle (1pt);
\draw (15.5,-2.1) -- (16,-.7);
\draw [fill=black] (16,-.7)  circle (1pt);
\draw (16,-.7)--(16.5,-.7);
\draw [fill=black] (16.5,-.7)  circle (1pt);
\draw (16.5,-.7)--(17,-1.4);
\draw [fill=black] (17,-1.4)  circle (1pt);
\draw (17,-1.4)--(17.5,-2.1);
\draw [fill=black] (17.5,-2.1)  circle (1pt);
\draw (17.5,-2.1)--(18,-.7);
\draw [fill=black] (18,-.7)  circle (1pt);
\draw (18.5,0)--(18,-.7);
\draw [fill=black] (18.5,0)  circle (1pt);

\draw (5,-1.4) node (a) {$|\mathscr{C}_1|$};
\draw (8,-2.8) node (b) {$|\mathscr{C}_2|$};
\draw[dotted,thick] [<-] (1.5,-1.4)--(4.2,-1.4);
\draw[dotted,thick] [->] (5.8,-1.4)--(8,-1.4);
\draw[dotted,thick] [<-] (1.5,-2.8)--(7.2,-2.8);
\draw[dotted,thick] [->] (8.8,-2.8)--(15,-2.8);
\draw (7.7,-.5) node {\scriptsize $-2$};
\draw (14.7,-1.5) node {\scriptsize $-4$};
\draw[dotted,thick] (8,0)--(8,-1.4);
\draw[dotted,thick] (15,0)--(15,-2.8);
\end{scriptsize}
\end{tikzpicture}
\end{center}
\caption{Component sizes as excursions of the exploration process}
\label{fig:exploration-comp}
\end{figure}
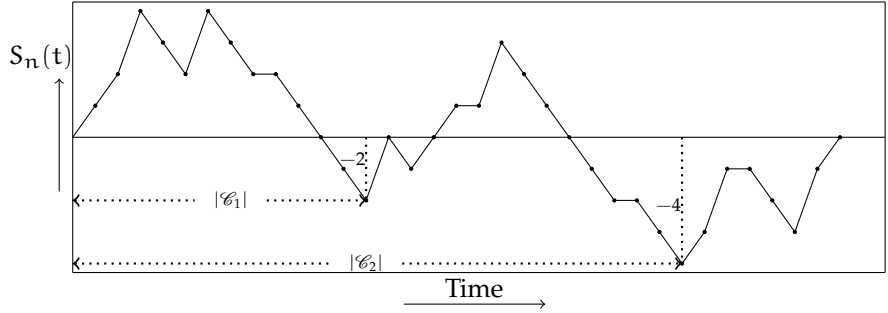
Therefore, the excursion lengths of the exploration process correspond to the sizes of the explored components, see Figure~\ref{fig:exploration-comp}. 
This property allows one to recover the scaling limits of the component sizes from the scaling limit of the exploration process.
\subsubsection{Size-biased exploration} \label{sec:size-biased-intro}
During the above mentioned exploration process, the vertices are explored in a size-biased manner with sizes proportional to their degrees, i.e., if we denote by $v_{\sss(i)}$ the $i$-th explored vertex, then 
 \begin{equation}
 \mathbb{P}\big(v_{\sss (i)}=j \vert v_{\sss (1)},v_{\sss (2)},...,v_{\sss (i-1)}\big)=\frac{d_{j}}{\sum_{k\notin \mathscr{V}_{i-1}} d_{k}},\quad \forall j\in \mathscr{V}_{i-1},
 \end{equation}
 where $\mathscr{V}_{i}$ denotes the first $i$ vertices to be discovered in the above exploration process.
 The following lemma is a consequence of the size-biased ordering, which keeps track of sum of the degrees and square of degrees of the explored vertices.
  This will allow us to track the drift and the quadratic variation of the exploration process \eqref{walk1-intro}:
\begin{lemma} \label{lem_con_1-intro} Suppose that \textrm{Assumption~\ref{assumption-finite-third-intro}} holds and denote $\sigma_r = \mathbb{E} [ D^r ] $. 
Then for all $t > 0 $, as $n \to \infty$,
\begin{equation} \label{lem_eq2-intro}
   \sup_{ u\leq t } \Big| n^{-2/3} \sum_{i=1}^{\lfloor n^{2/3} u \rfloor} d_{\sss (i)} - \frac{\sigma_{2}u}{\sigma_1} \Big| \xrightarrow{\mathbb{P}} 0, \quad \sup_{ u\leq t } \Big| n^{-2/3} \sum_{i=1}^{\lfloor n^{2/3} u \rfloor} d_{\sss (i)}^{2} - \frac{\sigma_{3}u}{\sigma_1} \Big| \xrightarrow{\mathbb{P}} 0.
  \end{equation} 
 \end{lemma}

\subsubsection{Analysis of the exploration process}\label{sec:explor-analysis-intro}
The next step is to obtain the scaling limit of the exploration process. 
Firstly, let us consider the simplified process
\begin{equation}\label{walk2-intro}
 s_{n}(0)=0, \quad s_{n}(i)=\sum_{j=1}^{i}(d_{\sss(j)} -2),
 \end{equation} 
which ignores the effect of cycles in the exploration process. 
Due to the close relation to the size-biased exploration, we can more easily describe the scaling limit of $s_n(t)$: 
\begin{proposition} \label{thm_main2-intro} Let $\bar{\mathbf{s}}_n = (s_n(t))_{t\geq 0}$ be given by $\bar{s}_n(t) = n^{-1/3} s_n(\lfloor tn^{2/3}\rfloor)$.  
 Under \textrm{Assumption~\ref{assumption-finite-third-intro}}, as $n \to \infty$,
  \begin{equation}\label{eq:thm_main2-intro}
   \bar{\mathbf{s}}_{n} \xrightarrow{\mathcal{L}} \mathbf{B} 
  \end{equation}
  with respect to  the Skorohod $J_{1}$ topology, where $B(t) = c_0W(t)+\lambda t -c_1 t^2$ with $W$ a standard Brownian motion.
\end{proposition}  
The contribution $\sum_{j=1}^ic_{\sss (j)}$ counts the number of surplus edges created upto time $i$, and one can expect it to be negligible from  earlier heuristics about small number of surplus edges. 
Thus Proposition~\ref{thm_main2-intro} also provides the scaling limit of $\mathbf{S}_n$, after corresponding rescaling.
The time scaling $n^{2/3}$ is due to our prior prediction that the component sizes are of this order.
The space scaling $n^{1/3}$ is the usual square root fluctuation with respect to the time component that arises for Brownian scaling limits.

Let us now give some details as to how Proposition~\ref{thm_main2-intro} can be proved.
Let $\mathscr{F}_j$ denote the natural sigma algebra which contains all the information about the explored graph up to time $j$. The Doob-Meyer decomposition yields that 
\begin{equation}
s_n(i) = M_n(i) +A_n(i),
\end{equation}where 
  \begin{equation}\label{drift-QV-intro}
   A_{n}(i)= \sum_{j=1}^{i} \mathbb{E}\big[d_{\sss(j)}-2 \vert \mathscr{F}_{j-1} \big], \quad 
   B_{n}(i)= \sum_{j=1}^{i} \var{d_{\sss (j)}\vert \mathscr{F}_{j-1}},
  \end{equation}
  $(B_n(i))_{i\geq 1}$ being the quadratic variation process of $(M_n(i))_{i\geq 1}$.
We consider the convergence of the drift part $A_n$ and the martingale part $M_n$ separately. 

\paragraph*{Convergence of the drift part.}
The negative quadratic drift in the limit of $\bar{\mathbf{s}}_n$ is a consequence of the size-biased reordering stated in Lemma~\ref{lem_con_1-intro}. 
Note that
\begin{align*}
 &\mathbb{E} \big[ d_{\sss(i)} -2 \vert \mathscr{F}_{i-1} \big]  = \frac{\sum_{j \notin \mathscr{V}_{i-1}} d_{j}(d_{j}-2)}{\sum_{j \notin \mathscr{V}_{i-1}} d_{j}}\\
 &= \frac{\sum_{j \in [n]} d_{j}(d_{j}-2)}{\sum_{j \in [n]} d_{j}}- \frac{\sum_{j \in \mathscr{V}_{i-1}} d_{j}(d_{j}-2)}{\sum_{j \in [n]} d_{j}} + \frac{\sum_{j \notin \mathscr{V}_{i-1}} d_{j}(d_{j}-2)\sum_{j \in \mathscr{V}_{i-1}} d_{j} }{\sum_{j \notin \mathscr{V}_{i-1}} d_{j}\sum_{j \in [n]} d_{j}} \\
 &= \frac{\lambda}{n^{1/3}} - \frac{\sum_{j \in \mathscr{V}_{i-1}} d_{j}^{2}}{\sum_{j \in [n]} d_{j}} + \frac{\sum_{j \notin \mathscr{V}_{i-1}} d_{j}^{2}\sum_{j \in \mathscr{V}_{i-1}} d_{j} }{\sum_{j \notin \mathscr{V}_{i-1}} d_{j}\sum_{j \in [n]} d_{j}} +o(n^{-1/3}).
\end{align*}
Therefore,
\begin{equation} \label{eq_cond_1-intro}
\begin{split}
&A_{n}(k) = \sum_{i=1}^{k}\mathbb{E} \big[ d_{\sss(i)} -2 \vert \mathscr{F}_{i-1} \big] \\
&= \frac{k\lambda}{n^{1/3}} - \sum_{i=1}^{k} \frac{\sum_{j \in \mathscr{V}_{i-1}} d_{j}^{2}}{\sum_{j \in [n]} d_{j}} + \sum_{i=1}^{k}\frac{\sum_{j \notin \mathscr{V}_{i-1}} d_{j}^{2}\sum_{j \in \mathscr{V}_{i-1}} d_{j} }{\sum_{j \notin \mathscr{V}_{i-1}} d_{j}\sum_{j \in [n]} d_{j}} + o(kn^{-1/3}).
\end{split}
\end{equation}
Now, $\sum_{j\in \mathscr{V}_{i-1}}d_j = o(\ell_n)$ uniformly over $i\leq tn^{2/3}$, since $\max_{i\in [n]}d_i = o(n^{1/3})$, by the uniform integrability of the third moment in Assumption~\ref{assumption-finite-third-intro}. 
Therefore $\sum_{j\notin \mathscr{V}_{i-1}}d_j \approx \ell_n$, and a similar argument yields that $\sum_{j\notin \mathscr{V}_{i-1}}d_j^2 \approx \sum_{j\in [n]}d_j^2$. 
Combining this with \eqref{lem_eq2-intro}, it follows that  
\begin{eq}\label{eq:drift-approximation-intro}
n^{-1/3} A_n(\lfloor u n^{2/3}\rfloor) \approx \lambda u - c_1  u^2,
\end{eq}where $c_1 = (\sigma_3\sigma_1-\sigma_2^2)/\sigma_1^3$. 
Notice that $c_1 = \var{D^*}/\E[D] > 0$, where $D^*$ is the size-biased version of the random variable $D$ appearing in Assumption~\ref{assumption-finite-third-intro}.
Thus, the drift term is negative and parabolic. 
In the above calculations, we see that the negative drift term arises from the depletion of degrees in the size-biased exploration.
As more vertices are explored, $\E[d_{\sss (i)}\vert \mathscr{F}_{i-1}]$ 
decreases by an amount proportional to $\sum_{j\in \mathscr{V}_{i-1}} d_j^2$.
Due to the finite third moment condition, Lemma~\ref{lem_con_1-intro} ensures that $\sum_{j\in \mathscr{V}_{i-1}} d_j^2$ increases linearly with time.
Thus the negative part in the drift term, which is just the sum of $\sum_{j\in \mathscr{V}_{i-1}} d_j^2$, is quadratic.
In this sense, the negative quadratic drift is related to the effect of depletion of degrees in sampling from the size-biased distribution without replacement.

\paragraph*{Convergence of the martingale part.}
The proof relies on the celebrated Martingale Functional Central Limit theorem (FCLT).  
The Martingale FCLT ensures convergence of martingales to Brownian motion provided that the limiting process has continuous sample paths, and the quadratic variation converges to a constant multiple of $t$. 
The latter condition arises due to L\'evy's characterization of Brownian motion as the unique process with quadratic variation $t$.
In this case, it is enough to show that 
 \begin{equation} \label{condition2}
  n^{-2/3}B_{n}(\lfloor un^{2/3}\rfloor )\xrightarrow{\mathbb{P}} c_0^{1/2} u.
 \end{equation} 
Again this can be deduced using Lemma~\ref{lem_con_1-intro}.
The increments of $B_{n}(\lfloor un^{2/3}\rfloor )$ in \eqref{drift-QV-intro} are given by the asymptotic finite variance of the size-biased distribution, which is equivalent to the finite third moment of the degree distribution. 
Thus, the finite third moment is essential from the point of view of the functional invariance principle.
The technical conditions for ensuring that the limiting process has continuous sample paths are explicitly stated in Section~\ref{c1:sec_proof}.

\subsubsection{Large components are explored early} \label{sec:large-com-early-intro}
To learn about the largest component sizes from Proposition~\ref{thm_main2-intro}, one first needs to check that the ordered vector of excursion lengths is a continuous function on a set $\mathcal{G} \subset \mathbb{D}(\R_+,\R)$, and the limiting process in \eqref{eq:thm_main2-intro} lies in $\mathcal{G}$ almost surely. 
This part of the argument follows using properties of 
 Brownian motion with a negative parabolic drift, see \cite{NP10a,A97,AL98}.
In order to ensure that the largest excursions of $\mathbf{B}$ in \eqref{eq:thm_main2-intro} correspond to the largest components in the critical random graph, it must be ensured that the largest components are explored in $O(n^{2/3})$ time. 
This is because, due to the time scaling by $n^{2/3}$, we loose information about the components explored in $\Omega (n^{2/3})$ time.
The following lemma ensures that no large component is explored after time $\Omega(n^{2/3})$:
 \begin{lemma}\label{lem:large-com-explored-early-intro}Let $\mathscr{C}_{\max}^{\sss \geq T}$ denote the largest component which is started to be explored after time $Tn^{2/3}$. Then, for any $\delta >0$,
 \begin{equation}\label{large-com-explored-early-intro}
  \lim_{T\to\infty}\limsup_{n\to\infty}\prob{|\mathscr{C}_{\max}^{\sss \geq T}|>\delta n^{2/3}}=0.
 \end{equation}
 \end{lemma}
The idea for the proof of Lemma~\ref{lem:large-com-explored-early-intro} is that due to the sequential matching of the half-edges after exploring the graph upto time $Tn^{2/3}$, the rest of the graph is again a configuration model. 
Moreover, the $\nu_n$ parameter for this new configuration model becomes 
 \begin{equation}\label{nu-n-i:nu-n:relation-intro}
 \nu_{n,Tn^{2/3}}= \nu_n-C_0 Tn^{-1/3}+o_{\sss\mathbb{P}}(n^{-1/3}).
\end{equation}
Thus, as we keep on exploring the graph, the rest of the graph becomes a configuration model that is more and more subcritical. 
Now the fact that the component sizes of a barely subcritical configuration model are $o(n^{2/3})$ can be leveraged.
However, the iterated limit in \eqref{large-com-explored-early-intro} requires explicit bounds on the required functionals of a ``slightly subcritical'' configuration model. For a formal deduction, see Lemma~\ref{c1:lem:large-com-explored-early}.

\subsubsection{Component sizes and surplus edges in the product topology} \label{sec:conv-prod-top-intro}
Let us now investigate how the exploration process can yield convergence of the surplus edges. 
At step $k+1$, we have discovered vertex $v_{\sss (k+1)}$ with degree $d_{\sss (k+1)}$, and since one half-edge has been used to discover $v_{\sss (k+1)}$, $d_{\sss (k+1)}-1$ half-edges can create surplus edges. 
There are $A_k = S_n(k) - \min_{j\leq k} S_n(j)$ many half-edges associated to the vertices that are discovered, but not yet explored completely.
Due to the uniform matching, $c_{\sss (k+1)}$, defined in \eqref{walk1-intro} satisfies 
\begin{eq}
\E[c_{\sss (k+1)}\vert \mathscr{F}_k,v_{\sss (k+1)}] \approx \frac{(d_{\sss (k+1)} -1)A_k}{\ell_n}.
\end{eq} 
Now, 
\begin{eq}
\mathbb{E}[ d_{\sss(k+1)} -1 \vert \mathscr{F}_k ] = \frac{\sum_{j \notin \mathscr{V}_k} d_j (d_j-1)}{\sum_{j \notin \mathscr{V}_k} d_j}  \approx \frac{\sum_{j \in [n]}d_j(d_j-1)}{\sum_{j \in [n]}d_j}+o_{\sss\mathbb{P}}(1) \approx 1,
\end{eq}so that 
\begin{eq}\label{eq:intro-rate-func-poisson}
\sum_{k\leq tn^{2/3}}\E[c_{\sss (k+1)} \vert \mathscr{F}_k] \approx \sum_{k\leq tn^{2/3}}\frac{A_k}{\sigma_1 n} \approx \frac{1}{\sigma_1}\int_0^t\bar{A}_n(u)\dif u,
\end{eq}where $\bar{A}_n(u) = n^{-1/3} A_n(\lfloor un^{2/3} \rfloor)$. 
Here $(\bar{A}_n(u))_{u\geq 0}$ converges to $(R(u))_{u\geq 0}$, where $R(u) = B(u) - \inf_{s\leq u} B(s)$ is the reflected version of the limit in~\eqref{eq:thm_main2-intro}.
This proves the following lemma:
\begin{lemma} \label{lem:surp:poisson-conv-intro} 
Let $N_n(k)$ be the number of surplus edges discovered up to time $k$ and $\bar{N}_n(u) = N_n(\lfloor un^{2/3} \rfloor)$. Then, as $n\to\infty$, $\bar{\mathbf{N}}_n\xrightarrow{\sss d} \mathbf{N},$ where $\mathbf{N}$ is the unique counting process such that the following is a martingale:
\begin{eq}
N(t) - \frac{1}{\sigma_1} \int_0^t R(u) \dif u.
\end{eq}
 \end{lemma}
Thus, the number of points $N_i$ in the $i$-th largest excursion $\gamma_i$ of $\mathbf{B}$ is distributed as a mixed Poisson random variable with parameter $(\sigma_1)^{-1} \int_{\gamma} R(u) \dif u$ (see Chapter~\ref{chap:thirdmoment} for a formal definition of excursions). 
At this moment, Proposition~\ref{thm_main2-intro} and Lemma~\ref{lem:surp:poisson-conv-intro} yield the convergence of the component sizes and surplus edges that are explored before time $O(n^{2/3})$. 
The scaling limits of the component sizes are the largest excursions of $\mathbf{B}$ in Proposition~\ref{thm_main2-intro}, and those of the surplus edges are given by mixed Poisson random variables with parameters being proportional to the areas under those excursions as given by Lemma~\ref{lem:surp:poisson-conv-intro}.
On the other hand, Lemma~\ref{lem:large-com-explored-early-intro} ensures that the largest components are explored in time $O(n^{2/3})$ during the exploration process.
This implies the finite-dimensional convergence of $\mathbf{Z}_n(\lambda)$:
\begin{theorem} \label{thm_surplus-intro-prod}
  Consider $\CM$ satisfying \textrm{Assumption~\ref{assumption-finite-third-intro}} and \eqref{eq:thm-intro-nu-n} for some $\lambda\in \mathbb{R}$. As $n\to\infty$,
  \begin{equation} \label{eqn_thm_surplus-intro-prod}
  \mathbf{Z}_n(\lambda) \dto \mathbf{Z}(\lambda)
  \end{equation} with respect to the product topology, where $\mathbf{Z}(\lambda)$ is the vector $(\gamma_i,N_i)_{i\geq 1}$ ordered as an element of $\Unot$.
 \end{theorem}

\subsubsection{{Convergence in the $\Unot$ topology}} \label{sec:Unot-conv-intro}
In order to complete the proof of Theorem~\ref{thm_surplus-intro}, it is now sufficient to show that $(\mathbf{Z}_n(\lambda))_{n\geq 1}$ is tight in $\Unot$, owing to the convergence in product topology in Theorem~\ref{thm_surplus-intro-prod}. 
The tightness is more technical, and the details will be provided in Section~\ref{c1:sec-key-ingredients}. 
However, let us state here the conditions that we need to verify in order to complete the proof. 
Let $\sC_i$ denote the $i$-th explored component, and $Y_i^n=n^{-2/3} |\mathscr{C}_i|$, $N_i^n=\SP(\sC_i)$. 
It is sufficient for the tightness of probability measures on $\Unot$ to prove that for any $\delta>0$
\begin{eq} \label{eqn_sufficient_for_U_0_convergence-intro}
\lim_{\varepsilon\to 0}\limsup_{n\to\infty}\PR\bigg( \sum_{Y_i^n\leq \varepsilon} (Y_i^n)^2> \delta \bigg)=0,\\ \lim_{\varepsilon\to 0}\limsup_{n\to\infty}\PR\bigg( \sum_{Y_i^n\leq \varepsilon} Y_i^n N_i^n> \delta \bigg)=0.
 \end{eq}  

\subsubsection{Degree distribution within components} \label{sec:degree-dist-comp-intro}
Define $v_k(G)$ as the number of vertices of degree $k$ in the connected graph $G$.
Then, 
  \begin{equation}\label{eqn_vertices_of_degree_k-ord-intro}
   v_k \big( \mathscr{C}_{\sss(j)} \big) = \frac{kr_k}{\mathbb{E}[D]} \big| \mathscr{C}_{\sss (j)} \big| +O_{\sss\mathbb{P}}\big((k^{-1}n^{1/3})\big).
  \end{equation}
Again this can be deduced from the size-biased exploration process. 
If $N_k(t)$ denotes the number of vertices of degree $k$ discovered up to time $t$, then for any $t>0$, uniformly over $k$,
  \begin{equation}
   \sup\limits_{u \leq t} \big| n^{-2/3} N_k(un^{2/3})-\frac{kn_k}{\ell_n} u \big| =  O_{\sss\mathbb{P}}((kn^{1/3})^{-1}).
  \end{equation}
This is due to the fact that, at each step during the exploration, we discover a vertex of degree $k$ with probability roughly $k n_k/\ell_n$.
Obviously, there will be depletion in the total number of half-edges and the total number of half-edges attached to vertices of degree $k$, but that depletion does not matter in the $n^{2/3}$ scale.  
Now an application of Lemma~\ref{lem:large-com-explored-early-intro} yields \eqref{eqn_vertices_of_degree_k-ord-intro}.  

The above analysis provides a detailed picture of the size and complexity of the critical components for percolation on $\CM$. 
Whenever the degree distribution satisfies an asymptotic finite third-moment condition, the scaling limit lies in the same universality class as for the Erd\H{o}s-R\'enyi random graph identified in \cite{A97}.
For Erd\H{o}s-R\'enyi random graphs the negative drift term takes a simpler form as the size-biased version of a Poisson random variable  again has a Poisson distribution.

\subsection{Infinite third-moment case} \label{sec:infinite-third-moment-intro}
We now continue with the case where $\E[D^3] = \infty$.
Since $\E[D^3]$ appears explicitly in the scaling limit of the exploration process in Section~\ref{sec:finite-third-moment} (see e.g.~\eqref{eq:drift-approximation-intro}), the scaling limit must be different in this case.  
This section is based on the results for $\mathbf{Z}_n(\lambda)$ from~\cite{DHLS16}. 
Throughout this section we will use the notation
\begin{gather*}
 \alpha= 1/(\tau-1),\qquad \rho=(\tau-2)/(\tau-1),\qquad \eta=(\tau-3)/(\tau-1),
\\
a_n= n^{\alpha}L(n),\qquad b_n=n^{\rho}(L(n))^{-1},\qquad c_n=n^{\eta} (L(n))^{-2},
\end{gather*}where $\tau\in (3,4)$ and $L(\cdot)$ is a slowly-varying function.
The results for $\mathbf{Z}_n(\lambda)$ are derived under the following assumptions on the degree sequence: 
\begin{assumption}\label{assumption-infinte-intro}
\normalfont Fix $\tau \in (3,4)$. Let $\boldsymbol{d}=(d_1,\dots,d_n)$ be a degree sequence (ordered in a non-increasing manner) such that the following conditions hold:
\begin{enumerate}[(i)] 
\item (\emph{High-degree vertices}) For any fixed $i\geq 1$, $d_i/a_n\to \theta_i,$
where $\boldsymbol{\theta}=(\theta_1,\theta_2,\dots)\in \ell^3_{\shortarrow}\setminus \ell^2_{\shortarrow}$. 
\item (\emph{Moment assumptions}) Let $D_n$ denote the degree of a vertex chosen uniformly at random from $[n]$, independently of $\mathrm{CM}_n(\boldsymbol{d})$. Then, $D_n\xrightarrow{\sss d} D$, for some integer-valued random variable $D$ and 
\begin{gather*}
 \frac{1}{n}\sum_{i\in [n]}d_i\to \mu:=\expt{D}, \quad \frac{1}{n}\sum_{i\in [n]}d_i^2\to \E[D^2], \\
 \lim_{K\to\infty}\limsup_{n\to\infty}a_n^{-3} \sum_{i=K+1}^{n} d_i^3=0.
\end{gather*}
\item (\emph{Critical window}) For some $\lambda \in \mathbb{R}$,
\begin{equation}\label{critical-window}
 \nu_n(\lambda):=\frac{\sum_{i\in [n]}d_i(d_i-1)}{\sum_{i\in [n]}d_i}=1+\lambda c_n^{-1}+o(c_n^{-1}).
\end{equation}
\item Let $n_1$ be the number of vertices of degree-one. Then $n_1=\Theta(n)$, which is equivalent to assuming that $\prob{D=1}>0$.
\end{enumerate}
\end{assumption} 
\noindent 
Assumption~\ref{assumption-infinte-intro} can be understood intuitively.
As in Section~\ref{sec:finite-third-moment}, we will set up an exploration process, which explores the components of $\CM$ in a size-biased manner. 
In this setting, we will see that the exploration process keeps on exploring vertices of high degree, resulting in jumps in the exploration process. 
Assumption~\ref{assumption-infinte-intro}~(i) is used to control the magnitude of these jumps. 
The scaling $a_n$ has the same order as $\max_{i\in [n]}D_i$, where $D_i$'s are i.i.d.~random variables satisfying $\PR(D_1\geq x) \propto L_0(x)x^{-(\tau-1)}$ for some slowly-varying function $L_0(\cdot)$.
The expectation and variance of the increments of the exploration process are  governed by the moment assumptions in Assumption~\ref{assumption-infinte-intro}. 
Of particular interest is the assumption on the third moment, which basically says that the variance of the increments is dictated by the contributions from the high-degree vertices only. 
The condition in Assumption~\ref{assumption-infinte-intro}~(iii) is the same criticality condition as in \eqref{eq:thm-intro-nu-n}.
The fact that the above set-up covers  $\mathrm{CM}_n(\bld{d},p_c(\lambda))$ can be established using an analogue of Lemma~\ref{prop_percolation_condition} in this setting. 
A key thing to note here is that if the degrees are an iid sample from a distribution $D$ with $\PR(D\geq x) \propto L_0(x)x^{-(\tau-1)}$, for some $\tau\in (3,4)$ and $L_0(\cdot)$ a slowly-varying function, then Assumption~\ref{assumption-infinte-intro} is satisfied; see Section~\ref{c2:important-examples}. 

Recall that $\mathbf{Z}_n(\lambda)$ denotes the vector of rescaled component sizes and surplus edges, ordered as an element of $\Unot$.
In this section, we rescale the component sizes by $b_n$. 
The following theorem describes the scaling limit of $\mathbf{Z}_n(\lambda)$ in the infinite third moment case:
\begin{theorem} \label{thm_surplus-intro-heavy}
  Consider $\CM$ satisfying \textrm{Assumption~\ref{assumption-infinte-intro}}. As $n\to\infty$, \linebreak
$  \mathbf{Z}_n(\lambda) \dto \mathbf{Z}(\lambda)$
with respect to the $\mathbb{U}^0_{\shortarrow}$ topology, where $\mathbf{Z}(\lambda)$ is some non-degenerate random vector which will be defined in the proof ideas and more formally in Theorem~\ref{c2thm:spls}.
 \end{theorem} 
\noindent 
The proof of  Theorem~\ref{thm_surplus-intro-heavy} can be approached by the steps outlined in Section~\ref{sec:finite-third-moment}. 
However, the techniques involved are substantially different, because, for example, the exploration process does not have a finite variance of the increment distribution. 
Below, we outline the analysis of the exploration process, and the necessary modifications to conclude that the largest components are explored in time $O(b_n)$. 
The asymptotics for the surplus edges follow identically, since \eqref{eq:intro-rate-func-poisson} holds here as well, the only difference arises due to different scaling limit of the exploration process.
We also discuss the scaling limit when the underlying graph is $\mathrm{UM}_n(\bld{d})$, i.e., when the configuration model is conditioned to be simple. 
This problem was stated as a conjecture in~\cite{Jo10} when the degrees are an iid sample from a power-law distribution with~$\tau\in (3,4)$. 

\subsubsection{The size-biased exploration process}
For technical tractability, we modify the exploration process.
We sequentially take active half-edges, pair them uniformly with an unpaired half-edge.
If the new half-edge is incident to a new vertex, then we declare all of its half-edges to be active. The paired half-edges are killed. 
If there are no active half-edges in the system, then we choose one unexplored vertex with probability proportional to its degree and declare all its half-edges active.
See Section~\ref{c2sec:conv-expl} for an exact description. 
The only difference with the exploration process in Section~\ref{sec:finite-third-moment} is that only one edge is created per step, and it is not necessary that new vertices are found at each step. 
Let $\mathscr{V}_l$ denote the set of vertices discovered up to time $l$ and $\mathcal{I}_i^n(l):=\ind{i\in\mathscr{V}_l}$.
The exploration process is given by $S_n(0)=0,$ and
\begin{eq}\label{eq:exploration-intro-tau34}
S_n(l) &= \# \text{active half-edges at }l - 2 \times \# \text{components explored upto }l\\
&= \sum_{i\in [n]} d_i \mathcal{I}_i^n(l)-2l.
\end{eq}
Suppose that $\mathscr{C}_{k}$ is the $k^{th}$ connected component explored by the above exploration process and define
$\tau_{k}=\inf \big\{ i:S_{n}(i)=-2k \big\}.$
Then  $\mathscr{C}_{k}$ is discovered between the times $\tau_{k-1}+1$ and $\tau_k$, and  $\tau_{k}-\tau_{k-1}-1$ gives the total number of edges in $\mathscr{C}_k$. 
However, since the surplus edges will be shown to be tight, the number of edges and the component sizes are asymptotically the same, after rescaling by $b_n$.
  Note that we can write
   \begin{equation}
    S_n(l)= \sum_{i\in [n]} d_i \mathcal{I}_i^n(l)-2l=\sum_{i\in [n]} d_i \left( \mathcal{I}_i^n(l)-\frac{d_i}{\ell_n}l\right)+\left( \nu_n(\lambda)-1\right)l.
   \end{equation} 
 Define the re-scaled version $\bar{\mathbf{S}}_n$ of $\mathbf{S}_n$ by $\bar{S}_n(t)= a_n^{-1}S_n(\lfloor b_nt \rfloor)$. Then, by Assumption~\ref{assumption-infinte-intro},
   \begin{equation} \label{eqn::scaled_process-intro}
    \bar{S}_n(t)= a_n^{-1} \sum_{i\in [n]}d_i\left( \mathcal{I}_i^n(tb_n)-\frac{d_i}{\ell_n}tb_n \right)+\lambda t +o(1).
   \end{equation}

\subsubsection{Analysis of the exploration process}
The exploration process given by \eqref{eqn::scaled_process-intro} has the following scaling limit: 
\begin{theorem} \label{thm::convegence::exploration_process-intro}
As $n\to\infty$,
 $ \bar{\mathbf{S}}_n \dto \bar{\mathbf{S}}_\infty$ with respect to the Skorohod $J_1$ topology. 
 The limit $\bar{\mathbf{S}}_\infty:=  (\bar{S}_\infty(t))_{t\geq 0} $  is given by 
 \begin{eq}\label{scaling-limit-heavy-intro}
 \bar{S}_\infty(t) =  \sum_{i=1}^{\infty} \theta_i\left(\mathcal{I}_i(t)- (\theta_i/\mu)t\right)+\lambda t,
 \end{eq}
 where $\mathcal{I}_i(s):=\ind{\xi_i\leq s }$ for $\xi_i\sim \mathrm{Exp}(\theta_i/\mu)$ independently, and $\mathrm{Exp}(r)$ denotes the exponential distribution with rate $r$.
\end{theorem}
\noindent The limit \eqref{scaling-limit-heavy-intro} is a jump-process. The vertices of degree $\Theta(a_n)$ keep getting explored with time $O(b_n)$, and since the space has been rescaled by~$a_n$, these create  macroscopic jumps in the exploration process.
Notice that $\bar{\mathbf{S}}_\infty$ does not have independent increments and therefore it is not a L\'evy process. 
This was termed as \emph{thinned L\'evy process} in \cite{BHL12}, since $\mathcal{I}_i(s)$ can be seen as a thinning of a Poisson process, with all points discarded except the first one.
Due to the absence of independent increments, most standard techniques 
from the stochastic-process limits literature do not work here. 

However from \eqref{scaling-limit-heavy-intro}, it is evident that the scaling limit only depends on the asymptotics of the high-degree vertices given in Assumption~\ref{assumption-infinte-intro}~(i), and the contributions coming from the lower-degree vertices should be asymptotically negligible.
With this in mind, define the truncated sum 
\begin{eq}\label{intro-tail-K-martingale}
M_n^K(l)=a_n^{-1} \sum_{i>K}d_i\Big( \mathcal{I}_i^n(l)-\frac{d_i}{\ell_n}l \Big).
\end{eq}
Recall that we have ordered the degree sequence in Assumption~\ref{assumption-infinte-intro}, so that the sum in \eqref{intro-tail-K-martingale} takes into account all the contributions in the exploration process except for the $K$ largest degrees.
With a proper estimate of the variance and the expectation of $M_n^K(l)$, along with maximal inequalities for supermartingales, it can be shown that (see Section~\ref{c2sec:conv-expl}) for any $\varepsilon >0$ and $T>0$,
\begin{eq}
\lim_{K\to\infty}\limsup_{n\to\infty}\PR\Big(\sup_{l\leq Tb_n}|M_n^K(l)|> \varepsilon \Big)=0.
\end{eq}
This implies that if we truncate the sum in \eqref{eqn::scaled_process-intro} at any fixed $K$, then it suffices to establish the iterated limit as $n\to\infty$ and then $K\to\infty$. 
Finally, to complete the proof of Theorem~\ref{thm::convegence::exploration_process-intro}, it suffices to consider the joint distributional limit of the processes $( \mathcal{I}_i^n(tb_n) )_{i\in[K],t\geq 0}$, since for any fixed $K\geq 1$,
\begin{equation}
a_n^{-1} \sum_{i=1}^Kd_i\left( \mathcal{I}_i^n(tb_n)-\frac{d_i}{\ell_n}tb_n \right) = \sum_{i=1}^K\theta_i\left( \mathcal{I}_i^n(tb_n)-\frac{d_i}{\ell_n}tb_n \right)+o(1).
\end{equation}
The following lemma characterizes the limit of $( \mathcal{I}_i^n(tb_n) )_{i\in[K],t\geq 0}$:
\begin{lemma} \label{lem::convergence_indicators}
   Fix any $K\geq 1$. As $n\to\infty$,
   \begin{equation}
    \left( \mathcal{I}_i^n(tb_n) \right)_{i\in[K],t\geq 0} \dto \left( \mathcal{I}_i(t) \right)_{i\in[K],t\geq 0}.
   \end{equation}
 \end{lemma}
\begin{proof} By noting that $(\mathcal{I}_i^n(tb_n))_{t\geq 0}$ are indicator processes, it is enough to show that 
\begin{eq}
 \prob{\mathcal{I}_i^n(t_ib_n)=0,\ \forall i\in [K]} &\to \prob{\mathcal{I}_i(t_i)=0,\ \forall i\in [K]} = \exp \Big( -\mu^{-1}\sum_{i=1}^{K} \theta_it_i\Big)
\end{eq} for any $t_1,\dots,t_K\in \mathbb{R}$. Now, 
\begin{eq} \label{lem::eqn::expression1}
 &\prob{\mathcal{I}_i^n(m_i)=0,\ \forall i\in [K]}=\prod_{l=1}^{\infty}\Big(1-\sum_{i\leq K:l\leq m_i}\frac{d_i}{\ell_n-\Theta(l)} \Big),\\
 &= \exp\Big( - \sum_{l=1}^{\infty}\sum_{i\leq K:l\leq m_i} \frac{d_i}{\ell_n}+o(1) \Big)= \exp\Big( -\sum_{i\in [K]} \frac{d_im_i}{\ell_n} +o(1) \Big).
\end{eq} Putting $m_i=t_ib_n$, Assumption~\ref{assumption-infinte-intro} gives
\begin{equation} \label{lem::eqn::expression2}
 \frac{m_id_i}{\ell_n}= \frac{\theta_it_i}{\mu} (1+o(1)),
\end{equation}
which completes the proof of Lemma \ref{lem::convergence_indicators}.
\end{proof}

\subsubsection{Convergence of component sizes and surplus edges}
As in the finite third moment case, one must ensure that the largest components are explored early during the exploration process, i.e., we need Lemma~\ref{lem:large-com-explored-early-intro} to hold under Assumption~\ref{assumption-infinte-intro} after replacing $n^{2/3}$ by $b_n$. 
One may try to adapt the argument of Lemma~\ref{lem:large-com-explored-early-intro}, but there is a more direct and simpler approach.
The idea is that since the critical behavior is primarily governed by the asymptotics of the high-degree vertices, removing the vertices of high degree makes the graph more and more subcritical, a feature that is not present in the finite third moment case. 
This idea can be leveraged to obtain the proof that large connected components are with high probability explored in time $O(b_n)$, as well as showing the $\ell^2_{\shortarrow}$ tightness for the vector of component sizes; see Proposition~\ref{c2:tightness-comp-size}.

Let $\mathcal{G}^{\sss[K]}$ be the random graph obtained by removing all edges attached to vertices $1,\dots,K$ and let $\boldsymbol{d}'$ be the obtained degree sequence.  
Now, conditionally on the set of removed half-edges, $\mathcal{G}^{\sss [K]}$ is still a configuration model with some degree sequence $\boldsymbol{d}'$ with $d_i'\leq d_i$ for all $i\in [n]\setminus [K]$ and $d_i'=0$ for $i\in [K]$. Further, the criticality parameter of $\mathcal{G}^{\sss [K]}$ satisfies 
 \begin{equation}\label{eqn:nu-K-intro}
  \begin{split}
   \nu^{\sss [K]}_n&= \frac{\sum_{i\in [n]} d_i'(d'_i-1)}{\sum_{i\in [n]} d_i'}\leq \frac{\sum_{i\in [n]}d_i(d_i-1)-\sum_{i=1}^Kd_i(d_i-1)}{\ell_n-2\sum_{i=1}^Kd_i}\\
   &=\nu_n-C_1n^{2\alpha -1}L(n)^2\sum_{i\leq K}\theta_i^2=\nu_n-C_1c_n^{-1}\sum_{i\leq K}\theta_i^2
  \end{split}
 \end{equation}for some constant $C_1>0$.
 Since $\boldsymbol{\theta}\notin \ell^2_{\shortarrow}$, $K$ can be chosen large enough such that $\nu^{\sss[K]}_n$ becomes arbitrarily small uniformly for all $n$. 
\eqref{eqn:nu-K-intro} plays the same role in the infinite third-moment case as \eqref{nu-n-i:nu-n:relation-intro} in the finite third-moment case. 
We refer the reader to Lemma~\ref{c2lem::tail_sum_squares} for the exact details. 

In Section~\ref{sec:finite-third-moment}, we have expressed the limiting number of surplus edges in Lemma~\ref{lem:surp:poisson-conv-intro} in terms of the reflected version of the scaling limit of the exploration process. 
This deduction holds in the infinite third moment case as well, but we have to replace $R(t)$ by the reflected version of $\mathbf{S}_{\infty}$.
Thus, the finite-dimensional convergence of $\mathbf{Z}_n(\lambda)$ follows. 
The tightness argument involves establishing \eqref{eqn_sufficient_for_U_0_convergence-intro}. 
Finally, an analogue of \eqref{eqn_vertices_of_degree_k-ord-intro} also holds in this case; see Proposition~\ref{c2prop-surp-u-0}.

\subsubsection{Proof for uniform simple graphs} 
All the scaling limit results for component sizes and surplus edges hold for the critical configuration model.
Let us now outline the strategy to transfer those results to the case of $\mathrm{UM}_n(\bld{d})$, the uniformly chosen simple graph with given degree distribution.
Recall that the law of $\mathrm{UM}_n(\bld{d})$ is the same as $\CM$, conditionally on the graph being simple.  
However, since the results about $\mathbf{Z}_n(\lambda)$ are distributional convergence results, it is not evident whether the results are still true conditionally on the graph being simple.
Particularly, in the infinite third-moment case, a related question about critical $\CM$ with iid degree sequence was stated as a conjecture in \cite{Jo10}.
\begin{theorem}\label{thm:simplicity-intro}
Conditionally on $\CM$ being simple, $\mathbf{Z}_n(\lambda)\dto \mathbf{Z}(\lambda)$, where $\mathbf{Z}(\lambda)$ is the scaling limit for $\CM$.
\end{theorem}
We refer the reader to Section~\ref{c2sec:simple-graphs} for the technical details and only explain the idea here. 
Recall from \cite[Theorem 1.1]{J09c} that 
\begin{eq}
\liminf_{n\to\infty}\prob{\CM \text{ is simple}}>0.
\end{eq}
Thus, the tightness of $\mathbf{Z}_n(\lambda)$ in the $\Unot$ topology follows directly, and we only need to prove that the finite-dimensional convergence remains valid.
Note that the graph $\mathrm{UM}_n(\bld{d})$ can be generated by sampling $\CM$ until we get a uniform simple graph. 
Suppose one can show that the exploration process does not encounter any self-loops or multiple edges in time $O(b_n)$ with high probability. 
Then, we can keep the graph explored up to time $O(b_n)$ and re-sample the uniform perfect matching on the half-edges explored after time $\Omega(b_n)$, until the latter one gives a simple graph. 
This will give us a sample from $\mathrm{UM}_n(\bld{d})$ with high probability. 
However, the large components are explored in time $O(b_n)$, and these components remain fixed under re-sampling for the latter construction.
Therefore, the finite-dimensional scaling limit for $\mathbf{Z}_n(\lambda)$ does not change, conditionally on the graph being simple. 
Let us now argue that no self-loops or multiple edges are explored before time~$O(b_n)$.

Let $\ell_n':=\ell_n-2Tb_n$. 
Let $v$ be a vertex being \emph{explored} before time $O(b_n)$, 
and  $(e_1,\dots,e_r)$ the set of half-edges of $v$. 
Note that, while pairing $e_i$, it creates a self-loop with probability at most $(d_{v}-i)/\ell_n'$ and creates a multiple edge with probability at most $(i-1)/\ell_n'$. 
Therefore, conditionally on $\mathscr{F}_{l-1}$, the expected number of self-loops/multiple edges discovered while exploring the vertex $v_l$ at the $l$-th step is at most $2d_{v_l}^2/\ell_n'$.
Thus, for any~$T>0$,
\begin{align*}
 &\expt{\#\{\text{self-loops or multiple edges discovered up to time }Tb_n}\\
 &\hspace{1cm}\leq \frac{2}{\ell_n'}\E\bigg[\sum_{i\in [n]}d_i^2\mathcal{I}^n_i(Tb_n)\bigg]\\
 &\hspace{1cm}=\frac{2}{\ell_n'}\E\bigg[\sum_{i=1}^K d_i^2\mathcal{I}^n_i(Tb_n)\bigg]+\frac{2}{\ell_n'}\E\bigg[\sum_{i=K+1}^nd_i^2\mathcal{I}^n_i(Tb_n)\bigg],
 \end{align*}where $\mathcal{I}^n_i(l)=\ind{i\in \mathscr{V}_l}$. 
Now, for every fixed $K\geq 1$,
\begin{equation}
 \frac{2}{\ell_n'}\E\bigg[\sum_{i=1}^K d_i^2\mathcal{I}^n_i(Tb_n)\bigg]\leq \frac{2}{\ell_n'} \sum_{i=1}^K d_i^2 \to 0,
\end{equation} since $2\alpha-1<0$. 
Moreover, the size-biased ordering of the vertices implies that 
\begin{eq}\label{eq:estimate-exploration-indicator-intro}
\prob{\mathcal{I}^n_i(Tb_n)=1}\leq Tb_nd_i/\ell_n'. 
\end{eq}
Therefore, for some constant $C>0$,
\begin{equation}
 \begin{split} 
  \frac{2}{\ell_n'}\E\bigg[\sum_{i=K+1}^nd_i^2\mathcal{I}_i^n(Tb_n)\bigg]\leq \frac{Tb_n}{\ell_n'^2}\sum_{i=K+1}^nd_i^3\leq C \bigg( a_n^{-3}\sum_{i=K+1}^nd_i^3 \bigg), 
 \end{split}
\end{equation}which, by Assumption~\ref{assumption-infinte-intro}, tends to zero if we first take $\limsup_{n\to\infty}$ and then take $\lim_{K\to\infty}$. Consequently, for any fixed $T>0$, as $n\to\infty$,
\begin{align*}
 \prob{\text{at least one self-loop/multiple edge is discovered before time }Tb_n}\to 0,
\end{align*}which yields Theorem~\ref{thm:simplicity-intro}. 

One may note that Theorem~\ref{thm:simplicity-intro} implies that the scaling limit of a critical $\CM$ is the same as that of a critical $\UM$. 
Using Janson's construction, this yields the scaling limit of $\mathrm{CM}_n(\bld{d}, p_c(\lambda))$, conditioned to be simple. 
This is different than $\mathrm{UM}_n(\bld{d}, p_c(\lambda))$, since here first the graph is conditioned on simplicity, and after that percolation is performed. 
However, the above argument goes through if we perform the exploration process on $\mathrm{CM}_n(\bld{d}, p_c(\lambda))$ directly.

\section{Evolution over the critical window} \label{sec:intro-Q2}
In this section, we will aim to describe the evolution of the vector $\mathbf{Z}_n(\lambda)$ asymptotically. 
We discuss the following theorem:
\begin{theorem}\label{thm:mul:conv-intro} 
Fix any $k\geq 1$, $-\infty<\lambda_1<\dots<\lambda_k<\infty$. Then, there exists a version $\mathbf{AMC}=(\mathrm{AMC}(\lambda))_{\lambda\in\R}$ of the augmented multiplicative coalescent such that, as $n\to\infty$,
\begin{equation}
 \left(\mathbf{Z}_n(\lambda_1), \dots\mathbf{Z}_n(\lambda_k)\right) \dto \left(\mathrm{AMC}(\lambda_1),\dots,\mathrm{AMC}(\lambda_k)\right)
\end{equation}with respect to the $(\mathbb{U}^0_{\shortarrow})^k$ topology.
\end{theorem}  
This will be based on the analysis in \cite{DHLS15,DHLS16}.
We first introduce the candidate for the limit, the \emph{augmented multiplicative coalescent} (AMC).
Next we describe an alternative construction  for the percolation process on $\CM$.
The main problem with the percolation process $(G_n(p))_{p\in [0,1]}$ on a random graph $G_n$ is that this process is non-Markovian, unless the information about $G_n$ is incorporated into the starting sigma-algebra. 
However, in the context of $\CM$, this special construction allows us to compare the percolation process over the critical window with another dynamically growing Markovian graph process.
Then we discuss how the new graph process gives rise to AMC. 
The proof outline in this section is illustrated for the infinite third-moment case.
Although the proof for the finite third-moment case \cite{DHLS15} was given for the evolution of the component sizes only, a similar proof holds there as well (See Remark~\ref{c2:rem-AMC-finite-third}). 
Also, for sake of simplicity, we will only consider the convergence of $(\mathbf{Z}_n(\lambda_1), \mathbf{Z}_n(\lambda_2))$, and the finite-dimensional convergence follows similarly. 
\subsection{Augmented multiplicative coalescent} 
Let us now describe in detail a Markov process $(\mathbf{Z}(\lambda))_{\lambda\in\R}$, called the \emph{augmented multiplicative coalescent} (AMC) process. 
Think of  a collection of particles in a system with $\mathbf{X}(\lambda)$ describing their masses (corresponding to limiting component sizes) and $\mathbf{Y}(\lambda)$ describing an additional attribute (corresponding to surplus edges) at time $\lambda$. 
Let $K_1,K_2>0$ be constants. 
The evolution of the system at time $\lambda$ takes place according to the following rule:
\begin{itemize}
\item[$\rhd$] For $i\neq j$, at rate $K_1X_i(\lambda)X_j(\lambda)$,  the $i$-th and $j$-th components merge and create a new component of mass $X_i(\lambda)+X_j(\lambda)$ and attribute $Y_i(\lambda)+Y_j(\lambda)$.  
\item[$\rhd$] For any $i\geq 1$, at rate $K_2X_i^2(\lambda)$, $Y_i(\lambda)$ increases to $Y_i(\lambda)+1$. 
\end{itemize}
Of course, at each event time, the indices are re-organized to give a proper element of $\mathbb{U}^0_{\shortarrow}$.
The case when ignoring $\mathbf{Y}(\lambda)$ is called the multiplicative coalescent which was studied extensively in \cite{A97,AL98} in the context of understanding the evolution of the component sizes. 
The augmented version was proposed in~\cite{BBW12} to study the joint evolution of component sizes and surplus edges.
 In \cite{BBW12}, the authors showed in \citep[Theorem 3.1]{BBW12} that this is a \emph{nearly} Feller process.
To understand this precisely, let $(T_{\lambda})_{\lambda \in \R}$ denote the semigroup of operators corresponding to augmented multiplicative coalescent. 
Let $f:\Unot\mapsto \R$ be a bounded continuous function, and $(\bld{z}_n)_{n\geq 1}$ be a sequence in $\Unot$ such that $\bld{z}_n \to \bld{z}$. Further assume that $\bld{z}  =(x_i,y_i)_{i\geq 1}\in \Unot$ is such that $\sum_i x_i = \infty$. 
Then, as $n\to\infty$
\begin{eq}\label{eq:near-Feller-intro}
T_\lambda(f(\bld{z}_n)) \to T_\lambda(f(\bld{z})).
\end{eq} 
Thus, \eqref{eq:near-Feller-intro} does not hold for any $\bld{z} \in \Unot$, which is the reason why this is called a \emph{nearly Feller} property.

\subsection{An alternative construction of the percolation process}
Consider the percolation process $(\mathrm{CM}_n(\bld{d},p))_{p\in [0,1]}$, coupled through the Harris coupling. 
We give an alternative construction of the percolation process $(\mathrm{CM}_n(\bld{d},p))_{p\in [0,1]}$, that allows us to study the evolution of the percolated graphs.
\begin{algo}\label{algo:cons-perc-intro} \normalfont Let $(U_e)_{e\in [\frac{\ell_n}{2}]}$ be a finite collection of iid uniform $[0,1]$ random variables.
Construct a collection of graphs $(G_n(p))_{p\in [0,1]}$ using the following two steps:
\begin{itemize}
 \item[\textrm{(S0)}] 
 Construct the process of edge arrivals $\bld{E}_n = (E_n(p))_{p\in [0,1]}$, where \linebreak $E_n(p) = \#\{e:U_e\leq~p\}$.
 \item[\textrm{(S1)}] Initially, $G_n(0)$ is a graph only consisting of isolated vertices with no paired half-edges. 
 At each time point $p$ where $E_n(p)$ has a jump, choose two unpaired half-edges uniformly at random and pair them. 
 The graph $G_n(p)$ is obtained by adding this edge to $G_n(p-)$. 
\end{itemize} 
\end{algo} 
Algorithm~\ref{algo:cons-perc-intro}~(S0) can be regarded as the birth of edges, and (S1) ensures that the edges of the graph $G_n(p)$ are obtained from a uniform perfect matching of the corresponding half-edges. 
The fact that $(G_n(p))_{p\in [0,1]}$ has the same distribution as the percolation process can be proved by showing that the finite-dimensional distributions are equal.
The special case that $G_n(p)$ has the same distribution as $\mathrm{CM}_n(\bld{d},p)$ was proved in \cite{F07} for each fixed $p$.
The finite-dimensional convergence requires generalizing those arguments, which is done in Section~\ref{c1:sec:perc-alt-cons}.

\subsection{Comparison to a Markovian dynamic construction} 
The problem with the alternative construction in Algorithm~\ref{algo:cons-perc-intro} is that (S1)
depends on the arrival of edges during the whole process.
Thus the graph process is non-Markovian. 
Let us now describe a dynamically evolving graph process from \cite{BBSX14} which is Markovian and at the same time approximates the percolation process over the critical window.
\begin{algo}\label{algo:dyn-cons-alt-intro} \normalfont 
Let $s_1(t)$ denote the number of unpaired half-edges at time~$t$. 
Thus $s_1(0) = \ell_n$.
Let $\Xi_n$ be an inhomogeneous Poisson process with rate $s_1(t)$ at time $t$. 
\begin{itemize}
\item[\textrm{(S1)}] At each event time of $\Xi_n$, choose two unpaired half-edges uniformly at random and pair them. 
Thus $s_1(t)$ decreases by two.
The graph $\mathcal{G}_n(t)$ is obtained by adding this edge to $\mathcal{G}_n(t-)$. 
\end{itemize} 
\end{algo} 
\noindent 
Note that $(\cG_n(t))_{t\geq 0}$ is Markovian.
In fact, many properties of this graph process such as the number of unpaired half-edges can be represented using the random time change  of a unit-rate Poisson process \cite{EK86}, and thus can be shown to converge to some solution of a differential equation. 
The reader is referred \cite{wormald1995differential} for an introduction to this differential equation method.

The graph process $(\cG_n(t))_{t\geq 0}$  turns out to approximate the percolation process in the critical regime. 
To state this formally, let us define
\begin{equation}\label{defn:t-n-lambda-intro}
t_c(\lambda)=\frac{1}{2}\log\bigg(\frac{\nu_n}{\nu_n-1}\bigg)+\frac{1}{2(\nu_n-1)}\frac{\lambda}{c_n}.
\end{equation}
\begin{proposition}\label{prop:coupling-whp-intro} Fix $-\infty<\lambda_1<\lambda_2<\infty$. 
There exists a coupling such that with high probability
\begin{eq}
 \mathcal{G}_n(t_c(\lambda)-\varepsilon_n)\subset &\mathrm{CM}_n(\bld{d},p_c(\lambda)) \subset\mathcal{G}_n(t_c(\lambda)+\varepsilon_n),\quad \forall \lambda \in [\lambda_1,\lambda_2], \\
 \mathrm{CM}_n(\bld{d},p_c(\lambda)-\varepsilon_n)&\subset \mathcal{G}_n(t_c(\lambda))\subset \mathrm{CM}_n(\bld{d},p_c(\lambda)+\varepsilon_n),\ \forall \lambda \in [\lambda_1,\lambda_2],
\end{eq}where $\varepsilon_{n}=n^{-\gamma_0}$, for some $\eta<\gamma_0<1/2$.
\end{proposition}
\noindent Notice the similarity between Algorithm~\ref{algo:cons-perc-intro}~(S1) and Algorithm~\ref{algo:dyn-cons-alt-intro}~(S1).
In both processes, two unpaired half-edges, chosen uniformly at random without replacement, are paired. 
We can couple the $k$-th uniform choice to be exactly the same for both  processes. 
Thus, it is enough to compare the total number of edges, i.e., the total number of times (S1) has been executed in both algorithms.
Since one edge is created per execution of (S1), it is enough to show that with high probability the following holds: For all $\lambda \in [\lambda_1,\lambda_2]$
\begin{equation}\label{eq:coup-reduc}
e(\mathcal{G}_n(t_c(\lambda)-\varepsilon_n))\leq e(\mathrm{CM}_n(\bld{d},p_c(\lambda))) \leq e(\mathcal{G}_n(t_c(\lambda)+\varepsilon_n)). 
\end{equation}
The quantity $e(\mathcal{G}_n(t_c(\lambda)\mp \varepsilon_n))$ can be estimated  using the differential equation method. 
After some computations, it can be shown that the sandwiching inequality in \eqref{eq:coup-reduc} holds with expectation of $e(\mathrm{CM}_n(\bld{d},p_c(\lambda)))$ (see Proposition~\ref{c1:prop:coupling-whp}). 
Therefore, it is enough to establish suitable concentration inequalities for $e(\mathrm{CM}_n(\bld{d},p_c(\lambda)))$, uniformly over $\lambda\in [\lambda_1,\lambda_2]$.
We can think of $\#\{U_e\leq p\}/(\ell_n/2)$ as an empirical distribution function on $[0,1]$. 
Thus, concentration inequalities of empirical measures such as the Dvoretzky-Kiefer-Wolfowitz inequality \cite{M90} can be applied to conclude~\eqref{eq:coup-reduc}.

As a consequence of Proposition~\ref{prop:coupling-whp-intro}, it is also enough to prove Theorem~\ref{thm:mul:conv-intro} for $(\mathcal{G}_n(t_c(\lambda)))_{\lambda\in \R}$. 
This is very handy, because the Markovian nature allows us to keep track of our functionals of interest. 
Further, Theorem~\ref{thm_surplus-intro-heavy} also holds for $\mathcal{G}_n(t_c(\lambda))$.
Therefore, in the later parts of this section, we consider $\mathbf{Z}_n(\lambda)$ for the graph $\mathcal{G}_n(t_c(\lambda))$ instead.

\subsection{Convergence to AMC}\label{sec:conv-to-AMC-intro}
We write $\mathscr{C}_{\sss (i)}(\lambda)$ for the $i$-th largest component of $\mathcal{G}_n(t_n(\lambda))$, and define $\mathcal{O}_i(\lambda)$ to be the number of unpaired half-edges in $\mathscr{C}_{\sss (i)}(\lambda)$. 
Think of $\mathcal{O}_i(\lambda)$ as the \emph{mass} of $\mathscr{C}_{\sss (i)}(\lambda)$. 
Let $\mathbf{Z}_n^o(\lambda)$ denote the vector in $\Unot$, where the $|\mathscr{C}_{\sss (i)}(\lambda)|$'s are replaced by $\mathcal{O}_i(\lambda)$'s. 
Firstly, using the differential equation method, it can be shown that with high probability 
\begin{eq}
\ell_n^o(\lambda) := \sum_{i\geq 1} \mathcal{O}_i(\lambda) =  \frac{n\mu(\nu-1)}{\nu}(1+\oP(1)).
\end{eq}
Moreover, during the evolution of Algorithm~\ref{algo:dyn-cons-alt-intro}, between time $[t_c(\lambda),t_c(\lambda+\dif \lambda)]$, the $i$-th and $j$-th largest components merge at rate
 \begin{equation}\label{rate:function-intro}
2\mathcal{O}_{i}(\lambda) \mathcal{O}_{j}(\lambda)\times\frac{1}{\ell_n^o(\lambda)-1}\times \frac{1}{2(\nu_n-1)c_n}\approx \frac{\nu}{\mu(\nu-1)^2} \big(b_n^{-1}\mathcal{O}_{i}(\lambda)\big)\big(b_n^{-1}\mathcal{O}_{j}(\lambda)\big),
\end{equation}and create a component with open half-edges $\mathcal{O}_{i}(\lambda)+\mathcal{O}_{j}(\lambda)-2$ and surplus edges $\mathrm{SP}(\mathscr{C}_{\sss(i)}(\lambda))+\mathrm{SP}(\mathscr{C}_{\sss(j)}(\lambda))$. 
Also, a surplus edge is created in $\mathscr{C}_{\sss(i)}(\lambda)$ at rate
\begin{align*}
\mathcal{O}_i(\lambda)(\mathcal{O}_i(\lambda)-1)\times\frac{1}{\ell_n^o(\lambda)-1}\times \frac{1}{2(\nu_n-1)c_n}\approx \frac{\nu}{2\mu(\nu-1)^2} \big(b_n^{-1}\mathcal{O}_{i}(\lambda)\big)^2,
\end{align*}and $\mathscr{C}_{\sss(i)}(\lambda)$ becomes a component with surplus edges  $\mathrm{SP}(\mathscr{C}_{\sss(i)}(\lambda))+1$  and open half-edges $\mathcal{O}_{i}(\lambda)-2$. 
Thus $(\mathbf{Z}_n^o(\lambda))_{\lambda\in [\lambda_1,\lambda_2]}$ does \emph{not} exactly evolve as an AMC process, but it is close. 
Let us now outline two key steps for reaching the proof of Theorem~\ref{thm:mul:conv-intro} from the above heuristics.

\paragraph*{Comparison to an exact AMC.} 
If $(\mathbf{Z}_n^o(\lambda))_{\lambda\in [\lambda_1,\lambda_2]}$ would evolve as an exact multiplicative coalescent, then $(\mathbf{Z}_n^o(\lambda_1),\mathbf{Z}_n^o(\lambda_2))$ would converge by an application of the nearly Feller property in \eqref{eq:near-Feller-intro}. 
Unfortunately, that is not the case, since two half-edges are lost after each pairing, which makes the masses deplete. 
If there were no such depletion of mass, then the vector of open half-edges, along with the surplus edges, would in fact merge as an exact AMC. 
Thus, one can modify the graph process, where after time $t_c(\lambda_1)$, the paired half-edges are replaced with a newly born half-edge to the corresponding vertex; see Section~\ref{c1:sec:modified-C1}. 
Let $(\bar{\mathbf{Z}}_n^o(\lambda))_{\lambda\in [\lambda_1,\lambda_2]}$ denote the corresponding quantity under this modified algorithm. 
Then,  $\mathbf{Z}_n^o(\lambda_1) = \bar{\mathbf{Z}}_n^o(\lambda_1)$ and the nearly Feller property of AMC yields that $(\bar{\mathbf{Z}}_n^o(\lambda_1),\bar{\mathbf{Z}}_n^o(\lambda_2))$ converges. 
Finally the convergence of $(\mathbf{Z}_n^o(\lambda_1),\mathbf{Z}_n^o(\lambda_2))$ is concluded by establishing that 
\begin{eq}
d_{\sss \mathbb{U}} (\mathbf{Z}_n^o(\lambda_2), \bar{\mathbf{Z}}_n^o(\lambda_2))) \pto 0,
\end{eq}where $d_{\sss \mathbb{U}}$ denotes the metric corresponding to the $\Unot$-topology.
We refer the reader to Section~\ref{c2sec:conv-amc} for the formal deduction.

\paragraph*{Open half-edge vs component sizes.}
Finally, Theorem~\ref{thm:mul:conv-intro} is about the joint convergence of component sizes and surplus edges.
Thus, in order to conclude Theorem~\ref{thm:mul:conv-intro}, it needs to be shown that, for each fixed $\lambda$, as $n\to\infty$, 
\begin{eq}
\mathrm{d}_{\sss\mathbb{U}}(\mathbf{Z}^o_n(\lambda),\kappa \mathbf{Z}_n(\lambda))\pto 0,
\end{eq}for some $\kappa >0$. 
This can be proved using our exploration process and martingale arguments. See Section~\ref{c2sec:conv-amc} for further details.

\section{Global metric structure in the infinite third-moment case} \label{sec:intro-Q3}
In this section, we consider the metric structure of $\mathscr{C}_{\sss (i)}(p_c(\lambda))$ for $\CM$ in the infinite third-moment setting. 
This section is based on the results in~\cite{BDHS17} Chapter~\ref{chap:mspace}. 
The global metric structure limit in the finite third-moment case has been derived in \cite{BBSX14,BS16}.
Suppose that $\CM$ satisfies Assumption~\ref{assumption-infinte-intro}~(i),~(ii).
For simplicity, we ignore the slowly-varying functions here.
The percolation parameter satisfies 
\begin{equation}
p_c(\lambda) = \frac{1}{\nu_n}\big(1+\lambda n^{-\frac{\tau-3}{\tau-1}}\big), \quad -\infty<\lambda<\infty.
\end{equation}
We simply write $\mathscr{C}_{\sss (i)}(\lambda)$ to denote the $i$-th largest component of $\mathrm{CM}_n(\bld{d},p_c(\lambda))$. 
Recall from Section~\ref{sec:key-question} that $\mathscr{C}_{\sss (i)}(\lambda)$ can be viewed as a random measured metric space. 
Write $\sS_*$ for the space of all measured metric spaces equipped with the Gromov weak topology (see Section~\ref{c3:sec:definitions-full}) and let $\sS_*^{\N}$ denote the corresponding product space with the accompanying product topology. 
The goal is to show the following theorem:

\begin{theorem}\label{thm:main-mspace-intro} There exists a sequence of random measured metric spaces \linebreak $(\mathscr{M}_i(\lambda))_{i\geq 1}$ such that on $\mathscr{S}_*^\N$, as $n\to\infty$,
\begin{equation}\label{eq:thm:main-mspace-intro}
 \big( n^{-\eta}\mathscr{C}_{\sss (i)}(\lambda) \big)_{i\geq 1} \dto  \big(\mathscr{M}_i(\lambda)\big)_{i\geq 1}.
\end{equation} 
\end{theorem}
The description of the limiting metric space appearing in Theorem~\ref{thm:main-mspace-intro} requires several definitions and concepts.
An interested reader is referred to Section~\ref{c3:sec:definitions-full} for an explicit description. 
The organization of this section is as follows: 
In Section~\ref{sec:BHS15-results-intro} we start by describing the results and proof ideas from \cite{BHS15}, where the above theorem was established in the context of Norros-Reittu random graphs. 
The results allow us to explain a \emph{universality theorem} in Section~\ref{sec:universality-theorem-intro}, which identifies a domain of attraction for the same scaling limits as \cite{BHS15}. 
In Section~\ref{sec:mspace-complete-proof-intro}, we argue how this universality theorem can be applied to deduce the scaling limit in Theorem~\ref{thm:main-mspace-intro}. 
In Section~\ref{sec:GHP-convergence-intro}, we describe the idea of establishing the so-called \emph{global mass lower bound} which can be used to improve the underlying topology of convergence in Theorem~\ref{thm:main-mspace-intro}.
Due to the technical nature of these results, the proof ideas will be more sketchy than the previous sections, with the detailed treatment left to Chapters~\ref{chap:mspace}, and~\ref{chap:mspace-GHP}.

\subsection{Scaling limit for Norros-Reittu model} \label{sec:BHS15-results-intro}
Before going into the proof ideas in \cite{BDHS17}, let us briefly describe the results from~\cite{BHS15}, along with key proof ideas.
In this section, we write $\NRx$ to denote the random graph obtained by keeping edge $(i,j)$ independently with probability $1-\e^{-qx_ix_j}$. 
Thus this is same as the Norros-Reittu model defined in Section~\ref{sec:intro-models}, where the normalization has been changed for the  sake of simplicity.
 The scaling limit result in \eqref{eq:thm:main-mspace-intro} was derived for $\NRx$, and the candidate scaling limit was identified for the heavy-tailed regime. 
\paragraph*{${\vp}$-trees and their limit.}
To create the context, first let us describe a random tree known as a $\vp$-tree.
Fix $m \geq 1$, and a non-increasing sequence $\vp = (p_i)_{i\in [m]}$ with $p_i > 0$ for all $i\in [m]$, and $\sum_{i\in [m]}p_i=1$.
Then the law of the $\vp$-tree, denoted by $\sT_m^{\mathbf{p}}$, is given by
\begin{equation}
\label{eqn:p-tree-def-intro}
	\PR(\sT_m^{\mathbf{p}}= \vt)= \prod_{v\in [m]} p_v^{d_v(\vt)}, \quad \text{ for any tree }\mathbf{t}\text{ on }m\text{ vertices.}
\end{equation}
In an ordered $\vp$-tree, children of each individual are assigned a uniform order. 
It was shown in \cite{CP99,P01} that the random tree $\sT_m^{\mathbf{p}}$, after assigning length $\sigmap := (\sum_{i\in [m]}p_i^2)^{1/2}$ to each edge, converges in distribution to the so-called inhomogeneous continuum random tree in the Gromov-Hausdorff topology. 
If for each fixed $i\geq 1$, $p_i/\sigmap \to \beta_i,$ for some $(\beta_i)_{i\geq 1} \in \ell^2_{\shortarrow}\setminus\ell^1_{\shortarrow}$, then the limiting object turns out to be structurally completely different from the classical Brownian continuum random tree. 
This case exhibits the \emph{heavy-tail effect}, in the sense that the limiting structure contain vertices, so-called \emph{hubs}, of infinite degree almost everywhere. 

\paragraph*{{A construction of the components of $\NRx$}.}
A novel construction of the connected components of $\NRx$ in the critical regime was proposed in \cite{BSW14,BHS15}. 
Let $\sC_{\sss (i)}(q)$ denote the $i$-th largest component of $\NRx$. 
Define, for $i\geq 1$,
	\begin{equation}\label{eq:p-n-a-NR-intro}
		\vp_n^{\sss(i)} := \bigg( \frac{x_v}{\sum_{v \in \sC_{\sss (i)}(q)}x_v }  \bigg)_{v \in \sC_{\sss (i)}(q)}, \quad a_n^{\sss(i)}:= q\bigg(\sum_{v\in \sC_{\sss (i)}(q)} x_v\bigg)^2.
	\end{equation}
Then, $\sC_{\sss (i)}(q)$ can be informally generated in the following two steps: 
\begin{algo}\label{algo:tilt-p-intro}\begin{itemize}
\item[$\mathrm{(S1)}$] Generate a tilted $\vp$-tree with $\mathbf{p}_n^i$ in  \eqref{eq:p-n-a-NR-intro}, and some tilting function $L(\cdot)$.
\item[$\mathrm{(S2)}$] Generate a mixed Poisson random variable $N$, and add $N$ many surplus edges.
\end{itemize}
\end{algo}
\noindent The formal description of the above algorithm is given in Section~\ref{c3:sec:definitions-full}, which requires several technical definitions. 
It turns out that all the functionals involved in the formal construction of Algorithm~\ref{algo:tilt-p-intro} depend only on the quantities defined in \eqref{eq:p-n-a-NR-intro}. 
Thus, the graph in Algorithm~\ref{algo:tilt-p-intro} can be generated for any $\vp$ and $a$ instead of specific choices in \eqref{eq:p-n-a-NR-intro}. 
We denote such a graph by~$\cG_m(\vp,a)$.

\paragraph*{Scaling limits for critical $\NRx$.}
The graph $\cG_m(\vp,a)$ is generated from a $\vp$-tree $\sT_m^{\mathbf{p}}$, after tilting the distribution by the function $L(\cdot) = L_m(\cdot)$, and then adding only finitely many shortcuts. 
Thus, provided that $(L_m)_{m\geq 1}$ is uniformly integrable, it is not difficult to imagine that the distance in $\cG_m(\vp,a)$ should scale similarly as the $\vp$-tree, which is $\sigmap$. 
However, one needs to establish a distributional convergence result which turns out to be significantly harder. 
The following result was proved in \cite{BHS15}, and the limiting object was identified as a function of the inhomogeneous continuum random tree after an appropriate tilt:
 \begin{theorem} \label{thm:BHS15-intro}
  Suppose that $\sigmap\to 0 $, and for each fixed $i\geq 1$, $p_i/\sigmap \to \beta_i$, where $\bld{\beta}=(\beta_i)_{i\geq 1}\in\ell^2_{\shortarrow}\setminus\ell^1_{\shortarrow}$. 
  Moreover, there exists a constant $\gamma >0$ such that $a\sigmap\to\gamma$. 
  Then, as $m\to\infty$,
\begin{eq}
\sigmap \mathcal{G}_m(\mathbf{p},a) \text{ converges in distribution in the Gromov-weak topology.}
\end{eq}  
\end{theorem}
\noindent The proof of Theorem~\ref{thm:BHS15-intro} consists of showing that the tilting functions are converging in distribution, and the operation of creating shortcuts on the space of ``trees" is continuous.
Next, the idea is to apply Theorem~\ref{thm:BHS15-intro} with the parameters  in~\eqref{eq:p-n-a-NR-intro}. 
Therefore, one needs to prove distributional convergence results for these parameters and also obtain the asymptotics of $\sigma(\vp_n^{\sss(i)})$.
These asymptotics can be obtained using exploration processes on $\NRx$. We refer to \cite{BHS15} for further details.

\subsection{The universality theorem}\label{sec:universality-theorem-intro}
In this section, we describe the \emph{universality} theorem which forms the basis of the results in \cite{BDHS17}.
A similar result for the Erd\H{o}s-R\'enyi universality class was established in~\cite{BBSX14}.
Let us first describe the universality theorem, and then discuss how this can be applied to obtain the proof of Theorem~\ref{thm:main-mspace-intro}.

The idea is to replace each of the vertices in the graph $\cG_m(\vp,a)$ by so-called \emph{blobs}.
Blobs are a collection $\{(M_i,\mathrm{d}_i,\mu_i)\}_{i\in [m]}$ of  connected, compact measured metric spaces.
Consider  an independent collection of random points $\mathbf{X}:=(X_{i,j})_{i,j\in [m]}$ such that $X_{i,j}\sim \mu_i$ for all $i,j$. Further, $\mathbf{X}$ is independent of~$\cG_m(\vp,a)$.
Put an edge of \emph{length} one between the pair of points
$$\{(X_{i,j},X_{j,i}): (i,j) \text{ is an edge of }\cG_m(\vp,a)\},$$ and denote the resulting graph by $\tilde{\cG}_m^{\sss \mathrm{bl}}(\vp,a)$.
$\tilde{\cG}_m^{\sss \mathrm{bl}}(\vp,a)$ inherits the metric from the graph-distance and the distances within blobs; see Section~\ref{c3:sec:univ-thm} for an exact description. 
Let $u_i:=\E[\mathrm{d}_i(X_i,X_i')]$ where $X_i,X_i'\sim \mu_i$ independently, $B_m:= \sum_{i\in [m]}p_iu_i$ and $\Delta_i = \mathrm{diameter}(M_i)$. 
\begin{theorem}[Universality theorem]\label{thm:univesalty-intro}
Suppose that the assumptions of \textrm{Theorem~\ref{thm:BHS15-intro}} hold, and additionally $\lim_{m\to\infty}\frac{\sigma(\mathbf{p})\max_{i\in [m]}\Delta_i}{B_m+1}=0.$
As $m\to\infty$,
\begin{equation}
\frac{\sigma(\mathbf{p})}{B_m+1}  \tilde{\mathcal{G}}_m^{\sss \mathrm{bl}}(\mathbf{p}, a) ,\text{ and } \sigmap \mathcal{G}_m(\mathbf{p},a) \text{ have the same limit}
 \end{equation}with respect to the Gromov-weak topology. 
 \end{theorem}
\noindent 
The proof of Theorem~\ref{thm:univesalty-intro} studies the effect of introducing surplus edges, or shortcuts in a $\vp$-tree. 
The proof uses the birthday construction of $\vp$-trees from \cite{CP99}.
A detailed proof is provided in Section~\ref{c3:sec:univ-thm}.

To this end, let us observe that Norros-Reittu random graphs have a direct relation to multiplicative coalescence in the following sense: 
Consider a system of $n$ vertices, with vertex $i$ having mass $x_i$.
Now, at rate $x_ix_i$, an edge is created between $i$ and $j$. 
The obtained graph at time $q$ is distributed as $\NRx$, defined in Section~\ref{sec:BHS15-results-intro}.
Also, if we track the sum of $x_i$'s in each component, it evolves as an exact multiplicative coalescent.
Thus the multiplicative coalescence evolution essentially gives rise to Norros-Reittu random graphs.
Next consider Algorithm~\ref{algo:dyn-cons-alt-intro}, and the modification of replacing open half-edges after time $t_c(\lambda_1)$  given in Section~\ref{sec:conv-to-AMC-intro}, where now $\lambda_1= \lambda_1(n) \to -\infty$. 
Therefore, at the beginning of the modification, the graph process is in the barely subcritical regime.
The modified process runs as an exact multiplicative coalescent, essentially giving rise to a \emph{superstructure} of a Norros-Reittu graph on top of the barely subcritical components. 
Let us denote this graph by $\bar{H}_n$, and denote the graph produced by Algorithm~\ref{algo:dyn-cons-alt-intro}  by $H_n$. 
Now one can apply Theorem~\ref{thm:univesalty-intro} to the graph $\bar{H}_n$, with the components in the barely subcritical regime serving as blobs.
Finally, we obtain the metric structure of $H_n$ by comparing its  structures with $\bar{H}_n$.

\subsection{Final steps in completing the proof}\label{sec:mspace-complete-proof-intro}
\paragraph*{Properties at the barely subcritical regime.} 
In order to estimate several functionals like the quantities in \eqref{eq:p-n-a-NR-intro}, and in Theorem~\ref{thm:univesalty-intro}, we need to obtain precise asymptotics of functionals of a barely subcritical configuration model, whose components serve as blobs. 
This is an interesting question in its own right, which was not studied previously for $\CM$, in the $\tau\in (3,4)$ universality class. 
In order to calculate the quantities in \eqref{eq:p-n-a-NR-intro}, we wish to verify the entrance boundary conditions for the ``behavior at $-\infty$" of the multiplicative coalescent \cite{AL98}, which characterize asymptotics of several functionals related to the multiplicative coalescent in terms of the asymptotics on the entrance boundary. 
The barely subcritical regime lies on the entrance boundary, and the verification of the Aldous-Limic entrance boundary yields the desired asymptotics. 
Further, one also needs to obtain bounds on the maximum diameter of the barely subcritical components $\max_{i\in [m]}\Delta_i$, and the within component average distance $B_m$ in order to apply Theorem~\ref{thm:univesalty-intro}.

\paragraph*{Structural comparison of $H_n$ and $\bar{H}_n$.}
With the asymptotics obtained in the barely subcritical regime, Theorem~\ref{thm:univesalty-intro} applies to $\bar{H}_n$. 
Finally, a structural comparison between the components of $H_n$ and $\bar{H}_n$ completes the proof of Theorem~\ref{thm:main-mspace-intro}.
Let $\sC_{\sss (i)}$ and $\bar{\sC}_{\sss (i)}$ denote the $i$-th largest component of $H_n$ and $\bar{H}_n$ respectively. 
Then the following structural comparisons allow us to conclude that the limit of largest connected components are the same in the Gromov-weak topology:
\begin{enumerate}
\item[$\rhd$] For any $i\geq 1$, $\mathscr{C}_{\sss (i)} \subset \bar{\mathscr{C}}_{\sss (i)}$ with high probability. 
Note that, under the replacement scheme of open half-edges, $\cup_{j\leq i}\mathscr{C}_{\sss (j)} \subset \cup_{j\leq i}\bar{\mathscr{C}}_{\sss (j)}$ almost surely for any $i\geq 1$. 
Therefore, this statement can be concluded by showing that the component sizes $(|\mathscr{C}_{\sss (i)}|)_{i\geq 1}$ and $(|\bar{\mathscr{C}}_{\sss (i)}|)_{i\geq 1}$ have the same limit. 
\item[$\rhd$] The ``mass" of $\bar{\mathscr{C}}_{\sss (i)}\setminus \mathscr{C}_{\sss (i)}$ converges to zero in probability.

\item[$\rhd$] For any pair of vertices $u,v\in \mathscr{C}_{\sss (i)}(\lambda)$, with high probability, the shortest path between them is exactly the same in $\mathscr{C}_{\sss (i)}$ and $\bar{\mathscr{C}}_{\sss (i)}$. 
This is obtained by showing that the number of surplus edges with at least one endpoint in  $\bar{\mathscr{C}}_{\sss (i)}\setminus \mathscr{C}_{\sss (i)}$ converges to zero in probability. 
\end{enumerate}
For further details about the formal statements and the verification of the above properties, the reader is referred to Section~\ref{c3:sec:proof-metric-mc}.

\subsection{Gromov-Hausdorff-Prokhorov convergence of the critical  components}\label{sec:GHP-convergence-intro}
We now describe the so-called global lower mass-bound property of the critical components. 
The property basically establishes a lower bound on the number of vertices within small neighborhood of the connected components.
If there is a single path of length $n^{\eta}\log(n)$, then since $\rho>\eta$ and $|\sC_{\sss (i)}(p_c(\lambda))| = \Theta_{\sss \PR}(n^{\rho})$, we do not see any members from that path if we sample finitely many points from $\sC_{\sss (i)}(p_c(\lambda))$ uniformly at random. 
For this reason, the Gromov-weak convergence does not take into whether there is a \emph{thin long path} in the component. 
Further, since the Gromov-weak convergence is defined on the space of complete separable metric spaces, the limit may not be a compact metric space. 
The global lower mass bound rules out the existence of such thin long paths.
Consequently, this implies that the scaling limit in Theorem~\ref{thm:main-mspace-intro} holds with respect to the Gromov-Hausdorff-Prokhorov (GHP) topology, the limiting metric space is compact, and the global distance related functionals (e.g.~the diameter) converges after the appropriate rescaling.
This relation between the Gromov-weak convergence and GHP convergence was studied in \cite{ALW16}.

We consider a critical configuration model and denote the $i$-th largest connected component of $\rCM_n(\bld{d})$ by $\csi$.
For each $v\in [n]$ and $\delta>0$, let $\mathcal{N}_v(\delta)$ denote the $\delta n^{\eta}$ neighborhood of $v$ in $\rCM_n(\bld{d})$. 
For each $i\geq 1$, define 
\begin{equation} \label{eq:m-i-defn}
 \mathfrak{m}_i^n(\delta) = \inf_{v\in\mathscr{C}_{\sss (i)}}n^{-\rho}| \mathcal{N}_v(\delta)|.
\end{equation}
\begin{theorem}[Global lower mass-bound]
\label{thm:gml-bound-intro} 
For any $\delta>0$, $(\mathfrak{m}_i^n(\delta)^{-1})_{n \geq 1}$ is a tight sequence. 
\end{theorem}
In Chapter~\ref{chap:mspace-GHP}, Theorem~\ref{thm:gml-bound-intro} is proved under a more general setting, but the proof requires some additional technical assumptions on the degree distribution on top of Assumption~\ref{assumption-infinte-intro} (see Assumption~\ref{c4:assumption1}).
The additional assumption is satisfied for power law degrees.
Following the above heuristic description, Theorem~\ref{thm:gml-bound-intro} now yields several interesting corollaries.  
Using the results from \cite{ALW16}, Theorem~\ref{thm:gml-bound-intro} establishes that the convergence in Theorem~\ref{thm:main-mspace-intro} holds with respect to the GHP topology. 
This in particular establishes that the limiting metric spaces in \cite{BHS15,BDHS17} are compact almost surely.
Due to technical reasons, some additional conditions are imposed on $\bld{\theta}$.
For example, the assumption is satisfied for $\theta_i \in [L_1(i)i^{-a_1}, L_2(i)i^{-a_2}]$, where $a_1,a_2 \in (1/3,1/2)$, and $L_1,L_2$ are slowly varying functions.
This is much less restrictive than assuming $\theta_i = i^{-\alpha}$ as in \cite{BHS15}.
The compactness of the limiting metric spaces in \cite{BHS15,BDHS17} was also established under some regularity conditions in a very recent preprint~\cite{BDW18} 
using a completely independent method as in this paper. 
In addition to the compactness of the limiting metric space, we also have the convergence of the diameters, i.e.,
\begin{eq}
\big(n^{-\eta}\diam(\sC_{\sss (i)}(p_c(\lambda)))\big)_{i\geq 1} \dto (X_i)_{i\geq 1}
\end{eq}
with respect to the product topology, where $(X_i)_{i\geq 1}$ is a non-degenerate random vector. 
In fact $X_i$ corresponds to the diameter of the limiting object of $\sC_{\sss (i)}(p_c(\lambda))$ from \cite{BDHS17}.

Let us just briefly describe the key ideas in the proof of Theorem~\ref{thm:gml-bound-intro}, which consists of two main steps. 
The first step is to show that the neighborhoods of the high-degree vertices, or \emph{hubs}, have mass  $\Theta(n^{\rho})$. 
Secondly, the probability of any small~$\varepsilon n^{\eta}$ neighborhood not containing hubs is arbitrarily small.
These two facts, summarized in Propositions~\ref{c4:prop:size-nbd bound} and~\ref{c4:prop:diamter-small-comp} below, together ensure that the total mass of any neighborhood of $\sC_{\sss(i)}$ of radius $\varepsilon n^{\eta}$ is bounded away from zero. 
These facts were proved in \cite{BHS15} in the context of inhomogeneous random graphs.
The main advantage in \cite{BHS15} was that the breadth-first exploration of components could be dominated by a branching process with \emph{mixed Poisson} progeny distribution that is \emph{independent of $n$}. 
The above facts allow one to use existing literature and estimate the probabilities that a long path exists in the branching process in \cite{BHS15}.
However, such a technique is specific to rank-one inhomogeneous random graphs and does not work in the cases where the above stochastic domination is not possible. 
This was partly a motivating reason for this work.
Moreover, along the way we derive results about exponential bounds for the number of edges in the large components (Proposition~\ref{c4:lem:volume-large-deviation}), and 
a coupling of the neighborhood exploration with a branching process with stochastically larger progeny distribution (Section~\ref{c4:sec:BP-approximation}), which is interesting in its own right.
The details are left to Chapter~\ref{chap:mspace-GHP}.

%
%
%
%
%
%

\section{Analysis in the infinite second-moment case}
\label{sec:intro-infinite-second}
We next discuss the critical behavior for percolation when the asymptotic empirical degree distribution is approximately a power law with exponent $\tau \in (2,3)$, i.e., the degree distribution has infinite second moment, but finite second moment. 
As canonical random graph models on which percolation acts, we take $\CM$, $\ECM$ and $\GRG$.
The latter two models only allow for single edges, which is the reason for  referring to them as models with a single-edge constraint.
In Section~\ref{sec:infinte-second-CM-intro}, we describe the results for  $\CM$, and in Section~\ref{sec:infinite-second-ECM-intro} those for $\ECM$ and $\GRG$.
The results are based on ongoing work \cite{DHL18}.

\subsection{Results for the configuration model} 
\label{sec:infinte-second-CM-intro}
We start by describing the assumptions on the degree distribution. 
Fix any $\tau \in (2,3)$. 
We denote 
\begin{equation}\label{eqn:notation-const-intro}
 \alpha= 1/(\tau-1),\quad \rho=(\tau-2)/(\tau-1),\quad \eta=(3-\tau)/(\tau-1).
\end{equation}
Note that the $\eta$ in \eqref{eqn:notation-const-intro} is different than in Section~\ref{sec:infinite-third-moment-intro}.
We assume the following about the degree sequences $(\bld{d}_n)_{n\geq 1}$:
\begin{assumption}
\label{assumption-infinite-second}
\normalfont  
\begin{enumerate}[(i)] 
\item  (\emph{High-degree vertices}) For any  $i\geq 1$, 
$ n^{-\alpha}d_i\to \theta_i,$
where the vector $\boldsymbol{\theta}=(\theta_1,\theta_2,\dots)\in \ell^2_{\shortarrow}\setminus \ell^1_{\shortarrow}$. 
\item  (\emph{Moment assumptions}) 
 $\lim_{n\to\infty}\frac{1}{n}\sum_{i\in [n]}d_i= \mu,$ and $$ \lim_{K\to\infty}\limsup_{n\to\infty}n^{-2\alpha} \sum_{i=K+1}^{n} d_i^2=0.$$
\end{enumerate}
\end{assumption}
\noindent Assumption~\ref{assumption-infinite-second}~(i) fixes the asymptotics of the high-degree vertices in a similar manner as Assumption~\ref{assumption-infinte-intro}~(i), and characterizes jumps of an associated exploration process that we describe in detail below. 
As before, we will consider a size-biased exploration process. 
Assumption~\ref{assumption-infinite-second}~(ii) says that the expectation of this size-biased distribution is carried predominantly by the contribution due to the hubs. 
Note that, under Assumption~\ref{assumption-infinite-second}, $\alpha>1/2$, and consequently the criticality parameter for $\CM$ satisfies $\nu_n= \Theta(n^{2\alpha-1})$, which tends to infinity as $n\to\infty$.
The critical behavior for percolation on $\CM$ is observed for values of $p$ given by  
\begin{equation}\label{eq:crit-value-tau23-intro}
p_c=p_c(\lambda):= \frac{\lambda}{\nu_n}(1+o(1)), \quad \lambda \in (0,\infty).
\end{equation}
The nature of the critical window for $\tau\in (2,3)$ is fundamentally different than in the finite second-moment case. 
Here, the graph becomes more and more subcritical (or supercritical) as $\lambda \to 0$ (or $\lambda\to \infty$), contrary to the $\lambda\to \mp \infty$ scenario for $\tau>4$ and $\tau\in (3,4)$.

%
To describe the results for the component sizes and the surplus edges,  recall that $\mathbf{Z}_n(\lambda)$ denotes the vector $(n^{-\rho}|\sCi(p_c(\lambda))|,\SP(\sCi(p_c(\lambda))))_{i\geq 1}$, ordered as an element of~$\Unot$.
A vertex is called isolated if it has degree zero in the graph $\CMP$. 
We define the component size corresponding to an isolated vertex to be zero. 
This is required because in this case $2\rho<1$. 
When we perform percolation with probability $p_c(\lambda)$, we see order $n$ isolated vertices and thus $n^{-2\rho}\times (\#$ isolated vertices$)$ tends to infinity, which destroys the $\ell^2_{\shortarrow}$-tightness of the component sizes. 
The following theorem gives the asymptotics for the component sizes and the complexity for $\CMP$: 
\begin{theorem}[Component sizes and complexity]\label{thm:tau23-main-intro} Under  Assumption~\ref{assumption-infinite-second}, as $n\to\infty$,
\begin{equation}
\mathbf{Z}_n(\lambda) \dto \mathbf{Z}(\lambda)
\end{equation}with respect to the $\Unot$ topology, where $\mathbf{Z}(\lambda)$ is some non-degenerate random described in detail in Theorem~\ref{c5:thm:main}.
\end{theorem}
Our next result shows that the diameter of the largest connected components is of constant order, which yields further insight into the distance structure of these critical components. 
Let $\Delta_i^n$ denote the diameter of $\sCi(p_c(\lambda))$.
\begin{theorem}[Diameter of largest clusters]\label{thm:diameter-large-comp} Under Assumption~\ref{assumption-infinite-second}, $(\Delta_i^n)_{n\geq1}$ is a tight sequence of random variables, for any $i\geq 1$.
\end{theorem}
%
%

In order to establish that \eqref{eq:crit-value-tau23-intro} gives the critical value, we further investigate the barely sub/supercritical regimes which are defined respectively by $p_n\ll p_c(\lambda)$ and $p_n \gg p_c(\lambda)$. 
We prove the following theorem for the barely-subcritical regime:
\begin{theorem}[Barely subcritical regime]\label{thm:barely-subcrit-tau23-intro} Suppose that $\frac{\log(n)}{\ell_n}\ll p_n\ll p_c(\lambda)$ and Assumption~\ref{assumption-infinite-second} holds. Then, for each fixed $i\geq 1$, as $n\to\infty$,
\begin{equation}
\frac{|\sCi(p_n)|}{n^\alpha p_n} \pto \theta_i.
\end{equation}
\end{theorem}
For the result about the barely supercritical regime, we need one further mild technical assumption, which is as follows:
Let $D_n^*$ denote the degree of a vertex chosen in a size-biased manner with the sizes being $(d_i/\ell_n)_{i\in [n]}$.
Then, there exists a constant $\kappa>0$ such that 
\begin{equation}\label{eq:asymp-laplace}
1 - \E[\e^{-tp_n^{1/(3-\tau)}D_n^*}] = \kappa t^{\tau-2} p_n^{(\tau-2)/(3-\tau)}(1+o(1)).
\end{equation}
Condition \eqref{eq:asymp-laplace} is related to the Abel-Tauberian theorem \cite[Chapter XIII.5]{Fel91}, but due to the joint scaling of $p_n$ and $D_n^*$, this has to be stated as an assumption.
In Chapter~\ref{chap:infinite-second}, we show that \eqref{eq:asymp-laplace} is satisfied 
if $d_i = (1-F)^{-1}(i/n)$ for some distribution function $F$ supported on non-negative integers, and $(1-F)(x) = C k^{-(\tau-1)},$ for  $k \leq x< k+1.$
The next theorem considers the barely supercritical regime:
\begin{theorem}[Barely supercritical regime]\label{thm:barely-supercrit-tau23-intro} 
Suppose that $p_n\gg p_c(\lambda)$, and Assumption~\ref{assumption-infinite-second} and \eqref{eq:asymp-laplace} holds. Then, as $n\to\infty$,
\begin{equation}
\frac{|\sC_{\sss (1)}(p_n)|}{np_n^{1/(3-\tau)}} \pto \frac{\mu\kappa^{1/(3-\tau)}}{2^{(\tau-2)/(3-\tau)}}, \quad \frac{\rE(\sC_{\sss (1)}(p_n))}{np_n^{1/(3-\tau)}} \pto \frac{\mu\kappa^{1/(3-\tau)}}{2^{(4-\tau)/(3-\tau)}},
\end{equation}and for all $i\geq 2$, $|\sCi(p_n)| = \oP(np_n^{1/(3-\tau)})$, $\rE(\sCi(p_n)) = \oP(np_n^{1/(3-\tau)})$, where $\rE(G)$ denotes the number of edges in the graph $G$.
\end{theorem}
In the next section, we will briefly discuss the main challenges in proving Theorem~\ref{thm:tau23-main-intro}. 
The reader is referred to Chapter~\ref{chap:infinite-second} for a further rigorous treatment of all the above mentioned results.

\subsubsection*{Proof ideas for Theorem~\ref{thm:tau23-main-intro}.}
The proof consists of two key steps:
Set up a suitable exploration process which converges to a stochastic process; 
and analyze the scaling limit of the exploration process. 

\paragraph*{The exploration process.}
In Section~\ref{sec:Janson-construction-intro}, Janson's construction in Algorithm~\ref{algo:janson-construction-intro} played a crucial role in representing the percolated configuration model as a configuration model, and thus one could use the exploration process on a configuration model to make conclusions about the percolated graph. 
Unfortunately, this technique does not work anymore when $p_c \to 0$, because in that case red vertices outnumber non-red vertices, which makes the discovery of the non-red vertices rare during the exploration process. 
However, we can still \emph{approximate} $\CMP$ by a suitable configuration model, which is described as follows:
\begin{algo}\label{algo-alt-cons-perc-tau23-intro}
\begin{itemize}\normalfont
 \item[(S0)] Retain each half-edge with probability $p_n$. If the total number of retained half-edges is odd, attach a \emph{dummy} half-edge with vertex 1.
 \item[(S1)] Perform a uniform perfect matching between the half-edges retained in~(S0).
Pair unpaired half-edges sequentially with a uniformly chosen unpaired half-edge until all half-edges are paired. 
 The paired half-edges create edges in the graph, and we call the resulting graph $\mathcal{G}_n(p_n)$.
\end{itemize}
\end{algo}
Let $\Mtilde{\bld{d}}=(\tilde{d}_1,\dots,\tilde{d}_n)$ be the degree sequence induced by Algorithm~\ref{algo-alt-cons-perc-tau23-intro}~(S1).
Then $\cG_n(p_n)$ is distributed as $\rCM_n(\Mtilde{\bld{d}})$.
Moreover, the following proposition states that it is enough to consider the scaling limit of $\mathbf{Z}_n(\lambda)$ in $\cG_n(p_n)$:
\begin{proposition} \label{prop:coupling-lemma-intro}
Let $p_n$ be such that $\ell_np_n \gg \log(n)$. 
Then there exists a sequence $(\varepsilon_n)_{n\geq 1}$ with $\varepsilon_n\to 0 $, and a coupling such that, with high probability,
\begin{equation}
\mathcal{G}_n(p_n(1-\varepsilon_n))\subset\mathrm{CM}_n(\bld{d},p_n)\subset\mathcal{G}_n(p_n(1+\varepsilon_n)).
\end{equation}
\end{proposition}
\noindent 
The proof of Proposition~\ref{prop:coupling-lemma-intro} is provided in Proposition~\ref{c5:prop:coupling-lemma}. 
We can now set up the exploration process on $\rCM_n(\Mtilde{\bld{d}})$.
Consider the same exploration algorithm and process as in the $\tau\in (3,4)$ case defined in \eqref{eq:exploration-intro-tau34} on the graph $\rCM_n(\Mtilde{\bld{d}})$. 
 Recall that if $\cI_i^n(l)$ denotes the indicator that vertex $i$ is discovered before time $l$, then the exploration process is given by
   \begin{equation}\label{eq:expl-process-crit-tau23-intro}
    S_n(l)= \sum_{i\in [n]} \tilde{d}_i \mathcal{I}_i^n(l)-2l.
   \end{equation} 
   Define the re-scaled version $\bar{\mathbf{S}}_n$ of $\mathbf{S}_n$ by $\bar{S}_n(t)= n^{-\rho}S_n(\lfloor tn^{\rho} \rfloor)$. Then,
   \begin{equation} \label{eqn::scaled_process-tau23-intro}
    \bar{S}_n(t)= n^{-\rho} \sum_{i\in [n]}\tilde{d}_i \mathcal{I}_i^n(tn^{\rho})-2t + o(1).
   \end{equation}
Now, using the estimate of the exploration probability in the above exploration process from  \eqref{eq:estimate-exploration-indicator-intro} 
\begin{eq}\label{error-larger-K-tau23-intro}
n^{-\rho}\E\bigg[\sum_{i>K}\tilde{d}_i\mathcal{I}_i^n(tn^{\rho})\Big\vert \Mtilde{\bld{d}}\bigg]\leq  \frac{t\sum_{i>K}\tilde{d}_i^2}{\tell_n -2tn^{\rho}}.
\end{eq}   
   Using Assumption~\ref{assumption-infinite-second}, along with the fact that $\Mtilde{d}_i\sim \mathrm{Bin}(d_i,p_n)$ independently over $i\in [n]$, it is not difficult to show that the probability that the final term in \eqref{error-larger-K-tau23-intro} is more than $\varepsilon$ tends to zero in the iterated limit \linebreak $\lim_{K\to\infty}\limsup_{n\to\infty}$. 
   See Lemma~\ref{c5:lem:perc-degrees} for more details. 
Therefore, it is enough to find the scaling limit of \eqref{eqn::scaled_process-tau23-intro} by truncating the sum upto the first $K$ terms and then taking the iterated limit as $n\to\infty$ and then $K\to\infty$.
Upon a closer inspection, one can verify that an analogue of Lemma~\ref{lem::convergence_indicators} is true for this case as well, which yields the following result:
\begin{theorem} \label{thm::convegence::exploration_process-tau23-intro}
Under  Assumption~\ref{assumption-infinite-second}, as $n\to\infty,$
 \begin{equation}
  \bar{\mathbf{S}}_n \dto \bar{\mathbf{S}}_\infty
 \end{equation} with respect to the Skorohod $J_1$ topology, where  the limiting process is defined by
 \begin{equation}\label{defn::limiting::process-tau23-intro}
\iS(t) =  \lambda\sum_{i=1}^{\infty} \theta_i\mathcal{I}_i(t)- 2 t,
\end{equation}for $\mathcal{I}_i(s):=\ind{\xi_i\leq s }$ with $\xi_i\sim \mathrm{Exp}(\theta_i/\mu)$ independently,
\end{theorem}

\paragraph*{Analysis of the limiting process.} 
The limiting process \eqref{defn::limiting::process-tau23-intro} has turned up for the first time in the critical random graph literature, and its description is not covered by the general framework provided by Aldous and Limic~\cite{AL98}.
One needs to establish several properties of the $\biS$ process to conclude that the rescaled component sizes converge to its ordered excursion lengths. 
For example, one first needs to show whether it is at all possible to order the excursion lengths.
Further, the function mapping a c\`adl\`ag function to its largest excursion is only continuous on a subset of \emph{good} c\`adl\`ag functions under the Skorohod $J_1$ topology, see Definition~\ref{c2defn::good_function}. 
Therefore, one needs to verify that the sample paths of \eqref{defn::limiting::process-tau23-intro} are good almost surely. 
The following proposition allows us to establish all those good properties:
\begin{proposition} \label{prop:limit-properties-tau23-intro}
\begin{itemize}
\item[(P1)] As $t\to\infty$, $\iS(t) \asto -\infty$. Thus, $\biS$ does not have an  excursion of infinite length almost surely.
\item[(P2)] For any $\delta >0$, $\biS$ has only finitely many excursions of length at least $\delta$ almost surely. 
\item[(P3)] Let $\mathcal{R}$ denote the set of excursion end-points of $\biS$. 
Then $\mathcal{R}$ does not have an isolated point.
\item[(P4)] For any $t>0$, $\PR(\iS(t)=\inf_{u\leq t}\iS(u))=0$.
\end{itemize}
\end{proposition}
\noindent 
The proof is mostly technical and is provided in Section~\ref{c5:sec:properties-exploration}.

\paragraph*{Completing the proof.} Let us now briefly outline the final ingredients of the proof of Theorem~\ref{thm::convegence::exploration_process-tau23-intro}. 
Firstly, if $A_l$ denotes the number of active half-edges after stage $l$ while implementing the exploration algorithm, then note that the probability of creating a surplus edge at time $i$ conditionally on $\mathscr{F}_{i-1}$ is given by
\begin{equation}
 \frac{A_{i-1}-1}{\tilde{\ell}_n-2i-1}= \frac{A_{i-1}}{\tilde{\ell}_n}(1+O(i/n))+O(n^{-1}),
\end{equation} uniformly for $i\leq Tn^{\rho}$ for any $T>0$. 
Therefore, the instantaneous rate of creating surplus edges at time $tn^{\rho}$, conditional on the past, is 
\begin{equation}\label{eqn:intensity-tau23-intro}
 n^{\rho}\frac{A_{\floor{tn^\rho}}}{n^{2\rho}\frac{\mu^2}{\sum_{i\geq 1}\theta_i^2}}\left( 1+o(1)\right) +o(1)= \frac{\sum_{i\geq 1}\theta_i^2}{\mu^2}\refl{\bar{S}_n(t)}\left( 1+o(1)\right) +o(1).
\end{equation}
As in the $\tau>4$ and $\tau\in (3,4)$ cases, this gives the asymptotics for the surplus edges within components. 
Finally, as in Sections~\ref{sec:finite-third-moment},~\ref{sec:infinite-third-moment-intro}, to conclude that the largest component sizes and surplus edges converge to ordered excursion lengths of $\biS$, one needs to show that the largest components are explored before time $O(n^{\rho})$, and $\mathbf{Z}_n(\lambda)$ is tight in $\Unot$.
The reader is referred to Section~\ref{c5:sec:CM-proof} for the final details of this proof.

\subsection{Effect of the single-edge constraint}
\label{sec:infinite-second-ECM-intro}
In this section, we will consider two random graph models that do not allow for self-loops and multiple-edges: the generalized random graph $\GRG$ and the erased configuration model $\ECM$. 
The critical window for percolation is given by 
\begin{equation}\label{eq:scaling-window-tau23-single-intro}
p_c =p_c(\lambda):= \lambda n^{-\frac{3-\tau}{2}}(1+o(1)), \quad \lambda \in (0,\infty).
\end{equation} 
Note that the critical value in \eqref{eq:scaling-window-tau23-single-intro} is strictly larger in order than \eqref{eq:crit-value-tau23-intro}.
In fact, this is the only regime of $\tau$ where the 
exponent for the critical window changes after deleting the self-loops and multiple edges of $\CM$.

Let us first state the result for $\GRG$. 
We assume the following about the sequence of weights:
\begin{assumption}\label{assumption-GRG-intro}
\normalfont
For some $\tau \in (2,3)$, consider the distribution function satisfying $(1-F)(x) = C x^{-(\tau-1)}$ and let $w_i = (1-F)^{-1}(i/n)$. 
\end{assumption} 
\noindent In the above case, if $W_n$ denotes the weight of a vertex chosen uniformly at random from $[n]$, then 
\begin{equation}
\E[W_n] = \frac{1}{n}	\sum_{i\in [n]}w_i \to \mu = \E[W].
\end{equation}
Moreover, $$n^{-\alpha}w_i  = c_{\sss F}i^{-\alpha},$$ for some constant $c_F >0$.
Throughout $c_{\sss F}$ will denote the special constant appearing above.
Assumption~\ref{assumption-GRG-intro} is strictly stronger than Assumption~\ref{assumption-infinite-second} in the sense that Assumption~\ref{assumption-GRG-intro} specifies not only the high-degree vertices but all the $w_i$'s. 
This is required in the proofs as one needs precise estimates of quantities like $\E[W_n \ind{W_n\geq K_n}]$.
See Lemma~\ref{c5:lem:order-estimates} for many such required estimates. 

Let $\sCi(p)$ denote the $i$-th largest component of $\rGRG(\bw,p)$, and define $W_{\sss (i)}(p):= \sum_{k\in \sCi(p)} w_k$. 
We will consider the scaling limits of $(W_{\sss (i)}(p_c))_{i\geq 1}$ and $(\sCi(p_c))_{i\geq 1}$. 
To describe the limiting object, consider the graph $G_\infty(\lambda)$ on the vertex set $\Z_+$, where the vertices $i$ and $j$ are joined independently by Poisson$(\lambda_{ij})$ many edges with $\lambda_{ij}$ given by 
\begin{equation}\label{defn:lambda-ij-intro}
\lambda_{ij}:=\lambda^2\int_0^\infty \theta_i(x) \theta_j(x) \dif x, \quad \theta_i(x):= \frac{c_{\sss F}^2i^{-\alpha} x^{-\alpha}}{\mu+c_{\sss F}^2i^{-\alpha}x^{-\alpha}}.
\end{equation}
Let $W_{\sss (i)}^{\infty}(\lambda)$ denote the $i$-th largest element of the set $$\Big\{\sum_{i\in \sC}\theta_i: \sC \text{ is a connected component of }G_{\infty}(\lambda)\Big\}.$$ 
The following describes the component sizes of $\mathrm{GRG}_n(\bw,p_c(\lambda))$:
\begin{theorem}[Critical regime for $\GRG$] \label{thm:main-crit-tau23-single-intro}
There exists an absolute constant $\lambda_0$ such that for any $\lambda\in (0,\lambda_0)$,
under \textrm{Assumption~\ref{assumption-GRG-intro}}, as $n\to \infty$,
\begin{equation} \label{eq:weight-single-edge-intro}
n^{-\alpha} (W_{\sss (i)}(p_c(\lambda)))_{i\geq 1} \dto (W_{\sss (i)}^{\infty}(\lambda))_{i\geq 1},
\end{equation} and 
\begin{equation} \label{eq:compsize-single-edge-intro}
(n^{\alpha}p_c)^{-1} (|\sC_{\sss (i)}(p_c(\lambda))|)_{i\geq 1} \dto (W_{\sss (i)}^{\infty}(\lambda))_{i\geq 1},
\end{equation} with respect to the $\ell^2_{\shortarrow}$ topology.
\end{theorem}
For the erased configuration model, we will assume that $\bld{d}$ satisfies Assumption~\ref{assumption-GRG-intro}.
Since $\bld{d}$ can only take integer values, the support of $F$ is taken to be the set of non-negative integers.
The limiting object for $\ECM$ is similar to that in $\GRG$ by now taking
\begin{equation}\label{defn:lambda-ij-ECM-intro}
\theta_i(x):=1-\e^{-\frac{c_{\sss F}^2i^{-\alpha}x^{-\alpha}}{\mu}}.
\end{equation}
\begin{theorem}[Critical regime for $\ECM$]\label{thm:main-ECM-intro}
There exists an absolute constant $\lambda_0$ such that for any $\lambda\in (0,\lambda_0)$,
under  Assumption~\ref{assumption-GRG-intro},
the scaling limit results in Theorem~\ref{thm:main-crit-tau23-single-intro} holds for $\mathrm{ECM}_n(\bld{d},p_c(\lambda))$ with identical limit objects described by \eqref{defn:lambda-ij-ECM-intro} above.
\end{theorem}
Next, we state the result about the barely subcritical regime under the single-edge constraint. 
The following result holds for percolation on both $\GRG$ and $\ECM$:
\begin{theorem}\label{thm:barely-subcrit-single-edge}
Suppose that \textrm{Assumption~\ref{assumption-GRG-intro}} holds and $p_n \ll p_c(\lambda)$. 
Then, for any fixed $i\geq 1$, as $n\to\infty$, 
\begin{equation}
\frac{|\sC_{\sss (i)}(p_n)|}{n^{\alpha}p_n} \pto c_{\sss F}i^{-\alpha},\quad \text{and}\quad \frac{W_{\sss (i)}(p_n)}{n^{\alpha}} \pto c_{\sss F}i^{-\alpha}.
\end{equation}
\end{theorem}
\noindent 
Under the single-edge constraint, the exact asymptotics in the barely supercritical case is left to future work. 
In the proofs under the single-edge constraint, coming up with a tractable exploration process for the clusters seems challenging.
The only tool we have is an estimate of the connection probabilities of hubs via an intermediate vertex, which allows us to estimate expectations of several moments of component sizes and total weights of those component. 
These are often referred to as susceptibility functions. 
The susceptibility functions allow us to ignore negligible contributions on the total weights of cluster using the first-moment method. 
Unfortunately, the first-moment method does not work in the barely-supercritical regime, or for high values of $\lambda$ in Theorems~\ref{thm:main-crit-tau23-single-intro},~\ref{thm:main-ECM-intro}. 
This is the reason for assuming $\lambda \in (0,\lambda_0)$ in those theorems. 
The proof for general $\lambda$ is an open question.

The critical window changes due to the single-edge constraint, as noted in \eqref{eq:crit-value-tau23-intro} and \eqref{eq:scaling-window-tau23-single-intro}.
However, in both cases, the critical window is the regime where hubs start getting connected. 
More precisely, the critical window is given by those values of $p$ such that for any fixed $i, j\geq 1$
\begin{eq}
\lim_{n\to\infty} \PR(i, j\text{ are in the same component of the }p \text{-percolated graph} ) \in (0,1).
\end{eq}
For the configuration model, hubs are connected directly with strictly positive probability.
In $\CM$, vertices $i$ and $j$ share $d_id_j/(\ell_n-1)$ edges in expectation.
Thus for hubs with $d_i=  O(n^{\alpha})$ and $d_j = O(n^{\alpha})$, $O(1)$ many edges survive after percolation in expectation in the critical window \eqref{eq:crit-value-tau23-intro}.
On the other hand, whenever $p \to 0$, hubs are never connected directly under the single-edge constraint.
We will see in Chapter~\ref{chap:infinite-second} and in the heuristic arguments below that the value $p_c$ in \eqref{eq:scaling-window-tau23-single-intro} is such that the hubs are connected to each other via intermediate vertices which have degree $\Theta(n^{\rho})$.
Intuitively, in the barely subcritical regime, all the hubs are in different components.
Hubs start forming the critical components as $p$ varies over the critical window, and finally in the barely super-critical regime the giant component is formed which contains all the hubs. 
This feature is also observed in the $\tau \in (3,4)$ case \cite{BHL12}.

In the next section, we only outline the proof of Theorem~\ref{thm:main-crit-tau23-single-intro}. 
The proof of Theorem~\ref{thm:main-ECM-intro} uses similar arguments, but additional complications arise due to the dependence between occurrence of edges. 
We refer the reader to Section~\ref{c5:sec:proof-GRG} for rigorous derivations of the above results.

\subsubsection*{Proof ideas for Theorem~\ref{thm:main-crit-tau23-single-intro}.}
The key idea of the proof is to first consider total weights of components. 
In this section, we will use the notation $C$ as a generic notation for a positive constant that may only depend on $F$.
Let $\sC(i)$ denote the component containing vertex $i$ and $W(i) = \sum_{v\in \sC(i)}w_v$. 
Further, let $W_k(i):= \sum_{v\in \sC(i),\ \dst(i,v) = k}w_v$, where $\dst(\cdot,\cdot)$ denotes the graph-distance.
The idea is to show that the primary contribution to $W(i)$ comes from vertices at finite, even distance.
This is because hubs are not connected directly, but via intermediate vertices.
To identify negligible contributions to $W(i)$, we use the first-moment method.
Note that $\E[W_k(i)] \leq \sum_{j \in [n]} w_j f_k(i,j)$, where $f_k(i,j)$ denotes the probability that there is a path of length $k$ from $i$ to~$j$. 
The key ingredient in our proof is the following lemma which allows us to compute $f_k(i,j)$.
\begin{lemma}[Two-hop connection probabilities]\label{lem:2-nbd-prob-intro}
For all $n\geq 1$, and $i,j\in [n]$, 
\begin{equation}\label{2hop-path-prob-pref-attachment-n-intro}
p_{ij}(2):= p_c^2 \sum_{v\in [n]} \frac{w_iw_v^2w_j}{(\ell_n+w_iw_v)(\ell_n+w_jw_v)} \leq \frac{C\lambda^2}{(i\wedge j)^{1-\alpha} (i\vee j)^{\alpha}}.
\end{equation}
\end{lemma}
Note that $p_{ij}(2)$ is the expected number of connections between $i$ and $j$ via an intermediate vertex. 
The upper bound in \eqref{2hop-path-prob-pref-attachment-n-intro} is exactly the same as the connection probabilities in a preferential attachment model \cite{DHH10,BR04} (See \cite[Lemma 2.2]{DHH10}).
Therefore, existing path-counting estimates for the preferential attachment model \cite[Lemma 2.4]{DHH10} yield, for $1-\alpha<b<\alpha$,
\begin{equation}\label{eq:geometric-decay-intro}
f_{2k}(i,j) \leq \frac{(C\lambda^2)^{k}}{(i\wedge j)^b(i\vee j)^{1-b}}.
\end{equation}
The geometric bound in \eqref{eq:geometric-decay-intro} gives
\begin{eq}\label{eq:path-count-geom-intro}
\E[W_{2k}(i)] &\leq \sum_{j\in [n]}w_{j} f_{2k}(i,j) \\
&\leq n^{\alpha}(c_0\lambda^2)^k \bigg[\sum_{j<i} \frac{1}{j^{\alpha}} \frac{1}{j^{b}i^{1-b}} + \sum_{j>i} \frac{1}{j^{\alpha}} \frac{1}{i^{b}j^{1-b}}\bigg] \\
&\leq C n^\alpha (c_0\lambda^2)^k \Big[\frac{1}{i^{1-b}} + \frac{1}{i^{\alpha}}\Big] \leq C (c_0\lambda^2)^k\frac{n^{\alpha}}{i^{1-b}}\\
&\leq C (c_0\lambda^2)^k w_i i^{b- (1-\alpha)}.
\end{eq}
The final term decays geometrically with $k$ when $c_0\lambda^2<1$.
This is the precise reason why the condition $\lambda\in (0,\lambda_0)$ is needed in Theorem~\ref{thm:main-crit-tau23-single-intro}.
Suitable upper bounds on $\E[W_{2k+1}(i)]$ can also be obtained using \eqref{eq:geometric-decay-intro}. 
Thus the next proposition follows using the first-moment method:
\begin{proposition} \label{prop:tail-particular-component-intro}
For any fixed $i\geq 1$, $\varepsilon > 0$, and $\lambda\in (0,\lambda_0)$
\begin{gather*}
\lim_{K\to\infty}\limsup_{n\to\infty}\PR\bigg(\sum_{k>K}W_{2k}(i) > \varepsilon n^{\alpha}\bigg) = 0,  \\ \lim_{n\to\infty} \PR\bigg(\sum_{k=0}^\infty W_{2k+1}(i) > \varepsilon n^{\alpha} \bigg) = 0.
\end{gather*}
\end{proposition}
Hence, the primary contribution to $W(i)$ comes from  weights of vertices at finite, even distance.
The next proposition goes one step further and says that, even among the vertices at finite, even distance, the major contribution comes from the hubs.
For $\delta>0$, define $V_{L}(\delta) := \{v:w_v>\delta n^{\alpha}\}$, and  $W_{k}(i,\delta):= \sum_{v\notin V_L(\delta), \dst(v,i)=k}w_v$. 
\begin{proposition}\label{prop-non-hub-contribution-K-nbd-intro}
For any fixed $i\geq 1$, $K\geq 1$, and $\varepsilon>0$, 
\begin{equation}
\lim_{\delta \to 0}\limsup_{n\to\infty}\PR\bigg(\sum_{k=1}^K W_{2k}(i,\delta)>\varepsilon n^{\alpha} \bigg) =0.
\end{equation}
\end{proposition}
Combining Propositions~\ref{prop:tail-particular-component-intro},~\ref{prop-non-hub-contribution-K-nbd-intro}, obtaining asymptotics of $W(i)$ now boils down to identifying the hubs which are in $\sC(i)$.
From \eqref{2hop-path-prob-pref-attachment-n-intro}, hubs are connected to each other via some intermediate vertices with probability bounded away from zero. 
Let $X_{ij}$ denote the number of paths of length 2 from $i$ to $j$.
Note that, for any $i,j\in V_L(\delta)$ ($i\neq j$),
\begin{eq}\label{eq:distn-Xij-intro}
X_{ij} = \sum_{v\neq i,j} \ber\bigg(\frac{w_iw_jw_v^2p_c^2}{(\ell_n+w_iw_v)(\ell_n+w_jw_v)}\bigg),
\end{eq}with the different Bernoulli random variables in the sum \eqref{eq:distn-Xij-intro} being independent.
Now, the primary contribution in the sum \eqref{eq:distn-Xij-intro} comes from vertices with weight $\Theta(n^{\rho})$.
In fact using some estimates of the moments of $\bld{w}$, we can show that 
\begin{gather*}
\sum_{v: w_v<\delta n^{\rho}}\ber\bigg(\frac{w_iw_jw_v^2p_c^2}{(\ell_n+w_iw_v)(\ell_n+w_jw_v)}\bigg) \leq C\delta^{3-\tau}, \\ \sum_{v: w_v>\delta^{-1} n^{\rho}}\ber\bigg(\frac{w_iw_jw_v^2p_c^2}{(\ell_n+w_iw_v)(\ell_n+w_jw_v)}\bigg) \leq C\delta^{\tau-1},
\end{gather*}
%
see Section~\ref{c5:sec:hub-structure}. Thus,
\begin{eq}\label{eq:Xij-approximation-intro}
X_{ij} &= \sum_{v: w_v\in [\delta n^{\rho},\delta^{-1}n^{\rho}]} \ber\bigg(\frac{w_iw_jw_v^2p_c^2}{(\ell_n+w_iw_v)(\ell_n+w_jw_v)}\bigg) + E(\delta,n) \\
&= X_{ij}(\delta) + E(\delta,n),
\end{eq} where for any $\varepsilon>0$ 
$\lim_{\delta \to 0}\limsup_{n\to\infty}\PR(E(\delta,n) >\varepsilon) = 0.$
Define 
\begin{eq}
\lambda_{ij}(\delta) = \sum_{v: w_v\in [\delta n^{\rho},\delta^{-1}n^{\rho}]} \frac{w_iw_jw_v^2p_c^2}{(\ell_n+w_iw_v)(\ell_n+w_jw_v)} .
\end{eq}
Thus, the above term can be approximated by a Poisson random variable using Stein's method and 
\begin{equation}
\lim_{\delta\to 0} \lim_{n\to\infty} \lambda_{ij}(\delta) = \lambda^2\int_0^\infty \theta_i(x) \theta_j(x) \dif x, \quad \theta_i(x):= \frac{c_{\sss F}^2i^{-\alpha} x^{-\alpha}}{\mu+ c_{\sss F}^2 i^{-\alpha}x^{-\alpha}}.
\end{equation}
Recall the description of the graph $G_{\infty}(\lambda)$ from Theorem~\ref{thm:main-crit-tau23-single-intro}.
The above proves that the limit of $n^{-\alpha}W(i)$ is basically $\sum_{v\in C(i)}v^{-\alpha}$, $C(i)$ being the connected component containing $i$ in $G_{\infty}(\lambda)$.
$|\sC(i)| = p_c(\lambda)W(i)(1+\oP(1))$ is proved using the second-moment method.
Finally, to prove the scaling limit of the ordered vector of component sizes and weights, we show that the vectors are tight in $\ell^2_{\shortarrow}$ in Section~\ref{c5:sec:l2-tightness}. 
This completes the sketch of the proof of Theorem~\ref{thm:main-crit-tau23-single-intro}.

\section{Summary of contributions}
In summary, we analyze the critical window for the percolation process on random graph models such as $\CM$, $\UM$, $\ECM$ and $\GRG$.
When the degree distribution satisfies a power law with exponent $\tau$, 
three universality classes arise depending on whether $\tau > 4$ (finite third moment), $\tau\in (3,4)$ (infinite third moment) and $\tau \in (2,3)$ (infinite second moment). 
Let us summarize the main contributions of this thesis below: 
\paragraph*{Component sizes and complexity for finite third-moment case.}
In Chapter~\ref{chap:thirdmoment}, we obtain precise asymptotics for the component sizes and the surplus edges for  $\mathrm{CM}_n(\boldsymbol{d},p_c(\lambda))$ and $\mathrm{UM}_n(\boldsymbol{d},p_c(\lambda))$ in the \emph{critical window} of the phase transition under a finite third-moment condition. 
The main contribution of this work is that we derive the strongest scaling limit results in the literature under optimal assumptions. 
This finite third-moment assumption is also necessary for Erd\H{o}s-R\'enyi type scaling limits, since, amongst other reasons, the third-moment appears in the scaling limit. 
Also, we prove the joint convergence of the component sizes and the surplus edges under a strong topology namely the $\Unot$-topology, which improves the previously known results \cite{R12,Jo10,NP10b} substantially. 
The re-scaled vector of component sizes (ordered in a decreasing manner) is shown to converge to the ordered excursion lengths of a reflected inhomogeneous Brownian motion with a negative parabolic drift.
Moreover, the surplus edges converge jointly with the component sizes under $\Unot$-topology to Poisson random variables with parameters being the areas under the above mentioned ordered excursion lengths. 

\paragraph*{Component sizes and complexity for infinite third-moment case.}
In Chapter~\ref{chap:secondmoment}, we consider the critical behavior for the component sizes and surplus edges in the infinite third-moment case. 
We consider a general set of assumptions, which include the case that the empirical degree distribution satisfies $\prob{D_n\geq k}\sim L_0(k)/k^{\tau-1}$ for some $\tau\in (3,4)$ and $L_0(\cdot)$ a slowly-varying function.
The largest connected components turn out to be of the order $n^{(\tau-2)/(\tau-1)}L(n)^{-1}$ and the width of the scaling window is of the order $n^{(\tau-3)/(\tau-1)}L(n)^{-2}$ for some slowly-varying function $L(\cdot)$. 
 The joint distribution of the re-scaled component sizes and the surplus edges is shown to converge in distribution to a suitable limiting random vector under $\Unot$-topology.
The scaling limits for the re-scaled ordered component sizes can be described in terms of the ordered excursions of a certain thinned L\'evy process that only depends on the asymptotics of the \emph{high}-degree vertices.
This universality class was first identified in \cite{BHL12} in the context of Norros-Reittu random graphs. 
Further, the scaling limits for the surplus edges can be described by Poisson random variables with the parameters being the areas under the excursions of the thinned L\'evy process. 
The results also hold conditioned on the graph being simple, thus solving an open question \cite[Conjecture 8.5]{Jo10}.

\paragraph*{Evolution of components and surplus edges.}
As  $\lambda$ increases over the critical window, the component sizes and surplus edges jointly evolve, with components merging with each other, and more surplus edges getting created.
In Chapters~\ref{chap:thirdmoment} and~\ref{chap:secondmoment}, the evolution of the component sizes and surplus edges is shown to converge to a version of the \emph{augmented multiplicative coalescent} process both in the finite third-moment and infinite third-moment regimes.
In fact, in the $\tau\in (3,4)$ case, our results imply that there exists a version of the augmented multiplicative coalescent process whose one-dimensional distribution can be described by the excursions of a thinned L\'evy process and a Poisson process with the intensity being proportional to the thinned L\'evy process, which is also novel.

\paragraph*{Metric structure of critical components in the infinite third-moment case.}
In Chapter~\ref{chap:mspace}, we consider the metric structure of the critical components for $\mathrm{CM}_n(d,p_c(\lambda))$, with degree-exponent $\tau\in (3,4)$.
In this context, candidate limit law of maximal components with each edge rescaled to have length $1/n^{(\tau-3)/(\tau - 1)}$ was established recently in \cite{BHS15}. 
In this work, we establish sufficient uniform asymptotic negligibility (UAN) conditions for a random graph model in the barely subcritical regime which, in combination with the appropriate merging dynamics of components as one increases edge density through the critical regime, implies convergence to limits obtained in~\cite{BHS15}. 
This result identifies the domain of attraction for the limit laws established in \cite{BHS15}, which holds for general sequences of dynamically evolving graphs.
 As a canonical example, we analyze the critical regime for percolation on the uniform random graph model (and the closely associated configuration model) with prescribed degree distribution that converges to a heavy-tailed degree distribution. 
 In order to carry out the above analysis and in particular check the UAN assumptions, we establish refined bounds for various susceptibility functionals and diameter in the barely subcritical regime of the configuration model which are of independent interest.
 In Chapter~\ref{chap:mspace-GHP}, we prove the global lower mass-bound property for these critical components, which proves the convergence of the largest components with respect to the Gromov-Hausdorff-Prokhorov topology.
 The latter yields the compactness of the scaling limit in Chapter~\ref{chap:mspace}, as well as scaling limits of global functionals like diameter.

\paragraph{Component sizes in the infinite second-moment case.}
In Chapter~\ref{chap:infinite-second}, we consider a new universality class which corresponds to the degree exponent $\tau\in (2,3)$. 
In this regime, the critical behavior is observed when the percolation probability tends to zero with the network size. 
We identify the critical window for the configuration model, the erased configuration model and the generalized random graph. 
The critical window for graphs with single-edges is given by $p_c \sim \lambda n^{-(3-\tau)/2}$, which is much larger than $p_c \sim n^{-(3-\tau)/(\tau-1)} $ for the multigraph $\CM$. 
This feature is unique to the critical behavior in the $\tau \in(2,3)$ regime.
The component sizes in both cases scale as $n^{\alpha} p_c$. 
 For the configuration model multigraph, we obtain scaling limits for the largest component sizes and surplus edges under a strong topology. 
Further, the diameter of the largest components is shown to be a tight sequence of random variables. 
To establish that the scaling limits correspond to the critical behavior, we further look at the near-critical behavior and derive the asymptotics for the component sizes in the so-called barely sub/supercritical regimes.
On the other hand, under the single-edge constraint, we identify the scaling limit of the largest component sizes in the part of the critical window, where the criticality parameter is sufficiently small. 
The proof where the criticality parameter can be arbitrary is an ongoing work.

 This is the first work on critical percolation on random graphs in the $\tau\in (2,3)$ setting, thus the techniques are novel. 
The primary difficulty in this setting is that the exploration process approach does not work. 
For the configuration model, this difficulty is circumvented by sandwiching the percolated graphs by two configuration models, which yield the same scaling limits for the component sizes. 
The main novelty in the proof of the configuration model is the analysis of the limiting exploration process. 
 On the other hand, in the single-edge constraint scenario, the proofs 
require a more detailed understanding of the structure of the critical components. 
It turns out that the hubs do not connect to each other directly, but there are some special vertices that interconnect hubs. 
This interconnected structure forms the \emph{core} of the critical components, and the 1-neighborhood of the core spans the critical components.
We primarily use path counting techniques here since the exploration process approach does not seem to work anymore. 
For path-counting, we compare the connection probabilities between the hubs to the connection probabilities in a preferential attachment model, which is interesting in its own right.

%
%
%
%

\chapter[Critical window: Finite third moment]{Critical window for the configuration model: finite third moment degrees}
\label{chap:thirdmoment}
{\small \paragraph*{Abstract.}
    We investigate the component sizes of the critical configuration model, as well as the related problem of critical percolation  on a supercritical configuration model.
 We show that, at criticality, the finite third moment assumption on the asymptotic degree distribution is enough to guarantee that the component sizes are $O(n^{2/3})$ and the re-scaled component sizes converge to the excursions of an inhomogeneous Brownian Motion with a parabolic drift. 
This identifies the minimal condition for the critical behavior to be in the Erd\H{o}s-R\'enyi universality class. 
 We use percolation to study the evolution of these component sizes while passing through the critical window and  show that the vector of percolation cluster-sizes, considered as a process in the critical window, converge to the multiplicative coalescent process in finite dimensions. This behavior was first observed for Erd\H{o}s-R\'enyi random graphs by Aldous (1997) and our results provide support for the empirical evidences that the nature of the phase transition for a wide array of random-graphs are universal in nature. 
 Further, we show that the re-scaled component sizes and surplus edges converge jointly under a strong topology, at each fixed location of the scaling window.} 
\vspace{.3cm}

\noindent {\footnotesize Based on the manuscript: Souvik Dhara, Remco van der Hofstad, Johan S.H. van Leeuwaarden, and Sanchayan Sen, \emph{Critical window for the configuration model: finite third moment degrees} (2016), Electronic Journal of Probability 22, no.~16, 1--33
}
\vfill

In this chapter, we focus on the critical behavior of the \emph{configuration model}, and critical percolation on these graphs when the empirical degree distribution satisfies a finite third-moment condition. 
We include detailed proofs of all scaling limit results mentioned in Chapter~\ref{chap:introduction} about the finite third-moment case.
The scaling limit result for the component sizes and surplus edges holds for the critical configuration model which includes critical percolation on the configuration model as a special case.
We also study percolation on a super-critical configuration model to show that the scaled vectors of component sizes at multiple locations of the percolation scaling window converge jointly to the finite-dimensional distributions of a multiplicative coalescent process. 
The scaling limit results show that component sizes and surplus edges of $\mathrm{CM}_n(\boldsymbol{d})$ in the critical regime, for a large collection of possible degree sequences~$\boldsymbol{d}$, lie in the same universality class as for the Erd\H{o}s-R\'enyi random graph~\cite{A97}.
Before stating the main results, we need to introduce some notation and concepts.

\section{Definitions and notation} \label{c1:sub_sec_notation}
Recall the definitions from Chapter~\ref{sec:notation-intro}.
For a triangular array of random variables $(f_{k,n})_{k,n\geq 1}$, we write phrases like $f_{k,n} = \OP(n^{\alpha})$ (respectively  $\oP(n^{\alpha})$), uniformly over $k\leq n^{\beta}$ to mean that $\sup_{k\leq n^{\alpha}}|f_{k,n}| = \OP(n^{\alpha})$ (respectively $\oP(n^{\alpha})$).
We also write $f_n=O_{\sss E}(a_n)$ (respectively $f_n=o_{\sss E}(a_n)$) to denote that $\sup_{n\geq 1}\expt{a_n^{-1}f_n}<\infty$ (respectively $\lim_{n\to\infty}\expt{a_n^{-1}f_n}=0$).

In this chapter, $\mathbf{B}^{\lambda}_{\mu,\eta}$ denotes an inhomogeneous Brownian motion with a parabolic drift, given by
\begin{equation}\label{c1:def:inhomogen:BM}
B^{\lambda}_{\mu,\eta}(s)=\frac{\sqrt{\eta}}{\mu} B(s) +\lambda s-\frac{\eta s^2}{2\mu^{3}},
\end{equation}
 where $\mathbf{B}= ( B(s) )_{s \geq 0}$ is a standard Brownian motion, and $\mu>0$, $\eta>0$ and $\lambda\in \R$ are constants. Define the reflected version of $\mathbf{B}^{\lambda}_{\mu,\eta}$ as
\begin{equation} \label{c1:defn::reflected-BM}
W^{\lambda}(s) = B^{\lambda}_{\mu,\eta}(s) - \min_{0 \leq t \leq s} B^{\lambda}_{\mu,\eta}(t).
\end{equation}
For a function $f\in \mathbb{C}[0,\infty)$, an interval $\gamma=(l,r)$ is called an \emph{excursion above past minima} or simply an \emph{excursion} of $f$ if $f(l)=f(r)=\min_{u\leq r}f(u)$ and $f(x)>f(r)$ for all $l<x<r$.   $|\gamma|=r(\gamma)-l(\gamma)$ will denote the length of the excursion $\gamma$.
\par Also, define the counting process of marks $\mathbf{N}^\lambda= ( N^\lambda(s) )_{s \geq 0}$ to be a unit-jump process with intensity $\beta W^\lambda(s)$ at time $s$ conditional on $( W^\lambda(u) )_{u \leq s}$ so that
\begin{equation} \label{c1:defn::counting-process}
N^\lambda(s) - \int\limits_{0}^{s} \beta W^\lambda(u)du
\end{equation} is a martingale (see \citep{A97}). 
For an excursion $\gamma$, let $N(\gamma)$ denote the number of marks in the interval $[l(\gamma),r(\gamma)]$.
\begin{remark} \label{c1:defn_U_0_process}\normalfont By \cite[Lemma 25]{A97}, the excursion lengths  of $\mathbf{B}^{\lambda}_{\mu,\eta}$ can be rearranged in decreasing order of length and the ordered excursion lengths can be considered as a vector in $\ell^2_{\shortarrow}$, almost surely. Let $\boldsymbol{\gamma}^\lambda = ( | \gamma^\lambda_{j}|)_{j\geq1}$ be the ordered excursion lengths of $\mathbf{B}^{\lambda}_{\mu,\eta}$. Then,  $( | \gamma_j^\lambda | , N(\gamma_j^\lambda))_{ j \geq 1} $ can be ordered as an element of  $\mathbb{U}^0_{\shortarrow}$ almost surely by \cite[Theorem 3.1~(iii)]{BBW12}. We denote this element of $\mathbb{U}^0_{\shortarrow}$ by $\mathbf{Z}(\lambda)= ((Y_j^\lambda,N_j^\lambda))_{j\geq 1}$ obtained from $(| \gamma_j^\lambda \big| , N(\gamma_j^\lambda))_{ j \geq 1} $.
\end{remark}
Finally, we define a Markov process $\mathbb{X}:=(\mathbf{X}(s))_{s\in\R}$ on $\mathbb{D}(\R,\ell^2_{\shortarrow})$, called the \emph{multiplicative coalescent process}. 
Think of $\mathbf{X}(s)$ as a collection of masses of some particles (possibly infinite) in a system at time $s$. 
Thus the $i^{th}$ particle has mass $X_i(s)$ at time~$s$. 
The evolution of the system takes place according to the following rule at time $s$: At rate $X_i(s)X_j(s)$,  particles $i$ and $j$ merge into a new particle of mass $X_i(s)+X_j(s)$.
This process has been extensively studied in \cite{A97,AL98}. In particular, Aldous~\cite[Proposition 5]{A97} showed that this is a Feller process.

\section{Main results}
\label{c1:sec_results}
 In this section, we discuss the main results in this chapter. 
 We start by recalling the definition of the configuration model from Chapter~\ref{sec:intro-models}, which is denoted by $\CM$.  
 Our results are twofold and concern (i) general $\mathrm{CM}_n(\boldsymbol{d})$ at criticality, and (ii) critical percolation on a super-critical configuration model, both under a finite third moment assumption.
\subsection{Configuration model results}
For each $n\geq 1$, let $\bld{d}=\bld{d}_n=(d_i)_{i\in [n]}$ be a degree sequence such that $\ell_n = \sum_{i\in [n]}d_i$ is even. 
We suppress $n$ in the notation of the degree sequence to simplify writing.
 We consider a sequence of configuration models $(\mathrm{CM}_n(\boldsymbol{d}))_{n\geq 1}$ satisfying the following conditions:
 \begin{assumption} \label{c1:assumption1}
 \normalfont Let $D_n$ denote the degree of a vertex chosen uniformly at random independently of the graph. Then the following holds as $n\to\infty$:
 \begin{enumerate}[(i)]
  \item \label{c1:assumption1-1}(\emph{Weak convergence of $D_n$}) $D_n \dto D$ for some random variable $D$ such that $\mathbb{E}[D^3] < \infty $.
 \item \label{c1:assumption1-2}(\emph{Uniform integrability of $D_n^3$})
    $\expt{D_n^3}= \frac{1}{n}\sum_{i\in [n]}d_{i}^{3} \to \mathbb{E}\big[ D^{3}\big].$
  \item \label{c1:assumption1-3}(\emph{Critical window}) 
   $\nu_{n}:= \frac{\sum_{i\in [n]}d_i(d_i-1)}{\sum_{i\in [n]}d_i} =1+\lambda n^{-1/3}+o(n^{-1/3}),$
  for some constant $\lambda\in \mathbb{R}$,
  \item \label{c1:assumption1-4}$\prob{D=1}>0$.
 \end{enumerate}
 \end{assumption}
  Suppose that $\mathscr{C}_{\sss(1)}$, $\mathscr{C}_{\sss(2)}$,... are the connected components of $\mathrm{CM}_{n}(\boldsymbol{d})$ in decreasing order of size. 
  In case of a tie, order the components according to the values of the minimum indices of vertices in those components. For a connected graph $G$, let $\surp{G}$:= $($number of edges in $G) - (|G| - 1)$ denote the number of surplus edges. Intuitively, this measures the deviation of $G$ from a tree-like structure. Let  $\sigma_{r}= \expt{D^r}$ and consider the reflected Brownian motion, the excursions, and the counting process $\mathbf{N}^\lambda$ as defined in Section \ref{c1:sub_sec_notation} with parameters
 \begin{equation}\label{c1:parameter}
 \mu:=\sigma_1, \quad \eta:= \sigma_{3} \mu - \sigma_{2}^{2},\quad \beta := 1/ \mu.
 \end{equation} Let $\boldsymbol{\gamma}^{\lambda}$ denote the vector of excursion lengths of the process $\mathbf{B}^\lambda_{\mu,\eta}$, arranged in non-increasing order. 
 The next two theorems are our main results for the critical configuration model:
 \begin{theorem}
 \label{c1:thm_main} Fix any $\lambda\in \mathbb{R}$. Under Assumption~\ref{c1:assumption1},
  \begin{equation}
   n^{-2/3}\big( | \mathscr{C}_{\sss (j)} | \big)_{j\geq 1} \xrightarrow{\mathcal{L}} \boldsymbol{\gamma}^\lambda
  \end{equation}
  with respect to the $\ell^2_{\shortarrow}$ topology.
 \end{theorem}
 Recall the definition of $\mathbf{Z}(\lambda)$ from Remark~\ref{c1:defn_U_0_process}. Order the vector component sizes and surplus edges $ \big( n^{-2/3}\big| \mathscr{C}_{\sss (j)} \big|, \surp{\mathscr{C}_{\sss (j)}} \big)_{j\geq 1}$ as an element of $\mathbb{U}^0_{\shortarrow}$ and denote it by $\mathbf{Z}_n(\lambda)$.
 \begin{theorem} \label{c1:thm_surplus}
  Fix any $\lambda\in \mathbb{R}$. Under Assumption~\ref{c1:assumption1},
  \begin{equation} \label{c1:eqn_thm_surplus}
  \mathbf{Z}_n(\lambda) \xrightarrow{\mathcal{L}} \mathbf{Z}(\lambda)
  \end{equation} with respect to the $\mathbb{U}^0_{\shortarrow}$ topology.
 \end{theorem}
In words, Theorem~\ref{c1:thm_main} gives the precise asymptotic distribution of the component sizes re-scaled by $n^{2/3}$ and Theorem~\ref{c1:thm_surplus} gives the asymptotic number of surplus edges in each component jointly with their sizes.

\begin{remark}\normalfont The strength of Theorems~\ref{c1:thm_main} and~\ref{c1:thm_surplus} lies in Assumption~\ref{c1:assumption1}. Clearly, Assumption~\ref{c1:assumption1} is satisfied when the distribution of $D$ satisfies an asymptotic power-law relation with finite third moment, i.e., $\PR(D\geq x)\sim x^{-(\tau-1)}(1+o(1))$ for some $\tau >4$. Also, if a random degree-sequence satisfies Assumption~\ref{c1:assumption1} with high probability, then Theorems~\ref{c1:thm_main} and~\ref{c1:thm_surplus} hold conditionally on the degrees. In particular, when the degree sequence consists of an i.i.d sample from a distribution with $\E[D^3]<\infty$ \cite{Jo10}, then Assumption~\ref{c1:assumption1} is satisfied almost surely. We will later see that  degree sequences in the percolation scaling window also satisfy Assumption~\ref{c1:assumption1}. 
\end{remark}

\subsection{Percolation results}
Bond percolation on a graph $G$ refers to deleting edges of $G$ independently with equal probability~$p$. 
In the case $G$ is a random graph, the deletion of edges is also independent of $G$.
 Consider bond percolation on $\mathrm{CM}_n(\boldsymbol{d})$ with probability $p_n$, yielding  $\mathrm{CM}_n(\boldsymbol{d}, p_n)$. 
 We assume the following:
\begin{assumption} \label{c1:assumption2} \normalfont
 \begin{enumerate}[(i)]
  \item \label{c1:assumption2-1} Assumption~\ref{c1:assumption1}~\eqref{c1:assumption1-1} and \eqref{c1:assumption1-2}  hold for the degree sequence and  $\mathrm{CM}_n(\boldsymbol{d})$ is super-critical, i.e.
  \begin{equation}
   \nu_n=\frac{\sum_{i\in [n]}d_i(d_i-1)}{\sum_{i\in [n]}d_i}\to \nu= \frac{\expt{D(D-1)}}{\expt{D}} >1.
  \end{equation}
  \item \label{c1:assumption2-2}(\emph{Critical window for percolation}) For some $\lambda\in \mathbb{R}$,
  \begin{equation}
  p_{n}=p_n(\lambda):=\frac{1}{\nu_n} \bigg( 1+ \frac{\lambda}{n^{1/3} }\bigg).
  \end{equation}
  \end{enumerate}
 \end{assumption}
Note that $p_n(\lambda)$, as defined in Assumption~\ref{c1:assumption2}~\ref{c1:assumption2-2}, is always non-negative for $n$ sufficiently large. 
Now, suppose $\tilde{d}_i\sim \mathrm{Bin}(d_i,\sqrt{p_n})$, $n_+:=\sum_{i\in [n]}(d_i-\tilde{d}_i)$ and $\tilde{n}=n+n_+$.  Consider the degree sequence $\Mtilde{\boldsymbol{d}}$ consisting of $\tilde{d}_i$ for $i\in [n]$ and $n_+$ additional vertices of degree 1, i.e. $\tilde{d}_i=1$ for $i\in [\tilde{n}]\setminus [n]$. 
We will show later that the degree $\tilde{D}_n$ of a random vertex from this degree sequence satisfies Assumption~\ref{c1:assumption1}~\eqref{c1:assumption1-1},~\eqref{c1:assumption1-2} almost surely for some random variable $\tilde{D}$ with $\mathbb{E}[\tilde{D}^3]<\infty$.
Moreover, $\tilde{n}/n\to 1+\mu(1-\nu^{-1/2}) = \zeta$ almost surely. 
Now, using the notation in Section~\ref{c1:sub_sec_notation}, define $\tilde{\gamma}_{j}^\lambda = \zeta^{2/3} \bar{\gamma}_j^\lambda$, where $\bar{\gamma}_j^\lambda$ is the $j^{th}$ largest excursion of the inhomogeneous Brownian motion $\mathbf{B}^\lambda_{\mu,\eta}$ with the parameters 
\begin{equation}\label{c1:value-parameters}
\mu=\mathbb{E}[\tilde{D}],\quad \eta=\mathbb{E}[\tilde{D}^3]\mathbb{E}[\tilde{D}]-\mathbb{E}^2[\tilde{D}^2], \quad \beta=1/\mathbb{E}[\tilde{D}].
\end{equation} Define the process $\tilde{\mathbf{N}}$ as in \eqref{c1:defn::counting-process}   with the parameter values given by \eqref{c1:value-parameters}. Denote  the $j^{th}$ largest cluster of $\mathrm{CM}_{n}(\boldsymbol{d},p_n(\lambda))$ by  $\mathscr{C}^{p}_{\sss (j)}(\lambda)$. 
Also, let $\mathbf{Z}_n^p(\lambda)$ denote the vector in $\mathbb{U}^0_{\shortarrow}$ obtained by rearranging critical percolation clusters (re-scaled by $n^{2/3}$) and their surplus edges and $\tilde{\mathbf{Z}}(\lambda)$ denote the vector in $\mathbb{U}^0_{\shortarrow}$ obtained by rearranging $( (\sqrt{\nu}| \tilde{\gamma}_j^\lambda | , \tilde{N}(\tilde{\gamma}^\lambda_j)))_{ j\geq 1} $.
    \begin{theorem} \label{c1:thm_percolation}  Under Assumption \ref{c1:assumption2},
    \begin{equation} \label{c1:eqn_thm_percolation}
   \mathbf{Z}_n^p(\lambda) \xrightarrow{\mathcal{L}}  \tilde{\mathbf{Z}}(\lambda)
  \end{equation}  with respect to the $\mathbb{U}^0_{\shortarrow}$ topology.
 \end{theorem}
 Next we consider the percolation clusters for multiple values of $\lambda$. 
 There is a very natural way to couple $(\mathrm{CM}_n(\boldsymbol{d},p_n(\lambda))_{\lambda\in \R}$ described as follows:
Suppose that each edge $(ij)$ of $\mathrm{CM}_n(\boldsymbol{d})$ has an associated i.i.d uniform random variable $U_{ij}$, and the $U_{ij}$'s are also independent of $\mathrm{CM}_n(\boldsymbol{d})$. Now, delete edge $(ij)$ if $U_{ij}>p_{n}(\lambda)$. The obtained graph is distributed as $\mathrm{CM}_n(\boldsymbol{d},p_n(\lambda))$. 
Moreover, if we fix the set of uniform random variables and change $\lambda$, this produces a coupling between the graphs $(\mathrm{CM}_n(\boldsymbol{d},p_n(\lambda))_{\lambda\in \R}$.
 The next theorem shows that the convergence of the component sizes holds jointly in finitely many locations within the critical window, under the above described coupling:
 \begin{theorem} \label{c1:thm_multiple_convergence}
Let us denote $\mathbf{C}_n(\lambda)=(n^{-2/3}|\mathscr{C}^p_{\sss (j)}(\lambda)| )_{j\geq 1}$. Suppose that 
Assumption \ref{c1:assumption2} holds. 
For any $k\geq 1$ and $-\infty<\lambda_0< \lambda_1<\dots<\lambda_{k-1}<\infty$,
 \begin{equation} \label{c1:eqn_thm_multiple_convergence}
 \big( \mathbf{C}_n(\lambda_0), \mathbf{C}_n(\lambda_1), \dots, \mathbf{C}_n(\lambda_{k-1})\big)  \xrightarrow{\mathcal{L}}\sqrt{\nu} (\tilde{\boldsymbol{\gamma}}^{ \lambda_0},\tilde{\boldsymbol{\gamma}}^{ \lambda_1}, \dots, \tilde{\boldsymbol{\gamma}}^{\lambda_{k-1}})
 \end{equation}
with respect to the $(\ell^2_{\shortarrow})^k$ topology. 
 \end{theorem}
\begin{remark}\normalfont \label{c1:rem:mult-coal-heuristics}
The coupling for the limiting process in Theorem~\ref{c1:thm_multiple_convergence} is given by the multiplicative coalescent process described in Section~\ref{c1:sub_sec_notation}. 
This will become more clear when we describe the ideas of the proof. 
To understand this intuitively, notice that the component $\mathscr{C}_{\sss (i)}^p(\lambda)$ consists of some paired half-edges which form the edges of the percolated graph, and some \emph{open half-edges} which were deleted due to percolation.
Denote by $\mathcal{O}_i^p(\lambda)$, the total number of open half-edges of $\mathscr{C}_{\sss (i)}^p(\lambda)$. 
One can think of $\mathcal{O}_i^p$ as the mass of $\mathscr{C}_{\sss (i)}^p$.
Now, as we change the value of the percolation parameter from $p_n(\lambda)$ to $p_n(\lambda+d\lambda)$, exactly one edge is added to the graph and the two endpoints are chosen  proportional to the number of open half-edges of the components  of $\mathrm{CM}_n(\boldsymbol{d},p_n(\lambda))$. 
 By the above heuristics, $\mathscr{C}_{\sss (i)}^p$ and $\mathscr{C}_{\sss (j)}^p$ merge at rate proportional to $\mathcal{O}_i^p\mathcal{O}_j^p$ and creates a component of mass $\mathcal{O}_i^p+\mathcal{O}_j^p-2$. 
Later, we will show that the mass of a component is approximately proportional to the component size. 
Therefore, the component sizes merge \emph{approximately} like the multiplicative coalescent over the critical scaling window.
\end{remark}

\begin{remark}\normalfont  Janson~\cite{J09} studied the phase transition of the maximum component size for percolation on a super-critical configuration model. 
The critical value was shown to be $p=1/\nu$. 
This is precisely the reason behind taking $p_n$ of the form given by  Assumption~\ref{c1:assumption2}~\eqref{c1:assumption2-2}. 
The width of the scaling window is intimately related to the asymptotics of the susceptibility function $\sum_i|\mathscr{C}_{\sss (i)}|^2/n$. 
In fact, if $\sum_i|\mathscr{C}_{\sss (i)}|^2\sim n^{1+\eta}$, then the width of the critical window turns out to be $n^{\eta}$ and the largest component sizes are of the order $n^{(1+\eta)/2}$.
This has been universally observed in the random graph literature \cite{A97,NP10b,R12,Jo10,BHL10,DLV14}, even when the scaling limit is not in the same universality class as Erd\H{o}s-R\'enyi random graphs \cite{BHL12,DHLS16} and the same turns out to be the case in this chapter.
\end{remark}

\begin{remark}\normalfont Theorems~\ref{c1:thm_main} and~\ref{c1:thm_surplus} also hold for configuration models conditioned on simplicity. We do not give a proof here. The arguments in \cite[Section 7]{Jo10} can be followed verbatim to obtain a proof of this fact. As a result,  Theorems~\ref{c1:thm_percolation} and~\ref{c1:thm_multiple_convergence} also hold, conditioned on simplicity.
\end{remark}

The rest of the chapter is organized as follows: In Section~\ref{c1:sec_literature_overview}, we give a brief overview of the relevant literature. This will enable the reader  to understand better the relation  of this work  to the large body of literature already present. Also, it will become clear why the choices of the parameters in  Assumption~\ref{c1:assumption1}~\eqref{c1:assumption1-3} and  Assumption~\ref{c1:assumption2}~\eqref{c1:assumption2-2} should correspond to the critical scaling window. We prove Theorems~\ref{c1:thm_main} and \ref{c1:thm_surplus} in Section~\ref{c1:sec_proof}.  In Section~\ref{c1:sec_vertex} we find the asymptotic degree distribution in each component. This is used along with Theorem~\ref{c1:thm_surplus} to establish Theorem~\ref{c1:thm_percolation} in Section~\ref{c1:sec_percolation}. In Section~\ref{c1:sec_multidimensional}, we analyze the evolution of the component sizes over the percolation critical window and prove Theorem~\ref{c1:thm_multiple_convergence}.

\section{Discussion} \label{c1:sec_discussion}
\subsection{Relation to precious work}\label{c1:sec_literature_overview}
\noindent \textbf{Erd\H{o}s-R\'enyi type behavior.} We first explain what `Erd\H{o}s-R\'enyi type behavior' means. The study of \emph{critical window} for random graphs started with the seminal paper \cite{A97} on Erd\H{o}s-R\'enyi random graphs with $p=n^{-1}(1+\lambda n^{-1/3})$. Aldous showed  in this regime that the largest components are of asymptotic size $n^{2/3}$ and the ordered component sizes (scaled by $n^{2/3}$) asymptotically have the same distribution as the ordered excursion lengths of a Brownian motion with a negative parabolic drift. 
Aldous also considered a natural coupling of the re-scaled vectors of component sizes as $\lambda$ varies, and viewed it as a dynamic $\ell^2_{\shortarrow}$-valued stochastic process. It was shown that the dynamic process can be described by a process called the \emph{standard multiplicative coalescent}, which has the Feller property. This implies the convergence of the component sizes jointly for different $\lambda$ values.
 In Theorem~\ref{c1:thm_multiple_convergence}, we show that similar results hold for the configuration model under a very general set of assumptions. Of course, for general configuration models, there is no obvious way to couple the graphs such that the location parameter in the scaling window varies and percolation seems to be the most natural way to achieve this. By \cite{F07, J09}, percolation on a configuration model can be viewed as a configuration model with a random degree sequence  and this is precisely the reason for studying percolation in this chapter.\vspace{.2cm}  \\ 
\textbf{Universality and optimal assumptions.} In \cite{BHL10} it was shown that, inside the critical scaling window, the ordered component sizes (scaled by $n^{2/3}$) of an inhomogeneous random graph with $$p_{ij}= 1- \exp{\bigg(\frac{-(1+\lambda n^{-1/3})w_iw_j}{\sum_{k \in [n]} w_k} \bigg)}$$ converge to the ordered excursion lengths of an inhomogeneous Brownian motion with a parabolic drift under only a finite third-moment assumption on the weight distribution. 
We establish a counterpart of this for the configuration model in Theorem~\ref{c1:thm_main}. 
Later Nachmias and Peres~\cite{NP10b} studied the case of percolation scaling window on random regular graphs; for percolation on the configuration model similar results were obtained by Riordan~\cite{R12} for bounded  maximum degrees. 
Joseph~\cite{Jo10} obtained the same scaling limits as Theorem~\ref{c1:thm_main} for the component sizes when the degrees form an i.i.d.~sample from a distribution having finite third moment. 
Theorems~\ref{c1:thm_surplus} and~\ref{c1:thm_percolation} prove stronger versions of all these existing results for the configuration model under the optimal assumptions. 
Further, in Theorem~\ref{c1:thm_multiple_convergence}, we give a dynamic picture for percolation cluster sizes in the critical window and show that this dynamics can be approximated by the multiplicative coalescent.\vspace{.2cm} \\
\textbf{Comparison to branching processes.} In \cite{MR95,JL09} the phase transition for the component sizes of $\mathrm{CM}_n(\boldsymbol{d})$ was identified in terms of the parameter $\nu=\mathbb{E}[D(D-1)]/\mathbb{E}[D]$.
The local neighborhoods of the configuration model can be approximated by a branching process $\mathcal{X}$ which has $\nu$ as its expected progeny and thus, when $\nu >1$, $\mathrm{CM}_n(\boldsymbol{d})$ has a  component $\mathscr{C}_{\max}$ of approximate size $\rho n$, where $\rho$ is the survival probability of $\mathcal{X}$. Further, the progeny distribution of $\mathcal{X}$ has finite variance when $\mathbb{E}[D^3]<\infty$. Now, for a branching process with mean $\approx 1+\varepsilon$ and finite variance $\sigma^2$, the survival probability is approximately $2\sigma^{-2}\varepsilon$ for small $\varepsilon >0$. This seems to suggest that the largest component size under Assumption~\ref{c1:assumption1} should be of the order $n^{2/3}$ since $\varepsilon=\Theta(n^{-1/3})$. Theorem~\ref{c1:thm_main} mirrors this intuition and shows that in fact all the largest component sizes are of the order $n^{2/3}$.

\subsection{Proof ideas} The proof of Theorem \ref{c1:thm_main} uses a standard functional central limit theorem argument. 
Indeed we associate a suitable semi-martingale with the graph obtained from an \emph{exploration algorithm} used to explore the connected components of $\mathrm{CM}_n(\boldsymbol{d})$. 
The martingale part is then shown to converge to an inhomogeneous Brownian motion, and the drift part is shown to converge to a parabola. 
The fact that the component sizes can be expressed in terms of the hitting times of the semi-martingale implies the finite-dimensional convergence of the component sizes. 
The convergence with respect to $\ell^2_{\shortarrow}$ is then concluded using size-biased point process arguments formulated by Aldous~\cite{A97}. 
Theorem~\ref{c1:thm_surplus} requires a careful estimate of the tail probability of the distribution of  surplus edges  when the component size is small and we obtain this using martingale estimates in Lemma~\ref{c1:lem_surplus_delta_bound}.  Theorem~\ref{c1:thm_percolation} is proved by showing that the percolated degree sequence satisfies Assumption~\ref{c1:assumption1} almost surely. 
Finally, we prove Theorem \ref{c1:thm_multiple_convergence} in Section~\ref{c1:sec_multidimensional}.
The key challenges here are that, for each fixed $n$, the components do not merge according to their component sizes, and that the components do not merge exactly like a multiplicative coalescent over the scaling window. 
Thus the main theme of the proof lies in approximating the evolution of the component sizes over the percolation scaling window with a suitable dynamic process that is an exact multiplicative coalescent. 

\section{{Proofs of Theorems \ref{c1:thm_main} and \ref{c1:thm_surplus}}}
\label{c1:sec_proof}
 \subsection{The exploration process} \label{c1:exploration} Let us explore the graph sequentially using a natural approach outlined in~\cite{R12}. At step $k$, divide the set of half-edges into three groups; sleeping half-edges $\mathcal{S}_{k}$, active half-edges $\mathcal{A}_{k}$, and  dead half-edges $\mathcal{D}_{k}$. The depth-first exploration process can be summarized in the following algorithm:
 \begin{algo}[DFS exploration]\label{c1:algo:1}\normalfont At $k=0$, $\mathcal{S}_{k}$ contains all the half-edges and $\mathcal{A}_{k}$, $\mathcal{D}_{k}$ are empty. While ($\mathcal{S}_{k}\neq\varnothing$ or $\mathcal{A}_{k}\neq\varnothing$) we do the following at  stage $k+1$:
 \begin{itemize}
 \item[S1] If $\mathcal{A}_{k}\neq\varnothing$, then take the smallest half-edge $a$ from $\mathcal{A}_{k}$.
 \item[S2] Take the half-edge $b$ from $\mathcal{S}_{k}$ that is paired to $a$. Suppose $b$ is attached to a vertex $w$ (which is necessarily  not discovered yet). Declare $w$ to be discovered,  let $r= d_w-1$ and $b_{w1}, b_{w2}, \dots b_{wr}$ be the half-edges of $w$ other than $b$. Declare $b_{w1}$, $b_{w2}$,..., $b_{wr} ,b$ to be smaller than all other half-edges in $\mathcal{A}_{k}$. 
 Also order the half-edges of $w$ among themselves as $b_{w1} > b_{w2}>\dots  >b_{wr} >b$. 
 Now identify $\mathcal{B}_{k} \subset \mathcal{A}_{k} \cup \{ b_{w1}, b_{w2},\dots, b_{wr} \}$ as the collection of all half-edges in $\mathcal{A}_{k} $ paired to one of the $b_{wi}$'s and the corresponding $b_{wi}$'s. Similarly identify $\mathcal{C}_{k} \subset \{b_{w1}, b_{w2},\dots,b_{wr} \}$ which is the collection of self-loops incident to $w$. Finally, declare $\mathcal{A}_{k+1} = \mathcal{A}_{k} \cup \{b_{w1}, b_{w2},\dots,b_{wr} \} \setminus \big(\mathcal{B}_{k} \cup \mathcal{C}_{k}\big)$, $\mathcal{D}_{k+1} = \mathcal{D}_{k} \cup \{ a,b \} \cup \mathcal{B}_{k} \cup \mathcal{C}_{k}$ and $\mathcal{S}_{k+1}= \mathcal{S}_{k} \setminus \big( \{ b \} \cup \{b_{w1}, b_{w2},...,b_{wr} \} \big)$. Go to stage $k+2$.
  \item[S3] If $\mathcal{A}_{k}= \varnothing$ for some $k$, then take out one half-edge $a$ from $\mathcal{S}_{k}$ uniformly at random and identify the vertex $v$ incident to it. Declare $v$ to be discovered. Let $r= d_v-1$ and assume that $a_{v1}$, $a_{v2}$,..., $a_{vr}$ are the half-edges of $v$ other than $a$ and identify the collection of half-edges involved in self-loops $\mathcal{C}_{k}$ as in Step 2. Order the half-edges of $v$ as $a_{v1}>a_{v2}>\dots>a_{vr}>a$. Set $\mathcal{A}_{k+1} = \{ a, a_{v1}$, $a_{v2}$,..., $a_{vr}\} \setminus \mathcal{C}_{k}$, $\mathcal{D}_{k+1} = \mathcal{D}_{k} \cup \mathcal{C}_{k}$, and $\mathcal{S}_{k+1}= \mathcal{S}_{k} \setminus  \{ a, a_{v1}, a_{v2},...,a_{vr} \} $. Go to stage $k+2$.
 \end{itemize}
 \end{algo}
 In words, we explore a new vertex at each stage and throw away all the half-edges involved in  a loop/multiple edge/cycle  with the vertex set already discovered before proceeding to the next stage.
 The ordering of the half-edges is such that the connected components of $\mathrm{CM}_{n}(\boldsymbol{d})$ are explored in the depth-first way. We call the half-edges of  $\mathcal{B}_{k} \cup \mathcal{C}_{k}$ $cycle$ half-edges because they create loops, cycles or multiple edges in the graph. Let
 \begin{equation}\label{c1:name:c-B-A}
 A_k:=| \mathcal{A}_{k} |,\quad c_{(k+1)}:= (| \mathcal{B}_{k} | +| \mathcal{C}_{k} | )/2 ,\quad  U_k:= | \mathcal{S}_{k} |.
  \end{equation} 
  Let $d_{\sss(j)}$ be the degree of the $j^{th}$ explored vertex and  define the following process:

 \begin{equation}\label{c1:walk1}
 S_{n}(0)=0, \quad S_{n}(i)=\sum_{j=1}^{i}(d_{\sss(j)} -2-2c_{\sss(j)}).
 \end{equation}
 
 The process $\mathbf{S}_{n}= ( S_{n}(i))_{i \in [n]}$ ``encodes the component sizes as lengths of path segments above past minima'' as discussed in \cite{A97}. Suppose $\mathscr{C}_{i}$ is the $i^{th}$ connected component explored by the above exploration process. Define
\begin{equation} \label{c1:eq:tau-k}
\tau_{k}=\inf \big\{ i:S_{n}(i)=-2k \big\}.
\end{equation}
Then  $\mathscr{C}_{k}$ is discovered between the times $\tau_{k-1}+1$ and $\tau_k$ and  $|\mathscr{C}_{k}|=\tau_{k}-\tau_{k-1}$.

\subsection{Size-biased exploration}
The vertices are explored in a size-biased manner with sizes proportional to their degrees, i.e., if we denote by $v_{\sss(i)}$ the $i^{th}$ explored vertex in Algorithm~\ref{c1:algo:1} and by $d_{\sss(i)}$ the degree of $v_{\sss(i)}$, then $\forall j\in \mathscr{V}_{i-1},$
 \begin{equation}
 \mathbb{P}\big(v_{\sss (i)}=j \vert v_{\sss (1)},v_{\sss (2)},...,v_{\sss (i-1)}\big)=\frac{d_{j}}{\sum_{k\notin \mathscr{V}_{i-1}} d_{k}}= \frac{d_{j}}{\sum_{k \in [n]} d_k - \sum_{k=1}^{i-1} d_{\sss (k)}},
 \end{equation}
 where $\mathscr{V}_{i}$ denotes the first $i$ vertices to be discovered in the above exploration process.
 The following lemma will be used crucially in the proof of Theorem~\ref{c1:thm_main}:
\begin{lemma} \label{c1:lem_con_1} Suppose that Assumption~\ref{c1:assumption1} holds and denote $\sigma_r = \mathbb{E} [ D^r ] $ and $\mu = \mathbb{E} [ D ]$. Then for all $t > 0 $, as $n \to \infty$,
\begin{equation} \label{c1:lem_eq2}
   \sup_{ u\leq t } \Big| n^{-2/3} \sum_{i=1}^{\lfloor n^{2/3} u \rfloor} d_{\sss (i)} - \frac{\sigma_{2}u}{\mu} \Big| \xrightarrow{\mathbb{P}} 0,
  \end{equation} and
  \begin{equation} \label{c1:lem_eq1}
   \sup_{ u\leq t } \Big| n^{-2/3} \sum_{i=1}^{\lfloor n^{2/3} u \rfloor} d_{\sss (i)}^{2} - \frac{\sigma_{3}u}{\mu} \Big| \xrightarrow{\mathbb{P}} 0.
  \end{equation}
 \end{lemma}
 The proof of this lemma follows from the two lemmas stated below:
 \begin{lemma}[{\cite[Lemma 8.2]{BSW14}}]\label{c1:lem_general}
Consider a weight sequence $(w_i)_{i\in [n]}$ and let $m=m(n)\leq n$ be increasing with $n$. Let $\{v(i)\}_{i\in[n]}$ be the size-biased reordering of indices $[n]$, where the size of index $i$ is $d_i/\ell_n$. 
Define $\gamma_n = \sum_{i\in[n]}w_id_i/\ell_n$ and  $Y(t) = (m\gamma_n)^{-1}\sum_{i=1}^{\lfloor mt\rfloor} w_{\sss v(i)}$. Further, let $d_{\max} = \max_{i\in [n]}d_i$, and $w_{\max} = \max_{i\in [n]}w_i$. Assume that 
\begin{equation}\label{c1:eq:conditions-size-biased}
 \lim_{n\to\infty}md_{\max}/\ell_n = 0,\quad \text{and}\quad \lim_{n\to\infty} (m\gamma_n)^{-1}w_{\max} = 0.
\end{equation}Then, for any $t>0$, as $n\to\infty$, $
 \sup_{u\leq t}|Y(t)-t|\pto 0.
$
 \end{lemma}
\begin{lemma} \label{c1:lem_d_max} Assumption~\ref{c1:assumption1} implies
 \begin{equation}
  \lim_{k\to \infty}\lim_{n\to \infty} \frac{1}{n} \sum_{j\in [n]}\mathbf{1}_{\{ d_{j}>k\} }d_{j}^r=0,\quad r=1,2,3.
 \end{equation}
For $r=3$, in particular, this implies $d_{\max}^{3}=o(n)$.
\end{lemma}

\subsection{Estimate of cycle half-edges} The following lemma gives an estimate of the number of cycle half-edges created up to time $t$. This result is proved in \cite{R12} for bounded degrees. In our case, it follows from Lemma~\ref{c1:lem_con_1} as we show below:
\begin{lemma} \label{c1:lem_back_edges}
 For Algorithm~\ref{c1:algo:1}, if $A_k=\big| \mathcal{A}_k \big|$, $ B_k:= \big| \mathcal{B}_k \big|$, and $ C_k:= \big| \mathcal{C}_k \big|$, then
 \begin{equation} \label{c1:lem_back_edges_B}
  \mathbb{E} \big[ B_k \vert \mathscr{F}_k \big] = (1+o_{\sss \mathbb{P}}(1))\frac{2A_k}{U_k}  + O_{\sss \mathbb{P}}(n^{-{2/3}})
 \end{equation} and
 \begin{equation} \label{c1:lem_back_edges_C}
  \mathbb{E} \big[ C_k \vert \mathscr{F}_k \big] = O_{\sss \mathbb{P}}(n^{-1})
 \end{equation} uniformly for $k \leq tn^{2/3}$ and any $t >0$, where $\mathscr{F}_k$ is the sigma-field generated by the information revealed up to stage $k$.   Further, all the $O_{\sss\mathbb{P}} $ and $o_{\sss\mathbb{P}} $ terms in \eqref{c1:lem_back_edges_B} and \eqref{c1:lem_back_edges_C} can be replaced by $O_{\sss E}$ and $o_{\sss E}$.
 \end{lemma}

 \begin{proof}
  Suppose $ U_k:= \big| \mathcal{S}_k \big|$. First note that by \eqref{c1:lem_eq2}
  \begin{equation}
   \frac{U_k}{n}= \frac{1}{n} \sum_{j \in [n]} d_j - \frac{1}{n} \sum_{j=1}^{k} d_{\sss (j)} = \mathbb{E} [ D ] + o_{\sss \mathbb{P}}(1)
  \end{equation} uniformly over $k \leq tn^{2/3}$. 
  Let $a$ be the half-edge that is being explored at stage $k+1$.  
  Now, each of the $(A_k-1)$ half-edges of $\mathcal{A}_k \setminus \{ a \}$ is equally likely to be paired with a half-edge of $v_{\sss (k+1)}$, thus creating two elements of $\mathcal{B}_k$. Also, given $\mathscr{F}_k$ and $v_{\sss (k+1)}$, the probability that a half-edge of $\mathcal{A}_k \setminus \{ a \}$ is paired to one of the half-edges of $v_{\sss (k+1)}$ is $(d_{\sss(k+1)}-1)/(U_{k}-1)$. Therefore,
 \begin{equation}\label{c1:lem_back_edges_eqn2}
 \begin{split}
 \mathbb{E} \big[ B_k \vert \mathscr{F}_k, v_{\sss(k+1)} \big] = 2(A_k-1)\frac{d_{\sss(k+1)}-1}{U_{k}-1}= 2 \big( d_{\sss(k+1)} -1 \big) \frac{A_k}{U_k-1} - 2 \frac{d_{\sss(k+1)} -1 }{U_k-1}.
 \end{split}
 \end{equation}
 Hence,
 \begin{equation} \label{c1:lem_back_edges_eqn1}
  \mathbb{E} \big[ B_k \vert \mathscr{F}_k \big] = 2 \mathbb{E}\big[ d_{\sss(k+1)} -1 \vert \mathscr{F}_k \big] \frac{A_k}{U_k-1} - 2 \frac{\mathbb{E}\big[ d_{\sss (k+1)} -1 \vert \mathscr{F}_k \big] }{U_k-1}.
 \end{equation}
 Now, using \eqref{c1:lem_eq2} and \eqref{c1:lem_eq1},
  \begin{equation}
 \mathbb{E}\big[ d_{\sss(k+1)} -1 \vert \mathscr{F}_k \big] = \frac{\sum_{j \notin \mathscr{V}_k} d_j (d_j-1)}{\sum_{j \notin \mathscr{V}_k} d_j}  = \frac{\sum_{j \in [n]}d_j^2}{\sum_{j \in [n]}d_j}-1+o_{\sss\mathbb{P}}(1) = 1+o_{\sss\mathbb{P}}(1).
 \end{equation} uniformly over $k \leq tn^{2/3}$, where the last step follows from Assumption~\ref{c1:assumption1}~\eqref{c1:assumption1-3}. 
 Further, using the fact $\PR(D=1)>0$, $U_k\geq c_0n $ for some constant $c_0>0$ uniformly over $k\leq tn^{2/3}$.
 Thus, \eqref{c1:lem_back_edges_eqn1} gives \eqref{c1:lem_back_edges_B}.  The fact that all the $O_{\sss\mathbb{P}}$, $o_{\sss\mathbb{P}}$ can be replaced by $O_{\sss E}$, $o_{\sss E}$ follows from $\sum_{j\in [n]}d_j^r-kd_{\max}^r \leq \sum_{j\notin\mathscr{V}_k}d_j^r\leq  \sum_{j\in [n]}d_j^r$ for $r=1,2$, together with $d_{\max}=o(n^{1/3})$. To prove \eqref{c1:lem_back_edges_C}, note that
 \begin{equation}
  \mathbb{E} \big[ C_k \vert \mathscr{F}_k, v_{\sss (k+1)} \big] = 2(d_{\sss (k+1)}-2)\frac{d_{\sss (k+1)}-1}{U_{k}-1}.
 \end{equation}
 By Assumption \ref{c1:assumption1} and \eqref{c1:lem_eq2} 
 \begin{equation}\label{c1:expt-dk2-F-k}\mathbb{E}[d_{\sss(k+1)}^2|\mathscr{F}_k] = \frac{\sum_{j\notin \mathscr{V}_k} d_j^3}{\sum_{j\notin \mathscr{V}_k} d_j} \leq \frac{\sum_{j\in [n]} d_j^3}{\sum_{j\in [n]} d_j+o_{\sss \mathbb{P}}(n^{2/3})} = O_{\sss\mathbb{P}}(1),
 \end{equation}uniformly for $k\leq tn^{2/3}$. Therefore,
 \begin{equation}\label{c1:lem_back_edges_eqn3}
 \begin{split}
  \mathbb{E} \big[ C_k \vert \mathscr{F}_k \big] = O_{\sss\mathbb{P}}(n^{-1})
  \end{split}
 \end{equation} uniformly over   $k \leq tn^{2/3}$. Again, the $O_{\sss\mathbb{P}}$ term can be replaced by $O_{\sss E}$, as argued before.
 \end{proof}

 \subsection{Key ingredients} \label{c1:sec-key-ingredients}
 For any $\mathbb{D}[0,\infty)$-valued process $\mathbf{X}_n$ define $\bar{X}_n(u):= n^{-1/3} X_n( \lfloor n^{2/3}u \rfloor )$ and  $ \bar{\mathbf{X}}_{n}:=( \bar{X}_n(u) )_{u \geq 0} $. The following result is the main ingredient for proving Theorem~\ref{c1:thm_main}. Recall the definition of $\mathbf{B}^\lambda_{\mu,\eta}$ from \eqref{c1:def:inhomogen:BM} with parameters given in \eqref{c1:parameter}. 
 \begin{theorem}[Convergence of the exploration process] \label{c1:thm_main1}
 Under Assumption~\ref{c1:assumption1}, as $n \to \infty$,
  \begin{equation}
   \bar{\mathbf{S}}_{n} \xrightarrow{\mathcal{L}} \mathbf{B}^\lambda_{\mu,\eta}
  \end{equation}
 with respect to the $Skorohod$ $J_{1}$ topology.
 \end{theorem}
  As in \cite{Jo10}, we will prove this by approximating  $\mathbf{S}_{n}$ by a simpler process defined as
\begin{equation} \label{c1:walk2} 
  s_{n}(0)=0, \quad s_{n}(i)=\sum_{j=1}^{i}(d_{\sss(j)} -2).
 \end{equation}
 Note that the difference between the processes $\mathbf{S}_n$ and $\mathbf{s}_n$ is due to the cycles, loops, and multiple-edges encountered during the exploration.
Following the approach of \cite{Jo10}, it will be enough to prove the following:
 \begin{proposition} \label{c1:thm_main2} Under Assumption~\ref{c1:assumption1}, as $n \to \infty$,
  \begin{equation}
   \bar{\mathbf{s}}_{n} \xrightarrow{\mathcal{L}} \mathbf{B}^\lambda_{\mu,\eta}
  \end{equation}
  with respect to  the $Skorohod$ $J_{1}$ topology.
\end{proposition}
 \begin{remark} \label{c1:remark_cont_cadlag} \normalfont It will be shown that the distributions of $\bar{\mathbf{S}}_n$ and $\bar{\mathbf{s}}_n$ are very close as $n\to\infty$, and therefore, Proposition~\ref{c1:thm_main2} implies Theorem \ref{c1:thm_main1}. This is achieved by proving that we will not see \emph{too many} cycle half-edges up to the time $\lfloor n^{2/3}u\rfloor$ for any  fixed $u >0$.
 \end{remark}
From here onwards we will look at the continuous versions of the processes $\bar{\mathbf{S}}_{n}$ and $\bar{\mathbf{s}}_{n}$ by linearly interpolating between the values at the jump points and write it using the same notation. It is easy to see that these continuous versions differ from their c\`adl\`ag versions by at most $n^{-1/3}d_{\max}=o(1)$ uniformly on $[0,T]$, for any $T>0$. Therefore, the convergence in law of the continuous versions implies the convergence in law of the c\`adl\`ag versions and vice versa.
Before proceeding to show that Theorem~\ref{c1:thm_main1} is a consequences of Proposition~\ref{c1:thm_main2}, we will need to bound the difference of these two processes in a suitable way. We need the following lemma.  Recall the definition of $c_{\sss (k+1)}:= (B_k+C_k)/2$ from \eqref{c1:name:c-B-A}.
\begin{lemma} \label{c1:cor_c_k}
 Fix $t >0$ and $M>0$ (large). Define the event
 $$E_{n}(t,M):=\big\{\max_{s\leq t}\{\bar{s}_{n}(s)-\min_{u\leq s} \bar{s}_{n}(u)\} < M \big\}.$$ Then
 \begin{equation} \label{c1:cor_c_k_eqn}
  \limsup\limits_{n \to \infty} \sum_{k \leq tn^{2/3}} \mathbb{E} \big[ c_{\sss(k)} \mathbf{1}_{E_{n}(t,M)} \big] < \infty.
 \end{equation}
 \end{lemma}

 \begin{proof} Lemma~\ref{c1:cor_c_k} is similar to \cite[Lemma 6.1]{Jo10}. We add a brief proof here.  Note that, for all large $n$, $A_k \leq Mn^{1/3}$ on $E_{n}(t,M)$, because
  \begin{equation}
   A_k = S_n(k)-\min\limits_{j \leq k} S_n(j)= s_n(k) - 2 \sum_{j=1}^{k} c_{\sss(j)} - \min\limits_{j \leq k} S_n(j) \leq s_n(k)-\min\limits_{j \leq k} s_n(j),
  \end{equation} where the last step follows by noting that $\min_{j \leq k} s_n(j) \leq \min_{j \leq k} S_n(j) + 2 \sum_{j=1}^{k} c_{\sss(j)}$.  By Lemma~\ref{c1:lem_back_edges},
  \begin{equation}
  \mathbb{E} \big[ c_{\sss(k)} \mathbf{1}_{E_{n}(t,M)} \big]  \leq \frac{Mn^{1/3}}{\mu n} + o(n^{-2/3})= \frac{M}{\mu}n^{-2/3} + o(n^{-2/3})
  \end{equation} uniformly for $k \leq tn^{2/3}$. 
  Summing over $1 \leq k \leq tn^{2/3}$ and taking the $\limsup$ completes the proof.
 \end{proof}
The proof of the fact that Theorem~\ref{c1:thm_main1} follows from Proposition~\ref{c1:thm_main2} and Lemma~\ref{c1:cor_c_k} is standard (see \cite[Section 6.2]{Jo10}) and we skip the proof for the sake of brevity. 
  From here onward the main focus of this section will be to prove Proposition~\ref{c1:thm_main2}. We use the martingale functional central limit theorem in a similar manner as~\cite{A97}.
  \begin{proof}[Proof of Proposition~\ref{c1:thm_main2}] Let $ \{\mathscr{F}_{i} \}_{i \geq 1} $ be the natural filtration defined in Lemma \ref{c1:lem_back_edges}. Recall the definition of $s_n(i)$ from \eqref{c1:walk2}. By the Doob-Meyer decomposition \cite[Theorem 4.10]{KS91} we can write
  \begin{equation}\label{c1:split_up}
    s_{n}(i) = M_{n}(i)+A_{n}(i),\quad s_{n}^{2}(i) = H_{n}(i)+B_{n}(i),
  \end{equation}
where
 \begin{subequations}
  \begin{equation} \label{c1:defn_martingale}
   M_{n}(i)= \sum_{j=1}^{i} \big(d_{\sss(j)}-\mathbb{E} \big[ d_{\sss(j)}\vert \mathscr{F}_{j-1} \big] \big),
  \end{equation}
  \begin{equation}
   A_{n}(i)= \sum_{j=1}^{i} \mathbb{E}\big[d_{\sss(j)}-2 \vert \mathscr{F}_{j-1} \big],
  \end{equation}
  \begin{equation}
   B_{n}(i)= \sum_{j=1}^{i} \big( \mathbb{E} \big[ d_{\sss(j)}^{2}\vert \mathscr{F}_{j-1} \big]-\mathbb{E}^{2} \big[ d_{\sss(j)}\vert \mathscr{F}_{j-1} \big] \big).
  \end{equation}
 \end{subequations}

Recall that for a discrete-time stochastic process $(X_n(i))_{i\geq 1}$, we denote $\bar{X}_n(t)=n^{-1/3}X_n(\lfloor tn^{2/3}\rfloor )$. 
Our result follows from the martingale functional central limit theorem  \cite[Theorem 2.1]{W07} if we can prove the following four conditions: For any $u >0$,
\begin{subequations}
 \begin{equation}\label{c1:condition1}
  \sup_{s\leq u}\big| \bar{A}_{n}(s)-\lambda s+\frac{\eta s^2}{2\mu^{3}}\big| \xrightarrow{\mathbb{P}} 0,
 \end{equation}
 \begin{equation} \label{c1:condition2}
  n^{-1/3}\bar{B}_{n}(u)\xrightarrow{\mathbb{P}} \frac{\eta}{\mu^{2}} u,
 \end{equation}
 \begin{equation} \label{c1:condition3}
  \mathbb{E}\big[\sup_{s\leq u}\big| \bar{M}_{n}(s)-\bar{M}_{n}(s-)\big| ^{2}\big] \to 0,
 \end{equation}
 and
 \begin{equation} \label{c1:condition4}
  n^{-1/3}\mathbb{E}\big[\sup_{s\leq u}\vert \bar{B}_{n}(s)-\bar{B}_{n}(s-)\vert\big] \to 0.
 \end{equation}
\end{subequations}
\par Indeed \eqref{c1:condition1} gives rise to the quadratic drift term of the limiting distribution. Conditions \eqref{c1:condition2}, \eqref{c1:condition3}, \eqref{c1:condition4} are the same as \cite[Theorem 2.1, Condition (ii)]{W07}. The facts that the jumps of both the martingale and the quadratic-variation process go to zero and that the quadratic variation process is converging to the quadratic variation of an inhomogeneous Brownian Motion, together imply the convergence of the martingale term. The validation of these conditions is given separately in the subsequent part of this section.
\end{proof}

\begin{lemma} The conditions \eqref{c1:condition2}, \eqref{c1:condition3}, and \eqref{c1:condition4} hold.
\end{lemma}

\begin{proof}
 Denote by $\sigma_{r}(n)=\frac{1}{n} \sum_{i\in [n]}d_{i}^{r},\: r=2,3$ and $\mu(n)=\frac{1}{n} \sum_{i\in [n]}d_{i}$. To prove \eqref{c1:condition2}, it is enough to prove that
 \begin{equation}
   n^{-2/3}B_{n}(\lfloor un^{2/3}\rfloor) \xrightarrow{\mathbb{P}}  \frac{\sigma_3 \mu - \sigma_{2}^{2}}{\mu^{2}} u.
 \end{equation}
 Recall that $\mathbb{E}[ d_{\sss(i)}^{2}\vert \mathscr{F}_{i-1} ] = \sum_{j \notin \mathscr{V}_{i-1}}d_{j}^{3}/\sum_{j \notin \mathscr{V}_{i-1}}d_{j}.$  
 Further, uniformly over $i \leq un^{2/3}$,
 \begin{equation}\label{c1:sum-deg-explored}
 \sum_{j \notin \mathscr{V}_{i-1}}d_{j} = \sum_{j \in [n]}d_{j} + \OP(d_{\max}i) = \ell_n+\oP(n).
 \end{equation}
  Assume that, without loss of generality, $j\mapsto d_{j}$ is non-increasing.  Then, uniformly over $i \leq un^{2/3}$,
 \begin{equation} \label{c1:eq_lem_1}
 \bigg| \sum_{j \notin \mathscr{V}_{i-1}}d_{j}^{3} - n\sigma_{3}(n) \bigg| \leq \sum_{j=1}^{un^{2/3}} d_{j}^{3}.
 \end{equation}
 For each fixed $k$,
 \begin{eq}
 \frac{1}{n} \sum_{j=1}^{un^{2/3}} d_{j}^{3} &\leq \frac{1}{n}\sum_{j=1}^{un^{2/3}} \mathbf{1}_{ \{ d_{j} \leq k \} } d_{j}^{3} + \frac{1}{n}\sum_{j \in [n]} \mathbf{1}_{ \{ d_{j} > k \} } d_{j}^{3} \\
 &\leq k^{3}un^{-1/3} + \frac{1}{n}\sum_{j \in [n]} \mathbf{1}_{ \{ d_{j} > k \} } d_{j}^{3} =o(1),
 \end{eq}
 where we first let $n \to \infty$ and then $k \to \infty$ and use Lemma \ref{c1:lem_d_max}. Therefore, the right-hand side of \eqref{c1:eq_lem_1} is $o(n)$ and we conclude that, uniformly over $i \leq un^{2/3}$,
 \begin{equation}
 \mathbb{E}\big[d_{\sss(i)}^{2}\vert \mathscr{F}_{i-1}\big] = \frac{\sigma_{3}}{\mu} + \oP(1).
 \end{equation}
A similar argument gives
  \begin{equation}
   \mathbb{E}\big[d_{\sss(i)} \vert \mathscr{F}_{i-1}\big] = \frac{\sigma_{2}}{\mu} + \oP(1),
  \end{equation}
and \eqref{c1:condition2} follows by noting that the error term is $\oP(1)$, uniformly over $i \leq un^{2/3}$.
 The proofs of \eqref{c1:condition3} and \eqref{c1:condition4} are rather short and we present them below. For \eqref{c1:condition3}, we bound
 \begin{align}
   \mathbb{E} \Big[ \sup_{s \leq u} &\vert \bar{M}_{n}(s) - \bar{M}_{n}(s -) \vert^{2} \Big]= n^{-2/3} \mathbb{E} \Big[ \sup_{k \leq un^{2/3}} \vert M_{n}(k) - M_{n}(k-1) \vert^{2} \Big] \nonumber\\
    & = n^{-2/3} \mathbb{E} \Big[ \sup_{k \leq un^{2/3}} \big| d_{\sss(k)} - \mathbb{E}[ d_{\sss(k)} \vert \mathscr{F}_{k-1}] \big|^{2} \Big]\nonumber\\
    & \leq n^{-2/3} \mathbb{E} \Big[ \sup_{k \leq un^{2/3}}  d_{\sss(k)}^{2} \Big] + n^{-2/3} \mathbb{E} \Big[ \sup_{k \leq un^{2/3}} \mathbb{E}^2\big[ d_{\sss(k)} \vert \mathscr{F}_{k-1} \big] \Big] \nonumber\\
    & \leq 2n^{-2/3}d_{\max}^2. 
    \end{align}
Similarly,  \eqref{c1:condition4} gives
 \begin{align}
n^{-1/3} \mathbb{E} \big[ \sup_{s \leq u} \vert &\bar{B}_{n}(s) - \bar{B}_{n}(s -) \vert^{2} \big]= n^{-2/3} \mathbb{E} \big[ \sup_{k \leq un^{2/3}} \vert B_{n}(k) - B_{n}(k-1) \vert \big] \hspace{2cm}\nonumber \\
& = n^{-2/3} \mathbb{E} \big[ \sup_{k \leq un^{2/3}} \mathrm{var} \big( d_{\sss(k)} \vert \mathscr{F}_{k-1} \big) \big]\\& \leq 2n^{-2/3} d_{\max}^{2},\nonumber
 \end{align} and Conditions \eqref{c1:condition3} and \eqref{c1:condition4} follow from Lemma \ref{c1:lem_d_max} using $d_{\max}=o(n^{1/3})$.
 \end{proof}
 Next, we prove Condition \eqref{c1:condition1} which requires some more work. Note that
\begin{align}\label{c1:the-split-up}
 &\mathbb{E} \big[ d_{\sss(i)} -2 \vert \mathscr{F}_{i-1} \big]  = \frac{\sum_{j \notin \mathscr{V}_{i-1}} d_{j}(d_{j}-2)}{\sum_{j \notin \mathscr{V}_{i-1}} d_{j}}\nonumber\\
 &= \frac{\sum_{j \in [n]} d_{j}(d_{j}-2)}{\sum_{j \in [n]} d_{j}}- \frac{\sum_{j \in \mathscr{V}_{i-1}} d_{j}(d_{j}-2)}{\sum_{j \in [n]} d_{j}} + \frac{\sum_{j \notin \mathscr{V}_{i-1}} d_{j}(d_{j}-2)\sum_{j \in \mathscr{V}_{i-1}} d_{j} }{\sum_{j \notin \mathscr{V}_{i-1}} d_{j}\sum_{j \in [n]} d_{j}} \nonumber\\
 &= \frac{\lambda}{n^{1/3}} - \frac{\sum_{j \in \mathscr{V}_{i-1}} d_{j}^{2}}{\sum_{j \in [n]} d_{j}} + \frac{\sum_{j \notin \mathscr{V}_{i-1}} d_{j}^{2}\sum_{j \in \mathscr{V}_{i-1}} d_{j} }{\sum_{j \notin \mathscr{V}_{i-1}} d_{j}\sum_{j \in [n]} d_{j}} +o(n^{-1/3}),
\end{align} where the last step follows from Assumption~\ref{c1:assumption1}~\eqref{c1:assumption1-3}.
Therefore,
\begin{equation} \label{c1:eq_cond_1}
\begin{split}
A_{n}(k) &= \sum_{i=1}^{k}\mathbb{E} \big[ d_{\sss(i)} -2 \vert \mathscr{F}_{i-1} \big] \\
&= \frac{k\lambda}{n^{1/3}} - \sum_{i=1}^{k} \frac{\sum_{j \in \mathscr{V}_{i-1}} d_{j}^{2}}{\sum_{j \in [n]} d_{j}} + \sum_{i=1}^{k}\frac{\sum_{j \notin \mathscr{V}_{i-1}} d_{j}^{2}\sum_{j \in \mathscr{V}_{i-1}} d_{j} }{\sum_{j \notin \mathscr{V}_{i-1}} d_{j}\sum_{j \in [n]} d_{j}} + o(kn^{-1/3}).
\end{split}
\end{equation}
 The following lemma estimates the sums on the right-hand side of \eqref{c1:eq_cond_1}:

 \begin{lemma} \label{c1:lem_con_2}
  For all $u>0$, as $n \to \infty$,
  \begin{equation} \label{c1:lem_eq3}
   \sup_{s \leq u} \bigg| n^{-1/3} \sum_{i=1}^{\lfloor s n^{2/3} \rfloor} \sum_{j=1}^{i-1} \frac{d_{\sss(j)}^{2}}{\ell_n}- \frac{\sigma_{3} s^{2}}{2\mu^{2}} \bigg| \xrightarrow{\mathbb{P}} 0
  \end{equation} and
  \begin{equation} \label{c1:lem_eq4}
   \sup_{s \leq u} \bigg| n^{-1/3} \sum_{i=1}^{\lfloor s n^{2/3} \rfloor} \sum_{j=1}^{i-1} \frac{d_{\sss(j)}}{\ell_n}- \frac{\sigma_{2} s^{2}}{2\mu^{2}} \bigg| \xrightarrow{\mathbb{P}} 0.
  \end{equation} Consequently, 
  \begin{equation} \label{c1:lem:estimate-drift}
\sup_{s \leq u} \bigg|n^{-1/3}\sum_{i=1}^{\lfloor s n^{2/3} \rfloor}\frac{\sum_{j \notin \mathscr{V}_{i-1}} d_{j}^{2}\sum_{j \in \mathscr{V}_{i-1}} d_{j} }{\sum_{j \notin \mathscr{V}_{i-1}} d_{j}\sum_{j \in [n]} d_{j}} - \frac{\sigma_{2}^{2}s^{2}}{2\mu^{3}}\bigg| \xrightarrow{\mathbb{P}} 0.
\end{equation}
 \end{lemma}
 \begin{proof}
  Notice that
   \begin{equation}\begin{split}
     &\sup_{s \leq u} \Big| n^{-1/3} \sum_{i=1}^{\lfloor s n^{2/3} \rfloor} \sum_{j=1}^{i-1} \frac{d_{\sss (j)}^{2}}{\ell_n}- \frac{\sigma_{3} s^{2}}{2\mu^{2}} \Big| = \sup_{k \leq un^{2/3}} \Big| n^{-1/3} \sum_{i=1}^{k} \sum_{j=1}^{i-1} \frac{d_{\sss (j)}^{2}}{\ell_n}- \frac{\sigma_{3} k^{2}}{2\mu^{2}n^{4/3}} \Big| \\
 & \leq \frac{1}{\ell_{n}} \sup_{k \leq un^{2/3}} \Big| n^{-1/3} \sum_{i=1}^{k} \Big(\sum_{j=1}^{i-1} d_{\sss (j)}^{2}- \frac{\sigma_{3} (i-1)}{\mu} \Big) \Big| \\
 &\hspace{1cm}+ \sup_{k \leq un^{2/3}} \Big| \frac{k \sigma_3}{2 \mu \ell_n n^{1/3}} \Big| + \sup_{k \leq un^{2/3}} \Big| \frac{k^2 \sigma_3}{2 \mu \ell_n n^{1/3}} -\frac{k^2 \sigma_3}{2 \mu^2 n^{4/3}} \Big|\\
 & \leq \frac{1}{\ell_n} n^{-1/3} un^{2/3} \sup_{i \leq un^{2/3}} \Big| \sum_{j=1}^{i}   d_{\sss (j)}^{2}- \frac{\sigma_{3} i}{\mu}  \Big| + o(1)+ \frac{\sigma_3 n^{-1/3}}{2 \mu} \Big| \frac{1}{\ell_n} - \frac{1}{n \mu} \Big| u^2 n^{4/3}\\
 & \leq \frac{u}{\mu + o(1)} \sup_{s \leq u} \Big| \Big(n^{-2/3}\sum_{j=1}^{\lfloor s n^{2/3} \rfloor} d_{\sss(j)}^{2}- \frac{\sigma_{3} s}{\mu} \Big) \Big| +o(1).
 \end{split}
   \end{equation} and \eqref{c1:lem_eq3} follows from \eqref{c1:lem_eq1} in Lemma~\ref{c1:lem_con_1}. The proof of \eqref{c1:lem_eq4} is similar and it follows from~\eqref{c1:lem_eq2}.
 We now show \eqref{c1:lem:estimate-drift}. Recall that $\sigma_2(n) = \frac{ 1}{n}\sum_{i \in [n]} d_i^2$ and observe
\begin{equation}
 \frac{1}{n}\sum_{j \notin \mathscr{V}_{i-1}} d_{j}^{2} = \sigma_{2}(n)-\frac{1}{n} \sum_{j \in \mathscr{V}_{i-1}} d_{j}^{2}= \sigma_2(n) +o_{\sss\mathbb{P}}(1)
\end{equation} uniformly over $i \leq un^{2/3}$ where we use Lemma \ref{c1:lem_con_1} to conclude the uniformity. Similarly, \eqref{c1:sum-deg-explored} implies that $\sum_{j \notin \mathscr{V}_{i-1}} d_{j} = \ell_n+o_{\sss\mathbb{P}}(n)$ uniformly over $i \leq un^{2/3}$. Therefore,
\begin{equation}
n^{-1/3}\sum_{i=1}^{k}\frac{\sum_{j \notin \mathscr{V}_{i-1}} d_{j}^{2}\sum_{j \in \mathscr{V}_{i-1}} d_{j} }{\sum_{j \notin \mathscr{V}_{i-1}} d_{j}\sum_{j \in [n]} d_{j}} = \frac{n\sigma_{2}(n)+o_{\sss\mathbb{P}}(n)}{\ell_n+o_{\sss\mathbb{P}}(n)} n^{-1/3}\sum_{i=1}^{k}\frac{\sum_{j \in \mathscr{V}_{i-1}} d_{j}}{\ell_n}
\end{equation}
 and Assumption~\ref{c1:assumption1}, combined with \eqref{c1:lem_eq4}, completes the proof.
\end{proof}
 \begin{lemma}
 Condition \eqref{c1:condition1} holds.
 \end{lemma}
 \begin{proof}
  The proof follows by using  Lemma \ref{c1:lem_con_2} in  \eqref{c1:eq_cond_1}.
 \end{proof}
 \subsection{Finite-dimensional convergence of the ordered component sizes} Note that the convergence of the exploration process in Theorem~\ref{c1:thm_main1} implies that, for any large $T>0$, the $k$-largest components explored up to time $Tn^{2/3}$ converge to the  $k$-largest excursions above past minima of $\mathbf{B}^\lambda_{\mu,\eta}$ up to time~$T$. Therefore, we can conclude the finite dimensional convergence of the ordered components sizes in the whole graph if we can show that the large components are explored \emph{early} by the exploration process. The following lemma formalizes the above statement:
 \begin{lemma}\label{c1:lem:large-com-explored-early}Let $\mathscr{C}_{\max}^{\sss \geq T}$ denote the largest component which is started exploring after time $Tn^{2/3}$ in Algorithm~\ref{c1:algo:1}. Then, for any $\delta >0$,
 \begin{equation}\label{c1:large-com-explored-early}
  \lim_{T\to\infty}\limsup_{n\to\infty}\prob{|\mathscr{C}_{\max}^{\sss \geq T}|>\delta n^{2/3}}=0.
 \end{equation}
 \end{lemma}
 \noindent Let us first state the two main ingredients to complete the proof of Lemma~\ref{c1:lem:large-com-explored-early}: 
 \begin{lemma}[{\cite[Lemma 5.2]{J09b}}]\label{c1:lem:janson-lemma}  Consider $\mathrm{CM}_n(\boldsymbol{d})$ with $\nu_n<1$ and let $\mathscr{C}(V_n)$ denote the component containing the vertex $V_n$, where $V_n$ is a vertex chosen uniformly at random independently of the graph~$\mathrm{CM}_n(\boldsymbol{d})$. Then,
 \begin{equation}
  \expt{|\mathscr{C}(V_n)|}\leq 1+\frac{\expt{D_n}}{1-\nu_n}.
 \end{equation}  
 \end{lemma}

 \begin{lemma}\label{c1:lem-time-nu-rel}Define, $\nu_{n,i}=\sum_{j\notin \mathscr{V}_{i-1}}d_j(d_j-1)/\sum_{j\notin \mathscr{V}_{i-1}}d_j.$ There exists  some constant $C_0>0$ such that for any $T>0$,
 \begin{equation}\label{c1:nu-n-i:nu-n:relation}
 \nu_{n,Tn^{2/3}}= \nu_n-C_0 Tn^{-1/3}+o_{\sss\mathbb{P}}(n^{-1/3}).
\end{equation}
 \end{lemma}
 \begin{proof}
  Using a similar split up as in \eqref{c1:the-split-up}, we have
\begin{equation}\label{c1:nu-n-i:nu-n:split}
 \nu_{n,i}= \nu_n+\frac{\sum_{j \in \mathscr{V}_{i-1}}d_j(d_j-1)}{\ell_n}- \frac{\sum_{j\notin \mathscr{V}_{i-1}}d_j(d_j-1)\sum_{j\in \mathscr{V}_{i-1}}d_j}{\ell_n\sum_{j\notin \mathscr{V}_{i-1}}d_j}.
\end{equation}Now, \eqref{c1:lem_eq2} and \eqref{c1:lem_eq1} give that, uniformly over $i\leq Tn^{2/3}$,\
\begin{subequations}
\begin{equation}
 \frac{\sum_{j\notin \mathscr{V}_{i-1}}d_j(d_j-1)}{\sum_{j\notin \mathscr{V}_{i-1}}d_j}= \frac{\sum_{j\in [n]}d_j(d_j-1)+o_{\sss\mathbb{P}}(n^{2/3})}{\sum_{j\in [n]}d_j+o_{\sss\mathbb{P}}(n^{2/3})}=1+o_{\sss\mathbb{P}}(n^{-1/3}),
\end{equation}
\begin{equation}
 \sum_{j\in \mathscr{V}_{i-1}}d_j(d_j-2)=\Big(\frac{\sigma_3}{\mu}-2\Big)(i-1)+o_{\sss \mathbb{P}}(n^{2/3}).
\end{equation}
\end{subequations}Further, note that $\sigma_3-2\mu = \mathbb{E}[D(D-1)(D-2)]+ \mathbb{E}[D(D-2)]>0$, by Assumption~\ref{c1:assumption1}~\eqref{c1:assumption1-3}, and~\eqref{c1:assumption1-4}. Therefore, \eqref{c1:nu-n-i:nu-n:split} gives \eqref{c1:nu-n-i:nu-n:relation}.
 \end{proof}
 \begin{proof}[Proof of Lemma~\ref{c1:lem:large-com-explored-early}] Let $i_{\sss T}:=\inf\{\tau_k:\tau_k>Tn^{2/3}\} $, where $\tau_k$ is defined by \eqref{c1:eq:tau-k}. 
 Thus, $i_{\sss T}$ denotes the first time we finish exploring a component after time $Tn^{2/3}$. Note that, conditional on the explored vertices up to time $i_{\sss T}$, the remaining graph $\bar{\mathcal{G}}$ is still a configuration model. Let $\bar{\nu}_n=\sum_{i\in \bar{\mathcal{G}}}d_i(d_i-1)/\sum_{i\in \bar{\mathcal{G}}}d_i$ be the criticality parameter of $\bar{\mathcal{G}}$. Then, using \eqref{c1:nu-n-i:nu-n:relation}, we can conclude that 
 \begin{equation}\label{c1:eq:bar-nu-n}
 \bar{\nu}_n\leq \nu_n- C_0Tn^{-1/3}+\oP(n^{-1/3}).
 \end{equation} 
Thus, as we explore more vertices, the graph becomes more subcritical. 
 Take $T>0$ such that $\lambda -C_0T <0$. Thus, with high probability, $\bar{\nu}_n<1$. 
 Denote the component corresponding to a randomly chosen vertex from $\bar{\mathcal{G}}$ by $\mathscr{C}^{\sss \geq T}(V_n)$, and the $i^{\sss th}$ largest component of $\bar{\mathcal{G}}$ by $\mathscr{C}_{\sss (i)}^{\sss \geq T}$. 
 Also, let $\bar{\PR}$ denote the probability measure conditioned on $\mathscr{F}_{i_{\sss T}}$, and let $\bar{\E}$ denote the corresponding expectation. 
 Now, for any $\delta >0$,
 \begin{equation}
  \begin{split} 
  \bar{\PR}\bigg( \sum_{i\geq 1}&|\mathscr{C}_{\sss (i)}^{\sss \geq T}|^2>\delta^2 n^{4/3}  \bigg)\leq \frac{1}{\delta^2 n^{4/3}}\sum_{i\geq 1}\bar{\E}\big( |\mathscr{C}^{\sss \geq T}_{\sss(i)}|^2\big)\\
  &\leq  \frac{1}{\delta^2 n^{1/3}}\bar{\E}\big( |\mathscr{C}^{\sss \geq T}(V_n)|\big)\leq \frac{1}{\delta^2(-\lambda+C_0T+\oP(1))},
  \end{split}
\end{equation} where the second step follows from the Markov inequality and the last step follows by combining Lemma~\ref{c1:lem:janson-lemma} and \eqref{c1:eq:bar-nu-n}. Noting that $\bar{\nu}_n<1$ with high probability, we get
\begin{equation}
 \limsup_{n\to\infty}\prob{|\mathscr{C}_{\max}^{\sss \geq T}|>\delta n^{2/3}}\leq \frac{C}{\delta^2 T},
\end{equation} for some constant $C>0$ and large $T>0$ and the proof follows.
 \end{proof}
 \begin{theorem}\label{c1:thm:conv-fd-comp}
  The convergence in Theorem~\ref{c1:thm_main} holds with respect to the product topology.
 \end{theorem}
 \begin{proof}
  The proof follows from Theorem~\ref{c1:thm_main1} and Lemma~\ref{c1:lem:large-com-explored-early}.
 \end{proof}
 \subsection{{Proof of Theorem~\ref{c1:thm_main}}}  \label{c1:sec_l2_tightness}
 The proof of Theorem~\ref{c1:thm_main} follows  using a similar argument as \cite[Section 3.3]{A97}. 
 However, the proof is a bit tricky since the components are explored in a size-biased manner with sizes being the total degree in the components (not the component sizes as in \citep{A97}). 
 For a sequence of random variables $\mathbf{Y}=(Y_i)_{i\geq 1}$ satisfying $\sum_{i\geq 1}Y_i^2<\infty$ almost surely, define $\boldsymbol{\xi}:=(\xi_i)_{i\geq 1}$ such that $\xi_i|\mathbf{Y}\sim \mathrm{Exp}(Y_i)$ and the coordinates of $\boldsymbol{\xi}$ are independent conditional on $\mathbf{Y}$.  
 For $a\geq 0$, let $\mathscr{S}(a):=\sum_{\xi_i\leq a}Y_i$. Then the \emph{size biased point process} is defined to be the random collection of points $\Xi:=\{(\mathscr{S}(\xi_i),Y_i)\}_{i\geq 1}$ (see \cite[Section 3.3]{A97}).  
 We will use Lemma 8, Lemma 14 and Proposition 15 from \cite{A97}. 
 Let $\mathfrak{C}:=\{\mathscr{C}: \mathscr{C}\text{ is a component of }\mathrm{CM}_n(\boldsymbol{d})\}$. 
 Consider the collection $\boldsymbol{\xi}:=(\xi(\mathscr{C}))_{\mathscr{C}\in \mathfrak{C}}$ such that conditional on $(\sum_{k\in \mathscr{C}}d_k, |\mathscr{C}|)_{\mathscr{C}\in \mathfrak{C}}$, $\xi(\mathscr{C})$ has an exponential distribution with rate $n^{-2/3}\sum_{k\in \mathscr{C}}d_k$ independently over $\mathscr{C}$. Then the order in which Algorithm~\ref{c1:algo:1} explores the components can be obtained by ordering the components according to their $\xi$-value.  Recall that $\mathscr{C}_i$ denotes the $i^{th}$ explored component by Algorithm~\ref{c1:algo:1} and let $D_i:=\sum_{k\in\mathscr{C}_i}d_k$. Define the size-biased point process
 \begin{equation}\Xi_n:=\Big(n^{-2/3} \sum_{j=1}^{i} D_j  , \hspace{.2cm}n^{-2/3} D_i \Big)_{i \geq 1}. 
 \end{equation}  Also define the point processes
 \begin{gather*} 
 \Xi_{n}^{'} := \Big( n^{-2/3} \sum_{j=1}^i \big| \mathscr{C}_j \big|, \hspace{.2cm} n^{-2/3} \big| \mathscr{C}_i\big|  \Big)_{i \geq 1}, \\
  \Xi_{\infty} := \big\{ \big( l(\gamma), |\gamma| \big):\text{ }\gamma \text{ an excursion of } \mathbf{B}^\lambda_{\mu,\eta} \big\}, 
 \end{gather*} where we recall that $l(\gamma)$ are the left endpoints of the excursions of $\mathbf{B}^\lambda_{\mu,\eta}$ and $|\gamma|$ is the length of the excursion $\gamma$ (see \eqref{c1:defn::reflected-BM}). Note that $\Xi_n'$ is not a size-biased point process. However, applying \cite[Lemma 8]{A97} and Theorem \ref{c1:thm_main1}, we get $ \Xi_{n}^{'} \xrightarrow{\sss \mathcal{L}} \Xi_{\infty}$.  We claim that 
 \begin{equation}\label{c1:conv:SBPP}
 \Xi_n \xrightarrow{\mathcal{L}} 2 \Xi_{\infty}.
 \end{equation}
 To verify the claim, note that \eqref{c1:lem_eq2}  and Assumption~\ref{c1:assumption1}~\eqref{c1:assumption1-3} together imply, for any $t>0$,
 \begin{equation}\label{c1:eqn:deg-close-comp}
  \sup_{u \leq t} \big| n^{-2/3} \sum_{i=1}^{\lfloor un^{2/3} \rfloor} d_{\sss(i)} - \frac{\sigma_2}{\mu}u \big|
  = \sup_{u \leq t} \big| n^{-2/3} \sum_{i=1}^{\lfloor un^{2/3} \rfloor} d_{\sss(i)} - 2u \big| \xrightarrow{\mathbb{P}} 0,
 \end{equation}since $\sigma_2/\mu =\E[D^2]/\E[D]=2$. 
  Thus, \eqref{c1:conv:SBPP} follows using \eqref{c1:eqn:deg-close-comp}.
 Now, the point process $2 \Xi_{\infty}$ satisfies all the conditions of \cite[Proposition 15]{A97} as shown by Aldous. 
 Thus, \cite[Lemma 14]{A97} gives
  \begin{align} \label{c1:eqn_tightness}
	\big\{ D_{\sss (i)}\big\}_{i \geq 1}\text{ is tight in } \ell^2_{\shortarrow}.
  \end{align} 
This implies that $\big( n^{-2/3} \big| \mathscr{C}_{\sss (i)} \big| \big)_{i \geq 1}$ is tight in $\ell^2_{\shortarrow}$ by simply observing that $|\mathscr{C}_i|\leq \sum_{k\in\mathscr{C}_i}d_k+1$. Therefore, the proof of Theorem~\ref{c1:thm_main} is complete using Theorem~\ref{c1:thm:conv-fd-comp}.
\qed

\subsection{{Proof of Theorem~\ref{c1:thm_surplus}}} \label{c1:sec_surplus_edges}
The proof of Theorem~\ref{c1:thm_surplus} is completed in two separate lemmas below. 
In Lemma~\ref{c1:lem:surp:poisson-conv} we first show that the convergence in Theorem~\ref{c1:thm_surplus} holds with respect to the $\ell^2_{\shortarrow}\times \mathbb{N}^{\infty}$ topology. The tightness of $(\mathbf{Z}_n)_{n\geq 1}$ with respect to the $\mathbb{U}^0_{\shortarrow}$ topology is ensured in Lemma~\ref{c1:sufficient-U0-conv-condn}.

 \begin{lemma} \label{c1:lem:surp:poisson-conv} Let $N_n^\lambda(k)$ be the number of surplus edges discovered up to time $k$ and $\bar{N}^\lambda_n(u) = N_n^\lambda(\lfloor un^{2/3} \rfloor)$. Then, as $n\to\infty$,
 \begin{equation}\bar{\mathbf{N}}_n^\lambda\dto \mathbf{N}^\lambda,
 \end{equation} where $\mathbf{N}^\lambda$ is defined in \eqref{c1:defn::counting-process}.
 \end{lemma}
 \begin{proof}
 Recall the definitions of $a$, $b$, $\mathcal{A}_k$, $\mathcal{B}_k$, $\mathcal{C}_k$, $\mathcal{S}_k$ from Section \ref{c1:exploration}. 
 Recall also that $ A_k:= \big| \mathcal{A}_k \big|$, $ B_k:= \big| \mathcal{B}_k \big|$, $ C_k:= \big| \mathcal{C}_k \big|$, $ U_k:= \big| \mathcal{S}_k \big|$, $c_{(k+1)}:= (\big| \mathcal{B}_{k} \big| +\big| \mathcal{C}_{k} \big|)/2 $ from Section~\ref{c1:exploration}. Notice that $A_k=S_n(k)-\min_{j \leq k} S_n(j)$. 
 From Lemma~\ref{c1:lem_back_edges}, we can conclude that, uniformly over $k\leq un^{2/3}$,
 \begin{equation} \label{c1:eqn_surplus_intensity}
  \mathbb{E} \big[ c_{(k+1)} \vert \mathscr{F}_k \big] = \frac{A_k}{\mu n}+ O_{\sss\mathbb{P}}(n^{-1}).
 \end{equation}
 The counting process $\mathbf{N}_n^\lambda$ has conditional intensity (conditioned on $\mathscr{F}_{k -1}$) given by \eqref{c1:eqn_surplus_intensity}. Writing the conditional intensity in  \eqref{c1:eqn_surplus_intensity} in terms of $\bar{\mathbf{S}}_n$, we get that the conditional intensity of the re-scaled process $\bar{\mathbf{N}}^\lambda_n$ is given by 
 \begin{equation} \label{c1:rate:surplus:scaled}
 \frac{1}{\mu} [\bar{S}_n(u)-\min_{\tilde{u} \leq u} \bar{S}_n(\tilde{u})]+ o_{\sss \PR}(1).
 \end{equation} 
 Denote by  $\bar{W}_n(u):=\bar{S}_n(u)-\min_{\tilde{u} \leq u} \bar{S}_n(\tilde{u})$ which is  the reflected version of~$\bar{\mathbf{S}}_n$. By Theorem~\ref{c1:thm_main},  
 \begin{equation}\bar{\mathbf{W}}_n\dto\mathbf{W}^\lambda,
 \end{equation} where $\mathbf{W}^\lambda$ is defined in \eqref{c1:defn::reflected-BM}. 
 Therefore, we can assume that there exists a probability space such that $\bar{\mathbf{W}}_n\to\mathbf{W}^\lambda$ almost surely. Using \cite[Theorem 1; Chapter 5.3]{LS89}, and the continuity of the sample paths of $\mathbf{W}^\lambda$, we conclude the proof.
\end{proof} 
\begin{lemma}\label{c1:sufficient-U0-conv-condn} The vector $(\mathbf{Z}_n)_{n\geq 1}$ is tight with respect to the $\mathbb{U}^0_{\shortarrow}$ topology.
\end{lemma}
The proof of Lemma~\ref{c1:sufficient-U0-conv-condn} makes use of the following crucial estimate of the probability that a component with small size has a very large number of surplus edges:
\begin{lemma} \label{c1:lem_surplus_delta_bound}
Assume that $\lambda <0.$ Let $V_n$ denote a vertex chosen uniformly at random, independent of the graph $\mathrm{CM}_n(\boldsymbol{d})$ and let $\mathscr{C}(V_n)$ denote the component containing $V_n$.  Let $\delta_k=\delta k^{-0.12}$. Then, for $\delta > 0$ (small),
\begin{equation}
 \prob{\surp{\mathscr{C}(V_n)}\geq K,|\mathscr{C}(V_n)|\in (\delta_K n^{2/3},2\delta_Kn^{2/3})}\leq \frac{C\sqrt{\delta}}{n^{1/3}K^{1.1}},
\end{equation}
 where $C$ is a fixed constant independent of $n,\delta, K$.
\end{lemma}
\begin{proof}[Proof of Lemma~\ref{c1:sufficient-U0-conv-condn}]
To simplify the notation, we write $Y_i^n=n^{-2/3} |\mathscr{C}_{\sss (i)}|$ and $N_i^n=$\# $\{$surplus edges in $\mathscr{C}_{\sss(i)}\}$. Let $Y_i$, $N_i$ denote the distributional limits of $Y_i^n$ and $N_i^n$ respectively. Recall from Remark~\ref{c1:defn_U_0_process} that $\mathbf{Z}(\lambda)$ is almost surely $\mathbb{U}^0_{\shortarrow}$-valued. Using Lemma~\ref{c1:lem:surp:poisson-conv}, the proof of Lemma~\ref{c1:sufficient-U0-conv-condn} is complete if we can show that, for any $\eta >0$
 \begin{equation} \label{c1:eqn_sufficient_for_U_0_convergence}
 \lim_{\varepsilon\to 0}\limsup_{n\to\infty}\PR\bigg( \sum_{Y_i^n\leq \varepsilon} Y_i^n N_i^n> \eta \bigg)=0.
 \end{equation} 
First, consider the case $\lambda <0$. For every $\eta,\varepsilon >0$ sufficiently small 
\begin{align}
  &\PR\bigg( \sum_{Y_i^n\leq \varepsilon} Y_i^n N_i^n> \eta \bigg)\leq \frac{1}{\eta}\E \bigg[\sum_{i=1}^{\infty}Y_i^n N_i^n \1_{\{ Y_i^n\leq \varepsilon\}} \bigg]\\
  &= \frac{n^{-2/3}}{\eta} \E \bigg[\sum_{i=1}^{\infty}|\mathscr{C}_{\sss(i)}| N_i^n \1_{\{ |\mathscr{C}_{\sss(i)}|\leq \varepsilon n^{2/3}\}} \bigg]= \frac{n^{1/3}}{\eta}\expt{\mathrm{SP}(\mathscr{C}(V_n))\1_{\{ |\mathscr{C}(V_n)|\leq \varepsilon n^{2/3}\}}}\nonumber\\
  &= \frac{n^{1/3}}{\eta}\sum_{k=1}^{\infty}\sum_{i\geq \log_2(1/(k^{0.12}\varepsilon))}\PR\bigg(\mathrm{SP}(\mathscr{C}(V_n))\geq k, |\mathscr{C}(V_n)|\in \bigg(\frac{n^{2/3}}{2^{i+1}k^{0.12}},\frac{n^{2/3}}{2^{i}k^{0.12}} \bigg] \bigg)\nonumber\\
  &\leq \frac{C}{\eta} \sum_{k=1}^{\infty}\frac{1}{k^{1.1}}\sum_{i\geq \log_2(1/(k^{0.12}\varepsilon))} 2^{-(1/2)i} \leq \frac{C}{\eta}\sum_{k=1}^{\infty}\frac{\sqrt{\varepsilon}}{k^{1.04}}  =O(\sqrt{\varepsilon}),
 \end{align}
  where we have used Lemma~\ref{c1:lem_surplus_delta_bound}. Therefore, \eqref{c1:eqn_sufficient_for_U_0_convergence} holds when $\lambda <0$. Now consider the case $\lambda >0$.  For $T>0$ (large), let \begin{equation}
 \mathcal{K}_n:=\{i: Y_i^n\leq \varepsilon, \mathscr{C}_{\sss (i)} \text{ is explored before }Tn^{2/3}\}.
\end{equation}Then, by applying the Cauchy-Schwarz inequality,
\begin{equation}\label{c1:eq:SP-C-T}
\begin{split}
\sum_{i\in \mathcal{K}_n}&Y_i^nN_i^n\leq \Big( \sum_{i\in \mathcal{K}_n}(Y_i^n)^2\Big)^{1/2}\times  \Big( \sum_{i\in \mathcal{K}_n}(N_i^n)^2\Big)^{1/2}\\
&\leq \Big( \sum_{i\in \mathcal{K}_n}(Y_i^n)^2\Big)^{1/2}\times (\# \text{ surplus edges  explored before }Tn^{2/3})
\end{split}
\end{equation}
For the case $\lambda >0$, we can use similar ideas as the proof of Lemma~\ref{c1:lem:large-com-explored-early}, i.e., we can run the exploration process till $Tn^{2/3}$ and the unexplored graph becomes a configuration model with negative criticality parameter for large $T>0$, by \eqref{c1:nu-n-i:nu-n:relation}.  Thus, the proof can be completed using \eqref{c1:eq:SP-C-T}, the $\ell^{2}_{\shortarrow}$ convergence of the component sizes given by Theorem~\ref{c1:thm_main} and Lemma~\ref{c1:lem:surp:poisson-conv}, and the proof for the case $\lambda<0$.
\end{proof}
\begin{proof}[Proof of Lemma~\ref{c1:lem_surplus_delta_bound}] 
To complete the proof of Lemma~\ref{c1:lem_surplus_delta_bound}, we will use martingale techniques coupled with Lemma~\ref{c1:lem:janson-lemma}. Fix $\delta > 0$ (small). First we describe another way of exploring $\mathscr{C}(V_n)$ which turns out to be convenient to work with.
 \begin{algo}[Exploring components of uniform vertices]\label{c1:algo:2}\normalfont Consider the following exploration of $\mathscr{C}(V_n)$: \begin{itemize}
 \item[(S0)] Initialize all half-edges to be alive. Choose a vertex from $[n]$ uniformly at random and declare all its half-edges active.
 \item[(S1)] In the next step, take any active half-edge and pair it uniformly with another alive half-edge. Kill these paired half-edges. Declare all the half-edges corresponding to the new vertex (if any) active. Keep repeating (S1) until the set of active half-edges is empty.
\end{itemize}
\end{algo}Unlike Algorithm~\ref{c1:algo:1}, we need not see a new vertex at each stage and we explore only two half-edges at each stage. In this proof, $\mathscr{F}_l$ denotes the sigma-field containing information revealed up to stage~$l$ by Algorithm~\ref{c1:algo:2} and $\mathscr{V}_l$ denotes the vertex set discovered up to time $l$. Recall that we denote by $D_n$ the degree of $V_n$. Define the exploration process~$\mathbf{s}_n'$ by,
\begin{equation}
 s_n'(0)=D_n,\ s_n'(l)= \sum_{i\in [n]} d_i\mathcal{I}_i^n(l)-2l,
\end{equation} where $\mathcal{I}_i^n(l)= \ind{i\in \mathscr{V}_l}$. 
Therefore, $s_n'(l)$ counts the number of active half-edges at time $l$, until $\mathscr{C}(V_n)$ is explored. Note that $\mathscr{C}(V_n)$ is explored when $\mathbf{s}'_n$ hits zero and the hitting time to zero gives the number of edges in $\mathscr{C}(V_n)$, since exactly one edge is being explored at each time step.
We will use a generic constant $C$ to denote a positive constant that can be different in different equations.  For $H>0$,  let \begin{equation} \label{c1:defn:gamma}
\gamma := \inf \{ l\geq 1: s_n'(l)\geq H \text{ or }  s_n'(l)= 0 \}\wedge 2\delta n^{2/3}.
\end{equation} Note that
\begin{equation}\label{c1:exploration:super_martingale}
\begin{split}
 \E[s_n'(l+1)&-s_n'(l)\vert \mathscr{F}_l]= \sum_{i\in [n]}d_i\prob{i\in \mathscr{V}_{l+1}\vert \mathscr{F}_l,\mathcal{I}_i^n(l) = 0} -2\\
 &= \frac{ \sum_{i\notin \mathscr{V}_l}d_i^2}{\ell_n-2l-1}-2\leq \frac{ \sum_{i\in [n]}d_i^2}{\ell_n-2l-1}-2\\
 & =\frac{\lambda}{n^{1/3}}+o(n^{-1/3})+\frac{2l+1}{\ell_n-2l-1}\times \frac{\sum_{i\in [n]}d_i^2}{\ell_n}   \leq 0
\end{split}
\end{equation} uniformly over $l\leq 2\delta n^{2/3}$ for all small $\delta >0$ and large $n$, where the last step follows from the fact that $\lambda<0$. Therefore, $\{s_n'(l)\}_{l= 1}^{2\delta n^{2/3}}$ is a super-martingale. The optional stopping theorem now implies
  \begin{equation}
   \mathbb{E}\left[D_n\right] \geq \mathbb{E}\left[s_n'(\gamma)\right] \geq H \mathbb{P}\left( s'_n(\gamma) \geq H \right).
  \end{equation} Thus,
  \begin{equation} \label{c1:eqn::bound_geq_H_at_stopping_time}
    \mathbb{P}\left( s'_n(\gamma) \geq H \right) \leq \frac{\expt{D_n}}{H}.
  \end{equation}
We put $H=n^{1/3}K^{1.1}/\sqrt{\delta}$. To simplify the notation, we write $s_n'[0,t]\in A$ to denote that $s_n'(l)\in A,$ for all $ l\in [0,t]$.  Notice that, for $K\geq 1$,
 \begin{equation}\label{c1:surp:sup:less}\begin{split}
  &\prob{\surp{\mathscr{C}(V_n)}\geq K,|\mathscr{C}(V_n)|\in (\delta_K n^{2/3},2\delta_Kn^{2/3})}\\
  &\leq \prob{s_n'(\gamma)\geq H}\\
  & \hspace{.5cm}+\prob{\surp{\mathscr{C}(V_n)}\geq K, s_n'[0,2\delta_K n^{2/3}]< H, s_n'[0,\delta_K n^{2/3}]>0}.
  \end{split}
 \end{equation}
 Here we have used the fact that if there is at least one surplus edge in $\mathscr{C}(V_n)$, the number of edges in $\mathscr{C}(V_n)$ is at least $\mathscr{C}(V_n)$. Therefore, $|\mathscr{C}(V_n)|>\delta_Kn^{2/3}$ implies $s_n'[0,\delta_K n^{2/3}]>0$.
 Let us denote the event that surplus edges appear at times  $l_1,\dots,l_K$,  $s_n'[0,2\delta_K n^{2/3}]< H$, and $s_n'[0,\delta_K n^{2/3}]>0$ by $\mathrm{SPB}(l_1,\dots,l_K)$.
   Now,
 \begin{equation}
 \begin{split}
   &\prob{\surp{\mathscr{C}(V_n)}\geq K, s_n'[0,2\delta_K n^{2/3}]< H, s_n'[0,\delta_K n^{2/3}]>0}\\
  &\hspace{.5cm}\leq \sum_{1\leq l_1<\dots< l_K\leq 2\delta_K n^{2/3}} \prob{\mathrm{SPB}(l_1,\dots,l_K)}\\
  &\hspace{.5cm}=\sum_{1\leq l_1<\dots<l_K\leq 2\delta_K n^{2/3}}\expt{\ind{0<s_n'[0,l_K-1]<H, \mathbf{SP}(l_K-1)=K-1}Y},
 \end{split}
 \end{equation}
 where
 \begin{align}
  Y&=\prob{K^{th}\text{ surplus occurs at }l_K,  s_n'[l_K,2\delta_K n^{2/3}]< H, s_n'[l_K,\gamma]>0\mid \mathscr{F}_{l_K-1} }\nonumber\\
  &\leq \frac{CK^{1.1}n^{1/3}}{\ell_n\sqrt{\delta}}\leq \frac{CK^{1.1}}{n^{2/3}\sqrt{\delta}}.
 \end{align}
 Therefore, using induction, 
 \begin{equation}\label{c1:exploration:bounded:surplus}
 \begin{split}
  &\prob{\surp{\mathscr{C}(V_n)}\geq K, s_n'[0,2\delta_K n^{2/3}]< H, s_n'[0,\delta_K n^{2/3}]>0}\\
  &\hspace{1cm}\leq C\bigg( \frac{K^{1.1}}{\sqrt{\delta}n^{2/3}}\bigg)^K\frac{(2\delta n^{2/3})^{K-1}}{K^{0.12(K-1)}(K-1)!}\sum_{l_1=1}^{2\delta_K n^{2/3}}\prob{|\mathscr{C}(V_n)|\geq l_1}\\
  &\hspace{1cm}\leq C \frac{\delta^{K/2}}{K^{1.1}n^{2/3}}  \expt{|\mathscr{C}(V_n)|},
  \end{split}
 \end{equation}where we have used the fact that $$\#\{1\leq l_2<\dots<l_k\leq 2\delta n^{2/3}\}\leq(2\delta n^{2/3})^{K-1}/(K-1)!$$ and the Stirling approximation for $(K-1)!$ in the last step. Since $\lambda <0$, we can use Lemma~\ref{c1:lem:janson-lemma} to conclude that for all sufficiently large $n$
 \begin{equation} \label{c1:expectation:random:vert:comp}
  \expt{|\mathscr{C}(V_n)|}\leq Cn^{1/3},
 \end{equation} for some constant $C>0$ and we get the desired bound for \eqref{c1:surp:sup:less}.
  The proof of Lemma~\ref{c1:lem_surplus_delta_bound} is now complete by applying \eqref{c1:eqn::bound_geq_H_at_stopping_time}~and~\eqref{c1:exploration:bounded:surplus} in \eqref{c1:surp:sup:less}.
\end{proof}

\section{{Degree distribution within components}} \label{c1:sec_vertex}
In this section, we compute the number of vertices of degree $k$ in each connected component at criticality. This will be useful in  Sections~\ref{c1:sec_percolation} and \ref{c1:sec_multidimensional}. Such an estimate was proved in \cite[Theorem 2.4]{JL09} for supercritical graphs under stronger moment assumptions.
\begin{lemma}\label{c1:lem:sec:vertex} Denote by $N_k(t)$ the number of vertices of degree $k$ discovered up to time $t$. For any $t>0$, uniformly over $k$,
  \begin{equation}
   \sup\limits_{u \leq t} \big| n^{-2/3} N_k(un^{2/3})-\frac{kn_k}{\ell_n} u \big| =  O_{\sss\mathbb{P}}((kn^{1/3})^{-1}).
  \end{equation}
 \end{lemma}
 \begin{proof}
  By setting $w_i=\mathbf{1}_{ \{ d_i=k \} }$ in Lemma~\ref{c1:lem_general}  we can directly conclude that 
  \begin{equation} \sup\limits_{u \leq t} \big| n^{-2/3} N_k(un^{2/3})-\frac{kn_k}{\ell_n} u \big|\xrightarrow{\mathbb{P}}0.
  \end{equation} However, one can repeat same arguments as leading to the proof of Lemma~\ref{c1:lem_general} and obtain that
  \begin{eq} \label{c1:lem_martingale_eqn}
  &\mathbb{P} \Big( \sup\limits_{u \leq t} \Big| n^{-2/3} N_k(un^{2/3})-\frac{kn_k}{\ell_n} u \Big| > \frac{A}{kn^{1/3}} \Big)  \\
  &\hspace{1cm} \leq \frac{3\Big( k^3 s^2 \frac{r_k}{( \mathbb{E}[D] )^2 } + \sqrt{  s \frac{k^3r_k}{\mathbb{E}[D]}} \Big)}{A}+ o(1).
  \end{eq} Now, we can use the finite third-moment assumption to conclude that the numerator in the right hand side can be taken to be uniform over $k$. Thus, the proof follows.
 \end{proof}
  Define $v_k(G):=$ the number of vertices of degree $k$ in the connected graph $G$.
As a corollary to Lemma~\ref{c1:lem:sec:vertex} and  \eqref{c1:large-com-explored-early}, we can deduce  that
  \begin{equation}\label{c1:eqn_vertices_of_degree_k-ord}
   v_k \big( \mathscr{C}_{\sss(j)} \big) = \frac{kr_k}{\mathbb{E}[D]} \big| \mathscr{C}_{\sss (j)} \big| +O_{\sss\mathbb{P}}\big((k^{-1}n^{1/3})\big).
  \end{equation}Moreover, the following holds: Let $\mathrm{ord}(\boldsymbol{x})$ denote the vector with elements of $\boldsymbol{x}$ ordered in a non-increasing manner. 
\begin{lemma} \label{c1:vertices_of_degree_k_ord}
For each $k\geq 1$ denote by $\mathbf{V}_k^n:=(n^{-2/3}v_k ( \mathscr{C}_j ))_{j\geq 1}$. Then, the sequence $\{\mathrm{ord}(\mathbf{V}_k^n)\}_{n\geq 1}$ is tight in $\ell^2_{\shortarrow}$.
\end{lemma}
\begin{proof}
 Note that for any $j\geq 1$, $v_k(\mathscr{C}_{\sss(j)})\leq |\mathscr{C}_{\sss(j)}|$ uniformly over $k$. The proof now follows from \eqref{c1:eqn_vertices_of_degree_k-ord} and $\ell^2_{\shortarrow}$ tightness of the component sizes given in Theorem~\ref{c1:thm_main}.
\end{proof}

\begin{remark}\normalfont Define $\mathbf{V}^n:=(n^{-2/3}v_k(\mathscr{C}_j))_{k,j\geq 1}$. Then $\{\mathrm{ord}(\mathbf{V}^n)\}_{n\geq 1}$ is also tight in $\ell^2_{\shortarrow}$.
\end{remark}

\section{Critical percolation}
\label{c1:sec_percolation}

Let $p = p_n \in (0,1)$ be the percolation parameter. Recall the notation $\mathrm{CM}_{n}(\boldsymbol{d},p)$ for the random graph obtained after deleting edges of $\mathrm{CM}_{n}(\boldsymbol{d})$ independently with probability $1-p$. Suppose, $\boldsymbol{d}'$ is the random degree sequence obtained after percolation. 
Fountoulakis~\cite{F07} showed that, given $\boldsymbol{d}'$, the law of $\mathrm{CM}_{n}(\boldsymbol{d},p)$  is same as the law of $\mathrm{CM}_{n}(\boldsymbol{d}')$.  We will use the following construction of $\mathrm{CM}_{n}(\boldsymbol{d},p)$ due to Janson~\cite{J09}:
\begin{algo} \label{c1:algo:3}
\normalfont \begin{itemize}
 \item[(S1)] For each half-edge $e$, let $v_e$ be the vertex to which $e$ is attached. With probability $1-\sqrt{p}$, one detaches $e$ from $v_e$ and associates $e$ to a new vertex $v'$. Color the new vertex $red$. This is done independently for every existing half-edge. Let $n_+$ be the number of red vertices created and $\tilde{n}=n+n_+$.  Suppose, $\Mtilde{\boldsymbol{d}} = ( \tilde{d}_i )_{i \in [\tilde{n}]}$ is the new degree sequence obtained by the above procedure, i.e. $\tilde{d}_i \sim \text{Bin} (d_i, \sqrt{p})$ for $i \in [n]$ and $\tilde{d}_i=1$ for $i \in [\tilde{n}] \setminus [n]$.
 \item[(S2)] Construct $\mathrm{CM}_{\tilde{n}}(\Mtilde{\boldsymbol{d}})$, independently of (S1).
 \item[(S3)] Delete all the red vertices.
 \end{itemize}
 \end{algo}
\begin{remark}\label{c1:rem:alt-s3}
 \normalfont
 It was argued in \cite{J09} that the obtained multigraph also has the same distribution as $\mathrm{CM}_{n}(\boldsymbol{d},p)$ if we replace (S3) by
 \begin{itemize}
 \item[(S3$'$)] Instead of deleting red vertices, choose any $n_+$ degree-one vertices uniformly at random, independently of (S1) and (S2), and delete them.
 \end{itemize}
\end{remark}
\begin{remark}\normalfont The construction of $\mathrm{CM}_{\tilde{n}}(\Mtilde{\boldsymbol{d}})$ in Algorithm~\ref{c1:algo:3} consists of two stages of randomization, the first one is described by (S1), and the second one by (S2). We will consider the following probability space to describe the randomization arising from Algorithm~\ref{c1:algo:3}~(S1): Suppose we have a sequence of degree sequences $(\boldsymbol{d})_{n\geq 1}$. 
Let  $\mathbb{P}_p^n$ denote the  probability measure induced on $\mathbb{N}^{\infty}$ by Algorithm~\ref{c1:algo:3}~(S1). 
Denote the product measure of $(\PR_p^n)_{n\geq 1}$ by $\mathbb{P}_p$. Thus (S1) is performed independently on $\boldsymbol{d}=\boldsymbol{d}(n)$ as $n$ varies.  All the almost sure statements in this section will be with  respect to the probability measure $\mathbb{P}_p$. 
\end{remark}
\begin{remark} \label{c1:remark-perc} \normalfont
The idea of the proof of Theorem~\ref{c1:thm_percolation} is as follows.  We show that $\Mtilde{\boldsymbol{d}}$, under Assumption~\ref{c1:assumption2}, satisfies Assumption~\ref{c1:assumption1} $\mathbb{P}_p$ almost surely and then estimate the number of vertices to be deleted from each component using Lemma~\ref{c1:lem:sec:vertex}. 
Since deleting a degree-one vertex does not break up any component, we can just subtract this from the component sizes of $\mathrm{CM}_{\tilde{n}}(\Mtilde{\boldsymbol{d}})$ to get the component sizes of $\mathrm{CM}_{n}(\boldsymbol{d},p_n(\lambda))$. 
Since the degree-one vertices do not get involved in surplus edges, deleting degree-one vertices does not change the number of surplus edges.
\end{remark}
\subsection{{Proof of Theorem~\ref{c1:thm_percolation}}}
We now consider the critical window corresponding to percolation. The goal is to prove Theorem~\ref{c1:thm_percolation}. Let $n_j$ and $\tilde{n}_j$ be the number of vertices of degree $j$ before and after performing Algorithm~\ref{c1:algo:3}~(S1) respectively. Further let 
\begin{equation}\tilde{\nu}_n = \frac{\sum_{i \in [\tilde{n}]}\tilde{d}_{i}\big( \tilde{d}_i-1 \big)}{\sum_{i \in [\tilde{n}]}\tilde{d}_{i}}.
\end{equation} 
For convenience we write $r_j=\mathbb{P}(D=j)$. Denote by $\tilde{n}_{jl}$, the number of vertices that had degree $l$ before and have degree $j$ after performing Algorithm~\ref{c1:algo:3}~(S1). Therefore, $\tilde{n}_{jl} \sim \text{Bin}\big( n_l,b_{lj}(\sqrt{p_n}) \big) $, where $b_{lj}(\sqrt{p_n})= \binom{l}{j} (\sqrt{p_n})^j (1-\sqrt{p_n})^{l-j}$. Using the strong law of large numbers for triangular arrays, note that $\mathbb{P}_p$ almost surely,
$\tilde{n}_{jl} = n_l b_{lj}(\sqrt{p_n}) + o(n_l) = nr_l b_{lj}(\sqrt{p_n}) + o(n_l).$  Now, $\sum_{l\geq 1}|n_l/n-r_l|\to 0$ and therefore, for all $j \geq 2$,  $\mathbb{P}_p$ almost surely
 \begin{equation}\label{c1:estimate:n-j-tilde}
  \frac{\tilde{n}_j}{n} = \frac{\sum_{l=j}^{\infty}\tilde{n}_{jl}}{n} = \sum_{l=j}^{\infty}r_lb_{lj}(\sqrt{p}_n)+ o(1).
 \end{equation}
 Also, $n_+ = \sum_{i \in [n]}\big( d_i - \tilde{d}_i \big) \sim \text{Bin}(\ell_n, 1-\sqrt{p_n})$. Therefore, using similar arguments as  \eqref{c1:estimate:n-j-tilde} again, $\mathbb{P}_p$ almost surely,
 \begin{equation} \label{c1:estimate_n+}
 \begin{split}
  \frac{n_+}{n} & = \mathbb{E}(D) \big( 1- \sqrt{p_n} \big) + o(1),  			
  \end{split}
 \end{equation}
\begin{equation} \label{c1:estimate_n_1_tilde}
\frac{\tilde{n}_{1}}{n} =\frac{\sum_{l=1}^{\infty}\tilde{n}_{1l}+n_+}{n}= \frac{\sum_{l=1}^{\infty}\tilde{n}_{1l}}{n}+\mathbb{E}(D) \big( 1- \sqrt{p_n} \big) + o(1),
\end{equation}
and
 \begin{equation}\label{c1:limit-n-tilde}
  \frac{\tilde{n}}{n} =1+\frac{n_+}{n}=1+\mathbb{E}(D) \big( 1- \sqrt{p_n} \big) + o(1).
 \end{equation}
Denote $\tilde{r}_l = \mathbb{P}( \tilde{D}=l )= \lim_{n \to \infty} \tilde{n}_l/\tilde{n}$. 
Let $\tilde{D}_n$ denote the degree of a uniformly chosen vertex from $[\tilde{n}]$, independently of the graph $\mathrm{CM}_{\tilde{n}}(\Mtilde{\boldsymbol{d}})$.
Thus, \eqref{c1:estimate:n-j-tilde} and \eqref{c1:limit-n-tilde} imply that $\tilde{D}_n\xrightarrow{\sss d} \tilde{D}$. The following lemma verifies the rest of the conditions for $\Mtilde{\boldsymbol{d}}$ in Assumption~\ref{c1:assumption1}:
 \begin{lemma} \label{c1:lem_percolation_condition} The statements below are true  $\mathbb{P}_p$ almost surely:
  \begin{enumerate}
   \item Under Assumption~\ref{c1:assumption2}~\ref{c1:assumption2-1} and for $r= 1,2,3$,
    \begin{equation}
     \frac{1}{\tilde{n}}\sum_{i \in [n]} \tilde{d}_i^r =  \frac{1}{\tilde{n}}\sum_{j \in [n]} j^r \tilde{n}_j \xrightarrow{n \to \infty}
     \mathbb{E} [ \tilde{D}^r ].
    \end{equation}
   \item Under Assumption~\ref{c1:assumption2},
    \begin{equation}
      \tilde{\nu}_n = 1+\lambda n^{-1/3}+o(n^{-1/3}).
    \end{equation}
\end{enumerate}
  \end{lemma}
\begin{proof}
We will make use of \cite[Corollary 2.27]{JLR00}. Suppose $Z_1$, $Z_2$, ..., $Z_N$ are independent random variables with $Z_i$ taking values in $\Lambda_i$ and $f:\prod_{i=1}^N \Lambda_i \to \mathbb{R}$ satisfies the following:
 If two vectors $z,z' \in \prod_{i=1}^N \Lambda_i$ differ only in the $i^{th}$ coordinate, then $| f(z)- f(z') | \leq c_i$ for some constant $c_i$.
Then, for any $t>0$, the random variable $X= f(Z_1, Z_2, \dots , Z_N)$ satisfies
\begin{equation} \label{c1:janson_lemma_bound}
 \mathbb{P} \Big( \big| X- \mathbb{E}[X] \big| > t \Big) \leq 2 \exp \Big( -\frac{t^2}{2\sum_{i =1}^{N} c_i^2} \Big).
\end{equation}
 Now let $I_{ij}$ denote the indicator of the $j^{th}$ half-edge corresponding to vertex $i$ to be kept after Algorithm~\ref{c1:algo:3}~(S1). Then $I_{ij} \sim \text{Ber} (\sqrt{p_n})$ independently for $j \in [d_i]$, $i \in [n]$. Let
 \begin{equation}
 \mathbf{I}:= (I_{ij})_{j \in [d_i], i \in [n]} \ \text{ and }\  f_1(\mathbf{I}):=\sum_{i\in [n]} \tilde{d_i}(\tilde{d}_i-1).
 \end{equation}Note that $f_1(\mathbf{I})=\sum_{i\in [\tilde{n}]}\tilde{d}_i(\tilde{d}_i-1)$ since the degree-one vertices do not contribute to the sum. One can check that, by changing the status of one half-edge corresponding to vertex $k$, we can change $f_1(\cdot)$ by at most $2(d_{k}+1)$. Therefore, \eqref{c1:janson_lemma_bound} yields
 \begin{eq}
 &\mathbb{P}_p \Big( \Big|\sum_{i\in [n]} \tilde{d_i}(\tilde{d}_i-1)- p_n \sum_{i\in [n]} d_i(d_i-1) \Big| >t \Big) \\
 &\hspace{.2cm}\leq 2 \exp \bigg( -\frac{t^2}{8\sum_{i \in [n]} d_i (d_{i}+1)^2}\bigg).
 \end{eq}
 By setting $t= n^{1/2+ \varepsilon}$ for some suitably small $\varepsilon >0$, using the finite third moment conditions and the Borel-Cantelli lemma we conclude that  $\mathbb{P}_p$ almost surely,
 \begin{equation}\sum_{i\in [n]} \tilde{d_i}(\tilde{d}_i-1)= p_n \sum_{i\in [n]} d_i(d_i-1) +O(n^{1/2+\varepsilon}), 
 \end{equation}and in particular,
  \begin{equation}\label{c1:estimate:nu-n-num}\sum_{i\in [\tilde{n}]} \tilde{d_i}(\tilde{d}_i-1)=\sum_{i\in [n]} \tilde{d_i}(\tilde{d}_i-1)= p_n \sum_{i\in [n]} d_i(d_i-1) +o(n^{2/3}).
 \end{equation}
 Similarly, take $f_2(\mathbf{I})=\sum_{i\in [n]}\tilde{d}_i(\tilde{d}_i-1)(\tilde{d}_i-2)$ and note that changing the status of one bond changes $f_2(\cdot)$ by at most $[2(d_k+1)]^2$. Thus, \eqref{c1:janson_lemma_bound} gives
 \begin{equation}
 \begin{split}  &\mathbb{P}_p \Big( \Big| f_2(\mathbf{I})- p_n^{3/2} \sum_{i\in [n]} d_i(d_i-1)(d_i-2) \Big| >t \Big) \\
 &\hspace{2cm}\leq 2 \exp \bigg( -\frac{t^2}{32\sum_{i \in [n]} d_i (d_{i}+1)^4} \bigg)\\
 &\hspace{2cm}  \leq \exp \bigg( -\frac{t^2}{32d_{\max}(d_{\max}+1)\sum_{i \in [n]}  (d_{i}+1)^3} \bigg),
  \end{split}
 \end{equation}which implies that, $\mathbb{P}_p$ almost surely,
 \begin{eq}\label{c1:estimate-third-mom-perc}
  \sum_{i\in [\tilde{n}]}\tilde{d}_i(\tilde{d}_i-1)(\tilde{d}_i-2)&=\sum_{i\in [n]}\tilde{d}_i(\tilde{d}_i-1)(\tilde{d}_i-2)\\
  &=p_n^{3/2}\sum_{i\in [n]} d_i(d_i-1)(d_i-2)+o(n),
 \end{eq}since $d_{\max}^2\sum_{i\in [n]}(d_i+1)^3=o(n^{5/3})$.
Now, to prove Lemma~\ref{c1:lem_percolation_condition}~{(1)}, note that the case $r=1$ follows by simply observing that $\sum_{i\in \tilde{n}}\tilde{d}_i=\sum_{i\in [n]}d_i$. The cases $r=2,3$ follow from \eqref{c1:estimate:nu-n-num} and \eqref{c1:estimate-third-mom-perc}.
Finally, to see Lemma~\ref{c1:lem_percolation_condition}~{(2)}, note that
\begin{equation}
\begin{split}
  \tilde{\nu}_n  & = \frac{\sum_{i \in [\tilde{n}]}\tilde{d}_i(\tilde{d}_i-1)}{\sum_{i \in [\tilde{n}]}\tilde{d}_i}
 								  =\frac{p_n \sum_{i \in [n]} d_i \big( d_i -1 \big) + o \big( n^{2/3} \big)}{\sum_{i \in [n]}d_i} \\
 								 &= \frac{p_n \sum_{i \in [n]} d_i(d_i-1)}{\sum_{i \in [n]} d_i}+o(n^{-1/3})= 1+\frac{\lambda}{n^{1/3}}+o(n^{-1/3}),
 \end{split}
\end{equation} by \eqref{c1:estimate:nu-n-num} and this completes the proof of Lemma~\ref{c1:lem_percolation_condition}.
\end{proof}
We will denote by $\tilde{\mathscr{C}}_{\sss (j)}$, the $j^{th}$ largest component of $\mathrm{CM}_{\tilde{n}}(\Mtilde{\boldsymbol{d}})$. 
To conclude Theorem~\ref{c1:thm_percolation} we also need to estimate the number of deleted vertices from each component. 
Recall from Remark~\ref{c1:rem:alt-s3} that $\mathrm{CM}_n(\boldsymbol{d},p_n(\lambda))$ can be obtained from $\mathrm{CM}_{\tilde{n}}(\Mtilde{\boldsymbol{d}})$ by deleting the relevant number of degree-one vertices \emph{uniformly} at random. Let $v^d_1(\tilde{\mathscr{C}}_{\sss (j)})$ be the number of degree-one vertices of $\tilde{\mathscr{C}}_{\sss (j)}$ that are deleted while creating $\mathrm{CM}_{n}(\boldsymbol{d},p_{n}(\lambda))$ from  $\mathrm{CM}_{\tilde{n}}(\Mtilde{\boldsymbol{d}})$. Since the vertices are to be chosen uniformly from all degree-one vertices, the number of vertices to be deleted from $\tilde{\mathscr{C}}_{\sss (j)}$ is asymptotically the total number of degree-one vertices in $\tilde{\mathscr{C}}_{\sss (j)}$ times the proportion of degree-one  vertices to be deleted. Therefore,
\begin{equation} \label{c1:degree_one_vertices}
\begin{split}
 v^d_1(\tilde{\mathscr{C}}_{\sss (j)}) &= \frac{n_+}{\tilde{n}_1}v_1(\tilde{\mathscr{C}}_{\sss (j)})+ o_{\sss\mathbb{P}}(n^{2/3})= \frac{n_+}{\tilde{n}_1} \frac{\tilde{n}_1}{\sum_{k=0}^{\infty} k\tilde{n}_k} \big| \tilde{\mathscr{C}}_{\sss (j)}\big| + o_{\sss\mathbb{P}}(n^{2/3})\\
 & = \frac{n_+}{\ell_n}\big| \tilde{\mathscr{C}}_{\sss (j)} \big|+ o_{\sss\mathbb{P}}(n^{2/3})= \frac{\mathbb{E}[D]\big(1-\sqrt{p}_n\big)}{\mathbb{E}[D]}\big| \tilde{\mathscr{C}}_{\sss (j)} \big|+ o_{\sss\mathbb{P}}(n^{2/3})\\
 & = \big(1-\sqrt{p}_n\big) \big| \tilde{\mathscr{C}}_{\sss (j)}\big| + o_{\sss\mathbb{P}}(n^{2/3}),
 \end{split}
\end{equation} where the third equality follows from \eqref{c1:eqn_vertices_of_degree_k-ord}. 
The proof of Theorem \ref{c1:thm_percolation} is now complete by using the $\ell^{2}_{\shortarrow}$ convergence in Lemma~\ref{c1:vertices_of_degree_k_ord}, \eqref{c1:degree_one_vertices} and Remark~\ref{c1:remark-perc}.

\section{{Joint convergence at multiple locations in the critical window}} \label{c1:sec_multidimensional}
We will prove Theorem~\ref{c1:thm_multiple_convergence} in this section. 
In Section~\ref{c1:sec:perc-alt-cons}, we give a construction of the joint distribution of the percolated graphs for different percolation parameters that are coupled in a way described in Theorem~\ref{c1:thm_multiple_convergence}.
In Section~\ref{c1:sec:dynamic-construction}, we compare the process of percolated graphs with a different graph process that turns out to be easier to work with.
As discussed in Remark~\ref{c1:rem:mult-coal-heuristics}, let the mass of a component be the number of open half-edges (re-scaled by $n^{2/3}$). 
The alternatively constructed graph process can be modified in such a way that the vector of masses evolves according to an \emph{exact} multiplicative coalescent as discussed in Section~\ref{c1:sec:modified-C1}. 
Thus the joint convergence result at multiple locations of the scaling window can be deduced for the modified process using the Feller property of the multiplicative coalescent.
Further, the modified process remains \emph{close} to the dynamic construction. 
In Section~\ref{c1:sec:open-he}, the vector of masses are shown to be asymptotically proportional to the component sizes and we combine all the above observations in Section~\ref{c1:sec-mul-conv-thm-proof} to complete the proof of Theorem~\ref{c1:thm_multiple_convergence}.
\subsection{Construction of the percolated graph process} 
\label{c1:sec:perc-alt-cons}
We start by explaining a way to construct the graph process $(\mathrm{CM}_n(\boldsymbol{d},p)_{p\in [0,1]}$. 
Fix any $p_1<p_2<\dots<p_m$ and consider $(\mathrm{CM}_n(\bld{d},p_i))_{i\in [m]}$. 
Recall that each edge $e$ of $\CM$ has an independent uniform $[0,1]$ random variable $U_e$ associated to it and  $\mathrm{CM}_n(\bld{d},p_i)$ is obtained from $\CM$
 by keeping only those edges $e$ with $U_e\leq p_i$. This couples the graphs $(\mathrm{CM}_n(\bld{d},p_i))_{i\in [m]}$. 
 Moreover, under this coupling, $\mathrm{CM}_n(\bld{d},p_i)$ is distributed as the graph obtained from edge percolation on $\mathrm{CM}_n(\bld{d},p_{i+1})$ with probability $p_i/p_{i+1}$ for all $i<m$. 
 The following two lemmas are modifications of \cite[Lemmas~3.1,~3.2]{F07} that lead to the construction of Algorithm~\ref{c1:algo:cons-perc} below.
  For a graph $G$, let $\rE(G)$ denote the set of edges of $G$. For a sub-graph $G$ of $\CM$, let $\mathcal{H}(G)$ denote the set of half-edges that are part of some edge in $G$ and $\mathcal{H} = \mathcal{H}(\CM)$.
\begin{lemma}\label{c1:lem:perc-cons-1} For $k_1\leq \dots\leq k_m$, conditionally on $\{|\rE(\mathrm{CM}_n(\bld{d},p_i))| = k_i:i\leq m\}$, the half-edges in $\mathrm{CM}_n(\bld{d},p_i)$ can be generated sequentially as follows: Let $k_0 = 0$, $\mathcal{H} (\mathrm{CM}_n(\bld{d},p_{0}))= \varnothing$. For each $i\leq m$, declare $\mathcal{H}(\mathrm{CM}_n(\bld{d},p_{i})) = \mathcal{H}(\mathrm{CM}_n(\bld{d},p_{i-1}))\cup \mathcal{H}_i$, where $\mathcal{H}_i$ is uniformly chosen among all the subsets of size $2k_i-2k_{i-1}$ of $\mathcal{H}\setminus \cup_{j<i}\mathcal{H}_i$.
\end{lemma} 
\begin{lemma}\label{c1:lem:perc-cons-2} Let $d_k(i,i+1) $ be the number of half-edges attached to vertex $k$ in the graph $\mathrm{CM}_n(\bld{d},p_{i+1})$ that are not in $\mathrm{CM}_n(\bld{d},p_i)$. For any $i\geq 1$, conditionally on the event $\{\bld{d}(j,j+1) = \bld{d}_0(j,j+1):j\leq m\}$ and $\mathcal{H}(\mathrm{CM}_n(\bld{d},p_{i-1}))$, the perfect matching of $\mathcal{H}(\mathrm{CM}_n(\bld{d},p_{i}))\setminus \mathcal{H}(\mathrm{CM}_n(\bld{d},p_{i-1}))$ constituting the edges $\rE(\mathrm{CM}_n(\bld{d},p_{i})\setminus\mathrm{CM}_n(\bld{d},p_{i-1}))$ is a uniform perfect matching, where we have assumed that $p_0 = 0$.
\end{lemma}
\begin{algo}\label{c1:algo:cons-perc} \normalfont Let $(U_i)_{i\geq 1}$ be a finite collection of i.i.d uniform $[0,1]$ random variables. Construct a collection of graphs $(G_n(\lambda))_{\lambda\in\R}$ using the following two steps:
\begin{itemize}
 \item[\textrm{(S0)}]  Construct the process $\bld{E}_n = (E_n(\lambda))_{\lambda\in\R}$, where $E_n(\lambda) = \#\{i:U_i\leq p_n(\lambda)\}$.
 \item[\textrm{(S1)}] Initially, $G_n(-\infty)$ is a graph only consisting of isolated vertices with no paired half-edges. At each time point $\lambda$ where $E_n(\lambda)$ has a jump, choose two unpaired half-edges uniformly at random and pair them. The graph $G_n(\lambda)$ is obtained by adding this edge to $G_n(\lambda-)$. 
\end{itemize} 
\end{algo} 
Algorithm~\ref{c1:algo:cons-perc}~(S0) can be regarded as the birth of edges and Algorithm~\ref{c1:algo:cons-perc} (S1) ensures that the edges of the graph $G_n(\lambda)$ are obtained from a uniform perfect matching of the corresponding half-edges. 
Using Lemmas~\ref{c1:lem:perc-cons-1}~and~\ref{c1:lem:perc-cons-2}, $(G_n(\lambda))_{\lambda\in\R}$ and $(\mathrm{CM}_n(\bld{d},p_n(\lambda)))_{\lambda\in \R}$ have the same finite-dimensional distributions. Therefore, $(G_n(\lambda))_{\lambda\in\R}$ and $(\mathrm{CM}_n(\bld{d},p_n(\lambda)))_{\lambda\in \R}$ have the exact same distribution. 
We complete this section by adding proofs of Lemmas~\ref{c1:lem:perc-cons-1},~and~\ref{c1:lem:perc-cons-2} which are in the same spirit as the arguments of \cite[Lemmas~3.1,~3.2]{F07}.
\begin{proof}[Proof of Lemma~\ref{c1:lem:perc-cons-1}] Assume that $k=2$ for the sake of simplicity. 
Observe that the total number of perfect matchings of $2k$ objects is given by $2k!/(k!2^k) = (2k-1)!!$.
Let $H_1$, $H_2$ be two disjoint subsets of $\mathcal{H}$ with $|H_1| = 2k_1$, $|H_2| = 2k_2-2k_1$.
Let $\mathcal{E}_1$ denote the event that a uniform perfect matching of all the half-edges contains also perfect matchings of the half-edges in $H_1$ and $H_2$. Then,
\begin{equation}\label{c1:prob-E1}
\prob{\mathcal{E}_1} = \frac{(2k_1-1)!! (2k_2-2k_1-1)!! (\ell_n-2k_2-1)!!}{(\ell_n-1)!!}.
\end{equation}
Also, for percolation on any (random) graph, conditional on the set of edges of the graph and the fact that $k$ edges have been retained by percolation, the choice of the retained edges is uniformly distributed among all subsets of size $k$ of the set of edges. Let $\mathcal{E}_2$ denote the event that $|\mathcal{H}(\mathrm{CM}_n(\bld{d},p_{1}))| = 2k_1$, and  $|\mathcal{H}(\mathrm{CM}_n(\bld{d},p_{2}))| = 2k_2$. It follows that
\begin{gather*}
\prob{\mathcal{H}(\mathrm{CM}_n(\bld{d},p_{2}))=H_1\cup H_2\mid \mathcal{E}_1,\mathcal{E}_2}= \frac{1}{\binom{\ell_n/2}{k_2}},\\
\prob{\mathcal{H}(\mathrm{CM}_n(\bld{d},p_{1}))=H_1\mid \mathcal{E}_1,\mathcal{E}_2,\mathcal{H}(\mathrm{CM}_n(\bld{d},p_{2}))=H_1\cup H_2} = \frac{1}{\binom{k_2}{k_1}}.
\end{gather*}
 Thus, conditional on the event $\mathcal{E}_2$, the probability that $\mathcal{H}(\mathrm{CM}_n(\bld{d},p_{1})) = H_1$ and $\mathcal{H}(\mathrm{CM}_n(\bld{d},p_{2}))\setminus \mathcal{H}(\mathrm{CM}_n(\bld{d},p_{1}))=H_2$
 is given by 
 \begin{equation}\label{c1:perfect-matching-1}
 \begin{split}
 \frac{(2k_1-1)!! (2k_2-2k_1-1)!! (\ell_n-2k_2-1)!!}{(\ell_n-1)!!} \frac{1}{\binom{\ell_n/2}{k_2}\binom{k_2}{k_1}} = \frac{1}{\binom{\ell_n}{2k_1}\binom{\ell_n-2k_1}{2k_2-2k_1}},
 \end{split}
 \end{equation}which does not depend on $H_1$ or $H_2$, and the proof follows.
\end{proof}
\begin{proof}[Proof of Lemma~\ref{c1:lem:perc-cons-2}]
 Fix two disjoint subsets $H_1$, $H_2$ of $\mathcal{H}$ such that $|H_1| = 2k_1$, $|H_2| = 2k_2-2k_1$.
 As in the proof of Lemma~\ref{c1:lem:perc-cons-1}, let $\mathcal{E}_2$ denote the event that $|\mathcal{H}(\mathrm{CM}_n(\bld{d},p_{1}))| = 2k_1$, and  $|\mathcal{H}(\mathrm{CM}_n(\bld{d},p_{2}))| = 2k_2$.
An identical argument as the proof of \eqref{c1:perfect-matching-1} now gives, conditionally on $\mathcal{E}_2$, the probability that $\mathcal{H}(\mathrm{CM}_n(\bld{d},p_1)) = H_1$, $\mathcal{H}(\mathrm{CM}_n(\bld{d},p_2))\setminus \mathcal{H}(\mathrm{CM}_n(\bld{d},p_1)) = H_2$, and given perfect matchings on $\mathcal{H}(\mathrm{CM}_n(\bld{d},p_{1}))$, $\mathcal{H}(\mathrm{CM}_n(\bld{d},p_2))\setminus \mathcal{H}(\mathrm{CM}_n(\bld{d},p_1))$ have been observed, is given by 
\begin{equation}\label{c1:perfect-matching-2}
\begin{split}
\frac{1}{\binom{\ell_n/2}{k_2}\binom{k_2}{k_1}}\frac{(\ell_n-2k_2-1)!!}{(\ell_n-1)!!}.
\end{split}
\end{equation} 
Let $\rD(H)$ denote the degree sequence induced by the set of half-edges $H$, and  $S$ denote the collection of \emph{disjoint} pairs $(H_1,H_2)$ such that $|H_1| = 2k_1$, $|H_2| = 2k_2-2k_1$, $\rD(H_1) = \bld{d}_0(0,1)$, and $\rD(H_2) = \bld{d}_0(1,2)$.
Then, conditionally on $\mathcal{E}_2$, the probability that $\bld{d}(0,1) = \bld{d}_0(0,1)$,  $\bld{d}(1,2) = \bld{d}_0(1,2)$, and given particular perfect matchings have been observed on $\mathcal{H}(\mathrm{CM}_n(\bld{d},p_1))$ and $\mathcal{H}(\mathrm{CM}_n(\bld{d},p_2))\setminus \mathcal{H}(\mathrm{CM}_n(\bld{d},p_1))$, is  
\begin{equation}\label{c1:perfect-matching-3}
 \sum_{(H_1,H_2)\in S} \frac{1}{\binom{\ell_n/2}{k_2}\binom{k_2}{k_1}}\frac{(\ell_n-2k_2-1)!!}{(\ell_n-1)!!}  = \frac{|S|}{\binom{\ell_n/2}{k_2}\binom{k_2}{k_1}}\frac{(\ell_n-2k_2-1)!!}{(\ell_n-1)!!}.
\end{equation}
Moreover, by Lemma~\ref{c1:lem:perc-cons-1}, the probability that $\bld{d}(0,1) = \bld{d}_0(0,1)$,  $\bld{d}(1,2) = \bld{d}_0(1,2)$, conditionally on $\mathcal{E}_2$, is given by 
  \begin{equation}\label{c1:perfect-matching-4}
   \frac{|S|}{\binom{\ell_n}{2k_1}\binom{\ell_n-2k_1}{2k_2-2k_1}}. 
  \end{equation}
Now,  \eqref{c1:perfect-matching-3} and \eqref{c1:perfect-matching-4} together yield that the probability that two particular perfect matchings are observed on $\mathcal{H}(\mathrm{CM}_n(\bld{d},p_1))$ and $\mathcal{H}(\mathrm{CM}_n(\bld{d},p_2))\setminus \mathcal{H}(\mathrm{CM}_n(\bld{d},p_1))$, conditional on $\bld{d}(0,1) = \bld{d}_0(0,1)$,  $\bld{d}(1,2) = \bld{d}_0(1,2)$ is given by 
\begin{equation}
\frac{1}{\binom{\ell_n/2}{k_2}\binom{k_2}{k_1}}\frac{(\ell_n-2k_2-1)!!}{(\ell_n-1)!!}\binom{\ell_n}{2k_1}\binom{\ell_n-2k_1}{2k_2-2k_1} = \frac{1}{(2k_1-1)!!(2k_2-2k_1-1)!!},
\end{equation}
 and the proof is complete.
\end{proof}

\subsection{The dynamic construction}\label{c1:sec:dynamic-construction}
Let us now describe a dynamic construction of $\CM$ that turns out to be easier to work with. This dynamic construction was introduced in \cite{BBSX14} to study the metric-space limits of the large components of the percolated configuration model. 
It will be shown that the graphs generated by this dynamic construction at a suitable range of time \emph{approximate} the  process $(\mathrm{CM}_n(\boldsymbol{d},p_n(\lambda)))_{\lambda\in\R}$.  
\begin{algo} \label{c1:algo-dyn-cons} \normalfont At time $t=0$, assume that there are $d_i$ \emph{open} half-edges associated with vertex $i$, for all $i\in [n]$. Associate i.i.d unit rate exponential clocks to each of the open half-edges. Each time an exponential clock rings, the corresponding half-edge selects another open half-edge uniformly at random and gets paired to it. The two paired half-edges are declared to be closed and the associated exponential clocks are removed. The process continues until the open half-edges are exhausted.
\end{algo}
Let $\mathcal{G}_n(t)$ denote the graph generated upto time $t$. Notice that $\mathcal{G}_n(\infty)$ is distributed as $\CM$ since each half-edge chooses to pair with another uniformly chosen open half-edge. 
Denote the total number of open-half-edges remaining at time $t$ while implementing Algorithm~\ref{c1:algo-dyn-cons} by $s_1(t)$. 
The graph process, given by Algorithm~\ref{c1:algo-dyn-cons}, can also be constructed as follows:
\begin{algo}\label{c1:algo:dyn-cons-alt} \normalfont Let $\Xi_n$ be an inhomogeneous Poisson process with rate $s_1(t)$ at time $t$. Let $e_1<e_2<\dots$ be the event times of $\Xi_n$.
\begin{itemize}
\item[\textrm{(S1)}] At each event time, choose two unpaired half-edges uniformly at random and pair them. The graph $\mathcal{G}_n(t)$ is obtained by adding this edge to $\mathcal{G}_n(t-)$. 
\end{itemize} 
\end{algo} 
Notice the similarity between Algorithm~\ref{c1:algo:cons-perc}~(S1) and Algorithm~\ref{c1:algo:dyn-cons-alt}~(S1). 
Now, the idea is to compare the number of half-edges that have been paired by Algorithms~\ref{c1:algo:cons-perc}~and~\ref{c1:algo:dyn-cons-alt}.
For that, we need the following lemma that describes the evolution of the count of the total number of open half-edges in Algorithm~\ref{c1:algo:dyn-cons-alt}:
\begin{lemma}[{\cite[Lemma 8.2]{BBSX14}}]\label{c1:lem:total-open-he}  Let $s_1(t)$ denote the total number of open half-edges at time $t$. Suppose that 
Assumption \ref{c1:assumption2} holds. Then, for any $T>0$ and some $1/3<\gamma<1/2$, 
\begin{equation}\label{c1:eqn:s-1-he}
 \sup_{t\leq T}\Big|\frac{1}{\ell_n}s_1(t)- \e^{-2t}\Big|= \oP(n^{-\gamma}).
\end{equation}
\end{lemma}
 Notice that the proof of \cite[Lemma 8.2]{BBSX14} is stated only under some more stringent assumptions, however the identical argument can be carried out under Assumption~\ref{c1:assumption2}.
The next proposition ensures that the graphs generated by percolation in Algorithm~\ref{c1:algo:cons-perc} and the dynamic construction in Algorithm~\ref{c1:algo-dyn-cons} are uniformly close in the critical window.
Define
\begin{equation}\label{c1:defn:t-n-lambda}
t_n(\lambda)=\frac{1}{2}\log\bigg(\frac{\nu_n}{\nu_n-1}\bigg)+\frac{1}{2(\nu_n-1)}\frac{\lambda}{n^{1/3}}.
\end{equation}
\begin{proposition}\label{c1:prop:coupling-whp} Fix $-\infty<\lambda_\star<\lambda^\star<\infty$. There exists a coupling such that with high probability
\begin{equation}
 \mathcal{G}_n(t_n(\lambda)-\varepsilon_n)\subset \mathrm{CM}_n(\bld{d},p_n(\lambda)) \subset\mathcal{G}_n(t_n(\lambda)+\varepsilon_n),\quad \forall \lambda \in [\lambda_\star,\lambda^\star]
\end{equation}where $\varepsilon_{n}=cn^{-\gamma_0}$, for some $1/3<\gamma_0<1/2$ and the constant $c$ does not depend on $\lambda$.
\end{proposition}
\begin{proof}
Notice the similarity between Algorithm~\ref{c1:algo:cons-perc}~(S1) and Algorithm~\ref{c1:algo:dyn-cons-alt}~(S1). 
Let $\#\mathrm{E}(G)$ denote the number of edges in a graph $G$. 
Suppose with high probability the following holds: $\forall \lambda \in [\lambda_\star,\lambda^\star]$
\begin{eq}\label{c1:eq:coup-reduc}
 \#\rE(\mathcal{G}_n(t_n(\lambda)-\varepsilon_n))\leq \#\rE(\mathrm{CM}_n(\bld{d},p_n(\lambda))) \leq \#\rE(\mathcal{G}_n(t_n(\lambda)+\varepsilon_n)), .
\end{eq} 
On the event $\{\#\rE(\mathrm{CM}_n(\bld{d},p_n(\lambda))) \leq \#\rE(\mathcal{G}_n(t_n(\lambda)+\varepsilon_n)), \forall \lambda \in [\lambda_\star,\lambda^\star]\}$, the choice of the uniform pair of half-edges at the $k^{th}$ pairing in Algorithm~\ref{c1:algo:cons-perc}~(S1)  can be taken to be exactly the same as the $k^{th}$ pairing in  Algorithm~\ref{c1:algo:dyn-cons-alt}~(S1).
 Under the above coupling $\mathrm{CM}_n(\bld{d},p_n(\lambda_\star)) \subset\mathcal{G}_n(t_n(\lambda_\star)+\varepsilon_n)$.
 Moreover, since $\#\rE(\mathrm{CM}_n(\bld{d},p_n(\lambda)))$ is dominated by  $\#\rE(\mathcal{G}_n(t_n(\lambda)+\varepsilon_n))$, uniformly over  $\lambda \in [\lambda_\star,\lambda^\star]$,  
the above coupling also yields $ \mathrm{CM}_n(\bld{d},p_n(\lambda)) \subset\mathcal{G}_n(t_n(\lambda)+\varepsilon_n)$ for all $\lambda\in [\lambda_\star,\lambda^\star]$. 
Further, on the event \linebreak $\{ \#\rE(\mathcal{G}_n(t_n(\lambda)-\varepsilon_n))\leq\#\rE(\mathrm{CM}_n(\bld{d},p_n(\lambda))) , \forall \lambda \in [\lambda_\star,\lambda^\star]\},$ under the same coupling, $\mathcal{G}_n(t_n(\lambda)-\varepsilon_n)\subset\mathrm{CM}_n(\bld{d},p_n(\lambda)) $ for all $\lambda\in [\lambda_\star,\lambda^\star]$. 
Thus, it remains to show \eqref{c1:eq:coup-reduc}. An application of Lemma~\ref{c1:lem:total-open-he} along with \eqref{c1:defn:t-n-lambda} yields, for some $1/3<\gamma_0<\gamma<1/2$, with high probability,
\begin{equation}\label{c1:edges-dyn-cons}
\bigg| \#\mathrm{E}(\mathcal{G}_n(t_n(\lambda))) - \bigg(\frac{\ell_n}{2\nu_n} +\frac{\lambda\ell_n}{2\nu_nn^{1/3}} +\frac{n\varepsilon_n(\nu_n-1)}{\nu_n}\bigg)\bigg| \leq n^{1-\gamma}, \quad \lambda \in [\lambda_\star,\lambda^\star].
\end{equation}
Notice that the total number of half-edges in $\mathrm{CM}_n(\bld{d},p_n(\lambda))$ follows a binomial distribution with parameters $\ell_n/2$ and $p_n(\lambda)$. Thus, with high probability,
\begin{equation}\label{c1:edges-perc}
 \bigg|\#\mathrm{E}(\mathrm{CM}_n(\bld{d},p_n(\lambda)))-\bigg(\frac{\ell_n}{2\nu_n} +\frac{\lambda\ell_n}{2\nu_nn^{1/3}}\bigg)\bigg|\leq n^{1-\gamma}, \quad \lambda \in [\lambda_\star,\lambda^\star].
\end{equation}The fact that the error can be chosen to be uniform over $\lambda\in [\lambda_\star,\lambda^\star]$ follows from the DKW inequality \cite{M90}. Thus,  \eqref{c1:edges-dyn-cons} and \eqref{c1:edges-perc} together show that, with high probability,
\begin{equation}
 \#\rE(\mathrm{CM}_n(\bld{d},p_n(\lambda))) \leq \#\rE(\mathcal{G}_n(t_n(\lambda)+\varepsilon_n)), \quad \forall \lambda \in [\lambda_\star,\lambda^\star].
\end{equation}The other part follows similarly and the proof is now complete.
\end{proof}
\begin{remark}\label{c1:rem:modified-prop-coup} \normalfont Notice that the proof of Proposition~\ref{c1:prop:coupling-whp} can be directly modified to show that there exists a coupling such that, with high probability,
\begin{equation}
 \mathrm{CM}_n(\bld{d},p_n(\lambda)-\varepsilon_n)\subset \mathcal{G}_n(t_n(\lambda))\subset \mathrm{CM}_n(\bld{d},p_n(\lambda)+\varepsilon_n),\quad \forall \lambda \in [\lambda_\star,\lambda^\star],
\end{equation}where $\varepsilon_{n}=cn^{-\gamma_0}$, for some $1/3<\gamma_0<1/2$ and the constant $c$ does not depend on~$\lambda$. Therefore, the scaling limits of different functionals like re-scaled component-sizes, surplus edges for $\mathcal{G}_n(t_n(\lambda))$ and $\mathrm{CM}_n(\bld{d},p_n(\lambda))$ are the same. 
\end{remark}
\subsection{The modified process}\label{c1:sec:modified-C1}
From here onward, we often augment $\lambda$ to a predefined notation to emphasize the dependence on~$\lambda$. 
We write $\mathscr{C}_{\sss (i)}(\lambda)$  for the $i^{th}$ largest component of $\mathcal{G}_n(t_n(\lambda))$ and define 
\begin{equation}
\mathcal{O}_i(\lambda)=\# \text{ open half-edges in }\mathscr{C}_{\sss (i)}(\lambda).
\end{equation}
Think of $\mathcal{O}_i(\lambda)$ as the \emph{mass} of the component $\mathscr{C}_{\sss (i)}(\lambda)$. 
Define the vector $\mathbf{C}_n(\lambda) = (n^{-2/3}|\mathscr{C}_{\sss (i)}(\lambda)|)_{i\geq 1}$, and $\mathbf{O}_n(\lambda) = (n^{-2/3}\mathcal{O}_i(\lambda))_{i \geq 1}$.
 Let $\ell_n^o(\lambda) = \sum_{i\geq 1}\mathcal{O}_i(\lambda)$. By Lemma~\ref{c1:lem:total-open-he} and \eqref{c1:defn:t-n-lambda}, $\ell_n^o(\lambda) \approx n\mu(\nu-1)/\nu$. Now, observe that, during the evolution of the graph process generated  by Algorithm~\ref{c1:algo-dyn-cons}, during the time interval  $[t_n(\lambda),t_n(\lambda+\dif \lambda)]$, the $i^{th}$ and $j^{th}$ ($i> j$) largest components, merge at rate 
 \begin{eq}\label{c1:rate:function}
&2\mathcal{O}_{i}(\lambda) \mathcal{O}_{j}(\lambda)\times\frac{1}{\ell_n^o(\lambda)-1}\times \frac{1}{2(\nu_n-1)n^{1/3}}\\
&\hspace{.5cm}\approx \frac{\nu}{\mu(\nu-1)^2} \big(n^{-2/3}\mathcal{O}_{i}(\lambda)\big)\big(n^{-2/3}\mathcal{O}_{j}(\lambda)\big),
\end{eq}and create a component with $\mathcal{O}_{i}(\lambda)+\mathcal{O}_{j}(\lambda)-2$ open half-edges.
Thus the open half-edges $(\mathbf{O}_n(\lambda))_{\lambda\in\R}$ does \emph{not} evolve as a multiplicative coalescent, but it is close. 
The fact that two half-edges are killed after pairing, makes the masses (the number of open half-edges) of the components  deplete. 
If there were no such depletion of mass, then the vector of open half-edges would in fact  merge as multiplicative coalescent. 
Let us formalize this  idea below:
\begin{algo}\label{c1:algo:modify-dyn-cons} \normalfont Initialize $\bar{\mathcal{G}}_n(t_n(\lambda_\star)) = \mathcal{G}_n(t_n(\lambda_\star))$.  Let $\mathscr{O}$ denote the set of open half-edges in the graph $\mathcal{G}_n(t_n(\lambda_\star))$, $\bar{s}_1 = |\mathscr{O}|$ and $\bar{\Xi}_n$ denote a Poisson process with rate $\bar{s}_1$. At each event time of the Poisson process $\bar{\Xi}_n$, select two half-edges from $\mathscr{O}$ and create an edge between the corresponding vertices. However, the selected half-edges are kept alive, so that they can be selected again.
\end{algo} 
\begin{remark}\label{c1:rem:modify-AMC}\normalfont The only difference between Algorithm~\ref{c1:algo:dyn-cons-alt} and Algorithm~\ref{c1:algo:modify-dyn-cons}, is that the \emph{paired} half-edges are not discarded and thus more edges are created by Algorithm~\ref{c1:algo:modify-dyn-cons}. Thus, there is a natural coupling between the graphs generated by Algorithms~\ref{c1:algo:dyn-cons-alt}~and~\ref{c1:algo:modify-dyn-cons} such that $\mathcal{G}_n(t_n(\lambda))\subset \bar{\mathcal{G}}_n(t_n(\lambda))$ for all $\lambda\in [\lambda_\star,\lambda^\star]$, with probability one. In the subsequent part of this section, we always work under this coupling. The extra edges that are created by Algorithm~\ref{c1:algo:modify-dyn-cons} will be called \emph{bad} edges.
\end{remark}
\begin{remark} \label{c1:rem:MC-exact-limit}\normalfont In the subsequent part of this chapter, we shall augment a predefined notation with a bar to denote the corresponding quantity for $\bar{\mathcal{G}}_n(t_n(\lambda))$. 
 Denote $\beta_n = (\bar{s}_1(\nu_n-1)n^{1/3})^{1/2}$ and $\bar{\mathbf{O}}_n'(\lambda)$ denote the vector $\ord((\beta_n^{-1}\bar{\mathcal{O}}_i(\lambda))_{i\geq 1})$.  
 By the description in Algorithm~\ref{c1:algo:modify-dyn-cons}, $(\bar{\mathbf{O}}_n'(\lambda))_{\lambda\geq \lambda_\star}$ evolves as a standard multiplicative coalescent.
Further, note that there exists a constant $c>0$ such that $\beta_n = cn^{2/3}(1+\oP(1))$ which enables us to deduce the scaling limit results for $(\bar{\mathbf{O}}_n(\lambda))_{\lambda\geq \lambda_\star}$ from $(\bar{\mathbf{O}}_n'(\lambda))_{\lambda\geq \lambda_\star}$.
\end{remark}

\subsubsection*{Multiplicative coalescent with mass and weight}
The Feller property of the multiplicative coalescent \cite[Proposition 5]{A97} ensures the joint convergence of the number of open half-edges in each component of $\bar{\mathcal{G}}_n(t_n(\lambda))$ at multiple values of $\lambda$ as we shall see below. 
To deduce the scaling limits involving the components sizes let us consider a dynamic process that is further augmented by a certain weight. 
Initially, the system consists of particles (possibly infinitely many) where particle $i$ has mass $x_i$, and weight $z_i$. 
Let $(X_i(t),Z_i(t))_{i\geq 1}$ denote the vector of masses, and weights at time $t$. 
The dynamics of the system is described as follows:
\begin{itemize}
\item[]   At time~$t$, particles $i$ and $j$ coalesce at rate $X_i(t)X_j(t)$ and create a particle with mass $X_i(t)+X_j(t)$, and weight $Z_i(t)+Z_j(t)$.
\end{itemize}
Denote by  $\mathrm{MC}_2(\mathbf{x},\mathbf{z},t)$ the vector $(X_i(t), Z_i(t))_{i\geq 1}$ with initial mass $\mathbf{x}$, and weight $\mathbf{z}$. 
We shall need the following theorem:
\begin{theorem}\label{c1:thm:AMC-2D}
Suppose that $(\mathbf{x}_n,\mathbf{z}_n) \to (\mathbf{x},\mathbf{x})$ in $(\ell^2_{\shortarrow})^2$. Then, for any $t\geq 0$
\begin{equation}
\mathrm{MC}_2(\mathbf{x}_n,\mathbf{z}_n,t)\dto \mathrm{MC}_2(\mathbf{x},\mathbf{x},t).
\end{equation}
\end{theorem}
\begin{proof}
For $\mathbf{x}_n = (x_i^n)_{i\geq 1}$ and $\mathbf{z}_n = (z_i^n)_{i\geq 1}$,  let $\mathbf{w}_n^+ = \mathrm{ord}(x_i^n\vee z_i^n)$, $\mathbf{w}_n^-=\mathrm{ord}(x_i^n\wedge z_i^n)$, where $\mathrm{ord}$ denotes the decreasing ordering of the elements. 
Notice that $\mathbf{w}_n^+ \to \mathbf{x}$, and $\mathbf{w}_n^- \to \mathbf{x}$ in $\ell^2_{\shortarrow}$.
Using the Feller property of the multiplicative coalescent \cite[Proposition 5]{A97}, it follows that
\begin{equation}\label{c1:limit-ub-lb}
 \mathrm{MC}_2(\mathbf{w}_n^+,\mathbf{w}_n^+,t)\dto  \mathrm{MC}_2(\mathbf{x},\mathbf{x},t), \quad \mathrm{MC}_2(\mathbf{w}_n^-,\mathbf{w}_n^-,t)\dto  \mathrm{MC}_2(\mathbf{x},\mathbf{x},t),
\end{equation}with respect to the $(\ell^2_{\shortarrow})^2$ topology. 
Now suppose that $\mathrm{MC}_2(\mathbf{w}_n^+,\mathbf{w}_n^+,t)$ and $\mathrm{MC}_2(\mathbf{w}_n^-,\mathbf{w}_n^-,t)$ are coupled through the subgraph coupling (see \cite[Page 838]{A97}). 
For $(\mathbf{x},\mathbf{z})\in (\ell_{\shortarrow}^2)^2$, denote $\|(\mathbf{x},\mathbf{z})\|_{\sss 22} = (\sum_{i\geq 1}x_i^2)^{1/2}+(\sum_{i\geq 1}z_i^2)^{1/2}$.
Under the subgraph coupling, \eqref{c1:limit-ub-lb} yields
\begin{equation}
\|\mathrm{MC}_2(\mathbf{w}_n^+,\mathbf{w}_n^+,t)\|_{\sss 22}^2-\|\mathrm{MC}_2(\mathbf{w}_n^-,\mathbf{w}_n^-,t)\|_{\sss 22}^2 \pto 0.
\end{equation}
Moreover,
\begin{equation}
 \|\mathrm{MC}_2(\mathbf{w}_n^-,\mathbf{w}_n^-,t)\|_{\sss 22}^2 \leq \|\mathrm{MC}_2(\mathbf{x}_n,\mathbf{z}_n,t)\|_{\sss 22}^2 \leq \|\mathrm{MC}_2(\mathbf{w}_n^+,\mathbf{w}_n^+,t)\|_{\sss 22}^2.
\end{equation}
Hence, using \cite[Corollary~18~(a)]{A97}, under the subgraph coupling,
\begin{eq}
 &\|\mathrm{MC}_2(\mathbf{w}_n^+,\mathbf{w}_n^+,t) - \mathrm{MC}_2(\mathbf{x}_n,\mathbf{z}_n,t)\|_{\sss 22}^2\\
 &\hspace{.2cm}\leq \|\mathrm{MC}_2(\mathbf{w}_n^+,\mathbf{w}_n^+,t)\|_{\sss 22}^2 - \|\mathrm{MC}_2(\mathbf{x}_n,\mathbf{z}_n,t)\|_{\sss 22}^2 \pto 0,
\end{eq}and the proof follows.
\end{proof}
\subsection{Asymptotics for the open half-edges}\label{c1:sec:open-he}
In this section, we show that the open half-edges in the components of $\mathcal{G}_n(t_n(\lambda))$ are \emph{approximately} proportional to the component sizes. 
This will enable us to apply Theorem~\ref{c1:thm:AMC-2D} for deducing the scaling limits of the required quantities for the graph $\bar{\mathcal{G}}_n(t_n(\lambda))$.
\begin{lemma}\label{c1:thm:open-comp}
 There exists a constant $\kappa > 0$ such that, for any $\lambda\in\R$ and $i\geq 1$,
 \begin{equation}\label{c1:open-he}
  \mathcal{O}_i(\lambda)= \kappa |\mathscr{C}_{\sss (i)}(\lambda)|+o_{\sss \PR}(b_n).
 \end{equation}Further, $(\mathbf{O}_n(\lambda))_{n\geq 1}$ is tight in  $\ell^2_{\shortarrow}$ and consequently $$n^{-4/3}\sum_{i\geq 1} (\mathcal{O}_i(\lambda)- \kappa |\mathscr{C}_{\sss (i)}(\lambda)|)^2 \pto 0.$$ 
\end{lemma}
\begin{proof} Let $(d_k^\lambda)_{k\in [n]}$ denote the degree sequence of $\mathrm{CM}_n(\bld{d},p_n(\lambda))$ and define
\begin{equation}
\mathcal{O}_i^p(\lambda) = \sum_{k\in \mathscr{C}_{\sss (i)}^p(\lambda)}(d_k-d_k^\lambda) = \sum_{k\in \mathscr{C}_{\sss (i)}^p(\lambda)}d_k-2(|\mathscr{C}_{\sss (i)}^p(\lambda)|-1+\mathrm{SP}(\mathscr{C}_{\sss (i)}^p(\lambda))).
\end{equation}
 Using Remark~\ref{c1:rem:modified-prop-coup} and the fact that the surplus edges in the large components is tight, it is enough to prove the lemma by replacing $\mathcal{O}_i(\lambda)$ by $\mathcal{O}_i^p(\lambda)$ and $\mathscr{C}_{\sss (i)}(\lambda)$ by $\mathscr{C}_{\sss (i)}^p(\lambda)$.
For a component $\tilde{\mathscr{C}}$ of $\mathrm{CM}_{\tilde{n}}(\Mtilde{\boldsymbol{d}})$, the corresponding component $\tilde{\mathscr{C}}^p$ in the percolated graph is obtained by cleaning up $R(\tilde{\mathscr{C}})$ red degree-one vertices, see Algorithm~\ref{c1:algo:3}. 
Thus, the number of open half-edges in $\tilde{\mathscr{C}}^p$ is given by 
\begin{equation}\label{c1:relation-deg-def}
 \sum_{k\in \tilde{\mathscr{C}}\cap [n]}d_k-\sum_{k\in \tilde{\mathscr{C}}\cap [n]}\tilde{d}_k+R(\tilde{\mathscr{C}}).
\end{equation} 
Now, all the three terms appearing in the right hand side of \eqref{c1:relation-deg-def} can be estimated using Lemma~\ref{c1:lem_general}. 
Indeed, we can consider weights $w_{i1} =d_i$, $w_{i2}=\tilde{d}_i$, and $w_{i3}=$ the number of red neighbors of vertex $i$ in $\mathrm{CM}_{\tilde{n}}(\Mtilde{\bld{d}})$. The conditions in~\eqref{c1:eq:conditions-size-biased} are satisfied by Lemma~\ref{c1:lem_percolation_condition}, and observing that 
\begin{equation}
 \max\{\max_iw_{i1}, \max_iw_{i2}, \max_{i}w_{i3}\}\leq d_{\max} = o(n^{1/3}).
\end{equation}
Note that, using an argument identical to Lemma~\ref{c1:lem_percolation_condition}, $(1/n) \sum_{i\in [\tilde{n}]}w_{ik}\tilde{d}_i$  converges $\PR_p$ almost surely, for all $k=1,2,3$. 
Now, \eqref{c1:open-he} is a consequence of Lemma~\ref{c1:lem:large-com-explored-early}.
Denote 
\begin{eq}
D_i = \sum_{k\in \tilde{\mathscr{C}}_{\sss (i)}\cap [n]}d_k&, \quad\tilde{D}_i = \sum_{k\in \tilde{\mathscr{C}}_{\sss (i)}\cap [n]}\tilde{d}_k,\\
 \mathbf{D}_n = \mathrm{ord}((D_i)_{i\geq 1})&,\quad \text{and}\quad \tilde{\mathbf{D}}_n = \mathrm{ord}((\tilde{D}_i)_{i\geq 1}).
\end{eq} 
Using \eqref{c1:eqn_tightness}, $(\tilde{\mathbf{D}}_n)_{n\geq 1}$ is tight in $\ell^2_{\shortarrow}$. Further $w_{i3}\leq d_i$ for all $i$. Thus, for the $\ell^2_{\shortarrow}$ tightness of $(\mathbf{O}_n(\lambda))_{n\geq 1}$, it is enough to show the $\ell^2_{\shortarrow}$ tightness of $(\mathbf{D}_n)_{n\geq 1}$.
Denote the conditional probability, conditioned on the uniform perfect matching in Algorithm~\ref{c1:algo:3}~(S2), by $\tilde{\PR}(\cdot)$. 
Notice that, since Algorithm~\ref{c1:algo:3}~(S1), and~(S2) are carried out independently, $\tilde{D}_i \sim \mathrm{Bin}(D_i,\sqrt{p_n})$ under $\tilde{\PR}$. 
Using standard concentration inequalities \cite[(2.9)]{JLR00}, it follows that 
\begin{equation}
 \tilde{\PR}(\tilde{D}_i< D_i\sqrt{p_n}(1-\sqrt{p_n}))\leq 2\e^{-D_ip_n^{3/2}/3},
\end{equation}and thus for $\mathcal{I} = \{k:D_k>n^{\varepsilon}\}$, the union bound yields
\begin{equation}\label{c1:domination-tilde-orginal-degree}
  \PR(\exists i\in \mathcal{I}:D_i > a\tilde{D}_i) \to 0,
\end{equation}for some constant $a>0$. Let $\mathcal{E}_n$ denote the corresponding event in \eqref{c1:domination-tilde-orginal-degree}. Thus, for any $\eta >0$, 
\begin{equation}
 \PR\bigg(n^{-4/3}\sum_{k>K,k\in \mathcal{I}}D_k^2>\eta\bigg) \leq \PR\bigg(n^{-4/3}\sum_{k>K}\tilde{D}_k^2>\frac{\eta}{a}\bigg)+ \PR(\mathcal{E}_n) \to 0,
\end{equation}if we first take the limit as $n\to\infty$, and then $K\to\infty$, and use the $\ell^2_{\shortarrow}$ tightness of $(\tilde{\mathbf{D}}_n)_{n\geq 1}$. 
Further, $\sum_{k\notin \mathcal{I}}D_k^{2}\leq n^{1+2\varepsilon}=o(n^{4/3})$, if $\varepsilon<1/6$. This completes the proof of the $\ell^2_{\shortarrow}$ tightness of $(\mathbf{D}_n)_{n\geq 1}$ and consequently that of $(\mathbf{O}_n(\lambda))_{n\geq 1}$.
\end{proof}

\subsection{{Proof of Theorem~\ref{c1:thm_multiple_convergence}}}
 \label{c1:sec-mul-conv-thm-proof}
We will consider the case $k=2$ only, since the case for general $k$ can be proved inductively. Fix $-\infty<\lambda_0<\lambda_1<\infty$.
Suppose that the modified Algorithm~\ref{c1:algo:modify-dyn-cons} starts at time $\lambda_{\star} = \lambda_0$. 
By Lemma~\ref{c1:thm:open-comp} and Theorem~\ref{c1:thm_percolation}, $(\mathbf{O}_n(\lambda_0),\kappa\mathbf{C}_n(\lambda_0))$ converges in distribution to $\kappa\sqrt{\nu}(\tilde{\bld{\gamma}}^{\lambda_0},\tilde{\bld{\gamma}}^{\lambda_0})$. Now, from Remark~\ref{c1:rem:MC-exact-limit}, an application of Theorem~\ref{c1:thm:AMC-2D} gives
\begin{equation}\label{c1:joint-conv-2}
 (\mathbf{C}_n(\lambda_0),\bar{\mathbf{C}}_n(\lambda_1)) \dto \sqrt{\nu}(\tilde{\bld{\gamma}}^{\lambda_0},\tilde{\bld{\gamma}}^{\lambda_1}).
\end{equation}
The fact that the limiting distribution corresponding to $\bar{\mathbf{C}}_n(\lambda_1)$ is equal to $\sqrt{\nu}\tilde{\bld{\gamma}}^{\lambda_1}$ follows from the Feller property of multiplicative coalescent, \cite[Theorem 2]{AL98}, and Theorem~\ref{c1:thm:AMC-2D}. For $\mathbf{x},\mathbf{y}\in \ell^2_{\shortarrow}$, denote $\mathbf{x}\preceq\mathbf{y}$ if $\mathbf{x}$ is the vector in decreasing order of elements $\{y_{ij}:i,j\geq 1\}$ such that $\sum_{j}y_{ij}\leq y_i$ for all $i\geq 1$. 
Thus if $\mathbf{y}$ is obtained by \emph{coalescing} elements of $\mathbf{x}$, then $\mathbf{x}\preceq\mathbf{y}$. 
Under the coupling in Remark~\ref{c1:rem:modify-AMC}, it follows that $\mathbf{C}_{n}(\lambda)\preceq \bar{\mathbf{C}}_{n}(\lambda)$ almost surely, for each $\lambda\geq \lambda_0$. 
Using \cite[Corollary~18~(a)]{A97}, it follows that
\begin{equation}\label{c1:bound-norm}
\|\bar{\mathbf{C}}_n(\lambda_1)-\mathbf{C}_n(\lambda_1)\|_{\sss 2}^2\leq \|\bar{\mathbf{C}}_n(\lambda_1)\|_{\sss 2}^2-\|\mathbf{C}_n(\lambda_1)\|_{\sss 2}^2,
\end{equation}where $\|\cdot\|_{\sss 2}$ denotes the $\ell^2$-norm. The final ingredient is the following straightforward lemma:

\begin{lemma}\label{c1:lem_ordering_convergence} Suppose $X_n$, $Y_n$ are non-negative random variables such that $X_n\leq Y_n$ a.s. and $X_n\xrightarrow{\mathcal{L}}X$, $Y_n\xrightarrow{\mathcal{L}}X$. Then, $$Y_n-X_n\xrightarrow{\mathbb{P}}0.$$
\end{lemma}
\begin{proof}
 Note that  $((X_n,Y_n))_{n\geq 1}$ is tight in $\mathbb{R}^2$. Thus, for any $(n'_i)_{i\geq 1}$ there exists a subsequence $(n_i)_{i\geq 1}\subset (n'_i)_{i\geq 1}$ such that $(X_{n_i},Y_{n_i})\xrightarrow{\sss \mathcal{L}}(Z_1,Z_2).$  Using the marginal distributional limits we get $Z_1\stackrel{\sss \mathcal{L}}{=} X$, $Z_2\stackrel{\sss \mathcal{L}}{=} X$.  Also the joint distribution of $(Z_1,Z_2)$ is concentrated on the line $y=x$ in the $xy$ plane.  Thus, $(X_{n_i},Y_{n_i})\xrightarrow{\sss \mathcal{L}} (X,X)$. This limiting distribution does not depend on the subsequence $(n_i)_{i\geq 1}$. Thus the tightness of $((X_n,Y_n))_{n\geq 1}$ implies $(X_{n},Y_{n})\xrightarrow{\sss \mathcal{L}} (X,X)$. The proof is now complete.
\end{proof}
\noindent Now, observe that $\|\mathbf{C}_n(\lambda_1)\|_{\sss 2}^2\leq \|\bar{\mathbf{C}}_n(\lambda_1)\|_{\sss 2}^2$ and $\|\mathbf{C}_n(\lambda_1)\|_{\sss 2}^2$, and  $\|\bar{\mathbf{C}}_n(\lambda_1)\|_{\sss 2}^2$ have the same distributional limit by Theorem~\ref{c1:thm_surplus}, and \eqref{c1:joint-conv-2}.  
Thus, using Lemma~\ref{c1:lem_ordering_convergence}, it follows that $\|\bar{\mathbf{C}}_n(\lambda_1)\|_{\sss 2}^2-\|\mathbf{C}_n(\lambda_1)\|_{\sss 2}^2\xrightarrow{\sss\mathbb{P}} 0$, and \eqref{c1:joint-conv-2}, \eqref{c1:bound-norm} yield 
\begin{equation}
 (\mathbf{C}_n(\lambda_0),\mathbf{C}_n(\lambda_1)) \dto \sqrt{\nu}(\tilde{\bld{\gamma}}^{\lambda_0},\tilde{\bld{\gamma}}^{\lambda_1}).
\end{equation}
Finally, the proof of Theorem~\ref{c1:thm_multiple_convergence} is completed by applying Proposition~\ref{c1:prop:coupling-whp}.\qed

\section{Conclusion} 
In this chapter, we have shown that whenever the third moment of the empirical degree distribution converges, the critical window for the configuration model is given by $\nu_n = 1+\lambda n^{-1/3}$, and the largest component sizes have $\Theta(n^{2/3})$ vertices and $O(1)$ surplus edges.
Theorem~\ref{c1:thm_surplus} identifies the precise limiting distribution of the rescaled component sizes and surplus edges, and the convergence results hold under $\Unot$ topology.
We apply these results to percolated $\CM$.
Analyzing the exploration process directly on percolated $\CM$ is difficult.
This is because many paired half-edges are not retained by percolation during the exploration, which changes the number of available half-edges of an unexplored vertex. 
For this reason, one has to keep updating the degree distribution of unexplored vertices, which becomes difficult to track when the maximum degree is unbounded \cite{NP10b}.
We circumvent this difficulty by using Janson's construction of the percolated configuration model.
Further, for the joint convergence of the percolated clusters over the critical window, we give a construction of the percolation process in Algorithm~\ref{c1:algo:cons-perc}, which allows us to approximate percolated graphs by a dynamically growing Markovian graph process. 
A further difficulty for proving Theorem~\ref{c1:thm_multiple_convergence} was that even the later Markovian graph process is only \emph{approximately} multiplicative coalescent.
The ideas for dealing with this approximate multiplicative coalescent are general, and we believe that these techniques are applicable to many other dynamic graph models. 

%


%
%
 
\cleardoublepage

\chapter[Critical window: Infinite third moment]{Heavy-tailed configuration models at criticality}
\label{chap:secondmoment}
{\small \paragraph*{Abstract.}
 We study the critical behavior of the component sizes for the configuration model when the tail of the degree distribution of a randomly chosen vertex is a regularly-varying function with exponent $\tau-1$ with $\tau\in (3,4)$.
 The component sizes are shown to be of the order $n^{(\tau-2)/(\tau-1)}L(n)^{-1}$ for some slowly-varying function $L(\cdot)$.
 We show that the re-scaled  ordered component sizes converge in distribution to the ordered excursions of a thinned L\'evy process.
 This proves that the scaling limits for the component sizes for these heavy-tailed configuration models are in a different universality class compared to those for the Erd\H{o}s-R\'enyi random graphs. 
 Also the joint  re-scaled vector of ordered component sizes and their surplus edges is shown to have a distributional limit under a strong topology. 
 Our proof resolves a conjecture by Joseph, \emph{Ann. Appl. Probab.} (2014) about the scaling limits of uniform simple graphs with i.i.d.~degrees in the critical window, and sheds light on the relation between the scaling limits obtained by Joseph and in this chapter, which appear to be quite different.
  Further, we use percolation to study the evolution of the component sizes and the  surplus edges within the critical scaling window, whose finite-dimensional distributions are shown to converge to the augmented multiplicative coalescent process introduced by Bhamidi et al., \emph{Probab. Theory Related Fields} (2014).
  The main results of this chapter  are proved under rather general assumptions on the vertex degrees. 
  We also discuss how these assumptions are satisfied by some of the frameworks that have been studied previously.
}

\vspace{.3cm} 
\noindent {\footnotesize Based on the manuscript: Souvik Dhara, Remco van der Hofstad, Johan S.H. van Leeuwaarden, and Sanchayan Sen, \emph{Heavy-tailed configuration models at criticality} (2016), arXiv:1612.00650
}

%

In this chapter, we focus on the critical behavior of the configuration model, and critical percolation when the third moment of the empirical degree distribution tends to infinity. 
We include detailed proofs of the scaling limit results related to component sizes and surplus edges described in Chapter~\ref{chap:introduction}. 
As in Chapter~\ref{chap:thirdmoment}, the scaling limit result are shown to hold for critical percolation on the configuration model. 
We also study the evolution of both component sizes and surplus edges over the critical window, and describe its asymptotic distribution by a version of the augmented multiplicative coalescent process. 
The scaling limits lie in the universality class identified in \cite{BHL12}, and are fundamentally different than in Chapter~\ref{chap:thirdmoment}.
The results in this chapter provide a detailed understanding of the component sizes and surplus edges for  heavy-tailed graphs in the critical window. 
 Before stating the main results, we need to introduce some
notation and concepts.

\section{Definitions and notation}\label{c2sec:notation}
Recall the definitions from Chapter~\ref{sec:notation-intro}.
Let $(\mathbb{U}^0_{\shortarrow})^k$ denote the $k$-fold product space of $\mathbb{U}^0_{\shortarrow}$. For any $\mathbf{z}\in \mathbb{U}_{\shortarrow}$, $\ord(\mathbf{z})$ will denote the element of $\mathbb{U}^0_{\shortarrow}$ obtained by suitably ordering the coordinates of $\mathbf{z}$. 
 We often use the boldface notation $\mathbf{X}$ for the process $( X(s) )_{s \geq 0}$, unless stated otherwise. $\mathbb{D}[I,E]$ will denote the space of c\`adl\`ag functions from an interval $I$ to the metric space $E=(E,\mathrm{d})$ equipped with Skorohod $J_1$-topology. 
Consider a decreasing sequence $ \boldsymbol{\theta}=(\theta_1,\theta_2,\dots)\in \ell^3_{\shortarrow}\setminus \ell^2_{\shortarrow}$. Denote by  $\mathcal{I}_i(s):=\ind{\xi_i\leq s }$ where $\xi_i\sim \mathrm{Exp}(\theta_i/\mu)$ independently, and $\mathrm{Exp}(r)$ denotes the exponential distribution with rate $r$.  Consider the process \begin{equation}\label{c2defn::limiting::process}
\bar{S}^\lambda_\infty(t) =  \sum_{i=1}^{\infty} \theta_i\left(\mathcal{I}_i(t)- (\theta_i/\mu)t\right)+\lambda t,
\end{equation}for some $\lambda\in\mathbb{R}, \mu >0$ and define the reflected version of $\bar{S}_\infty^{\lambda}(t)$ by
\begin{equation} \label{c2defn::reflected-Levy}
 \refl{ \bar{S}_\infty^{\lambda}(t)}= \bar{S}_\infty^{\lambda}(t) - \min_{0 \leq u \leq t} \bar{S}_\infty^{\lambda}(u).
\end{equation}The process of the form \eqref{c2defn::limiting::process} was termed \emph{thinned} L\'evy processes in \cite{BHL12} (see also \cite{AHKL16,HKL14}), since the summands are thinned versions of  Poisson processes.
For any function $f\in \mathbb{D}[0,\infty)$,  define $\ubar{f}(x)=\inf_{y\leq x}f(y)$.  $\mathbb{D}_+[0,\infty)$ is the subset of $\mathbb{D}[0,\infty)$ consisting of functions with positive jumps only. Note that $\ubar{f}$ is continuous when $f\in \mathbb{D}_+[0,\infty)$. An \emph{excursion} of a function $f\in \mathbb{D}_+[0,T]$ is an interval $(l,r)$ such that 
\begin{eq}\label{c2def:excursion}\min\{f(l-),f(l)&\}=\ubar{f}(l)=\ubar{f}(r)=\min\{f(r-),f(r)\} \\
\quad \text{and}\quad &f(x)>\ubar{f}(r),\ \forall x\in (l,r)\subset [0,T].
\end{eq}Excursions of a function $f\in \mathbb{D}_+[0,\infty)$ are defined similarly.
 We will use $\gamma$ to denote an excursion, as well as the length of the excursion $\gamma$ to simplify notation. 

 Also, define the counting process $\mathbf{N}$ to be the Poisson process that has intensity $\refl{ \bar{S}_\infty^{\lambda}(t)}$ at time $t$ conditional on $( \bar{S}_\infty^{\lambda}(u) )_{u \leq t}$. Formally, $\mathbf{N}$ is characterized as the counting process for which 
\begin{equation} \label{c2defn::counting-process}
N(t) - \int\limits_{0}^{t}\refl{ \bar{S}_\infty^{\lambda}(u)}\dif u
\end{equation} is a martingale.  We use  the notation $N(\gamma)$ to denote the number of marks in the interval $\gamma$.

Finally, we define a Markov process $(\mathbf{Z}(s))_{s\in\R}$ on $\mathbb{D}(\R,\mathbb{U}^0_{\shortarrow})$, called the augmented multiplicative coalescent (AMC) process. Think of  a collection of particles in a system with $\mathbf{X}(s)$ describing their masses and $\mathbf{Y}(s)$ describing an additional attribute at time $s$. Let $K_1,K_2>0$ be constants. The evolution of the system takes place according to the following rule at time $s$:
\begin{itemize}
\item[$\rhd$] For $i\neq j$, at rate $K_1X_i(s)X_j(s)$,  the $i^{th}$ and $j^{th}$ component merge and create a new component of mass $X_i(s)+X_j(s)$ and attribute $Y_i(s)+Y_j(s)$.  
\item[$\rhd$] For any $i\geq 1$, at rate $K_2X_i^2(s)$, $Y_i(s)$ increases to $Y_i(s)+1$. 
\end{itemize}Of course, at each event time, the indices are re-organized to give a proper element of $\mathbb{U}^0_{\shortarrow}$.
This process was first introduced in \cite{BBW12} to study the joint behavior of the component sizes and the surplus edges over the critical window.
 In \cite{BBW12}, the authors extensively study the properties of the  standard version of AMC, i.e., the case  $K_1=1,K_2=1/2$ and showed in \citep[Theorem 3.1]{BBW12} that this is a (nearly) Feller process, a property that will play a crucial rule in the final part of this chapter.

\begin{remark}\normalfont Notice that the summation term in \eqref{c2defn::limiting::process}, after replacing $\theta_i$ by $\mu\theta_i$, is of the form 
$V^{\boldsymbol{\theta}}(s)= \mu^{\alpha}\sum_{i=1}^{\infty} \big(\theta_i\ind{\xi_i\leq s}-\theta_i^2s \big),$ where $\xi_i\sim \mathrm{Exp}(\theta_i)$ independently over $i$ and  $\boldsymbol{\theta}\in \ell^3_{\shortarrow}\setminus\ell^2_{\shortarrow}$. Therefore, by \cite[ Lemma 1]{AL98}, the process $\refl{\bar{\mathbf{S}}_\infty^{\lambda}}$ has no infinite excursions almost surely and only finitely many excursions with length at least $ \delta$, for any $\delta >0$.
\end{remark}

\section{Main results}
\subsection{Main results for critical configuration models}\label{c2sec:results}
 Throughout this chapter we will use the shorthand notation
\begin{eq}\label{c2eqn:notation-power*}
 \alpha= 1/(\tau-1),\qquad \rho=(\tau-2)/(\tau-1),\qquad \eta=(\tau-3)/(\tau-1),\\
a_n= n^{\alpha}L(n),\qquad b_n=n^{\rho}(L(n))^{-1},\qquad c_n=n^{\eta} (L(n))^{-2},
\end{eq}where $\tau\in (3,4)$ and $L(\cdot)$ is a slowly-varying function.
We state our results under the following assumptions:

\begin{assumption}\label{c2assumption1}
\normalfont Fix $\tau \in (3,4)$. Let $\boldsymbol{d}=(d_1,\dots,d_n)$ be a degree sequence such that the following conditions hold:
\begin{enumerate}[(i)] 
\item \label{c2assumption1-1} (\emph{High-degree vertices}) For any fixed $i\geq 1$, 
\begin{equation}\label{c2defn::degree}
 \frac{d_i}{a_n}\to \theta_i,
\end{equation}where $\boldsymbol{\theta}=(\theta_1,\theta_2,\dots)\in \ell^3_{\shortarrow}\setminus \ell^2_{\shortarrow}$. 
\item \label{c2assumption1-2} (\emph{Moment assumptions}) Let $D_n$ denote the degree of a vertex chosen uniformly at random from $[n]$, independently of $\mathrm{CM}_n(\boldsymbol{d})$. Then, $D_n\dto D$, for some integer-valued random variable $D$ and 
\begin{eq}
 \frac{1}{n}\sum_{i\in [n]}d_i\to \mu:=\expt{D},& \quad \frac{1}{n}\sum_{i\in [n]}d_i^2\to \E[D^2],\\
 \quad  \lim_{K\to\infty}\limsup_{n\to\infty}\bigg(a_n^{-3} &\sum_{i=K+1}^{n} d_i^3\bigg)=0.
\end{eq}
\item \label{c2assumption1-3} (\emph{Critical window}) For some $\lambda \in \mathbb{R}$,
\begin{equation}\label{c2critical-window}
 \nu_n(\lambda):=\frac{\sum_{i\in [n]}d_i(d_i-1)}{\sum_{i\in [n]}d_i}=1+\lambda c_n^{-1}+o(c_n^{-1}).
\end{equation}
\item \label{c2assumption1-4} Let $n_1$ be the number of vertices of degree-one. Then $n_1=\Theta(n)$, which is equivalent to assuming that $\prob{D=1}>0$.
\end{enumerate}
\end{assumption} 
Note that Assumption~\ref{c2assumption1}~\eqref{c2assumption1-1}-\eqref{c2assumption1-2} implies $\liminf_{n\to\infty}\E[D_n^3]=\infty$. 
The following three results hold for any $\mathrm{CM}_n(\boldsymbol{d})$ satisfying Assumption~\ref{c2assumption1}:

\begin{theorem}\label{c2thm::conv:component:size} Consider $\mathrm{CM}_n(\boldsymbol{d})$ with the degrees satisfying \textrm{Assumption~\ref{c2assumption1}}. Denote the $i^{th}$-largest cluster of $\mathrm{CM}_n(\boldsymbol{d})$ by $\mathscr{C}_{\sss (i)}$. Then, 
 \begin{equation}
  \left( b_n^{-1}|\mathscr{C}_{\sss (i)}|\right)_{i\geq 1}\dto\left(\gamma_i(\lambda)\right)_{i\geq 1},
 \end{equation}
 with respect to the $\ell^2_{\shortarrow}$-topology where $\gamma_i(\lambda)$ is the length of the $i^{th}$ largest excursion of the process $\bar{\mathbf{S}}_\infty^{\lambda} $, while $b_n$ and the constants $\lambda, \mu$ are defined in \eqref{c2eqn:notation-power*} and \textrm{Assumption~\ref{c2assumption1}}.
\end{theorem}  

\begin{theorem}\label{c2thm:spls}
 Consider $\mathrm{CM}_n(\boldsymbol{d})$ with the degrees satisfying \textrm{Assumption~\ref{c2assumption1}}. Let $\mathrm{SP}(\mathscr{C}_{\sss (i)})$ denote the number of surplus edges in  $\mathscr{C}_{\sss (i)}$  and define the vectors $\mathbf{Z}_n:= \ord( b_n^{-1}|\mathscr{C}_{\sss (i)}|,\mathrm{SP}(\mathscr{C}_{\sss (i)}))_{i\geq 1}$, $\mathbf{Z}:=\ord(\gamma_i(\lambda), N(\gamma_i))_{i\geq 1}$. Then, as $n\to\infty$,
 \begin{equation}\label{c2thm:eqn:spls}
  \mathbf{Z}_n\dto\mathbf{Z}
 \end{equation}with respect to the $\mathbb{U}^0_{\shortarrow}$ topology, where $\mathbf{N}$ is  defined in \eqref{c2defn::counting-process}.
\end{theorem}
\begin{theorem}\label{c2thm:simple-graph}
 The results in \textrm{Theorems~\ref{c2thm::conv:component:size}} and\textrm{~\ref{c2thm:spls}} also hold for uniform random graphs with degree $\bld{d}$.
\end{theorem}
\begin{remark}\normalfont
The only previous work to understand  the critical behavior of the configuration model with heavy-tailed degrees was by \citet{Jo10} where the degrees were assumed to be i.i.d an sample from an exact power-law distribution and the results were obtained for the component sizes of $\mathrm{CM}_n(\boldsymbol{d})$ (Theorem~\ref{c2thm::conv:component:size}). 
We will see that Assumption~\ref{c2assumption1} is satisfied for i.i.d degrees in Section~\ref{c2sec:iid-degrees}. Thus, a quenched version of \cite[Theorem~8.3]{Jo10} follows from our results. Further, if the degrees are chosen approximately as the weights chosen in \cite{BHL12}, then our results continue to hold. This sheds light on the relation between the scaling limits in \cite{BHL12}~and~\cite{Jo10} (see Remark~\ref{c2rem:jos-vs-this}).
Moreover, Theorem~\ref{c2thm:simple-graph} resolves \cite[Conjecture 8.5]{Jo10}.
\end{remark}

\begin{remark}\label{c2rem:gen-functional1}\normalfont The conclusions of Theorems~\ref{c2thm::conv:component:size},~\ref{c2thm:spls},~and~\ref{c2thm:simple-graph} hold for more general functionals of the components. Suppose that each vertex $i$ has a weight $w_i$ associated to it and let $\mathscr{W}_i$ denote the total weight of the component $\mathscr{C}_{\sss (i)}$, i.e., $\mathscr{W}_i = \sum_{k\in \mathscr{C}_{\sss (i)}} w_k$. Then, under some regularity conditions on the weight sequence $\bld{w}=(w_i)_{i\in [n]}$, in Section~\ref{c2sec:comp-functional} we will  show  that the scaling limit for $\mathbf{Z}^w_n:= \ord( b_n^{-1}\mathscr{W}_i,\mathrm{SP}(\mathscr{C}_{\sss (i)}))_{i\geq 1}$ is given by  $\mathbf{Z} = \ord(\kappa\gamma_i(\lambda), N(\gamma_i))_{i\geq 1}$, where the constant $\kappa$ is given by  
$$\kappa = \lim_{n\to\infty}\frac{\sum_{i\in [n]}d_iw_i}{\sum_{i\in[n]}d_i}.$$ Observe that, for $w_i= \ind{d_i=k}$, $\mathscr{W}_i$ gives the asymptotic number of vertices of degree $k$ in the $i^{th}$ largest component.
\end{remark}

\begin{remark}\normalfont It might not be immediate why we should work with Assumption~\ref{c2assumption1}. We will see in Section~\ref{c2sec:example} that Assumption~\ref{c2assumption1} is satisfied by the degree sequences in some important and natural cases. The reason to write the assumptions in this form is to make the properties of the degree distribution explicit (e.g. in terms of moment conditions and the asymptotics of the highest degrees) that jointly lead to this universal critical limiting behavior. We explain the significance of Assumption~\ref{c2assumption1} in more detail in Section~\ref{c2sec:discussion}.
\end{remark}

\subsection{Percolation on heavy-tailed configuration models} 
Percolation refers to deleting each edge of a graph independently with probability $1-p$.  Consider  percolation on a configuration model $\mathrm{CM}_n(\boldsymbol{d})$ under the following assumptions:
\begin{assumption} \label{c2assumption2} \normalfont 
 \begin{enumerate}[(i)]
  \item \label{c2assumption2-1} Assumption~\ref{c2assumption1}~\eqref{c2assumption1-1},~and~\eqref{c2assumption1-2} hold for the degree sequence and $\mathrm{CM}_n(\boldsymbol{d})$ is super-critical, i.e.,
  \begin{equation}
   \nu_n=\frac{\sum_{i\in [n]}d_i(d_i-1)}{\sum_{i\in [n]}d_i}\to \nu =\frac{\expt{D(D-1)}}{\expt{D}}>1.
  \end{equation}
  \item \label{c2assumption2-2} (\emph{Critical window for percolation}) The percolation parameter $p_n$ satisfies
  \begin{equation}
  p_{n}=p_n(\lambda):=\frac{1}{\nu_n} \big( 1+ \lambda c_n^{-1}+o(c_n^{-1}) \big)
  \end{equation}for some $\lambda\in \mathbb{R}$.
  \end{enumerate}
 \end{assumption}
\noindent Let $\mathrm{CM}_n(\boldsymbol{d},p_n(\lambda))$ denote the graph obtained through percolation on $\mathrm{CM}_n(\boldsymbol{d})$ with bond retention probability $p_n(\lambda)$. The following result gives the asymptotics for the ordered component sizes and the surplus edges for $\mathrm{CM}_n(\boldsymbol{d},p_n(\lambda))$:
\begin{theorem}\label{c2thm:percolation}
Consider $\mathrm{CM}_n(\boldsymbol{d},p_n(\lambda))$ satisfying \textrm{Assumption~\ref{c2assumption2}}.
Let $\tilde{\mathbf{S}}_{\infty}^{\lambda}$ denote the process in \eqref{c2defn::limiting::process} with $\theta_i$ replaced by $\theta_i/\sqrt{\nu}$, and
 $\mathscr{C}_{\sss(i)}^p$ denote the $i^{th}$ largest component of $\mathrm{CM}_n(\boldsymbol{d},p_n)$ and let $\mathbf{Z}_n^p(\lambda):=\ord( b_n^{-1}|\mathscr{C}_{\sss(i)}^p|,\mathrm{SP}(\mathscr{C}_{\sss(i)}^p))_{i\geq 1}$,  $\mathbf{Z}^p(\lambda):=\ord((\nu^{1/2}\tilde{\gamma}_i(\lambda),N(\tilde{\gamma}_i(\lambda)))_{i\geq 1}$, where $\tilde{\gamma}_i(\lambda)$ is the largest excursion of $\tilde{\mathbf{S}}_{\infty}^{\lambda}$. Then, for any $\lambda\in \mathbb{R}$, as $n\to \infty$,
\begin{equation}\label{c2eqn:perc:limit}
 \mathbf{Z}_n^p(\lambda)\dto \mathbf{Z}^p(\lambda)
\end{equation}with respect to the $\mathbb{U}^0_{\shortarrow}$ topology.
\end{theorem}
 Now, consider a graph $\mathrm{CM}_n(\boldsymbol{d})$ satisfying Assumption~\ref{c2assumption2}~\eqref{c2assumption2-1}. To any edge $(ij)$ between vertices $i$ and $j$ (if any), associate an independent uniform random variable $U_{(ij)}$. Note that the graph obtained by keeping only those edges satisfying $U_{(ij)}\leq p_n(\lambda)$ is distributed as $\mathrm{CM}_n(\boldsymbol{d},p_n(\lambda))$. This construction naturally couples the graphs $(\mathrm{CM}_n(\boldsymbol{d},p_n(\lambda)))_{\lambda\in\R}$ using the same set of uniform random variables. Our next result shows that the evolution of the component sizes and the surplus edges of  $\mathrm{CM}_n(\boldsymbol{d},p_n(\lambda))$, as $\lambda$ varies, can be described by a version of the augmented multiplicative coalescent process described in Section~\ref{c2sec:notation}:
\begin{theorem}\label{c2thm:mul:conv} Suppose that \textrm{Assumption~\ref{c2assumption2}} holds, and $\ell_n/n = \mu +o(n^{-\zeta})$ for some $\eta<\zeta<1/2$.
Fix any $k\geq 1$, $-\infty<\lambda_1<\dots<\lambda_k<\infty$. Then, there exists a version $\mathbf{AMC}=(\mathrm{AMC}(\lambda))_{\lambda\in\R}$ of the augmented multiplicative coalescent such that, as $n\to\infty$,
\begin{equation}
 \left(\mathbf{Z}_n^p(\lambda_1), \dots\mathbf{Z}_n^p(\lambda_k)\right) \dto \left(\mathbf{AMC}(\lambda_1),\dots,\mathbf{AMC}(\lambda_k)\right)
\end{equation}with respect to the $(\mathbb{U}^0_{\shortarrow})^k$ topology, where at each $\lambda$, $\mathrm{AMC}(\lambda)$ is distributed as the limiting object in~\eqref{c2eqn:perc:limit}.
\end{theorem}  
\begin{remark}\label{c2:rem-AMC-finite-third}
\normalfont
 Theorem~\ref{c2thm:mul:conv} also holds when $\E[D_n^3]\to \E[D^3]<\infty$ with $\alpha=\eta=1/3$, $\rho =2/3$ and $L(n)=1$. This improves \cite[Theorem 4]{DHLS15}, which was proved only for the cluster sizes.
\end{remark}
\begin{remark}\normalfont Theorem~\ref{c2thm:mul:conv}, in fact, shows that there exists a version of the AMC process whose distribution at each fixed $\lambda$ can be described by the excursions of a thinned L\'evy process and an associated Poisson process. 
This did not appear in \cite{BBW12,BM14}, since the scaling limits in their settings were described in terms of the excursions of a Brownian motion with parabolic drift.   
\end{remark}
\begin{remark}\label{c2rem:additional-assumption}\normalfont The additional assumption in Theorem~\ref{c2thm:mul:conv} about the asymtotics  $\ell_n/n$ is required only in one place for Proposition~\ref{c2prop:coupling-whp} and the rest of the proof works under Assumption~\ref{c2assumption2} only. That is why we have separated this assumption from the set of conditions in Assumption~\ref{c2assumption2}. 
It is worthwhile mentioning that the condition is not stringent at all, e.g., we will see that this condition is satisfied under the two widely studied set-ups in Section~\ref{c2sec:example}. 
\end{remark}

\begin{remark}\normalfont 
As we will see in Section~\ref{c2sec:conv-amc}, the proof of Theorem~\ref{c2thm:mul:conv} can be extended to more general functionals of the components. For example, the evolution of the number of degree $k$ vertices along with the surplus edges can also be described by an AMC process. The key idea here is that these component functionals become approximately proportional to the component sizes in the critical window  and thus the scaling limit for the component functionals becomes a constant multiple of the scaling limit for the component sizes.
\end{remark}

\section{Important examples}\label{c2:important-examples}
\subsection{Power-law degrees with small perturbation}\label{c2sec:example}
As discussed in the introduction, our main goal is to obtain results for the critical configuration model with $\prob{D_n\geq k}\sim L_0(k)k^{-(\tau-1)}$ for some $\tau\in (3,4)$. 
In this section, we consider such an example and show that the conditions of Assumption~\ref{c2assumption1} are satisfied. Thus, the results in Section~\ref{c2sec:results} hold for \CM\   in the following set-up that is closely related to the model studied in~\cite{BHL12} for rank-1 inhomogeneous random graphs.
\par Fix $\tau\in (3,4)$. Suppose that $F$ is the distribution function of a discrete non-negative random variable $D$ such that
 \begin{equation} \label{c2defn::F}
 G(x)=1-F(x)= \frac{C_FL_0(x)}{x^{\tau-1}}(1+o(1))\quad \text{as } x\to\infty,
 \end{equation}
where $L_0(\cdot)$ is a slowly-varying function so that the tail of the distribution is decaying  like a regularly-varying function. Recall that the inverse of a locally bounded non-increasing function $f:\R\to\R$ is defined as $f^{-1}(x):=\inf\{y:f(y)\leq x\}$. Therefore, using \cite[Theorem 1.5.12]{BGT89}, \begin{equation}\label{c2eqn:inverse}
 G^{-1}(x)=\frac{C_F^{1/(\tau-1)}L(1/x)}{x^{1/(\tau-1)}}(1+o(1))\quad \text{as } x\to 0,
\end{equation}where $L(\cdot)$ is another slowly-varying function. Note that \cite[Theorem 1.5.12]{BGT89} is stated for positive exponents only. Since our exponent is negative, the asymptotics in \eqref{c2eqn:inverse} holds for $x\to 0$. 
Suppose that the random variable $D$ is such that
\begin{equation}\label{c2defn:critical}
 \nu=\frac{\expt{D(D-1)}}{\expt{D}}=1.
\end{equation}Define the degree sequence $\boldsymbol{d}_{\lambda}$ by taking the degree of the $i^{th}$ vertex to be
 \begin{equation}\label{c2defn::degree:powerlaw}
  d_i=d_i(\lambda):=G^{-1}(i/n)+\delta_{i,n}(\lambda),
  \end{equation}where the $\delta_{i,n}(\lambda)$'s are non-negative integers satisfying the asymptotic equivalence 
  \begin{equation}\label{c2defn:delta:i} \delta_{i,n}(\lambda) \sim \lambda G^{-1}(i/n) c_n^{-1},\quad \text{as } n\to\infty.
  \end{equation}The $\delta_{i,n}(\lambda)$'s are chosen in such a way that Assumption~\ref{c2assumption1}~\eqref{c2assumption1-4} is satisfied.
Fix any $K\geq 1$. 
Notice that  \eqref{c2eqn:inverse} and \eqref{c2defn:delta:i} imply that, for all large enough $n$ (independently of $K$), the first $K$ largest degrees $(d_i)_{i\in [K]}$ satisfy 
\begin{equation}\label{c2defn:high-degree}
 d_i= \left(\frac{n^{\alpha}C_F^{\alpha}L(n/i)}{i^{\alpha}}\right) \left( 1+\lambda c_n^{-1}+o(c_n^{-1})\right).
\end{equation}
Therefore, $\boldsymbol{d}_{\lambda}$ satisfies Assumption~\ref{c2assumption1}~\eqref{c2assumption1-1} with $\theta_i=(C_F/i)^{\alpha}$. 
The next two lemmas verify Assumption~\ref{c2assumption1}~\eqref{c2assumption1-2},~\eqref{c2assumption1-3}:

\begin{lemma} \label{c2lem::order_moments}The degree sequence $\boldsymbol{d}_{\lambda}$ defined in \eqref{c2defn::degree:powerlaw} satisfies \textrm{Assumption~\ref{c2assumption1}~\eqref{c2assumption1-2}}.
\end{lemma}
\begin{proof}
 Note that, by \eqref{c2defn:high-degree}, $d_1^2=o(n)$.
 Also, since $G^{-1}$ is non-increasing 
 \begin{equation}
  \int_{0}^1G^{-1}(x)\dif x-\frac{d_1}{n} \leq \frac{1}{n}\sum_{i\in [n]} G^{-1}(i/n)\leq \int_{0}^1G^{-1}(x)\dif x.
 \end{equation}
  Therefore,
 \begin{eq}\label{c2tot-deg-error-1}
  \frac{1}{n}\sum_{i\in [n]} d_i&=\frac{1}{n}\sum_{i\in [n]}G^{-1}(i/n)(1+O(c_n^{-1}))\\
  &= \int_{0}^1G^{-1}(x)\dif x +O(d_1/n)+O(c_n^{-1})= \expt{D}+O(b_n^{-1}).
 \end{eq} 
 Similarly, $\sum_{i\in [n]}d_i^2=n\E[D^2]+O(d_1^2)=n\E[D^2]+o(n)$. 
 To prove the condition involving the third-moment, we use Potter's theorem \cite[Theorem 1.5.6]{BGT89}. 
 First note that $3\alpha-1 = (4-\tau)/(\tau-1)>0$ since $\tau\in (3,4)$. 
 Fix $0<\delta<\alpha -1/3$ and $A>1$ and choose $C=C(\delta,A)$ such that for all $i\leq nC^{-1}$, $L(n/i)/L(n)< Ai^{\delta}$. Therefore, \eqref{c2eqn:inverse} implies 
 \begin{equation}\label{c2eq:3rd-moment-example}
 a_n^{-3} \sum_{i>K}d_i^3\leq A \sum_{i>K}i^{-3\alpha+3\delta}+ \frac{\sup_{1\leq x\leq C}L(x)^3}{L(n)^3} \sum_{i>nC^{-1}}i^{-3\alpha}.
 \end{equation} 
 From our choice of $\delta$, $-3\alpha+3\delta <-1$ and therefore $\sum_{i\geq 1}i^{-3\alpha+3\delta}<\infty$. By \cite[Lemma 1.3.2]{BGT89},  $\sup_{1\leq x\leq C}L(x)^3<\infty$. Moreover, $\sum_{i>nC^{-1}}i^{-3\alpha}=O(n^{1-3\alpha})$ and $1-3\alpha<0$. Thus,  the proof follows by first taking $n\to \infty$ and then $K\to\infty$. 
\end{proof}
\begin{lemma}\label{c2lem::nu-n} The degree sequence $\boldsymbol{d}_{\lambda}$ defined in \eqref{c2defn::degree:powerlaw} satisfies  \textrm{Assumption~\ref{c2assumption1}~\eqref{c2assumption1-3}}, i.e., there exists $\lambda_0\in \R$ such that
 \begin{equation} \label{c2eqn::lem::nu_n}
  \nu_n(\lambda)= 1+(\lambda+\lambda_0) c_n^{-1}+o(c_n^{-1}).
 \end{equation}
\end{lemma} 
\begin{proof}
 Firstly, Lemma \ref{c2lem::order_moments} guarantees  the convergence of the second moment of the degree sequence. However, \eqref{c2eqn::lem::nu_n} is more about obtaining sharper asymptotics for $\nu_n(\lambda)$. We use similar arguments as in  \cite[Lemma 2.2]{BHL12}. Denote  $\nu_n:=\nu_n(0)$. Note that $\nu_n(\lambda)=\nu_n(1+\lambda c_n^{-1})+o(c_n^{-1})$, so it is enough to verify that \begin{equation}
\nu_n=1+\lambda_0c_n^{-1}+o(c_n^{-1}).
\end{equation}
Consider $d_i(0)$ as given in \eqref{c2defn::degree:powerlaw} with $\lambda=0$. Lemma~\ref{c2lem::order_moments} implies
\begin{equation}
 \nu_n = \frac{\sum_{i\in [n]} d_i(0)^2}{n\expt{D}}-1+o(c_n^{-1}).
\end{equation}
 Fix any $K\geq 1$.  We have
 \begin{equation}
  \int_{K/n}^{1} G^{-1}(u)^2 \dif u - \frac{d_K^2}{n}\leq \frac{1}{n}\sum_{i=K+1}^n d_i^2 \leq \int_{K/n}^{1} G^{-1}(u)^2 \dif u.
 \end{equation}Now by \eqref{c2defn::degree:powerlaw}, $d_K^2/n= \Theta(K^{-2\alpha}L(n/K)^2n^{-\eta})$. Therefore, 
 \begin{eq} \label{c2expr::nu_difference}
  &\nu-\nu_n \\
  &= \frac{1}{\expt{D}}\left( \sum_{i=1}^{K} \int_{(i-1)/n}^{i/n}G^{-1}(u)^2 \dif u- \frac{1}{n}\sum_{i=1}^K d_i^2 \right)+O(K^{-2\alpha}L(n/K)^2n^{-\eta}).
 \end{eq} Again, using \eqref{c2defn::degree:powerlaw}, 
 \begin{equation}\label{c2expr::first_sum}
  \frac{1}{n}\sum_{i=1}^K d_i^2  
  = n^{-\eta} \sum_{i=1}^K \left( \frac{C_F}{i}\right)^{2\alpha}L(n/i)^2+o(c_n^{-1})=c_n^{-1}\sum_{i=1}^K \left( \frac{C_F}{i}\right)^{2\alpha}+\varepsilon(c_n,K),
 \end{equation} where the last equality follows using the fact that $L(\cdot)$ is a slowly-varying function. Note that the error term $\varepsilon(c_n,K)$ in \eqref{c2expr::first_sum} satisfies \linebreak $\lim_{n\to\infty} c_n\varepsilon(c_n,K)=0$ for each fixed $K\geq 1$. Again,
 \begin{equation} \label{c2expr::second_sum}
 \begin{split}
   \sum_{i=1}^K \int_{(i-1)/n}^{i/n} G^{-1}(u)^2 \dif u &= n^{-\eta} \sum_{i=1}^K \int_{(i-1)}^{i}\left( \frac{C_F}{u}\right)^{2\alpha}L(n/u)^2\dif u+o(c_n^{-1})\\
   &=c_n^{-1}\sum_{i=1}^K \int_{(i-1)}^{i}\left( \frac{C_F}{u}\right)^{2\alpha}\dif u +\varepsilon'(c_n,K),
   \end{split}
 \end{equation}where $\lim_{n\to\infty} c_n\varepsilon'(c_n,K)=0$ for each fixed $K\geq 1$.
 Thus combining \eqref{c2expr::nu_difference}, \eqref{c2expr::first_sum}, and  \eqref{c2expr::second_sum} and first letting $n\to \infty$ and then $K\to\infty$, we get
 \begin{equation}
  \lim_{n\to\infty} c_n(\nu_n-\nu)=\lambda_0,
 \end{equation}where 
 \begin{equation}\lambda_0 =-\frac{C_F^{2\alpha}}{\expt{D}}\sum_{i=1}^{\infty} \left(  \int_{i-1}^{i}u^{-2\alpha}\dif u-i^{-2\alpha} \right).
 \end{equation}
 Using Euler-Maclaurin summation \cite[Page 333]{H49} it can be seen that  $\lambda_0$ is finite which completes the proof. 
\end{proof}

\begin{remark}
 \normalfont
 Note that if we add approximately $cn^{1-\eta}$ ($c>0$ is a constant) ones in the degree sequence given in \eqref{c2defn::degree:powerlaw}, then  we end up with another configuration model for which $\lim_{n\to\infty}n^{\eta}(\nu_n-\nu)= \zeta'$ with $\zeta> \zeta'$. Similarly, deleting $cn^{1-\eta}$ ones from the degree sequence increases the new $\zeta$ value. This gives an obvious way to perturb the degree sequence in such a way that the configuration model is in different locations within the critical scaling window. In our proofs, we will only use the precise asymptotics of the \emph{high}-degree vertices. Thus, a small (suitable) perturbation in the degrees of the \emph{low}-degree vertices does not change the scaling behavior fundamentally, except for a change in the location inside the scaling window.
\end{remark}  

\begin{remark} \normalfont If $\nu$ in \eqref{c2defn:critical} is larger than one, then the degree sequence satisfies Assumption~\ref{c2assumption2}. Therefore, the results for critical percolation also hold in this setting. 
\eqref{c2tot-deg-error-1} implies that the additional assumption in Theorem~\ref{c2thm:mul:conv} is also satisfied.

\end{remark}
\subsection{Random degrees sampled from a power-law distribution}\label{c2sec:iid-degrees}
  We now consider the set-up discussed in \cite{Jo10}. Let $D_1,\dots, D_n$ be i.i.d samples from a distribution $F$, where $F$ is defined in \eqref{c2defn::F}. Therefore, the asymptotic relation in \eqref{c2eqn:inverse} holds.
   Consider the random degree sequence $\boldsymbol{d}$ where $d_i=D_{\sss(i)}$, $D_{\sss(i)}$ being the $i^{th}$ order statistic of $(D_1,\dots, D_n)$. We show that $\boldsymbol{d}$ satisfies Assumption~\ref{c2assumption1} almost surely under a suitable coupling. We use a coupling from \cite[Section 13.6]{B68}. 
   Let $(E_1,E_2,\dots )$ be an i.i.d sequence of unit rate exponential random variables and let $\Gamma_i:= \sum_{j=1}^iE_j$. Let
  \begin{equation}
   \bar{d}_i=\bar{D}_{\sss(i)}=G^{-1}(\Gamma_i/\Gamma_{n+1}).
  \end{equation}   
It can be checked that $(d_1,\dots,d_n)\stackrel{\sss d}{=}(\bar{d}_1,\dots,\bar{d}_n)$ and therefore, we will ignore the bar in the subsequent notation. Note that, by the strong law of large numbers, $\Gamma_{n+1}/n \asto 1$. Thus,  for each fixed $i\geq 1$, $\Gamma_{n+1}/(n\Gamma_i)\asto 1/\Gamma_i$. 
Using \eqref{c2eqn:inverse}, we see that $\boldsymbol{d}$ satisfies Assumption~\ref{c2assumption1}~\eqref{c2assumption1-1} almost surely under this coupling with $\theta_i=(C_F/\Gamma_i)^{\alpha}$. 
The first two conditions of Assumption~\ref{c2assumption1}~\eqref{c2assumption1-2} are trivially satisfied by $\boldsymbol{d}$  almost surely using the strong law of large numbers.
 Using the third condition, we first claim that
\begin{equation}\label{c2eq:gamma-summable}
\PR\bigg(\sum_{i=1}^{\infty}\Gamma_i^{-3\alpha}<\infty\bigg)=1. 
\end{equation}To see \eqref{c2eq:gamma-summable}, note that $\Gamma_i$ has a Gamma distribution with shape parameter $i$ and scale parameter 1.  Thus, for $i>3\alpha$,
\begin{equation}
\begin{split} \E[\Gamma_i^{-3\alpha}] = \frac{\Gamma(i-3\alpha)}{\Gamma(i)}=i^{-3\alpha}(1+O(1/i)),
\end{split}
\end{equation}where $\Gamma (x)$ is the Gamma function and the last equality follows from  Stirling's approximation. Therefore,
 \begin{equation}
 \E\bigg[\sum_{i=1}^{\infty}\Gamma_i^{-3\alpha}\bigg]=\sum_{i=1}^{\infty}\E\big[\Gamma_i^{-3\alpha}\big]<\infty
\end{equation}and \eqref{c2eq:gamma-summable} follows. Now, using the fact that $\Gamma_{n+1}/n\asto 1$, we can use arguments identical to \eqref{c2eq:3rd-moment-example} to show that $\lim_{K\to\infty}\limsup_{n\to\infty} a_n^{-3}\sum_{i>K}d_i^3=0$ on the event $\{\sum_{i=1}^{\infty}\Gamma_i^{-3\alpha}<\infty\}\cap \{\Gamma_{n+1}/n\to 1\}$. Thus, we have shown that the third condition of Assumption~\ref{c2assumption1}~\eqref{c2assumption1-2} holds almost surely.
\par To see Assumption~\ref{c2assumption1}~\eqref{c2assumption1-3}, an argument similar to Lemma~\ref{c2lem::nu-n} can be carried out to prove that 
\begin{equation}
 \lim_{n\to\infty}c_n(\nu_n-\nu)\asto \Lambda_0,
\end{equation}where 
\begin{equation}\label{c2defn:Lambda-0}\Lambda_0:= -\frac{C_F^{2\alpha}}{\expt{D}}\sum_{i=1}^{\infty} \left(  \int_{\Gamma_{i-1}}^{\Gamma_i}u^{-2\alpha}\dif u-\Gamma_i^{-2\alpha} \right).
\end{equation}Therefore, the results in Section~\ref{c2sec:results} hold conditionally on the degree sequence if we assume the degrees to be i.i.d samples from a distribution of the form \eqref{c2defn::F}. 
For the percolation results, notice that the additional condition in Theorem~\ref{c2thm:mul:conv} is a direct consequence of the convergence rates of sums of i.i.d sequences of random variables~\cite[Corollary 3.22]{K2006}.

\begin{remark}\label{c2rem:jos-vs-this}\normalfont Let us recall the limiting object obtained in \cite[Theorem 8.1]{Jo10} and compare this with the limiting object $\bar{\mathbf{S}}_\infty^{\sss \Lambda_0}$, defined in \eqref{c2defn::limiting::process} with $\Lambda_0$ given by \eqref{c2defn:Lambda-0}. We will prove an analogue of \cite[Theorem 8.1]{Jo10} in Theorem~\ref{c2thm::convegence::exploration_process}. 
Although we use a different exploration process from \cite{Jo10}, the fact that the component sizes are \emph{huge} compared to the number of cycles in a component, means that one can prove Theorem~\ref{c2thm::convegence::exploration_process} for the exploration process in \cite{Jo10} also. 
This will indirectly imply that Joseph's limiting exploration process in \cite[Theorem 8.1]{Jo10} obeys the law of $\bar{\mathbf{S}}_\infty^{\sss \Lambda_0}$, averaged out over the $\Gamma$-values. 
This is counter intuitive, given the vastly different descriptions of the two processes; for example our process does not have independent increments. 
We do not have a direct way to prove the above mentioned claim. 
\end{remark}
\section{Discussion}\label{c2sec:discussion}
 \textbf{Assumptions on the degree distribution.} Let us now briefly explain the significance of Assumption~\ref{c2assumption1}. Unlike the finite third-moment case \cite{DHLS15}, the high-degree vertices dictate the scaling limit in Theorem~\ref{c2thm::conv:component:size} and therefore it is essential to fix their asymptotics through Assumption~\ref{c2assumption1}~\eqref{c2assumption1-1}. Assumption~\ref{c2assumption1}~\eqref{c2assumption1-3} defines the critical window of the phase transition and Assumption~\ref{c2assumption1}~\eqref{c2assumption1-4} is reminiscent of the fact that a configuration model with negligibly small amount of degree-one vertices is always supercritical. Assumption~\ref{c2assumption1}~\eqref{c2assumption1-2} states the finiteness of the first two moments of the degree distribution and fixes the asymptotic order of the third moment. The order of the third moment is crucial in our case. The derivation of the scaling limits for the components sizes is based on the analysis of a walk which encodes the information about the component sizes in terms of the excursions above its past minima \cite{A97,R12,DHLS15,BHL12,BHL10}. Now, the increment distribution turns out to be the size-biased distribution with the sizes being the degrees. Therefore, the third-moment assumption controls the variance of the increment distribution. Another viewpoint is that the components can be locally approximated by a branching process $\mathcal{X}_n$ with the variance of the same order as the third moment of the degree distribution. Thus Assumption~\ref{c2assumption1}~\eqref{c2assumption1-2} controls the order of the survival probability of $\mathcal{X}_n$, which is intimately related to the asymptotic size of the largest components. \vspace{.2cm}
\\ 
 \textbf{Connecting the barely subcritical and supercritical regimes.} The barely subcritical (supercritical) regime corresponds to the case when $\nu_n(\lambda_n)=1+\lambda_nc_n^{-1}$ for some $\lambda_n\to -\infty$ ($\lambda_n\to \infty$) and $\lambda_n = o(c_n^{-1})$. \citet{J08} showed that the size of the $k^{th}$ largest cluster for a subcritical configuration model (i.e., the case $\nu_n\to\nu$ and $\nu<1$) is $d_k/(1-\nu)$ (see \cite[Remark 1.4]{J08}). In \cite{BDHS17}, we show that this is indeed the case for the entire barely subcritical regime, i.e., the size of the $k^{th}$ largest cluster is $d_k/(1-\nu_n(\lambda_n))=\Theta(b_n|\lambda_n|^{-1})$.
 In the barely supercritical case, the giant component can be \emph{locally} approximated  by a  branching process $\mathcal{X}_n$ having variance of the order $a_n^3/n$ and the size of the giant component is of the order $n\rho_n$, where $\rho_n$ is the survival probability of $\mathcal{X}_n$~\cite{HJL16}. 
The asymptotic size of the giant component turns out to be $\Theta(b_n|\lambda_n|)$.
 Therefore, the fact that the sizes of the maximal components in the critical scaling window are $\Theta(b_n)$ for $\lambda_n=\Theta(1)$ proves a continuous phase transition property for the configuration model within the whole critical regime. 
  \vspace{.2cm}\\
\textbf{Percolation.} The main reason to study percolation in this chapter is to understand the evolution of the component sizes and the surplus edges over the critical window in Theorem~\ref{c2thm:mul:conv}. 
It turns out that a precise characterization of the evolution of the percolation clusters is necessary for understanding the minimal spanning tree of the giant component with i.i.d.~weights on each edge \cite{ABGM13}.
Also, since the percolated configuration model is again a configuration model \cite{F07,J09}, the natural way to study the evolution of the clusters sizes of configuration models over the critical window is through percolation. \vspace{.2cm} \\
\textbf{Universality.} The limiting object in Theorem~\ref{c2thm::conv:component:size} is identical to that in \cite[Theorem 1.1]{BHL12} for rank-1 inhomogeneous random graphs. 
Thus, $\mathrm{CM}_n(\boldsymbol{d})$ with regularly-varying tails falls onto the domain of attraction of the new universality class studied in \cite{BHL12}. This is again confirming the predictions made by statistical physicists that the nature of the phase transition does not depend on the precise details of the model. Our scaling limit fits into the general class of limits predicted in \cite{AL98}. 
In the notation of \cite[(6)]{AL98}, the scaling limits $\mathrm{CM}_n(\boldsymbol{d})$, under Assumption~\ref{c2assumption1}, give rise to the case $\kappa=0$. 
To understand this, let us discuss some existing works. In \cite{A97,AP00,Jo10,BHL10}, and Chapter~\ref{chap:thirdmoment} the limiting component sizes are  described by the excursions of a Brownian motion with a parabolic drift. 
All these models had a common property: if the component sizes in the barely subcritical regime are viewed as masses then (i) these masses merge as approximate multiplicative coalescents in the critical window, and (ii) each individual mass is negligible/``dust" compared to the sum of squares of the masses in the barely subcritical regime. Indeed, (ii) is observed in \cite[(10)]{A97}, \cite[(4)]{AP00}. 
In the case of \cite{BHL12} and this chapter, the barely-subcritical component sizes do not become negligible due to the existence of the high-degree vertices (see \cite[Theorem 1.3]{BHL12}). As discussed in \cite[Section 1.4]{AL98}, these \emph{large} barely-subcritical clusters can be thought of as nuclei, not interacting with each other and ``sweeping up the smaller clusters in such a way that the relative masses converge''.  It will be fascinating to find a class of random graphs, used to model real-life networks, that has both the nuclei and a good amount of dust in the barely-subcritical regime, so that the scaling limits predicted by \cite{AL98} can be observed in complete generality.
 \vspace{.2cm}\\
\textbf{Component sizes and the width of the critical window}. We have already discussed how the width of the scaling window and the order of the maximal degrees should lead the asymptotic size of the components to be of the order $b_n$. For the finite third-moment case, the size of the largest component is of the order $n^{2/3}\gg b_n$. We do not have a very intuitive explanation for the reduced sizes of the components except for the fact that a similar property is true for the survival probability of a slightly supercritical branching process. The width of the critical window decreases by a factor of $L(n)^{-2}$ as compared to \cite{BHL12} if the size of the high-degree vertices increases by a factor of $L(n)$ (see \eqref{c2eqn:notation-power*}). Indeed, an increase in the degrees of the high-degree vertices is expected to start the merging of the barely-subcritical nuclei earlier, resulting in an increase in the width of the critical window. The fact that the width decreases by a factor of $L(n)^{-2}$ comes out of our calculations. \vspace{.2cm} \\
\textbf{Overview of the proofs}.  The proofs of Theorems~\ref{c2thm::conv:component:size} and \ref{c2thm:spls} consist of two important steps. First, we define an exploration algorithm on the graph that explores one edge of the graph at each step. The algorithm produces a walk, termed exploration process, that encodes the information about the number of edges in the explored components in terms of the hitting times to its past minima. In Section~\ref{c2sec:conv-expl}, the exploration process, suitably rescaled, is shown to converge. The surplus edges in the components are asymptotically negligible  compared to the component sizes; these two facts together give us the finite-dimensional scaling limit of the re-scaled component sizes. The proof of Theorem~\ref{c2thm::conv:component:size} follows from the asymptotics of the susceptibility function in Section~\ref{c2:tightness-comp-size}.
 The joint convergence of the component sizes and surplus edges is proved by verifying a uniform tightness condition on the surplus edges in Section~\ref{c2sec:surplus}. Then, in Section~\ref{c2sec:simple-graphs}, we exploit the idea that the large components are explored before any self-loops or multiple edges are created and conclude the proof of Theorem~\ref{c2thm:simple-graph}. The proof of Theorem~\ref{c2thm:percolation} is completed by showing that the percolated degree sequence is again a configuration model satisfying Assumption~\ref{c2assumption1}. Section~\ref{c2sec:conv-amc} is devoted to the proof of Theorem~\ref{c2thm:mul:conv} which exploits different properties of the augmented multiplicative coalescent process.

\section{Convergence of the exploration process}\label{c2sec:conv-expl}
  We start by describing how the connected components in the graph can be explored while generating the random graph simultaneously:
\begin{algo}[Exploring the graph]\label{c2algo-expl}\normalfont  
The algorithm carries along vertices that can be alive, active, exploring and killed and half-edges that can be alive, active or killed. 
We sequentially explore the graph as follows:
\begin{itemize}
\item[(S0)] At stage $i=0$, all the vertices and the half-edges are \emph{alive} but none of them are \emph{active}. Also, there are no \emph{exploring} vertices. 
\item[(S1)]  At each stage $i$, if there is no active half-edge at stage $i$, choose a vertex $v$ proportional to its degree among the alive (not yet killed) vertices and declare all its half-edges to be \emph{active} and declare $v$ to be \emph{exploring}. If there is an active vertex but no exploring vertex, then declare the \emph{smallest} vertex to be exploring.
\item[(S2)] At each stage $i$, take an active half-edge $e$ of an exploring vertex $v$ and pair it uniformly to another alive half-edge $f$. Kill $e,f$. If $f$ is incident to a vertex $v'$ that has not been discovered before, then declare all the half-edges incident to $v'$ active, except $f$ (if any). 
If $\mathrm{degree}(v')=1$ (i.e. the only half-edge incident to $v'$ is $f$) then kill $v'$. Otherwise, declare $v'$ to be active and larger than all other vertices that are alive. After killing $e$, if $v$ does not have another active half-edge, then kill $v$ also.

\item[(S3)] Repeat from (S1) at stage $i+1$ if not all half-edges are already killed.
\end{itemize}
\end{algo}
Algorithm~\ref{c2algo-expl} gives a \emph{breadth-first} exploration of the connected components of $\mathrm{CM}_n(\boldsymbol{d})$. Define the exploration process by
   \begin{equation}\label{c2defn:exploration:process}
    S_n(0)=0,\quad
     S_n(l)=S_n(l-1)+d_{(l)}J_l-2,
    \end{equation} where $J_l$ is the indicator that a new vertex is discovered at time $l$ and $d_{(l)}$ is the degree of the new vertex chosen at time $l$ when $J_l=1$.  Suppose $\mathscr{C}_{k}$ is the $k^{th}$ connected component explored by the above exploration process and define
$\tau_{k}=\inf \big\{ i:S_{n}(i)=-2k \big\}.$
Then  $\mathscr{C}_{k}$ is discovered between the times $\tau_{k-1}+1$ and $\tau_k$, and  $\tau_{k}-\tau_{k-1}-1$ gives the total number of edges in $\mathscr{C}_k$.
 Call a vertex \emph{discovered} if it is either active or killed. Let $\mathscr{V}_l$ denote the set of vertices discovered up to time $l$ and $\mathcal{I}_i^n(l):=\ind{i\in\mathscr{V}_l}$. Note that 
   \begin{equation}
    S_n(l)= \sum_{i\in [n]} d_i \mathcal{I}_i^n(l)-2l=\sum_{i\in [n]} d_i \left( \mathcal{I}_i^n(l)-\frac{d_i}{\ell_n}l\right)+\left( \nu_n(\lambda)-1\right)l.
   \end{equation} 
   Recall the notation in \eqref{c2eqn:notation-power*}. Define the re-scaled version $\bar{\mathbf{S}}_n$ of $\mathbf{S}_n$ by $\bar{S}_n(t)= a_n^{-1}S_n(\lfloor b_nt \rfloor)$. Then, by Assumption~\ref{c2assumption1}~\eqref{c2assumption1-3},
   \begin{equation} \label{c2eqn::scaled_process}
    \bar{S}_n(t)= a_n^{-1} \sum_{i\in [n]}d_i\left( \mathcal{I}_i^n(tb_n)-\frac{d_i}{\ell_n}tb_n \right)+\lambda t +o(1).
   \end{equation}Note the similarity between the expressions in \eqref{c2defn::limiting::process} and \eqref{c2eqn::scaled_process}. We will prove the following:
   \begin{theorem} \label{c2thm::convegence::exploration_process} Consider the process $\bar{\mathbf{S}}_n:= (\bar{S}_n(t))_{t\geq 0}$ defined in \eqref{c2eqn::scaled_process} and recall the definition of  $\bar{\mathbf{S}}_\infty:=  (\bar{S}_\infty(t))_{t\geq 0} $ from \eqref{c2defn::limiting::process}. Then, 
 \begin{equation}
  \bar{\mathbf{S}}_n \dto \bar{\mathbf{S}}_\infty
 \end{equation} with respect to the Skorohod $J_1$ topology.
\end{theorem}
 The proof of Theorem~\ref{c2thm::convegence::exploration_process} is completed by showing that the summation term in  \eqref{c2eqn::scaled_process} is predominantly carried by the first \emph{few} terms and the limit of the first few terms gives rise to the limiting process given in \eqref{c2defn::limiting::process}. 
 Fix $K\geq 1$ to be large. 
 Denote by $\mathscr{F}_l$ the sigma-field containing the information generated up to time $l$ by Algorithm~\ref{c2algo-expl}. 
 Also, let $\Upsilon_l$ denote the set of time points  up to time $l$ when a component was discovered and $\upsilon_l=|\Upsilon_l|$. 
 Note that we have lost $2(l-\upsilon_l)$ half-edges by time $l$. 
 Thus, on the set $\{\mathcal{I}_i^n(l)=0\}$,
  \begin{equation}\label{c2eqn:increment:indicator}
  \begin{split}
   \prob{\mathcal{I}_i^n(l+1)=1\big| \mathscr{F}_l}=
   \begin{cases} \frac{d_i}{\ell_n-2(l-\upsilon_l)-1} & \text{ if } l\notin \Upsilon_l,\\
    \frac{d_i}{\ell_n-2(l-\upsilon_l)} & \text{ otherwise } 
   \end{cases}
   \end{split}
  \end{equation}and,  uniformly over $l\leq Tb_n$, 
  \begin{equation}\label{c2eq:prob-ind}
  \prob{\mathcal{I}_i^n(l+1)=1\big| \mathscr{F}_l} \geq \frac{d_i}{\ell_n}\quad\text{ on the set } \{\mathcal{I}_i^n(l)=0\}.
  \end{equation} 
 Denote $
  M_n^K(l)=a_n^{-1} \sum_{i>K}d_i\big( \mathcal{I}_i^n(l)-\frac{d_i}{\ell_n}l \big).$ Then,
 \begin{equation}
 \begin{split}
  &\E\big[M_n^K(l+1)-M_n^K(l) \big| \mathscr{F}_l\big]=\E\bigg[\sum_{i=K+1}^na_n^{-1}d_i \left(\mathcal{I}^n_i(l+1)-\mathcal{I}_i^n(l)-\frac{d_i}{\ell_n}\right)\Big| \mathscr{F}_l\bigg]\\
  &= \sum_{i=K+1}^n a_n^{-1}d_i \left(\E\big[\mathcal{I}^n_i(l+1)\big| \mathscr{F}_l\big]\ind{\mathcal{I}_i^n(l)=0} - \frac{d_i}{\ell_n} \right)\geq 0.
  \end{split}
 \end{equation} Thus $(M_n^K(l))_{l= 1}^{Tb_n}$ is a sub-martingale.  Further, \eqref{c2eqn:increment:indicator} implies that, uniformly for all $l\leq Tb_n$,
  \begin{equation}\label{c2prob-ind-lb}
   \prob{\mathcal{I}_i^n(l)=0} \geq \left(1-\frac{d_i}{\ell_n'} \right)^l,
  \end{equation} where $\ell_n'=\ell_n-2Tb_n-1$.
  Thus, Assumption~\ref{c2assumption1}~\eqref{c2assumption1-2} gives
  \begin{equation} \begin{split}
    \big| &\E[M_n^K(l)]\big|= a_n^{-1} \sum_{i=K+1}^n d_i\left( \prob{\mathcal{I}_i^n(l)=1}-\frac{d_i}{\ell_n}l \right)
    \\& \leq a_n^{-1} \sum_{i=K+1}^n d_i\left( 1-\left(1-\frac{d_i}{\ell_n'} \right)^l-\frac{d_i}{\ell_n'}l \right)+a_n^{-1}l\sum_{i\in [n]}d_i^2\left(\frac{1}{\ell_n'}-\frac{1}{\ell_n}\right)\\
    &\leq \frac{l^2}{2\ell_n'^2 a_n } \sum_{i=K+1}^n d_i^3+o(1) 
    \\&\leq \frac{T^2n^{2\rho}n^{3\alpha}L(n)^3}{2\ell_n'^2L(n)^2n^{\alpha}L(n)}\left( a_n^{-3}\sum_{i=K+1}^{n}d_i^3\right)+o(1)\\
    &=C\bigg(a_n^{-3}\sum_{i=K+1}^{n}d_i^3\bigg)+o(1),
  \end{split}
  \end{equation} for some constant $C>0$, where we have used the fact that 
  \begin{eq}
  &a_n^{-1}l\sum_{i\in [n]}d_i^2\Big(\frac{1}{\ell_n'}-\frac{1}{\ell_n}\Big)=O(n^{2\rho+1-\alpha-2}/L(n)^3)\\
  &\hspace{.5cm}=O(n^{(\tau-4)/(\tau-1)}/L(n)^3)=o(1),
  \end{eq} uniformly for $l\leq Tb_n$. Therefore, uniformly over $l\leq Tb_n$,
  \begin{equation}\label{c2expectation::M_n^K}
   \lim_{K\to\infty}\limsup_{n\to\infty}\big| \E[M_n^K(l)]\big|=0.
  \end{equation} Now, note that for any $(x_1,x_2,\dots)$, $0\leq a+b \leq x_i$ and $a,b>0$ one has $\prod_{i=1}^R(1-a/x_i)(1-b/x_i)\geq \prod_{i=1}^R (1-(a+b)/x_i)$. Thus, by \eqref{c2eqn:increment:indicator}, for all $l\geq 1$ and $i\neq j$, 
  \begin{equation}\label{c2neg:correlation}
  \prob{\mathcal{I}_i^n(l)=0, \mathcal{I}_j^n(l)=0}\leq \prob{\mathcal{I}_i^n(l)=0}\prob{\mathcal{I}_j^n(l)=0}
  \end{equation} and therefore $\mathcal{I}_i^n(l)$ and $\mathcal{I}^n_j(l)$ are negatively correlated. Observe also that, uniformly over $l\leq Tb_n$, 
  \begin{eq}\label{c2var-ind-ub}
   &\var{\mathcal{I}_i^n(l)}\leq  \prob{\mathcal{I}_i^n(l)=1} \\
   &\leq \sum_{l_1=1}^l\prob{\text{vertex  }i \text{ is first discovered at stage }l_1 }\leq \frac{ld_i }{\ell_n'}.
  \end{eq}  
  Therefore, using the negative correlation in \eqref{c2neg:correlation}, uniformly over $l\leq Tb_n$, 
  \begin{equation} \label{c2variance::M_n^k}
   \begin{split}
    \var{M_n^K(l)}&\leq a_n^{-2}\sum_{i=K+1}^n d_i^2 \var{\mathcal{I}_i^n(l)} \leq \frac{l}{\ell_n'a_n^2}\sum_{i=K+1}^n d_i^3 \leq Ca_n^{-3}\sum_{i=K+1}^{n}d_i^3,
   \end{split}
  \end{equation}for some constant $C>0$ and by using Assumption~\ref{c2assumption1}~\eqref{c2assumption1-2} again,
  \begin{equation}
   \lim_{K\to\infty}\limsup_{n\to\infty} \var{M_n^K(l)}= 0,
  \end{equation}uniformly for $l\leq Tb_n$.
 Now we can use the super-martingale inequality \cite[Lemma 2.54.5]{RW94} stating that for any super-martingale $(M(t))_{t\geq 0}$, satisfying $M(0)=0$, 
 \begin{equation}\label{c2eqn:supmg:ineq}
  \varepsilon \prob{\sup_{s\leq t}|M(s)|>3\varepsilon}\leq 3\expt{|M(t)|}\leq 3\left(|\expt{M(t)}|+\sqrt{\var{M(t)}}\right).
 \end{equation}
  Using \eqref{c2expectation::M_n^K}, \eqref{c2variance::M_n^k}, and \eqref{c2eqn:supmg:ineq}, together with the fact that $(-M_n^K(l))_{l=1}^{Tb_n}$ is a super-martingale, we get, for any $\varepsilon >0 $,
 \begin{equation} \label{c2tail::martingale}
  \lim_{K\to\infty}\limsup_{n\to\infty}\PR\bigg(\sup_{l\leq Tb_n}|M_n^K(l)|> \varepsilon \bigg)=0.
 \end{equation}
Define the truncated exploration process
\begin{equation}\label{c2eqn::truncated_scaled_process}
  \bar{S}_n^K(t)= a_n^{-1} \sum_{i=1}^Kd_i\left( \mathcal{I}_i^n(tb_n)-\frac{d_i}{\ell_n}tb_n \right)+\lambda t.
\end{equation}
 Define $\mathcal{I}_i^n(tb_n)=\mathcal{I}_i^n(\floor{tb_n})$ and recall that $\mathcal{I}_i(s):=\ind{\xi_i\leq s }$ where $\xi_i\sim \mathrm{Exp}(\theta_i/\mu)$. 
 \begin{lemma} \label{c2lem::convergence_indicators}
   Fix any $K\geq 1$. As $n\to\infty$,
   \begin{equation}
    \left( \mathcal{I}_i^n(tb_n) \right)_{i\in[K],t\geq 0} \dto \left( \mathcal{I}_i(t) \right)_{i\in[K],t\geq 0}.
   \end{equation}
 \end{lemma}
\begin{proof} By noting that $(\mathcal{I}_i^n(tb_n))_{t\geq 0}$ are indicator processes, it is enough to show that 
\begin{equation}
 \prob{\mathcal{I}_i^n(t_ib_n)=0,\ \forall i\in [K]} \to \prob{\mathcal{I}_i(t_i)=0,\ \forall i\in [K]} = \exp \Big( -\mu^{-1}\sum_{i=1}^{K} \theta_it_i\Big).
\end{equation} for any $t_1,\dots,t_K\in \mathbb{R}$. Now, 
\begin{equation} \label{c2lem::eqn::expression1}
 \prob{\mathcal{I}_i^n(m_i)=0,\ \forall i\in [K]}=\prod_{l=1}^{\infty}\Big(1-\sum_{i\leq K:l\leq m_i}\frac{d_i}{\ell_n-\Theta(l)} \Big),
\end{equation}where the $\Theta(l)$ term arises from the expression in \eqref{c2eqn:increment:indicator} and noting that $\upsilon_l\leq l$. Taking logarithms on both sides of \eqref{c2lem::eqn::expression1} and using the fact that $l\leq \max m_i=\Theta(b_n)$ we get 
\begin{equation}\label{c2lem::eqn::ex1}
 \begin{split}
  \prob{\mathcal{I}_i^n(m_i)=0\, \forall i\in [K]}&= \exp\Big( - \sum_{l=1}^{\infty}\sum_{i\leq K:l\leq m_i} \frac{d_i}{\ell_n}+o(1) \Big)\\
  &= \exp\Big( -\sum_{i\in [K]} \frac{d_im_i}{\ell_n} +o(1) \Big).
 \end{split}
\end{equation} Putting $m_i=t_ib_n$, Assumption~\ref{c2assumption1}~\eqref{c2assumption1-1},~\eqref{c2assumption1-2} gives
\begin{equation} \label{c2lem::eqn::expression2}
 \frac{m_id_i}{\ell_n}= \frac{\theta_it_i}{\mu} (1+o(1)).
\end{equation}
Hence \eqref{c2lem::eqn::ex1} and \eqref{c2lem::eqn::expression2}  complete the proof of Lemma \ref{c2lem::convergence_indicators}.
\end{proof}

\begin{proof}[Proof of Theorem~\ref{c2thm::convegence::exploration_process}]
 The proof of Theorem~\ref{c2thm::convegence::exploration_process} now follows from \eqref{c2eqn::scaled_process}, \eqref{c2tail::martingale} and Lemma~\ref{c2lem::convergence_indicators} by first taking the limit as $n\to \infty$ and then taking the limit as $K\to\infty$.
\end{proof}

\begin{theorem}\label{c2thm:conv:refl:process} Recall the definition of $\refl{\bar{\mathbf{S}}_\infty}$ from \eqref{c2defn::reflected-Levy}. As $n\to\infty$,
\begin{equation}
 \refl{\bar{\mathbf{S}}_n} \dto \refl{\bar{\mathbf{S}}_\infty}.
\end{equation}
\end{theorem} 
\begin{proof}
 This follows from Theorem~\ref{c2thm::convegence::exploration_process} and the fact that the reflection is Lipschitz continuous with respect to the Skorohod $J_1$ topology (see \cite[Theorem 13.5.1]{W02}).
\end{proof}

This also implies that Algorithm~\ref{c2algo-expl} explores the \emph{large} components before time $Tb_n$ for large $T$. 
Next, we show that the function mapping an element of $\mathbb{D}[0,\infty)$ to its largest excursions, is continuous on a special subset $A$ of $\mathbb{D}[0,\infty)$ and the process $\refl{\bar{\mathbf{S}}_{\infty}}$ has sample paths in $A$ almost surely. Therefore, Theorem~\ref{c2thm::convegence::exploration_process} gives the scaling limit of the number of edges in the components ordered as a non-increasing sequence. 
Finally, we show that the number of surplus edges discovered up to time  $Tb_n$ are negligible and thus the convergence of the component sizes in Theorem~\ref{c2thm::conv:component:size} follows.

\subsection{Tightness of the component sizes}
\label{c2:tightness-comp-size}The following proposition establishes a uniform tail summability condition that is required for the tightness of the (scaled) ordered vector of component sizes with respect to the $\ell^2_{\shortarrow}$ topology: 
\begin{proposition}\label{c2prop-l2-tightness} For any $\varepsilon >0$,
\begin{equation}
 \lim_{K\to\infty}\limsup_{n\to\infty}\PR\bigg(\sum_{i>K}|\mathscr{C}_{\sss(i)}|^2>\varepsilon b_n^2\bigg)=0.
\end{equation}
\end{proposition}
Roughly speaking, the proof is based on the fact that the graph, obtained by removing a large number of high-degree vertices, yields a graph that approaches subcriticality. More precisely,  we prove Lemma~\ref{c2lem::tail_sum_squares} below to complete the proof of Proposition~\ref{c2prop-l2-tightness}. This fact is not true for the finite third-moment setting  \cite{DHLS15}. 
However, since the large-degree vertices guide the scaling behavior in the infinite third-moment case, the observation in Lemma~\ref{c2lem::tail_sum_squares} saves some computational complexity, and gives a different proof of the $\ell^2_{\shortarrow}$ tightness than the arguments with size-biased point processes originally described in \cite{A97}. 
\begin{lemma} \label{c2lem::tail_sum_squares} Consider $\mathrm{CM}_n(\boldsymbol{d})$ satisfying \textrm{Assumption~\ref{c2assumption1}}.
Let $\mathcal{G}^{\sss[K]}$ be the random graph obtained by removing all edges attached to vertices $1,\dots,K$ and let $\boldsymbol{d}'$ be the obtained degree sequence. Suppose $V_n$ is a random vertex of $\mathcal{G}^{\sss[K]}$ chosen independently of the graph and let $\mathscr{C}^{\sss[K]}(V_n)$ be the corresponding component. Let $\{\mathscr{C}_{\sss (i)}^{\sss [K]}:i\geq 1 \}$ be the components of $\mathcal{G}^{\sss [K]}$, ordered according to their sizes. Then,
\begin{equation}\label{c2expt-cvn-K-removed}
\lim_{K\to\infty} \limsup_{n\to\infty} c_n^{-1}\expt{|\mathscr{C}^{\sss [K]}(V_n)|}=0.
\end{equation}Consequently, for any $\varepsilon > 0$,
\begin{equation}\label{c2tail-sum-squares-K}
 \lim_{K\to\infty}\limsup_{n\to\infty} \PR\bigg(\sum_{i\geq 1}\big|\mathscr{C}_{\sss (i)}^{\sss [K]}\big|^2> \varepsilon b_n^2\bigg)=0.
\end{equation}
\end{lemma}

\begin{proof}
We make use of a result due to \citet{J09b} regarding bounds on the susceptibility functions for the configuration model. 
In fact, \cite[Lemma 5.2]{J09b} shows that, for any configuration model $\mathrm{CM}_n(\boldsymbol{d})$ with $\nu_n<1$, 
 \begin{equation}\label{c2bound::expt-cluster-size}
  \expt{|\mathscr{C}(V_n)|} \leq 1+\frac{\expt{D_n}}{1-\nu_n}.
 \end{equation}  Now, conditional on the set of removed half-edges, $\mathcal{G}^{\sss [K]}$ is still a configuration model with some degree sequence $\boldsymbol{d}'$ with $d_i'\leq d_i$ for all $i\in [n]\setminus [K]$ and $d_i'=0$ for $i\in [K]$. Further, the criticality parameter of $\mathcal{G}^{\sss [K]}$ satisfies 
 \begin{equation}\label{c2eqn:nu-K}
  \begin{split}
   \nu^{\sss [K]}_n&= \frac{\sum_{i\in [n]} d_i'(d'_i-1)}{\sum_{i\in [n]} d_i'}\leq \frac{\sum_{i\in [n]}d_i(d_i-1)-\sum_{i=1}^Kd_i(d_i-1)}{\ell_n-2\sum_{i=1}^Kd_i}\\
   &=\nu_n-C_1n^{2\alpha -1}L(n)^2\sum_{i\leq K}\theta_i^2=\nu_n-C_1c_n^{-1}\sum_{i\leq K}\theta_i^2
  \end{split}
 \end{equation}for some constant $C_1>0$.
 Since $\boldsymbol{\theta}\notin \ell^2_{\shortarrow}$, $K$ can be chosen large enough such that $\nu^{\sss[K]}_n < 1$ uniformly for all $n$. Also $\sum_{i\in [n]}d'_i=\ell_n+o(n)$ for each fixed $K$. Let $\mathbb{E}_K[\cdot]$ denote the conditional expectation, conditioned on the set of removed half-edges. Using \eqref{c2bound::expt-cluster-size} on $\mathcal{G}^{\sss [K]}$, we get
 \begin{equation}
 \begin{split}
  \mathbb{E}_K\big[|\mathscr{C}^{\sss[K]}(V_n)|\big]&\leq \frac{C_2}{1-\nu^{\sss[K]}_n}\leq \frac{C_2}{1-\nu_n+C_1c_n^{-1}\sum_{i\leq K}\theta_i^2}\leq  \frac{C_2c_n}{-\lambda+C_1\sum_{i\leq K}\theta_i^2},
 \end{split}
 \end{equation}for some constant $C_2>0$. Using the fact that $\boldsymbol{\theta}\notin \ell^2_{\shortarrow}$, this concludes the proof of \eqref{c2expt-cvn-K-removed}. The proof of \eqref{c2tail-sum-squares-K} follows from \eqref{c2expt-cvn-K-removed} by using the Markov inequality and the observation that
 \begin{equation} \label{c2comp-vs-rnd-choice}
 \mathbb{E}\bigg[\sum_{i\geq 1}|\mathscr{C}_{\sss (i)}^{\sss [K]}|^2\bigg]= n\expt{|\mathscr{C}^{\sss [K]}(V_n)|}.
 \end{equation}
\end{proof} 
\begin{proof}[Proof of Proposition~\ref{c2prop-l2-tightness}] Denote the sum of squares of the component sizes excluding the components containing vertices $1,2,\dots, K$ by $\mathscr{S}_K$.
 Note that \begin{equation}
 \sum_{i>K}|\mathscr{C}_{\sss (i)}|^2\leq \mathscr{S}_K\leq \sum_{i\geq 1} |\mathscr{C}_{\sss(i)}^{\sss [K]}|^2. 
 \end{equation}  Thus, Proposition~\ref{c2prop-l2-tightness} follows from Lemma~\ref{c2lem::tail_sum_squares}.
\end{proof}

\subsection{Large components are explored early}
An important consequence of Proposition~\ref{c2prop-l2-tightness} is that after time $\Theta(b_n)$, Algorithm~\ref{c2algo-expl} does not explore large components. The precise statement needed to complete our proof is given below. This is an essential ingredient to conclude the convergence of the component sizes from the convergence of the exploration process since Theorem~\ref{c2thm::convegence::exploration_process} only gives information about the components explored on the time scale of the order $b_n$.
\begin{lemma}\label{c2lem:no-large-comp-later} Let $\mathscr{C}_{\max}^{\sss \geq T}$ be the largest among those components which are started exploring after time $Tb_n$ by \textrm{Algorithm~\ref{c2algo-expl}}. Then, for any $\varepsilon >0$,
\begin{equation}
 \lim_{T\to\infty}\limsup_{n\to\infty}\prob{|\mathscr{C}_{\max}^{\sss \geq T}|>\varepsilon b_n}=0.
\end{equation}

\end{lemma}

\begin{proof}

 Define the event $$\mathscr{A}_{\sss K,T}^n:= \{\text{all the vertices of }[K] \text{ are explored before time }Tb_n\}.$$ Recall the definition of $\mathscr{C}_{\sss (i)}^{\sss [K]}$ from Lemma~\ref{c2lem::tail_sum_squares}. Firstly, note that 
 \begin{equation}\label{c2eq:CgeqT1}
  \prob{|\mathscr{C}_{\max}^{\sss \geq T}|>\varepsilon b_n, \ \mathscr{A}_{\sss K,T}^n}\leq \PR\bigg(\sum_{i\geq 1}\big|\mathscr{C}_{\sss (i)}^{\sss [K]}\big|^2> \varepsilon^2 b_n^2\bigg).
 \end{equation}Moreover, using \eqref{c2eqn:increment:indicator} and the fact that $d_jb_n=\Theta(n)$, we get 
 \begin{equation}\label{c2eq:CgeqT2}
  \begin{split} \prob{(\mathscr{A}_{\sss K,T}^{n})^c}&=\prob{\exists j\in [K]: j \text{ is not explored before }Tb_n}\\
  &\leq \sum_{j=1}^K \prob{j \text{ is not explored before }Tb_n}\\
  &\leq \sum_{j=1}^K\left(1-\frac{d_j}{\ell_n-\Theta(Tb_n)} \right)^{Tb_n}\leq \sum_{j=1}^K \e^{-CT},
  \end{split}
 \end{equation} where $C>0$ is a constant that may depend on $K$. Now, by \eqref{c2eq:CgeqT1},
 \begin{equation}
  \prob{|\mathscr{C}_{\max}^{\sss \geq T}|>\varepsilon b_n}\leq \PR\bigg(\sum_{i\geq 1}\big|\mathscr{C}_{\sss (i)}^{\sss [K]}\big|^2> \varepsilon^2 b_n^2\bigg) + \prob{(\mathscr{A}_{\sss K,T}^{n})^c}.
 \end{equation} The proof follows by taking $\limsup_{n\to\infty}$, $\lim_{T\to\infty}$, $\lim_{K\to\infty}$ respectively and using  \eqref{c2tail-sum-squares-K}, \eqref{c2eq:CgeqT2}. 
\end{proof}

Define the set of excursions of a function $f$ by
\begin{equation}
\mathcal{E}:= \{(l,r): (l,r) \text{ is an excursion of }f\}.
\end{equation}  
We also denote  the set of excursion end-points by $\mathcal{Y}$, i.e.,
\begin{equation}
\mathcal{Y}:= \{r>0: (l,r)\in \mathcal{E}\}.
\end{equation}
\begin{defn}\label{c2defn::good_function}\normalfont A  function $f\in \mathbb{D}_+[0,T]$ is said to be \emph{good} if the following holds:
\begin{enumerate}[(a)]
\item  $\mathcal{Y}$ does not have an isolated point and the complement of $\cup_{(l,r)\in \mathcal{E}}(l,r)$ has Lebesgue measure zero;
\item $f$ does not attain a local minimum at any point of $\mathcal{Y}$.
\end{enumerate} 
\end{defn}

\begin{remark}\label{c2rem:cont-at-r} \normalfont We claim that if a function $f\in \mathbb{D}_+[0,T]$ is good, then $f$ is continuous on $\mathcal{Y}$. To see this, fix  any $\delta>0$ and denote the set of excursions of length at least $\delta$ by $\mathcal{E}_\delta$. Let $r$ be the excursion endpoint of an excursion in $\mathcal{E}_\delta$ and suppose that $f(r)>f(r-)$. Thus, there is no excursion endpoint in $(r-\delta,r)$. Moreover, since $f$ is right-continuous, there exists $\delta '>0 $ such that $f(x)>f(r-)+\varepsilon$ for all $x\in (r,r+\delta')$, where $\varepsilon = (f(r)-f(r-))/2>0$. Thus there is no excursion endpoint on $(r-\delta,r+\delta')$ and thus $r$ is an isolated point contradicting Definition~\ref{c2defn::good_function}. We conclude that $f$ is continuous at excursion endpoints of the excursions in $\mathcal{E}_{\delta}$, and since $\delta>0$ is arbitrary the claim is established.
\end{remark}
Let $\mathcal{L}_i(f)$ be the length of the $i^{th}$ largest excursion of $f$ and define $\Phi_m:\mathbb{D}_+[0,T]\to \mathbb{R}^m$ by 
 \begin{equation}
 \Phi_m(f)= (\mathcal{L}_1(f), \mathcal{L}_2(f), \dots, \mathcal{L}_m(f)).
 \end{equation}Note that $\Phi_m(\cdot)$ is well-defined for any good function  defined in  Definition~\ref{c2defn:good_function-infty}.  
 \begin{lemma}\label{c2lem::good:function:continuity} Suppose that $f\in \mathbb{D}_+[0,T]$ is good. Then, $\Phi_m$ is continuous at $f$ with respect to the subspace topology on $\mathbb{D}_+[0,T]$ induced by the Skorohod $J_1$ topology.
 \end{lemma}
\begin{proof}
 We extend the arguments of \cite[Proposition 22]{NP10b}. The proof  here is for $m=1$ and  similar arguments hold for $m>1$.  Let $\mathfrak{L}$ denote the set of continuous functions $\Lambda:\mathbb{R}_+\to\mathbb{R}_+$  that are strictly increasing and $\Lambda(0)=0, \Lambda(T)=T$.   Suppose $E_1=(l,r)$ is the longest excursion of $f$ on $[0,T]$, thus $\Phi_1(f)=r-l$. For any $\varepsilon > 0$ (small), choose $\delta >0$ such that 
 \begin{equation}\label{c2f-large-interval}
 f(x)> \min\{f(r-),f(r)\}+\delta\quad \forall x\in (l+\varepsilon,r-\varepsilon).
 \end{equation}
 Let $||\cdot||$ denote the sup-norm on $[0,T]$.  Take any sequence of functions $f_n\in \mathbb{D}_+[0,T]$ such that $f_n\to f$, i.e., there exists $\{\Lambda_n\}_{n\geq 1} \subset \mathfrak{L}$ such that for all large enough $n$, 
 \begin{equation}\label{c2f_n-f-close}
 ||f_n\circ \Lambda_n-f||< \frac{\delta}{6}\  \text{ and }\  ||\Lambda_n-I||< \varepsilon,
 \end{equation}where $I$ is the identity function. 
 Now, by Remark~\ref{c2rem:cont-at-r}, $f$ is continuous at $r$. This implies that $f(r-)=f(r)$, and using \eqref{c2f-large-interval} and \eqref{c2f_n-f-close}, for all large enough $n$,
 \begin{equation} \label{c2liminf_f_excursion}
  f_n(y)> f_n\circ \Lambda_n(r)+\frac{2\delta}{3}\quad \forall y \in (l+2\varepsilon, r-2\varepsilon).
 \end{equation}Further, using the continuity of $f$ at $r$, $f_n(r)\to f(r)$ and thus, for all sufficiently large $n$,
 \begin{equation}
  |f_n\circ\Lambda_n(r)-f_n(r)|\leq |f_n\circ\Lambda_n(r)-f(r)|+|f_n(r)-f(r)|< \frac{\delta}{3}.
 \end{equation} Hence, \eqref{c2liminf_f_excursion} implies that, for all sufficiently large $n$,
 \begin{equation}
  f_n(y)> f_n(r)+\frac{\delta}{3}\quad \forall y\in (l+2\varepsilon, r-2\varepsilon).
 \end{equation} Thus, for any $\varepsilon>0$, we have
\begin{equation}\label{c2eq:liminf-exc}
\liminf_{n\to\infty}\Phi_1(f_n)\geq r-l-4\varepsilon= \Phi_1(f)-4\varepsilon.
\end{equation}
Now we turn to a suitable upper bound on $\limsup_{n\to\infty}\Phi_1(f_n)$. First, we claim that one can find $r_1,\dots,r_k\in \mathcal{Y}$ such that $r_1\leq \Phi_1(f)+\varepsilon, T-r_k< \Phi_1(f)+\varepsilon, $ and $r_i-r_{i-1}\leq \Phi_1(f)+\varepsilon, \forall i=2, \dots,k$. The claim is a consequence of Definition~\ref{c2defn::good_function}~(a). 
 Now, Definition \ref{c2defn::good_function}~(b) implies that for any small $\varepsilon >0$, there exists $\delta >0$ and $x_i\in (r_i,r_i+\varepsilon)$ such that $f(r_i)-f(x_i)> \delta$ $\forall i$. Again, since $r_i$ is a continuity point of $f$, $f_n(r_i)\to f(r_i)$. Thus, using \eqref{c2f_n-f-close}, for all large enough $n$,
 \begin{equation}
  f_n(r_i)-f_n(\Lambda_n(x_i))> \frac{\delta}{2}.
 \end{equation}
 Now, $\Lambda_n(x_i)\in (r_i,r_i+\varepsilon)$ for all sufficiently large $n$, since $x_i\in (r_i,r_i+\varepsilon)$. Thus, for all large enough $n$, there exists a point $z_i^n\in (r_i,r_i+\varepsilon)$ such that
 \begin{equation}
  f_n(r_i)-f_n(z_i^n)> \frac{\delta}{2}.
 \end{equation}Also the function $f_n$ only has positive jumps and $\ubar{f}_n(r_i)\to \ubar{f}(r_i)$, as $\ubar{f}_n$ is continuous, where we recall that $\ubar{f}(x)=\inf_{y\leq x}f(y)$. Therefore, $f_n$ must have an excursion end point on $(r_i,r_i+\varepsilon)$ for all large enough $n$. Also, using the fact that the complement of $\cup_{(l,r)\in \mathcal{E}}(l,r)$ has Lebesgue measure zero, $f$ has an excursion endpoint  $r_i^0\in(l_i-\varepsilon,l_i)$. The previous argument shows that $f_n$ has to have an excursion endpoint in $(r_i^0, r_i^0+\varepsilon)$ and thus in $(l_i-\varepsilon,l_i+\varepsilon)$, for all large $n$. Therefore, for any $\varepsilon >0$,
 \begin{equation}\label{c2eq:limsup-exc}
 \limsup_{n\to\infty} \Phi_1(f_n)\leq \Phi_1(f)+3\varepsilon.
\end{equation}  Hence the proof follows from  \eqref{c2eq:liminf-exc} and \eqref{c2eq:limsup-exc}.
\end{proof}

\begin{remark}\label{c2rem:excursion-area-cont} \normalfont For $f\in\mathbb{D}_{+}[0,T]$, let $\mathcal{A}_i(f)$ denote the area under the excursion $\mathcal{L}_i(f)$. Let $(f_n)_{n\geq 1}$ be a sequence of functions on $f\in\mathbb{D}_{+}[0,\infty)$ such that $f_n\to f$, with respect to the Skorohod $J_1$ topology, where $f$ is good. 
Then, \eqref{c2f_n-f-close}, \eqref{c2eq:liminf-exc} and \eqref{c2eq:limsup-exc} also implies that $(\mathcal{A}_1(f_n),\dots,\mathcal{A}_m(f_n))$ converges to $(\mathcal{A}_1(f),\dots,\mathcal{A}_m(f))$, for any $m\geq 1$.
\end{remark}

\begin{defn}\label{c2defn:good_function-infty}\normalfont A stochastic process $\mathbf{X}\in \mathbb{D}_+[0,\infty)$ is said to be good if 
\begin{enumerate}[(a)]
\item The sample paths are good almost surely when restricted to $[0,T]$, for every fixed $T>0$;
\item $\mathbf{X}$ does not have an infinite excursion almost surely;
\item For any $\varepsilon >0$, $\mathbf{X}$ has only finitely many excursions of length more than $\varepsilon$ almost surely.
\end{enumerate} 
\end{defn}
\begin{lemma}\label{c2lem:levy-as-good}The thinned L\'evy process $\mathbf{S}_\infty^\lambda$ defined in \eqref{c2defn::limiting::process} is good.
\end{lemma}
\begin{proof}
 Let us make use of the properties of the process $\mathbf{S}_\infty^\lambda$ that were established in \citep{AL98}. $\mathbf{S}_\infty^\lambda$ satisfies Definition~\ref{c2defn:good_function-infty}~(b),(c) by \cite[(8)]{AL98}.   The fact that the excursion endpoints of $\mathbf{S}_\infty^\lambda$ do not have any isolated points almost surely follows directly from \citep[Proposition 14 (d)]{AL98}. Further,  \citep[Proposition 14 (b)]{AL98} implies that, for any $u>0$, $\prob{S_\infty^\lambda(u)=\inf_{u'\leq u}S_\infty^\lambda(u')}=0$. Taking the integral with respect to the Lebesgue measure and interchanging the limit by using Fubini's theorem, we conclude that almost surely 
 \begin{equation}
  \int_0^T \ind{S_\infty^\lambda(u)=\inf_{u'\leq u}S_\infty^\lambda(u')}\mathrm{d}u=0,
 \end{equation}which verifies Definition~\ref{c2defn::good_function}~(a). Now, let $\mathbf{L}$ be the L\'evy process defined as
 \begin{equation}\label{c2defn:perfect-levy}
  L(t)=\sum_{i=1}^{\infty} \theta_i\left(\mathcal{N}_i(t)- (\theta_i/\mu)t\right)+\lambda t,
 \end{equation}where $(\mathcal{N}_i(t))_{t\geq 0}$ is a Poisson process with rate $\theta_i$ which are independent for different $i$. Via the natural coupling that states $\mathcal{I}_i(t)\leq \mathcal{N}_i(t)$, we can assume that $S_\infty^\lambda(t)\leq L(t)$ for all $t>0$. Using \cite[Theorem VII.1]{Ber96},
 \begin{equation}\label{c2defn:perfect-levy-1}
  \inf \{t>0: L(t)<0\}=0, \quad \text{almost surely.}
 \end{equation}
 Moreover, for any stopping time $\mathcal{T}>0$, $(S_\infty^\lambda(\mathcal{T}+t)-S_\infty^\lambda(\mathcal{T}))_{t\geq 0}$, conditioned on the sigma-field $\sigma (S_\infty^\lambda(s):s\leq \mathcal{T})$, is distributed as a process  defined in \eqref{c2defn::limiting::process} for some random $\boldsymbol{\theta}$ and $\Lambda$.  Now we can take $\mathcal{T}$ to be an excursion endpoint and  an application of \eqref{c2defn:perfect-levy-1} verifies Definition~\ref{c2defn::good_function}~(b).
\end{proof}

As described in Section~\ref{c2sec:conv-expl}, the excursion lengths of the exploration process $\bar{\mathbf{S}}_n$ gives the total number of edges in the explored components. \linebreak 
Lemma~\ref{c2lem:surp:poisson-conv} below estimates the number of surplus edges in the components explored upto time $\Theta(b_n)$. This enables us to compute the scaling limits for the component sizes using the results from the previous section and complete the proof of Theorem~\ref{c2thm::conv:component:size}.
\begin{lemma} \label{c2lem:surp:poisson-conv} Let $N_n^\lambda(k)$ be the number of surplus edges discovered up to time $k$ and $\bar{N}^\lambda_n(u) = N_n^\lambda(\lfloor ub_n \rfloor)$. Then, as $n\to\infty$,
 \begin{equation}
 (\bar{\mathbf{S}}_n,\bar{\mathbf{N}}_n^\lambda)\dto (\mathbf{S}_{\infty}^\lambda,\mathbf{N}^\lambda),
 \end{equation} where $\mathbf{N}^\lambda$ is defined in \eqref{c2defn::counting-process}.
 \end{lemma}
 \begin{proof}
  We write $
N_n^{\lambda}(l)=\sum_{i=2}^l \xi_i$,
where $\xi_i=\ind{\mathscr{V}_i=\mathscr{V}_{i-1}}$. Let $A_i$ denote the number of active half-edges after stage $i$ while implementing Algorithm~\ref{c2algo-expl}. Note that 
\begin{equation}
 \prob{\xi_i=1\vert \mathscr{F}_{i-1}}=\frac{A_{i-1}-1}{\ell_n-2i-1}= \frac{A_{i-1}}{\ell_n}(1+O(i/n))+O(n^{-1}),
\end{equation} uniformly for $i\leq Tb_n$ for any $T>0$. 
Therefore, the instantaneous rate of change of the re-scaled process $\bar{\mathbf{N}}^{\lambda}$ at time $t$, conditional on the past, is 
\begin{equation}\label{c2eqn:intensity}
 b_n\frac{A_{\floor{tb_n}}}{n\mu}\left( 1+o(1)\right) +o(1)= \frac{1}{\mu}\refl{\bar{S}_n(t)}\left( 1+o(1)\right) +o(1).
\end{equation}
Recall from  Theorem~\ref{c2thm:conv:refl:process} that $\mathrm{refl}(\bar{\mathbf{S}}_n)\dto \mathrm{refl}(\bar{\mathbf{S}}_\infty)$.  Then, by the Skorohod representation theorem, we can assume that $\mathrm{refl}(\bar{\mathbf{S}}_n)\to \mathrm{refl}(\bar{\mathbf{S}}_\infty)$ almost surely on some probability space. Observe that $(\int_{0}^t \refl{\bar{S}_{\infty}(u)}du)_{t\geq 0}$ has continuous sample paths. Therefore, the conditions of \cite[Corollary 1, Page 388]{LS89} are satisfied and the proof is complete.
 \end{proof}
\begin{theorem} \label{c2thm::component-sizes-finite-dim}
 For any $m\geq 1$, as $n\to \infty$
 \begin{equation}
  b_n^{-1}\big( |\mathscr{C}_{\sss (1)}|, |\mathscr{C}_{\sss (2)}|, \dots, |\mathscr{C}_{\sss(m)}|\big) \dto (\gamma_1(\lambda), \gamma_2(\lambda), \dots, \gamma_m(\lambda))
 \end{equation} with respect to the product topology, where $\gamma_i(\lambda)$ is the $i^{th}$ largest excursion of $\bar{\mathbf{S}}_{\infty}$ defined in \eqref{c2defn::limiting::process}.
\end{theorem}
\begin{proof}
Fix any $m\geq 1$. Let $\mathscr{C}_{\sss (i)}^{\sss T}$ be the $i^{th}$ largest component explored by Algorithm~\ref{c2algo-expl} up to time $Tb_n$. Denote by $\mathscr{D}_{\sss (i)}^{\sss \mathrm{ord},T}$ the $i^{th}$ largest value of $(\sum_{k\in \mathscr{C}_{\sss(i)}^{\sss T}}d_k)_{i\geq 1}$. Let $g:\R^m\mapsto \R$ be a bounded continuous function. By Lemma~\ref{c2lem:levy-as-good} the sample paths of $\bar{\mathbf{S}}_{\infty}$ are almost surely good. Thus, using Theorem~\ref{c2thm::convegence::exploration_process}, Lemma~\ref{c2lem::good:function:continuity} gives
\begin{eq}
 \lim_{n\to\infty} &\expt{g\Big((2b_n)^{-1}\big( \mathscr{D}_{\sss(1)}^{\sss\mathrm{ord},T}, \mathscr{D}_{\sss(2)}^{\sss\mathrm{ord},T}, \dots, \mathscr{D}_{\sss(m)}^{\sss\mathrm{ord},T}\big)\Big)} \\
 &= \expt{g\big(\gamma_1^{\sss T}(\lambda), \gamma_2^{\sss T}(\lambda), \dots, \gamma_m^{\sss T}(\lambda)\big)},
\end{eq}where $\gamma_i^{\sss T}(\lambda)$ is the $i^{th}$ largest excursion of $\bar{\mathbf{S}}_{\infty}$ restricted to $[0,T]$. Now the support of the joint distribution of $(\gamma_i^{\sss T}(\lambda))_{i\geq 1}$ is concentrated on$\{(x_1,x_2,\dots): x_1>x_2>\dots\}$. Thus, using Lemma~\ref{c2lem:surp:poisson-conv}, it follows that 
\begin{equation}\label{c2finite-dim-D}
\lim_{n\to\infty} \expt{g\Big(b_n^{-1}\big( |\mathscr{C}_{\sss (1)}^{\sss T}|, |\mathscr{C}_{\sss (2)}^{\sss T}|, \dots, |\mathscr{C}_{\sss (m)}^{\sss T}|\big)\Big)} = \expt{g\big(\gamma_1^{\sss T}(\lambda), \gamma_2^{\sss T}(\lambda), \dots, \gamma_m^{\sss T}(\lambda)\big)}.
\end{equation}Since $\mathbf{S}_{\infty}^{\lambda}$ satisfies Definition~\ref{c2defn:good_function-infty}~(b), (c), it follows that
\begin{equation}\label{c2no-inf-exc}
 \lim_{T\to\infty}\expt{g\big(\gamma_1^{\sss T}(\lambda), \gamma_2^{\sss T}(\lambda), \dots, \gamma_m^{\sss T}(\lambda)\big)}=\expt{g\big(\gamma_1(\lambda), \gamma_2(\lambda), \dots, \gamma_m(\lambda)\big)}
\end{equation}
 Finally, using Lemma~\ref{c2lem:no-large-comp-later}, the proof of Theorem~\ref{c2thm::component-sizes-finite-dim} is completed by \eqref{c2finite-dim-D} and \eqref{c2no-inf-exc}.

\end{proof}

\begin{proof}[Proof of Theorem~\ref{c2thm::conv:component:size}]
The proof of Theorem~\ref{c2thm::conv:component:size} now follows directly from Theorem~\ref{c2thm::component-sizes-finite-dim} and Proposition~\ref{c2prop-l2-tightness}.
\end{proof}

\section{Proof of Theorem~\ref{c2thm:spls}}\label{c2sec:surplus}
 The goal of this section is to prove the joint convergence of the component sizes and the surplus edges as described in Theorem~\ref{c2thm:spls}. We start with a preparatory lemma:
\begin{lemma}\label{c2lem:spls:prod} The convergence in \eqref{c2thm:eqn:spls} holds with respect to the $\ell^2_{\shortarrow}\times \mathbb{N}^{\infty}$ topology.
\end{lemma}
\begin{proof}
 Note that Lemma~\ref{c2lem:no-large-comp-later} already states that we do not see large components being explored after the time  $Tb_n$ for large $T>0$. Thus the proof is a consequence of Lemmas~\ref{c2lem::good:function:continuity},~\ref{c2lem:surp:poisson-conv}, Remark~\ref{c2rem:excursion-area-cont} and Theorem~\ref{c2thm::conv:component:size}.
\end{proof}
Recall the definition of the metric $\mathrm{d}_{\sss \mathbb{U}}$ from Chapter~\ref{c1:sub_sec_notation}. 
Using Lemma~\ref{c2lem:spls:prod}, it now remains to obtain a uniform summability condition on the tail of the sum of products of the scaled component sizes and surplus edges. This is formally stated in Proposition~\ref{c2prop-surp-u-0} below. The proof is completed in the similar spirit as the finite third-moment case \cite{DHLS15}.
 \begin{proposition}\label{c2prop-surp-u-0} For any $\varepsilon >0$,
 \begin{equation}
 \lim_{\delta\to 0}\limsup_{n\to\infty} \PR\bigg(\sum_{i: |\mathscr{C}_{(i)}|\leq \delta b_n }|\mathscr{C}_{\sss(i)}|\times \surp{\mathscr{C}_{\sss(i)}}> \varepsilon b_n\bigg)=0.
 \end{equation}
 \end{proposition} The following estimate will be the crucial ingredient to complete the proof of Proposition~\ref{c2prop-surp-u-0}. The proof of Lemma~\ref{c2lem:sp-cv-n} is postponed to Appendix~\ref{c2sec_appendix} since this uses similar ideas as \cite{DHLS15}.
\begin{lemma} \label{c2lem:sp-cv-n}
Assume that $\limsup_{n\to\infty} c_n(\nu_n-1)<0$. Let $V_n$ denote a vertex chosen uniformly at random, independently of the graph $\mathrm{CM}_n(\boldsymbol{d})$ and let $\mathscr{C}(V_n)$ denote the component containing $V_n$.  Let $\delta_k=\delta k^{-0.12}$. Then, for $\delta > 0$ sufficiently small,
\begin{equation}
 \prob{\surp{\mathscr{C}(V_n)}\geq K,|\mathscr{C}(V_n)|\in (\delta_K b_n,2\delta_Kb_n)}\leq \frac{C\sqrt{\delta}}{a_nK^{1.1}}
\end{equation}
 where $C$ is a fixed constant independent of $n,\delta, K$. 
 \end{lemma}

\begin{proof}[Proof of Proposition~\ref{c2prop-surp-u-0} using Lemma~\ref{c2lem:sp-cv-n}]First consider the case $\lambda<0$. Fix any $\varepsilon, \eta >0$. Note that
\begin{equation}\label{c2eq:u-0-calc}
 \begin{split}
  &\PR\bigg( \sum_{|\mathscr{C}_{\sss (i)}|\leq \varepsilon b_n} |\mathscr{C}_{\sss (i)}|\surp{\mathscr{C}_{\sss (i)}}> \eta b_n \bigg)\leq \frac{1}{\eta b_n}\E \bigg[\sum_{i=1}^{\infty}|\mathscr{C}_{\sss (i)}| \surp{\mathscr{C}_{\sss (i)}} \1_{\{ |\mathscr{C}_{\sss (i)}|\leq \varepsilon b_n\}} \bigg]
  \\&= \frac{a_n}{\eta}\expt{\mathrm{SP}(\mathscr{C}(V_n))\1_{\{ |\mathscr{C}(V_n)|\leq \varepsilon b_n\}}}\\
  &= \frac{a_n}{\eta}\sum_{k=1}^{\infty}\sum_{i\geq \log_2(1/(k^{0.12}\varepsilon))}\PR\bigg(\mathrm{SP}(\mathscr{C}(V_n))\geq k, |\mathscr{C}(V_n)|\in \Big(\frac{b_n}{2^{i+1}k^{0.12}}, \frac{b_n}{2^{i}k^{0.12}}\Big] \bigg)\\
  &\leq \frac{C}{\eta} \sum_{k=1}^{\infty}\frac{1}{k^{1.1}}\sum_{i\geq \log_2(1/(k^{0.12}\varepsilon))} 2^{-i/2} \leq \frac{C}{\eta}\sum_{k=1}^{\infty}\frac{\sqrt{\varepsilon}}{k^{1.04}}  =O(\sqrt{\varepsilon}),
 \end{split}
\end{equation} where the last-but-two step follows from Lemma~\ref{c2lem:sp-cv-n}.
 The proof of Proposition~\ref{c2prop-surp-u-0} now follows for $\lambda <0$. 
\par Now consider the case $\lambda > 0$. Fix a large integer $R\geq 1$ such that $\lambda - \sum_{i=1}^R\theta_i^2<0$. This can be done because $\boldsymbol{\theta}\notin \ell^{2}_{\shortarrow}$. Using \eqref{c2eq:CgeqT1}, for any $\eta>0 $, it is possible to choose $T>0$ such that for all sufficiently large $n$,
\begin{equation}\label{c2eq:early-expl}
 \prob{\text{all the vertices }1,\dots,R \text{ are explored within time }Tb_n }> 1-\eta.
\end{equation} Let $T_e$ denote the first time after $Tb_n$ when we finish exploring a component. By Theorem~\ref{c2thm::convegence::exploration_process}, $(b_n^{-1}T_e)_{n\geq 1}$ is a tight sequence. Let $\mathcal{G}^*_T$ denote the graph obtained by removing the components explored up to time $T_e$. Then, $\mathcal{G}^*_T$ is again a configuration model conditioned on its degrees. Let $\nu_n^*$ denote the value of the criticality parameter for $\mathcal{G}^*$. Note that 
\begin{equation}
 \sum_{i\notin \mathscr{V}_{\sss T_e}}d_i\geq \ell_n-2Tb_n \implies  \sum_{i\notin \mathscr{V}_{\sss T_e}}d_i = \ell_n+\oP(n),
\end{equation}and thus conditionally on $\mathscr{F}_{\sss T_e}$ and the fact that $(1,\dots, R)$ are explored within time $Tb_n$,
\begin{equation}\label{c2ub-nu*}
 \begin{split}
 \nu_n^*\leq \frac{\sum_{i\in [n]}d_i^2-\sum_{i=1}^Rd_i^2}{\sum_{i\notin \mathscr{V}_{\sss T_e}}d_i}-1=1+c_n^{-1}\big(\lambda-\sum_{i=1}^R\theta_i^2\big) +o(c_n^{-1}).
 \end{split}
\end{equation}
Therefore, combining \eqref{c2eq:early-expl}, \eqref{c2ub-nu*}, we can use Lemma~\ref{c2lem:sp-cv-n}  on $\mathcal{G}^*_T$ since $c_n(\nu_n^*-1)<0$. Thus, if $\mathscr{C}_{\sss(i)}^*$ denotes the $i^{th}$ largest component of $\mathcal{G}_T^*$, then
\begin{equation}\label{c2surplus-positivelamba1} \lim_{T\to\infty}\lim_{\delta\to 0}\limsup_{n\to\infty}\PR\bigg(\sum_{i: |\mathscr{C}_{(i)}^*|\leq \delta b_n }|\mathscr{C}_{\sss(i)}^*|\times \surp{\mathscr{C}_{\sss(i)}^*}> \varepsilon b_n\bigg)=0.
\end{equation} To conclude the proof for the whole graph $\mathrm{CM}_n(\boldsymbol{d})$ (with $\lambda >0$), let $$\mathcal{K}_n^T:=\{i:|\mathscr{C}_{\sss(i)}|\leq \delta b_n, |\mathscr{C}_{\sss(i)}| \text{ is explored before the time }T_e  \}.$$ Note that
\begin{equation}
 \begin{split}
  \sum_{i \in \mathcal{K}_n^T}|\mathscr{C}_{\sss (i)}|\cdot \mathrm{SP}(\mathscr{C}_{\sss (i)})&\leq \Big( \sum_{i\in \mathcal{K}_n}|\mathscr{C}_{\sss (i)}|^2\Big)^{1/2}\times \Big(\sum_{i\in \mathcal{K}_n}\mathrm{SP}(\mathscr{C}_{\sss (i)})^2 \Big)^{1/2}\\
  &\leq  \bigg( \sum_{|\mathscr{C}_{\sss (i)}|\leq \delta b_n}|\mathscr{C}_{\sss (i)}|^2\bigg)^{1/2}\times \mathrm{SP}(T_e),
 \end{split}
\end{equation}where $\mathrm{SP}(t)$ is the number of surplus edges explored up to time $tb_n$ and we have used the fact that $\sum_{i\in\mathcal{K}_n}\mathrm{SP}(\mathscr{C}_{\sss (i)})^2\leq (\sum_{i\in\mathcal{K}_n}\mathrm{SP}(\mathscr{C}_{\sss (i)}))^2\leq \mathrm{SP}(T_e)^2$. From Lemma~\ref{c2lem:surp:poisson-conv} and Proposition~\ref{c2prop-l2-tightness} we can conclude that for any $T>0$,
\begin{equation}\label{c2surplus-positivelamba2}
 \lim_{\delta\to 0}\limsup_{n\to\infty}\PR\bigg(\sum_{i \in \mathcal{K}_n^T}|\mathscr{C}_{\sss (i)}|\cdot \mathrm{SP}(\mathscr{C}_{\sss (i)})>\varepsilon b_n\bigg)=0.
\end{equation}
  The proof is now complete for the case $\lambda > 0$ by combining \eqref{c2surplus-positivelamba1} and \eqref{c2surplus-positivelamba2}.
\end{proof}
\section{Proof for simple graphs} \label{c2sec:simple-graphs}
In this section, we give a proof of Theorem~\ref{c2thm:simple-graph}. Let $\PR_s(\cdot)$ (respectively $\E_s[\cdot]$) denote the probability measure (respectively the expectation) conditionally on the graph $\mathrm{CM}_n(\boldsymbol{d})$ being simple. For any process $\mathbf{X}$ on $\mathbb{D}([0,\infty),\R)$, we define $\mathbf{X}^T:=(X(t))_{t\leq T}$. Thus the truncated process $\mathbf{X}^T$ is $\mathbb{D}([0,T],\R)$-valued. Now, by \cite[Theorem 1.1]{J09c}, $\liminf_{n\to\infty}\PR(\mathrm{CM}_n(\boldsymbol{d})\text{ is simple})>0$. This fact ensures that, under the conditional measure $\PR_s$,  $(b_n^{-1}|\mathscr{C}_{\sss (i)}|)_{i\geq 1}$ is tight with respect to the $\ell^2_{\shortarrow}$ topology.
 Therefore, to conclude Theorem~\ref{c2thm:simple-graph}, it suffices to show that the exploration process $\bar{\mathbf{S}}_n$, defined in \eqref{c2eqn::scaled_process},  has the  same limit (in distribution) under $\PR_s$ as obtained in Theorem~\ref{c2thm::convegence::exploration_process} so that the finite-dimensional limit of $(b_n^{-1}|\mathscr{C}_{\sss (i)}|)_{i\geq 1}$ remains unchanged under $\PR_s$. Thus, it is enough to show that for any bounded continuous function $f:\mathbb{D}([0,T],\R)\mapsto\R$,
\begin{equation}\label{c2diff:expt:bounded:cont}
 \big| \E[f(\bar{\mathbf{S}}_n^T)]-\E_s[f(\bar{\mathbf{S}}_n^T)]\big|\to 0.
\end{equation}
Let $\ell_n':=\ell_n-2Tb_n$. We first estimate the number of multiple edges or self-loops discovered in the graph up to time $Tb_n$. 
Let $v_l$ denote the \emph{exploring} vertex in the breadth-first exploration given by Algorithm~\ref{c2algo-expl}, $d_{v_l}$  the degree of $v_l$ and  $(e_1,\dots,e_r)$ the ordered set of active half-edges of $v_l$ when $v_l$ is declared to be exploring. Note that, for $l\leq Tb_n$, $e_i$ creates a self-loop with probability at most $(d_{v_l}-i)/\ell_n'$ and creates a multiple edge with probability at most $(i-1)/\ell_n'$. Therefore,
\begin{equation}
 \expt{\#\{\text{self-loops/multiple edges discovered while exploring}v_l\}|\mathscr{F}_{l-1}}\leq \frac{2d_{v_l}^2}{\ell_n'}.
\end{equation}Thus, for any $T>0$,
\begin{equation}
\begin{split}
 &\expt{\#\{\text{self-loops or multiple edges discovered up to time }Tb_n}\\
 &\leq \frac{2}{\ell_n'}\E\bigg[\sum_{i\in [n]}d_i^2\mathcal{I}^n_i(Tb_n)\bigg]\\
 &=\frac{2}{\ell_n'}\E\bigg[\sum_{i=1}^K d_i^2\mathcal{I}^n_i(Tb_n)\bigg]+\frac{2}{\ell_n'}\E\bigg[\sum_{i=K+1}^nd_i^2\mathcal{I}^n_i(Tb_n)\bigg],
 \end{split}
\end{equation}where $\mathcal{I}^n_i(l)=\ind{i\in \mathscr{V}_l}$. Now, using Assumption~\ref{c2assumption1}~\eqref{c2assumption1-1}, for every fixed $K\geq 1$,
\begin{equation}
 \frac{2}{\ell_n'}\E\bigg[\sum_{i=1}^K d_i^2\mathcal{I}^n_i(Tb_n)\bigg]\leq \frac{2}{\ell_n'} \sum_{i=1}^K d_i^2 \to 0,
\end{equation} since $2\alpha-1<0$. Moreover, recall from  \eqref{c2eq:prob-ind} that $\prob{\mathcal{I}^n_i(Tb_n)=1}\leq Tb_nd_i/\ell_n'$. Therefore, for some constant $C>0$,
\begin{equation}
 \begin{split} 
  \frac{2}{\ell_n'}\E\bigg[\sum_{i=K+1}^nd_i^2\mathcal{I}_i^n(Tb_n)\bigg]\leq \frac{Tb_n}{\ell_n'^2}\sum_{i=K+1}^nd_i^3\leq C \bigg( a_n^{-3}\sum_{i=K+1}^nd_i^3 \bigg), 
 \end{split}
\end{equation}which, by Assumption~\ref{c2assumption1}~\eqref{c2assumption1-2}, tends to zero if we first take $\limsup_{n\to\infty}$ and then take $\lim_{K\to\infty}$. Consequently, for any fixed $T>0$, as $n\to\infty$,
\begin{equation}
 \prob{\text{at least one self-loop or multiple edge is discovered before time }Tb_n}\to 0.
\end{equation}
Now,
 \begin{equation}\label{c2simple-after-Tbn-enough}
 \begin{split}
  &\expt{f(\bar{\mathbf{S}}_n^T) \ind{\CM \text{ is simple}}}\\
  & =\expt{f(\bar{\mathbf{S}}_n^T) \ind{\text{no self-loops or multiple  edges found after }Tb_n}}+o(1)\\
  &= \expt{f(\bar{\mathbf{S}}_n^T) \prob{\text{no self-loops or multiple  edges found after  }T b_n\vert \mathscr{F}_{\sss Tb_n}}} +o(1).
  \end{split}
 \end{equation}Define, $T_e=\inf\{l\geq Tb_n: \text{a component is finished exploring at time } l\}$. Using the fact that $(b_n^{-1}T_e)_{n\geq 1}$ is a tight sequence, the limit of the expected number of loops or multiple edges discovered between time $Tb_n$ and $T_e$ is again zero. As in the proof of Proposition~\ref{c2prop-surp-u-0}, consider the graph $\mathcal{G}^*$, obtained by removing the components obtained up to time $T_e$. Thus, $\mathcal{G}^*$ is a configuration model, conditioned on its degree sequence. Let $\nu^*_n$ be the criticality parameter. Then, we claim that $\nu^*_n\xrightarrow{\sss \PR} 1$. To see this note that $\sum_{i\notin \mathscr{V}_{\sss T_e}}d_i=\ell_n+\oP(n)$. 
Further, note that by Assumption~\ref{c2assumption1}~\eqref{c2assumption1-2} \eqref{c2eqn:increment:indicator}, for any $t>0$,
\begin{equation}\label{c2finite-expt-di2}
\begin{split}
\limsup_{n\to\infty}  \E\bigg[a_n^{-2}\sum_{i\in [n]}d_i^2\mathcal{I}_i(tb_n)\bigg] \leq\limsup_{n\to\infty} a_n^{-2}tb_n\frac{\sum_{i\in [n]}d_i^3}{\ell_n-2tb_n}<\infty,
\end{split}
\end{equation} 
which implies that  
 $\sum_{i\notin \mathscr{V}_{\sss T_e}}d_i^2=\sum_{i\in [n]}d_i^2+\oP(n)$ and thus the claim is proved. Since the degree distribution has finite second moment, using  \cite[Theorem 7.11]{RGCN1} we get \begin{equation}\label{c2prob-of-simple}
\prob{\mathcal{G}^* \text{ is simple}\Big\vert\mathscr{F}_{\sss T_e}}\pto \e^{-3/4}.
\end{equation}Now using  \eqref{c2simple-after-Tbn-enough}, \eqref{c2prob-of-simple} and the dominated convergence theorem, we conclude that
\begin{equation}
 \expt{f(\bar{\mathbf{S}}_n^T) \ind{\CM \text{ is simple}}}=  \expt{f(\bar{\mathbf{S}}_n^T)}\prob{\CM \text{ is simple}}+o(1).
\end{equation}Therefore, \eqref{c2diff:expt:bounded:cont} follows and the proof of Theorem~\ref{c2thm:simple-graph} is complete.
\qed
\section{Scaling limits for component functionals}\label{c2sec:comp-functional}
Suppose that vertex $i$ has an associated weight $w_i$. The total weight of the component $\mathscr{C}_{\sss (i)}$ is denoted by $\mathscr{W}_i = \sum_{k\in \mathscr{C}_{\sss (i)}} w_k$. The goal of this section is to derive the scaling limits for $(\mathscr{W}_i )_{i\geq 1}$ when the weight sequence satisfies some regularity conditions given below:
\begin{assumption}\label{c2assumption-weight} \normalfont The weight sequences $\bld{w} = (w_i)_{i\in [n]}$ satisfies
\begin{enumerate}[(i)]
\item \label{c2assumption-weight-1} $\sum_{i\in [n]}w_i = O(n)$, and  $\lim_{n\to\infty}\frac{1}{\ell_n}\sum_{i\in [n]} d_i w_i = \mu_{w}$.
\item \label{c2assumption-weight-2} $\max\{\sum_{i\in [n]}d_iw_i^2,\sum_{i\in [n]}d_i^2w_i\} = O(a_n^3)$.
\end{enumerate}
\end{assumption}
\begin{theorem}\label{c2thm:comp-functionals}Consider $\CM$ satisfying \textrm{Assumption~\ref{c2assumption1}} and a weight sequence $\bld{w}$ satisfying \textrm{Assumption~\ref{c2assumption-weight}}. Denote $\mathbf{Z}^w_n= \ord( b_n^{-1}\mathscr{W}_i,\mathrm{SP}(\mathscr{C}_{\sss (i)}))_{i\geq 1}$ and $\mathbf{Z}^w:=\ord(\mu_{w}\gamma_i(\lambda), N(\gamma_i))_{i\geq 1}$, where $\gamma_i(\lambda)$, and  $N(\gamma_i)$ are defined in \textrm{Theorem~\ref{c2thm:spls}}. As $n\to\infty$,
\begin{equation}\label{c2eq:weight-conv-amc}
 \mathbf{Z}_n^w \dto \mathbf{Z}^w,
\end{equation}with respect to the $\mathbb{U}^0_{\shortarrow}$ topology.
\end{theorem}
The proof of Theorem~\ref{c2thm:comp-functionals} can be decomposed in two main steps: the first one is to obtain the finite-dimensional limits of $\mathbf{Z}_n^w$ and then prove the $\mathbb{U}^0_{\shortarrow}$ convergence. The finite-dimensional limit is a consequence of the fact that the total weight of the clusters is approximately equal to the cluster sizes. The argument for the tightness with respect to the $\mathbb{U}^0_{\shortarrow}$ topology is similar to Propositions~\ref{c2prop-l2-tightness}~and~\ref{c2prop-surp-u-0} and therefore we only provide a sketch with pointers to all the necessary ingredients. Recall that $\mathcal{I}_i^n(l) = \ind{i\in \mathscr{V}_l}$, where $\mathscr{V}_l$ is the set of discovered vertices upto time $l$ by Algorithm~\ref{c2algo-expl}.
\begin{lemma}\label{c2lem:weight-approx-size}Under \textrm{Assumptions~\ref{c2assumption1},~\ref{c2assumption-weight}}, for any $T>0$, 
\begin{equation}\label{c2weight-expl-prop}
 \sup_{u\leq T}\bigg| \sum_{i\in [n]} w_i\mathcal{I}_i^n(ub_n)-\frac{\sum_{i\in [n]}d_iw_i}{\ell_n}ub_n\bigg|=\OP(a_n).
\end{equation} Consequently, for each fixed $i\geq 1$,
\begin{equation}\label{c2weight-approx-size}
   \mathscr{W}_i = \mu_{w} \big| \mathscr{C}_{\sss (i)} \big| +\oP(b_n).   
  \end{equation}
\end{lemma}
\begin{proof}Fix any $T>0$. Define ,
\begin{equation}
W_n(l)=\sum_{i\in [n]}w_i\mathcal{I}_i^n(l) - \frac{\sum_{i\in [n]}d_iw_i}{\ell_n} l.
\end{equation}
 The goal is to use the supermartingale inequality \eqref{c2eqn:supmg:ineq} in the same spirit as in the proof of \eqref{c2tail::martingale}. Firstly, observe from \eqref{c2eq:prob-ind} that
\begin{equation}
 \begin{split}
 &\E[W_n(l+1)-W_n(l) | \mathscr{F}_l]\\
 &=\E\bigg[\sum_{i\in [n]}w_i \left(\mathcal{I}^n_i(l+1)-\mathcal{I}_i^n(l)\right)\Big| \mathscr{F}_l\bigg]-\frac{\sum_{i\in [n]}d_iw_i}{\ell_n}\\
  &= \sum_{i\in [n]} w_i\E\big[\mathcal{I}^n_i(l+1)\big| \mathscr{F}_l\big]\ind{\mathcal{I}_i^n(l)=0} -\frac{\sum_{i\in [n]}d_iw_i}{\ell_n}\geq 0,
  \end{split}
 \end{equation} 
uniformly over $l\leq Tb_n$ and therefore, $(W_n(l))_{l=1}^{Tb_n}$ is a sub-martingale. Let $\ell_n'=\ell_n-2Tb_n-1$. Using \eqref{c2prob-ind-lb}, we compute
\begin{equation}\label{c2expt-W-n}
\begin{split}
 \big|\E[W_n(l)]\big|&=\sum_{i\in [n]}w_i\left(\prob{\mathcal{I}_i^n(l)=1}-\frac{d_i}{\ell_n}\right)\\
 &\leq \sum_{i\in [n]} w_i\bigg(1-\bigg(1-\frac{d_i}{\ell_n'}\bigg)^l-\frac{d_i}{\ell_n'}l\bigg) + l\sum_{i\in [n]}w_i \bigg(\frac{d_i}{\ell_n'}-\frac{d_i}{\ell_n}\bigg)\\
 &\leq 2(2Tb_n)^2 \frac{\sum_{i\in [n]}d_i^2w_i}{\ell_n'^2}=O(b_n^2a_n^3/n^2) = O(a_n),
\end{split}
\end{equation}uniformly over $l\leq Tb_n$. Also, using \eqref{c2neg:correlation}, \eqref{c2var-ind-ub}, and Assumption~\ref{c2assumption-weight}~\eqref{c2assumption-weight-2},
\begin{equation}\label{c2var-W-n}
\mathrm{Var}(W_n(l))\leq \sum_{i\in [n]} w_i^2 \mathrm{var}(\mathcal{I}_i^n(l))\leq Tb_n\frac{\sum_{i\in [n]}d_iw_i^2}{\ell_n'} =O(a_n^2),
\end{equation}uniformly over $l\leq Tb_n$. Using \eqref{c2eqn:supmg:ineq}, \eqref{c2expt-W-n} and \eqref{c2var-W-n}, we conclude the proof of \eqref{c2weight-expl-prop}.
The proof of \eqref{c2weight-approx-size} follows using Lemma~\ref{c2lem:no-large-comp-later} and simply observing that $a_n=o(b_n)$. 
\end{proof}
\begin{proof}[Proof of Theorem~\ref{c2thm:comp-functionals}]
Lemma~\ref{c2lem:weight-approx-size} ensures the finite-dimensional convergence in \eqref{c2eq:weight-conv-amc}. Thus, the proof is complete if we can show that, for any $\varepsilon > 0$ 
\begin{subequations}
\begin{equation}\label{c2eq:suff-U0-conv-1}
 \lim_{K\to \infty}\limsup_{n\to\infty}\PR\bigg(\sum_{i>K}\mathscr{W}_i^2>\varepsilon b_n^2\bigg)=0,
\end{equation}and
\begin{equation}\label{c2eq:suff-U0-conv-2}
\lim_{\delta \to 0}\limsup_{n\to\infty}\PR\bigg( \sum_{|\mathscr{C}_{\sss (i)}|\leq \delta b_n} \mathscr{W}_i\times \surp{\mathscr{C}_{\sss (i)}}> \varepsilon b_n\bigg) = 0.
\end{equation}
\end{subequations}
The arguments for proving \eqref{c2eq:suff-U0-conv-1}, and \eqref{c2eq:suff-U0-conv-2} are similar to those for ropositions~\ref{c2prop-l2-tightness}, and~\ref{c2prop-surp-u-0} and thus we only sketch a brief outline. 
Denote $\ell_n^w = \sum_{i\in [n]}w_i$. 
The main ingredient to the proof of Proposition~\ref{c2prop-l2-tightness} is Lemma~\ref{c2lem::tail_sum_squares}, and the proof of Lemma~\ref{c2lem::tail_sum_squares} uses the fact that the expected sum of squares of the cluster sizes can be written in terms of susceptibility functions in \eqref{c2comp-vs-rnd-choice} and then we made use of the estimate for the susceptibility function in \eqref{c2bound::expt-cluster-size}. Let $V_n'$ denote a vertex chosen according to the distribution $(w_i/\ell_n^w)_{i\in [n]}$, independently of the graph. Notice that for any $\mathrm{CM}_n(\bld{d})$,
\begin{equation}\label{c2W-vs-WVn}
 \E\bigg[\sum_{i\geq 1 }\mathscr{W}_i^2\bigg]=   \ell_n^w\E\big[ \mathscr{W}(V_n')\big]. 
\end{equation}
Now, \cite[Lemma 5.2]{J09b} can be extended using an identical argument  to compute the weight-based susceptibility function in the right hand side of \eqref{c2W-vs-WVn}. See Lemma~\ref{c2lem:gen-path-count} given in Appendix~\ref{c2sec:appendix-gen-path-counting}. The proof of \eqref{c2eq:suff-U0-conv-2} can also be completed using an identical argument as Proposition~\ref{c2prop-surp-u-0} by observing that 
\begin{equation}
 \PR\bigg( \sum_{|\mathscr{C}_{\sss (i)}|\leq \delta b_n} \mathscr{W}_i\times\surp{\mathscr{C}_{\sss (i)}}> \varepsilon b_n \bigg)\leq \frac{\ell_n^w}{\varepsilon b_n}\expt{\mathrm{SP}(\mathscr{C}(V_n'))\1_{\{ |\mathscr{C}(V_n')|\leq \delta b_n\}}}.
\end{equation} 
Moreover, an analogue of Lemma~\ref{c2lem:sp-cv-n} also holds for $V_n'$ (see Appendix~\ref{c2sec_appendix}), and the proof of \eqref{c2eq:suff-U0-conv-2} can now be completed in an identical manner as the proof of Proposition~\ref{c2prop-surp-u-0}.
\end{proof}
While studying percolation in the next section, we will need an estimate for the proportion of degree-one vertices in the large components. In fact, an application of Theorem~\ref{c2thm:comp-functionals}, yields the following result about the degree composition of the largest clusters:
\begin{corollary}\label{c2cor:degree-k} Consider $\mathrm{CM}_n(\boldsymbol{d})$ satisfying \textrm{Assumption~\ref{c2assumption1}}. Let $v_k(G)$ denote the number of vertices of degree $k$ in the graph $G$.  Then, for any fixed $i\geq 1$,
\begin{equation} \label{c2eqn_vertices_of_degree_k}
   v_k \big( \mathscr{C}_{\sss(i)} \big) = \frac{kr_k}{\mu} \big| \mathscr{C}_{\sss (i)} \big| +\oP(b_n),   
  \end{equation}
  where $r_k=\mathbb{P}(D=k)$. 
  Denote $\mathbf{Z}^k_n= \ord( b_n^{-1} v_k ( \mathscr{C}_{\sss(i)} ),\mathrm{SP}(\mathscr{C}_{\sss (i)}))_{i\geq 1}$, $\mathbf{Z}^k:=\ord(\frac{kr_k}{\mu}\gamma_i(\lambda), N(\gamma_i))_{i\geq 1}$, where $\gamma_i(\lambda)$, and  $N(\gamma_i)$ are defined in \textrm{Theorem~\ref{c2thm:spls}}. As $n\to\infty$,
\begin{equation}\label{c2eq:conv-amc-k}
 \mathbf{Z}_n^k \dto \mathbf{Z}^k,
\end{equation}with respect to the $\mathbb{U}^0_{\shortarrow}$ topology. 
\end{corollary}
\begin{proof}
 The proof follows directly from Theorem~\ref{c2thm:comp-functionals} by putting $w_i = \ind{d_i=k}$. The fact that this weight sequence satisfies Assumption~\ref{c2assumption-weight} is a consequence of Assumption~\ref{c2assumption1}.
\end{proof}

\section{Percolation}\label{c2sec:perc} 
In this section, we study critical percolation on the configuration model for fixed $\lambda\in\R$ and complete the proof of Theorem~\ref{c2thm:percolation}. 
As discussed earlier, $\mathrm{CM}_n(\boldsymbol{d},p)$ is obtained by first constructing $\mathrm{CM}_n(\boldsymbol{d})$ and then deleting each edge with probability $1-p$, independently of each other, and  the graph $\CM$. 
An interesting property of the configuration model is that $\mathrm{CM}_n(\boldsymbol{d},p)$ is also  distributed as a configuration model conditional on the degrees \cite{F07}. The rough idea here is to show that the degree distribution of $\mathrm{CM}_n(\bld{d},p_n(\lambda))$ satisfies Assumption~\ref{c2assumption1}, where $p_n(\lambda)$ is given by Assumption~\ref{c2assumption2}. This allows us to invoke Theorem~\ref{c2thm:spls} and complete the proof of Theorem~\ref{c2thm:percolation}. 
Recall from Assumption~\ref{c2assumption2} that $\nu=\lim_{n\to\infty}\nu_n>1$, and $p_n=p_n(\lambda)=\nu_n^{-1}(1+\lambda c_n^{-1})$. We start by describing an algorithm due to \citet{J09} that is easier to work with.
\begin{algo}[Construction of $\mathrm{CM}_n(\boldsymbol{d},p_n)$] \label{c2algo:perc}
\normalfont Initially, vertex $i$ has $d_i$ half-edges incident to it. For each half-edge $e$, let $v_e$ be the vertex to which $e$ is incident.
\begin{itemize}
 \item[(S1)]  With probability $1-\sqrt{p_n}$, one detaches $e$ from $v_e$ and associates $e$ to a new vertex $v'$ of degree-one. Color the new vertex $red$. This is done independently for every existing half-edge and we call this whole process $explosion$. Let $n_+$ be the number of red vertices created by explosion and $\tilde{n}=n+n_+$.  Denote the degree sequence obtained from the above procedure by $\Mtilde{\boldsymbol{d}} = ( \tilde{d}_i )_{i \in [\tilde{n}]}$, i.e., $\tilde{d}_i \sim \text{Bin} (d_i, \sqrt{p_n})$ for $i \in [n]$ and $\tilde{d}_i=1$ for $i \in [\tilde{n}] \setminus [n]$; 
 \item[(S2)] Construct $\mathrm{CM}_{\tilde{n}}(\Mtilde{\boldsymbol{d}})$ independently of (S1);
 \item[(S3)] Delete all the red vertices and the edges attached to them.
 \end{itemize}
 \end{algo}
It was also shown in \cite{J09} that the obtained multigraph has the same distribution as $\mathrm{CM}_{n}(\boldsymbol{d},p)$ if we replace (S3) by
 \begin{itemize}
 \item[(S3$'$)] Instead of deleting red vertices, choose $n_+$ degree-one vertices uniformly at random without replacement, independently of (S1), and (S2) and delete them. 
 \end{itemize}

\begin{remark}\normalfont Notice that Algorithm~\ref{c2algo:perc}~(S1) induces a probability measure $\mathbb{P}_p^n$ on $\N^\infty$. Denote their product measure by $\mathbb{P}_p$. In words, for different $n$, (S1) is carried out independently.  All the almost sure statements about the degrees in this section will be with  respect to the probability measure $\mathbb{P}_p$. 
\end{remark} 
 Let us first show that $\Mtilde{\boldsymbol{d}}$ also satisfies Assumption~\ref{c2assumption1}~\eqref{c2assumption1-2}. Note that the total number of half-edges remains unchanged during the explosion in Algorithm~\ref{c2algo:perc}~(S1) and therefore, $\sum_{i\in [\tilde{n}]}\tilde{d}_i=\sum_{i\in [n]}d_i$ and by Assumption~\ref{c2assumption2}~\eqref{c2assumption2-1},
\begin{equation}\label{c2eq:perc-mu}
 \frac{1}{n}\sum_{i\in [\tilde{n}]}\tilde{d}_i \to \mu \quad \PR_p \text{ a.s.}
\end{equation}This verifies the first moment condition in Assumption~\ref{c2assumption1}~\eqref{c2assumption1-2} for the percolated degree sequence $\PR_p$ a.s.
 Let $I_{ij}$:= the indicator of the $j^{th}$ half-edge corresponding to vertex $i$ being kept after the explosion. Then $I_{ij} \sim \text{Ber} (\sqrt{p_n})$ independently for $i \in [n]$, $j \in [d_i]$. Let 
 \begin{equation} 
 \mathbf{I}:= (I_{ij})_{j \in [d_i], i \in [n]}  \quad\text{and} \quad  f_1(\mathbf{I}):=\sum_{i\in [n]} \tilde{d_i}(\tilde{d}_i-1).
 \end{equation}Note that $f_1(\mathbf{I})=\sum_{i\in [\tilde{n}]}\tilde{d}_i(\tilde{d}_i-1)$ since the degree-one vertices do not contribute to the sum. One can check that by changing the status of one half-edge corresponding to vertex $k$ we can change $f_1$ by at most $2(d_{k}+1)$. Therefore an application  of \cite[Corollary 2.27]{JLR00} yields 
 \begin{eq}\label{c2eqn::prob:ineq:third}
 &\mathbb{P}_p \Big( \Big|\sum_{i\in [n]} \tilde{d_i}(\tilde{d}_i-1)- p_n \sum_{i\in [n]} d_i(d_i-1) \Big| >t \Big)\\
 & \leq 2 \exp \bigg( -\frac{t^2}{2\sum_{i \in [n]} d_i (d_{i}+1)^2}\bigg).
 \end{eq}
Now by Assumption~\ref{c2assumption2}~\eqref{c2assumption2-1}, $\sum_{i\in [n]}d_i^3=O(a_n^3)$. If we set $t=n^{1-\varepsilon}c_n^{-1}$, then $t^2/(\sum_{i\in [n]}d_i^3)$ is of the order $n^{\alpha-2\varepsilon}/L(n)$. Thus, choosing $\varepsilon<\alpha/2$, using \eqref{c2eqn::prob:ineq:third} and the Borel-Cantelli lemma we conclude that 
 \begin{equation}\label{c2perc:2nd-moment}\sum_{i\in [n]} \tilde{d_i}(\tilde{d}_i-1)= p_n \sum_{i\in [n]} d_i(d_i-1) +o(nc_n^{-1}) \quad \mathbb{P}_p\text{ a.s.} 
 \end{equation}Thus, using Assumption~\ref{c2assumption2}, the second moment condition in Assumption~\ref{c2assumption1} \eqref{c2assumption1-2} is verified for the percolated degree sequence $\PR_p$ a.s. Let $\tilde{d}_{\sss (i)}$ denote the $i^{th}$ largest value of $(\tilde{d}_i)_{i\in [\tilde{n}]}$. The third-moment condition in Assumption~\ref{c2assumption1}~\eqref{c2assumption1-2} is obtained by noting that $\tilde{d}_i\leq d_i$ for all $i\in [n]$ and  
 \begin{equation}\label{c2eq:third-moment-perc}\begin{split}
  &\lim_{K\to\infty}\limsup_{n\to\infty}a_n^{-3}\sum_{i=K+1}^{\tilde{n}}\tilde{d}_{\sss (i)}^{3}\leq \lim_{K\to\infty}\limsup_{n\to\infty}a_n^{-3}\sum_{i=K+1}^{\tilde{n}}\tilde{d}_{i}^{3}\\ &\leq \lim_{K\to\infty}\limsup_{n\to\infty}a_n^{-3}\Big(\sum_{i=K+1}^{n}\tilde{d}_{i}^{3}+ n_+\Big)\leq  \lim_{K\to\infty}\limsup_{n\to\infty}a_n^{-3}\Big(\sum_{i=K+1}^{n}d_{i}^{3}+ n_+\Big), 
\end{split}
\end{equation}which tends to zero $\PR_p$ a.s., where we have used Assumption~\ref{c2assumption2}~\eqref{c2assumption2-1} and the fact that $a_n^{-3}n_+\to 0$, $\PR_p$ a.s., which follows by observing that $n_+\sim\mathrm{Bin}(\ell_n,1-\sqrt{p_n})$. To see that $\Mtilde{\boldsymbol{d}}$ satisfies Assumption~\ref{c2assumption1}~\eqref{c2assumption1-3} note that by \eqref{c2perc:2nd-moment},
 \begin{equation}
  \frac{\sum_{i\in [\tilde{n}]}\tilde{d}_i(\tilde{d}_i-1)}{\sum_{i\in [\tilde{n}]}\tilde{d}_i}=p_n\frac{\sum_{i\in [n] }d_i(d_i-1)}{\sum_{i\in [n] }d_i}+o(c_n^{-1})=1+\lambda c_n^{-1}+o(c_n^{-1})
 \end{equation}$\PR_p \text{ a.s.,}$ where the last step follows from Assumption~\ref{c2assumption2}~\eqref{c2assumption2-2}. Assumption~\ref{c2assumption1}~\eqref{c2assumption1-4} is trivially satisfied by $\Mtilde{\boldsymbol{d}}$. 
Finally, in order to verify Assumption~\ref{c2assumption1}~\eqref{c2assumption1-1}, it suffices to show that 
 \begin{equation}\label{c2eq:assump-1-perc}
  \frac{\tilde{d}_{\sss(i)}}{a_n}\to \theta_i\sqrt{p}, \quad \PR_p \text{ a.s.,}
 \end{equation}where $p=1/\nu$.
 Recall that $\Mtilde{d}_i\sim \mathrm{Bin}(d_i,\sqrt{p_n})$. A standard concentration inequality for the binomial distribution \cite[(2.9)]{JLR00} yields that, for any $0<\varepsilon\leq 3/2$, 
\begin{equation} 
 \PR(|\tilde{d}_i-d_i\sqrt{p_n}|>\varepsilon d_i\sqrt{p_n})\leq 2\mathrm{exp}(-\varepsilon^2d_i\sqrt{p_n}/3),
 \end{equation}
  and using the Borel-Cantelli lemma it follows that $\mathbb{P}_p$ almost surely, $\Mtilde{d}_i=d_i\sqrt{p_n}(1+o(1))$ for all fixed $i$. 
Moreover, an application of \eqref{c2eq:third-moment-perc} yields that 
 \begin{equation}
  \lim_{K\to\infty}\limsup_{n\to\infty}a_n^{-3}\max_{i>K}\tilde{d}_i^3  =0.
 \end{equation}
  Now, since $\boldsymbol{\theta}$ is an ordered vector, the proof of \eqref{c2eq:assump-1-perc} follows.
 
To summarize, the above discussion in \eqref{c2eq:perc-mu},~\eqref{c2perc:2nd-moment},~\eqref{c2eq:third-moment-perc},~and~\eqref{c2eq:assump-1-perc} yields that the degree sequence $\Mtilde{d}$ satisfies all the conditions in Assumption~\ref{c2assumption1}. Therefore, Theorem~\ref{c2thm:spls} can be applied to $\mathrm{CM}_{\tilde{n}}(\Mtilde{\boldsymbol{d}})$.
Denote by $\tilde{\mathscr{C}}_{(i)}$ the $i^{th}$ largest component of $\mathrm{CM}_{\tilde{n}}(\Mtilde{\boldsymbol{d}})$. 
Let $\tilde{\mathbf{Z}}_n= \ord( b_n^{-1}|\tilde{\mathscr{C}}_{(i)}|,\mathrm{SP}(\tilde{\mathscr{C}}_{(i)})_{i\geq 1}$ and $\tilde{\mathbf{Z}}:=\ord(\tilde{\gamma}_i(\lambda), N(\tilde{\gamma}_i))_{i\geq 1}$, where $\gamma_i(\lambda)$, and  $N(\gamma_i)$ are defined in \textrm{Theorem~\ref{c2thm:percolation}}. Now,  Theorem~\ref{c2thm:spls} implies
\begin{equation}
\tilde{\mathbf{Z}}_n \dto \tilde{\mathbf{Z}},
\end{equation}with respect to the $\mathbb{U}^0_{\shortarrow}$ topology.

Since the percolated degree sequence satisfies Assumption~\ref{c2assumption1} $\mathbb{P}_p$ a.s., \eqref{c2eqn_vertices_of_degree_k} holds for $\tilde{\mathscr{C}}_{\sss (i)}$ also. Let $v^d_1(\tilde{\mathscr{C}}_{\sss (i)})$ be the number of degree-one vertices of $\tilde{\mathscr{C}}_{\sss (i)}$ which are deleted while creating the graph $\mathrm{CM}_{n}(\boldsymbol{d},p_{n})$ from  $\mathrm{CM}_{\tilde{n}}(\Mtilde{\boldsymbol{d}})$. Since the vertices are to be chosen uniformly from all degree-one vertices as described in (S3$'$), 
\begin{equation} \label{c2degree_one_vertices}
\begin{split}
 v^d_1(\tilde{\mathscr{C}}_{\sss (i)}) &= \frac{n_+}{\tilde{n}_1}v_1(\tilde{\mathscr{C}}_{\sss (i)})+ \oP(b_n)= \frac{n_+}{\tilde{n}_1} \frac{\tilde{n}_1}{\ell_n} \big| \tilde{\mathscr{C}}_{\sss (i)}\big| + \oP(b_n) = \frac{n_+}{\ell_n}\big| \tilde{\mathscr{C}}_{(i)} \big|+ \oP(b_n)\\
 & = \frac{\mu\big(1-\sqrt{p}_n\big)+o(1)}{\mu+o(1)}\big| \tilde{\mathscr{C}}_{\sss (i)} \big|+ \oP(b_n) = \big(1-\sqrt{p}_n\big) \big| \tilde{\mathscr{C}}_{\sss (i)}\big| + \oP(b_n),
 \end{split}
\end{equation}where the last-but-one equality follows by observing that $n_+\sim\mathrm{Bin}(\ell_n,1-\sqrt{p_n})$. Now, notice that by removing degree-one vertices, the components are not broken up, so the  vector of component sizes for percolation can be obtained by just  subtracting the number of red vertices from the component sizes of $\mathrm{CM}_{\tilde{n}}(\Mtilde{\boldsymbol{d}})$. Moreover, the removal of degree-one vertices does not effect the count of surplus edges.
Therefore, the proof of Theorem~\ref{c2thm:percolation} is complete by using Corollary~\ref{c2cor:degree-k}. 

\section{Convergence to augmented multiplicative coalescent}
\label{c2sec:conv-amc}
Let us give an overview of the organization of this section:
In Section~\ref{c2sec:dyn-cons-coup}, we discuss an alternative dynamic  construction that approximates the  percolated graph process, coupled in a natural way.
This construction enables us to compare the coupled percolated graphs with a dynamic construction.
Then, we describe a modified system that evolves as an exact augmented multiplicative coalescent and the rest of the section is devoted to comparing the exact augmented multiplicative coalescent and the corresponding quantities for the graphs generated by the dynamic construction.
The ideas are similar to \cite[Section 8]{DHLS15}, and we only give the overall idea and the necessary details specific to this chapter.

\subsection{The dynamic construction and the coupling} \label{c2sec:dyn-cons-coup}
Let us consider graphs generated dynamically as follows: 
\begin{algo}\label{c2algo:dyn-cons-2} \normalfont Let $s_1(t)$ be the total number of unpaired or \emph{open} half-edges at time $t$, and  $\Xi_n$ be an inhomogeneous Poisson process with rate $s_1(t)$ at time $t$. 
\begin{itemize}
\item[\textrm{(S0)}] Initially, $s_1(0)=\ell_n$, and $\mathcal{G}_n(0)$ is the empty graph on vertex set $[n]$. 
\item[\textrm{(S1)}] At each event time of $\Xi_n$, choose two open half-edges uniformly at random and pair them. The graph $\mathcal{G}_n(t)$ is obtained by adding this edge to $\mathcal{G}_n(t-)$. Decrease $s_1(t)$ by two. Continue until $s_1(t)$ becomes zero.
\end{itemize} 
\end{algo} 
Notice that $\mathcal{G}_n(\infty)$ is distributed as $\CM$ since an open  half-edge is paired with another uniformly chosen open half-edge. 
The next proposition ensures that the graph process generated by Algorithm~\ref{c2algo:dyn-cons-2} \emph{sandwiches} the graph process $(\mathrm{CM}_n(\bld{d},p_n(\lambda)))_{\lambda\in\R}$. This result was proved in \cite[Proposition~28]{DHLS15}. The proof is identical under Assumption~\ref{c2assumption2} and therefore is omitted here.
Define,
\begin{equation}\label{c2defn:t-n-lambda}
t_n(\lambda)=\frac{1}{2}\log\bigg(\frac{\nu_n}{\nu_n-1}\bigg)+\frac{1}{2(\nu_n-1)}\frac{\lambda}{c_n}.
\end{equation}
\begin{proposition}\label{c2prop:coupling-whp} Fix $-\infty<\lambda_\star<\lambda^\star<\infty$. There exists a coupling such that with high probability
\begin{subequations}
\begin{equation}
 \mathcal{G}_n(t_n(\lambda)-\varepsilon_n)\subset \mathrm{CM}_n(\bld{d},p_n(\lambda)) \subset\mathcal{G}_n(t_n(\lambda)+\varepsilon_n),\quad \forall \lambda \in [\lambda_\star,\lambda^\star]
\end{equation} and 
\begin{equation}\label{c2eq:prop-coup-2}
 \mathrm{CM}_n(\bld{d},p_n(\lambda)-\varepsilon_n)\subset \mathcal{G}_n(t_n(\lambda))\subset \mathrm{CM}_n(\bld{d},p_n(\lambda)+\varepsilon_n),\quad \forall \lambda \in [\lambda_\star,\lambda^\star]
\end{equation}
\end{subequations}where $\varepsilon_{n}=cn^{-\gamma_0}$, for some $\eta<\gamma_0<1/2$ and the constant $c$ does not depend on $\lambda$.
\end{proposition}
From here onward, we augment $\lambda$ to a predefined notation to emphasize the dependence on~$\lambda$. 
We write $\mathscr{C}_{\sss (i)}(\lambda)$  for the $i^{th}$ largest component of $\mathcal{G}_n(t_n(\lambda))$ and define 
\begin{equation}
\mathcal{O}_i(\lambda)=\# \text{ open half-edges in }\mathscr{C}_{\sss (i)}(\lambda).
\end{equation}
Think of $\mathcal{O}_i(\lambda)$ as the \emph{mass} of the component $\mathscr{C}_{\sss (i)}(\lambda)$. 
Let $\mathbf{Z}_n^o(\lambda)$ denote the vector of the number of open half-edges (re-scaled by $b_n$) and surplus edges of $\mathcal{G}_n(t_n(\lambda))$, ordered as an element of $\mathbb{U}^0_{\shortarrow}$. 
For a process $\mathbf{X}$, we will write $\mathbf{X}[\lambda_\star,\lambda^\star]$ to denote the restricted process $(X(\lambda))_{\lambda\in[\lambda_\star,\lambda^\star]}$. 
 Let $\ell_n^o(\lambda) = \sum_{i\geq 1}\mathcal{O}_i(\lambda)$. 
Note that 
\begin{equation}\label{c2eq:asympt-ell-n-o}
\ell_n^o(\lambda) =  \frac{n\mu(\nu-1)}{\nu}(1+\oP(1)).
\end{equation} 
\eqref{c2eq:asympt-ell-n-o} is a consequence of \cite[Lemma 8.2]{BBSX14} since the proof only uses the facts that $|\ell_n/n-\mu|=o(n^{-\gamma})$ for all $\gamma<1/2$, and $\sum_{i\in [n]}d_i(d_i-1)/\ell_n\to\nu$.  
Now, observe that, during the evolution of the graph process generated  by Algorithm~\ref{c2algo:dyn-cons-2}, during the time interval $[t_n(\lambda),t_n(\lambda+\dif \lambda)]$, the $i^{th}$ and $j^{th}$ ($i> j$) largest components, merge at rate 
 \begin{equation}\label{c2rate:function}
2\mathcal{O}_{i}(\lambda) \mathcal{O}_{j}(\lambda)\times\frac{1}{\ell_n^o(\lambda)-1}\times \frac{1}{2(\nu_n-1)c_n}\approx \frac{\nu}{\mu(\nu-1)^2} \big(b_n^{-1}\mathcal{O}_{i}(\lambda)\big)\big(b_n^{-1}\mathcal{O}_{j}(\lambda)\big),
\end{equation}and create a component with $\mathcal{O}_{i}(\lambda)+\mathcal{O}_{j}(\lambda)-2$ open half-edges and \linebreak $\mathrm{SP}(\mathscr{C}_{\sss(i)}(\lambda))+\mathrm{SP}(\mathscr{C}_{\sss(j)}(\lambda))$ surplus edges. 
Also, a surplus edge is created in $\mathscr{C}_{\sss(i)}(\lambda)$ at rate
\begin{equation}\label{c2rate:function-spls}
\mathcal{O}_i(\lambda)(\mathcal{O}_i(\lambda)-1)\times\frac{1}{\ell_n^o(\lambda)-1}\times \frac{1}{2(\nu_n-1)c_n}\approx \frac{\nu}{2\mu(\nu-1)^2} \big(b_n^{-1}\mathcal{O}_{i}(\lambda)\big)^2,
\end{equation}and $\mathscr{C}_{\sss(i)}(\lambda)$ becomes a component with surplus edges  $\mathrm{SP}(\mathscr{C}_{\sss(i)}(\lambda))+1$  and open half-edges $\mathcal{O}_{i}(\lambda)-2$. Thus $\mathbf{Z}_n^o[\lambda_\star,\lambda^\star]$ does \emph{not} evolve as an AMC process but it is close. 
The fact that two half-edges are killed after pairing, makes the masses (the number of open half-edges) of the components deplete. If there were no such depletion of mass, then the vector of open half-edges, along with the surplus edges, would in fact  merge as an augmented multiplicative coalescent. 
Let us define the modified process \cite[Algorithm~7]{DHLS15} that in fact evolves as augmented multiplicative coalescent:

{\begin{algo}\label{c2algo:modify-dyn-cons} \normalfont Initialize $\bar{\mathcal{G}}_n(t_n(\lambda_\star)) = \mathcal{G}_n(t_n(\lambda_\star))$.  Let $\mathscr{O}$ denote the set of open half-edges in the graph $\mathcal{G}_n(t_n(\lambda_\star))$, $\bar{s}_1 = |\mathscr{O}|$ and $\bar{\Xi}_n$ denote a Poisson process with rate $\bar{s}_1$. At each event time of the Poisson process $\bar{\Xi}_n$, select two half-edges from $\mathscr{O}$ and create an edge between the corresponding vertices. However, the selected half-edges are kept alive, so that they can be selected again.
\end{algo} 

\begin{remark}\label{c2rem:modify-AMC}\normalfont The only difference between Algorithms~\ref{c2algo:dyn-cons-2}~and~\ref{c2algo:modify-dyn-cons}, is that the \emph{paired} half-edges are not discarded and thus more edges are created by Algorithm~\ref{c2algo:modify-dyn-cons}. Thus, there is a natural coupling between the graphs generated by Algorithms~\ref{c2algo:dyn-cons-2}~and~\ref{c2algo:modify-dyn-cons} such that $\mathcal{G}_n(t_n(\lambda))\subset \bar{\mathcal{G}}_n(t_n(\lambda))$ for all $\lambda\in [\lambda_\star,\lambda^\star]$, with probability one. In the subsequent part of this section, we will always work under this coupling. The extra edges that are created by Algorithm~\ref{c2algo:modify-dyn-cons} will be called \emph{bad} edges.
\end{remark}
 In the subsequent part of this chapter, we will augment a predefined notation with a bar to denote the corresponding quantity for $\bar{\mathcal{G}}_n(t_n(\lambda))$. 
 Denote $\beta_n = (\bar{s}_1(\nu_n-1)c_n)^{1/2}$ and $\bar{\mathbf{Z}}_n^{o,{\sss \mathrm{scl}}}(\lambda)$ denote the vector \linebreak $\ord(\beta_n^{-1}\bar{\mathcal{O}}_i(\lambda),\mathrm{SP}(\bar{\mathscr{C}}_{\sss (i)}(\lambda)))_{i\geq 1}$. 
 Using an argument identical to \eqref{c2rate:function},~and~\eqref{c2rate:function-spls}, it follows that $\bar{\mathbf{Z}}_n^{o,{\sss \mathrm{scl}}}[\lambda_\star,\lambda^\star]$ evolves as a standard augmented multiplicative coalescent.
 Note that there exists a constant $c>0$ such that $\beta_n = cb_n(1+\oP(1))$, and therefore the scaling limit of any finite-dimensional distributions of $\bar{\mathbf{Z}}_n^{o}[\lambda_\star,\lambda^\star]$ can be obtained from $\bar{\mathbf{Z}}_n^{o,{\sss \mathrm{scl}}}[\lambda_\star,\lambda^\star]$.

\subsubsection{Augmented multiplicative coalescent with mass and weight}
The near Feller property of the  augmented multiplicative coalescent \cite[Theorem 3.1]{BBW12} ensures the joint convergence of the number of open half-edges in each component together with the surplus edges of $\bar{\mathcal{G}}_n(t_n(\lambda))$. 
To deduce the scaling limits involving the components sizes let us consider a dynamic process that is further augmented by weight.
Initially, the system consists of particles (possibly infinitely many) where particle $i$ has mass $x_i$, weight $z_i$ and an attribute $y_i$. 
Let $(X_i(t),Z_i(t),Y_i(t))_{i\geq 1}$ denote masses, weights, and attribute values  at time $t$. 
The dynamics of the system is described as follows: At time $t$,
\begin{itemize}
\item[$\rhd$]  particles $i$ and $j$ coalesce at rate $X_i(t)X_j(t)$ and create a particle with mass $X_i(t)+X_j(t)$, weight $Z_i(t)+Z_j(t)$ and attribute $Y_i(t)+Y_j(t)$.
\item[$\rhd$]  for each $i$, attribute $Y_i(t)$ increases by 1 at rate $ Y_i^2(t)/2$.
\end{itemize}
For $(\bld{x},\bld{y}), (\bld{z},\bld{y})\in\mathbb{U}^0_{\shortarrow}$, we write $(\bld{x},\bld{z},\bld{y})$ for $((\bld{x},\bld{y}),(\bld{z},\bld{y}))\in (\mathbb{U}^0_{\shortarrow})^2$. 
Denote by  $\mathrm{MC}_2(\bld{x},\bld{z},t)$ and $\mathrm{AMC}_2(\bld{x},\bld{z},\bld{y},t)$ respectively the vector $(X_i(t), Z_i(t))_{i\geq 1}$ and $(X_i(t), Z_i(t),Y_i(t))_{i\geq 1}$ with initial mass $\bld{x}$, weight $\bld{z}$ and attribute value $\bld{y}$. 
We will need the following theorem:
\begin{theorem}\label{c2thm:AMC-2D}
Suppose that $(\bld{x}_n,\bld{z}_n,\bld{y}_n) \to (\bld{x},\bld{x},\bld{y})$ in $(\mathbb{U}^0_{\shortarrow})^2$ and $\sum_{i}x_i=\infty$. Then, for any $t\geq 0$
\begin{equation}
\mathrm{AMC}_2(\bld{x}_n,\bld{z}_n,\bld{y_n})\dto \mathrm{AMC}_2(\bld{x},\bld{x},\bld{y}).
\end{equation}
\end{theorem}
\begin{proof}
By \cite[Theorem 29]{DHLS15},
 \begin{equation}\label{c2mul-coal-2d}
  \mathrm{MC}_2(\bld{x}_n,\bld{z}_n,t) \dto \mathrm{MC}_2(\bld{x},\bld{x},t).
 \end{equation}
For $\bld{x}_n = (x_i^n)_{i\geq 1}$, and $\bld{z}_n = (z_i^n)_{i\geq 1}$  let $\bld{w}_n^+ = \mathrm{sort}(x_i^n\vee z_i^n)$, $\bld{w}_n^-=\mathrm{sort}(x_i^n\wedge z_i^n)$, where $\mathrm{sort}$ denotes the decreasing ordering of the elements. 
Notice that $\bld{w}_n^+ \to \bld{x}$, and $\bld{w}_n^- \to \bld{x}$ in $\ell^2_{\shortarrow}$.
Let us denote by $\mathrm{AMC}_1(\bld{x},\bld{y},t)$ the usual augmented multiplicative coalescent process at time $t$ with starting state $(\bld{x},\bld{y})$.
Now, since $\sum_i x_i =\infty$, we can use the near Feller property \cite[Theorem~3.1]{BBW12} to conclude that $\mathrm{AMC}_1(\bld{x}_n,\bld{y}_n,t)\dto \mathrm{AMC}_1(\bld{x},\bld{y},t)$.
Moreover, $\mathrm{AMC}_2(\bld{w}_n^+,\bld{w}_n^+,y_n,t)$ and 
$\mathrm{AMC}_2(\bld{w}_n^-,\bld{w}_n^-,y_n,t)$ converges to the same limit. 
For $(\bld{x},\bld{z},\bld{y})\in(\mathbb{U}^0_{\shortarrow})^2$, if $S_{\sss \mathrm{pr}}(\bld{x},\bld{z},\bld{y}) = \sum_{i}z_iy_i$, then under the subgraph coupling
\begin{equation}
S_{\sss\mathrm{pr}}(\mathrm{AMC}_2(\bld{w}_n^+,\bld{w}_n^+,y_n,t))- S_{\sss\mathrm{pr}}(\mathrm{AMC}_2(\bld{w}_n^-,\bld{w}_n^-,y_n,t)) \pto 0,
\end{equation}
which implies that
\begin{eq}\label{c2AMC-ded}
 &(\mathrm{AMC}_1(\bld{x}_n,\bld{y}_n,t),S_{\sss\mathrm{pr}}(\mathrm{AMC}_2(\bld{x}_n,\bld{z}_n,y_n,t))) \\
 &\hspace{.6cm}\dto (\mathrm{AMC}_1(\bld{x},\bld{y},t),S_{\sss\mathrm{pr}}(\mathrm{AMC}_2(\bld{x},\bld{x},y,t))).
\end{eq}
Now, using \eqref{c2mul-coal-2d}, \eqref{c2AMC-ded}, an application of  \cite[Lemma 4.11]{BBW12} concludes the proof. 
\end{proof}
\subsection{Asymptotics for the open half-edges}
The following lemma shows that the number of open half-edges in $\mathcal{G}_n(t_n(\lambda))$ is \emph{approximately} proportional to the component sizes. 
This will enable us to apply Theorem~\ref{c2thm:AMC-2D} for deducing the scaling limits of the required quantities for the graph $\bar{\mathcal{G}}_n(t_n(\lambda))$.
\begin{lemma}\label{c2thm:open-comp}
 There exists a constant $\kappa > 0$ such that, for any $i\geq 1$,
 \begin{equation}\label{c2open-he}
  \mathcal{O}_i(\lambda)= \kappa |\mathscr{C}_{\sss (i)}(\lambda)|+o_{\sss \PR}(b_n).
 \end{equation}Further, $(\mathbf{Z}_n^o(\lambda))_{n\geq 1}$ is tight in  $\mathbb{U}^0_{\shortarrow}$. 
\end{lemma}
\begin{proof} Let $(d_k^\lambda)_{k\in [n]}$ denote the degree sequence of $\mathrm{CM}_n(\bld{d},p_n(\lambda))$ and define
\begin{equation}
\mathcal{O}_i^p(\lambda) = \sum_{k\in \mathscr{C}_{\sss (i)}^p(\lambda)}(d_k-d_k^\lambda) = \sum_{k\in \mathscr{C}_{\sss (i)}^p(\lambda)}d_k-2(|\mathscr{C}_{\sss (i)}^p(\lambda)|-1+\mathrm{SP}(\mathscr{C}_{\sss (i)}^p(\lambda))).
\end{equation}
 Using \eqref{c2eq:prop-coup-2} and the fact that the number of surplus edges in the large components are tight, it is enough to prove the lemma by replacing $\mathcal{O}_i(\lambda)$ by $\mathcal{O}_i^p(\lambda)$ and $\mathscr{C}_{\sss (i)}'(\lambda)$ by $\mathscr{C}_{\sss (i)}^p(\lambda)$.
For a component $\tilde{\mathscr{C}}$ of $\mathrm{CM}_{\tilde{n}}(\tilde{\boldsymbol{d}})$, the corresponding component in the percolated graph is obtained by cleaning up $R(\tilde{\mathscr{C}})$ red degree-one vertices. Thus, the degree deficiency of that percolated cluster is given by 
\begin{equation}\label{c2relation-deg-def}
 \sum_{k\in \tilde{\mathscr{C}}\cap [n]}d_k-\sum_{k\in \tilde{\mathscr{C}}\cap [n]}\tilde{d}_k+2R(\tilde{\mathscr{C}}).
\end{equation} Now, all the three terms appearing in the right hand side of \eqref{c2relation-deg-def} can be estimated using Theorem~\ref{c2thm:comp-functionals}, where we recall from Section~\ref{c2sec:perc} that $\Mtilde{\bld{d}}$ satisfies Assumption~\ref{c2assumption1}. The proof is now complete.
\end{proof}
For an element $\mathbf{z} = (x_i,y_i)_{i\geq 1} \in \mathbb{U}^0_{\shortarrow}$ and a constant $c>0$, denote $c \mathbf{z} = (c x_i,y_i)_{i\geq 1}$.  Thus, Lemma~\ref{c2thm:open-comp} states that, for each fixed $\lambda$, $\mathbf{Z}^o_n(\lambda)$ is close to $\kappa \bZ_n(\lambda)$. The following lemma states that formally:
\begin{corollary}\label{c2cor-fixed-lambda}  For each fixed $\lambda$, as $n\to\infty$, $\mathrm{d}_{\sss\mathbb{U}}(\mathbf{Z}^o_n(\lambda),\kappa \bZ_n(\lambda))\pto 0$.
\end{corollary}
\begin{proof}Let
$\pi_k,T_k:\mathbb{U}^0_{\shortarrow}\mapsto\mathbb{U}^0_{\shortarrow}$ be the functions such that for $\mathbf{z}=((x_i,y_i))_{i\geq 1}$, $\pi_k(\mathbf{z})$ consists of only $(x_i,y_i)$ for $i\leq k$ and zeroes in other coordinates, and $T_k(\mathbf{z})$ consists only of $(x_i,y_i)$ for $i>k$. Thus,
\begin{equation}\label{c2split-up-fixed-lambda}
\begin{split}
 \mathrm{d}_{\mathbb{U}}\left(\mathbf{Z}_n^o(\lambda),\kappa\mathbf{Z}_n(\lambda)\right)\leq\mathrm{d}_{\mathbb{U}}\left(\pi_K(\mathbf{Z}_{n}^o(\lambda)),\pi_K(\kappa\mathbf{Z}_n(\lambda))\right)&\\
 +\|T_K(\mathbf{Z}_n^o(\lambda))\|_{\sss\mathbb{U}}+\|T_K(\kappa\mathbf{Z}_n(\lambda))\|_{\sss\mathbb{U}}&.
 \end{split}
\end{equation}
 Now, for each fixed $K\geq 1$ the first term in the right hand side of \eqref{c2split-up-fixed-lambda} converges in probability to zero, by \eqref{c2open-he}. Also, using the tightness of both $(\bZ_n(\lambda))_{n\geq 1}$ and $(\bZ_n^o(\lambda))_{n\geq 1}$ with respect to the $\mathbb{U}^0_{\shortarrow}$ topology, it follows that for any $\varepsilon>0$,
\begin{equation}
\lim_{K\to\infty}\lim_{n\to\infty}\prob{\|T_K(\mathbf{Z}_n(\lambda))\|_{\sss \mathbb{U}}>\varepsilon}=\lim_{K\to\infty}\lim_{n\to\infty}\prob{\|T_K(\mathbf{Z}_n^o(\lambda))\|_{\sss \mathbb{U}}>\varepsilon}=0,
\end{equation}and the proof is now complete.
\end{proof}

\subsection{Comparison between the dynamic construction and the modified process}
Suppose that, at time $\lambda_\star$, we have colored the components $(\mathscr{C}_{\sss (i)}(\lambda_\star))_{i\in [M]}$  blue, say, and then let Algorithms~\ref{c2algo:dyn-cons-2}~and~\ref{c2algo:modify-dyn-cons} evolve. 
Additionally, we color all the components blue that get connected to one of the blue components during the evolution. 
Let $\mathcal{C}_M(\lambda)$, $\bar{\mathcal{C}}_M(\lambda)$ denote the union of all such blue components in $\mathcal{G}_n(t_n(\lambda))$ and $\bar{\mathcal{G}}_n(t_n(\lambda))$. 
In this section, we show that (i) no bad edges are created that are surplus edge of some component, (ii) $|\bar{\mathcal{C}}_M(\lambda)|-|\mathcal{C}_M(\lambda)|$ is asymptotically negligible, 
(iii) no bad edge is created between the large components, and (iv) with sufficiently large probability, the largest components of $\bar{\mathcal{G}}_n(t_n(\lambda))$ are contained within $\bar{\mathcal{C}}_M(\lambda)$, where $M$ is large.
These facts together ensure that the scaling limit for the largest connected components and surplus edges of $\mathcal{G}_n(t_n(\lambda))$ and $\bar{\mathcal{G}}_n(t_n(\lambda))$ are identical.
Consider the coupled evolution of Algorithms~\ref{c2algo:dyn-cons-2}~and~\ref{c2algo:modify-dyn-cons}. 
Thus, in the modified set-up, more components get merged due to the creation of bad-edges.
Denote $\mathcal{B}_M(\lambda) = |\bar{\mathcal{C}}_M(\lambda)|-|\mathcal{C}_M(\lambda)|$ and $B_{\sss \mathrm{SP}}(\lambda)$ the number of bad-edges that are created as surplus edge of some component. 

\begin{lemma} \label{c2lem:bad-estimation}
For any $\lambda\geq \lambda_\star$, $B_{\sss \mathrm{SP}}(\lambda)\pto 0$ and 
for all $M\geq 1$,  $b_n^{-1} \mathcal{B}_M(\lambda) \pto 0$.  
\end{lemma}
\begin{proof}
 Before going into the proof, recall Algorithm~\ref{c2algo:modify-dyn-cons}, and all the definitions. 
 A bad edge is created if, during some event time of $\bar{\Xi}_n$, a half-edge from $\mathscr{O}$ is selected that was already selected before. 
 Now, for some given pair $(e_0,f_0)$, $e_0\neq f_0$, the number of ways in which one can choose a pair $(e,f)$, $e\neq f$ such that $e=e_0$, or $f=f_0$, is given by $2\bar{s}_1-3$. 
 Thus, the bad edges are created between times $[t_n(\lambda),t_n(\lambda+\dif \lambda)]$ at rate $(2(\nu_n-1)\bar{s}_1c_n)^2/ (2\bar{s}_1-3)$.
 Denote $\mathcal{I}_M = \mathcal{I}_M(\lambda)= \{i:\bar{\mathscr{C}}_{\sss (i)}(\lambda)\subset \bar{\mathcal{C}}_M(\lambda)\}$.
 The created bad edge adds an additional mass of $|\bar{\mathscr{C}}_{\sss (i)}(\lambda)|$ to $\bar{\mathcal{C}}_M(\lambda)$ if one end is from $\bar{\mathcal{C}}_M(\lambda)$ (for which there are $\sum_{i\in \mathcal{I}_M}\mathcal{O}_i(\lambda)$ possibilities) and the other half-edge is in $\bar{\mathscr{C}}_{\sss (i)}(\lambda)$.
 The created bad edge is a surplus edge if both of its endpoints come from the same component. 
For any semi-martingale $(Y_t)_{t\geq 0}$, we write $\mathrm{D}(Y)(t)$ and $\mathrm{QV}(Y)(t)$, respectively to denote the compensator and the quadratic variation, i.e., 
\begin{equation}
Y_t - \mathrm{D}(Y)(t), \quad\text{and}\quad (Y_t - \mathrm{D}(Y)(t))^2 - \mathrm{QV}(Y)(t)
\end{equation}are both martingales.
Now, $\mathrm{D}(B_{\sss\mathrm{SP}}(\lambda))\geq 0$, $\mathrm{D}(b_n^{-1}\mathcal{B}_1(\lambda)) \geq 0$, and for some constants $C_1,C_2>0$
\begin{subequations}
\begin{equation}
\begin{split}
\mathrm{D}(B_{\sss\mathrm{SP}})(\lambda) &= \int_{\lambda_\star}^{\lambda} \frac{2\bar{s}_1-3}{4(\nu_n-1)^2\bar{s}_1^2c_n^2} \sum_{i\geq 1} \binom{\bar{\mathcal{O}}_i(\lambda')}{2}\dif \lambda'\\
&\leq \frac{C_1 n}{b_n^2}\int_{\lambda_\star}^{\lambda} \|\bar{\mathbf{O}}_n(\lambda')\|_{\sss 2}^2\dif \lambda'+\oP(1)\\
 & \leq  \frac{C_1 n}{b_n^2}(\lambda^\star -\lambda_\star) \|\bar{\mathbf{O}}_n(\lambda^\star)\|_{\sss 2}^2 +\oP(1),
 \end{split}
\end{equation}
\begin{equation}
\begin{split}
&\mathrm{D}(b_n^{-1}\mathcal{B}_1)(\lambda)
\leq b_n^{-1}\int_{\lambda_\star}^{\lambda} \frac{2\bar{s}_1-3}{4(\nu_n-1)^2\bar{s}_1^2c_n^2} \sum_{i\in \mathcal{I}_M(\lambda)}\bar{\mathcal{O}}_i(\lambda')\sum_{i\geq 1} \bar{\mathcal{O}}_i(\lambda') |\bar{\mathscr{C}}_{\sss (i)}(\lambda')| \dif \lambda'\\
  & \leq  \frac{C_2n}{b_n^2}\int_{\lambda_\star}^{\lambda} \bigg(b_n^{-1}\sum_{i=1}^M\bar{\mathcal{O}}_{\sss (i)}(\lambda')\bigg) \|\bar{\mathbf{O}}_n(\lambda')\|_{\sss 2}\|\bar{\mathbf{C}}_n(\lambda')\|_{\sss 2} \dif \lambda' +\oP(1)\\
  & \leq \frac{C_2n}{b_n^2}(\lambda^\star-\lambda_\star) \bigg(b_n^{-1}\sum_{i=1}^M\bar{\mathcal{O}}_{\sss (i)}(\lambda^\star)\bigg) \|\bar{\mathbf{O}}_n(\lambda^\star)\|_{\sss 2}\|\bar{\mathbf{C}}_n(\lambda^\star)\|_{\sss 2} +\oP(1),
\end{split}
\end{equation}
\end{subequations}
where $\bar{\mathcal{O}}_{\sss (i)}$ denotes the $i^{th}$ largest value of $(\bar{\mathcal{O}}_i)_{i\geq 1}$. 
Further,
\begin{subequations}
\begin{equation}
\mathrm{QV}(B_{\sss\mathrm{SP}})(\lambda) \leq \frac{C_1 n}{b_n^2}(\lambda^\star -\lambda_\star) \|\bar{\mathbf{O}}_n(\lambda^\star)\|_{\sss 2}^2 +\oP(1),
\end{equation}and
\begin{equation}
\begin{split}
&\mathrm{QV}(b_n^{-1}\mathcal{B}_1)(\lambda)
\leq b_n^{-2}\int_{\lambda_\star}^{\lambda} \frac{2\bar{s}_1-3}{4(\nu_n-1)^2\bar{s}_1^2c_n^2} \sum_{i\in \mathcal{I}_M(\lambda)}\bar{\mathcal{O}}_i(\lambda')\sum_{i\geq 1} \bar{\mathcal{O}}_i(\lambda') |\bar{\mathscr{C}}_{\sss (i)}(\lambda')|^2 \dif \lambda'\\
  & \leq  \frac{C_2n}{b_n^2}\int_{\lambda_\star}^{\lambda} \bigg(b_n^{-1}\sum_{i=1}^M\bar{\mathcal{O}}_{\sss (i)}(\lambda')\bigg)\big(b_n^{-1}|\bar{\mathscr{C}}_{\sss (1)}(\lambda')|\big) \|\bar{\mathbf{O}}_n(\lambda')\|_{\sss 2}\|\bar{\mathbf{C}}_n(\lambda')\|_{\sss 2} \dif \lambda' +\oP(1)\\
  &\leq \frac{C_2n}{b_n^2}(\lambda^\star-\lambda_\star) \bigg(b_n^{-1}\sum_{i=1}^M\bar{\mathcal{O}}_{\sss (i)}(\lambda^\star)\bigg)\big(b_n^{-1}|\bar{\mathscr{C}}_{\sss (1)}(\lambda^\star)|\big) \|\bar{\mathbf{O}}_n(\lambda^\star)\|_{\sss 2}\|\bar{\mathbf{C}}_n(\lambda^\star)\|_{\sss 2} +\oP(1).
\end{split}
\end{equation}
\end{subequations}
Recall that using Lemma~\ref{c2thm:open-comp}, an application of Theorem~\ref{c2thm:AMC-2D} yields $(\bar{\mathbf{Z}}_n(\lambda))_{n\geq 1}$ is tight in $\mathbb{U}^0_{\shortarrow}$.
The proof now follows using the fact that $n/b_n^{2}\to 0$.
\end{proof}
Suppose that a bad edge is being created at time $\lambda'$. 
Now, this bad edge may be created by choosing the open half-edges from $\mathscr{C}_{\sss (i)}(\lambda')$ and $\mathscr{C}_{\sss (j)}(\lambda')$ for $1\leq i,j\leq M$. For fixed $M$, let $F_M(\lambda)$ denote the number of such bad-edges created upto time $\lambda$. Using an argument identical to Lemma~\ref{c2lem:bad-estimation} one can show the following:
\begin{lemma}\label{c2lem:no-bad-large}
For any $\lambda\geq \lambda_\star$ and $M\geq 1$, $F_M(\lambda)\pto 0$.
\end{lemma}
The following is the last ingredient that will be needed in the proof:
\begin{lemma} \label{c2lem:large-comp-contain} 
Fix any $\lambda \in [\lambda_\star,\lambda^\star]$. For any $\varepsilon>0$, and $K\geq 1$, there exists $M = M(\varepsilon,K)$ such that
\begin{equation}
\limsup_{n\to\infty}\prob{\bar{\mathscr{C}}_{\sss (1)}(\lambda), \dots,\bar{\mathscr{C}}_{\sss (K)}(\lambda) \text{ are not contained in }\bar{\mathcal{C}}_M(\lambda)}\leq \varepsilon.
\end{equation}
\end{lemma}
\begin{proof}
 Let $\mathcal{I}_M:= \{i:\bar{\mathscr{C}}_{\sss (i)}(\lambda)\subset \bar{\mathcal{C}}_M(\lambda)\}$. It is enough to show that, for any $\varepsilon>0$, there exists $M$ such that 
 \begin{equation}\label{c2trunc-large-1}
  \limsup_{n\to\infty}\PR\bigg(\sum_{i\notin\mathcal{I}_M}|\bar{\mathscr{C}}_{\sss (i)}(\lambda)|^2>\varepsilon b_n^2 \bigg)\leq \varepsilon.
 \end{equation} 
For any $M\geq 1$, consider the merging dynamics of Algorithm~\ref{c2algo:modify-dyn-cons}, where at time $\lambda_\star$, all the components $(\bar{\mathscr{C}}_{\sss (i)}(\lambda_\star))_{i\in [M]}$ are removed. We refer to the above evolution as $M$-truncated system.
We augment a previously defined notation with a superscript $>M$ to denote the corresponding quantity for the $M$-truncated system.
We assume that the $M$-truncated system and the modified system are coupled in a natural way that at each event time of the modified truncated system, an edge is created in the $M$-truncated system if both the half-edges are selected from the outside of $\cup_{i=1}^M\bar{\mathscr{C}}_{\sss (i)}(\lambda_\star)$. 
Under this coupling,
\begin{equation}\label{c2trunc-large-2}
\sum_{i\notin\mathcal{I}_M}|\bar{\mathscr{C}}_{\sss (i)}(\lambda)|^2 \leq \sum_{i\geq 1} |\bar{\mathscr{C}}_{\sss (i)}^{\sss >M}(\lambda)|^2.
\end{equation}
Now, using Lemma~\ref{c2thm:open-comp}, an application of Theorem~\ref{c2thm:AMC-2D} yields that \linebreak $(\bar{\mathbf{Z}}_n(\lambda))_{n\geq 1}$ is tight in $\mathbb{U}^0_{\shortarrow}$.
Thus the proof now follows.
\end{proof}

\subsection{Proof of Theorem~\ref{c2thm:mul:conv}}
We now have all the ingredients to complete the proof of Theorem~\ref{c2thm:mul:conv}. 
For simplicity in writing, we only give a proof for the case $k=2$ since the proof for general $k$ is identical.
Take $\lambda_\star = \lambda_1$. 
Using Lemma~\ref{c2thm:open-comp}, Theorem~\ref{c2thm:AMC-2D} implies
\begin{equation}\label{c2eq:conv-modi}
(\bar{\mathbf{Z}}_n(\lambda_1),\bar{\mathbf{Z}}_n(\lambda_2))\dto (\bar{\mathbf{Z}}(\lambda_1),\bar{\mathbf{Z}}(\lambda_1,\lambda_2)),
\end{equation}for some random elements $\mathbf{Z}(\lambda_1),\mathbf{Z}(\lambda_1,\lambda_2)$ of $\mathbb{U}^0_{\shortarrow}$.
Now, $\bar{\mathbf{Z}}_n(\lambda_1) = \mathbf{Z}_n(\lambda_1)$.
Moreover, using Lemmas~\ref{c2lem:bad-estimation},~\ref{c2lem:no-bad-large},~and~\ref{c2lem:large-comp-contain} and the facts that both $(\bar{\mathbf{Z}}_n(\lambda_2))_{n\geq 1}$ and $(\mathbf{Z}_n(\lambda_2))_{n\geq 1}$ converge, it follows that (see the argument in Corollary~\ref{c2cor-fixed-lambda})
\begin{equation}
\mathrm{d}_{\sss \mathbb{U}}(\bar{\mathbf{Z}}_n(\lambda_2),\mathbf{Z}_n(\lambda_2))\pto 0.
\end{equation} 
Thus, $(\mathbf{Z}_n(\lambda_1),\mathbf{Z}_n(\lambda_2))$ converge jointly. 
Moreover, the limiting object $\mathbf{Z}(\lambda_1,\lambda_2)$ appearing in \eqref{c2eq:conv-modi} does not depend on $\lambda_1$ by Theorem~\ref{c2thm:spls}. 
Now, using induction, there exists a version of the augmented multiplicative coalescent $\mathbf{AMC} = (\mathrm{AMC}(\lambda))_{\lambda\in\R}$ such that for any $k \geq 1$
\begin{equation}
 (\mathbf{Z}_n(\lambda_1),\dots,\mathbf{Z}_n(\lambda_k))\dto (\mathrm{AMC}(\lambda_1),\dots,\mathrm{AMC}(\lambda_k)).
\end{equation}
Finally, the proof of Theorem~\ref{c2thm:spls} is completed by using Proposition~\ref{c2prop:coupling-whp}.  \qed
\section{Conclusion}
In this chapter, we have shown that, when the third moment of the empirical degree distribution tends to infinity, the critical window for the configuration model is primarily dictated by the vertices of highest degree or hubs.
In fact, the asymptotics of hubs completely specify the scaling limits for the component sizes and the surplus edges. 
The proof techniques in this case is completely different than in Chapter~\ref{chap:thirdmoment}.
Since the increment distribution of the exploration process has infinite third moment, the analysis does not fall under the framework of invariance principles such as Martingale FCLT. 
Moreover, since the limiting process does not have independent increments, general methods for stochastic process convergence are not applicable.
The core of the analysis rests on the fact that the hubs cause jumps in the exploration process, and the contribution due to the low-degree vertices turns out to be asymptotically negligible. 
The proof for $\ell^2_{\shortarrow}$ convergence is also more direct in this case, as compared to the size-biased point processes in Chapter~\ref{c1:sec_l2_tightness}.
The fact that the graph becomes more subcritical only after the removal of the hubs plays a crucial role in the analysis.
As in Chapter~\ref{chap:thirdmoment}, we use Janson's construction to study critical percolation.
Further, the evolution over the critical window is studied for both the component sizes and the surplus edges.
In fact, the scaling limit in Theorem~\ref{c2thm:mul:conv} establishes that there exists a version of augmented multiplicative coalescent with finite-dimensional distributions being described by excursions of a thinned L\'evy process.

%

\begin{subappendices}
\section{Appendix: Path counting}\label{c2sec:appendix-gen-path-counting} In this section,  we derive a generalization of \cite[Lemma 5.1]{J09b} by extending their argument.  
Let $V_n'$ denote  the vertex chosen according to the distribution $F_n$  on $[n]$, independently of the graph. 
Also, let $D_n'$ denote the degree of $V_n'$, $D_n$ denote the degree of a uniformly chosen vertex (independently of the graph) and  $\mathscr{C}(v)$ denote the connected component containing $v$.
\begin{lemma}\label{c2lem:gen-path-count}
 Let $\bld{w} = (w_i)_{i\in [n]}$ be a weight sequence and consider $\mathrm{CM}_n(\bld{d})$ such that $\nu_n<1$. Then,
 \begin{equation}
  \E\bigg[\sum_{i\in \mathscr{C}(V_n')}w_i\bigg]\leq \E\big[w_{V_n'}\big]+\frac{ \expt{D_n'}\E\big[ D_nw_{V_n}\big]}{\expt{D_n}(1-\nu_n)}.
 \end{equation}
\end{lemma}
\begin{proof} Consider all possible paths of length $l$ starting from $V_n'$ and the $w$-value at the end of those paths. 
If we sum over all such paths together with a sum over all possible $l$, then we obtain an upper bound on $\sum_{i\in \mathscr{C}(V_n')}w_i$. Write $\E_v[\cdot]$ for the expectation conditional on $V_n'=v$.
 Thus,
 \begin{equation}
  \begin{split}
   \E_v\bigg[\sum_{i\in \mathscr{C}(V_n')}w_i\bigg]\leq w_v+d_v \sum_{l\geq 1}\sum_{\substack{x_1,\dots,x_l\\x_i\neq x_j, \forall i\neq j}} \frac{\prod_{i=1}^{l-1}d_{x_i}(d_{x_i}-1)d_{x_l}w_{x_l}}{(\ell_n-1)\dots(\ell_n-2l+1)}.
  \end{split}
 \end{equation} 
 Now, using the exactly same arguments as \cite[Lemma 5.1]{J09b}, it follows that
 \begin{equation}
  \E\bigg[\sum_{i\in \mathscr{C}(V_n')}w_i\bigg]\leq \E\big[w_{V_n'}\big] + \frac{ \expt{D_n'}\expt{D_nw_{V_n}}}{\expt{D_n}}\sum_{l\geq 1}\nu_n^{l-1},
 \end{equation}and this completes the proof.
\end{proof}
\section{Appendix: Proof of Lemma~\ref{c2lem:sp-cv-n}}
\label{c2sec_appendix}
\begin{proof}[\nopunct]
The proof is an adaptation of the proof of \cite[Lemma 20]{DHLS15}. 
Let $V_n'$ denote the vertex chosen according to the distribution $F_n$  on $[n]$, independently of the graph and let $D_n'$ denote the degree of $V_n'$. 
Suppose that $\limsup_{n\to\infty}\E[D_n']<\infty$.
We use a generic constant $C$ to denote a positive constant independent of $n,\delta,K$. 
Consider the graph exploration described in Algorithm~\ref{c2algo-expl}, but now we start by choosing vertex $V_n'$ at Stage 0 and declaring all its half-edges active. 
The exploration process is still given by \eqref{c2defn:exploration:process} with $S_n(0)=D_n'$.
Note that $\mathscr{C}(V_n')$ is explored when $\mathbf{S}_n$ hits zero. For $H>0$,  let \begin{equation} \label{c2defn:gamma}
\gamma := \inf \{ l\geq 1: S_n(l)\geq H \text{ or }  S_n(l)= 0 \}\wedge 2\delta_K b_n.
\end{equation} Note that
\begin{equation}\label{c2exploration:super_martingale}
\begin{split}
 \E\bigg[S_n(l+1)-S_n(l)&\vert \left( \mathcal{I}_i^n(l)\right)_{i=1}^n\bigg]= \sum_{i\in [n]}d_i\prob{i\notin \mathscr{V}_l, i\in \mathscr{V}_{l+1}\vert \left( \mathcal{I}_i^n(l)\right)_{i=1}^n} -2\\
 &= \frac{ \sum_{i\notin \mathscr{V}_l}d_i^2}{\ell_n-2l-1}-2\leq \frac{ \sum_{i\in [n]}d_i^2}{\ell_n-2l-1}-2\\:
 & =\lambda c_n^{-1}+o(c_n^{-1})+\frac{2l+1}{\ell_n-2l-1}\times \frac{\sum_{i\in [n]}d_i^2}{\ell_n}   \leq 0
\end{split}
\end{equation} uniformly over $l\leq 2\delta_K b_n$ for all small $\delta >0$ and large $n$, where the last step follows from the fact that $\lambda<0$. Therefore, $\{S_n(l)\}_{l= 1}^{2\delta_Kb_n}$ is a super-martingale. The optional stopping theorem now implies
  \begin{equation}
   \mathbb{E}\left[D_n'\right] \geq \mathbb{E}\left[S_n(\gamma)\right] \geq H \mathbb{P}\left( S_n(\gamma) \geq H \right).
  \end{equation} Thus,
  \begin{equation} \label{c2eqn::bound_geq_H_at_stopping_time}
    \mathbb{P}\left( S_n(\gamma) \geq H \right) \leq \frac{\expt{D_n'}}{H}.
  \end{equation}
Put $H=a_nK^{1.1}/\sqrt{\delta}$. To simplify the writing, we  write $S_n[0,t]\in A$ to denote that $S_n(l)\in A,$ for all $ l\in [0,t]$.  Notice that
 \begin{equation}\label{c2surp:sup:less}\begin{split}
  &\prob{\surp{\mathscr{C}(V_n')}\geq K,|\mathscr{C}(V_n')|\in (\delta_K b_n,2\delta_Kb_n)}\\
  &\leq \prob{S_n(\gamma)\geq H}+\prob{\surp{\mathscr{C}(V_n')}\geq K, S_n[0,2\delta_K b_n]< H, S_n[0,\delta_K b_n]>0}.
  \end{split}
 \end{equation}Define $A(l_1,\dots,l_K)$ to be the event that $\{\text{surpluses occur at times } l_1,\dots,l_K,$ and $S_n[0,2\delta_K b_n]< H, S_n[0,\delta_K b_n]>0\}$. Now,
 \begin{equation}
  \begin{split}
   &\prob{\surp{\mathscr{C}(V_n')}\geq K, S_n[0,2\delta_K b_n]< H, S_n[0,\delta_K b_n]>0}\\
  &\leq \sum_{1\leq l_1<\dots< l_K\leq 2\delta_K b_n} \prob{A(l_1,\dots,l_K)}\\
  &=\sum_{1\leq l_1<\dots<l_K\leq 2\delta_K b_n}\expt{\ind{0<S_n[0,l_K-1]<H, \mathrm{SP}(l_K-1)=K-1}Y},
  \end{split}
 \end{equation}
 where
 \begin{equation}
 \begin{split}
  Y&=\prob{K^{th}\text{ surplus occurs at time }l_K,  S_n[l_K,2\delta_Kb_n]< H, S_n[l_K,\gamma]>0\mid \mathscr{F}_{l_K-1} }\\
  &\leq \frac{CK^{1.1}a_n}{\ell_n\sqrt{\delta}}\leq \frac{CK^{1.1}}{b_n\sqrt{\delta}}.
 \end{split}
 \end{equation}Therefore, using induction, \eqref{c2surp:sup:less} yields
 \begin{equation}\label{c2exploration:bounded:surplus}
 \begin{split}
  &\prob{\surp{\mathscr{C}(V_n')}\geq K, S_n[0,2\delta_K b_n]< H, S_n[0,\delta_K b_n]>0}\\
  &\leq C\bigg( \frac{K^{1.1}}{\sqrt{\delta}b_n}\bigg)^K\frac{(2\delta b_n)^{K-1}}{K^{0.12(K-1)}(K-1)!}\sum_{l_1=1}^{2\delta_K b_n}\prob{|\mathscr{C}(V_n')|\geq l_1}\\
  &\leq C \frac{\delta^{K/2}}{K^{1.1}b_n}  \expt{|\mathscr{C}(V_n')|},
  \end{split}
 \end{equation}where we have used the fact that $\#\{1\leq l_2,\dots,l_K\leq 2\delta b_n\}=(2\delta b_n)^{K-1}/(K-1)!$ and Stirling's approximation for $(K-1)!$ in the last step. Since $\lambda <0$, we can use Lemma~\ref{c2lem:gen-path-count} to conclude that for all sufficiently large $n$
 \begin{equation} \label{c2expectation:random:vert:comp}
  \expt{|\mathscr{C}(V_n)|}\leq Cc_n,
 \end{equation} for some constant $C>0$ and we get the desired bound for \eqref{c2surp:sup:less}.
  The proof of Lemma~\ref{c2lem:sp-cv-n} is now complete.
\end{proof}
\end{subappendices}

%
%
\cleardoublepage

\chapter[Metric space limit]{Metric space limit for critical components in the infinite third-moment regime}
\label{chap:mspace}
{\small \paragraph*{Abstract.}
This chapter establishes general universality principles for random network models whose component sizes in the critical regime lie in the multiplicative coalescent universality class with heavy-tailed degrees resulting in hubs. 
For graphs whose components evolve exactly as a multiplicative coalescent in this regime, scaling limits for the metric structure of maximal components were derived in \cite{BHS15}. 
In this chapter, we derive sufficient uniform asymptotic negligibility conditions for general network models to satisfy in the barely subcritical regime such that, if the evolution of the components can be \emph{approximated} by a multiplicative coalescent as one transitions from the barely subcritical regime through the critical regime, then the maximal components belong to the same universality class as in~\cite{BHS15}. 
As a canonical example, we study critical percolation on configuration models with heavy-tailed degrees. 
Of independent interest,  we derive refined asymptotics for various susceptibility functions and the maximal diameter in the barely subcritical regime. 
 These estimates, coupled with the universality result,  allow us to derive the asymptotic metric structure of the large components through the critical scaling window for percolation.
}

\vspace{.3cm} 
\noindent {\footnotesize Based on the manuscript: 
Shankar Bhamidi, Souvik Dhara, Remco van der Hofstad, Sanchayan Sen; 
\emph{Universality for critical heavy-tailed network models: Metric structure of maximal components} (2017),	arXiv:1703.07145}
\vfill
The aim of this chapter is to understand universality principles for the metric structure of the critical components when the degree distribution satisfies an infinite third-moment condition. 
To describe our results, we start with an analogy. 
In classical limit theorems for sums of independent random variables, there are two major steps: 
 (1) Identifying possible limit laws (e.g. normal distribution, stable laws, etc.), and (2) understanding uniform asymptotic negligibility conditions under which sums of random variables (appropriately re-scaled) converge to the appropriate limit.  
In the context of critical random graphs with degree-exponent $\tau\in (3,4)$, candidate limit law of maximal components with each edge rescaled to have length $n^{-(\tau-3)/(\tau - 1)}$ was established in \cite{BHS15}. 
 In this chapter we establish sufficient uniform asymptotic negligibility (UAN) conditions for a random graph model in the barely subcritical regime which, coupled with appropriate merging dynamics of components as one increases edge density through the critical regime, implies convergence to limits established in~\cite{BHS15}.   
	This is described in Theorem \ref{c3:thm:univesalty}. 
	 As a canonical example, we analyze the critical regime for percolation on the uniform random graph model and the configuration model with a prescribed heavy-tailed degree distribution (see Theorems~\ref{c3:thm:main}~and~\ref{c3:thm:main-simple}).
	 Of independent interest, we obtain refined estimates for various susceptibility functionals and bounds on the diameter of largest connected components in the barely-subcritical regime for the configuration model; these are described in Theorems~\ref{c3:thm:susceptibility} and~\ref{c3:thm:diam-max}. 
%

\paragraph*{Organization of the chapter.}
 Section~\ref{c3:sec:res} describes the canonical random graph model motivating this work and describes associated results. 
A full description of the limit objects and notions of convergence of   metric space valued random variables are deferred to Section~\ref{c3:sec:definitions-full}. 
Section~\ref{c3:sec:discussion} has a detailed discussion of related work and relevance of this work. 
Section~\ref{c3:sec:univ-thm} describes and proves the general universality result. 
Sections~\ref{c3:sec:entrance-bdd-proofs} and \ref{c3:sec:proof-metric-mc}  prove results about the configuration model. 

\section{Main results} 
\label{c3:sec:res}
Owing to technical overhead, the statement of our main universality result is deferred to Section~\ref{c3:sec:univ-thm}. In this section, we present the results about the largest connected components obtained via percolation on the uniform random graph model and the configuration model. 
We defer definitions of the limiting objects as well as notions of convergence of measured metric spaces to Section~\ref{c3:sec:definitions-full}.   
\subsection{Critical percolation on the configuration model: the metric structure}
%
For $p>0$, define metric space
$\ell^p_{\shortarrow}=\{(x_i)_{i\geq 1}: x_i> x_{i+1} , \sum_i x_i^p<\infty\},$  with the metric $d(\bld{x}, \bld{y})= ( \sum_{i} |x_i-y_i|^p)^{1/p}.$
 Fix $\tau\in (3,4)$.  Throughout this chapter we set: 
\begin{equation}\label{c3:eqn:notation-const}
 \alpha= 1/(\tau-1),\qquad \rho=(\tau-2)/(\tau-1),\qquad \eta=(\tau-3)/(\tau-1).
\end{equation}
 \begin{assumption}[Degree sequence]\label{c3:assumption1}
\normalfont  For each $n\geq 1$, let $\bld{d}=\boldsymbol{d}_n=(d_i)_{i\in [n]}$ be a degree sequence. 
We assume the following about $(\boldsymbol{d}_n)_{n\geq 1}$ as $n\to\infty$:
\begin{enumerate}[(i)] 
\item \label{c3:assumption1-1} (\emph{High degree vertices}) For $i\geq 1$, 
$n^{-\alpha}d_i\to \theta_i$, where $\boldsymbol{\theta}=(\theta_i)_{i\geq 1}\in \ell^3_{\shortarrow}\setminus \ell^2_{\shortarrow}$. 
\item \label{c3:assumption1-2} (\emph{Moment assumptions}) 
Let $D_n$ denote the degree of a vertex chosen uniformly at random, independently of $\mathrm{CM}_n(\boldsymbol{d})$. Then, $D_n$ converges in distribution to some discrete random variable $D$, and 
\begin{eq}
 \frac{1}{n}\sum_{i\in [n]}d_i\to &\mu := \E[D], \quad \frac{1}{n}\sum_{i\in [n]}d_i^2 \to \mu_2:=\E[D^2],\\
  &\lim_{K\to\infty}\limsup_{n\to\infty}n^{-3\alpha} \sum_{i=K+1}^{n} d_i^3=0.
\end{eq}
\end{enumerate}
\end{assumption}   
As discussed in Chapter~\ref{chap:introduction}, the component sizes of $\CM$ undergo a phase transition \cite{JL09,MR95} depending on the parameter
\begin{equation}
 \nu_n = \frac{\sum_{i\in [n]}d_i(d_i-1)}{\sum_{i\in [n]}d_i} \to \nu =\frac{\expt{D(D-1)}}{\expt{D}}.
\end{equation} 
Precisely, when $\nu>1$, $\CM$ is super-critical in the sense that there exists a unique \emph{giant} component whp, and when $\nu<1$, all the components have size $\oP(n)$.
  In this chapter, we will always assume that 
\begin{equation}\label{c3:eq:super-crit-start}
 \nu>1, \quad \text{i.e.}\quad  \CM \text{ is supercritical.}
\end{equation}
 The focus of this chapter is to study \emph{critical percolation} on $\CM$. Percolation refers to deleting each edge of a graph independently with probability $1-p$.  
 Let $\CMPN$, and $\mathrm{UM}_n(\bld{d},p_n)$ denote the graphs obtained from percolation with probability $p_n$ on graphs $\mathrm{CM}_n(\boldsymbol{d})$ and $\mathrm{UM}_n(\bld{d})$, respectively. 
 For $p_n\to p$, it was shown in \cite{J09} that the critical point for the phase transition of the component sizes is $p=1/\nu$. 
 The critical window for percolation was studied in Chapters~\ref{chap:thirdmoment} and~\ref{chap:secondmoment} to obtain the asymptotics of the largest component sizes and corresponding surplus edges. 
 In this chapter, we will assume that $\CMP2$ is in the critical window, i.e.,
 \begin{equation}\label{c3:eq:critical-window-defn}
  p_n = p_n(\lambda) = \frac{1}{\nu_n}+\frac{\lambda}{n^{\eta}}+o(n^{-\eta}).
 \end{equation}
Let $\mathscr{C}_{\sss (i)}^p(\lambda)$ denote the $i$-th largest component of $\rCM(\bld{d},p_n(\lambda))$. 
Each component $\mathscr{C}$ can be viewed as a measured metric space with (i) the metric being the graph distance where each edge has length one, (ii) the measure being proportional to the counting measure, i.e., for any $A\subset \mathscr{C}$, the measure of $A$ is given by $\mu_{\sss \mathrm{ct},i}(A) = |A|/|\mathscr{C}|$. 
For a generic measured metric space $M = (M,\mathrm{d},\mu)$ and $a>0$, write $aM$ to denote the measured metric space $(M,a\mathrm{d},\mu)$.  
Write $\sS_*$ for the space of all measured metric spaces equipped with the Gromov weak topology (see Section \ref{c3:sec:defn:GHP-weak}) and let $\sS_*^{\N}$ denote the corresponding product space with the accompanying product topology. For each $n\geq 1$, view $\big( n^{-\eta}\mathscr{C}_{\sss (i)}^p(\lambda) \big)_{i\geq 1} $ as an object in $\sS_*^{\N}$ by appending an infinite sequence of empty metric spaces after enumerating the components in $\mathrm{CM}_n(\bld{d},p_n(\lambda))$. 
The main results for the configuration model are as follows:
\begin{theorem}\label{c3:thm:main}Consider $\mathrm{CM}_n(\bld{d},p_n(\lambda))$ satisfying \textrm{Assumption~\ref{c3:assumption1}}, \eqref{c3:eq:super-crit-start} and \eqref{c3:eq:critical-window-defn} for some $\lambda\in\R$. There exists a sequence of random measured metric spaces $(\mathscr{M}_i(\lambda))_{i\geq 1}$ such that on $\mathscr{S}_*^\N$ as $n\to\infty$
\begin{equation}\label{c3:eq:thm:main}
 \big( n^{-\eta}\mathscr{C}_{\sss (i)}^p(\lambda) \big)_{i\geq 1} \dto  \big(\mathscr{M}_i(\lambda)\big)_{i\geq 1}.
\end{equation} 
\end{theorem}
\begin{theorem}\label{c3:thm:main-simple}
Under \textrm{Assumption~\ref{c3:assumption1}}, \eqref{c3:eq:super-crit-start} and \eqref{c3:eq:critical-window-defn} for some $\lambda\in\R$, the convergence in \eqref{c3:eq:thm:main} also holds for the components of $\mathrm{UM}_n(\bld{d},p_n(\lambda))$, with the identical limiting object.
\end{theorem}

\begin{remark}\normalfont
The limiting objects are precisely described in Section~\ref{c3:sec:descp-limit}. The conclusion of Theorem~\ref{c3:thm:main} holds if the measure $\mu_{\sss ct , i }$ on $\mathscr{C}_{\sss (i)}^p$ is replaced by more general measures, see
Remark~\ref{c3:rem:switch-measure}.
\end{remark}

\subsection{Mesoscopic properties of the critical clusters: barely subcritical regime} 
One of the main ingredients in the proof of Theorem~\ref{c3:thm:main} is a refined analysis of various susceptibility functions in the \emph{barely subcritical} regime which are of independent interest.
In this context we prove general statements about the susceptibility functions applicable not just to percolation on the supercritical configuration model, rather to any barely subcritical configuration model.
Since percolation on a configuration model yields a configuration model \cite{F07,J09}, the above yields susceptibility functions for percolation on a configuration model as a special case.

\begin{assumption}[Barely subcritical degree sequence]
	\label{c3:assumption-w}
\normalfont 
Let $\boldsymbol{d}'=(d_1',\dots,d_n')$ be a degree sequence and let $w_{\sss (\cdot )}:[n]\mapsto\R_+$ be such that
\begin{enumerate}[(i)] 
\item Assumption~\ref{c3:assumption1} holds for $\bld{d}'$ with some $\bld{c}\in \ell^3_{\shortarrow}\setminus \ell^2_{\shortarrow}$, and 
\begin{eq}
 \lim_{n\to\infty}\frac{1}{n}\sum_{i\in [n]}d_i'&=\mu_{d},\quad \lim_{n\to\infty}\frac{1}{n}\sum_{i\in [n]}w_i=\mu_{w},\\
 & \lim_{n\to\infty}\frac{1}{n}\sum_{i\in [n]}d_i'w_i=\mu_{d,w}.
 \end{eq} 
 \item $\max\{ \sum_{i\in [n]} w_i^3, \sum_{i\in [n]}d_i'^2 w_i,\sum_{i\in [n]}d_i' w_i^2 \} = O(n^{3\alpha}). $ \vspace{.1cm}
 \item (\emph{Barely subcritical regime}:) There exists $0<\delta<\eta$ and $\lambda_0>0$ such that
\begin{equation}\label{c3:defn:barely-subcrit}
 \nu_n'=\frac{\sum_{i\in [n]}d_i'(d_i'-1)}{\sum_{i\in [n]}d_i'}=1-\lambda_0 n^{-\delta}+o(n^{-\delta}).
\end{equation}
\end{enumerate}
\end{assumption} 
We will consider a configuration model, where vertex $i$ has degree $d_i'$ and weight $w_i$. 
Let $\mathscr{C}'_{\sss (i)}$ denote the $i$-th largest component of $\rCM_n(\bld{d}')$, $\mathscr{W}_{\sss (i)}=\sum_{k\in\mathscr{C}_{\sss (i)}'}w_k$ and define the weight-based susceptibility functions as
\begin{equation}\label{c3:defn:susc-w}
 s_r^\star = \frac{1}{n}\sum_{i\geq 1}\mathscr{W}_{\sss (i)}^r\ ;\ r\geq 1, \qquad
 s_{pr}^\star = \frac{1}{n}\sum_{i\geq 1} \mathscr{W}_{\sss (i)}|\mathscr{C}_{\sss(i)}'|.
\end{equation}
Also, define the weighted distance-based susceptibility as 
\begin{equation}\label{c3:defn:susc-dist}
\mathcal{D}_n^\star=\frac{1}{n}\sum_{i,j\in [n]}w_iw_j\mathrm{d}(i,j)\ind{i,j \text{ are in the same connected component}},
\end{equation} where $\mathrm{d}$ denotes the graph distance in the component $\mathscr{C}_{\sss (k)}'$ for which $i,j\in \mathscr{C}_{\sss (k)}'$. 
The goal is to show that the quantities defined in \eqref{c3:defn:susc-w} and \eqref{c3:defn:susc-dist} satisfy asymptotic regularity conditions.
 These are summarized in the following theorem:
\begin{theorem}[Susceptibility functions]\label{c3:thm:susceptibility} Under \textrm{Assumption~\ref{c3:assumption-w}}, as $n\to\infty$, 
 \begin{eq}
  n^{-\delta}s_2^\star\pto \frac{\mu_{d,w}^2}{\mu_d\lambda_0}, &\quad n^{-\delta}s_{pr}^\star \pto \frac{\mu_{d,w}}{\lambda_0}, \quad n^{-(\alpha+\delta)}\mathscr{W}_{\sss (j)}\pto \frac{\mu_{d,w}}{\mu_d\lambda_0}c_j, \\
  n^{-3\alpha-3\delta+1}s_3^\star&\pto \bigg(\frac{\mu_{d,w}}{\mu_d\lambda_0}\bigg)^3\sum_{i=1}^{\infty}c_i^3, \quad n^{-2\delta}\mathcal{D}_n^\star\pto \frac{\mu_{d,w}^2}{\mu_d\lambda_0^2}. 
 \end{eq}
\end{theorem}
 For a connected graph $G$, $\Delta(G)$ denotes the diameter of the graph, and for any arbitrary graph~$G$, $\Delta_{\max}(G):=\max \Delta(\mathscr{C})$, where the maximum is taken over all connected components ${\mathscr{C}\subset G}$. 
 We simply write $\Delta_{\max}$ for $\Delta_{\max}(\rCM_n(\bld{d}'))$. 
\begin{theorem}[Maximum diameter]\label{c3:thm:diam-max} Under \textrm{Assumption~\ref{c3:assumption-w}}, as $n\to\infty$, $$\PR(\Delta_{\max}>6n^{\delta}\log(n))\to 0.$$
\end{theorem}

\begin{remark}\normalfont 
Note that $w_i=1$ for all $i\in [n]$ implies that $\mathscr{W}_{\sss (i)}=|\mathscr{C}_{\sss (i)}|$, and thus Theorem~\ref{c3:thm:susceptibility} holds for the usual susceptibility functions defined in terms of the component sizes (cf. \cite{J09b}).  
In the proof of Theorem~\ref{c3:thm:main}, we will require a more general weight function.
\end{remark}

\section{Discussion}\label{c3:sec:discussion}
In this section, we describe related work and discuss the relevance of the results in this chapter. 
\paragraph*{Related work.}
A wide array of universality conjectures have been postulated about functionals of network models. 
Of particular relevance here is \cite{braunstein2003optimal,braunstein2007optimal}, where via simulations the so-called strong disorder regime (which in the extremal case is the minimal spanning tree where edges have i.i.d.~positive random edge weights) was studied.
From the probabilisitic combinatorics community 
both for the universality principle and the results for the configuration model, a major role is played by the \emph{multiplicative coalescent}. 
The process was rigorously constructed in \cite{A97} whilst a complete description of the entrance boundary of this Markov process was laid out in \cite{AL98}. 
In the context of the critical regime for random graphs especially with heavy tails, \emph{component sizes} for the closely related rank-one random graph model were derived in \cite{H13,BHL12}. 
These were then extended to the configuration model in \cite{Jo10} culminating in a complete description of component sizes and surplus edges in \cite{DHLS16}. 
Rigorous results for the metric-space structure of components in the heavy-tailed regime was first derived in \cite{BHS15}; the limiting objects are described in Section~\ref{c3:sec:descp-limit}. 
In this chapter, we develop general conditions under which the metric structure of the critical components are identical to those for critical rank-one inhomogeneous random graphs as derived in~\cite{BHS15}, and hence derive Theorem~\ref{c3:thm:main}.
Similar universality principles for $\tau > 4$ were derived in \cite{BBSX14}. 
\paragraph*{Proof techniques in the barely subcritical regime.}
 A key part of the contributions of this chapter is a refined analysis of the barely subcritical regime for $\CM$ in the heavy-tailed regime;
  for related results see e.g.  \cite{J08,J09b,janson2008susceptibility,janson2009susceptibility} and the references therein. The bounds in this chapter, in particular the extension to the barely subcritical regime are new. 
The proof techniques are also novel, and involve a combination of generalizing path-counting techniques \cite{J09b}, formalizing branching process heuristics, as well as leveraging the differential equation method  \cite{wormald1995differential} to analyze various susceptibility functions in the barely subcritical regime. 
\begin{remark} \normalfont 
In an ongoing work, \cite{CG17} derives the scaling limit of the maximal components at criticality for $\CM$ when the degrees are i.i.d samples from a power-law distribution with $\tau\in (3,4)$, and \cite{GHS17} investigates the properties of these limiting objects, obtained via appropriate tilts of Levy trees \cite{duquesne-legall-levy-tree-book}.  
Interestingly, the description of the limiting objects in the i.i.d setting turns out to be quite different than Theorem~\ref{c3:thm:main}.
It will be interesting to explore the connections between the results in the above papers and the current work. 
\end{remark}

\section{Convergence of metric spaces, discrete structures and limit objects}
\label{c3:sec:definitions-full}

The aim of this section is to define the proper notion of convergence relevant to this chapter (Section~\ref{c3:sec:defn:GHP-weak}),
set-up  discrete structures required in the statement and proof of the universality result in Theorem~\ref{c3:thm:univesalty} (Sections~\ref{c3:defn:super-graph}, \ref{c3:sec:tree-space}, \ref{c3:sec:p-trees}), and describe limit objects that arise in Theorem~\ref{c3:thm:main} (Sections~\ref{c3:sec:inhom-cont-tree} and~\ref{c3:sec:descp-limit}).

\subsection{Gromov-weak topology}
\label{c3:sec:defn:GHP-weak}
A complete separable measured metric space (denoted by $(X, \mathrm{d} , \mu)$) is a complete, separable metric space $(X,\mathrm{d})$ with an associated probability measure $\mu$ on the Borel sigma algebra $\cB(X)$.
The Gromov-weak topology is defined on $\mathscr{S}_0$, the space of 
all complete and separable measured metric spaces  (see \cite[Section 2.1.2]{BHS15}, \cite{GPW09,G07}).
 The notion is formulated based on the philosophy of finite-dimensional convergence.
 Two measured metric spaces $(X_1,\dst_1,\mu_1)$, $(X_2,\dst_2,\mu_2)$ 
 are considered to be equivalent if there exists an isometry  $\psi:\mathrm{support}(\mu_1)\mapsto \mathrm{support}(\mu_2)$ such that $\mu_2=\mu_1 \circ \psi^{-1}$. 
 Let $\mathscr{S}_*$ be the space of all equivalence classes of $\mathscr{S}_0$. 
 We abuse the notation by not distinguishing between a metric space and its corresponding equivalence class. 
 Fix $l\geq 2$, $(X,\dst,\mu)\in \mathscr{S}_*$. 
 Given any collection of points $\mathbf{x}=(x_1,\dots,x_l)\in X^l$, define $\mathrm{D}(\mathbf{x}):=(\dst(x_i,x_j))_{ i,j\in [l]}$ to be the matrix of pairwise distances of the points in $\mathbf{x}$. A function $\Phi: \mathscr{S}_*\mapsto \R$ is called a polynomial if there exists a bounded continuous function $\phi:\R^{l^2}\mapsto \R$ such that 
\begin{equation}\label{c3:defn:polynomials}
 \Phi((X,\dst,\mu))=\int \phi(\mathrm{D}(\mathbf{x}))d\mu^{\sss\otimes l},
\end{equation}
where $\mu^{\sss\otimes l}$ denotes the $l$-fold product measure. A sequence $\{(X_n,\dst_n,\mu_n)\}_{n\geq 1}$ $\subset \mathscr{S}_*$ is said to converge to $(X,\dst,\mu)\in \mathscr{S}_*$ if and only if $\Phi((X_n,\dst_n,\mu_n))\to \Phi((X,\dst,\mu))$ for all polynomials $\Phi$ on~$\mathscr{S}_*$. By \cite[Theorem 1]{GPW09},  $\mathscr{S}_*$ is a Polish space under the Gromov-weak topology.

\subsection{Super graphs}\label{c3:defn:super-graph} 
Our \emph{super graphs} consist of three main ingredients: 1) a collection of metric spaces called \emph{blobs}, 2) a graphical super-structure determining the connections between the blobs, 3) connection points or junction points at each blob. In more detail,  super graphs contain the following structures:
 \begin{enumerate}[(a)]
 \item \label{c3:defn:blob:supstrct:2} \textbf{Blobs}: A collection $\{(M_i,\dst_i,\mu_i)\}_{i\in [m]}$ of  connected, compact measured metric spaces.
  \item \textbf{Superstructure}: A (random) graph $\mathcal{G}$ with vertex set $[m]$. The graph has a weight sequence $\mathbf{p}=(p_i)_{i\in [m]}$ associated to the vertex set $[m]$. We regard $M_i$ as the $i$-th vertex of $\mathcal{G}$. 
  \item \textbf{Junction points}:  An independent collection of random points $\mathbf{X}:=(X_{i,j}:i,j\in [m])$ such that $X_{i,j}\sim \mu_i$ for all $i,j$. Further, $\mathbf{X}$ is independent of~$\mathcal{G}$.
 \end{enumerate} Using these three ingredients, define a metric space $(\bar{M},\bar{\dst}, \bar{\mu})=\Gamma(\mathcal{G},\mathbf{p}, \mathbf{M},\mathbf{X})$, with $\bar{M}= \sqcup_{i\in [m]}M_i$, by putting an edge of \emph{length} one between the pair of points
$\{(X_{i,j},X_{j,i}): (i,j) \text{ is an edge of }\mathcal{G}\}.$
 The distance metric $\bar{\dst}$ is the natural metric obtained from the graph distance and the inter-blob distance on a path. More precisely, for any $x,y\in \bar{M}$ with $x\in M_{j_1}$ and $y\in M_{j_2}$,
 \begin{equation}\label{c3:defn:distance-metric-blob}
  \bar{\mathrm{d}}(x,y)= \inf \Big\{k+\dst_{j_1}(x,X_{j_1,i_1})+\sum_{l=1}^{k-1} \dst_{i_l}(X_{i_l,i_{l-1}},X_{i_{l+1},i_{l}})+\dst_{j_2}(X_{j_2,i_{k-1}},y)\Big\},
 \end{equation}where the infimum is taken over all paths $(i_1,\dots,i_{k-1})$ in $\mathcal{G}$ and all $k\geq 1$ and we interpret $i_0$ and $i_{k}$ as $j_1$ and $j_2$ respectively.
  The measure $\bar{\mu}$ is given by  $\bar{\mu}(A):=\sum_{i\in [m]}p_i\mu_i(A \cap M_i)$, for any measurable subset $A$ of $\bar{M}$. Note that there is a one-to-one correspondence between the components of $\mathcal{G}$ and $\Gamma(\mathcal{G},\mathbf{p},\mathbf{M},\mathbf{X})$ as the blobs are connected. 
  \subsection{Space of trees with edge lengths, leaf weights, root-to-leaf measures, and blobs}\label{c3:sec:tree-space}
In the proof of the main results we need the following spaces built on top of the space of discrete trees. 
The first space $\vT_{\sss IJ}$ was formulated in \cite{AP99,AP00a} where it was used to study trees spanning a finite number of random points sampled from an inhomogeneous continuum random tree (as described in the next section).

\subsubsection{The space $\vT_{\sss IJ}$.} Fix $I\geq 0$ and $J\geq 1$. Let $\vT_{\sss IJ}$ be the space of trees with each element $\vt\in \vT_{\sss IJ}$ having the following properties:
\begin{enumerate}
	\item There are exactly $J$ leaves labeled $1+, \ldots, J+$, and the tree is rooted at the labeled vertex $0+$.
	\item There may be extra labeled vertices (called hubs) with  labels in $\set{1,\ldots, I}$. (It is possible that only some, and not all, labels in $\set{1,\ldots, I}$ are used.)
	\item Every edge $e$ has a strictly positive edge length $l_e$.
\end{enumerate}
A tree $\vt\in \vT_{\sss IJ}$ can be viewed as being composed of two parts: 
(1) $\shape(\vt)$ describing the shape of the tree (including the labels of leaves and hubs) but ignoring edge lengths. The set of all possible shapes $\vT_{\sss IJ}^{\sss \shape}$ is obviously finite for fixed $I, J$.
(2) The edge lengths $\vl(\vt):= (l_e:e\in \vt)$. 
We will consider the product topology on $\vT_{\sss IJ}$ consisting of the discrete topology on $\vT_{\sss IJ}^{\sss \shape}$ and the product topology on $\bR^{\sss \mathrm{E}}$, where $\mathrm{E}$ is the number of edges of $\vt$.

\subsubsection{The space $T_{\sss IJ}^*$.} \label{c3:sec:space-of-trees}
Along with the three attributes above in $\vT_{\sss IJ}$, the trees in this space have the following two additional properties. Let $\cL(\vt):= \set{1+, \ldots, J+}$ denote the collection of leaves in~$\vt$. 
Then every leaf $v \in \cL(\vt) $ has the following attributes:

\begin{enumerate}
	\item[(d)] \textbf{Leaf weights:} A strictly positive number $A(v)$. 
	\item[(e)] \textbf{Root-to-leaf measures:} A probability measure $\nu_{\vt,v}$ on the path $[0+,v]$ connecting the root and the leaf $v$. 
\end{enumerate}
The path $[0+,v]$ for each $v \in \cL(\vt) $, can be viewed as a compact measured metric space with the measure being  $\nu_{\vt,v}$.
Let $\cX$ denote the space of compact measured metric spaces endowed with the Gromov-Hausdorff-Prokhorov topology (see \cite[Section~2.1.1]{BHS15}).
In addition to the topology on $\vT_{\sss IJ}$, the space $\vT_{\sss IJ}^*$ with the additional two attributes inherits the product topology on $\bR^{\sss J}$ due to leaf weights and $\mathcal{X}^{\sss J}$ due to the paths $[0+,v]$ endowed with $\nu_{\vt,v}$ for each $v \in \cL(\vt) $.
For consistency, we add a conventional state $\partial$ to the spaces $\vT_{\sss IJ}$ and $\vT_{\sss IJ}^*$. 
Its use will be made clear in Section~\ref{c3:sec:univ-thm}.  
  
   For all instances in this chapter, the shape of a tree $\mathrm{shape}(\mathbf{t})$ will be viewed as a subgraph of a graph with $m$ vertices. In that case, the tree will be assumed to inherit the vertex labels from the original graph. We will often write $\mathbf{t}\in T^{*m}_{\sss IJ}$ to emphasize the fact that the vertices of $\mathbf{t}$ are labeled from a subset of $[m]$.

 \subsubsection{The space $\overline{T}^{*m}_{\sss IJ}$.}
 \label{c3:sec:T-space-extended} We enrich the space $T^{*m}_{\sss IJ}$ with some additional elements to accommodate the blobs. 
 Consider $\mathbf{t}\in T_{\sss IJ}^{*m}$ and construct $\bar{t}$ as follows:
Let $(M_i,\dst_i,\mu_i)_{i\in [m]}$ be a collection of blobs and $\mathbf{X}=(X_{ij}:i,j\in [m])$ be the collection of junction points as defined in Section~\ref{c3:defn:super-graph}. Construct the metric space $\bar{t}$ with elements in $\bar{M}(\mathbf{t})= \sqcup_{i \in \mathbf{t}}M_i$, by putting an edge of `length' one between the pair of vertices $ \{(X_{i,j},X_{j,i}): (i,j) \text{ is an edge of }\mathbf{t}\}.$ 
The distance metric is given by~\eqref{c3:defn:distance-metric-blob}.
 The path from the leaf $v$ to the root $0+$ now contain blobs. 
 Replace the root-to-leaf measure by $\bar{\nu}_{\vt,v}(A):= \sum_{i\in [0+,v]}\nu_{\vt,v}(i)\mu(M_i\cap A)$ for $A\subset \sqcup_{i\in [0+,v]} M_i$, where $\nu_{\vt,v}$ is the root-to-leaf measure on $[0+,v]$ for $\mathbf{t}$. 
 Notice that $T_{\sss IJ}^{*m}$ can be viewed as a subset of $\overline{T}_{\sss IJ}^{*m}$. 
 In the proof of the universality theorem in Section~\ref{c3:sec:univ-thm}, the blobs will be a fixed collection and, therefore, any $\vt\in T_{\sss IJ}^{*m}$ corresponds to a unique $\bar{\vt}\in\overline{T}_{\sss IJ}^{*m}$. 
  
%
%
%
%

\subsection{\textbf{p}-trees}\label{c3:sec:p-trees}
For fixed $m \geq 1$, write $\bT_m$ and $\bT_m^{\sss \ord}$ for the collection of all rooted trees with vertex set $[m]$ and rooted ordered trees with vertex set $[m]$ respectively. 
An ordered rooted tree is a rooted tree where children of each individual are assigned an order.
We define a random tree model called $\vp$-trees \cite{CP99,P01}, and their corresponding limits, the so-called inhomogeneous continuum random trees, which play a key role in describing the limiting metric spaces.   
Fix $m \geq 1$, and a probability mass function $\vp = (p_i)_{i\in [m]}$ with $p_i > 0$ for all $i\in [m]$.
 A $\vp$-tree is a random tree in $\bT_m$, with law as follows: For any fixed $\vt \in \bT_m$ and $v\in \vt$, write $d_v(\vt)$ for the number of children of $v$ in the tree $\vt$. Then the law of the $\vp$-tree, denoted by $\PR_{\text{tree}}$, is defined as
\begin{equation}
\label{c3:eqn:p-tree-def}
	\PR_{\text{tree}}(\vt) = \PR_{\text{tree}}(\vt; \vp) = \prod_{v\in [m]} p_v^{d_v(\vt)}, \quad \vt \in \bT_m.
\end{equation}
Generating a random $\vp$-tree $\mathscr{T}\sim \PR_{\text{tree}}$ and then assigning a uniform random order on the children of every vertex $v\in \mathscr{T}$ gives a random element with law $\PR_{\ord}(\cdot ; \vp)$ given by
\begin{equation}
\label{c3:eqn:ordered-p-tree-def}
	\PR_{\ord}(\vt) = \PR_{\ord}(\vt; \vp) = \prod_{v\in [m]} \frac{p_v^{d_v(\vt)}}{(d_v(\vt)) !}, \quad \vt \in \bT_m^{\ord}.
\end{equation}

\subsubsection{The birthday construction of \textbf{p}-trees.} \label{c3:sec:p-tree-birthday}
We now describe a construction of $\vp$-trees, formulated in~\cite{CP99}, that is relevant to this work.  
Let $\vY:=(Y_0, Y_1, \ldots)$ be a sequence of i.i.d.~random variables with distribution $\vp$. 
Let $R_0=0$ and for $l\geq 1$, let $R_l$ denote the $l$-th repeat time, i.e., $R_l=\min\big\{k>R_{l-1}: Y_k\in\{Y_0,\hdots,Y_{k-1}\}\big\}.$
Now consider the directed graph formed via the edges
$
\cT(\vY):= \set{(Y_{j-1}, Y_j): Y_j\notin \set{Y_0, \ldots, Y_{j-1}}, j\geq 1}.$
This gives a tree which we view as rooted at $Y_0$. 
The following striking result was shown in \cite{CP99}:
\begin{theorem}[{\cite[Lemma~1 and Theorem~2]{CP99}}]
	\label{c3:thm:pit-cam-birthday}
	The random tree $\cT(\vY)$, viewed as an object in $\bT_m$, is distributed as a $\vp$-tree with distribution \eqref{c3:eqn:p-tree-def} independently of $Y_{R_1-1}, Y_{R_2-1}, \ldots$ which are i.i.d with distribution~$\vp$.
\end{theorem}

 \begin{remark}\normalfont
 	\label{c3:rem:p-tree-dist} The independence between the sequence $Y_{R_1-1}, Y_{R_2-1}, \ldots$ and the constructed $\vp$-tree $\cT(\vY)$ is truly remarkable. 
In particular,  let $\cT_r\subset \cT(\vY)$ denote the subtree with vertex set $\set{Y_0, Y_1, \ldots, Y_{R_r-1}}$, namely the tree constructed in the first $R_r$ steps. 
Further take $\tilde{\mathbf{Y}}=(\tilde{Y_1}, \ldots \tilde{Y_r}) $ an i.i.d.~sample from $\vp$ and then construct the subtree $\mathcal{S}_r$ spanned by $\tilde{\mathbf{Y}}$.
Then the above result (formalized as \cite[Corollary 3]{CP99}) implies that 
%
	\begin{equation}
	\label{c3:eqn:p-tree-joint-sample}
		(\tilde{Y_1}, \tilde{Y_2},\ldots, \tilde{Y_r}; \cS_r)\stackrel{d}{=} (Y_{R_1-1}, Y_{R_2-1}, \ldots Y_{R_r -1}; \cT_r).
	\end{equation}
	We will use this fact in Section~\ref{c3:sec:univ-thm} to complete the proof of the universality theorem.
 \end{remark}
 
\subsubsection{Tilted $\mathbf{p}$-trees and connected components of $\mathrm{NR}_n(\bld{x},t)$.} 
Consider the vertex set $[n]$ and assign weight $x_i$ to vertex $i$. 
Now, connect each pair of vertices $i, j$ ($i\neq j$) independently with probability $
		q_{ij}:= 1-\exp(- tx_i x_j).$
The resulting random graph, denoted by $\mathrm{NR}_n(\bld{x},t)$, is known as the Norros-Reittu model or the Poisson graph process \cite{RGCN1}.
For a connected component $\cC\subseteq \mathrm{NR}_n(\bld{x},t)$, 
let $\mass(\cC):= \sum_{i\in \cC} x_i$ and, for any $t\geq 0$, $(\cC_i(t))_{i\geq 1}$ denotes the components in decreasing order of their mass sizes. 
In this section, we describe results from \cite{BSW14} that gave a method of constructing  connected components of $\mathrm{NR}_n(\bld{x},t)$, conditionally on the vertices of the components. 
This construction involves tilted versions of $\vp$-trees introduced in Section \ref{c3:sec:p-trees}. 
Since these trees are parametrized via a driving probability mass function (pmf) $\vp$, it will be easy to parametrize various random graph constructions in terms of pmfs as opposed to vertex weights~$\bld{x}$. Proposition~\ref{c3:prop:generate-nr-given-partition} will relate vertex weights to pmfs.

Fix $n\geq 1$ and $\cV\subset [n]$, and write 
$\bG_{\cV}^{\con}$ for the space of all simple connected graphs with vertex set~$\cV$.
For fixed $a > 0$, and probability mass function $\vp = (p_v)_{ v \in \cV}$, define probability distributions $\PR_{\con}(\cdot; \vp, a, \cV)$ on $\bG_{\cV}^{\con}$ as follows: For $i,j \in \cV$, denote
\begin{equation}
\label{c3:eqn:qij-def-vp}
	q_{ij}:= 1-\exp(-a p_i p_j).
\end{equation}
Then, for $G \in \bG_{\cV}^{\con},$
\begin{equation}
	\label{c3:eqn:pr-con-vp-a-cV-def}
	\PR_{\con}(G; \vp, a, \cV): = \frac{1}{Z(\vp,a)} \prod_{(i,j)\in E(G)} q_{ij} \prod_{(i,j)\notin E(G)} (1-q_{ij}), 
\end{equation}
where $Z(\vp,a)$ is the normalizing constant.
Now let $\cV^{\sss(i)} $ be the vertex set of $\cC_i(t)$ for $i \geq 1$, and note that $(\cV^{\sss(i)})_{i\geq 1}$ denotes a random finite partition of the vertex set $[n]$. 
The next proposition yields a construction of the random (connected) graphs $(\cC_{i}(t))_{i\geq 1}$:

\begin{proposition}[{\protect{\cite[Proposition 6.1]{BSW14}}}]
	\label{c3:prop:generate-nr-given-partition}
	Given the partition $(\cV^{\sss(i)})_{i\geq 1}$, define,  for $i\geq 1$,
	\begin{equation}\label{c3:eq:p-n-a-NR}
		\vp_n^{\sss(i)} := \left( \frac{x_v}{\sum_{v \in \cV^{\sss(i)}}x_v } : v \in \cV^{\sss(i)} \right), \quad  a_n^{\sss(i)}:= t\bigg(\sum_{v\in \cV_{\sss(i)}} x_v\bigg)^2.
	\end{equation}
	For each fixed $i \geq 1$, let $G_i \in  \bG_{\cV^{\sss(i)}}^{\con}$ be a connected simple graph with vertex set $\cV^{\sss(i)}$. Then
	\begin{equation}
		\PR\left(\cC_i(t) = G_i, \;\; \forall i \geq 1\ \big|\ (\cV^{\sss(i)})_{i\geq 1} \right) = \prod_{i\geq 1} \PR_{\con}( G_i; \vp_n^{\sss(i)}, a_n^{\sss(i)}, \cV^{\sss(i)}).
	\end{equation}
\end{proposition}
\begin{algo}\normalfont
The random graph $\mathrm{NR}_n(\bld{x},t)$ can be generated in two stages:
\begin{enumerate}
	\item[(S0)] Generate the random partition $(\cV^{\sss(i)})_{i\geq 1}$ of the vertices into different components.
\item[(S1)] Conditional on the partition, generate the internal structure of each component following the law of $\PR_{\con}(\cdot ; \vp^{\sss(i)}, a^{\sss(i)}, \cV^{\sss(i)})$, independently across different components.
\end{enumerate}
\end{algo}
\noindent Let us now describe an algorithm to generate such  connected components using the distribution in~\eqref{c3:eqn:pr-con-vp-a-cV-def}. To ease notation, let $\cV = [m]$ for some $m\geq 1$ and fix a probability mass function $\vp$ on $[m]$ and a constant $a>0$ and write $\PR_{\con}(\cdot):= \PR_{\con}(\cdot;\vp,a,[m])$ on $\bG_m^{\con}:= \bG_{[m]}^{\con}$. 
As a matter of convention, we view ordered rooted trees via their planar embedding using the associated ordering to determine the relative locations of siblings of an individual. 
We think of the left-most sibling as the ``oldest''. 
Further, in a depth-first exploration, we explore the tree from left to right. 
Now given a planar rooted tree $\vt\in \bT_m$, let $\rho$ denote the root and for every vertex $v\in [m]$, let $[\rho,v]$ denote the path connecting $\rho$ to $v$ in the tree. Given this path and a vertex $i\in [\rho,v]$, write $\RC(i,[\rho,v])$ for the set of all children of $i$ that fall to the right of $[\rho,v]$. 
Define $\fP(v,\vt):= \cup_{i\in [m]} \RC(i,[\rho,v]).$
	In the terminology of \cite{ABG09,BHS15}, $\fP(v,\vt)$ denotes the set of endpoints of all \emph{permitted edges} emanating from $v$. 
The surplus edges of the graph $G$, sampled from $\PR_{\mathrm{con}}(\cdot)$, are formed only between $v$ and $\fP(v,\vt)$, as $v$ varies.	
	Define
	\begin{equation}
	\label{c3:eqn:amv-def}
		\dA_{\sss(m)}(v):= \sum_{i\in [\rho,v]} \sum_{j\in [m]} p_j \ind{j\in \RC(i,[\rho,v])}.
	\end{equation}
Let $(v(1), v(2), \ldots, v(m))$ denote the order of the vertices in the depth-first exploration of the tree~$\vt$. 
Let $y^*(0)=0$ and $y^*(i) = y^*(i-1) + p_{v(i)}$ and define
	\begin{equation}
	\label{c3:eqn:ant-def-new}
		A_{\sss(m)}(u) = \dA_{\sss(m)}(u),\text{ for } u\in (y^*(i-1), y^*(i)],\quad\text{and}\quad \bar{A}_{\sss(m)}(\cdot):= a A_{\sss(m)}(\cdot),
	\end{equation}where $a$ is defined in \eqref{c3:eqn:qij-def-vp}.
Define the function
		\begin{equation}
		\label{c3:eqn:Lambda-tree}
			\Lambda_{\sss(m)}(\vt) := a\sum_{v\in [m]} p_v \dA_{\sss(m)}(v).
		\end{equation}
Finally, let $ \mathrm{E}(\vt)$ denote the set of edges of $\vt$, $\mathscr{T}_m^{\mathbf{p}}$ the $\vp$-tree defined in \eqref{c3:eqn:ordered-p-tree-def}, $\fP(\vt)=\cup_{v\in [m]}\fP(v,\vt)$, and define the tilt function  $L : \bT_m^{\ord} \to \bR_+$ by
\begin{equation}
\label{c3:eqn:ltpi-def}
	\displaystyle L(\vt)=\displaystyle L_{\sss (m)}(\vt):= \prod_{(k,\ell)\in \mathrm{E}(\vt)} \left[\frac{\exp(a p_k p_{\ell})- 1}{ap_k p_{\ell}} \right] \exp\bigg(\sum_{(k,\ell) \in \sP(\vt)} a p_k p_{\ell}\bigg),
\end{equation} for $\vt \in \bT_m^{\ord}$.
Recall the (ordered) $\vp$-tree distribution from \eqref{c3:eqn:ordered-p-tree-def}. Using $L(\cdot)$ to tilt this distribution results in the distribution
\begin{equation}
	\label{c3:eqn:tilt-ord-dist-def}
	\PR_{\ord}^\star( \vt) := \PR_{\ord}(\vt) \cdot \frac{L(\vt)}{\E_{\ord}[ L(\mathscr{T}^{\vp}_m)]}, \qquad \vt \in \bT_m^{\ord}.
\end{equation}
While all of these objects depend on the tree $\vt$, we suppress this dependence to ease notation. 


\begin{algo}\label{c3:algo:construction-Pcon}\normalfont
Let $\tilde{\mathcal{G}}_m(\vp,a)$ denote a random graph sampled from $\PR_{\mathrm{con}}(\cdot)$. 
This algorithm gives a construction of $\tilde{\mathcal{G}}_m(\vp,a)$, proved in \cite{BHS15}.
\begin{enumerate}
	\item[(S1)] \textbf{Tilted $\vp$-tree:}
	Generate a tilted ordered $\vp$-tree $\mathscr{T}^{\vp,\star}_m$ with distribution \eqref{c3:eqn:tilt-ord-dist-def}. Now consider the (random) objects $\fP(v,\mathscr{T}^{\vp,\star}_m)$ for $v\in [m]$ and the corresponding (random) functions $\dA_{\sss(m)}(\cdot)$ on $[m]$ and $A_{\sss(m)}(\cdot)$ on $[0,1]$.
		\item[(S2)] \textbf{Poisson number of possible surplus edges:} Let $\cP$ denote a rate-one Poisson process on $\bR_+^2$ that is independent of all other randomness and define
	\begin{equation}
	\label{c3:eqn:poisson-pp}
		\bar{A}_{\sss(m)}\cap {\cP}:= \set{(s,t)\in \cP: s\in [0,1], t\leq \bar{A}_{\sss(m)}(s)}.
	\end{equation}
	 Write $\bar{A}_{\sss(m)}\cap {\cP}:= \{(s_j,t_j):1\leq j\leq N_{\sss(m)}^\star\}$ where $N_{\sss(m)}^\star = |\bar{A}_{\sss(m)}\cap {\cP}|$.
We next use the set $\{(s_j, t_j):1\leq j\leq N_{\sss(m)}^\star\}$ to generate pairs of points $\set{(\cL_j,\cR_j): 1\leq j\leq N_{\sss(m)}^\star}$ in the tree that will be joined to form the surplus edges.
	 \item[(S3)] \textbf{``First'' endpoints:} Fix $j$ and suppose $s_j \in (y^*(i-1), y^*(i)]$ for some $i\geq 1$, where $y^*(i)$ is as given right above \eqref{c3:eqn:ant-def-new}. Then the \emph{first endpoint} of the surplus edge corresponding to $(s_j, t_j)$ is $\cL_j:= v(i)$, where $v(i)$ is defined right below \eqref{c3:eqn:amv-def}.
	 \item[(S4)] \textbf{``Second'' endpoints:} Note that in the interval $(y^*(i-1), y^*(i)]$, the function $\bar{A}_{\sss(m)}$ is of constant height $a\dA_{\sss(m)}(v(i))$. We will view this height as being partitioned into sub-intervals of length $a p_u$ for each element $u\in \fP(v(i),\mathscr{T}^{\vp,\star}_m)$, the collection of endpoints of permitted edges emanating from $\cL_k$. (Assume that this partitioning is done according to some preassigned rule, e.g., using the order of the vertices in $\fP(v(i),\mathscr{T}^{\vp,\star}_m)$.) Suppose $t_j$ belongs to the interval corresponding to $u$. Then the \emph{second endpoint} is $\cR_j = u$. Form an edge between $(\cL_j, \cR_j) = (v(i),u)$.
\item[(S5)] In this construction, it is possible that one creates more than one surplus edge between two vertices. Remove any multiple surplus edges. 
This has vanishing probability in our applications.
\end{enumerate}
\end{algo}
\begin{defn}\label{c3:defn:p-tree-graph}
Consider the connected random graph $\tilde{\mathcal{G}}_m(\mathbf{p},a)$, given by \textrm{Algorithm~\ref{c3:algo:construction-Pcon}}, viewed as a measured metric space via the graph distance and each vertex $v$ is assigned measure $p_v$.
\end{defn}
The following lemma describes the law of $\tilde{\mathcal{G}}_m(\mathbf{p},a)$:
\begin{lemma}[{\cite[Lemma 4.10]{BHS15}}]
	\label{c3:lem:lk-rk-equivalent}
	The random graph $\tilde{\cG}_m(\vp,a)$ generated by \textrm{Algorithm~\ref{c3:algo:construction-Pcon}} has the same law as $\PR_{\mathrm{con}}(\cdot)$. 
	Further, conditionally on $\mathscr{T}^{\vp,\star}_m$,
	\begin{enumerate}
		\item $N_{\sss(m)}^\star$ has Poisson distribution with mean $\Lambda_{\sss(m)}(\mathscr{T}_m^{\vp,\star})$ where $\Lambda_{\sss(m)}$ is as in \eqref{c3:eqn:Lambda-tree};
		\item conditionally further on $N_{\sss(m)}^\star=k$, the first endpoints $(\cL_j)_{j\in [k]}$ can be generated in an i.i.d fashion by sampling from the vertex set $[m]$ with probability distribution
		$
		\cJ^{\sss(m)}(v) \propto p_v \dA_{\sss(m)}(v),$ $v\in [m].
		$;
		\item conditionally further on $N_{\sss(m)}^\star=k$ and the first endpoints $(\cL_j)_{j\in [k]}$, generate the second endpoints in an i.i.d. fashion where conditionally on $\cL_j = v$, the probability distribution of $\cR_j$ is given by
	\begin{equation}
	\label{c3:eqn:right-end-pt-prob}
		Q_{v}^{\sss(m)}(y):= \begin{cases}
			\sum_{u} p_u \ind{u\in \RC(y,[\rho,v])}/\dA_{\sss(m)}(v) & \text{ if } y\in [\rho,v],\\
			0 & \text{ otherwise },
		\end{cases}
	\end{equation}	
 and create an  edge between $\cL_j$ and $\cR_j$ for $1\leq j\leq k$.
	\end{enumerate}
\end{lemma}

\label{c3:sec:inhom-cont-tree}
In a series of papers \cite{AMP04,AP99,AP00a} it was shown that $\vp$-trees, under various assumptions, converge to inhomogeneous continuum random trees that we now describe. 
Recall from \cite{LG05,E06} that a real tree is a metric space $(\mathscr{T},d)$ that satisfies the following for every pair $a,b\in \mathscr{T}$:
\begin{enumerate}
	\item There is a \textbf{unique} isometric map $f_{a,b}\colon [0,d(a,b)]\to \mathscr{T}$ such that $f_{a,b}(0)$ $=a,~ f_{a,b}(d(a,b)) =b$.
	\item For any continuous one-to-one map $g:[0,1]\to \mathscr{T}$ with $g(0)=a$ and $g(1)=b$, we have $g([0,1]) = f_{a,b}([0,d(a,b)])$.
\end{enumerate}

\noindent \textbf{Construction of the ICRT:}
Given $\bld{\beta}\in\ell^2_{\shortarrow}\setminus \ell^1_{\shortarrow}$ with $\sum_i\beta_i^2=1$, we will now define the inhomogeneous continuum random tree $\mathscr{T}^{\bld{\beta}}$. We mainly follow the notation in \cite{AP00a}. Assume that we are working on a probability space $(\Omega, \cF,\PR_{\bld{\beta}})$ rich enough to support the following:
\begin{enumerate}
	\item For each $i\geq 1$, let $\cP_i:= (\xi_{i,1}, \xi_{i,2}, \ldots)$ be rate $\beta_i$ Poisson processes that are independent for different~$i$. The first point of each process $\xi_{i,1}$ is special and is called a \emph{joinpoint}, while the remaining points $\xi_{i,j}$ with $j\geq 2$ will be called \emph{$i$-cutpoints} \cite{AP00a}.
	\item Independently of the above, let $\mvU=(U_j^{\sss(i)})_{i,j\geq 1}$ be a collection of i.i.d.~uniform $(0,1)$ random variables. These are not required to construct the tree but will be used to define a certain function on the tree.
\end{enumerate}
  The \textbf{random} real tree (with marked vertices) $\icrt$ is then constructed as follows:
\begin{enumerate}
	\item  Arrange the cutpoints $\set{\xi_{i,j}: i\geq 1, j\geq 2}$ in increasing order as $0< \eta_1 < \eta_2 < \cdots$. The assumption that $\sum_i \beta_i^2 <\infty$ implies that this is possible. For every cutpoint $\eta_k=\xi_{i,j}$, let $\eta_k^*:=\xi_{i,1}$ be the corresponding joinpoint.
	\item Next, build the tree inductively. Start with the branch $[0,\eta_1]$. Inductively assuming that we have completed step $k$, attach the branch $(\eta_k, \eta_{k+1}]$ to the joinpoint $\eta_k^*$ corresponding to $\eta_k$.
\end{enumerate}
Write $\mathscr{T}_0^{\bld{\beta}}$ for the corresponding tree after one has used up all the branches $[0,\eta_1]$, $\set{(\eta_k, \eta_{k+1}]: k\geq 1}$.
Note that for every $i\geq 1$, the joinpoint $\xi_{i,1}$ corresponds to a vertex with infinite degree. Label this vertex $i$. The ICRT $\icrt$ is the completion of the marked metric tree $\mathscr{T}^{\bld{\beta}}_0$. As argued in \cite[Section~2]{AP00a}, this is a real-tree as defined above which can be viewed as rooted at the vertex corresponding to zero. We call the vertex corresponding to joinpoint $\xi_{i,1}$  \textbf{hub} $i$. Since $\sum_i \beta_i = \infty$, one can check that hubs are almost everywhere dense on $\icrt$.

%
%

%
The uniform random variables $(U_j^{\sss(i)})_{i,j\geq 1}$ give rise to a natural ordering on $\icrt$ (or a planar embedding of $\icrt$) as follows: 
For $i\geq 1$, let $(\mathscr{T}_j^{\sss(i)})_{j\geq 1}$ be the collection of subtrees hanging off the $i$-th hub. Associate $U_j^{\sss(i)}$ with the subtree $\mathscr{T}_j^{\sss(i)}$, and think of $\mathscr{T}_{j_1}^{\sss(i)}$ appearing ``to the right of" $\mathscr{T}_{j_2}^{\sss(i)}$ if $U_{j_1}^{\sss(i)}< U_{j_2}^{\sss(i)}$. This is the natural ordering on $\icrt$ when it is being viewed as a limit of ordered $\vp$-trees. We can think of the pair $(\icrt, \mvU)$ as the \textbf{ordered ICRT}.

\subsection{Continuum limits of components}
\label{c3:sec:descp-limit}
The aim of this section is to give an explicit description of the limiting (random) metric spaces in Theorem \ref{c3:thm:main}.  We start by constructing a specific metric space using the tilted version of the ICRT in Section~\ref{c3:sec:tilt-icrt}. 
Then we describe the limits of maximal components in Section~\ref{c3:sec:limit-component}.

\subsubsection{Tilted ICRTs and vertex identification}
\label{c3:sec:tilt-icrt}
Let $(\Omega, \cF, \PR_{\beta})$ and $\icrt$ be as in Section \ref{c3:sec:inhom-cont-tree}. 
In \cite{AP00a}, it was shown that one can associate a natural probability measure $\mu$, called the \textbf{mass measure}, to $\icrt$, satisfying $\mu(\mathfrak{L}(\icrt))=1$. 
Here we recall that $\mathfrak{L}(\cdot)$ denotes the set of leaves. 
Before moving to the desired construction of the random metric space, we will need to define some more quantities that describes the asymptotic analogues of the quantities appearing in Algorithm~\ref{c3:algo:construction-Pcon}.
Similarly to \eqref{c3:eqn:amv-def}, define
\begin{equation}
\label{c3:eqn:limit-da}
\dA_{\sss(\infty)}(y)=\sum_{i\geq 1}\beta_{i}\bigg(\sum_{j\geq 1}U_j^{\sss(i)}\times\mathbf{1}\{y\in \mathscr{T}_j^{\sss(i)} \}\bigg).\end{equation}
It was shown in \cite{BHS15} that $\dA_{\sss(\infty)}(y)$ is finite for almost every realization of $\icrt$ and for $\mu$-almost every $y\in\icrt$. For $y\in \icrt$, let $[\rho,y]$ denote the path from the root $\rho$ to $y$. For every $y$, define a probability measure on $[\rho,y]$ as
\begin{equation}
\label{c3:eqn:right-end-prob-inft}
Q_{y}^{\sss(\infty)}(v):= \frac{\beta_i U_{j}^{\sss(i)}}{ \dA_{\sss(\infty)}(y)}, \quad\mbox{ if } v\mbox{ is the }i\mbox{-th hub and } y\in \mathscr{T}_j^{\sss(i)}\mbox{ for some }j. 	
\end{equation}
Thus, this probability measure is concentrated on the hubs on the path from $y$ to the root.
Let $\gamma >0$ be a constant.
The choice of the function $\gamma$ is indicated in Assumption~\ref{c3:assm:BHS15}. 
Informally, the construction goes as follows:
We will first tilt the distribution of the original ICRT $\icrt$ using the exponential functional
	\begin{equation}
	\label{c3:eqn:ltheta-def}
		L_{\sss(\infty)}(\icrt, \mvU):= \exp\bigg(\gamma\int_{y\in \icrt} \dA_{\sss(\infty)}(y)\mu(dy)\bigg)
	\end{equation}
	to get a tilted tree $\tilicrt$. 
	We then generate a random but finite number $N_{\sss(\infty)}^\star$ of pairs of points $\{(x_k, y_k):1\leq k\leq N_{\sss(\infty)}^\star\}$ that will provide the surplus edges. 
	The final metric space is obtained by creating ``shortcuts" by identifying the points $x_k$ and $y_k$. Formally the construction proceeds in four steps:
\begin{enumerate}
	\item \textbf{Tilted ICRT:}  
Define ${\PR}_{\beta}^\star$ on $\Omega$ by
\begin{equation}
\frac{d {{\PR}}_{\beta}^\star}{d{{\PR}}_{\beta}}=\frac{\exp\big(\gamma\int_{y\in\mathscr{T}^{\bld{\beta}}_{(\infty)}}\dA_{(\infty)}(y)\mu(dy) \big)}{\E\Big[\exp\big(\gamma\int_{x\in \mathscr{T}^{\bld{\beta}}_{(\infty)}} \dA_{\sss(\infty)}(x)\mu(dx) \big)\Big]}. 
\end{equation}
The expectation in the denominator is with respect to the original measure ${\PR}_{\beta}$. 
Write $(\tilicrt, \mu^\star)$ and $\mvU^{\star}=(U_j^{(i), \star})_{i,j\geq 1}$ for the tree and the mass measure on it, and the associated random variables under this change of measure.
 \item \textbf{Poisson number of identification points:} Conditionally on the object $((\tilicrt, \mu^\star), \mvU^{\star})$, generate $N_{\sss(\infty)}^\star$ having a $\mathrm{Poisson}(\Lambda_{\sss(\infty)}^\star)$ distribution, where
\begin{equation}
\Lambda_{\sss(\infty)}^\star:= \gamma\int_{y\in \tilicrt}\dA_{\sss(\infty)}(y)\mu^\star(dy)
=\gamma\sum_{i\geq 1}\beta_{i}\bigg[\sum_{j\geq 1}U_j^{(i), \star}\mu^\star(\mathscr{T}_j^{\sss(i), \star})\bigg].
\end{equation}
Here, $(\mathscr{T}_j^{\sss(i), \star})_{j\geq 1}$ denotes the collection of subtrees of hub $i$ in $\tilicrt$. 

\item \textbf{``First" endpoints (of shortcuts): } Conditionally on (a) and (b), sample $x_k$ from $\tilicrt$  with density proportional to $\dA_{\sss(\infty)}(x)\mu^\star(dx)$ for $1\leq k\leq N_{\sss(\infty)}^\star$.
\item \textbf{``Second'' endpoints (of shortcuts) and identification:} Having chosen $x_k$, choose $y_k$ from the path $[\rho, x_k]$ joining the root $\rho$ and $x_k$ according to the probability measure $Q_{x_k}^{\sss (\infty)}$ as in \eqref{c3:eqn:right-end-prob-inft} but with $U_j^{(i),\star}$ replacing $U_j^{\sss (i)}$.
(Note that $y_k$ is always a \textbf{hub} on $[\rho, x_k]$.)  Identify $x_k$ and $y_k$, i.e., form the quotient space by introducing the equivalence relation $x_k\sim y_k$ for $1\leq k\leq N_{\sss(\infty)}^\star$.
\end{enumerate}
\begin{defn}\label{c3:def:limiting-space} Fix $\gamma\geq 0$ and $\bld{\beta}\in\ell^2_{\shortarrow}\setminus \ell^1_{\shortarrow}$ with $\sum_i\beta_i^2=1$. Let $\cG_{\sss (\infty)}(\bld{\beta},\gamma)$ be the metric measure space constructed via the four steps above equipped with the measure inherited from the mass measure on $\tilicrt$.	
\end{defn}
\subsubsection{Scaling limit for the component sizes and surplus edges}
\label{c3:sec:comp-size-scaling}
Let us describe the scaling limit results for the component sizes and the surplus edges ($\#\text{edges}-\#\text{vertces}+1$) for the largest components of $\mathrm{CM}_n(\bld{d},p_n(\lambda))$ from Chapter~\ref{chap:secondmoment}. 
Although we need to define the limiting object only for describing the limiting metric space, the convergence result will turn out to be crucial in Section~\ref{c3:sec:proof-metric-mc} in the proof of Theorem~\ref{c3:thm:main}, and therefore we state it here as well.
Consider a decreasing sequence $ \boldsymbol{\theta}\in \ell^3_{\shortarrow}\setminus \ell^2_{\shortarrow}$. 
Denote by  $\mathcal{I}_i(s):=\ind{\zeta_i\leq s }$ where $\zeta_i\sim \mathrm{Exp}(\theta_i)$ independently, and $\mathrm{Exp}(r)$ denotes the exponential distribution with rate $r$.  
 Consider the process 
 \begin{equation}\label{c3:defn::limiting::process}
\bar{S}^\lambda_\infty(t) =  \sum_{i=1}^{\infty} \theta_i\left(\mathcal{I}_i(t)- \theta_it\right)+\lambda t,
\end{equation}for some $\lambda\in\mathbb{R}$. 
Define the reflected version of $\bar{S}_\infty^{\lambda}(t)$ by
$ \mathrm{refl}\big( \bar{S}_\infty^{\lambda}(t)\big)= \bar{S}_\infty^{\lambda}(t) - \inf_{0 \leq u \leq t} \bar{S}_\infty^{\lambda}(u).$
The processes of the form \eqref{c3:defn::limiting::process} were termed \emph{thinned} L\'evy processes in \cite{BHL12} since the summands are thinned versions of Poisson processes.
Let $(\Xi_i(\bld{\theta},\lambda))_{i \geq 1}$, $(\xi_i(\bld{\theta},\lambda))_{i\geq 1}$, respectively, denote the vector of excursions and excursion-lengths, ordered according to the excursion lengths in a decreasing manner.
Denote the vector $(\xi_i(\bld{\theta},\lambda))_{i\geq 1}$ by $\bld{\xi}(\bld{\theta},\lambda)$.
The fact that $\bld{\xi}(\bld{\theta},\lambda)$ is always well defined follows from \cite[Lemma 1]{AL98}.
Also, define the counting process of marks $\mathbf{N}$ to be a Poisson process that has intensity $\refl{ \bar{S}_\infty^{\lambda}(t)}$ at time $t$ conditional on $( \refl{ \bar{S}_\infty^{\lambda}(u)} )_{u \leq t}$. 
We use  the notation $\mathscr{N}_i(\bld{\theta},\lambda)$ to denote the number of marks within the i-th largest excursion $\Xi_i(\bld{\theta},\lambda)$.

For a connected graph $G$, let $\mathrm{SP}(G) = \#\text{edges}-\#\text{vertices}+1$ denote its surplus edges. 
In the context of this chapter, we simply write $\xi_i$, $\bld{\xi}$ and $\mathscr{N}_i$ respectively for $\xi_i(\bld{\theta}/(\mu\nu),\lambda/\mu)$, $\bld{\xi}(\bld{\theta}/(\mu\nu),\lambda/\mu)$ and $\mathscr{N}_i(\bld{\theta}/(\mu\nu),\lambda/\mu)$.
\begin{proposition}\label{c3:prop:comp-size}
Under \textrm{Assumption~\ref{c3:assumption1}}, as $n\to\infty$,
\begin{equation}\label{c3:eq:limit-objetct-comp-size}
\big(n^{-\rho}|\mathscr{C}_{\sss (i)}^p(\lambda)|,\surp{\mathscr{C}_{\sss (i)}^p(\lambda)}\big)_{i\geq 1} \dto \Big(\frac{1}{\nu}\xi_i,\mathscr{N}_i\Big)_{i\geq 1},
 \end{equation}with respect to the topology on the product space $\ell^2_\shortarrow\times \N^\N$.
\end{proposition}
Proposition~\ref{c3:prop:comp-size} was proved in Chapter~\ref{chap:secondmoment}.
The limiting object in Theorem~\ref{c2thm::conv:component:size} is stated in a slightly different form compared to the right-hand side of \eqref{c3:eq:limit-objetct-comp-size}. 
However, the limiting objects are identical in distribution with suitable rescaling of time and space, and  by observing that $r\mathrm{Exp}(r) \stackrel{\sss d}{=} \mathrm{Exp}(1)$, where $\mathrm{Exp}(r)$ denotes an exponential random variable with rate $r$ (See Appendix~\ref{c3:sec:appendix-rescaling}). In fact, the arguments in Appendix~\ref{c3:sec:appendix-rescaling} establish the following lemma that will be used extensively in Section~\ref{c3:sec:proof-metric-mc}.
\begin{lemma} \label{c3:lem:rescale}  For $\eta_1,\eta_2>0$, $\bld{\theta}\in \ell^3_{\shortarrow}\setminus \ell^2_{\shortarrow}$ and $\lambda\in\R$,
$ \bld{\xi}(\eta_1\bld{\theta},\eta_2\lambda)\eqd \frac{1}{\eta_1}\bld{\xi}\big(\bld{\theta},\frac{\eta_2}{\eta_1^2}\lambda\big).
$
\end{lemma}

\subsubsection{Limiting component structures}
\label{c3:sec:limit-component}
We are now all set to describe the metric space $M_i$ appearing in Theorem~\ref{c3:thm:main}.
Recall the graph $\mathcal{G}_{\infty}(\bld{\beta},\gamma)$ from Definition~\ref{c3:def:limiting-space}.
Using the notation of Section~\ref{c3:sec:comp-size-scaling}, write $\xi_i^*$ for $\xi_i((\mu(\nu-1))^{-1}\bld{\theta}, (\mu(\nu-1)^2)^{-1}\nu^2\lambda)$ and $\Xi_i^*$ for the excursion corresponding to $\xi_i^*$.
Note that $\xi_i^*$ has the same distribution as $(\nu-1) \xi_i/\nu$, where $\xi_i$ is as in Proposition~\ref{c3:prop:comp-size}. 
Then the limiting space $M_i$ is distributed as 
\begin{equation}
 M_i \eqd \frac{\nu}{\nu-1}\frac{\xi_i^*}{\big(\sum_{v\in \Xi_i^*}\theta_v^2\big)^{1/2}} \mathcal{G}_\infty ( \bld{\theta}^{\sss (i)}, \gamma^{\sss (i)}  ),
\end{equation}where $\bld{\theta}^{\sss (i)} = \big(\frac{\theta_j}{\sum_{v\in \Xi_i^*}\theta_v^2}: j\in \Xi_i^*\big)$ and $\gamma^{\sss (i)}=\frac{\xi_i^*}{\mu(\nu-1)}\big(\sum_{v\in \Xi_i^*}\theta_v^2\big)^{1/2}$.

\section{Universality theorem} \label{c3:sec:univ-thm}
In this section, we develop universality principles that enable us to derive the scaling limits of the components for graphs that can be compared with the critical rank-one inhomogeneous random graph in a suitable sense. 
Our universality theorem closely resembles that in \cite[Theorem~6.4]{BBSX14} which was developed in a different context to derive the scaling limits of the components for general inhomogeneous random graphs with a finite number of types and the configuration model with an exponential moment condition on the degrees. 
We first state the relevant result from~\cite{BHS15} that was used in the context of rank-one inhomogeneous random graphs and then state our main result below. 
The convergence of metric spaces is with respect to the Gromov-weak topology, unless stated otherwise.
Recall the measured metric spaces $\tilde{\mathcal{G}}_m(\vp,a)$ and $\mathcal{G}_{\infty}(\bld{\beta},\gamma)$ defined in \textrm{Definitions~\ref{c3:defn:p-tree-graph} and~\ref{c3:def:limiting-space}}.
\begin{assumption} \label{c3:assm:BHS15}  \normalfont \begin{enumerate}[(i)]
 \item  Let $\sigmap := \big(\sum_i p_i^2\big)^{1/2}$. As $m\to\infty$, $\sigmap\to 0 $, and for each fixed $i\geq 1$, $p_i/\sigmap \to \beta_i$, where $\bld{\beta}=(\beta_i)_{i\geq 1}\in\ell^2_{\shortarrow}\setminus\ell^1_{\shortarrow}$, $\sum_i\beta_i^2=1$.
 \item Recall $a$ from \eqref{c3:eqn:qij-def-vp}. There exists a constant $\gamma >0$ such that $a\sigmap\to\gamma$. 
\end{enumerate}
\end{assumption}
Assumption~\ref{c3:assm:BHS15}~(i) is a sufficient condition for the convergence of $\vp$-trees \cite{CP99} when the edges are assigned edge-length $\sigmap$.
Assumption~\ref{c3:assm:BHS15}~(ii) is required for the tilting function $L(\cdot)$ to converge.
This suggests that the tilted $\vp$-tree in Algorithm~\ref{c3:algo:construction-Pcon}~(S1) converges to tilted ICRT.
In \cite{BHS15}, the above fact was proved, and it was further shown that the scaling limit holds after the shortcuts are created during Algorithm~\ref{c3:algo:construction-Pcon}~(S2) onwards.

 \begin{theorem}[{\cite[Theorem 4.5]{BHS15}}] \label{c3:thm:BHS15}
  Under \textrm{Assumption~\ref{c3:assm:BHS15}},
$ \sigmap \tilde{\mathcal{G}}_m(\mathbf{p},a) \dto \mathcal{G}_{\sss (\infty)}(\bld{\beta},\gamma)$, as $m\to\infty$.
\end{theorem}

For each $m\geq 1$, fix a collection of blobs $\mathbf{M}_m:= \{(M_i, \dst_i, \mu_i):{i\in [m]} \}$. Recall the definition of super graphs from Section~\ref{c3:defn:super-graph} and denote 
\begin{equation}\label{c3:defn:blob-Gm}
\tilde{\mathcal{G}}_m^{\sss \mathrm{bl}}(\mathbf{p},a) = \Gamma(\tilde{\mathcal{G}}_m(\mathbf{p},a),\mathbf{p},\mathbf{M}_m,\mathbf{X}),
\end{equation}where $\mathbf{X} = (X_{ij})_{i,j\in [m]}$, $X_{ij}\sim \mu_i$ independently for each $i$. 
Moreover, $\mathbf{X}$ is independent of the graph $\tilde{\mathcal{G}}_m(\mathbf{p},a)$.
Let $u_i:=\E[\dst_i(X_i,X_i')]$ where $X_i,X_i'\sim \mu_i$ independently and $B_m:= \sum_{i\in [m]}p_iu_i$. 
Let $\Delta_i:=\diam(M_i) $, $\Delta_{\max}:=\max_{i\in [m]}\Delta_i$.

\begin{assumption}[Maximum inter-blob-distance] \label{c3:assm:blob-diameter}
 \normalfont 
 $\lim_{m\to\infty}\frac{\sigma(\mathbf{p})\Delta_{\max}}{B_m+1}=0.$ 
\end{assumption}
Assumption~\ref{c3:assm:blob-diameter} basically says that the blobs have negligible diameter compared to the average distances in the metric space.
The next theorem is the universality theorem, which basically says that the introduction of the blobs does not change the scaling limits in \cite[Theorem 4.5]{BHS15} if the distances are normalized accordingly. 
\begin{theorem}[Universality theorem]\label{c3:thm:univesalty}
Under \textrm{Assumptions~\ref{c3:assm:BHS15}},~\textrm{\ref{c3:assm:blob-diameter}}, 
as $m\to\infty$,
\begin{equation}
 \frac{\sigma(\mathbf{p})}{B_m+1}  \tilde{\mathcal{G}}_m^{\sss \mathrm{bl}}(\mathbf{p}, a) \dto \mathcal{G}_{\sss (\infty)}(\bld{\beta},\gamma).
 \end{equation}
\end{theorem} 

\subsection{Completing the proof of Theorem~\ref{c3:thm:univesalty}}
This section is devoted to the proof of Theorem~\ref{c3:thm:univesalty}. 
To simplify notation, we write $\tilde{\mathcal{G}}_m$, $\tilde{\mathcal{G}}_m^{\sss \mathrm{bl}}$ respectively instead of $\tilde{\mathcal{G}}_m(\mathbf{p}, a)$ and $\tilde{\mathcal{G}}_m^{\sss \mathrm{bl}}(\mathbf{p}, a)$.
\begin{lemma}[{\cite[Lemma 4.11]{BHS15}}]\label{c3:lem:spls-tight-p} Recall the definition of $N_{\sss(m)}^\star$ from \textrm{Algorithm~\ref{c3:algo:construction-Pcon}}.
The sequence of random variables $(N_{\sss(m)}^\star)_{m\geq 1}$ is tight. 
\end{lemma}
Recall the definition of Gromov-weak topology from Section~\ref{c3:sec:defn:GHP-weak}. 
Fix some $l\geq 1$ and take any bounded continuous function $\phi: \R^{\sss l^2}\mapsto\R$. 
We simply write $\Phi(X)$ for $\Phi((X,\dst,\mu))$. 
\paragraph*{\textbf{Key step 1.}}
Let us write the scaled metric spaces as $\tilde{\cG}_m^{\sss \mathrm{s}} = \sigmap \tilde{\mathcal{G}}_m$ and $\tilde{\cG}_m^{\sss \mathrm{bl},\mathrm{s}} = \frac{\sigmap}{B_m+1} \tilde{\mathcal{G}}_m^{\sss \mathrm{bl}}$.
Using Theorem~\ref{c3:thm:BHS15} it is enough to show that
\begin{equation}\label{c3:eq:uni-th:red1}
 \lim_{m\to\infty} \big|\E\big[\Phi\big(\tilde{\cG}_m^{\sss\mathrm{bl}, \mathrm{s}}\big)\big] - \E\big[\Phi(\tilde{\cG}_m^{\sss \mathrm{s}})\big] \big|=0.
\end{equation}
The above is the main step in the proof of Theorem~\ref{c3:thm:univesalty}. 
By \eqref{c3:eq:uni-th:red1}, we need only to compare the structure of $\tilde{\mathcal{G}}_m^{\sss\mathrm{bl}}$ to that of $\tilde{\mathcal{G}}_m$.
\paragraph*{\textbf{Key step 2.}}For any $K\geq 1$, 
$$ \bigg|\E\big[\Phi(\tilde{\cG}_m^{\sss \mathrm{s}})\big]- \sum_{k=0}^{K}\E\big[\Phi(\tilde{\cG}_m^{\sss \mathrm{s}})\1\{N_{\sss(m)}^\star = k\}\big]\bigg| \leq \|\phi\|_{\infty}\prob{N_{\sss(m)}^\star \geq K+1},$$ and the same inequality also holds for $\tilde{\cG}_m^{\sss\mathrm{bl}, \mathrm{s}}$.
Thus, using Lemma~\ref{c3:lem:spls-tight-p}, 
the proof of \eqref{c3:eq:uni-th:red1} reduces to showing that, for each fixed $k\geq 1$, 
\begin{equation}\label{c3:eq:blobvsG-reduction1}
 \lim_{m\to\infty}\Big|\E\big[\Phi\big(\tilde{\cG}_m^{\sss \mathrm{bl},\mathrm{s}}\big)\1\{N_{\sss(m)}^\star = k\}\big]-\E\big[\Phi(\tilde{\cG}_m^{\sss \mathrm{s}})\1\{N_{\sss(m)}^\star = k\}\big]\Big|=0.
\end{equation}
\paragraph*{\textbf{Main aim of this section.}} 
Below, we define a function $\mathsf{g}_{\phi}^k(\cdot)$ on 
the space $\overline{T}_{IJ}^*$ which captures the behavior of pairwise distances after creating $k$ surplus edges. 
Under Assumption~\ref{c3:assm:blob-diameter}, we show that the introduction of blobs changes the distances within the tilted $\vp$-trees and the $\mathsf{g}_\phi^k$ values negligibly. This completes the proof of \eqref{c3:eq:blobvsG-reduction1}. 

For any fixed $k\geq 0$, consider $\mathbf{t}\in T^*_{\sss I,(k+l)}$ with root $0+$, leaves $\bld{i}=(1+,\dots,$ $(k+l)+)$ and root-to-leaf measures $\nu_{\vt,i}$ on the path $[0+,i+]$ for all $1\leq i\leq k+l$. 
We create a graph $G(\mathbf{t})$ by sampling, for each $1\leq i\leq k$, points $i(s)$ on $[0+,i+]$ according $\nu_{\vt,i}$ and connecting $i+$ with $i(s)$. 
Let $\dst_{\sss G(\mathbf{t})}$ denote the distance on $G(\mathbf{t})$ given by the sum of edge lengths in the shortest path. 
Then, the function $\mathsf{g}_{\phi}^k:T_{\sss IJ}^*\mapsto\R$ is defined as

\begin{subequations}
\begin{equation}\label{c3:defn:g-phi-1}
 \mathsf{g}_{\phi}^k(\mathbf{t}) = \E\big[\phi\big(\dst_{\sss G(\mathbf{t})}(i+,j+):k+1\leq i,j\leq k+l\big)\big]\ind{\mathbf{t}\neq \partial},
\end{equation}where $\partial$ is a forbidden state defined as follows: 
Given any $\mathbf{t}\in T_{\sss IJ}^*$,  and a set of vertices $\bld{v}=(v_1,\dots,v_r)$, we denote the subtree of $\mathbf{t}$ spanned by $\bld{v}$ with $\mathbf{t}(\bld{v})$. We declare $\mathbf{t}(\bld{v})=\partial$ if either two vertices in $\bld{v}$ are the same or one of them is an ancestor of another vertex in $\bld{v}$. Thus, if $\mathbf{t}(\bld{v})\neq \partial$, the tree $\mathbf{t}(\bld{v})$ necessarily has $r$ leaves.
Notice that the expectation in \eqref{c3:defn:g-phi-1} is over the choices of $i(s)$-values only.
In our context, $\vt$ is always considered as a subgraph of the graph on vertex set $[m]$ and thus we assume that $\vt$ has inherited the labels from the corresponding graph.
Thus $\vt\in T^{*m}_{\sss I,(k+l)}$.
There is a natural way to extend $\mathsf{g}_{\phi}^k(\cdot)$ to $\overline{T}_{\sss IJ}^{*m}$ as follows:
Consider $\bar{\mathbf{t}}\in \overline{T}_{\sss IJ}^{*m}$ and the corresponding $\mathbf{t}\in T_{\sss IJ}^{*m}$ (see Section~\ref{c3:sec:T-space-extended}). 
Let $0+$, $\bld{i}$, $(\nu_{\vt,i})_{i\in [k+l]}$  and $(i(s))_{i\in [k+l]}$ be as defined above. 
Let $\bar{G}(\bar{\mathbf{t}})$ denote the metric space by introducing an edge of length one between $X_{i+i(s)}$ and $X_{i(s)i+}$, where $X_{ij}$ has distribution  $\mu_i$ for all $j\geq 1$, independently of each other and other shortcuts. 
For $k+1\leq i\leq k+l$, $X_i\in M_{x_i}$ have distribution $\mu_{x_i}$ independently for all $i\geq 1$. 
Let $\bar{\dst}_{\sss \bar{G}(\mathbf{\bar{t}})}$ denote the distance on $\bar{G}(\mathbf{\bar{t}})$. Then, let
\begin{equation}\label{c3:defn:g-phi-2}
 \mathsf{g}_{\phi}^k(\bar{\mathbf{t}}) = \E\big[\phi\big(\bar{\dst}_{\sss \bar{G}(\mathbf{\bar{t}})}(X_i,X_j):k+1\leq i,j\leq k+l\big)\big]\ind{\mathbf{t}\neq \partial},
\end{equation}
\end{subequations}where the expectation is taken over the collection of random variables $X_{i+i(s)}$ and $X_{i(s)i+}$.  
At this moment, we urge the reader to recall the construction in Algorithm~\ref{c3:algo:construction-Pcon}, Lemma~\ref{c3:lem:lk-rk-equivalent} and all the associated notations.
Now, conditional on $\mathscr{T}_m^{ \mathbf{p},\star}$, we can construct the tree $\mathscr{T}_m^{ \mathbf{p},\star}(\tilde{\mathbf{V}}_m^{k,k+l})$ where 
\begin{enumerate}[(a)]
 \item $\tilde{\mathbf{V}}_m^{k,k+l}= (\tilde{V}_1^m,\dots,\tilde{V}_k^m, V_{k+1}^m,\dots , V_{k+l}^m)$ is an independent collection of vertices;
 \item $\tilde{V}_i^m$ is distributed as $\cJ^{\sss (m)}(\cdot)$, for $1\leq i\leq k$ and $V_i^m$ is distributed as $\mathbf{p}$, for $k+1 \leq i \leq k+l$.
\end{enumerate} Note that, by \cite[(4.30)]{BHS15}, $ \lim_{m\to\infty} \PR(\mathscr{T}_m^{ \mathbf{p},\star}(\tilde{\mathbf{V}}_m^{k,k+l})= \partial)=0.$
Now, whenever $\mathscr{T}_m^{ \mathbf{p},\star}(\tilde{\mathbf{V}}_m^{k,k+l})\neq \partial$, $\mathscr{T}_m^{ \mathbf{p},\star}(\tilde{\mathbf{V}}_m^{k,k+l})$ can be considered as an element of $T^{*m}_{\sss I,k+l}$  using the leaf-weights $(\mathfrak{G}_{\sss (m)}(\tilde{V}_i))_{i=1}^k$, $(\mathfrak{G}_{\sss (m)}(V_i))_{i=k+1}^{k+l}$ and root-to-leaf measures given by $(Q_{\tilde{V}_i}^m(\cdot))_{i=1}^k$, $(Q_{V_i}^m(\cdot))_{i=k+1}^{k+l}$.
 Let $\bar{\mathscr{T}}_m^{ \mathbf{p},\star}(\tilde{\mathbf{V}}_m^{k,k+l})$ denote the element corresponding to $\mathscr{T}_m^{ \mathbf{p},\star}(\tilde{\mathbf{V}}_m^{k,k+l})$ with blobs. 
 Thus, $\bar{\mathscr{T}}_m^{ \mathbf{p},\star}(\tilde{\mathbf{V}}_m^{k,k+l})$ is viewed as an element of $\overline{T}_{\sss IJ}^{*m}$. 
Let $\mathbf{V}_m=(V_1,\dots,V_{k+l})$ be an i.i.d.~collection of random variables with distribution $\mathbf{p}$. Let $\E_{\mathbf{p},\star}$ denote the expecation conditionally on $\mathscr{T}_m^{\mathbf{p},\star}$ and $N_{\sss(m)}^\star$. The proof of \eqref{c3:eq:blobvsG-reduction1} now reduces to
\begin{eq} \label{c3:eq:blobvsG-reduction2}
&\Big|\E\big[\Phi\big(\tilde{\cG}_m^{\sss \mathrm{bl},\mathrm{s}}\big)\1\{N_{\sss(m)}^\star = k\}\big]-\E\big[\Phi(\tilde{\cG}_m^{\sss\mathrm{s}})\1\{N_{\sss(m)}^\star = k\}\big]\Big|\\
&\hspace{.5cm}= \bigg| \E\bigg[\E_{\mathbf{p},\star}\Big[\mathsf{g}_{\phi}^{k}\Big(\frac{\sigma(\mathbf{p})}{B_m+1}\bar{\mathscr{T}}_m^{\mathbf{p},\star}(\tilde{\mathbf{V}}_m^{k,k+l}) \Big)\Big]\ind{N_{\sss (m)}^{\star}=k}\bigg]\\
&\hspace{1cm}- \E\Big[\E_{\mathbf{p},\star}\big[\mathsf{g}_{\phi}^{k}\big(\sigma(\mathbf{p})\mathscr{T}_m^{\mathbf{p},\star}(\tilde{\mathbf{V}}_m^{k,k+l})\big)\big]\ind{N_{\sss (m)}^{\star}=k}\Big]\bigg| +o(1).
\end{eq}
Notice that the tilting does not affect the blobs themselves but only the superstructure. Recall also the definition of the tilting function $L(\cdot)$ from \eqref{c3:eqn:ltpi-def}. 
Using the fact that $\cJ^{\sss (m)}(v) \propto p_v \mathfrak{G}_{\sss (m)}(v)$, 
\begin{equation}
\begin{split}
\E_{\mathbf{p},\star}\big[\mathsf{g}_{\phi}^{k}\big(\sigma(\mathbf{p})\mathscr{T}_m^{\mathbf{p},\star}(\tilde{\mathbf{V}}_m^{k,k+l})\big) \big]= \frac{\E_{\mathbf{p},\star}\big[\prod_{i=1}^k\mathfrak{G}_{\sss (m)}(V_i)\mathsf{g}_{\phi}^{k}\big(\sigma(\mathbf{p})\mathscr{T}_m^{\mathbf{p},\star}(\mathbf{V}_m)\big) \big]}{\big(\E_{\mathbf{p},\star}[\mathfrak{G}_{\sss (m)}(V_1)]\big)^k},
\end{split}
\end{equation} and an identical expression holds by replacing $\sigma(\mathbf{p})\mathscr{T}_m^{\mathbf{p},\star}$ by $\frac{\sigma(\mathbf{p})}{B_m+1}\bar{\mathscr{T}}_m^{\mathbf{p}}$.
Denote the expectation conditionally on $\mathscr{T}_m^{\mathbf{p}}$ and $N_{\sss (m)}$ by $\E_{\mathbf{p}}$ and simply write $\bar{\mathscr{T}}_m^{\mathbf{p},{\sss \mathrm{s}}}$, $\mathscr{T}_m^{\mathbf{p},{\sss \mathrm{s}}}$ for $\frac{\sigma(\mathbf{p})}{B_m+1}\bar{\mathscr{T}}_m^{\mathbf{p}}(\mathbf{V}_m)$,  $\sigma(\mathbf{p})\mathscr{T}_m^{\mathbf{p}}(\mathbf{V}_m)$ respectively. Now, \eqref{c3:eq:blobvsG-reduction2} simplifies to 
\begin{align}\label{c3:eq:blobvsG-reduction3}
 &\Big|\E\big[\Phi\big(\tilde{\cG}_m^{\sss \mathrm{bl},\mathrm{s}}\big)\1\{N_{\sss(m)}^\star = k\}\big]-\E\big[\Phi(\tilde{\cG}_m^{\sss\mathrm{s}})\1\{N_{\sss(m)}^\star = k\}\big]\Big|\nonumber\\
 &\hspace{.6cm}\leq \frac{1}{\expt{L(\mathscr{T}_m^{\mathbf{p}})}}\bigg|\E\bigg[\frac{\E_{\mathbf{p}}\big[\prod_{i=1}^k\mathfrak{G}_{\sss (m)}(V_i)\mathsf{g}_{\phi}^{k}\big(\bar{\mathscr{T}}_m^{\mathbf{p},{\sss \mathrm{s}}} \big)\big]}{\big(\E_{\mathbf{p}}[\mathfrak{G}_{\sss (m)}(V_1)]\big)^k} L(\mathscr{T}_m^{\mathbf{p}})\ind{N_{\sss (m)}=k}\bigg] \nonumber\\
 & \hspace{1cm} - \E\bigg[\frac{\E_{\mathbf{p}}\big[\prod_{i=1}^k\mathfrak{G}_{\sss (m)}(V_i)\mathsf{g}_{\phi}^{k}\big(\mathscr{T}_m^{\mathbf{p},{\sss \mathrm{s}}}\big) \big]}{\big(\E_{\mathbf{p}}[\mathfrak{G}_{\sss (m)}(V_1)]\big)^k}L(\mathscr{T}_m^{\mathbf{p}})\ind{N_{\sss (m)}=k}\bigg] \bigg|.
\end{align}
\begin{proposition}\label{c3:prop:g-blob-close} As $m\to\infty$,
 $\big|\mathsf{g}_{\phi}^{k}\big(\bar{\mathscr{T}}_m^{\mathbf{p},{\sss \mathrm{s}}}\big)-\mathsf{g}_{\phi}^{k}\big(\mathscr{T}_m^{\mathbf{p},{\sss \mathrm{s}}} \big) \big|\pto 0.$
\end{proposition} 
We first show that it is enough to prove Proposition~\ref{c3:prop:g-blob-close} to complete the proof of \eqref{c3:eq:blobvsG-reduction3}, but before that we first need to state some results. 
The proof of Proposition~\ref{c3:prop:g-blob-close} is deferred till the end of this section.
\begin{lemma}[{\cite[Proposition 4.8, Theorem 4.15]{BHS15}}] \label{c3:lem:uni-int-tilt} $(L(\mathscr{T}_m^{\mathbf{p}}))_{m\geq 1}$ is uniformly integrable. Also, for each $k\geq 0$, the quantity
	\begin{align}\label{c3:eqn:jt-convg}
&\bigg(\Ep\bigg[\frac{\dA_{\sss(m)}(V_1^{\sss(m)})}{\sigma(\vp)}\bigg],  ~\Ep\bigg[\bigg(\prod_{i=1}^k \frac{\dA_{\sss(m)}(V_i^{\sss(m)})}{\sigma(\vp)}\bigg) \mathsf{g}_\phi^{k}\big(\mathscr{T}_m^{\mathbf{p},{\sss \mathrm{s}}}\big)  \bigg]\bigg)
	\end{align} converges in distribution to some random variable.
\end{lemma}
\begin{fact}\label{c3:fact:unif-int} Consider three sequences of random variables  $(X_{m})_{m\geq 1}$, $(Y_m)_{m\geq 1}$ and $(Y_m')_{m\geq 1}$ with (i) $(X_m)_{m\geq 1}$ is uniformly integrable, (ii) $(Y_m)_{m\geq 1}$ and $(Y_m')_{m\geq 1}$ are almost surely bounded and (iii) $Y_m-Y_m'\xrightarrow{\sss \PR} 0$. Then, as $m\to\infty$, $$\expt{|X_mY_m-X_mY_m'|}\to 0.$$
\end{fact} 
\begin{fact}\label{c3:fact:2} Suppose that $(X_m)_{m\geq 1}$ is a sequence of random variables such that for every $m\geq 1$, there exists a further sequence $(X_{m,r})_{r\geq 1}$ satisfying (i) for each fixed $m\geq 1$, $X_{m,r}\xrightarrow{\sss \PR} 0$ as $r\to\infty$, and (ii) $\lim_{r\to\infty}\limsup_{m\to\infty}\PR(|X_m-X_{m,r}|>\varepsilon)=0$ for any $\varepsilon>0$. Then $X_m \xrightarrow{\sss \PR} 0$ as $m\to\infty$.
\end{fact} 
\begin{proof}[Proof of \eqref{c3:eq:blobvsG-reduction3} from Proposition~\ref{c3:prop:g-blob-close}]
 By Lemma~\ref{c3:lem:uni-int-tilt} and Fact~\ref{c3:fact:unif-int}, the proof of \eqref{c3:eq:blobvsG-reduction3} reduces to showing
\begin{equation}\label{c3:eq:blobvsG-reduction4}
\Ep\bigg[\bigg(\prod_{i=1}^k \frac{\dA_{\sss(m)}(V_i^{\sss(m)})}{\sigma(\vp)}\bigg) \big(\mathsf{g}_{\phi}^{k}\big(\bar{\mathscr{T}}_m^{\mathbf{p},{\sss \mathrm{s}}}\big)-\mathsf{g}_{\phi}^{k}\big(\mathscr{T}_m^{\mathbf{p},{\sss \mathrm{s}}} \big) \big)  \bigg] \pto 0.
\end{equation}
Let $X_m$ denote the term inside the expectation in \eqref{c3:eq:blobvsG-reduction4}.
Further, sample the set of leaves $\mathbf{V}_m$ independently $r$ times on the same tree $\mathscr{T}_m^\vp$ and let $X_{m}^i$ denote the observed value of $X_m$ in the $i$-th sample. 
Further, let $X_{m,r} = r^{-1}\sum_{i=1}^rX_m^i$.
Obviously, condition (i) in Fact~\ref{c3:fact:2} is satisfied due to Proposition~\ref{c3:prop:g-blob-close}. 
To verify condition (ii), note that $\Ep(X_{m,r}) = X_m$ and therefore Chebyshev's inequality yields
\begin{eq}
\prob{|X_m-X_{m,r}|>\varepsilon}&\leq \frac{\E[X_m^2]}{\varepsilon^2r} \leq \frac{4\|\phi\|_{\infty}^2}{\varepsilon^2r}\frac{\expt{\dA_{\sss (m)}(V_1^m)^{2k}}}{\sigmap^{2k}}.
\end{eq}
The final term is uniformly bounded over $m\geq 1$ and vanishes as $r\to\infty$.
For an interested reader, using the notation of \cite{BHS15}, the above is a consequence of the fact that $\|\dA_{\sss (m)}\|_{\infty}\leq \|F^{\mathrm{exc},\vp}\|_{\infty}$ and \cite[Lemma 4.9]{BHS15}.
\end{proof}

In this section, we will use the notion of Gromov-Hausdorff-Prokhorov topology on the collection of $(X,\dst,\mu)$, where $(X,\dst)$ is a compact metric space and $\mu$ is a probability measure on the corresponding Borel sigma algebra. 
Without re-defining all the required notions, we refer the reader to \cite[Section~2.1.1]{BHS15}.
 We further recall the notation $\mathrm{dis}$ for distortion and $D(\mu;\mu_1,\mu_2)$ for discrepancy of measures as defined in \cite[Section~2.1.1]{BHS15}.
Denote the root of $\mathscr{T}_m^{\vp}(\mathbf{V}_m)$ by $0+$ and the $j$th leaf by $j+$.
Let $\mathscr{M}_j^m:= \{[0+,j+],\dst,\nu_j\}$ be the random measured metric space 
 with the corresponding root-to-leaf measure $\nu_j$. 
 Let $\vmbar{12}{\mathscr{M}}_j^m:= \{ \bar{M}_j, \bar{\dst}, \bar{\nu}_j \}$ be the measured metric space   with $\bar{M}_j:=\sqcup_{i\in [0+,j+]}M_i$ and the induced root-to-leaf measure \linebreak $\bar{\nu}_j(A)=\sum_{i\in [0+,j+]}\nu_j(\{i\})\mu_i(A\cap M_i)$. 
For convenience, we have suppressed the dependence on $\mathscr{T}_m^{\vp}(\mathbf{V}_m)$ in the notation. 
 $\mathscr{M}_j^m$ is coupled to $\vmbar{12}{\mathscr{M}}_j^m$ in the obvious way that the superstructure of $\vmbar{12}{\mathscr{M}}_j^m$ is given by $\mathscr{M}_j^m$.
 We need the following lemma to prove Proposition~\ref{c3:prop:g-blob-close}:
\begin{lemma}\label{c3:lem:path:GHP-conv}For $j\geq 1$, as $m\to\infty$, $\dGHP \big(\sigma(\mathbf{p})\mathscr{M}_j^m, \frac{\sigma(\mathbf{p})}{B_m+1}\vmbar{12}{\mathscr{M}}_j^m \big) \pto 0.$
\end{lemma}

\begin{proof}
 We prove this for $j=1$ only. The proof for $j\geq 2$ is identical. 
 For $x\in \bar{M}_1$, we denote its corresponding vertex label by $i(x)$, i.e., $i(x)=k$ iff $x\in M_k$. Consider the correspondence $C_m$ and the measure $\mathfrak{m}$ on the product space $[0+,1+]\times \bar{M}_1$ defined as
 \begin{equation}\label{c3:corr-blob-vs-path}
 C_m:=\{(i,x): i\in [0+,1+], x\in M_i\}, \quad \mathfrak{m}(\{i\}\times A)=\nu_1(\{i\})\bar{\nu}_1(A\cap M_i).
 \end{equation} 
%
  Note that the discrepancy of $\mathfrak{m}$ satisfies $D(\mathfrak{m}; \nu_1,\bar{\nu}_1)=0$. Further, $\mathfrak{m}(C_m^c)=0.$ Therefore, Lemma~\ref{c3:lem:path:GHP-conv} follows if we can prove that
 \begin{equation}\label{c3:distortion:expression}
  \dis(C_m):= \sup_{x,y\in \bar{M}_1} \Big\{ \sigma(\mathbf{p})\dst(i(x),i(y))-\frac{\sigma(\mathbf{p})}{B_m+1}\bar{\dst}(x,y) \Big\} \pto 0.
 \end{equation}
To simplify the expression for $\dis(C_m)$, suppose that $i(x)$ is an ancestor of $i(y)$ on the path from $0+$ to $1+$. Then, 
\begin{gather*}
 \dst(i(x),i(y))=\dst(0+,i(y))-\dst(0+,i(x)),\\
 \bar{\dst}(x_0,y)-\bar{\dst}(x_0,x) \leq \bar{\dst}(x,y)\leq \bar{\dst}(x_0,y)-\bar{\dst}(x_0,x)+2\Delta_{\max},
\end{gather*}for any $x_0\in M_{\sss 0+}$. 
This implies that
\begin{equation}\label{c3:eq:distortion-reduction}
\begin{split}
 &\sup_{x,y\in \bar{M}_1} \Big\{ \sigma(\mathbf{p})\dst(i(x),i(y))-\frac{\sigma(\mathbf{p})}{B_m+1}\bar{\dst}(x,y) \Big\}\\
 &\hspace{.8cm}\leq 2 \sup_{y \in \bar{M}_1}\Big\{ \sigma(\mathbf{p})\dst(0+,i(y))-\frac{\sigma(\mathbf{p})}{B_m+1}\bar{\dst}(x_0,y) \Big\}+\frac{2 \sigmap \Delta_{\max}}{B_m+1}.
 \end{split}
\end{equation}
Further, replacing $y$ by any other point $y'$ in the right hand side in \eqref{c3:eq:distortion-reduction} incurs an error of at most $\Delta_{\max}$. Now, write the path $[0+,1+]$ as $0+=i_0\to i_1\to \dots \to i_{R^*-2} \to i_{R^*-1}=1+.$ Then
\begin{equation}\label{c3:eqn:estimate-dis}
 \dis(C_m)\leq 2\sup_{k\leq R^*-1} \Big| \sigmap \dst(i_0,i_k)-\frac{\sigmap}{B_m+1}\bar{\dst}(X_{\sss i_0,i_1}, X_{\sss i_k,i_{k+1}}) \Big|+\frac{6\sigmap \Delta_{\max}}{B_m+1},
\end{equation}where $(X_{i,j})_{i,j\in [m]}$ are the junction-points.
Using Assumption~\ref{c3:assm:blob-diameter} and \eqref{c3:eqn:estimate-dis}, it is now enough to show that for any $\varepsilon >0$, 
\begin{equation}\label{c3:prob:dst:paths-mg}
 \lim_{m\to\infty} \PR\bigg(\sup_{k\leq R^*-1}\Big| \sigmap \dst(i_0,i_k)-\frac{\sigmap}{B_m+1}\dst(X_{\sss i_0,i_1}, X_{\sss i_k,i_{k+1}}) \Big| >\varepsilon\bigg) = 0.
\end{equation}
Denote the term inside the above supremum  by $Q_k$. Then,
\begin{equation}\label{c3:Qk-simplif}
\begin{split}
 Q_k&:=\bigg[ \sigmap \dst(i_0,i_k)-\frac{\sigmap}{B_m+1}\bar{\dst}\big(X_{\sss i_0,i_1}, X_{\sss i_k,i_{k+1}}\big) \bigg]\\
 &=\bigg[ \sigmap k - \frac{ \sigmap}{B_m+1}\bigg(k+\sum_{j=1}^k\dst_{i_j}\big(X_{\sss i_j,i_{j-1}}, X_{\sss i_j, i_{j+1}}\big)\bigg)\bigg]\\
 &=\frac{\sigmap}{B_m+1} \bigg[\sum_{j=1}^k\Big( B_m-\dst_{i_j}\big(X_{\sss i_j,i_{j-1}}, X_{\sss i_j, i_{j+1}}\big)  \Big) \bigg].
 \end{split}
\end{equation}
 Recall the construction of the path $[0+,1+]$ via the birthday problem from Section~\ref{c3:sec:p-tree-birthday}. Take $\mathbf{J}:=(J_i)_{i\geq 1}$ such that $J_i$ are i.i.d.~samples from $\mathbf{p}$. 
 Further let $\boldsymbol{\xi}:= (\xi_i)_{i\in [m]}$ be an independent sequence such that $\xi_i$ is the distance between two points, chosen randomly from $M_i$ according to $\mu_i$. 
 Further, let $\mathbf{J}$ and $\boldsymbol{\xi}$ be independent. Then $R^*$ can be thought of as the first repeat time of the sequence $\mathbf{J}$. Thus, $Q_k$ in \eqref{c3:Qk-simplif} has the same distribution as 
 \begin{equation}
  \begin{split}
    \hat{Q}_k:= \frac{\sigmap}{B_m+1} \sum_{i=1}^k\big(   B_m -\xi_{J_i} \big).
  \end{split}
 \end{equation}
From the birthday construction $\expt{\xi_{J_i}}= \sum_{i\in [m]}p_iu_i=B_m$ and $(\xi_{J_i})_{i \geq 1}$ is an independent sequence. Therefore, $(\hat{Q}_k)_{k\geq 0}$ is a martingale. Further, 
\begin{equation}
\mathrm{Var}(\hat{Q}_k)\leq \bigg(\frac{\sigmap}{B_m+1}\bigg)^2k\Delta_{\max}\sum_{i\in [m]}p_iu_i=\frac{\sigmap^2k\Delta_{\max}B_m}{(B_m+1)^2}.
\end{equation} Thus, by Doob's martingale inequality \cite[Lemma 2.54.5]{RW94}, it follows that, for any $\varepsilon>0$ and $T>0$,
\begin{equation}\label{c3:eqn:supbound:Q}
 \PR\bigg(\sup_{k\leq T}|\hat{Q}_k|>\varepsilon \bigg)\leq \frac{T\sigmap^2\Delta_{\max}B_m}{(B_m+1)^2\varepsilon^2}.
\end{equation}Recall from \cite[Theorem 4]{CP99} that $(\sigmap R^*)_{m\geq 1}$ is a tight sequence of random variables. 
The proof now follows using Assumption~\ref{c3:assm:blob-diameter}.
%
%
%
%
%
\end{proof}
\begin{proof}[Proof of Proposition~\ref{c3:prop:g-blob-close} using Lemma~\ref{c3:lem:path:GHP-conv}] 
We use objects defined in \eqref{c3:corr-blob-vs-path}, \eqref{c3:distortion:expression} in the proof of Lemma~\ref{c3:lem:path:GHP-conv} for all the path metric spaces with $j\leq k$. We assume that we are working on a probability space such that the convergence \eqref{c3:distortion:expression} holds almost surely for all $j\leq k$. 
To summarize, for fixed $\varepsilon>0$ and for each $j\leq k$, we can choose the correspondence $C_m^j$ and a measure $\mathfrak{m}_j$ of $[0+,j+]\times \bar{M}_j$ satisfying (i) $(i,X_{ik})\in C_m^j$, for all $i,k\in [0+,j+]$, (ii) $\dis(C_m^j)<\varepsilon/2k$ almost surely, and (iii) $D(\mathfrak{m}_j;\nu_j,\bar{\nu}_j)=0$ and $\mathfrak{m}_j((C^j_m)^c)=0$. 
Recall the definitions of the function $\mathsf{g}_{\phi}^k$ from \eqref{c3:defn:g-phi-1}, \eqref{c3:defn:g-phi-2} and the associated graphs $G(\cdot)$, $\bar{G}(\cdot)$. We simply write $G$ and $\bar{G}$ for $G(\sigma(\mathbf{p})\mathscr{T}_m^{\mathbf{p}}(\mathbf{V}_m))$ and $\bar{G}(\frac{\sigma(\mathbf{p})}{B_m+1}\bar{\mathscr{T}}_m^{\mathbf{p}}(\mathbf{V}_m))$, respectively. 
Let $\mathfrak{m}^{\sss \otimes k}$ denote the $k$-fold product measure of $\mathfrak{m}_j$ for $j\leq k$. 
We denote the graph distance on a graph $H$ by $\dst_{\sss H}$.
Note that 
\begin{equation}\label{c3:g-phi-compare}
\begin{split}
&\Big|\mathsf{g}_{\phi}^{k}\big(\sigma(\mathbf{p})\mathscr{T}_m^{\mathbf{p}}(\mathbf{V}_m)\big)-\mathsf{g}_{\phi}^{k}\Big(\frac{\sigma(\mathbf{p})}{B_m+1}\bar{\mathscr{T}}_m^{\mathbf{p}}(\mathbf{V}_m) \Big) \Big|\\
&\leq \E\big[\big|\phi\big((\dst_{\sss G}(i+,j+))_{i,j=k+1}^{k+l}\big)- \phi\big((\dst_{\sss \bar{G}}(X_i,X_j))_{i,j=k+1}^{k+l}\big)\big|\big],
\end{split}
\end{equation} where $X_i\sim \mu_i$ independently for $i\in [m]$, and the above expectation is with respect to the measure $\mathfrak{m}^{\sss \otimes k}$.
Recall the notation while defining $\mathsf{g}_\phi^k(\cdot)$ in \eqref{c3:defn:g-phi-1}, \eqref{c3:defn:g-phi-2}. Notice that for any point $k\in [0+,i+]$ and $x_k\in M_k$ and $x_{i(s)}\in M_{i(s)}$,
\begin{equation}\label{c3:shortcut-contribute}
 | \dst_{\sss \mathbf{t}}(k,i(s))- \dst_{\sss\bar{\mathbf{t}}}(x_k,x_{i(s)})|\leq \frac{\varepsilon}{2k}.
\end{equation}
Now, for any path $i+$ to $j+$ in $G$, we can essentially take the same path from $X_i$ to $X_j$ in $\bar{G}$ and take the corresponding inter-blob paths on the way. 
The distance traversed in $\bar{G}$ in this way gives an upper bound on $\dst_{\sss \bar{G}}(X_i,X_j)$. Notice that, by \eqref{c3:shortcut-contribute}, taking a shortcut contributes at most $\varepsilon / 2k$ to the difference of the distance traveled in $G$ and $\bar{G}$. Also, traversing a shortcut edge contributes $\sigmap B_m/(B_m+1)$ and there are at most $k$ shortcuts on the path. 
Furthermore, it may be required to reach the relevant junction points from $X_i$ and $X_j$ and that contributes at most $2 \sigmap \Delta_{\max}/(B_m+1)$. Thus, for $k+1\leq i,j\leq k+l$, and sufficiently large $m$,
\begin{equation}
 \dst_{\sss \bar{G}}(X_i,X_j)\leq \dst_{\sss G}(i+,j+)+\frac{\varepsilon}{2}+\frac{k\sigmap B_m}{B_m+1}+\frac{2\sigmap \Delta_{\max}}{B_m+1}\leq \dst_{\sss G}(i+,j+)+\varepsilon. 
\end{equation}
%
By symmetry we can conclude the lower bound also and the continuity of $\phi(\cdot)$ (see \cite[Theorem 4.18]{BHS15}) along with \eqref{c3:g-phi-compare} completes the proof of Proposition~\ref{c3:prop:g-blob-close}.
\end{proof} 

\label{c3:sec:entrance-bdd-proofs}
At this moment, we urge the reader to recall the definitions from \eqref{c3:defn:barely-subcrit}, \eqref{c3:defn:susc-w} and \eqref{c3:defn:susc-dist}. 
The configuration model graphs considered in this section will be assumed to have degree sequence $\bld{d}'$ and the vertices have an associated weight sequence $\bld{w}$ such that  Assumption~\ref{c3:assumption-w} is satisfied. 
We treat the different terms arising in Theorem~\ref{c3:thm:susceptibility} in different subsections. 
\subsection{Analysis of $s_2^\star$}\label{c3:sec:s2} 
The asymptotics of $s_2^\star$ is a consequence of the Chebyshev inequality. In the following lemma, we compute its mean and variance. 
Consider the size-biased distribution on the vertex set $[n]$ with sizes $(w_i)_{i\in [n]}$.
 Let $V_n$ and $V_n^{*}$, respectively, denote a vertex chosen  uniformly at random and according to the size-biased distribution, independently of the underlying graph $\rCM_n(\bld{d}')$. 
Let $D_n'$, $W_n$ (respectively $D_n^*$, $W_n^*$) denote the degree and weight of $V_n$ (respectively $V_n^*$). 
For a vertex $v\in [n]$, let $\mathscr{W}(v):=\sum_{k\in\mathscr{C}'(v)}w_k$, where $\mathscr{C}'(v)$ denotes the component of $\rCM_n(\bld{d}')$ containing $v$.

\begin{lemma}\label{c3:lem:expt-weight-random} Under \textrm{Assumption~\ref{c3:assumption-w}}, (i) $
 \expt{\mathscr{W}(V_n^*)}= \frac{\expt{D_n^*}\expt{D_n' W_n}}{\expt{D_n'}(1-\nu_n')}(1+o(1))$, (ii) 
 $  \E\big[\big(\mathscr{W}(V_n^*)\big)^2\big] = O\big(\frac{n^{3\alpha-1}}{(1-\nu_n')^3}\big)$, and $\E\big[\big(\mathscr{W}(V_n^*)\big)^3\big]= o(n^{1+2\delta}).$
\end{lemma}
 \begin{proof}[Asymptotics of $s_2^\star$]Denote $\ell_n^w=\sum_{i\in [n]}w_i$. Firstly, if $\E_{\bld{d}'}$ denotes the conditional expectation given $\mathrm{CM}_n(\bld{d}')$, then for any $r\geq 1$,
\begin{equation}\label{c3:eq:identity-moment}
\begin{split}
\E_{\bld{d}'}\big[\big(\mathscr{W}(V_n^*)\big)^{r-1}\big] 
= \sum_{i\geq 1}\sum_{k\in \mathscr{C}_{\sss (i)}'}\frac{w_k}{\ell_n^w}\bigg(\sum_{k\in \mathscr{C}_{\sss (i)}'}w_k\bigg)^{r-1} = \frac{1}{\ell_n^w}\sum_{i\geq 1}\big(\mathscr{W}_{\sss (i)}\big)^r.
 \end{split}
\end{equation} Therefore, using Lemma~\ref{c3:lem:expt-weight-random} and \eqref{c3:defn:barely-subcrit}, it follows from Assumption~\ref{c3:assumption-w} that
\begin{equation} \label{c3:expt:s2}
 n^{-\delta}\expt{s_2^\star}= \frac{\ell_n^w}{n} n^{-\delta}\expt{\mathscr{W}(V_n^*)}\to \frac{\mu_{d,w}^2}{\mu_d\lambda_0},
\end{equation}where we have used the fact that $\expt{D_n^*}\to \mu_{d,w}/\mu_w$. It remains to compute the variance. Let $U_n^*$ denote another vertex chosen in a size-biased way with the sizes being $(w_i)_{i\in [n]}$, independently of the graph $\mathrm{CM}_n(\bld{d}')$ and $V_n^*$.  
Then \eqref{c3:eq:identity-moment} yields
 \begin{equation}\label{c3:eq:moment2-s2}
 \begin{split}
  &\expt{(s_2^{\star})^2}
  =\frac{1}{n^2}\E\bigg[\sum_{i,j\geq 1}\mathscr{W}_{\sss (i)}^2\mathscr{W}_{\sss (j)}^2\bigg]
  =\frac{1}{n^2}\E\bigg[\sum_{i\geq 1}\mathscr{W}_{\sss (i)}^4\bigg]+\frac{1}{n^2}\E\bigg[\sum_{i\neq j}\mathscr{W}_{\sss (i)}^2\mathscr{W}_{\sss (j)}^2\bigg]\\
  & = \frac{\ell_n^w}{n}\frac{1}{n}\expt{\big(\mathscr{W}(V_n^*)\big)^3}+\bigg(\frac{\ell_n^w}{n}\bigg)^2\expt{\mathscr{W}(U_n^*)\mathscr{W}(V_n^*)\ind{U_n^*\notin \mathscr{C}'(V_n^*)}},
 \end{split}
\end{equation}
where the second term in the third equality follows using similar arguments as in \eqref{c3:eq:identity-moment}. 
Denote the last two terms of \eqref{c3:eq:moment2-s2} by $(\mathbf{I})$ and $(\mathbf{II})$ respectively.
To estimate~$(\mathbf{II})$, observe that, conditionally on the graph $\mathscr{C}'(V_n^*)$, the graph obtained by removing $\mathscr{C}'(V_n^*)$ from $\rCM_n(\bld{d}')$ is again a configuration model with the induced degree sequence $\Mtilde{\boldsymbol{d}}$ and number of vertices $\tilde{n}$. Let $\tilde{\nu}_n$ denote the corresponding criticality parameter. 
 In the proof of Lemma~\ref{c3:lem:expt-weight-random}~(i), we will see that the upper bound holds whenever $\tilde{\nu}_n<1$ (see Remark~\ref{c3:rem:ub-susc}). Thus, let us first show that 
\begin{equation}\label{c3:fact:graphs}
\text{for all sufficiently large }n,\ \prob{\tilde{\nu}_n<1\mid \mathscr{C}'(V_n^*)}=1, \text{ almost surely.}
\end{equation}
Denote $\ell_n' = \sum_{i\in [n]}d_i'$. To see \eqref{c3:fact:graphs}, first notice that 
 \begin{equation}
\begin{split}
 &\tilde{\nu}_n-1 = \nu_n'-1-\frac{\sum_{j \in \mathscr{C}'(V_n^*)} d_{j}'(d_{j}'-2)}{\sum_{j \in [n]} d_{j}'}+(\tilde{\nu}_n-1)\frac{\sum_{j \in \mathscr{C}'(V_n^*)} d_{j}'}{\ell_n'}.
%
 \end{split}
\end{equation}
Moreover, for any connected graph $\mathcal{G}$, $\sum_{i\in \mathcal{G}}d_i'(d_i'-2)\geq -2$ (this can be proved by induction) 
so that
$(\tilde{\nu}_n-1)\big(1-\sum_{j \in \mathscr{C}'(V_n^*)} d_{j}'/\ell_n'\big)\leq \nu_n'-1+\frac{2}{\ell_n'}. $
The proof of \eqref{c3:fact:graphs} now follows. 
As mentioned above, now we can apply the upper bound from Lemma~\ref{c3:lem:expt-weight-random}. 
Therefore,
 \begin{equation}\label{c3:weight-conditioned-comp}
\begin{split}
 &\expt{\mathscr{W}(U_n^*)\ind{U_n^* \notin \mathscr{C}'(V_n^*)} \big\vert \mathscr{C}'(V_n^*)}\\
 &= \frac{\sum_{i\notin \mathscr{C}'(V_n^*)}w_i}{\ell_n^w}\expt{\mathscr{W}(U_n^*) \big\vert \mathscr{C}'(V_n^*), U_n^* \notin \mathscr{C}'(V_n^*)}\\
 &\leq \frac{\big(\sum_{i\in [n]}d_i'w_i\big)^2}{\ell_n^w\ell_n'(1-\nu_n'-2/\ell_n')}= \expt{\mathscr{W}(V_n^*)}\big(1+O(1/n)\big).
 \end{split}
\end{equation}
Thus,
\begin{equation}\label{c3:cross-prod-var-s2}
 \expt{\mathscr{W}(U_n^*)\mathscr{W}(V_n^*)\ind{U_n^* \notin \mathscr{C}'(V_n^*)} }\leq \big(\expt{\mathscr{W}(V_n^*)}\big)^2\big(1+O(1/n)\big).
\end{equation} 
We conclude that \eqref{c3:eq:moment2-s2}, \eqref{c3:cross-prod-var-s2} together with Lemma~\ref{c3:lem:expt-weight-random} implies that $\var{s_2^\star}=o(n^{2\delta})$.
 Thus, we can use the Chebyshev inequality and \eqref{c3:expt:s2} to conclude that $$n^{-\delta}s_2^\star \xrightarrow{\sss \PR} \mu_{d,w}^2/(\mu_d\lambda_0).$$ 
 \end{proof} 
\begin{proof}[Proof of Lemma~\ref{c3:lem:expt-weight-random}~$(i)$]  
We use path-counting techniques for configuration models from \cite[Lemma 5.1]{J09b}. 
  Let $\mathcal{A}(k,l)$ denote the event that there exists a path of length $l$ from $V_n^*$ to $k$ and $\mathcal{A}'(k,l)$ the event that there exist two different paths, one of length $l$ and another one of length at most $l$, from $V_n^*$ to $k$. Notice that 
\begin{subequations}
\begin{equation}\label{c3:sus1-expr-1}
 \begin{split}
  \expt{\mathscr{W}(V_n^*)} &\leq \E\bigg[\sum_{k\in [n]}w_k\ind{V_n^{*}\leadsto k}\bigg]=\expt{W_n^*} +  \sum_{l\geq 1}\sum_{k\in [n]}w_{k}\prob{\mathcal{A}(k,l)},
 \end{split}
 \end{equation}
\begin{equation}\label{c3:sus1-expr-2}
 \expt{\mathscr{W}(V_n^*)} \geq \sum_{l\geq 1}\sum_{k\in [n]}w_{k}\prob{\mathcal{A}(k,l)}-\sum_{l\geq 1}\sum_{k\in [n]}w_{k}\prob{\mathcal{A}'(k,l)}.
\end{equation}
\end{subequations}
  Now, by Assumption~\ref{c3:assumption-w}, \eqref{c3:sus1-expr-1} yields
\begin{equation}\label{c3:sus-simpl-1}
 \begin{split}
  &\expt{\mathscr{W}(V_n^*)} 
  \leq \expt{W_n^*}+ \expt{D_n^{*}}\sum_{l = 1}^{\infty}\sum_{k\in [n]}w_k\sum_{x_i\neq x_j, \forall i\neq j}\frac{\prod_{i=1}^{l-1}d_{x_i}'(d_{x_i}'-1)d_{k}'}{(\ell_n'-1)\cdots (\ell_n'-2l+1)}\\
  &\leq \expt{W_n^*}+\frac{\expt{D_n^*}\expt{D_n' W_n}}{\expt{D_n'}-1/n}\sum_{l=1}^{\infty}\nu_n'^{l-1} = \frac{\expt{D_n^*}\expt{D_n' W_n}}{\expt{D_n'}(1-\nu_n')}(1+o(1)),
 \end{split}
 \end{equation}where in the third step, we have used the fact that $$\sum_{x_i\neq x_j, \forall i\neq j}\frac{\prod_{i=1}^{l-1}d_{x_i}'(d_{x_i}'-1)}{(\ell_n'-1)\cdots (\ell_n'-2l+1)}\leq \nu_n'^{l-1}$$ from \cite[Lemma 5.1]{J09}.
 For the computation of the lower bound, 
 observe that
 \begin{equation}\label{c3:sus-lb}
  \begin{split}
   &\sum_{l=1}^{n^{\eta}/\log(n)}\sum_{k\in [n]}w_{k}\prob{\mathcal{A}(k,l)} 
   \\&\geq \E[D_n^*]\sum_{l= 1}^{n^{\eta}/\log(n)}\sum_{k\in [n]} w_k\sum_{x_i\neq x_j, \forall i\neq j}\frac{1}{\ell_n'^{l-1}}\prod_{i=1}^{l-1}d_{x_i}'(d_{x_i}'-1)d_{k}'\\
    &\geq \frac{\E[D_n^*]\expt{D_n'W_n}}{\expt{D_n'}}\sum_{l = 1}^{n^{\eta}/\log(n)}\Big(\nu_n'^{l-1}-\frac{d_1'n^{\eta}\big(\sum_{i\in [n]}d_i'(d_i'-1)\big)^{l-2}}{\log(n)\ell_n'^{l-1}}\Big)  \\
    & = \frac{\expt{D_n^*}\expt{D_n' W_n}(1-(\nu_n')^{n^{\eta}/\log(n)})}{\expt{D_n'}(1-\nu_n')}(1+o(1))\\
    &=\frac{\expt{D_n^*}\expt{D_n' W_n}}{\expt{D_n'}(1-\nu_n')}(1+o(1)),
  \end{split}
 \end{equation}where we have used the fact that $d_1'l\leq d_1'n^{\eta}/\log(n)$ and inclusion-exclusion to obtain the third step, and 
 \eqref{c3:defn:barely-subcrit}, $d_1'n^{\eta}/\ell_n'= c_1/\mu_d(1+o(1))$ and the fact that $\delta<\eta$ in the one-but-last step.
 To complete the proof of Lemma~\ref{c3:lem:expt-weight-random}, we need to have an upper bound on the last term of~\eqref{c3:sus1-expr-2}. 
 Observe that if $\mathcal{A}'(k,l)$ happens, then one of the structures in Figure~\ref{c3:fig:multi-paths} occurs.

\begin{figure}
\centering
\begin{subfigure}{.2\textwidth}
\centering
	\includegraphics[width=.5\textwidth]{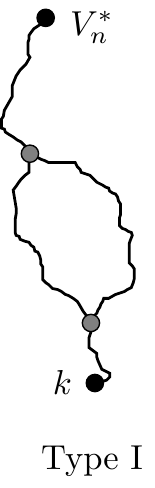}
\end{subfigure}
\begin{subfigure}{.2\textwidth}
\centering
	\includegraphics[width=.45\textwidth]{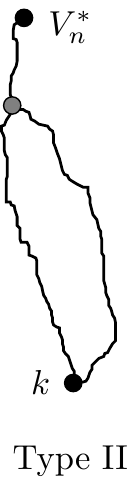}
\end{subfigure}
\begin{subfigure}{.2\textwidth}
\centering
	\includegraphics[width=.5\textwidth]{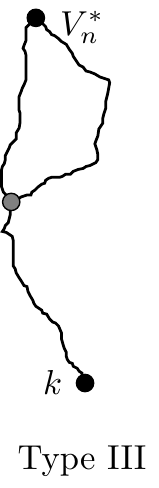}
\end{subfigure}
\begin{subfigure}{.2\textwidth}
\centering
	\includegraphics[width=.5\textwidth]{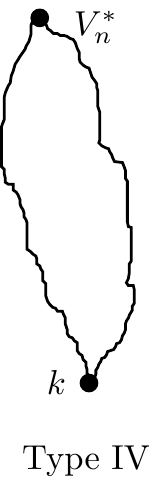}
\end{subfigure}
\caption{Possible paths corresponding to $\mathcal{A}'(k,l)$.}\label{c3:fig:multi-paths}
\end{figure}

Denote by $\mathcal{A}'(k,l,i)$ the event that the structure of type $i$ ($i$=I, II, III, IV) in Figure~\ref{c3:fig:multi-paths} appears. 
We use the notation $C$ to denote a generic constant.
Using an argument identical to \eqref{c3:sus-simpl-1}, and applying Assumption~\ref{c3:assumption-w}, it follows that
\begin{equation}
\label{c3:eqn:paths1}
\begin{split}
  &\sum_{l\geq 1}\sum_{k\in [n]}w_{k}\prob{\mathcal{A}'(k,l,\mathrm{I})}
  \\
  &\leq C\expt{D_n^*}\expt{D_n'W_n}\sum_{l\geq 1}\sum_{r\geq 1} \frac{\sigma_3(n)^2}{\ell_n'}(l-1)(l-2)\nu_n'^{l+r-4}\nonumber\\
&\leq C \frac{n^{6\alpha-3}}{1-\nu_n'} \sum_{l\geq 3}(l-1)(l-2)\nu_n'^{l-3}\leq C\frac{n^{6\alpha-3}}{(1-\nu_n')^4} =o(n^{\delta}),
\end{split}
\end{equation}where we have used the fact $\delta<\eta$ in the last step. 
Identical arguments can be carried out to conclude that $\sum_{l\geq 1}\sum_{k\in [n]}w_{k}\prob{\mathcal{A}'(k,l,i)} = o(n^{\delta})$, $i=\mathrm{II},\mathrm{III},\mathrm{IV}$.
Combining this with \eqref{c3:sus-lb} and applying them to \eqref{c3:sus1-expr-2}, it follows that
\begin{equation}\label{c3:susc-lb2}
 \expt{\mathscr{W}(V_n^*)}\geq  \frac{\expt{D_n^*}\expt{D_n' W_n}}{\expt{D_n'}(1-\nu_n')}(1+o(1)),
\end{equation}and the proof of Lemma~\ref{c3:lem:expt-weight-random} is now complete using \eqref{c3:sus-simpl-1}.
\end{proof}
\begin{proof}[Proof of Lemma~\ref{c3:lem:expt-weight-random}~$(ii)$]
Notice that 
\begin{equation}
\big(\mathscr{W}(V_n^*)\big)^r = \sum_{k_1,\dots,k_r\in [n]}w_{k_1}\cdots w_{k_r}\ind{V_n^*\leadsto k_1,\dots, V_n^*\leadsto k_r}. 
\end{equation}
\begin{figure}\centering
\begin{subfigure}{.4\textwidth}
\centering
	\includegraphics[width=.5\textwidth]{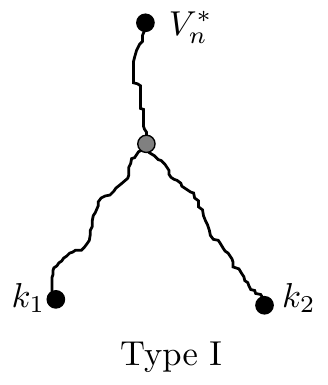}
\end{subfigure}
\begin{subfigure}{.4\textwidth}
\centering
	\includegraphics[width=.5\textwidth]{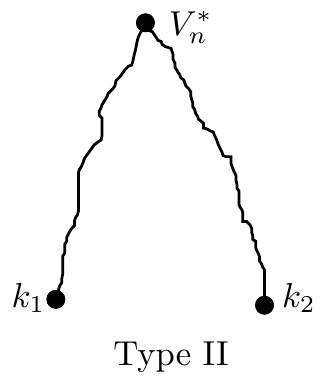}
\end{subfigure}
\caption{Possible paths when $V_n^*\leadsto k_1$, and $V_n^*\leadsto k_2$.}\label{c3:fig:paths2}
\end{figure}
We can again count the contribution due to the different types of paths in Figure~\ref{c3:fig:paths2} by using similar argument as in \eqref{c3:sus1-expr-1} to compute the second moment. 
Ignoring the re-computation, it follows that
\begin{equation}\label{c3:W-moment2-ub}
  \begin{split}\E\Big[\big(\mathscr{W}(V_n^*)\big)^2\Big] &\leq\expt{(W_n^*)^2}+\frac{\expt{D_n^*}(\expt{D_n'W_n})^2\sigma_3(n)}{(\expt{D_n'})^3(1-\nu_n')^3}\\
  &+\frac{\expt{D_n^*(D_n^*-1)}(\expt{D_n'W_n})^2}{(\expt{D_n'})^2(1-\nu_n')^2},
  \end{split}
\end{equation}which gives rise to the desired $O(\cdot)$ term. 
For the third moment, the leading contributions arise from the structures given in Figure~\ref{c3:fig:paths3}. 
See Appendix~\ref{c3:sec:appendix-path-counting} for a detailed computation.
\begin{figure} \centering
\begin{subfigure}{.2\textwidth}
 \centering
	\includegraphics[width=.8\textwidth]{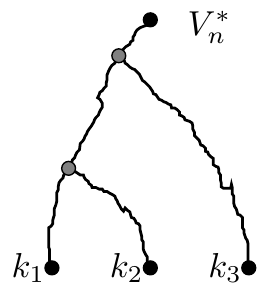}
\end{subfigure}
\begin{subfigure}{.2\textwidth}
\centering
	\includegraphics[width=.8\textwidth]{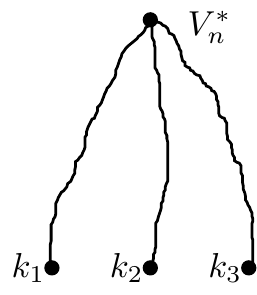}
\end{subfigure}
\begin{subfigure}{.2\textwidth}
\centering
	\includegraphics[width=.8\textwidth]{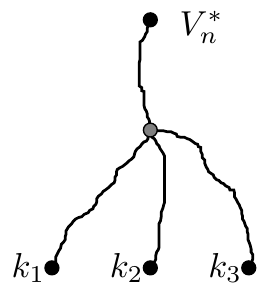}
\end{subfigure}
\begin{subfigure}{.2\textwidth}
\centering
	\includegraphics[width=.8\textwidth]{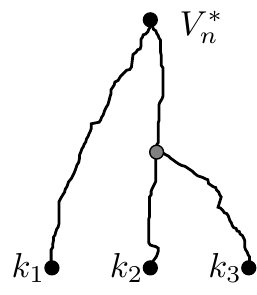}
\end{subfigure}
\caption{Possible paths when $V_n^*\leadsto k_1$, $V_n^*\leadsto k_2$, and $V_n^*\leadsto k_3$.}\label{c3:fig:paths3}
\end{figure}
\end{proof}
\begin{remark}\label{c3:rem:ub-susc} \normalfont The  upper bounds of $\E[(\mathscr{W}(V_n^*))^r]$ in \eqref{c3:sus-simpl-1} and \eqref{c3:W-moment2-ub} 
hold for any configuration model for which the value of the criticality parameter is less than one. The precise assumptions were needed to estimate the orders of these terms. 
\end{remark}
\begin{remark}\label{c3:rem:spr-asymp} \normalfont The method used to obtain the asymptotics of $s_2^\star$ can be followed verbatim to obtain the asymptotics of $s_{pr}^\star$. Indeed, notice that 
\begin{equation}
 \expt{s_{pr}^\star} = \frac{1}{n}\E\bigg[\sum_{i\geq 1}\mathscr{W}_{\sss (i)}|\mathscr{C}'_{\sss(i)}|\bigg] = \expt{\mathscr{W}(V_n)}.
\end{equation} 
A similar identity for the second moment of $s_{pr}^\star$ also holds.
\end{remark}

The main aim of this section is to prove the following proposition which will be required to obtain the asymptotics of $s_3^\star$, as well as $\mathscr{W}_{\sss (i)}$:
\begin{proposition}\label{c3:prop:tail3-bare-subcrit}
Suppose that \textrm{Assumption~\ref{c3:assumption-w}} holds. 
For any $\varepsilon > 0$,
\begin{equation}
 \lim_{K\to\infty}\limsup_{n\to\infty}\PR\bigg(\sum_{i>K}\big(\mathscr{W}_{\sss (i)}\big)^3>\varepsilon n^{3(\alpha+\delta)}\bigg)=0.
\end{equation}
\end{proposition}
\begin{proof}
 Let $\mathcal{G}^{\sss K}$ denote the graph obtained by deleting all the edges incident to vertices $\{1,\dots,K\}$. In this proof, a superscript $K$ to any previously defined object will correspond to the object in $\mathcal{G}^K$. Note that $\mathcal{G}^{\sss K}$ is again distributed as a configuration model conditioned on the new degree sequence $\bld{d}^{\sss K}$. 
 Firstly, for each fixed $K$, there exists a constant $C_1>0$ such that 
  \begin{equation}\label{c3:eqn:nu-K}
  \begin{split}
   \nu^{\sss K}_n&= \frac{\sum_{i\in [n]} d_i^{\sss K}(d^{\sss K}_i-1)}{\sum_{i\in [n]} d_i^{\sss K}}\leq \frac{\sum_{i\in [n]}d_i'(d_i'-1)-\sum_{i=1}^Kd_i'(d_i'-1)}{\ell_n'-2\sum_{i=1}^Kd_i'}\\
   &=\nu_n'-C_1n^{2\alpha -1}\sum_{i\leq K}c_i^2\leq\nu_n'-C_1n^{-\delta}\sum_{i\leq K}c_i^2,
  \end{split}
  \end{equation}where we have used the fact that $\delta<\eta$ in the last step. Since $\nu_n^{\sss K}<1$, we can apply the upper bound in \eqref{c3:W-moment2-ub} (see Remark~\ref{c3:rem:ub-susc}) and it follows that
  \begin{equation}\label{c3:est:3rd-susc-order}
  \begin{split}
   \frac{1}{n}\E\bigg[\sum_{i\geq 1}\big(\mathscr{W}_{\sss (i)}^{\sss K}\big)^3\bigg]&= \frac{\ell_n^w}{n}\expt{\big(\mathscr{W}^{\sss K}(V_n^*)\big)^2}
   \leq C\frac{n^{3(\alpha+\delta)-1}}{1+C_1\sum_{i=1}^Kc_i^2},
   \end{split}
  \end{equation}for some constant $C>0$, and therefore, using the Markov inequality and the fact that $\bld{c}\in \ell^3_{\shortarrow}\setminus \ell^2_{\shortarrow}$, it follows that, for any $\varepsilon>0$,
\begin{equation}\label{c3:eq:delete-K-susc}
\lim_{K\to\infty}\limsup_{n\to\infty}\PR\bigg(\sum_{i\geq 1}\big(\mathscr{W}_{\sss (i)}^{\sss K}\big)^3>\varepsilon n^{3(\alpha+\delta)}\bigg) = 0.
\end{equation}  
    Now, the proof is complete by  observing that $
   \sum_{i>K}(\mathscr{W}_{\sss (i)})^3\leq \sum_{i\geq 1}(\mathscr{W}_{\sss (i)}^{\sss K})^3.$
  \end{proof}
\begin{remark}\label{c3:rem:3rd-suscep}\normalfont Notice that the proof of Proposition~\ref{c3:prop:tail3-bare-subcrit} can be modified to conclude the similar results for  $\sum_{i>K}\big(\mathscr{W}_{\sss (i)}\big)^2|\mathscr{C}_{\sss (i)}'|$ and  $\sum_{i>K}\big(\mathscr{W}_{\sss (i)}\big)^2|\mathscr{C}_{\sss (i)}'|$.
%
Indeed, \eqref{c3:est:3rd-susc-order} can be replaced by observing that the following identities hold:
$$ \E[\sum_{i\geq 1}(\mathscr{W}_{\sss (i)}^{\sss K})^2|\mathscr{C}_{\sss (i)}'^{\sss K}|] = n  \E[\mathscr{W}^{\sss K}(V_n)], \quad \E[\sum_{i\geq 1}\mathscr{W}_{\sss (i)}^{\sss K}(|\mathscr{C}_{\sss (i)}'^{\sss K}|)^2] = \ell_n^w \E[|\mathscr{C}'^{\sss K}(V_n^*)|^2].$$
\end{remark}
\subsection{Barely sub-critical masses}\label{c3:sec:barely-subcritical-mass}
We only prove the asymptotics of $\mathscr{W}_{\sss (i)}$ in Theorem~\ref{c3:thm:susceptibility}.
Then the asymptotics of $s_3^\star$  follow by a direct application of Proposition~\ref{c3:prop:tail3-bare-subcrit}. 
The idea is to obtain the asymptotics for $\mathscr{W}(j)$ for each fixed $j$.
 We will see that Proposition~\ref{c3:prop:tail3-bare-subcrit} implies that $\mathscr{W}_{\sss (j)}=\mathscr{W}(j)$ with high probability. Consider the breadth-first exploration of the graph  starting from vertex $j$ as follows:
 \begin{algo} \label{c3:algo:explor-barely-sub}\normalfont
 The algorithm carries along three disjoint sets of half-edges: \emph{active, neutral, dead}. 
\begin{itemize}
\item[(S0)] At stage $i=0$, the half-edges incident to $j$ are active and all the other half-edges are neutral. Order the initially active half-edges arbitrarily.
\item[(S1)]  At each stage, take the \emph{largest} half-edge $e$ and pair it with another half-edge $f$, chosen uniformly at random from the set of half-edges that are either active or neutral. If $f$ is neutral, then the vertex $v$ to which $f$ is incident, is not discovered yet. 
Declare the half-edges incident to $v$ to be active and \emph{larger} than all other active vertices (choose any order between the half-edges of $v$). 
Declare $e,f$ to be dead.
\item[(S2)] Repeat from (S1) until the set of active half-edges is empty.
\end{itemize}
\end{algo}
\noindent Define  the process $\mathbf{S}_n^j$ by
$S_n^j(l)= S_n^j(l-1)+d_{\sss (l)}'J_l-2,$ and $S_n^j(0)= d_j'$,
where $J_l$ is the indicator that a new vertex is discovered at time $l$ and $d_{\sss (l)}'$ is the degree of the discovered vertex, if any. Let $L:=\inf\{l\geq 1:S_n^j(l)=0\}$. 
 By convention, we assume that $S_n^j(l)=0$ for $l>L$. 
 Let $\mathscr{V}_l$ denote the vertex set discovered upto time $l$ excluding $j$ and $\mathcal{I}_i^n(l):=\ind{i\in\mathscr{V}_l}$. Define $\mathcal{I}_j^n(l)\equiv 0$. 
 Also, let $\mathscr{F}_l$ denote the sigma-field containing all the information upto time $l$ in Algorithm~\ref{c3:algo:explor-barely-sub}. 
Note that 
   \begin{equation}
   \begin{split}
    S_n^j(l)&= d_j'+\sum_{i\in [n]} d_i' \mathcal{I}_i^n(l)-2l=d_j'+\sum_{i\in [n]} d_i' \bigg( \mathcal{I}_i^n(l)-\frac{d_i'}{\ell_n'}l\bigg)+\left( \nu_n'-1\right)l.
    \end{split}
   \end{equation}  
    Consider the re-scaled process $\bar{\mathbf{S}}^j_n$ defined as $\bar{S}^j_n(t)= n^{-\alpha}S_n^j(\floor{tn^{\alpha+\delta} })$. Then, using Assumption~\ref{c3:assumption-w},
   \begin{equation} \label{c3:eqn::scaled_process_j}
    \bar{S}_n^j(t)= c_j+n^{-\alpha} \sum_{i\in [n]}d_i'\bigg( \mathcal{I}_i^n(tn^{\alpha+\delta})-\frac{d_i'}{\ell_n'}tn^{\alpha+\delta} \bigg)- \lambda_0 t +o(1).
   \end{equation}The following three lemmas determine the asymptotics of $\mathscr{W}_{\sss (j)}$ and $s_3^\star$: 
\begin{lemma} \label{c3:lem:exploration::subcritical-j}
 Let $L_j$ be the function with  $L_j(t)=c_j-\lambda_0t$ for $t\in [0,c_j\lambda_0^{-1}]$ and $L_j(t)=0$ for $t>c_j\lambda_0^{-1}$. Then, under \textrm{Assumption~\ref{c3:assumption-w}}, as $n\to\infty$,
  $\bar{\mathbf{S}}^j_n \xrightarrow{\sss\PR} L_j$ with respect to the Skorohod $J_1$ topology. 
\end{lemma}
\begin{lemma}\label{c3:lem:weight-prop-l} For any $T>0$,
$\sup_{l\leq Tn^{\alpha+\delta}}\Big|\sum_{i\in [n]}w_i\mathcal{I}_i^n(l)- \frac{\sum_{i\in [n]}d_i'w_i}{\sum_{i\in [n]}d_i'}l\Big|=\oP(n^{\alpha+\delta}).$
\end{lemma}
\begin{lemma}\label{c3:lem:hdeg-hweight}
 Fix any $j\geq 1$. Then with high probability $\mathscr{W}(j) = \mathscr{W}_{\sss (j)}$.
\end{lemma}
\begin{proof}[Asymptotics of $\mathscr{W}_{\sss (j)}$]Note that, since the exploration process explores one edge at each time, Lemma~\ref{c3:lem:exploration::subcritical-j} implies that (see e.g. \cite[Theorem 13.6.4]{W02})
\begin{equation}\label{c3:eq:weight-j-subcrit-1}
 \frac{1}{2n^{\alpha+\delta}}\sum_{k\in \mathscr{C}'(j)}d_k'\pto \frac{c_j}{\lambda_0}.
\end{equation}Moreover, Lemma~\ref{c3:lem:weight-prop-l} yields that 
\begin{eq}\label{c3:eq:weight-j-subcrit-2}
\frac{1}{n^{\alpha+\delta}}\mathscr{W}(j)&= \frac{1}{n^{\alpha+\delta}}\sum_{k\in \mathscr{C}'(j)}w_k \\
&= \frac{\sum_{i\in [n]}d_i'w_i}{\ell_n'n^{\alpha+\delta}}\frac{1}{2}\sum_{k\in \mathscr{C}'(j)}d_k'+\oP(1) \pto \frac{ \mu_{d,w}}{\mu_d \lambda_0}c_j.
\end{eq}Now the asymptotics of $\mathscr{W}_{\sss (j)}$ in Theorem~\ref{c3:thm:susceptibility} follows by an application of Lemma~\ref{c3:lem:hdeg-hweight} under Assumption~\ref{c3:assumption-w}.
\end{proof}
Next we provide a proof for Lemma~\ref{c3:lem:hdeg-hweight} subject to Lemmas~\ref{c3:lem:exploration::subcritical-j}, ~\ref{c3:lem:weight-prop-l}. 
The proofs of Lemmas~\ref{c3:lem:exploration::subcritical-j} and~\ref{c3:lem:weight-prop-l} are similar to \cite[Section 4]{DHLS16} and are provided in Appendix~\ref{c3:sec:appendix-barely-subcrit}.
\begin{proof}[Proof of Lemma~\ref{c3:lem:hdeg-hweight}]
To simplify the writing, we only give a proof for $j=1$, the general case follows similarly. 
Define the event $\mathrm{A}_i:= \{\mathscr{W}(i)>\mathscr{W}(1), i\notin \cup_{k\leq i-1}\mathscr{C}'(k)\}$ and let $\mathrm{A}_{\sss >K}=\cup_{i>K}\mathrm{A}_i$.
Fix $r$ such that $c_2<r\lambda_0\mu_d/\mu_{d,w}<c_1$ and define the event $\mathrm{B}=\{\mathscr{W}(1)>rn^{\alpha+\delta}\}$. 
Then, for any $K\geq 2$,
\begin{equation}\label{c3:eq:weight-large-split}
\prob{\mathscr{W}(1) \neq \mathscr{W}_{\sss (1)}}
\leq \sum_{i=2}^K\prob{\mathrm{A}_i\cap \mathrm{B}} + \prob{\mathrm{A}_{\sss >K}\cap \mathrm{B}} +\prob{\mathrm{B}^c}.
\end{equation}
Firstly, notice that due to the choice of $r$, \eqref{c3:eq:weight-j-subcrit-1} and \eqref{c3:eq:weight-j-subcrit-2} implies that $\PR(\mathrm{B}^c)\to 0$. 
Moreover, for each fixed $i\geq 2$, $\prob{\mathrm{A}_i\cap \mathrm{B}}\leq \PR(\mathscr{W}(i)>rn^{\alpha+\delta})\to 0$.
Further, recall \eqref{c3:eq:delete-K-susc} and the relevant notation. 
Note that $$
\prob{\mathrm{A}_{\sss >K}\cap \mathrm{B}}\leq \PR\bigg(\sum_{i\geq 1}\big(\mathscr{W}_{\sss (i)}^{\sss K}\big)^3>r^3n^{3(\alpha+\delta)}\bigg).$$
Thus, the proof follows from \eqref{c3:eq:weight-large-split} by taking first the limit as $n\to\infty$, and then as $K\to\infty$ and using \eqref{c3:eq:delete-K-susc}.
\end{proof}
\subsection{Mesoscopic typical distances}
Recall the definition of $\mathcal{D}_n^\star$ from \eqref{c3:defn:susc-dist}.
 In this section, we obtain the asymtotics of $\mathcal{D}_n^\star$ in Theorem~\ref{c3:thm:susceptibility} using a similar analysis as in Section~\ref{c3:sec:s2}. 
Again the proof involves the Chebyshev inequality where the moments are estimated using path counting. 
We sketch the computation of $\E[\mathcal{D}_n^\star]$.
%
Recall the notations $U_n^*$, $V_n^*$, 
$\mathcal{A}(k,l)$ 
and $\mathcal{A}'(k,l)$ 
from Section~\ref{c3:sec:s2}. 
Note that
\begin{equation}
\begin{split}
 \expt{\mathcal{D}_n^\star} &= \frac{1}{n}\E\bigg[\sum_{i,k\in [n]}w_iw_k\dst(i,k)\ind{k\in\mathscr{C}'(i)}\bigg]\\
 &\leq  \frac{\ell_n^w}{n}\sum_{k\in [n]}w_k\sum_{l\geq 1}l\prob{\mathcal{A}(k,l)}=   \frac{\ell_n^w}{n} \sum_{l\geq 1} l\sum_{k\in [n]}w_k\prob{\mathcal{A}(k,l)},
\end{split}
\end{equation}
and
\begin{equation}
\expt{\mathcal{D}_n^\star} \geq \frac{\ell_n^w}{n} \sum_{l\geq 1} l \bigg(\sum_{k\in [n]}w_k\big(\prob{\mathcal{A}(k,l)}-\prob{\mathcal{A}'(k,l)}\big)\bigg).
\end{equation}
Now compare the terms above  to \eqref{c3:sus1-expr-1}, \eqref{c3:sus1-expr-2}. The only difference is that there is an extra multiplicative $l$ here which amounts to differentiating  with respect to $\nu_n'$ in the obtained bounds. Thus, we can repeat an argument identical to \eqref{c3:sus-simpl-1}, \eqref{c3:susc-lb2} to obtain that
\begin{equation}
 \expt{\mathcal{D}^\star_n} = \frac{\expt{W_n}\expt{D_n^*}\expt{D_n'W_n}}{\expt{D_n'}(1-\nu_n')^2}(1+o(1))= \frac{(\expt{D_n'W_n})^2}{\expt{D_n'}(1-\nu_n')^2}(1+o(1)).
\end{equation}
The variance terms can also be computed similarly. Due to the presence of a factor $l^2$ in the second moment, we have to differentiate the upper-bounds twice with respect to $\nu_n'$. Again, the identical arguments as \eqref{c3:weight-conditioned-comp} can be applied to show that $\var{\mathcal{D}_n^{\star}}=o(n^{4\delta})$. 
This completes the proof of the asymptotics of $\mathcal{D}_n^{\star}$.
\subsection{Maximum diameter: Proof of Theorem~\ref{c3:thm:diam-max}}
Firstly, let us investigate the diameter of $\mathscr{C}'(i)$. 
Notice that, if $\Delta(\mathscr{C}'(i))>6n^{\delta}\log(n)$, then there exists at least one path of length at least $3n^{\delta}\log(n)$ starting from $i$. 
Now, the expected number of such paths is at most \linebreak $\sum_{l=3n^{\delta}\log(n)}^n \expt{P_l}$, where $P_l$ denotes the number of paths of length $l$, starting from vertex $i$ and we have used the fact that a vertex disjoint path can be of size at most $n$. Again, the path-counting technique yields $\expt{P_l}\leq d_i'\nu_n'^{l-1}$. Thus, for some constant $C>0$,
\begin{eq}  
  \prob{\Delta(\mathscr{C}'(i))>6n^{\delta}\log(n)}&\leq \sum_{l=3n^{\delta}\log(n)}^n\expt{P_l}\\
  &\leq C d_i' n (\nu_n')^{3n^{\delta}\log(n)}\leq C\frac{d_i'}{n^2},
  \end{eq}
where in the last step we have used \eqref{c3:defn:barely-subcrit}.
Thus, the proof of Theorem~\ref{c3:thm:diam-max} follows using the union bound.
%
\section{Metric space limit for percolation clusters}
\label{c3:sec:proof-metric-mc}
Finally, the aim of this section is to complete the proof of Theorem~\ref{c3:thm:main}. 
We start by defining the multiplicative coalescent process \cite{A97,AL98} that will play a pivotal role in this section:
\begin{defn}[Multiplicative coalescent]\label{c3:defn:mul-coalescent} \normalfont
Consider a (possibly infinite) collection of particles and let $\mathbf{X}(s)=(X_i(s))_{i\geq 1}$ denote the collection of masses of those particles at time $s$. 
Thus the $i$-th particle has mass $X_i(s)$ at time~$s$. 
The evolution of the system takes place according to the following rule at time $s$: At rate $X_i(s)X_j(s)$,  particles $i$ and $j$ merge into a new particle of mass $X_i(s)+X_j(s)$.
\end{defn}
Before going into the details, let us describe the general idea and the organization of this section. 
Extending the approach of \cite{BBSX14}, we  consider a dynamically growing process of graphs that approximates the percolation clusters in the critical window (see Chapter~\ref{chap:thirdmoment}). 
Now, the graphs generated by this dynamic evolution satisfy: 
(i) In the critical window, the components merge \emph{approximately} as the multiplicative coalescent where the \emph{mass} of each component is approximately proportional to the component size; 
(ii) the masses of the barely  sub-critical clusters satisfy \emph{nice} properties due to Theorem~\ref{c3:thm:susceptibility}. 
In Section~\ref{c3:sec:entrance-boundary-open-he}, we derive the required properties in the barely subcritical regime for the dynamically growing graph process using Theorems~\ref{c3:thm:susceptibility}~and~\ref{c3:thm:diam-max}.
In Section~\ref{c3:sec:coupling-mul-coal}, we modify the dynamic process such that the components merge exactly as multiplicative coalescent.
Since the exact multiplicative coalescent corresponds to the rank-one inhomogeneous case, thinking of these barely subcritical clusters as blobs, we use the universality theorem (Theorem~\ref{c3:thm:univesalty}) in Section~\ref{c3:sec:modified-graph} to determine the metric space limits of the largest components of the modified graph (Theorem~\ref{c3:thm:mspace-limit-modified}). 
Section~\ref{c3:sec:struct-compare} is devoted to the structural comparison of the modified graph and the original graph, and we finally complete the proof of Theorems~\ref{c3:thm:main}. 
   in Section~\ref{c3:sec:proof-thm1}.
  The proof of Theorem~\ref{c3:thm:main-simple} is given in Section~\ref{c3:sec:simple}.

\begin{algo}[The dynamic construction]\label{c3:algo:dyn-cons} \normalfont Let $\mathcal{G}_n(t)$ be the graph obtained up to  time $t$ by the following dynamic construction:
\begin{itemize}
 \item[(S0)] A half-edge can either be alive or dead. Initially, all the half-edges are alive. All the half-edges have an independent unit rate  exponential clock attached to them.
 \item[(S1)] Whenever a clock rings, we take the corresponding half-edge, kill it and pair it with a half-edge chosen uniformly at random among the alive half-edges. The paired half-edge is also killed and the exponential clocks associated with killed half-edges are discarded.
\end{itemize} 
\end{algo}  Since a half-edge is paired with another unpaired half-edge, chosen uniformly at random from the set of all unpaired half-edges, the final graph $\mathcal{G}_n(\infty)$ is distributed as $\mathrm{CM}_n(\boldsymbol{d})$. 
Define
\begin{equation}\label{c3:eq:def-crit-time}
 t_c(\lambda)=\frac{1}{2}\log\bigg(\frac{\nu_n}{\nu_n-1}\bigg)+\frac{\nu_n}{2(\nu_n-1)}\frac{\lambda}{n^\eta}.
\end{equation}
We denote the $i$-th largest component of $\mathcal{G}_n(t)$ by $\mathscr{C}_{\sss (i)}(t)$.
In the subsequent part of this chapter, we will derive the metric space limit of $(\mathscr{C}_{\sss (i)}(t_c(\lambda)))_{i\geq 1}$. 
The following lemma (see Proposition~\ref{c2prop:coupling-whp}) enables us to switch to the conclusions for the largest clusters of $\rCM_n(\bld{d},p_n(\lambda))$:
\begin{lemma}[{\cite[Proposition 24]{DHLS16}}]\label{c3:lem:coupling-whp} There exists $\varepsilon_{n}=o(n^{-\eta})$ and a coupling such that, with high probability,
\begin{gather*}
 \mathcal{G}_n(t_c(\lambda)-\varepsilon_{n})\subset \mathrm{CM}_n(\bld{d},p_n(\lambda)) \subset\mathcal{G}_n(t_c(\lambda)+\varepsilon_n),\\
  \mathrm{CM}_n(\bld{d},p_n(\lambda)-\varepsilon_n) \subset \mathcal{G}_n(t_c(\lambda)\subset \mathrm{CM}_n(\bld{d},p_n(\lambda)+\varepsilon_n).
  \end{gather*}
\end{lemma}
Let $\omega_i(t)$ denote the number of unpaired/open half-edges incident to  vertex~$i$ at time $t$ in Algorithm~\ref{c3:algo:dyn-cons}. 
We end this section by understanding the evolution of some functionals of the degrees and the open half-edges in the graph $\mathcal{G}_n(t)$.
Let $s_1(t)$ denote the total number of unpaired half-edges at time $t$. Denote also $s_2(t)=\sum_{i\in [n]} \omega_i(t)^2$, $s_{d,\omega}(t)=\sum_{i\in [n]}d_i\omega_i(t)$. 
Further, we write $\mu_n=\ell_n/n$.
\begin{lemma}\label{c3:lem:total-open-he}  Under \textrm{Assumption~\ref{c3:assumption1}}, the quantities 
 $\sup_{t\leq T}|\frac{1}{n}s_1(t)-\mu_n\e^{-2t}|$, $\sup_{t\leq T}|\frac{1}{n} s_2(t) - \mu_n\e^{-4t}(\nu_n+\e^{2t})|,$ $\sup_{t\leq T}|\frac{1}{n}s_{d,\omega}(t) - \mu_n(1+\nu_n)\e^{-2t}|\}$ are  $\OP(n^{-1/2})$, for any $T>0$.
\end{lemma}
\begin{proof}
 The proof uses the differential equation method \cite{wormald1995differential}. 
 Notice that, after each exponential clock rings in Algorithm~\ref{c3:algo:dyn-cons}, $s_1(t)$ decreases by two. Let $Y$ denote a unit rate Poisson process. Using the random time change representation \cite{EK86},
 \begin{equation}\label{c3:eq:diff-eqn}
  s_1(t) = \ell_n - 2 Y\bigg(\int_{0}^t s_1(u)\mathrm{d} u\bigg) = \ell_n +M_n(t)-2\int_{0}^t s_1(u)\mathrm{d} u,
 \end{equation}where $\bld{M}_n$ is a martingale. 
 Now, the quadratic variation of $\bld{M}_n$ satisfies $ \langle M_n \rangle (t)\leq 4t\ell_n = O(n),$ which implies that $\sup_{t\leq T}|M_n(t)|= \OP(\sqrt{n}).$ 
 Moreover, notice that the function $f(t)=\mu_n\e^{-2t}$ satisfies \linebreak $f(t)=\mu_n-2\int_0^tf(u)\mathrm{d}u$. 
  Therefore,
 \begin{equation}
 \begin{split}
  \sup_{t\leq T}\bigg|\frac{1}{n}s_1(t)- \mu_n\e^{-2t}\bigg| &\leq \sup_{t\leq T}\frac{|M_n(t)|}{n}+2\int_0^T\sup_{t\leq u}\bigg|\frac{1}{n}s_1(t)- \mu_n\e^{-2t}\bigg| \mathrm{d}u.
 \end{split}
 \end{equation} Using Gr\H{o}nwall's inequality \cite[Proposition 1.4]{M86}, it follows that
 \begin{equation}\label{c3:eq:diff-eqn-gronwall}
   \sup_{t\leq T}\bigg|\frac{1}{n}s_1(t)- \mu_n\e^{-2t}\bigg|\leq \e^{2T} \sup_{t\leq T}\frac{|M_n(t)|}{n} =\OP(n^{-1/2}),
 \end{equation}as required. 
For $s_2(t)$, note that if half-edges corresponding to vertices $i$ and $j$ are paired, then $s_2$ changes by $-2\omega_i-2\omega_j+2$ and if two half-edges corresponding to $i$ are paired, $s_2$ then changes by $-4\omega_i+4$. Thus,
\begin{equation}
\begin{split}
&\sum_{i\in [n]}\omega_i(t)^2\\
&=\sum_{i\in [n]}d_i^2+M_n'(t)+\int_0^t \sum_{i\neq j}\frac{\omega_i(u)\omega_j(u)(-2\omega_i(u)-2\omega_j(u)+2)}{s_1(u)-1}\\
 &\hspace{3cm}+\int_0^t\sum_{i\in [n]}\frac{\omega_i(u)(\omega_i(u)-1)(-4\omega_i(u)+4)}{s_1(u)-1}\\
 & = n\mu_n(1+\nu_n)+M_n'(t)+\int_0^t(-4s_2(u)+2s_1(u))\mathrm{d}u+\OP(1),
 \end{split}
\end{equation}where $\bld{M}_n'$ is a martingale with quadratic variation given by $\langle M_n'\rangle (t) =O(n)$. 
Again, an estimate equivalent to \eqref{c3:eq:diff-eqn-gronwall} follows using Gr\H{o}nwall's inequality.
Notice also that when a clock corresponding to vertex $i$ rings and it is paired to vertex $j$, then $s_{d,\omega}$ decreases by $d_i+d_j$. 
Thus,
 \begin{equation}
  \begin{split}
   s_{d,\omega}(t)&=\sum_{i\in [n]}d_i^2+M_n''(t)-\int_0^t \sum_{i\neq j}\frac{\omega_i(u)\omega_j(u)(d_i+d_j)}{s_1(u)-1} \mathrm{d}u\\
   &\hspace{2cm}-\int_0^t \sum_{i\in [n]}\frac{\omega_i(u)(\omega_i(u)-1)2d_i}{s_1(u)-1}\mathrm{d}u\\
   & = n\mu_n(1+\nu_n)+M_n''(t)- 2 \int_0^ts_{d,\omega}(u)\mathrm{d}u,
  \end{split}
 \end{equation}where $\bld{M}_n''$ is a martingale with quadratic variation given by $\langle M_n''\rangle (t) \leq 2t \sum_{i\in [n]}d_i^2=O(n)$. 
 The proof of Lemma~\ref{c3:lem:total-open-he} is now complete. 

\end{proof}

\subsection{Entrance boundary for open half-edges}\label{c3:sec:entrance-boundary-open-he}
Define
\begin{equation}\label{c3:eq:def-subcrit-time}
 t_n = \frac{1}{2}\log\bigg(\frac{\nu_n}{\nu_n-1}\bigg)-\frac{\nu_n}{2(\nu_n-1)}\frac{1}{n^{\delta}}, \quad 0< \delta < \eta.
\end{equation} The goal is to show that the open half-edges satisfy the entrance boundary conditions. 
Let $\bld{d}(t)=(d_i(t))_{i\in [n]}$ denote the degree sequence of $\mathcal{G}_n(t)$ constructed by Algorithm~\ref{c3:algo:dyn-cons}. Recall that $\mathcal{G}_n(t)$ is a configuration model conditionally on~$\bld{d}(t)$. 
Let us first derive the asymptotics of $\nu_n(t_n)$.
Recall that $\omega_i(t)$ denotes the number of open half-edges adjacent to vertex $i$ in $\mathcal{G}_n(t)$. Notice that 
\begin{equation}
 \nu_n(t_n)=\frac{\sum_{i\in [n]}(d_i-\omega_i(t_n))^2}{\ell_n-s_1(t_n)} -1 = \frac{\sum_{i\in [n]}d_i^2-2s_{d,\omega}(t_n)+s_2(t_n)}{\ell_n-s_1(t_n)}-1.
\end{equation}
Using Lemma~\ref{c3:lem:total-open-he} and Assumption~\ref{c3:assumption1}, 
\begin{equation}\label{c3:nu-calc-1}
 \begin{split}
  &\frac{1}{n}(\ell_n-s_1(t_n))=\mu_n (1-\e^{-2t_n}) +\oP(n^{-\delta})
    = \frac{\mu_n}{\nu_n}\Big(1-\frac{\nu_n}{n^{\delta}}\Big)+\oP(n^{-\delta}),
 \end{split}
\end{equation}
\begin{equation}\label{c3:nu-calc-2}
\begin{split}
 \frac{1}{n}\bigg(\sum_{i\in [n]}d_i^2-2s_{d,\omega}(t_n)+s_2(t_n)\bigg) 
 & = \frac{\mu_n}{\nu_n}\Big(2-\frac{3\nu_n}{n^{\delta}}\Big)+\oP(n^{-\delta}).
\end{split}
\end{equation}
Thus, \eqref{c3:nu-calc-1} and \eqref{c3:nu-calc-2} yield that
$\nu_n(t_n) 
= 1- \nu_n n^{-\delta} +\oP(n^{-\delta}).$
We aim to apply the results for the barely sub-critical regime in Theorem~\ref{c3:thm:susceptibility} to the number of open half-edges $\bld{\omega}(t_n)=(\omega_i(t_n))_{i\in [n]}$. 
Notice that, by  Lemma~\ref{c3:lem:total-open-he} and Assumption~\ref{c3:assumption1}, $\bld{\omega}(t_n)$ and $\bld{d}(t_n)$ satisfy Assumption~\ref{c3:assumption-w} with
\begin{equation}\label{c3:eq:mu-mudw-open-he}
 \mu_{\omega} = \frac{\mu(\nu-1)}{\nu},\quad \mu_d = \frac{\mu}{\nu}, \quad \mu_{d,\omega}  = \frac{\mu(\nu-1)}{\nu} ,\quad c_i = \frac{\theta_i}{\nu}.
\end{equation} 
Consider the quantities $s_2^\star$, $s_3^\star$, $\mathcal{D}_n^\star$ with the weights being the number of open half-edges and denote them by $s_2^\omega$, $s_3^\omega$, $\mathcal{D}_n^\omega$ respectively.
Denote $f_i(t)= \sum_{k\in \mathscr{C}_{\sss (i)}(t)}\omega_k(t)$ and $\bld{f}(t) = (f_i(t))_{i\geq 1}$.
The following theorem summarizes the entrance boundary conditions for $\bld{f}(t)$: 

\begin{theorem}\label{c3:th:open-he-entrance}Under \textrm{Assumption~\ref{c3:assumption1}}, as $n\to\infty$,
\begin{gather*}
 n^{-\delta}s_2^\omega \pto \frac{\mu (\nu-1)^2}{\nu^2},\quad n^{-\delta}s_{pr}^\omega \pto\frac{\mu (\nu-1)}{\nu^2}, \quad 
 n^{-(\alpha+\delta)}f_i(t_n) \pto \big(\frac{\nu-1}{\nu^2}\big)\theta_i, \\
  n^{-3\alpha-3\delta+1}s_3^\omega \pto \big(\frac{\nu-1}{\nu^2}\big)^3 \sum_{i=1}^\infty \theta_i^3, \quad n^{-2\delta} \mathcal{D}_n^\omega \pto \frac{\mu (\nu-1)^2}{\nu^3}.
 \end{gather*}
\end{theorem}
\begin{remark}\label{c3:rem:entrance-comp-size} \normalfont Setting $w_i =1$ for all $i$, we get the entrance boundary conditions for the component sizes also. In this case $\mu_d=\mu_{d,w}=\mu/\nu$. Augmenting a predefined notation with $c$ in the superscript to denote the component susceptibilities, it follows that 
\begin{align*}
 n^{-\delta}s_2^c \pto \frac{\mu}{\nu^2},\quad n^{-(\alpha+\delta)}|\mathscr{C}_{\sss (i)}(t_n)|\pto \frac{\theta_i}{\nu^2},\quad n^{-3\alpha-3\delta+1}s_3^c\pto \frac{1}{\nu^6}\sum_{i=1}^\infty \theta_i^3.
\end{align*}
\end{remark}
\subsection{Coupling with the multiplicative coalescent}
\label{c3:sec:coupling-mul-coal}
Recall the definitions of $t_c(\lambda)$ and $t_n$ from \eqref{c3:eq:def-crit-time} and \eqref{c3:eq:def-subcrit-time}.  
 Now, let us investigate the dynamics of $\bld{f}(t) $ starting from time $t_n$. Notice that, in the time interval $ [t_n,t_c(\lambda)]$, components with masses $f_i(t)$ and $f_j(t)$ merge at rate 
 \begin{equation}
  f_i(t)\frac{f_j(t)}{s_1(t)-1}+f_j(t)\frac{f_i(t)}{s_1(t)-1}=\frac{2f_i(t)f_j(t)}{s_1(t)-1}\approx \frac{2\nu f_i(t)f_j(t)}{\mu(\nu-1)n}, 
 \end{equation}and create a component with $f_i(t)+f_j(t)-2$ open half-edges. 
 Thus $\bld{f}(t)$ does not \emph{exactly} evolve as a multiplicative coalescent, but it is close. 
 Now, we define an exact multiplicative coalescent that approximates the above process:
 \begin{algo}[Modified process] \label{c3:algo:modified-MC} \normalfont Conditionally on $\mathcal{G}_n(t_n)$, associate a rate $2/(s_1(t_n)-1)$ Poisson process $\mathcal{P}(e,f)$ to each of pair of unpaired-half-edges $(e,f)$. 
An edge $(e,f)$ is created between the vertices incident to $e$ and $f$ at the instance when $\mathcal{P}(e,f)$ rings. 
We denote the graph obtained at time $t$ by $\bar{\mathcal{G}}_n(t)$. 
 \end{algo}
\begin{proposition}\label{c3:prop-coupling-order}
There exists a coupling such that $\mathcal{G}_n(t)\subset\bar{\mathcal{G}}_n(t)$ for all $t>t_n$ with probability one.
\end{proposition}
\begin{proof}
Recall the construction of $\mathcal{G}_n(t)$ from Algorithm~\ref{c3:algo:dyn-cons}. 
We modify (S1) as follows: whenever two half-edges are paired, we do not kill the corresponding half-edges and do not discard the associated exponential clocks. 
Instead we reset the corresponding exponential clocks. 
The graphs generated by this modification of Algorithm~\ref{c3:algo:dyn-cons} has the same distribution as $\bar{\mathcal{G}}_n(t)$, conditionally on $\mathcal{G}_n(t_n)$.
Moreover, the above also gives a natural coupling such that $\mathcal{G}_n(t)\subset\bar{\mathcal{G}}_n(t)$, by viewing the event times of Algorithm~\ref{c3:algo:dyn-cons} as a thinning of the event times of the modified process.
\end{proof}
Henceforth, we will always assume that we are working on a probability space such that Proposition~\ref{c3:prop-coupling-order} holds.
 The connected components at time $t_n$, $(\mathscr{C}_{\sss (i)}(t_n))_{i\geq 1}$ are regarded as blobs.
Thus, for $t\geq t_n$, the graph $\bar{\mathcal{G}}_n(t)$ should be viewed as a super-graph with the superstructure being determined by the edges appearing after time $t_n$ in Algorithm~\ref{c3:algo:modified-MC}.  
  Let us denote the ordered connected components of $\bar{\mathcal{G}}_n(t)$ by $(\bar{\mathscr{C}}_{\sss (i)}(t))_{i\geq 1}$. 
  The components of $\bar{\mathcal{G}}_n(t)$ can be regarded as a union of the blobs. 
  For a component $\mathscr{C}$, we use the notation 
  $\bl(\mathscr{C})$ to denote the collection of indices corresponding to the blobs within $\mathscr{C}$ given by $\{b: \mathscr{C}_{\sss (b)}(t_n)\subset\mathscr{C}\}$.
  Denote 
 $$ \bar{\mathcal{F}}_{i}(t) = \sum_{b\in \bl(\bar{\mathscr{C}}_{\sss (i)}(t))}f_{b}(t_n).$$
 The $\bar{\mathcal{F}}$-value is regarded as the mass of component $\bar{\mathscr{C}}_{\sss (i)}(t)$ at time $t$.
Note that for the modified process in Algorithm~\ref{c3:algo:modified-MC}, conditionally on $\mathcal{G}_n(t_n)$, at time $t\in[t_n,t_c(\lambda)]$, $\bar{\mathscr{C}}_{\sss (i)}(t)$ and $\bar{\mathscr{C}}_{\sss (j)}(t)$ merge at exact rate $2 \bar{\mathcal{F}}_i(t)\bar{\mathcal{F}}_j(t)/(s_1(t_n)-1)$ and the new component has mass $\bar{\mathcal{F}}_i(t)+\bar{\mathcal{F}}_j(t)$. 
Thus, the vector of masses $(\bar{\mathcal{F}}_i(t))_{i\geq 1}$ merge as an exact multiplicative coalescent.

\subsection{Properties of the modified process}
\label{c3:sec:modified-graph}
Notice that, conditionally on $\mathcal{G}_n(t_n)$, blobs $b_i$ and $b_j$ are connected in $\bar{\mathcal{G}}_n(t_c(\lambda))$ with probability $p_{ij}$ equal to
\begin{equation}\label{c3:eq:pij-value}
 1- \exp\Big(-f_{b_i}(t_n) f_{b_j}(t_n)\Big[\frac{1}{n^{1+\delta}}\frac{\nu^2 }{\mu(\nu-1)^2} + \frac{1}{n^{1+\eta}}\frac{\nu^2}{\mu(\nu-1)^2}\lambda\Big] (1+\oP(1))\Big),
\end{equation}
where the $\oP(\cdot)$ term appearing above is uniform in $i,j$.
Thus, using Theorem~\ref{c3:th:open-he-entrance}, \eqref{c3:eq:pij-value} is of the form $1-\e^{-qx_i^nx_j^n(1+\oP(1))}$ with
\begin{equation}\label{c3:eq:parameters-inhom}
 x_i^n = n^{-\rho}f_{b_i}(t_n), \quad q = \frac{1}{\sigma_2(\bld{x}^n)} + \frac{\nu^2}{\mu(\nu-1)^2}\lambda,
\end{equation}where $\sigma_r(\bld{x}^n) = \sum (x_i^n)^r$. 
By Theorem~\ref{c3:thm:susceptibility}, the sequence $\bld{x}^n$ satisfies the entrance boundary conditions of \cite{AL98}, i.e.,
\begin{equation}
\begin{split}
 \frac{\sigma_3(\bld{x}^n)}{(\sigma_2(\bld{x}^n))^3} \pto \frac{1}{\mu^3(\nu-1)^3}\sum_{i=1}^\infty \theta_i^3,\quad \frac{x_i^n}{\sigma_2(\bld{x}^n)}\pto \frac{1}{\mu(\nu-1)}\theta_i, \quad \sigma_2(\bld{x}^n)\pto 0.
 \end{split}
\end{equation}
To simplify the notation, we write $\bar{\mathcal{F}}_i(\lambda)$ for  $\bar{\mathcal{F}}_i(t_c(\lambda))$ and $\bar{\mathscr{C}}_{\sss (i)}(\lambda)$ for $\bar{\mathscr{C}}_{\sss (i)}(t_c(\lambda))$.
The following result is a consequence of \cite[Proposition~7]{AL98} and Lemma~\ref{c3:lem:rescale}:
\begin{proposition}\label{c3:thm:modified-open-he-limit} As $n\to\infty$, $\big(n^{-\rho}\bar{\mathcal{F}}_i(\lambda)\big)_{i\geq 1} \dto \frac{\nu-1}{\nu}\bld{\xi}$
with respect to the $\ell^2_{\shortarrow}$ topology, where $\bld{\xi}$ is defined in \textrm{Proposition~\ref{c3:prop:comp-size}}.
\end{proposition}
\noindent We next relate $(\bar{\mathcal{F}}_i(\lambda))_{i\geq 1}$ to $(\bar{\mathscr{C}}_{\sss (i)}(\lambda))_{i\geq 1}$, for each fixed $i$:
\begin{proposition}\label{c3:thm:mod-comp-openhe}
 As $n\to\infty$, 
  $\bar{\mathcal{F}}_i(\lambda) = (\nu-1)|\bar{\mathscr{C}}_{\sss (i)}(\lambda)| + \oP(n^{\rho}).$
 Consequently,
  $\big(n^{-\rho}|\bar{\mathscr{C}}_{\sss (i)}(\lambda)|\big)_{i\geq 1} \dto \frac{1}{\nu}\bld{\xi}$
with respect to the product topology.
\end{proposition}
\noindent We will need the following lemma, the proof of which is same as \cite[Lemma 8.2]{BSW14}. 
\begin{lemma}[{\cite[Lemma 8.2]{BSW14}}]\label{c3:lem:size-biased} Consider two ordered weight sequences $\bld{x} = (x_i)_{i\in [m]} $ and $\bld{y} = (y_i)_{i\in [m]}$. 
Consider the size-biased reordering $(v(1),v(2),\dots)$ of  $[m]$ with respect to the weights $\bld{x}$ and let $V(i):= \{v(1),\dots,v(i)\}$. Denote $m_{rs} = \sum_i x_i^{r}y_i^{s}$, define $c_n = m_{11}/m_{10}$ and assume that $c_n>0$ for each $n$. 
Suppose that the following conditions hold:
\begin{equation}\label{c3:eq:size-biased-conditions}
 \frac{lm_{21}}{m_{10}m_{11}} \to 0, \quad \frac{m_{12}m_{10}}{lm_{11}^2}\to 0,\quad \frac{lm_{20}}{m_{10}^2}\to 0, \quad \text{ as }n\to\infty.
\end{equation}
Then, as $n\to\infty$,
 $\sup_{k\leq l}\big|\frac{1}{lc_n}\sum_i y_i \ind{i\in V(k)}-\frac{k}{l}\big| \pto 0.$
\end{lemma}
\begin{proof}[Proof of Proposition~\ref{c3:thm:mod-comp-openhe}]
We only prove the asymptotic relation of $\bar{\mathcal{F}}_1(\lambda)$ and $|\mathscr{C}_{\sss (1)}(\lambda)|$.
Consider the breadth-first exploration of the \emph{supestructure} of graph $\bar{\mathcal{G}}_n(t_c(\lambda))$ (which is also a rank-one inhomogeneous random graph) using the Aldous-Limic construction from \cite[Section 2.3]{AL98}. 
Notice that the vertices are explored in a size-biased manner with the sizes being $\bld{x} = (x_i)_{i\geq 1}$, where $x_i = n^{-\rho} f_i(t_n)$.  
Let $v(i)$ be the $i$-th vertex explored.
Further, let $\bar{\mathscr{C}}_{\sss (i)}^{\sss \mathrm{st}}(\lambda)$ denote the component $\bar{\mathscr{C}}_{\sss (i)}(\lambda)$, where the blobs have been shrunk to single vertices.
Then, from \cite{AL98}, one has the following:
\begin{enumerate}[(i)]
\item there exists random variables $m_L,m_R$ such that $\bar{\mathscr{C}}_{\sss (i)}^{\sss \mathrm{st}}(\lambda)$ is explored between $m_L+1$ and $m_R$;
\item $\sum_{i\leq m_R}x_{\sss v(i)}$ is tight;
\item $\sum_{i=m_L+1}^{m_R}x_{\sss v(i)} \dto \gamma$, where $\gamma$ is some non-degenerate, positive random variable. 
\end{enumerate}  
Let $y_i = n^{-\rho}|\mathscr{C}_{\sss (i)}(t_n)|$. Using  Theorem~\ref{c3:th:open-he-entrance},  Remark~\ref{c3:rem:3rd-suscep} and Remark~\ref{c3:rem:entrance-comp-size}, it follows that
 $\sum_i x_i^{r} y_i^{s} = \OP(n^{3\delta-3\eta});$ for $r+s=3$, $\sum_ix_i = \OP(n^{1-\rho}),$ and 
 $\sum_{i} x_i^{r} y_i^{s} = \OP(n^{-2\rho+1+\delta});$ for  $r+s = 2$.
Below, we show that 
\begin{equation}\label{c3:eqn:ratio-comp-he}
 \frac{\sum_{i = m_L+1}^{m_R}y_{\sss v(i)}}{\sum_{i = m_L+1}^{m_R}x_{\sss v(i)}}\times \frac{\sum_ix_i^2}{\sum_i x_i y_i} \pto 1.
\end{equation} The proof of Proposition~\ref{c3:thm:mod-comp-openhe} follows from \eqref{c3:eqn:ratio-comp-he} by observing that
$ \frac{\sum_{i}x_i^2}{\sum_ix_iy_i} = \frac{s_2^\omega(t_n)}{s_{pr}^\omega(t_n)}\pto \nu-1,$ and using Theorem~\ref{c3:th:open-he-entrance}.
To prove \eqref{c3:eqn:ratio-comp-he}, we will now apply Lemma~\ref{c3:lem:size-biased}.  Denote $m_0 = \sum_i x_i/\sum_ix_i^2$ and consider $l = 2Tm_0$ for some fixed $T>0$.
 Using Theorem~\ref{c3:th:open-he-entrance}, an application of Lemma~\ref{c3:lem:size-biased} yields
\begin{align*}
 \sup_{k\leq 2Tm_0} \bigg|\sum_{i=1}^k x_{\sss v(i)} - \frac{k}{m_0}\bigg|\pto 0.
\end{align*}
Now, for any $\varepsilon > 0$, $T>0$ can be chosen large enough such that $\sum_{i=1}^{m_R}x_{\sss v(i)}>T$ has probability at most $\varepsilon$ and on the event 
$\big\{\sup_{k\leq 2Tm_0} \big|\sum_{i=1}^k x_{\sss v(i)} - \frac{k}{m_0}\big|\leq \varepsilon \big\}\cap \big\{\sum_{i=1}^{m_R}x_{\sss v(i)}\leq T\big\},$ 
one has $m_L<m_R<2Tm_0$. Thus, it follows that 
\begin{equation}\label{c3:weight-total-larg-comp-1}
 \bigg|\sum_{i=m_L+1}^{m_R}x_{\sss v(i)} - \frac{m_R-m_L}{m_0}\bigg| \pto 0.
\end{equation}An identical argument as above shows that 
\begin{equation}\label{c3:weight-total-large-comp-2}
 \bigg|\sum_{i=m_L+1}^{m_R}y_{\sss v(i)} - \frac{m_R-m_L}{m_0'}\bigg| \pto 0,
\end{equation}where $m_0'=\sum_{i}x_i/\sum_ix_iy_i$. The proof of \eqref{c3:eqn:ratio-comp-he} now follows from \eqref{c3:weight-total-larg-comp-1} and \eqref{c3:weight-total-large-comp-2}. 
The asymptotic distribution for $(n^{-\rho}|\bar{\mathscr{C}}_{\sss (i)}|)$ can be obtained using Proposition~\ref{c3:thm:modified-open-he-limit} and  Lemma~\ref{c3:lem:rescale}.
\end{proof}
Recall that $\omega_i(t_n)$ denotes the number of open-half edges attached to vertex $i$ in the graph $\mathcal{G}_n(t_n)$. 
We now equip $\bar{\mathscr{C}}_{\sss (i)}(\lambda)$ with the probability measure $\mu_{\sss \mathrm{fr}}^i$ given by $\mu_{\sss \mathrm{fr}}^i(A) = \sum_{k\in A}\omega_k(t_n)/\mathcal{F}_i(\lambda)$ for $A\subset \bar{\mathscr{C}}_{\sss (i)}(\lambda)$, and denote the corresponding measured metric space by $\bar{\mathscr{C}}^{\sss \mathrm{fr}}_{\sss (i)}(\lambda)$.
\begin{theorem}\label{c3:thm:mspace-limit-modified} Under \textrm{Assumption~\ref{c3:assumption1}}, as $n\to\infty$,
\begin{equation}\label{c3:eq:limit-component-free}
 \big(n^{-\eta}\bar{\mathscr{C}}_{\sss (i)}^{\sss \mathrm{fr}}(\lambda)\big)_{i\geq 1} \dto (M_{i})_{i\geq 1},
\end{equation}with respect to the $\mathscr{S}_*^\N$ topology, where $M_i$ is defined in Section~\ref{c3:sec:limit-component}. 
\end{theorem}
\begin{proof}
We just consider the metric space limit of $\bar{\mathscr{C}}_{\sss (i)}^{\sss \mathrm{fr}}(\lambda)$ for each fixed $i\geq 1$ and the joint convergence in \eqref{c3:eq:limit-component-free} follows using the joint convergence of different functionals used throughout the proof. 
Recall the notation $\bl(\mathscr{C}):=\{b: \mathscr{C}_{\sss (b)}(t_n)\subset\mathscr{C}\}$ for a component $\mathscr{C}$.
Now, $\bar{\mathscr{C}}^{\sss \mathrm{fr}}_{\sss (i)}(\lambda)$ can be seen as a super-graph as defined in Section~\ref{c3:defn:super-graph} with \begin{enumerate}
\item the collection of blobs $\{\mathscr{C}_{\sss (b)}(t_n):b\in \bl(\bar{\mathscr{C}}_{\sss (i)}(\lambda))\}$ and within-blob measure $\mu_b$ given by $\mu_b(A) = \sum_{k\in A}\omega_k(t_n)/f_b(t_n)$, $A\subset \mathscr{C}_{\sss (b)}(t_n)$,  $b\in  \bl(\bar{\mathscr{C}}_{\sss (i)}(\lambda))$;
\item the superstructure consisting of the edges appearing during $[t_n,t_c(\lambda)]$ in Algorithm~\ref{c3:algo:modified-MC} and weight sequence $(f_b(t_n)/\bar{\mathcal{F}}_i(\lambda):b\in  \bl(\bar{\mathscr{C}}_{\sss (i)}(\lambda)))$.
\end{enumerate}  
Let $\dst(\cdot,\cdot)$ denote the graph distance on $\bar{\mathscr{C}}_{\sss (i)}(\lambda)$ and define 
 \begin{equation}
  u_b = \sum_{i,j\in \mathscr{C}_{\sss (b)}(t_n)} \frac{\omega_i\omega_j}{f_b^2(t_n)}\dst(i,j), \quad B_n^{\sss (i)}= \frac{\sum_{b \in \bl(\bar{\mathscr{C}}_{\sss (i)}(\lambda))}x_bu_b}{\sum_{b\in \bl(\bar{\mathscr{C}}_{\sss (i)}(\lambda))}x_b}.
 \end{equation}
 Here $u_b$ gives the average distance within blob $\mathscr{C}_{\sss (b)}(t_n)$.
Using Lemma~\ref{c3:lem:size-biased}, we will show  
   \begin{equation}\label{c3:eq:measure-limit}
    B_n^{\sss (i)} \times \frac{\sum_i x_i^2}{\sum_i x_i^2u_i}\pto 1.
   \end{equation}
The argument is the same as the proof of \eqref{c3:eqn:ratio-comp-he}. We only have to ensure that \eqref{c3:eq:size-biased-conditions} holds with $y_i = x_iu_i$. 
   Thus, we need to show that 
   \begin{equation}\label{c3:cond-check-dist-blob}
    \frac{n^{\rho-\delta}\sum_ix_i^3u_i}{\sum_ix_i\sum_ix_i^2u_i}\pto 0,\quad \text{and}\quad \frac{\sum_i x_i^3 u_i^2\sum_ix_i}{n^{\rho-\delta}\big(\sum_ix_i^2u_i\big)^2}\pto 0.
   \end{equation}First of all, notice that, by Lemma~\ref{c3:lem:total-open-he} and Theorem~\ref{c3:th:open-he-entrance},
   \begin{equation} \label{c3:order:cn}
   \begin{split}
    c_n &= \frac{\sum_i x_i^2u_i}{\sum_i x_i} = (1+\oP(1))\frac{\nu n^{-1+\rho}}{\mu(\nu-1)} n^{-2\rho} \sum_b f_b^2(t_n) \sum_{i,j\in b} \frac{\omega_i\omega_j}{f_b^2(t_n)}\dst(i,j) \\
    &=  \frac{\nu n^{-1+\rho}}{\mu(\nu-1)} n^{1-2\rho} \mathcal{D}_n^\omega  = n^{2\delta-\rho} \frac{\nu-1}{\nu^2}(1+\oP(1)). 
    \end{split} 
   \end{equation}
   Also, recall from Theorem~\ref{c3:thm:diam-max} that $u_{\max}= \max_{b}u_b =\OP(n^{\delta}\log(n))$.  
   Now, 
   \begin{eq}
    \frac{n^{\rho-\delta}\sum_ix_i^3u_i}{\sum_i x_i \sum_ix_i^2u_i} &\leq \frac{n^{\rho-\delta}u_{\max}\sum_{i}x_i^3}{\sum_{i} x_i \sum_ix_i^2u_i}  = \OP\bigg(\frac{n^{\rho-\delta}n^{\delta}\log(n)n^{-3\rho}n^{3\alpha+3\delta}}{n^{1-\rho}n^{2\delta-\rho}n^{1-\rho}}\bigg) \\
    &= \OP(n^{\delta-\eta}\log(n))=\oP(1),\\
      \frac{\sum_i x_i^3 u_i^2\sum_ix_i}{n^{\rho-\delta}\big(\sum_ix_i^2u_i\big)^2}&\leq \frac{x_{\max}u_{\max}\sum_ix_i}{n^{\rho-\delta}\sum_ix_i^2u_i} = \OP\bigg(\frac{n^{-\rho}n^{\alpha+\delta}n^{\delta}\log(n)}{n^{\rho-\delta}n^{2\delta-\rho}}\bigg) \\
      &= \OP(n^{\delta-\eta}\log(n))=\oP(1),
   \end{eq}and \eqref{c3:cond-check-dist-blob} follows, and hence the proof of \eqref{c3:eq:measure-limit} also follows.    
   Recall that the superstructure of $\bar{\mathcal{G}}_n(t_c(\lambda))$ has the same distribution as a $\mathrm{NR}_n(\bld{x}, q)$ random graph with the parameters given by \eqref{c3:eq:parameters-inhom}. 
   Thus, using Proposition~\ref{c3:prop:generate-nr-given-partition}, we now aim to use Theorem~\ref{c3:thm:univesalty} on $\mathscr{C}_{\sss (i)}^{\sss \mathrm{fr}}(\lambda)$ with the blobs being $(\mathscr{C}_{\sss (i)}(t_n))_{i\geq 1}$, and $\mathbf{p}_n^{\sss (i)}$, $a_n^{\sss (i)}$ given by \eqref{c3:eq:p-n-a-NR}. 
 Define $\Upsilon_n^{\sss (i)} = \big(p_b/\sigma(\mathbf{p}_n^{\sss (i)}):b\in \bl(\bar{\mathscr{C}}_{\sss (i)}(\lambda))\big)$. 
 Let $\mathcal{N}(\R_+)$ denote the space of all counting measures equipped with the vague topology and denote the product space $\mathbb{S} = \R_+^3\times \mathcal{N}(\R_+)$. Define 
   \begin{equation}
    \mathscr{P}_n = \Big(a_n^{\sss (i)}\sigma(\mathbf{p}_n^{\sss (i)}), \sum_{b\in \bl(\bar{\mathscr{C}}_{\sss (i)}(\lambda))}x_b, \frac{1}{\sigma_2^2(\bld{x}^n)}\sum_{b\in \bl(\bar{\mathscr{C}}_{\sss (i)}(\lambda))}x_b^2, \Upsilon_n^{\sss (i)}\Big)_{i\geq 1},
   \end{equation}  viewed as an element of $\mathbb{S}^\N$.  
  Recall the definition of $\xi_i^*$ and $\Xi_i^*$ from Section~\ref{c3:sec:limit-component}.
   Define
   {\small\begin{equation}
    \mathscr{P}^\infty = \bigg( \frac{\xi_i^*}{\mu(\nu-1)}\bigg(\sum_{v\in \Xi_i^*}\theta_v^2\bigg)^{1/2},\ \xi_i^*,\ \frac{1}{\mu^2(\nu-1)^2}\sum_{v\in \Xi_i^*}\theta_v^2,\ \bigg(\frac{\theta_j}{\sum_{v\in \Xi_i^*}\theta_v^2}: j\in \Xi_i^*\bigg) \bigg)_{i\geq 1} 
   \end{equation} }
   The following is a consequence of \cite[Proposition 5.1, Lemma 5.4]{BHS15}:
   \begin{equation}\label{c3:con-com-parameters}
    \sigma(\mathbf{p}_n^{\sss (i)})\pto 0,\quad  \text{and} \quad \mathscr{P}_n\dto \mathscr{P}_{\infty} \text{ on } \mathbb{S}^\N.
   \end{equation}  Without loss of generality, we assume that the convergence in \eqref{c3:con-com-parameters} holds almost surely. Now, using \eqref{c3:eq:measure-limit}, it follows that
   \begin{align*}
    &\frac{\sigma(\mathbf{p}_n^{\sss (i)})}{1+B_n^{\sss (i)}} = \frac{\sigma_2(\bld{x}^n)\big( \sum_{v\in \Xi_i^*}\theta_v^2\big)^{1/2}}{\mu(\nu-1)\xi_i^*}\times\frac{\sum_i x_i^2}{\sum_i x_i^2u_i}(1+o(1))\\
  &= \frac{\sigma_2^2(\bld{x}^n)\big( \sum_{v\in \Xi_i^*}\theta_v^2\big)^{1/2}}{\mu(\nu-1)\xi_i^*\sum_i x_i^2u_i} (1+o(1))= n^{-\eta}\frac{\nu-1}{\nu}  \frac{1}{\xi_i^*}\bigg( \sum_{v\in \Xi_i^*}\theta_v^2\bigg)^{1/2}(1+o(1)),
   \end{align*}    
    where the last step follows from Theorem~\ref{c3:th:open-he-entrance}, \eqref{c3:order:cn} and \eqref{c3:con-com-parameters}. The proof of Theorem~\ref{c3:thm:mspace-limit-modified} is now complete using Theorem~\ref{c3:thm:univesalty}.
\end{proof}
\subsection{Properties of the original process}
\label{c3:sec:struct-compare}
 Let us denote the ordered components of $\mathcal{G}_n(t_c(\lambda))$ simply by $(\mathscr{C}_{\sss (i)}(\lambda))_{i\geq 1}$. 
To prove Theorem~\ref{c3:thm:main}, we need to compare functionals of $\mathscr{C}_{\sss (i)}(\lambda)$ and $\bar{\mathscr{C}}_{\sss (i)}(\lambda)$ that describe the structures of these graphs. 
Firstly, the following is a direct consequence of Lemma~\ref{c3:lem:coupling-whp} and Proposition~\ref{c3:prop:comp-size}:
\begin{proposition}\label{c3:thm:comp-functionals-original} Let $(\mathscr{C}_{\sss (i)}(\lambda))_{i\geq 1}$ denote the ordered vector of components sizes of the graph $\mathcal{G}_n(t_c(\lambda))$. Then, 
$\big(n^{-\rho}|\mathscr{C}_{\sss (i)}(\lambda)|,\surp{\mathscr{C}_{\sss (i)}(\lambda)}\big)_{i\geq 1} \xrightarrow{\sss d} (\frac{1}{\nu}\xi_i,\mathscr{N}_i)_{i\geq 1}$
as $n\to\infty$,
with respect to the topology on $\ell^2_\shortarrow\times \N^\N$, where the limiting objects are defined in Proposition~\ref{c3:prop:comp-size}.
\end{proposition}
\noindent Now, conditionally on $\mathcal{G}_n(t_n)$, $\mathscr{C}_{\sss (i)}(\lambda)$ can also be viewed as consisting of blobs $(\mathscr{C}_{\sss (i)}(t_n))_{i\geq 1}$ and a superstructure connecting the blobs. 
Denote 
\begin{equation}
  \mathcal{F}_{i}(\lambda) = \sum_{b\in \bl(\mathscr{C}_{\sss (i)}(\lambda))}f_{b}(t_n).
 \end{equation} 
 The components consist of surplus edges within the blobs and the surplus edges in the superstructure. 
Let $\mathrm{SP}'(\mathscr{C}_{\sss (i)}(\lambda))$ denote the number of surplus edges in the superstructure of $\mathscr{C}_{\sss (i)}(\lambda)$. 
The following proposition relates the superstructure and components:
\begin{proposition}\label{c3:thm:comp-functional-original}
Assume that $\eta/2<\delta<\eta$. 
Then, for each $1\leq i\leq K$, the following hold:
\begin{enumerate}[(a)]
\item With high probability, $\mathrm{SP}'(\mathscr{C}_{\sss (i)}(\lambda)) = \mathrm{SP}(\mathscr{C}_{\sss (i)}(\lambda))$. Consequently, there are no surplus edges within blobs in $\mathscr{C}_{\sss (i)}(\lambda)$ with high probability;
\item $\mathcal{F}_i(\lambda)/|\mathscr{C}_{\sss (i)}(\lambda)|\pto \nu-1$. Consequently, $(n^{-\rho}\mathcal{F}_i(\lambda))_{i\geq 1}$ and $(n^{-\rho}\bar{\mathcal{F}}_i(\lambda))_{i\geq 1}$ have the same distributional limit as \textrm{Proposition~\ref{c3:thm:modified-open-he-limit}}.
\end{enumerate}
\end{proposition}
\noindent We start by explaining the idea of the proof. 
Since $\mathrm{SP}'(\mathscr{C}_{\sss (i)}(\lambda))\leq \mathrm{SP}(\mathscr{C}_{\sss (i)}(\lambda))$ almost surely, for Part (a) it suffices to show that 
 \begin{equation}\label{c3:eq:surp-sam-dist-limit}
\mathrm{SP}'(\mathscr{C}_{\sss (i)}(\lambda))\text{ and }\mathrm{SP}(\mathscr{C}_{\sss (i)}(\lambda))\text{ have the same distributional limit}.
 \end{equation}
 Let $\mathcal{G}_n'$ denote the graph obtained from $\mathcal{G}_n(t_c(\lambda))$ by shrinking each blob to a single node. 
 Then, $\mathrm{SP}'(\cdot)$ can be viewed as the surplus edges in the components of $\mathcal{G}_n'$. 
The graph $\mathcal{G}_n'$ can also be viewed to be constructed dynamically as in Algorithm~\ref{c3:algo:dyn-cons} with the degree sequence being $(f_i(t_n))_{i\geq 1}$.
In the following, we investigate the relations between $\mathcal{G}_n(t_n)$ and $\mathcal{G}_n'$ carefully. 
Lemma~\ref{c3:lem:total-open-he} implies that the number of unpaired half-edges in $\mathcal{G}_n(t_n)$ that are paired in $\mathcal{G}_n(t_c(\lambda))$ is given by
\begin{equation}\label{c3:eq:estimate-he-modi-supst}
s_1(t_n)-s_1(t_c(\lambda))  = n \mu_n( n^{-\delta} + \lambda n^{-\eta} )+\oP(n^{1-\gamma}), \quad \text{for some }\eta<\gamma.
\end{equation} 
 \begin{algo}\label{c3:algo:perc-for-blob}\normalfont
  Define
 $ \pi_n = \frac{\nu_n}{\nu_n-1}( n^{-\delta} + \lambda n^{-\eta} )$
   and associate $f_i(t_n)$ half-edges to the vertex $i$ of $\mathcal{G}_n'$. Construct the graph $\mathcal{G}_n'(\pi_n)$ as follows:
  \begin{enumerate}
  \item[(S1)] Retain each half-edge independently with probability $\pi_n$. 
  \item[(S2)] Create a uniform perfect matching between the retained half-edges and obtain $\mathcal{G}_n'(\pi_n)$ by creating edges corresponding to any two pair of matched half-edges.
  \end{enumerate}
 \end{algo}
 In (S1), if the total number of retained half-edges is odd, then add an extra half-edge to vertex~1. However, this will be ignored in the computations since it does not make any difference in the asymptotic computations.
Notice that $a_i$, the number of half-edges attached to $i$ that are retained by Algorithm~\ref{c3:algo:perc-for-blob}~(S1), is distributed as $\mathrm{Bin}(f_i(t_n),\pi_n)$, independently for each $i$. 
Thus the number of half-edges in the graph $\mathcal{G}_n'(\pi_n)$ is distributed as a $\mathrm{Bin}(s_1(t_n),\pi_n)$ random variable.
We claim that there exists $\varepsilon_n = o(n^{-\eta})$ and a coupling such that, with high probability 
\begin{equation}\label{c3:eq:blob-graphs}
 \mathcal{G}_n'(\pi_n-\varepsilon_n)\subset \mathcal{G}_n'\subset \mathcal{G}_n'(\pi_n+\varepsilon_n).
\end{equation}
The proof follows from an identical argument as Lemma~\ref{c3:lem:coupling-whp} using the estimate \eqref{c3:eq:estimate-he-modi-supst} and standard concentration inequalities for binomial random variables. 
We skip the proof here and refer the reader to Chapter~\ref{c1:sec_multidimensional}.
We now continue to analyze $\mathcal{G}_n'(\pi_n)$, keeping in mind that the relation \eqref{c3:eq:blob-graphs} allows us make conclusions for $\mathcal{G}_n'$.
%
%
%
To analyze the component sizes and the surplus edges of the components of $\mathcal{G}'(\pi_n)$ we first need some regularity conditions on $\bld{a}$, the degree sequence of $\mathcal{G}_n'(\pi_n)$, as summarized in the following lemma:
\begin{lemma}\label{c3:lem:a-asymp}
For some $\eta/2<\delta<\eta$, as $n\to\infty$,
\begin{gather*} 
 n^{-\alpha}a_i\pto\frac{\theta_i}{\nu}, \qquad  \frac{a_i}{\sum_i a_i}n^{\rho-\delta} \pto \frac{\theta_i}{\mu\nu}, \\ 
 \nu_n(\bld{a}) = \frac{\sum_ia_i(a_i-1)}{\sum_ia_i} = 1+ \lambda n^{-\eta+\delta} + \oP(n^{-\eta+\delta}),
 \end{gather*} and for any $\varepsilon>0$,
 \begin{equation}\label{c3:eq:a-third-moment}
 \lim_{K\to\infty}\limsup_{n\to\infty}\PR\bigg(\sum_{i>K}a_i^3>\varepsilon n^{3\alpha}\bigg) = 0.  
 \end{equation} 
\end{lemma}  
\begin{proof} 
Using Theorem~\ref{c3:th:open-he-entrance} and the fact that $a_i\sim \mathrm{Bin}(f_i(t_n),\pi_n)$, one gets $n^{-\alpha}a_i = (1+\oP(1))\theta_i/\nu$.
Moreover, $\sum_ia_i\sim\mathrm{Bin}(\sum_if_i(t_n),\pi_n)$
and $\sum_ia_i = (1+\oP(1))\pi_n \sum_if_i(t_n)$ yield the required asymptotics for  $a_i/\sum_ja_j$.
 Let $I_{ij}$ be the indicator of the $j$-th half-edge corresponding to vertex $i$ is kept in Algorithm~\ref{c3:algo:perc-for-blob}~(S1). Then $I_{ij} \sim \text{Ber} (\pi_n)$ independently for $j \in [f_i(t_n)]$, $i\geq 1$. 
 Note that, by changing the status of one half-edge corresponding to vertex $k$, we can change $\sum_ia_i(a_i-1)$ by at most $2(f_k(t_n)+1)$. Therefore we can apply \cite[Corollary 2.27]{JLR00} to conclude that 
 \begin{equation}\label{c3:eqn::prob:ineq:third}
 \begin{split}
 &\mathbb{P} \bigg( \Big|\sum_{i} a_i(a_i-1)- \pi_n^2 \sum_{i} f_i(t_n)(f_i(t_n)-1) \Big| >t \Big| (f_i(t_n))_{i\geq 1}\bigg) \\
 &\hspace{1cm}\leq 2 \exp \bigg( \frac{-t^2}{2\sum_{i} f_i(t_n) (f_{i}(t_n)+1)^2}\bigg).
 \end{split}
 \end{equation}
Observe that $\pi_n^2\sum_{i} f_i(t_n)(f_i(t_n)-1) = \Theta_{\sss \PR}(n^{1-\delta})$.
  Take $t = \varepsilon n^{1-\delta}$ and recall that $\sum_if_i^3(t_n) = \Theta_{\sss \PR}(n^{3\alpha+3\delta})$.
It is easy to check that $(2-3\alpha)/5 > \eta /2$, and therefore one can choose $\eta/2<\delta<\eta$ such that $t^2/\sum_if_i^3(t_n) \to\infty$. Thus,
\begin{equation}\label{c3:eq:a-i-1}
 \sum_{i}a_i(a_i-1) = (1+\oP(1)) \pi_n^2 \sum_i f_i(t_n)(f_i(t_n)-1).
\end{equation}
Therefore,  Theorem~\ref{c3:th:open-he-entrance} yields the required asymptotics for $\nu_n(\bld{a})$.
To see \eqref{c3:eq:a-third-moment}, note that
 $\E\big[\sum_{i>K}a_i(a_i-1)(a_i-2)\mid(f_i(t_n))_{i\geq 1}\big] = \pi_n^3\sum_{i>K}f_i(t_n)^3,$
 and the proof follows again by using the condition on $s_3^\omega$ in Theorem~\ref{c3:th:open-he-entrance}.
\end{proof}
\noindent From here onward, we assume that $\delta>0$ is such that Lemma~\ref{c3:lem:a-asymp} holds. 
Consider the  exploration of the graph $\mathcal{G}_n'(\pi_n)$ via Algorithm~\ref{c3:algo:explor-barely-sub}, but now the first vertex is chosen proportional to its degree.
Define the exploration process by $\bld{S}_n$ similarly as the process $\bld{S}_n^j(l)$ in Section~\ref{c3:sec:barely-subcritical-mass}.
 Call a vertex \emph{discovered} if it is either active or killed. Let $\mathscr{V}_l$ denote the set of vertices discovered up to time $l$ and $\mathcal{I}_i^n(l):=\ind{i\in\mathscr{V}_l}$. Note that 
   \begin{equation}\label{c3:expl-process-CM}
    S_n(l)= \sum_{i} a_i \mathcal{I}_i^n(l)-2l=\sum_{i} a_i \left( \mathcal{I}_i^n(l)-\frac{a_i}{\ell^a_n}l\right)+\left( \nu_n(\bld{a})-1\right)l,
   \end{equation}where $\ell_n^a = \sum_i a_i$. Consider the re-scaled version $\bar{\mathbf{S}}_n$ of $\mathbf{S}_n$ defined as $\bar{S}_n(t)=n^{-\alpha}S_n(\lfloor tn^{\rho-\delta} \rfloor)$. 
   Define the limiting process
   \begin{equation}\label{c3:eq:lim-blob-process}
   S(t) = \sum_{i=1}^\infty\frac{\theta_i}{\nu}\bigg(\ind{E_i\leq t}-\frac{\theta_i}{\mu\nu}t\bigg) + \lambda t,
   \end{equation}where $E_i\sim\mathrm{Exp}(\theta_i/(\mu\nu))$ independently for $i\geq 1$. 
   The following proposition describes the scaling limit of $\bar{\mathbf{S}}_n$:
   \begin{proposition} \label{c3:thm::convegence::exploration-process-blob} As $n\to\infty$, $\bar{\mathbf{S}}_n \dto \mathbf{S}$ with respect to the Skorohod $J_1$ topology.
\end{proposition}The proof of Proposition~\ref{c3:thm::convegence::exploration-process-blob} can be carried out using similar ideas as Chapter~\ref{chap:secondmoment} Theorem~\ref{c2thm::convegence::exploration_process}. 
A sketch of the proof is given in Appendix~\ref{c3:sec:appendix-perc-blob}. 
 The excursion lengths of the exploration process give the number of edges in the explored components. 
 Now, at each step $l$, the probability of discovering a surplus edge, conditioned on the past, is approximately the proportion of half-edges that are active. 
 Note that the number of active half-edges is the reflected version of $\mathbf{S}_n$ given by $\mathrm{refl}(S_n(t)) = S_n(t) - \inf_{u\leq t} S_n(u)$. 
 Thus, conditionally on $(S_n(l))_{l\leq tn^{\rho-\delta}}$, the rate at which a  surplus edge appears at time $tn^{\rho-\delta}$ is approximately  
$n^{\rho-\delta}\frac{\mathrm{refl}(S_n(tn^{\rho-\delta}))}{\sum_{i}a_i} = \frac{1}{\mu}\refl{\bar{S}_n(t)}(1+\oP(1)).$
 Therefore, Proposition~\ref{c3:thm::convegence::exploration-process-blob} implies that for each $K\geq 1$, there exists components $C_1,\dots,C_K \subset \mathcal{G}_n'(\pi_n)$ such that
 \begin{equation}\label{c3:eq:blob-perc-limit}
  \big(n^{-\rho +\delta}|C_i|, \mathrm{SP} (C_i)\big)_{i\in [K]} \dto \big(\xi_i, \mathscr{N}_i \big)_{i\in[K]},
 \end{equation} where $\xi_i$ and $\mathscr{N}_i$ are defined in Proposition~\ref{c3:prop:comp-size}.
 Here we have also used the fact that the ordered excursion lengths of the process $(S(t))_{t\geq 0}$, defined in \eqref{c3:eq:lim-blob-process},
 are identically distributed as the ordered excursion lengths of $(S(t)/\mu)_{t\geq 0}$.
 Note that $C_i$ in \eqref{c3:eq:blob-perc-limit} may not be the $i$-th largest component of $\mathcal{G}_n'(\pi_n)$ as we have not established that the $i$-th largest component is explored by time $\Theta(n^{\rho-\delta})$. However, that is not required for our purposes.  
We can now combine \eqref{c3:eq:blob-graphs} and \eqref{c3:eq:blob-perc-limit} to obtain the asymptotics for the number of blobs in the largest connected components and $\mathrm{SP}'(\cdot)$.
Denote $\mathscr{B} (\mathscr{C}) = |\mathfrak{B}(\mathscr{C})|$ for a component $\mathscr{C}\subset \mathcal{G}_n(t_c(\lambda))$. 
The following is a direct consequence of \eqref{c3:eq:blob-graphs} and \eqref{c3:eq:blob-perc-limit}:
\begin{lemma}\label{c3:thm:superstrcut-comp-surp-original}
 For $K\geq 1$, there exist components $\mathscr{C}_1,\dots,\mathscr{C}_K \subset \mathcal{G}_n(t_c(\lambda))$ such that the following convergence holds:
  $$(n^{-\rho +\delta}\mathscr{B}(\mathscr{C}_i), \mathrm{SP}' (\mathscr{C}_i))_{i\in [K]} \dto (\xi_i, \mathscr{N}_i  )_{i\in[K]}.
 $$
\end{lemma}
\noindent Next we show that the components $\sC_i$ in Lemma~\ref{c3:thm:superstrcut-comp-surp-original} indeed correspond to the $i$-th largest component of $\mathcal{G}_n(t_c(\lambda))$:
\begin{lemma}\label{c3:thm:comp-blob-same-whp}
 For any $K\geq 1$, $\mathscr{C}_i = \mathscr{C}_{\sss (i)}(\lambda)$, $\forall i\in [K]$ with high probability.
\end{lemma}
\begin{proof}
 Notice that $\sum_{j\leq i}|\mathscr{C}_j|\leq \sum_{j\leq i}|\mathscr{C}_{\sss (j)} (\lambda)|$ for all $i\in [K]$, almost surely. 
 Thus, it is enough to prove that $|\mathscr{C}_i|$ and  $|\mathscr{C}_{\sss (i)} (\lambda)|$ involve the same re-scaling factor and have the same scaling limit. 
We again make use of the inclusions in graphs in \eqref{c3:eq:blob-graphs}.
Algorithm~\ref{c3:algo:explor-barely-sub} explores the components of $\mathcal{G}_n'(\pi_n)$ in a size-biased manner with the sizes being $(a_i)_{i\geq 1}$. 
An application of Lemma~\ref{c3:lem:size-biased} with $y_i = \mathscr{C}_{\sss (i)}(t_n)$ yields that, for any $t>0$, uniformly for $l\leq tn^{\rho-\delta}$,
 \begin{equation}\label{c3:blob-vs-comp1}
  \sum_i|\mathscr{C}_{\sss (i)}(t_n)| \mathcal{I}_i^n(l) = \sum_i |\mathscr{C}_{\sss (i)}(t_n)| \frac{a_i}{\sum_i a_i}l +\oP(n^{\rho}).
 \end{equation} 
Since $a_i\sim\mathrm{Bin}(f_i(t_n),\pi_n)$, we can apply concentration inequalities like \cite[Corollary 2.27]{JLR00} and use the asymptotics from Theorem~\ref{c3:th:open-he-entrance} to conclude that
 \begin{equation}\label{c3:blob-vs-comp2}
 \begin{split}
  n^{-\delta}\frac{\sum_i a_i|\mathscr{C}_{\sss (i)}(t_n)|}{\sum_ia_i} 
  & = \frac{\frac{\mu(\nu-1)}{\nu^2}}{\frac{\mu(\nu-1)}{\nu}}(1+\oP(1))=\frac{1}{\nu}(1+\oP(1)).
  \end{split}
 \end{equation}
 Thus, \eqref{c3:blob-vs-comp1} and \eqref{c3:blob-vs-comp2}, together with \eqref{c3:eq:blob-graphs}, imply that
$  \frac{\nu |\mathscr{C}_i|}{n^{\delta}\mathscr{B}(\mathscr{C}_i)} \xrightarrow{\sss \PR} 1,$ and it follows from Lemma~\ref{c3:thm:superstrcut-comp-surp-original} and Lemma~\ref{c3:lem:rescale} that 
$  (n^{-\rho} |\mathscr{C}_i|)_{i\in [K]} \xrightarrow{\sss d} (\frac{1}{\nu}\xi_i )_{i\in[K]}. $
\end{proof}

\begin{proof}[Proof of Proposition~\ref{c3:thm:comp-functional-original}]
We are now finally in the position to prove Proposition~\ref{c3:thm:comp-functional-original}.
Using Lemmas~\ref{c3:thm:superstrcut-comp-surp-original},~\ref{c3:thm:comp-blob-same-whp}, and Proposition~\ref{c3:thm:comp-functionals-original} together with \eqref{c3:eq:blob-graphs}, we directly conclude Part~(a) from \eqref{c3:eq:surp-sam-dist-limit}.
For Part~(b),  we can follow the same arguments as~\eqref{c3:blob-vs-comp1} to conclude that, uniformly for $l\leq tn^{\rho-\delta}$,
 \begin{equation}\label{c3:blob-vs-comp3}
  \sum_if_i(t_n) \mathcal{I}_i^n(l) = \sum_i f_i(t_n) \frac{a_i}{\sum_i a_i}l +\oP(n^{\rho}),
 \end{equation} where 
$  n^{-\delta}\frac{\sum_i a_if_i(t_n)}{\sum_ia_i} = \frac{\nu-1}{\nu} (1+\oP(1)).$
 Now, \eqref{c3:blob-vs-comp1} and \eqref{c3:blob-vs-comp3} together with \eqref{c3:eq:blob-graphs} prove Part~(b).
\end{proof}
 In the final part of the proof, we will also need an estimate of the surplus edges in the components $\bar{\mathscr{C}}_{\sss (i)}(\lambda)$, that can be obtained by following the exact same argument as the proof outline of Lemma~\ref{c3:thm:superstrcut-comp-surp-original}. Recall that the superstructure on the graph $\bar{\mathcal{G}}_n(t_c(\lambda))$ is a rank-one inhomogeneous random graph $\mathrm{NR}_n(\bld{x},q)$. 
 The connection probabilities given by \eqref{c3:eq:pij-value} can be written as $1-\exp(-z_i(\lambda)z_j(\lambda)/\sum_kz_k(\lambda))$, where 
 \begin{equation}
  z_i(\lambda) = \frac{f_i(t_n)\sum_jf_j(t_n)}{\sum_j f_j^2(t_n)}\big(1+\lambda n^{-\eta+\delta}+\oP(n^{-\eta+\delta})\big) .
 \end{equation}
 Moreover, using Theorem~\ref{c3:th:open-he-entrance}, it follows that 
 \begin{gather*}
  n^{-\alpha}z_i(\lambda) \pto \frac{\theta_i}{\nu}, \qquad \frac{z_i(\lambda)}{\sum_jz_j(\lambda)} \pto \frac{\theta_i}{\mu\nu}, \\ \nu_n(\bld{z}) = \frac{\sum_iz_i^2(\lambda)}{\sum_i z_i(\lambda)} = 1+\lambda n^{-\eta+\delta}+\oP(n^{-\eta+\delta}).
 \end{gather*}
 Now, we may consider the breadth-first exploration of the above graph  and define the exploration process 
$  S_n^{\sss \mathrm{NR}}(l) = \sum_i z_i(\lambda)\mathcal{I}_i^n(l) - l,$ as in \eqref{c3:expl-process-CM}. The only thing to note here is that the component sizes are not necessarily encoded by the excursion lengths above the past minima of $\mathbf{S}_n^{\sss \mathrm{NR}}$.
 However, if $  \tilde{S}_n^{\sss \mathrm{NR}}(l) = \sum_i \mathcal{I}_i^n(l) - l$, then it can be shown that (see \cite[Lemma 3.1]{BHL12}) $\tilde{\mathbf{S}}_n^{\sss \mathrm{NR}}$ and $\mathbf{S}_n^{\sss \mathrm{NR}}$ have the same distributional limit. 
 Thus, a conclusion identical to Proposition~\ref{c3:thm::convegence::exploration-process-blob} follows for $\bar{\mathcal{G}}_n(t_c(\lambda))$.  
 Due to the size-biased exploration of the components one can also obtain analogues of Lemmas~\ref{c3:thm:superstrcut-comp-surp-original}~and \ref{c3:thm:comp-blob-same-whp} for $\bar{\mathcal{G}}_n(t_c(\lambda))$. 
 This explains the following proposition:
\begin{proposition}\label{c3:thm:surp-superstruc-modi}
 For fixed  $K\geq 1$,
$  (n^{-\rho+\delta}\mathscr{B}(\bar{\mathscr{C}}_{\sss (i)}(\lambda)), \mathrm{SP}' (\bar{\mathscr{C}}_{\sss (i)}(\lambda)))_{i\in [K]} \dto (\xi_i, \mathscr{N}_i  )_{i\in[K]},$ as $n\to\infty$.
\end{proposition}

\subsection{Completing the proof of Theorem~\ref{c3:thm:main}}
\label{c3:sec:proof-thm1}
In this section, we finally conclude the proof of Theorem~\ref{c3:thm:main}.
Recall Theorem~\ref{c3:thm:mspace-limit-modified} and the terminologies therein.
Let $n^{-\eta}\mathscr{C}_{\sss (i)}^{\sss \mathrm{fr}}(\lambda)$ denote the measured metric space with measure $\mu_{\sss \mathrm{fr}}^i$ and the distances multiplied by $n^{-\eta}$.
At this moment, let us recall the relevant properties $\mathscr{C}_{\sss (i)}(\lambda)$ and $\bar{\mathscr{C}}_{\sss (i)}(\lambda)$: 
\begin{enumerate}[(A)]
\item By Proposition~\ref{c3:prop-coupling-order}, $\cup_{j\leq i}\mathscr{C}_{\sss (j)}(\lambda) \subset \cup_{j\leq i}\bar{\mathscr{C}}_{\sss (j)}(\lambda)$ almost surely for any $i\geq 1$. 
Therefore, applying Propositions~\ref{c3:thm:mod-comp-openhe}~and~\ref{c3:thm:comp-functionals-original}, it follows that with high probability $\mathscr{C}_{\sss (i)}(\lambda) \subset \bar{\mathscr{C}}_{\sss (i)}(\lambda)$ for any fixed $i\geq 1$. 
\item By Proposition~\ref{c3:thm:comp-functional-original}~(b), $\bar{\mathcal{F}}_i(\lambda) - \mathcal{F}_i(\lambda) \pto 0$ and consequently $\mu_{\sss \mathrm{fr}}^i(\bar{\mathscr{C}}_{\sss (i)}(\lambda)\setminus \mathscr{C}_{\sss (i)}(\lambda)) \pto 0$. 
\item By Propositions~\ref{c3:thm:comp-functional-original}~(a)~and~\ref{c3:thm:surp-superstruc-modi}, the number of surplus edges with one endpoint in  $\bar{\mathscr{C}}_{\sss (i)}(\lambda)\setminus \mathscr{C}_{\sss (i)}(\lambda)$ converges in probability to zero. 
Moreover, with high probability there is no surplus edge within the blobs.
This implies that, for any pair of vertices $u,v\in \mathscr{C}_{\sss (i)}(\lambda)$, with high probability, the shortest path between them is exactly the same in $\mathscr{C}_{\sss (i)}(\lambda)$ and in $\bar{\mathscr{C}}_{\sss (i)}(\lambda)$.
\end{enumerate}
Thus, from the definition of Gromov-weak convergence in Section~\ref{c3:sec:defn:GHP-weak}, an application of Theorem~\ref{c3:thm:mspace-limit-modified} yields that
 $\big(n^{-\eta}\mathscr{C}_{\sss (i)}^{\sss \mathrm{fr}}\big)_{i\geq 1} \xrightarrow{\sss d} (M_{i} )_{i\geq 1},$ 
The only thing remaining to show is that we can replace the measure $\mu_{\sss\mathrm{fr}}^i$ by $\mu_{\sss\mathrm{ct},i}$. Now, using Propositions~\ref{c3:thm:comp-functionals-original}~and~\ref{c3:thm:comp-functional-original}~(b), it is enough to show  that
\begin{equation}\label{c3:counting-measure-switch}
 \sum_{b\in \bl( \mathscr{C}_{\sss (i)}(\lambda))}\big|f_{b}(t_n) - (\nu-1) |\mathscr{C}_{\sss(b)}(t_n)|\big| = \oP(n^{\rho}).
\end{equation}
Indeed, during the breadth-first exploration of the superstructure of $\mathcal{G}_n(t_c(\lambda))$, the blobs are explored in a size-biased manner with the sizes being $(f_i(t_n))_{i\geq 1}$. 
Therefore, one can again use Lemma~\ref{c3:lem:size-biased}. 
Recall that, by Lemma~\ref{c3:thm:comp-blob-same-whp}, for any $\varepsilon>0$, one can choose $T>0$ so large that the probability of exploring $\mathscr{C}_{\sss (i)}(\lambda)$ within time $Tn^{\rho-\delta}$ is at least $1-\varepsilon$. 
Thus, if $\mathscr{V}^b_l$ denotes the set of blobs explored before time $l$, then, for any $T>0$,  
\begin{align*}
  &\sum_{b\in \mathscr{V}^b_{\sss Tn^{\rho-\delta}}}\big|f_{b}(t_n) - (\nu-1)|\mathscr{C}_{\sss (b)}(t_n)|\big|\\
  &\hspace{1cm}= (1+\oP(1)) Tn^{\rho-\delta} \sum_i \frac{f_i(t_n)}{\sum_if_i(t_n)}\big|f_{\sss i}(t_n) - (\nu-1)|\mathscr{C}_{\sss (i)}(t_n)|\big|.
\end{align*}
Using the Cauchy-Schwarz inequality and Theorem~\ref{c3:th:open-he-entrance} it now follows that the above term is $o(n^{\rho})$.
Therefore \eqref{c3:counting-measure-switch} follows.
Finally the proof of Theorem~\ref{c3:thm:main} is complete using Lemma~\ref{c3:lem:coupling-whp}.\qed

\begin{remark} \label{c3:rem:switch-measure} \normalfont
The fact that the measure can be changed from $\mu_{\sss \mathrm{fr}}^i$ to $\mu_{\sss\mathrm{ct},i}$ in $n^{-\eta}\mathscr{C}_{\sss (i)}^{\sss \mathrm{fr}}$ follows only from~\eqref{c3:counting-measure-switch}, which again follows from the entrance boundary conditions. 
However, the entrance boundary conditions in Theorem~\ref{c3:thm:susceptibility} hold for weight sequences $\bld{w}=(w_i)_{i\in [n]}$ under general assumptions (see Assumption~\ref{c3:assumption-w}). 
Therefore, one could also replace the measure $\mu_{\sss\mathrm{ct},i}$ by $\mu_{w,i}$, where $\mu_{w,i} = \sum_{i\in A}w_i/\sum_{k\in \mathscr{C}_{\sss (i)}(\lambda)}w_i$ and $\bld{w}$ satisfies Assumption~\ref{c3:assumption-w}.
\end{remark}

\subsection{Proof of Theorem~\ref{c3:thm:main-simple}}
\label{c3:sec:simple}
The argument is related to Chapter~\ref{c2sec:simple-graphs}.
Using \cite[Lemma 3.2]{F07}, the random graph $\mathrm{CM}_n(\bld{d},p_n(\lambda))$, conditionally on its degree sequence $\bld{d}^p$, is distributed as $\mathrm{CM}_n(\bld{d}^p)$.
To complete the proof of Theorem~\ref{c3:thm:main-simple}, consider the exploration algorithm given by Algorithm~\ref{c3:algo:explor-barely-sub}, now on the graph $\mathrm{CM}_n(\bld{d},p_n(\lambda))$, conditionally on the degree sequence $\bld{d}^p$.
The starting vertex is chosen in a size biased manner with sizes proportional to the degrees $\bld{d}^p$. 
For convenience, we denote $X = (\mathscr{C}_{\sss (i)}^p(\lambda))_{i\leq K}$ in this section. 
Consider a bounded continuous function $f:(\mathscr{S}_*)^K\mapsto \R$, where we recall $\mathscr{S}_*$ from Section~\ref{c3:sec:defn:GHP-weak}. 
Recall from \cite[Theorem 1.1]{J09c} that $$\liminf_{n\to\infty}\prob{\CM \text{ is simple}}>0.$$
Thus, it is enough to show that 
\begin{equation}\label{c3:simple-fact}
\expt{f(X)\ind{\CM \text{ is simple}}} - \expt{f(X)} \prob{\CM \text{ is simple}} \to 0.
\end{equation}
Now, for any $T>0$, let $\mathcal{A}_{n,T}$ denote the event that $X$ is explored before time $Tn^{\rho}$ by the exploration algorithm. 
Using \cite[Lemma 13]{DHLS16}, it follows that 
\begin{equation}
\lim_{T\to\infty}\limsup_{n\to\infty} \prob{\mathcal{A}_{n,T}^c} = 0.
\end{equation}
Let $X_{ T}$ denote the random vector consisting of $K$ largest ones among the components explored before time $Tn^{\rho}$.
 Thus,
\begin{align*}
&\lim_{T\to\infty}\limsup_{n\to\infty}\expt{f(X)\ind{\CM \text{ is simple}}\1_{\mathcal{A}_{n,T}^c}} \\
&\hspace{1cm}\leq \|f\|_{\infty} \lim_{T\to\infty}\limsup_{n\to\infty}\prob{\mathcal{A}_{n,T}^c}=0, 
\end{align*}which implies that
\begin{equation}\label{c3:simple-fact-3}
\lim_{T\to\infty}\limsup_{n\to\infty}\big|\expt{\big(f(X)-f(X_{T})\big)\ind{\CM \text{ is simple}}} \big| = 0.
\end{equation}
Further, let $\mathcal{B}_{n,T}$ denote the event that a vertex $v$ is explored before time $Tn^{\rho}$ such that $v$ is involved in a self-loop or a multiple edge in $\CM$. 
For any fixed vertex $v$, the $i$-th half edge creates a self-loop in $\CM$ with probability at most $(d_{v}-i)/(\ell_n-1)$ and creates a multiple edge with probability at most $(i-1)/(\ell_n-1)$ so that the probability of $v$ creating a self-loop or a multiple edge is at most $d_v^2/(\ell_n-1)$. 
Let $\mathcal{I}_i^n(l)$ denote the indicator that vertex $i$ is discovered upto time $l$ and note that Algorithm~\ref{c3:algo:explor-barely-sub} will explore the vertices in a size-biased manner with sizes being $\bld{d}^p$.
Let $\PR_p$ (respectively $\E_p$) denote the conditional probability (respectively expectation), conditionally on $\bld{d}^p$. 
Thus, 
\begin{align*}
 \PR_p(\mathcal{B}_{n,T})&\leq \frac{1}{\ell_n-1}\E_p\bigg[\sum_{i\in [n]}d_i^2\mathcal{I}^n_i(Tn^{\rho})\bigg]\\
 &=\frac{1}{\ell_n-1}\bigg(\E_p\bigg[\sum_{i=1}^K d_i^2\mathcal{I}^n_i(Tn^{\rho})\bigg]+\E_p\bigg[\sum_{i=K+1}^nd_i^2\mathcal{I}^n_i(Tn^{\rho})\bigg]\bigg).
 \end{align*} Now, using Assumption~\ref{c3:assumption1}, for every fixed $K\geq 1$,
\begin{equation}
 \frac{1}{\ell_n-1}\E_p\bigg[\sum_{i=1}^K d_i^2\mathcal{I}^n_i(Tn^\rho)\bigg]\leq \frac{2}{\ell_n} \sum_{i=1}^K d_i^2 \pto 0.
\end{equation} 
Further, $\PR_p(\mathcal{I}^n_i(Tn^{\rho})=1)\leq Tn^{\rho}d_i^p/(\sum_{i\in [n]}d_i^p-2Tn^{\rho})$. 
Therefore, 
\begin{equation}\label{c3:simple-calc-1}
 \begin{split} 
  \frac{2}{\ell_n}\E_p\bigg[\sum_{i=K+1}^nd_i^2\mathcal{I}_i^n(Tn^{\rho})\bigg]&\leq \frac{2Tn^{\rho}}{\ell_n \sum_{i\in [n]}d_i^p}\sum_{i=K+1}^nd_i^2 d_i^p\\
  &\leq \OP(1)\bigg(n^{-3\alpha}\sum_{i=K+1}^nd_i^3 \bigg), 
 \end{split}
\end{equation}
where the last step follows using $\sum_{i\in [n]}d_i^p \sim 2\times\mathrm{Bin}(\ell_n/2,p_n(\lambda))$, standard concentration inequalities for the binomial distribution, and the fact that $d_i^p\leq d_i$ for all $i\in [n]$.
Now, by Assumption~\ref{c3:assumption1}, the final term in \eqref{c3:simple-calc-1} tends to zero in probability if we first take $\limsup_{n\to\infty}$ and then take $\lim_{K\to\infty}$. 
Consequently, for any fixed $T>0$, 
\begin{equation}\label{c3:eq:B-n-t}
 \lim_{n\to\infty}\prob{\mathcal{B}_{n,T}}= 0.
\end{equation}
Let $\mathcal{E}_{n,T}$ denote the event that no self-loops or multiple edges are attached to the vertices in $\CM$ that are discovered after time $Tn^{\rho}$. Then \eqref{c3:simple-fact-3} and \eqref{c3:eq:B-n-t} implies that
\begin{eq}\label{c3:simple-fact-2}
&\lim_{n\to\infty}\expt{f(X)\ind{\CM \text{ is simple}}} = \lim_{T\to\infty}\lim_{n\to\infty}\expt{f(X_T)\1_{\mathcal{E}_{n,T}}}\\
&= \lim_{T\to\infty}\lim_{n\to\infty}\expt{f(X_T)\prob{\mathcal{E}_{n,T}\vert \mathscr{F}_{Tn^{\rho}}}} \\
&= \lim_{T\to\infty}\lim_{n\to\infty}\expt{f(X_T)\prob{\mathcal{E}_{n,T}\vert \mathscr{F}_{Tn^{\rho}}, \mathcal{B}_{n,T}}}.
\end{eq}
Let $\cG^*_{Tn^{\rho}}$ denote the graph obtained from $\CM$ after removing the vertices discovered upto time $Tn^{\rho}$.
Then $\cG^*_{Tn^{\rho}}$ is distributed as a configuration model  conditional on its degree sequence. 
Thus conditional on $\mathscr{F}_{Tn^{\rho}}\cap \mathcal{B}_{n,T}$, $\mathcal{E}_{n,T}$ happens if and only if $\cG^*_{Tn^{\rho}}$ is simple.
Now, an argument similar to \eqref{c2prob-of-simple} in Chapter~\ref{chap:secondmoment} can be applied to conclude that 
\begin{equation}
\prob{\cG^*_{Tn^{\rho}} \text{ is simple}\vert \mathscr{F}_{tn^{\rho}}} - \prob{\CM \text{ is simple}} \pto 0,
\end{equation}and using \eqref{c3:simple-fact-2},  \eqref{c3:simple-fact} follows, and the proof of Theorem~\ref{c3:thm:main-simple} is now complete. \qed

\section{Conclusion}
We have obtained the scaling limit for the metric structure of the ordered component sizes of the critical percolation clusters for $\CM$ when the empirical degree distribution has diverging third moment. 
The key ingredient of the proof is a universality principle in Theorem~\ref{c3:thm:univesalty}, which basically says that after replacing each nodes by metric spaces having small diameter, the scaling limit for the rank-one inhomogeneous random graphs does not change even if typical distances change. 
This work provides a general framework to establish the scaling limit for networks which are in the same universality class as identified in \cite{BHS15}. 
The overall idea for percolation on $\CM$ is not very specific  to the underlying model, and could be applicable to other types of inhomogeneous random graphs. 
An analogous framework for the Erd\H{o}s-R\'enyi universality class was established in \cite{BBSX14}. 
The underlying topology for the convergence of metric spaces in \cite{BBSX14} was taken to be Gromov-Hausdorff-Prokhorov topology, which turns out to be strictly stronger than the Gromov-weak topology considered here.
In the next chapter, we strengthen Theorem~\ref{c3:thm:main} to Gromov-Hausdorff-Prokhorov topology under some additional mild assumptions on the degree sequence.

\begin{subappendices}
\section{Proof of Proposition~\ref{c3:prop:comp-size}}\label{c3:sec:appendix-rescaling}
Note that due to the difference in the choice of $p_n(\lambda)$ in Chapter~\ref{chap:secondmoment} Assumption~\ref{c2assumption1} and this chapter, $\lambda$ must be replaced by $\lambda \nu$.
Let $\mathcal{E}(\cdot)$ denote the operator that maps a process to its ordered vector of excursion lengths, and $\mathcal{A}(\cdot)$  maps a process to the vector of areas under those excursions. 
Let us use $\mathrm{Exp}(b)$ as a generic notation to write an exponential random variable with rate $b$. 
All the different exponential random variables will be assumed to independent.
Now, 
 \begin{equation}\label{c3:eq:excursion-param}
 \begin{split}
  &\frac{1}{\sqrt{\nu}}\mathcal{E}\bigg(\sum_{i\geq 1}\frac{\theta_i}{\sqrt{\nu}}\big(\ind{\mathrm{Exp}(\theta_i/(\mu\sqrt{\nu}))\leq t} - (\theta_i/(\mu\sqrt{\nu}))t\big)+\lambda\nu t\bigg)\\
  &\hspace{1cm} \eqd \frac{1}{\nu} \mathcal{E}\bigg(\sum_{i\geq 1}\frac{\theta_i}{\sqrt{\nu}}\big(\ind{\mathrm{Exp}(\theta_i/(\mu\nu))\leq u} - (\theta_i/(\mu\nu))u\big)+\lambda u\sqrt{\nu}\bigg)\\
  &\hspace{1cm} \eqd \frac{1}{\nu} \mathcal{E}\bigg(\sum_{i\geq 1}\frac{\theta_i}{\mu\nu}\big(\ind{\mathrm{Exp}(\theta_i/(\mu\nu))\leq u} - (\theta_i/(\mu\nu))u\big)+\frac{\lambda}{\mu} u\bigg),
   \end{split}
 \end{equation} where the last step follows by rescaling the space by $\mu\sqrt{\nu}$ and noting that the rescaling of space does not affect excursion lengths.
 Again, 
 \begin{eq}
  &\mathcal{A}\bigg(\sum_{i\geq 1}\frac{\theta_i}{\mu\sqrt{\nu}}\big(\ind{\mathrm{Exp}(\theta_i/(\mu\sqrt{\nu}))\leq t} - (\theta_i/(\mu\sqrt{\nu}))t\big)+\frac{\lambda\nu}{\mu} t\bigg)\\
  & \hspace{1cm}\eqd \mathcal{A}\bigg(\sum_{i\geq 1}\frac{\theta_i}{\mu\nu}\big(\ind{\mathrm{Exp}(\theta_i/(\mu\nu))\leq u } - (\theta_i/(\mu\nu))t\big)+\frac{\lambda}{\mu} u\bigg),
 \end{eq}which is obtained by rescaling both the space and time by $\sqrt{\nu}$.  
 Thus, the proof follows.
\section{Computation for $\E[\mathscr{W}(V_n^*)^3]$} 
\label{c3:sec:appendix-path-counting}
Recall that 
\begin{equation}
\big(\mathscr{W}(V_n^*)\big)^r = \sum_{k_1,\cdots,k_r\in [n]}w_{k_1}\dots w_{k_r}\ind{V_n^*\leadsto k_1,\dots, V_n^*\leadsto k_r}. 
\end{equation}
For the third moment, the leading contributions arise from the structures given in Figure~3.
Thus,
\begin{equation}\label{c3:W-moment3-ub}
 \begin{split}
  &\expt{\big(\mathscr{W}(V_n^*)\big)^3}\\
  &\leq \expt{(W_n^*)^3}+\frac{\expt{D_n^*}(\expt{D_n'W_n})^3\sigma_3(n)^2}{(\expt{D_n'})^5(1-\nu_n')^5}\\
  &\hspace{.5cm}+\frac{\expt{D_n^*(D_n^*-1)(D_n^*-2)}(\expt{D_n'W_n})^3}{(\expt{D_n'})^3(1-\nu_n')^3}\\
  & \hspace{1cm}+\frac{\expt{D_n^*}(\expt{D_n'W_n})^3\sigma_4(n)}{(\expt{D_n'})^4(1-\nu_n')^4}+ \frac{\expt{D_n^*(D_n^*-1)}(\expt{D_n'W_n})^3\sigma_3(n)}{(\expt{D_n'})^4(1-\nu_n')^4}\\
  &=O(n^{4\alpha-1}+n^{6\alpha-2+3\delta}+n^{4\alpha-1+3\delta}+n^{4\alpha-1+4\delta}+n^{6\alpha-2+4\delta})=o(n^{2\delta+1}),
 \end{split}
\end{equation} and the proof follows.
\section{Proofs of Lemmas \ref{c3:lem:exploration::subcritical-j}~and~\ref{c3:lem:weight-prop-l}}
\label{c3:sec:appendix-barely-subcrit}
\begin{proof}[Proof of Lemma~6.3]
Recall the representation of $\bar{S}_n^j(t)$ from \eqref{c3:eqn::scaled_process_j}. 
It is enough to show that 
\begin{eq}
\sup_{t\in [0,T]} & n^{-\alpha} \bigg|\sum_{i\in [n]}d_i'\bigg( \mathcal{I}_i^n(tn^{\alpha+\delta})-\frac{d_i'}{\ell_n'}tn^{\alpha+\delta} \bigg)\bigg|\\
&= \sup_{t\in [0,T]}n^{-\alpha}|M_n(tn^{\alpha+\delta})| \pto 0.
\end{eq}
Fix any $T>0$ and define $\ell_n'(T)=\ell_n'-2Tn^{\alpha+\delta}-1$, and $ M_n'(l) = $ \linebreak $\sum_{i\in [n]} d_i'(\mathcal{I}_i^n(l)- (d_i/\ell_n'(T))l)$. Note that 
\begin{eq}
 \sup _{t\in [0,T]} &n^{-\alpha}|M_n(tn^{\alpha+\delta})-M_n'(tn^{\alpha+\delta})| \\
 &\leq Tn^{\delta}\frac{(2Tn^{\alpha+\delta}-1)\sum_{i\in [n]}d_i'^2}{\ell_n'(T)^2} = \oP(1),
\end{eq}and thus the proof reduces to showing that 
\begin{equation} \label{c3:tail::martingale}
\sup_{t\in [0,T]}n^{-\alpha}|M_n'(tn^{\alpha+\delta})| \pto 0.
\end{equation}
Note that, uniformly over $l\leq Tn^{\alpha+\delta}$, 
  \begin{equation}\label{c3:eq:prob-ind}
  \prob{\mathcal{I}_i^n(l+1)=1\mid \mathscr{F}_l} \leq \frac{d_i'}{\ell_n'(T)}\quad\text{ on the set } \{\mathcal{I}_i^n(l)=0\}.
  \end{equation} 
 Therefore,
 \begin{align*}
  &\E\big[M_n'(l+1)-M_n'(l) \mid \mathscr{F}_l\big]\\
  &=\E\bigg[\sum_{i\in [n]} n^{-\alpha}d_i' \left(\mathcal{I}^n_i(l+1)-\mathcal{I}_i^n(l)-\frac{d_i'}{\ell_n'(T)}\right)\Big| \mathscr{F}_l\bigg]\\
  &= \sum_{i\in [n]} n^{-\alpha}d_i' \left(\E\big[\mathcal{I}^n_i(l+1)\big| \mathscr{F}_l\big]\ind{\mathcal{I}_i^n(l)=0} - \frac{d_i'}{\ell_n'(T)} \right)\leq 0.
  \end{align*} Thus $(M_n'(l))_{l= 1}^{Tn^{\alpha+\delta}}$ is a super-martingale.  Further, uniformly for all $l\leq Tn^{\alpha+\delta}$,
  \begin{equation}\label{c3:prob-ind-lb}
   \prob{\mathcal{I}_i^n(l)=0} \leq \left(1-\frac{d_i'}{\ell_n'} \right)^l.
  \end{equation}
  Thus, Assumption~2 gives
  \begin{align*}
    &n^{-\alpha}\big| \E[M_n'(l)]\big| \\&\leq n^{-\alpha} \sum_{i\in [n]} d_i'\left( 1-\left(1-\frac{d_i'}{\ell_n'} \right)^l-\frac{d_i'}{\ell_n'}l \right)+n^{-\alpha}l\sum_{i\in [n]}d_i'^2\left(\frac{1}{\ell_n'(T)}-\frac{1}{\ell_n'}\right)\\
    &\leq \frac{l^2}{2\ell_n'^2 n^{\alpha} } \sum_{i\in [n]} d_i'^3+o(1) = o(1),
  \end{align*} 
 where we have used the fact that 
  $$n^{-\alpha}l\sum_{i\in [n]}d_i'^2(1/\ell_n'(T)-1/\ell_n')=O(n^{2\rho+1-\alpha-2})=O(n^{(\tau-4)/(\tau-1)}),$$ uniformly for $l\leq Tn^{\alpha+\delta}$ and, in the last step, that fact that $\delta<\eta$. Therefore, uniformly over $l\leq Tn^{\alpha+\delta}$,
  \begin{equation}\label{c3:expectation::M_n^K}
  \lim_{n\to\infty}\big| \E[M_n'(l)]\big|=0.
  \end{equation} 
  Now, note that for any $(x_1,x_2,\dots)$, $0\leq a+b \leq x_i$ and $a,b>0$ one has $\prod_{i=1}^R(1-a/x_i)(1-b/x_i)\geq \prod_{i=1}^R (1-(a+b)/x_i)$. Thus, for all $l\geq 1$ and $i\neq j$, 
  \begin{equation}\label{c3:neg:correlation}
  \prob{\mathcal{I}_i^n(l)=0, \mathcal{I}_j^n(l)=0}\leq \prob{\mathcal{I}_i^n(l)=0}\prob{\mathcal{I}_j^n(l)=0}
  \end{equation} and therefore $\mathcal{I}_i^n(l)$ and $\mathcal{I}^n_j(l)$ are negatively correlated. Observe also that, uniformly over $l\leq Tb_n$, 
  \begin{eq}\label{c3:var-ind-ub}
   &\var{\mathcal{I}_i^n(l)}\leq  \prob{\mathcal{I}_i^n(l)=1} \\
   &\leq \sum_{l_1=1}^l\prob{\text{vertex  }i \text{ is first discovered at stage }l_1 }\leq \frac{ld_i' }{\ell_n'(T)}.
  \end{eq}  
  Therefore, using the negative correlation in \eqref{c3:neg:correlation}, uniformly over $l\leq Tn^{\alpha+\delta}$, 
  \begin{equation} \label{c3:variance::M_n^k}
   \begin{split}
    n^{-2\alpha}\var{M_n'(l)}&
    \leq \frac{l}{\ell_n'(T)n^{2\alpha}}\sum_{i\in [n]} d_i'^3 =o(1).
   \end{split}
  \end{equation} Now we can use the super-martingale inequality \cite[Lemma 2.54.5]{RW94} stating that for any super-martingale $(M(t))_{t\geq 0}$, with $M(0)=0$, 
 \begin{equation}\label{c3:eqn:supmg:ineq}
  \varepsilon \prob{\sup_{s\leq t}|M(s)|>3\varepsilon}\leq 3\expt{|M(t)|}\leq 3\left(|\expt{M(t)}|+\sqrt{\var{M(t)}}\right).
 \end{equation}
  Thus \eqref{c3:tail::martingale} follows using \eqref{c3:expectation::M_n^K}, \eqref{c3:variance::M_n^k}, and \eqref{c3:eqn:supmg:ineq}.
\end{proof}
\begin{proof}[Proof of Lemma~6.4]
Fix any $T>0$ and recall that $\ell_n(T)=\ell_n'-2Tn^{\alpha+\delta}-1$. Denote $W(l)=\sum_{i\in [n]}w_i\mathcal{I}_i^n(l)$. Firstly, observe that
\begin{align*}
  &\E[W(l+1)-W(l) \mid \mathscr{F}_l]\\
  &= \sum_{i\in [n]} w_i\E\big[\mathcal{I}^n_i(l+1)\mid \mathscr{F}_l\big]\ind{\mathcal{I}_i^n(l)=0} \leq \frac{\sum_{i\in [n]}d_i'w_i}{\ell_n'(T)},
  \end{align*} 
uniformly over $l\leq Tn^{\alpha+\delta}$. Therefore, $(\tilde{W}(l))_{l=1}^{Tn^{\alpha+\delta}}$ is a super-martingale, where $\tilde{W}(l)=W(l)-(\sum_{i\in [n]}d_i'w_i/\ell_n')l$. Again, the goal is to use \eqref{c3:eqn:supmg:ineq}. Using \eqref{c3:prob-ind-lb}, we can show that
$ \big|\E[\tilde{W}(l)]\big|
 =o(n^{\alpha+\delta}),$
uniformly over $l\leq Tn^{\alpha+\delta}$. Also, using \eqref{c3:neg:correlation} and \eqref{c3:var-ind-ub} and Assumption~\ref{c3:assumption-w},
$\mathrm{Var}(\tilde{W}(l))\leq \sum_{i\in [n]}w_i^2 \mathrm{var}(\mathcal{I}_i^n(l))= o(n^{2(\alpha+\delta)}),$ uniformly over $l\leq Tn^{\alpha+\delta}$. Finally, using \eqref{c3:eqn:supmg:ineq},  we conclude the proof.
\end{proof}
\section{Proof sketch for Proposition~\ref{c3:thm::convegence::exploration-process-blob}}
\label{c3:sec:appendix-perc-blob}
The proof of Proposition~\ref{c3:thm::convegence::exploration-process-blob} can be carried out using similar ideas as Chapter~\ref{chap:secondmoment} Theorem~\ref{c2thm::convegence::exploration_process}. 
The key idea to prove Proposition~\ref{c3:thm::convegence::exploration-process-blob} is that the scaling limit is governed by the vertices having large degrees only. More precisely, for any $\varepsilon > 0$ and $T>0$,
\begin{equation}
 \lim_{K\to\infty}\limsup_{n\to\infty}\PR\bigg(\sup_{t\leq T}n^{-\alpha}\bigg|\sum_{i>K}a_i\Big( \mathcal{I}_i^n(tn^{\rho-\delta})-\frac{a_i}{\ell_n^a}tn^{\rho-\delta} \Big)\bigg| > \varepsilon \bigg) = 0.
\end{equation} This can be proved using martingale estimates. Thus, if one considers the truncated sum 
\begin{align*}
\sum_{i\leq K} a_i \left( \mathcal{I}_i^n(l)-\frac{a_i}{\ell^a_n}l\right)+\left( \nu_n(\bld{a})-1\right)l,
\end{align*}with the first $K$ (fixed) terms
it is enough to show that the iterated limit of the truncated process (first taking $\lim_{n\to\infty}$ and then $\lim_{K\to\infty}$) converges to $\mathbf{S}$ with respect to the Skorohod $J_1$ topology. Now, using the fact that $a_i/\sum_ia_i\xrightarrow{\sss \PR} \theta_i/(\mu\nu)$, and the fact that the vertices are explored in a size-biased manner with sizes being $(a_i)_{i\geq 1}$, it follows that (see Chapter~\ref{chap:secondmoment} \linebreak Lemma~\ref{c2lem::convergence_indicators}), for each fixed $K\geq 1$,
\begin{equation}
 \big(\mathcal{I}_i^n(tn^{\rho-\delta})\big)_{i\in [K],t\geq 0} \dto \big(\ind{\mathrm{Exp}(\theta_i/(\mu\nu))\leq t}\big)_{i\in [K],t\geq 0}.
\end{equation}
This concludes the proof of Proposition~\ref{c3:thm::convegence::exploration-process-blob}. 
\end{subappendices}
%
%
%
%
 
\cleardoublepage

\chapter[Global lower mass-bound for critical components]{Global lower mass-bound for critical configuration models in the heavy-tailed regime}
\label{chap:mspace-GHP}
{\small \paragraph*{Abstract.}
We establish the global lower mass-bound property for largest connected components in the critical window of phase transition for configuration model when the degree distribution has an infinite third moment.
The scaling limit of the critical percolation clusters, viewed as measured metric spaces, was established in \cite{BDHS17} with respect to the Gromov-weak topology.
Our result extends those scaling limit results to hold under the stronger Gromov-Hausdorff-Prokhorov topology.
This implies convergence of global functionals such as the diameters of the critical components.
Further, our result establishes compactness of the random metric spaces, which arise as scaling limits of critical clusters in the heavy-tailed regime. 
}

\vspace{.3cm} 
\noindent {\footnotesize Based on the preprint: 
Shankar Bhamidi, Souvik Dhara, Remco van der Hofstad, Sanchayan Sen;
\emph{Global lower mass-bound for critical configuration models in the heavy-tailed regime} (2018)}
\vfill
%
%
%
%
%
%
%


Any connected graph $\sC$ can be viewed as a metric space with the distance between points given by $a \dst (\cdot,\cdot)$ for some constant $a>0$, where $\dst(\cdot,\cdot)$ is used as a generic notation to denote the  graph-distance (i.e., number of edges in the shortest path). 
Suppose that each vertex $i$ is assigned a \emph{mass}~$w_i$ so that there is a natural probability measure associated to the Borel sigma-algebra on $(\sC,a \dst)$ with the measure given by  $\mu (A) = \sum_{v\in A}w_v/\sum_{v\in \sC}w_v $ for any $A\subset \sC$. 
We denote the above metric space with a measure by $(\sC,a,\bld{w})$.
Fix any $\delta>0$ and define the $\delta$-lower mass of $(\sC,a,\bld{w})$ by 
\begin{eq}\label{c4:eq:defn:GLM}
\fm(\delta):=  \frac{\inf_{v\in \sC}\sum_{u: a \dst(v,u) \leq \delta} w_u}{\sum_{u\in \sC}w_u}.
\end{eq}
For a sequence $(\sC_n,a_n,\bld{w}_n)_{n\geq 1}$ of graphs viewed as metric spaces endowed with a measure,  the global lower mass-bound property is defined as follows:
\begin{defn}[Global lower mass-bound property \cite{ALW16}]\normalfont 
For $\delta>0$, let $\fm_n(\delta)$ denote the $\delta$-lower mass of $(\sC_n,a_n,\bld{w}_n)$.
Then $(\sC_n,a_n,\bld{w}_n)_{n\geq 1}$ is said to satisfy the global lower mass-bound property if and only if $\sup_{n\geq 1} \fm_n(\delta)^{-1}<\infty$ for any $\delta >0$.
When the sequence $(\sC_n)_{n\geq 1}$ is a collection of random graphs, $(\sC_n,a_n,\bld{w}_n)_{n\geq 1}$ is said to satisfy the global lower mass-bound property if and only if $(\fm_n(\delta)^{-1})_{n\geq 1}$ is a tight sequence of random variables for any $\delta >0$.
\end{defn}
The aim of this chapter is to prove the global lower mass-bound property for connected components of a configuration model at criticality, when the third moment of the empirical degree distribution tends to infinity.  
Informally speaking, the global lower mass-bound property ensures that all the small neighborhoods have mass bounded away from zero, so that the graph does not have any \emph{light spots} and the total mass is well-distributed over the whole graph.
This has several interesting consequences in the theory of critical random graphs, which we discuss in detail below after the formal statement of the result.
We start by defining the configuration model and state the precise assumptions, followed by a formal statement of the main result. 
Subsequently, we discuss some implications of this result in the context of recent scaling limit results for critical percolation on a configuration model.

\section{Main results}
Fix $\tau\in (3,4)$. Throughout this chapter we will use the shorthand notation
\begin{equation}\label{c4:eqn:notation-const}
 \alpha= 1/(\tau-1),\quad \rho=(\tau-2)/(\tau-1),\quad \eta=(\tau-3)/(\tau-1).
\end{equation}
Further, we assume the following conditions on the degree sequences of $\CM$:
\begin{assumption}[Degree sequence]\label{c4:assumption1}
\normalfont 
For each $n\geq 1$, let $\bld{d}=\boldsymbol{d}_n=(d_1,\dots,d_n)$ be a degree sequence satisfying $d_1\geq d_2\geq\ldots\geq d_n$. 
We assume the following about $(\boldsymbol{d}_n)_{n\geq 1}$ as $n\to\infty$:
\begin{enumerate}[(i)] 
\item \label{c4:assumption1-1} (\emph{High-degree vertices}) For each fixed $i\geq 1$, 
\begin{equation}\label{c4:defn::degree}
 n^{-\alpha}d_i\to \theta_i,
\end{equation}
where $\boldsymbol{\theta}=(\theta_1,\theta_2,\dots)\in \ell^3_{\shortarrow}\setminus \ell^2_{\shortarrow}$. 
\item \label{c4:assumption1-2} (\emph{Moment assumptions}) 
Let $D_n$ denote the degree of a typical vertex, i.e. a vertex chosen uniformly at random, independently of $\mathrm{CM}_n(\boldsymbol{d})$. Then, $D_n$ converges in distribution to some discrete random variable $D$ and 
\begin{eq}
 \frac{1}{n}\sum_{i\in [n]}d_i\to \mu := \E[D], \quad \frac{1}{n}&\sum_{i\in [n]}d_i^2 \to \mu_2:=\E[D^2],\\ 
 \lim_{K\to\infty}\limsup_{n\to\infty}n^{-3\alpha} &\sum_{i=K+1}^{n} d_i^3=0.
\end{eq}
\item For all sufficiently large $n$, the following holds uniformly over $i\in [n]$ and $x>0$:
\begin{eq}\label{c4:eq:assumption-degree-tail}
n^{-\alpha} \sum_{j:d_j>xn^{\alpha}}d_j \geq x^{-b},\quad \Big(\frac{d_i}{n^{\alpha}}\Big)^{b-1} n^{-2\alpha} \sum_{j\leq i}d_j^2>C
\end{eq} for some $b\in (1,2)$. Further, $\limsup_{n\to\infty}\sum_{i\geq 1}\e^{-n^{-2\alpha}\sum_{j=1}^{i}d_j^2}<\infty.$

\item \label{c4:assumption1-4} Let $n_1$ be the number of degree-one vertices. Then $n_1=\Theta(n)$, which is equivalent to assuming that $\prob{D=1}>0$.
\item The weight sequence $\bld{w} = (w_i)_{i\in [n]}$ satisfies 
$$\lim_{n\to\infty}\frac{1}{\ell_n}\sum_{i\in [n]} d_i w_i = \mu_{w}, \quad \max\bigg\{\sum_{i\in [n]}d_iw_i^2,\sum_{i\in [n]}d_i^2w_i\bigg\} = O(n^{3\alpha}).$$
\end{enumerate}
\end{assumption}
Assumption~\ref{c4:assumption1}~(i)--(iii) are the general set of assumptions on the degree distribution under which the scaling limit for the component sizes, surplus edges and the metric structure of critical configuration model was proved in \cite{DHLS16,BDHS17}. 
These assumptions are applicable for a configuration model with power-law degree distribution with exponent $\tau\in (3,4)$.
More precisely, if $F$ is a distribution function on non-negative integers  satisfying $(1-F)(x) = C x^{\tau-1}$, then Assumptions~\ref{c4:assumption1}~(i)--(iv) is satisfied when (a) $d_i=(1-F)^{-1}(i/n)$, (b) $d_i$'s are i.i.d.~samples from $F$ \cite[Section 2]{DHLS16}. 
Thus, above assumptions are applicable for configuration model with power-law degree distribution with exponent $\tau\in (3,4)$.
We note that Assumption~\ref{c4:assumption1}~(iv) is required for technical purposes, which was not required in \cite{DHLS16,BDHS17}.
Assumptions~\ref{c4:assumption1}~(v) for the weight sequence is satisfied for $w_i = 1$ or $w_{i} = d_i$ for all $i\in [n]$.
$w_i = 1$ is equivalent to the normalized counting measure on $\sC$.
%
%
Moreover, we assume that the configuration model lies within the critical window of the phase transition, i.e., for some $\lambda\in \R$,
\begin{equation}\label{c4:defn:criticality}
\nu_n=\frac{\sum_{i\in [n]}d_i(d_i-1)}{\sum_{i\in [n]}d_i} =  1 + \lambda n^{-\eta} + o(n^{-\eta}).
\end{equation}   
We denote the $i$-th largest connected component of $\rCM_n(\bld{d})$ by $\csi$.
For each $v\in [n]$ and $\delta>0$, let $\mathcal{N}_v(\delta)$ denote the $\delta n^{\eta}$ neighborhood of $v$ in $\rCM_n(\bld{d})$. 
For each $i\geq 1$, define 
\begin{equation} \label{c4:eq:m-i-defn}
 \mathfrak{m}_i^n(\delta) = \inf_{v\in\mathscr{C}_{\sss (i)}}n^{-\rho}\sum_{k\in \mathcal{N}_v(\delta)} w_k.
\end{equation}
For $\CM$ satisfying Assumption~\ref{c4:assumption1} and \eqref{c4:defn:criticality}, the total mass of components $n^{-\rho}\sum_{v\in \csi} w_v$ is known to converge to some non-degenerate random variable with support $(0,\infty)$ \cite[Theorem 21]{DHLS16}. 
Therefore, it is enough to rescale by $n^{\rho}$ in \eqref{c4:eq:m-i-defn} instead of the total weight of the components as given in \eqref{c4:eq:defn:GLM}.
The following theorem is the main result of this chapter:
\begin{theorem}[Global lower mass-bound]
\label{c4:thm:gml-bound} 
Suppose that \textrm{Assumption~\ref{c4:assumption1}} and \eqref{c4:defn:criticality} holds.
Then, for each  fixed $i\geq 1$, $(\sC_{\sss (i)},n^{-\eta},\bld{w})_{n\geq 1}$ satisfies global lower mass-bound, i.e., for any $\delta>0$, the sequence $(\mathfrak{m}_i^n(\delta)^{-1})_{n \geq 1}$ is tight. 
\end{theorem}
\noindent By the results of \cite{J09c}, under Assumption~\ref{c4:assumption1},
Therefore, 
\begin{equation}
\liminf_{n\to\infty} \PR(\CM \text{ is simple})>0.
\end{equation} 
This immediately implies the following corollary:
\begin{corollary}\label{c4:cor:GLM-uniform}
Under \textrm{Assumption~\ref{c4:assumption1}} and \eqref{c4:defn:criticality}, the largest components of \linebreak $\UM$ also satisfies the global lower mass-bound property.
\end{corollary}
Next we state another important corollary, which says that the global lower mass-bound property is also satisfied by critical percolation clusters of $\CM$ and $\UM$. 
To this end, let us assume that 
\begin{equation}
 \lim_{n\to\infty}\frac{\sum_{i\in [n]}d_i(d_i-1)}{\sum_{i\in [n]}d_i} = \nu >1.
\end{equation} 
$\CM$ is super-critical in the sense that there exists a unique \emph{giant} component whp for $\nu>1$,, and when $\nu<1$, all the components have size $\oP(n)$  \cite{JL09,MR95}.
 Percolation refers to deleting each edge of a graph independently with probability $1-p$.  
 The critical window for percolation was studied in \cite{DHLS16,BDHS17}, and is defined by the values of $p$ given by
 \begin{equation}\label{c4:eq:critical-window-defn}
  p_c(\lambda) = \frac{1}{\nu_n}+\frac{\lambda}{n^{\eta}}+o(n^{-\eta}).
 \end{equation}
Let $\sC_{\sss (i)}(p_c(\lambda))$ denote the $i$-th largest component of the graph obtained by percolation with probability $p_c(\lambda)$ on the graph $\CM$. 
Then the following result holds:
\begin{corollary}\label{c4:cor:GLM-percoltion}
Under \textrm{Assumption~\ref{c4:assumption1}} and \eqref{c4:eq:critical-window-defn}, $(\sC_{\sss (i)}(p_c(\lambda)),n^{-\eta},\bld{w})$ satisfies the global lower mass-bound property, for each fixed $i\geq 1$.
\end{corollary}

\subsection{Discussion} 
\paragraph*{Gap between Gromov-weak and GHP convergence.}
For formal definitions of the Gromov-weak topology, and Gromov-Hausdorff-Prokhorov (GHP) \linebreak topology on the space of compact measured metric spcaes, we refer the reader to \cite{BHS15,GPW09,ALW16}.
The Gromov-weak topology is an analogue of finite-dimensional convergence, since it takes into account distances between a finite number of sampled points from the underlying metric space.
Thus, global functionals such as the diameter is not continuous with respect to this topology. 
Further, under the Gromov-weak convergence, the limit of compact measured metric spaces may not be compact.  
On the other hand, GHP convergence imposes a stronger topology which takes care of both the above points.
The global lower mass (GLM) bound property acts as a bridge between these two notions of convergence. 
In fact, Gromov-weak convergence and GLM-bound together imply GHP-convergence when the support of the limiting measure is the entire limiting space \cite[Theorem 6.1]{ALW16}, in which case the limiting metric space is always compact. 
Thus, given Gromov-weak convergence, in order to derive convergence of global functionals like diameter, it is desirable to establish the GLM-bound.

%
\paragraph*{Scaling limit of critical percolation clusters.}
The scaling limit for largest critical percolation clusters $\sC_{\sss (i)}(p_c(\lambda))$, viewed as a measured metric space, was derived in Chapter~\ref{chap:mspace} with respect to the Gromov-weak topology.
Following the above discussion, Corollary~\ref{c4:cor:GLM-percoltion} establishes that the convergence in Chapter~\ref{chap:mspace} holds with respect to the GHP topology. 
This in particular establishes that the limiting metric spaces in \cite{BHS15,BDHS17} are compact almost surely.
Due to Assumption~\ref{c4:assumption1}~(iv), some additional conditions are imposed on $\bld{\theta}$.
For example, the assumption is satisfied for $\theta_i \in [L_1(i)i^{-a_1}, L_2(i)i^{-a_2}]$, where $a_1,a_2 \in (1/3,1/2)$, and $L_1,L_2$ are slowly varying functions.
This is much less restrictive than assuming $\theta_i = i^{-\alpha}$ as in \cite{BHS15}.
The compactness of the limiting metric spaces in \cite{BHS15,BDHS17} was also established under some regularity conditions in a very recent preprint~\cite{BDW18} 
using independent methods as in this chapter. 
In addition to the compactness of the limiting metric space, we also have the convergence of the diameters, i.e.,
\begin{eq}
\big(n^{-\eta}\diam(\sC_{\sss (i)}(p_c(\lambda)))\big)_{i\geq 1} \dto (X_i)_{i\geq 1}
\end{eq}
with respect to the product topology, where $(X_i)_{i\geq 1}$ is a non-degenerate random vector. 
In fact $X_i$ corresponds to the diameter of the limiting object of $\sC_{\sss (i)}(p_c(\lambda))$ from \cite{BDHS17}.

\paragraph*{Proof ideas and technical motivation for this work.} 
The key idea of the proof of Theorem~\ref{c4:thm:gml-bound} consists of two main steps. 
The first step is to show that the neighborhoods of the high-degree vertices, called \emph{hubs}, have mass  $\Theta(n^{\rho})$. 
Secondly, the probability of any small~$\varepsilon n^{\eta}$ neighborhood  not containing hubs is arbitrarily small.
These two facts, summarized in Propositions~\ref{c4:prop:size-nbd bound} and~\ref{c4:prop:diamter-small-comp} below, together ensure that the total mass of any neighborhood of $\sC_{\sss(i)}$ of radius $\varepsilon n^{\eta}$ is bounded away from zero. 
These two facts were proved in \cite{BHS15} in the context of inhomogeneous random graphs.
However, the proof techniques are completely different here.
The main advantage in \cite{BHS15} was that the breadth-first exploration of components could be dominated by a branching process with \emph{mixed Poisson} progeny distribution that is \emph{independent of $n$}. 
The above facts allow one to use existing literature and estimate the probabilities that a long path exists in the branching process in \cite{BHS15}.
However, such a technique is specific to rank-one inhomogeneous random graphs and does not work in the cases where the above stochastic domination is not possible. 
This was partly a motivating reason for this work.
Moreover, the final section contains many results about exponential bounds for the number of edges in the large components (Proposition~\ref{c4:lem:volume-large-deviation}), 
a coupling of the neighborhood exploration with a branching process with stochastically larger progeny distribution (Section~\ref{c4:sec:BP-approximation}), which is interesting in its own right.

\paragraph*{Organization of this chapter.}
The rest of this chapter is organized as follows: In Section~\ref{c4:sec:proof-main-thm}, we state two key propositions, one involving the total mass of small neighborhoods, and the second one involving a bound on the diameter. 
The proof of Theorem~\ref{c4:thm:gml-bound} is completed in Section~\ref{c4:sec:proof-main-thm}.
In Section~\ref{c4:sec:total-mass-hubs} we derive the required bounds on the total mass of small neighborhoods.
In Section~\ref{c4:sec:diameter-after-removal} we obtain the required bounds on the diameter. 

\section{Proof of Theorem~\ref{c4:thm:gml-bound}} \label{c4:sec:proof-main-thm}

In this section, we first state the two key propositions in Propositions~\ref{c4:prop:size-nbd bound}, and~\ref{c4:prop:diamter-small-comp}, and then complete the proof of Theorem~\ref{c4:thm:gml-bound}.
The following shows that hub $i$ has sufficient mass close to it with high probability:
\begin{proposition}
\label{c4:prop:size-nbd bound}
For each fixed $i\geq 1$ and $\varepsilon_2>0$, there exists $\delta_{i, \varepsilon_2}>0$ and $n_{i,\varepsilon_2}\geq 1$ such that, for any $\delta\in (0,\delta_{i,\varepsilon_2}]$ and $n\geq n_{i,\varepsilon_2}$, 
\begin{equation}\label{c4:eq:size-nbd bound}
\PR\bigg(\sum_{k\in\mathcal{N}_i(\delta)}w_k\leq  \theta_i \delta n^{\rho}\bigg)\leq \frac{\varepsilon_2}{2^{i+1}}.
\end{equation}
\end{proposition}

Denote by $\mathcal{G}^{\sss >K}_n$ the graph obtained by removing the vertices $1,\dots,K$ having the largest degrees and the associated edges from $\CM$. 
Note that $\mathcal{G}^{\sss >K}_n$ is a configuration model conditional on its degree sequence.
Let $\Delta^{\sss >K}$ denote the maximum of the diameters of the connected components of $\mathcal{G}^{\sss >K}_n$. 
For a component $\sC \subset \CM$, we write $\Delta(\sC)$ to denote its diameter.
The following proposition shows that the diameter of all components of $\mathcal{G}^{\sss >K}_n$ is small with high probability:

\begin{proposition}\label{c4:prop:diamter-small-comp}
Assume that \textrm{Assumption~\ref{c4:assumption1}} holds.  Then, for any $\varepsilon_1,\varepsilon_2 > 0$, there exists $K =K(\varepsilon_1,\varepsilon_2)$ and $n_0=n_{0}(\varepsilon_1,\varepsilon_2)$ such that for all $n\geq n_0$,
\begin{equation}\label{c4:eq:diamter-small-comp}
 \prob{\Delta^{\sss >K}>\varepsilon_1 n^{\eta}}\leq \frac{\varepsilon_2}{4}. 
\end{equation}
\end{proposition}
\begin{proof}[Proof of Theorem~\ref{c4:thm:gml-bound}] 
Fix any $i\geq 1$ and $\varepsilon_1,\varepsilon_2>0$. 
Let us choose $K$ and $n_0$ satisfying~\eqref{c4:eq:diamter-small-comp}. 
In view of Proposition~\ref{c4:prop:size-nbd bound}, let $\delta_0 = \min\{\varepsilon_1,\delta_{1,\varepsilon_2},\dots,\delta_{K,\varepsilon_2}\}/2$, and $n_0' = \max\{n_0,n_{1,\varepsilon_2},\dots,n_{K,\varepsilon_2}\}$. 
Thus, for all $n\geq n_0'$, \eqref{c4:eq:size-nbd bound} is satisfied for all $i\in [K]$.
Define 
\begin{equation}
F_1 := \{\Delta^{\sss >K}< \varepsilon_1 n^{\eta}/2\}, \quad F_2 := \{\Delta(\csi )>\varepsilon_1 n^{\eta}/2\}. 
\end{equation}
Notice that, on the event $F_1\cap F_2$, it must be that one of the vertices $1,2,\dots, K$ belongs to~$\csi$, and the union of the neighborhoods of $[K]$ of radius $\lceil\varepsilon_1 n^{\eta}/2\rceil+1 \approx \varepsilon_1 n^{\eta}/2$ covers~$\csi $. 
Therefore, given any vertex $v\in \csi$, $\cN_v(\varepsilon_1)$ contains at least one of the neighborhoods~$(\cN_j(\varepsilon_1/2))_{j\in [K]}$. 
This observation yields that
\begin{equation}
 \inf_{v\in\csi}n^{-\rho}\sum_{k\in \mathcal{N}_v(\varepsilon_1)} w_k \geq \min_{j\in [K]} n^{-\rho}\sum_{k\in \mathcal{N}_j(\varepsilon_1/2)} w_k \geq \min_{j\in [K]} n^{-\rho}\sum_{k\in \mathcal{N}_j(\delta_0)} w_k.
\end{equation} 
Thus, for all $n\geq n_0'$
\begin{equation}\label{c4:eq:f1-f2}
\begin{split}
 &\PR\bigg(F_1\cap F_2 \cap \bigg\{\inf_{v\in\csi}n^{-\rho}\sum_{k\in \mathcal{N}_v(\varepsilon_1)} w_k  \leq  \theta_K \delta_0 \bigg\}\bigg)\\
 &\hspace{.5cm}\leq \sum_{j\in [K]} \PR\bigg(\sum_{k\in\mathcal{N}_j(\delta)}w_k\leq  \theta_j \delta_0 n^{\rho}\bigg) \leq \frac{\varepsilon_2}{2} .
 \end{split}
\end{equation}
Further, on the event $F_2^c$, $\sum_{k \in \mathcal{N}_v(\varepsilon_1)} w_k = \sum_{k\in\csi}w_k$ for all $v\in \csi$. 
Moreover, using \cite[Theorem 21]{DHLS16}, it follows that $n^{-\rho}\sum_{k\in\csi}w_k$ converges in distribution to a random variable with  strictly positive support.
Using the Portmanteau theorem, the above implies that for any $\delta_0'>0$, there exists $\tilde{n}_0 = \tilde{n}_0(\varepsilon_2,\delta_0')$ such that, for all $n\geq \tilde{n}_0$, 
\begin{equation}
\PR\bigg(n^{-\rho}\sum_{k\in\csi}w_k\leq \delta_0'\bigg)\leq \frac{\varepsilon_2}{4}.
\end{equation}
Therefore,
\begin{equation}\label{c4:eq:f1-f2c}
\PR\bigg( F_2^c\cap \bigg\{\inf_{v\in\csi}n^{-\rho}\sum_{k\in \mathcal{N}_v(\varepsilon_1)} w_k  \leq \delta_0' \bigg\}\bigg)\leq \frac{\varepsilon_2}{4}.
\end{equation} 
Now, using \eqref{c4:eq:f1-f2}, \eqref{c4:eq:f1-f2c} together with Proposition~\ref{c4:prop:diamter-small-comp}, it follows that, for any $n\geq \max\{n_0',\tilde{n}_0\}$,
\begin{eq}
\PR\bigg(\inf_{v\in\csi}n^{-\rho}\sum_{k\in \mathcal{N}_v(\varepsilon_1)} w_k  \leq \min\{\delta_0',\theta_K\delta_0\} \bigg)\leq \varepsilon_2.
\end{eq}
This completes the proof of Theorem~\ref{c4:thm:gml-bound}.
\end{proof}

\section{Lower bound on the total mass of neighborhoods of hubs} 
\label{c4:sec:total-mass-hubs}
In this section, we prove Proposition~\ref{c4:prop:size-nbd bound}.
\begin{proof}[Proof of Proposition~\ref{c4:prop:size-nbd bound}]
Let us denote the component of $\CM$ containing vertex~$i$  by $\cs(i)$.
Consider the breadth-first exploration of $\cs(i)$ starting from vertex~$i$, given by
the following algorithm:
\begin{algo}[Exploring the graph]\label{c4:algo-expl}\normalfont  
The algorithm carries along vertices that can be alive, active, exploring and killed and half-edges that can be alive, active or killed. 
We sequentially explore the graph as follows:
\begin{itemize}
\item[(S0)] At stage $l=0$, all the vertices and the half-edges are \emph{alive}, and only the half-edges associated to vertex $i$ are \emph{active}. Also, there are no \emph{exploring} vertices except $i$. 
\item[(S1)]  At each stage $l$, if there is no active half-edge, choose a vertex $v$ proportional to its degree among the alive (not yet killed) vertices and declare all its half-edges to be \emph{active} and declare $v$ to be \emph{exploring}. If there is an active vertex but no exploring vertex, then declare the \emph{smallest} vertex to be exploring.
\item[(S2)] At each stage $l$, take an active half-edge $e$ of an exploring vertex $v$ and pair it uniformly to another alive half-edge $f$. Kill $e,f$. If $f$ is incident to a vertex $v'$ that has not been discovered before, then declare all the half-edges incident to $v'$ active, except $f$ (if any). 
If $\mathrm{degree}(v')=1$ (i.e. the only half-edge incident to $v'$ is $f$) then kill $v'$. Otherwise, declare $v'$ to be active and larger than all other vertices that are alive. After killing $e$, if $v$ does not have another active half-edge, then kill $v$ also.

\item[(S3)] Repeat from (S1) at stage $l+1$ if not all half-edges are already killed.
\end{itemize}
\end{algo}
\noindent  Call a vertex \emph{discovered} if it is either active or killed. Let $\mathscr{V}_l$ denote the set of vertices discovered up to time $l$ and $\mathcal{I}_i^n(l):=\ind{i\in\mathscr{V}_l}$.
Define the exploration process by
\begin{equation}\label{c4:def:exploration-process}
    S_n(l)= d_i+\sum_{j\neq i} d_j \mathcal{I}_j^n(l)-2l=d_i+\sum_{j\neq i} d_j \left( \mathcal{I}_j^n(l)-\frac{d_j}{\ell_n}l\right)+\bigg( \frac{1}{\ell_n} \sum_{j\neq i}d_j^2-2\bigg)l.
   \end{equation} 
Note that the exploration process keeps track of the number of active half-edges.
Thus, $\cs(i)$ is explored when $\bld{S}_n$ hits zero.
Moreover, since one edge is explored at each step, the hitting time to zero is the total number of edges in $\cs(i)$.
Define the re-scaled version $\bar{\bld{S}}_n$ of~$\bld{S}_n$ by $\bar{S}_n(t)= n^{-\alpha}S_n(\lfloor tn^{\rho} \rfloor)$. 
Then, by Assumption~\ref{c4:assumption1} and \eqref{c4:defn:criticality},
   \begin{equation} \label{c4:eqn::scaled_process}
    \bar{S}_n(t)= \theta_i-\frac{\theta_i^2 t}{\mu}+ n^{-\alpha} \sum_{j\neq i}      d_j\left(\mathcal{I}_j^n(tn^\rho)-\frac{d_j}{\ell_n}tn^{\rho} \right)+\lambda t +o(1).
   \end{equation}
Using arguments similar to \cite[Theorem 8]{DHLS16}, it can be shown that 
\begin{eq}\label{c4:eq:dist-conv-S}
\bar{\bld{S}}_n\xrightarrow{\sss d} \bld{S}_\infty,
\end{eq}
 with respect to the Skorohod $J_1$ topology, where
\begin{equation}
S_\infty(t) = \theta_i - \frac{\theta_i^2 t}{\mu} +\sum_{j\neq i}\theta_j\Big(\mathcal{I}_j(t)- \frac{\theta_jt}{\mu}\Big)+\lambda t,
\end{equation}with $\mathcal{I}_j(s):=\ind{\xi_j\leq s }$ and $\xi_j\sim \mathrm{Exponential}(\theta_j/\mu)$ independently.

Let $h_n(u)$ (respectively $h_\infty(u)$) denote the first hitting time of $\bar{\bld{S}}_n$ (respectively $\bld{S}_\infty$) to~$u$. 
More precisely, 
\begin{eq}
h_n(u):= \inf\{t: \bar{S}_n(t) \leq u \text{ or }  \bar{S}_n(t - ) \leq u\}, 
\end{eq} and define $h_\infty(u)$ similarly by replacing $\bar{S}_n(t)$ by $\bar{S}_\infty(t)$ above.
Note that for any $u>0$, $h_{\infty}(u) < h_{\infty}(u - )$ implies that $\bar{\bld{S}}_\infty(t)$ has a jump at $u$, which is a zero probability event. 
Thus, \cite[Chapter VI.2, Proposition 2.11]{JS03} is applicable and together with the convergence in \eqref{c4:eq:dist-conv-S}, this yields 
\begin{eq}\label{c4:eq:hitting-time-conv}
n^{-\rho} h_n(u)\xrightarrow{\sss d} h_\infty(u)
\end{eq} for any $u>0$.
Further, the distribution of $h_{\infty} (u)$ do not contain any atoms.
This follows using \cite[Lemma 3.5]{BHL12}. 
Now an application of Portmanteau theorem yields that
%
there exist $\beta_{\varepsilon_2,i}>0$ and $n_{i,\varepsilon_2}\geq 1$ such that, for all $n\geq n_{i,\varepsilon_2}$,
\begin{equation}
 \PR(h_n(\theta_i/2)\leq n^{\rho} \beta_{\varepsilon_2,i})\leq \frac{\varepsilon_2}{2^{i+1}}.
\end{equation} 
Firstly the goal is to show that there exists a $\delta_{i,\varepsilon}$ such that for any $\delta\in (0,\delta_{i,\varepsilon_2}]$,  
\begin{eq}
\sum_{k\in \cN_i(\delta)}d_k \leq \theta_i\delta n^{\rho}\quad \implies \quad h_n(\theta_i/2)\leq n^{\rho} \beta_{\varepsilon_2,i}.
\end{eq}
Recall that $\mathcal{N}_v(\delta)$ denotes the $\delta n^{\eta}$ neighborhood of $v$ in $\rCM_n(\bld{d})$.
To this end, let $\partial(j)$ denote the set of vertices at distance $j$ from $i$. 
Let $E_{j1}$ denote the total number of edges 
between  vertices in $\partial(j)$ and $\partial(j-1)$, and let $E_{j2}$ denote the number of edges within $\partial(j-1)$.
Define $E_j = E_{j1}+E_{j2}$.
Fix any $\delta<2\beta_{\varepsilon,i}/\theta_i$. 
Note that if $\sum_{k\in\mathcal{N}_i(\delta)}d_k\leq \theta_i \delta n^{\rho}$, then the total number of edges in $\mathcal{N}_i(\delta)$ is at most $\theta_i \delta n^{\rho}/2$. 
Thus there exists $j\leq \delta n^{\eta}$ such that $E_j\leq\theta_i\delta n^{\rho}/2\delta n^{\eta} = \theta_in^{\alpha}/2 $. 
This implies that $\bld{S}_n$ must go below $\theta_i n^\alpha/2$ before exploring all the vertices in $\cN_i(\delta)$.
This is because we are exploring the components in a breadth-first manner and $\bar{\bld{S}}_n$ keeps track of the number of active half-edges which are the potential connections to vertices at the next level.
Since one edge is explored in each time step, and we rescale time by $n^{\rho}$,  this implies that 
\begin{equation}
h_n(\theta_i/2)\leq \frac{1}{2} n^{-\rho} \sum_{k\in\mathcal{N}_i(\delta)}d_k\leq \delta\theta_i/2 \leq \beta_{\varepsilon,i}.
\end{equation}
Therefore, for all $n\geq n_{i,\varepsilon}$,
\begin{equation} \label{c4:eq:dk-bound}
\PR\bigg(\sum_{k\in\mathcal{N}_i(\delta)}d_k\leq \theta_i\delta n^{\rho}\bigg)\leq \PR(h_n(\theta_i/2)\leq \beta_{\varepsilon,i})\leq \frac{\varepsilon}{2^{i+1}}.
\end{equation}
Finally, to conclude Proposition~\ref{c4:prop:size-nbd bound} from \eqref{c4:eq:dk-bound}, we use the following result from \cite[Lemma 22]{DHLS16}:
For any $T>0$, 
\begin{equation}\label{c4:weight-expl-prop}
 \sup_{u\leq T}\bigg| \sum_{i\in [n]} w_i\mathcal{I}_i^n(un^{\rho})-\frac{\sum_{i\in [n]}d_iw_i}{\ell_n}un^{\rho}\bigg|=\oP(n^{\rho}).
\end{equation} 
Note that \cite[Lemma 22]{DHLS16} does not use $\sum_{i\in [n]}w_i  = O(\ell_n)$ from \cite[Assumption 3]{DHLS16}, and thus it is omitted in Assumption~\ref{c4:assumption1}.
The proof of Proposition~\ref{c4:prop:size-nbd bound} now follows.
\end{proof}

\section{Diameter after removing hubs}
\label{c4:sec:diameter-after-removal}
Recall the definition of the graph $\cG_n^{\sss >K}$ from Proposition~\ref{c4:prop:diamter-small-comp}. 
If we keep on exploring $\cG_n^{\sss >K}$ in a breadth-first manner using Algorithm~\ref{c4:algo-expl} and ignore the cycles created, we get a random tree. 
The idea is to couple neighborhoods of $i$ in $\cG_n^{\sss >K}$ with a suitable branching process such that the progeny distribution of the branching process dominates the number of children of each vertices in the breadth-first tree.
Therefore, if there is a long path in $\cG_n^{\sss >K}$ which makes the diameter large, that long path must be present in the branching process as well under the above coupling.
In this way, the question about the diameter of $\cG_n^{\sss >K}$ reduces to the question about the height of a branching process. 
To estimate the height suitably, we use a beautiful recent technique by Addario-Berry from \cite{A17} which allows one to relate the height of a branching process to the sum of inverses of the associated breadth-first random walk.

In Section~\ref{c4:sec:asymp-edges}, we establish large deviation bounds for the number of edges within components.
This allows us to come up with the desired coupling in Section~\ref{c4:sec:BP-approximation}. 
In Section~\ref{c4:sec:height-vs-rw}, we analyze the breadth-first random walk to show that the height of the branching process being larger than $\varepsilon n^{\eta}$ has small probability. 
These bounds are different from those derived in \cite{A17} since the branching process depends on $n$ and there is a joint scaling involved between the distances and the mean of the branching process.


\subsection{Asymptotics for the number of edges} \label{c4:sec:asymp-edges}
For a graph $G$, let $\rE(G)$ denote the number of edges in $G$.
 \begin{proposition}\label{c4:lem:volume-large-deviation} There exists $\varepsilon_0>0$ such that the following holds: For all $\varepsilon \in (0,\varepsilon_0)$, there exists $\delta>0$ such that for all sufficiently large $n$
 \begin{eq}
 \PR(\rE(\sC (i))> n^{\rho+\varepsilon}) \leq C\e^{-C n^{\delta}},
 \end{eq}for some absolute constant $C>0$ and for all $i\in [n]$.  
 \end{proposition}
Consider exploring $\CM$ with Algorithm~\ref{c4:algo-expl}, and the associated exploration process defined in \eqref{c4:def:exploration-process}. 
Let us denote by $d_{\sss (l)}$ the degree of the vertex found at step $l$. 
If no new vertex is found at step $l$, then $d_{\sss (l)} = 0$.
Also, let $\mathscr{F}_l$ denote the sigma algebra containing all the information revealed by the exploration process upto time $l$.
Thus, 
\begin{eq}
S_n(0) = d_i, \quad \text{and}\quad  S_n(l) = S_n(l-1) + (d_{\sss (l)}-2).
\end{eq}
Using the Doob-Meyer decomposition, one can write 
\begin{equation}
S_n(l) = M_n(l) + A_n(l), 
\end{equation}where $M_n$ is a martingale with respect to $(\mathscr{F}_l)_{l\geq 1}$. 
The drift $A_n$ and the quadratic variation $\langle M_n \rangle$ of $M_n$ are given by 
\begin{equation}
   A_{n}(l)= \sum_{j=1}^{l} \mathbb{E}\big[d_{\sss(j)}-2 \vert \mathscr{F}_{j-1} \big], \quad 
   \langle M_n \rangle(l)= \sum_{j=1}^{l} \var{d_{\sss(j)}\vert \mathscr{F}_{j-1}} .
  \end{equation}
  Fix $\varepsilon_0 = (4-\tau)/(\tau-1)$. 
  We use $C$ as a generic notation for an absolute constant whose value can be different in different places.
  We will show that for any $\varepsilon \in (0,\varepsilon_0)$, there exists $\delta$ such that the following two lemmas hold with $t_n := n^{\rho +\varepsilon}$:
\begin{lemma}\label{c4:lem:small-martingale}
For all sufficiently large $n$, $\PR(n^{-(\alpha+\varepsilon)} M_n(t_n) > 1) \leq C\e^{-C n^{\delta}}$. 
\end{lemma}
\begin{lemma}\label{c4:lem:drift-superlinear}
For all $K\geq 1$ the following bound holds sufficiently large $n$:
 \begin{eq}
 \PR\bigg(n^{-(\alpha+\varepsilon)} A_n(t_n) \geq -C\sum_{i=1}^K\theta_i^2\bigg) \leq C\e^{-C n^{\delta}}.
 \end{eq} 
\end{lemma}  
\begin{proof}[Proof of Proposition~\ref{c4:lem:volume-large-deviation} subject to Lemmas~\ref{c4:lem:small-martingale},~\ref{c4:lem:drift-superlinear}] 
Note that, we can choose $K\geq 1$ such that $\sum_{i=1}^K\theta_i^2$ is arbitrarily large as $\bld{\theta}\notin \ell^2_{\shortarrow}$.
Thus, if $n^{-(\alpha+\varepsilon)} M_n(t_n)$ $\leq 1$ and $n^{-(\alpha+\varepsilon)}A_n(t_n) \leq -\sum_{i=1}^K\theta_i^2$, then $n^{-(\alpha + \varepsilon)}S_n(t_n) <0$, and therefore $\sC(i)$ must be explored before time $t_n$.
Thus, Lemmas~\ref{c4:lem:small-martingale} and~\ref{c4:lem:drift-superlinear} together complete the proof of Proposition~\ref{c4:lem:volume-large-deviation}.
\end{proof}
\begin{proof}[Proof of Lemma~\ref{c4:lem:small-martingale}]
Firstly note that $\varepsilon_0 < \alpha$ and therefore $t_n = o(n)$.
Thus, uniformly over $j\leq t_n$, 
\begin{equation}
\var{d_{\sss (j)} \vert \mathscr{F}_{j-1}} \leq \E[d_{\sss (j)}^2 \vert \mathscr{F}_{j-1}]  = \frac{\sum_{j\notin \mathscr{V}_{j-1}} d_j^3}{\ell_n - 2j +2} \leq \frac{\sum_{j\in [n]} d_j^3}{\ell_n - 2t_n+2}\leq Cn^{3\alpha - 1},
\end{equation}so that almost surely,
\begin{equation}\label{c4:eq:bound-QV}
\langle M_n\rangle (t_n) \leq t_n Cn^{3\alpha-1} = C n^{2\alpha + \varepsilon}.
\end{equation}
Also, $d_{\sss (j)} \leq C n^{\alpha}$ almost surely. 
We can now use Freedman's inequality \cite[Proposition 2.1]{Fre75} to conclude that 
\begin{eq}
\PR(M_n(t_n) > n^{\alpha + \varepsilon}) \leq \exp \bigg(- \frac{n^{2\alpha + 2\varepsilon}}{2  (n^{\alpha} n^{\alpha+ \varepsilon} + Cn^{2\alpha+ \varepsilon})}\bigg) \leq C\e^{-Cn^{\varepsilon}},
\end{eq} 
and the proof follows.
\end{proof}

\begin{proof}[Proof of Lemma~\ref{c4:lem:drift-superlinear}]
Note that 
\begin{eq}\label{c4:the-split-up}
 &\mathbb{E} \big[ d_{\sss(i)} -2 \vert \mathscr{F}_{i-1} \big]  = \frac{\sum_{j \notin \mathscr{V}_{i-1}} d_{j}^2}{\ell_n-2i+1}-2\\
 &\hspace{2cm}= \frac{1}{\ell_n}\sum_{j \in [n]} d_{j}(d_{j}-2)- \frac{1}{\ell_n} \sum_{j \in \mathscr{V}_{i-1}} d_{j}^2 + \frac{(2i-1)\sum_{j \notin \mathscr{V}_{i-1}} d_{j}^2 }{\ell_n (\ell_n-2i+1)} \\
 &\hspace{2cm}\leq  \lambda n^{-\eta} - \frac{1}{\ell_n} \sum_{j \in \mathscr{V}_{i-1}} d_{j}^2 + \frac{(2i-1)}{(\ell_n-2i+1)^2} \sum_{j \in [n]} d_{j}^{2} +o(n^{-\eta})
\end{eq}uniformly over $i\leq t_n$.
%
Therefore, for all sufficiently large $n$, 
\begin{eq}\label{c4:upper-bound-drift-large}
A_n(t_n) &\leq \lambda n^{\alpha+\varepsilon} - \frac{1}{\ell_n}\sum_{i=1}^{t_n} \sum_{j \in \mathscr{V}_{i-1}} d_{j}^{2} + \frac{Ct_n^2}{\ell_n}+ o(n^{\alpha+\varepsilon}) \\
&= \lambda n^{\alpha+\varepsilon} - \frac{1}{\ell_n}\sum_{i=1}^{t_n} \sum_{j \in \mathscr{V}_{i-1}} d_{j}^{2} + o(n^{\alpha+\varepsilon}),
\end{eq}where in the last step we have used the fact that $\varepsilon< (4-\tau)/(\tau-1)$.
Let us denote the second term above by (A).
To analyze~(A), define the event 
\begin{eq}
\cA_n:= \big\{\exists j: d_j>n^{\alpha - \varepsilon/2}, j\notin \mathscr{V}_{t_n/2}\big\}.
\end{eq}
Thus, for all sufficiently large $n$,
\begin{eq}\label{c4:eq:estimate-bad-event-exponential}
\PR(\cA_n) \leq \sum_{j: d_j>n^{\alpha-\varepsilon/2}} \bigg(1-\frac{d_i}{\ell_n-2t_n}\bigg)^{t_n} \leq n \e^{-n^{\varepsilon/2}}.
\end{eq}
On the event $\cA_n^c$, 
\begin{eq}\label{c4:eq:expression-A}
\mathrm{(A)} &= \frac{1}{\ell_n} \sum_{i=1}^{t_n} \sum_{j\in [n]}d_j^2\mathbf{1}\{j\in \mathscr{V}_{i-1}\} \\
&\geq \frac{1}{\ell_n} \sum_{i = \frac{t_n}{2}+1}^{t_n} \sum_{j=1}^K d_j^2 \mathbf{1}\{j\in \mathscr{V}_{i-1}\} \geq Cn^{\alpha+\varepsilon} \sum_{j=1}^K \theta_j^2.  
\end{eq}
Combining \eqref{c4:upper-bound-drift-large}, \eqref{c4:eq:estimate-bad-event-exponential} and \eqref{c4:eq:expression-A}  now completes the proof.
\end{proof}

\subsection{Coupling with Branching processes} \label{c4:sec:BP-approximation}
Define the event $\cK_n:= \{\rE(\mathscr{C}_{\sss (i)})>n^{\rho+\varepsilon}\}$. 
On the event~$\mathcal{K}_n^c$, we can couple the breath-first exploration starting from vertex $i$ with a suitable branching process. 
Consider the branching process $\mathcal{X}_n(i)$ starting with $d_i$ individuals, and the progeny distribution $\bar{\xi}_n$ given by 
\begin{equation}\label{c4:upperbounding-BP}
\begin{split}\prob{\bar{\xi}_n=k}=\bar{p}_k=
\begin{cases}
\frac{(k+1)n_{k+1}}{\ubar{\ell}_n} \quad &\text{for } k\geq 1,\\
\frac{n_1-2n^{\rho+\varepsilon}}{\ubar{\ell}_n} \quad &\text{for } k=0,
\end{cases}
\end{split}
\end{equation}where $\ubar{\ell}_n=\ell_n-2n^{\rho+\varepsilon}$.
 Note that, at each step of the exploration, we have at most $(k+1)n_{k+1}$ half-edges that are incident to vertices having $k$ further unpaired half-edges. 
 Further, on the event $\mathcal{K}_n^c$, we have at least $\ubar{\ell}_n$ choices for pairing. 
Therefore, the number of active half-edges discovered at each step in the breadth-first exploration of the neighborhoods of $i$ is stochastically dominated by  $\bar{\xi}_n$.
This proves the next proposition, which we state after setting up some notation. 
Recall that $\cG_n^{\sss >i-1}$ denotes the graph obtained by deleting vaertices $[i-1]$ and the associated edges from $\CM$.
Let $\partial_i(r)$ denote the number of vertices at distance $r$ from $i$ in the graph $\cG_n^{\sss >i-1}$. 
Let $\bar{\xi}_n(i)$ denote the random variable with the distribution in \eqref{c4:upperbounding-BP} truncated in such a way that $\{d_1,\dots,d_{i-1}\}$ are excluded from the support.
More precisely,
\begin{eq}
\PR(\bar{\xi}_n(i)=k) =  
\begin{cases}
\frac{(k+1)n_{k+1}}{L} \quad &\text{for } 1\leq k\leq d_i,\\
\frac{n_1-2n^{\rho+\varepsilon}}{L} \quad &\text{for } k=0,
\end{cases}
\end{eq}where $L$ is the normalizing constant.
Let $\cX_{n,\mathrm{res}}(i)$ denote the branching process starting with~$d_i$ individuals and progeny distribution $\bar{\xi}_n(i)$ and let $\bar{\partial}_i(r)$ denotes the number of individuals at generation $r$ of $\mathcal{X}_n(i)$.
Then the above stochastic domination argument immediately yields the next proposition:

\begin{proposition}\label{c4:prop:coupling-uppperbound}
For all $r\geq 1$ and $i\in [n]$ and $n\geq 1$: 
\begin{eq}
\PR(\partial_i(r)\neq \varnothing) \leq \PR(\bar{\partial}_i(r)\neq \varnothing)+ \PR(\rE(\sC(i))>n^{\rho+\varepsilon}).
\end{eq}
\end{proposition}
\noindent Before going into the next section, we note that, by Assumption~\ref{c4:assumption1},
\begin{eq} \label{c4:eq:computation-mean-BP}
\bar{\nu}_n(i)&=\expt{\bar{\xi}^n(i)}=\frac{1}{\bar{\ell}_n}\sum_{j\geq  i} d_j(d_j-1)= \frac{1}{\ell_n}\sum_{j \geq  i} d_j(d_j-1) +O(n^{-\alpha + \epsilon })\\
&\leq 1 - \bigg(Cn^{-2\alpha}\sum_{j\leq i} d_j^2\bigg)n^{-\eta}  +o(n^{-\eta }).
\end{eq}
Thus for $i$ large and $n\geq n_0$, 
\begin{eq} \label{c4:eq:BP-expt}
\expt{\bar{\xi}^n(i)}\leq 1 - \beta_i n^{-\eta}, \ \text{ where }\ \beta_i=C\sum_{j\leq i}\theta_j^2.
\end{eq}
This fact will be crucially used in the next section.

\subsection{Estimating heights of trees via random walks} \label{c4:sec:height-vs-rw}
Consider a branching process $\cX_{n,\mathrm{res}}(i)$ starting with $d_i$ individuals, and progeny distribution $\bar{\xi}_n(i)$ given by \eqref{c4:upperbounding-BP}.
Thus the progeny distribution satisfies 
\begin{equation} \label{c4:progeny-BP}
\E[\bar{\xi}_n(i)] \leq 1 - \beta_i n^{-\eta}, \quad \var{\bar{\xi}_n(i)} \leq C n^{3\alpha - 1},
\end{equation} where the choices of
$\beta_i$'s are given by \eqref{c4:eq:BP-expt}.
We will prove the following theorem in this section:
\begin{theorem}\label{c4:lem:boundary-small-prob} 
Fix any $\varepsilon>0$ and let $r_0= \varepsilon n^\eta/2$. 
Then for all $i\in [n]$
\begin{equation}
\PR(\bar{\partial}_i(r_0)\neq \varnothing)\leq C2^{-\beta_i/C},
\end{equation}for some large constant $C>0$.
\end{theorem}
The estimate in Theorem~\ref{c4:lem:boundary-small-prob} is interesting in its own right and do not follow from previous asymptotic results in~\cite{A17,Kor17}.
This is due to the dependence of the branching process and the height on $n$.
%
In the proof of Theorem~\ref{c4:lem:boundary-small-prob}, we leverage the high-level ideas from~\cite{A17}.
%
Define the breadth-first random walk by
\begin{equation}\label{c4:eq:random-walk-tree}
s_n(u) = s_n(u-1) + \zeta_u -1, \quad s_n(0) = d_i,
\end{equation}
where $(\zeta_u)_{u\geq 0}$ are i.i.d.~observations from the distribution of $\bar{\xi}_n(i)$. 
Define the function 
\begin{equation}
H_n(t) = \sum_{u\in [0,t)}\frac{1}{s_n(u)},
\end{equation}and $\sigma = \inf\{u:s_n(u) = 0\}$.
It was shown in \cite[Proposition 1.7]{A17} that the height of a branching process is at most $3H_n(\sigma)$. 
Thus Theorem~\ref{c4:lem:boundary-small-prob} can be concluded from the following estimate:
%
\begin{proposition}\label{c4:prop:RW-hitting-estimate}
For any $\varepsilon>0$ and $i\in [n]$,
\begin{equation}
\PR(H_n(\sigma)\geq \varepsilon n^{\eta}) \leq C 2^{-\varepsilon \beta_i/C },
\end{equation}for some large constant $C>0$.
\end{proposition}
\noindent 
Denote $I_l:=[2^{l-1} d_i,2^{l+2} d_i)$ for $l>0$, and $I_l: = [d_i2^{l-2}, d_i2^{l+1})$ for $l<0$.
Note that $I_l$'s are not disjoint intervals.
We decompose the possible values of the random walk \eqref{c4:eq:random-walk-tree} into different scales. 
At each time $t$, the scale of $s_n(t)$, denoted by $\scl(s_n(t))$, is an integer. 
Suppose that $\scl(s_n(u)) = l$ for some $u>0$.
A change of scale occurs when $\bld{s}_n$ leaves $I_l$. 
That is, at time $T:= \inf\{t>u: s_n(t)\notin I_l\}$, a change of scale occurs, and the new scale is given by $\scl(s_n(T)) = l'$, where $l'\in \Z$ is such that $s_n(T)\in (2^{l'-1}d_i,2^{l'}d_i]$. 
Now, the next change of scale occurs at time $T':= \inf\{t>T: s_n(t)\notin I_{l'}\}$, and the scale remains the same until $T'$, i.e., $\scl(s_n(t)) = l'$ for all $T\leq t<T'$.
Define 
\begin{eq}
H_{nl}(t):= \sum_{u \in [0,t), \ \scl(s_n(u))=l}\frac{1}{s_n(u)}, \quad \text{so that} \quad H_n(t) = \sum_{l\in \Z} H_{nl}(t).
\end{eq} 
\noindent Denote $T_{nl}(t):= \#\{u\in [0,t): \scl(s_n(u))=l\}$, and note that for $l>0$
 \begin{eq}
 2^{l-1} d_i H_{nl}(t)\leq T_{nl} (t)\leq 2^{l+2}d_i H_{nl}(t),
 \end{eq}and a similar inequality holds for $l<0$.
Therefore, for any $x>0$ and $l>0$, 
\begin{equation}\label{c4:eq:T-H-relation}
\PR\Big(H_{nl}(\sigma)\geq \frac{x}{2^{l-1}d_i}\Big)\leq \PR(T_{nl}(\sigma)\geq x),
\end{equation}and a similar inequality holds for $l<0$.
Thus the proof of Proposition~\ref{c4:prop:RW-hitting-estimate} follows from a careful estimate of the final term in  \eqref{c4:eq:T-H-relation}, which is given by the next lemma.
Let $\bld{s}_n'$ be a random walk given by the same recursion relation as \eqref{c4:eq:random-walk-tree}, except only that $s_n'(0)\in I_l$.
Let $\sigma_{nl} : = \min\{t\geq 1: s_n'(t) \notin I_l\}$ and
 $r_{nl}: = \min\{t\geq 1:\sup_{x\in I_l}\PR_x(\sigma_{nl}>t)\leq 1/2\}$.
\begin{lemma}\label{c4:lem:time-spent-l-ub} 
For all $n\geq 1$, and $l\in \Z$:
\begin{equation}
\PR(T_{nl}(\sigma)\geq a r_{nl})\leq C\min\{1,2^{-l}\}2^{-a/C},
\end{equation}for some large constant $C>0$.
\end{lemma}
\begin{proof}
Firstly, note that $T_{nl}(\sigma)\neq 0$ if and only if $\scl(s_n(u)) =l$ for some $u<\sigma$. 
The number of upcrossings of an interval $[a,b]$ by $s_n$ is defined to be the supremum of the integers $k$ such that there exists times $(u_j,t_j)_{j=1}^k$ satisfying $u_1<t_1<u_2<\dots<t_k$, and $s_n(u_j)<a<b<s_n(t_j)$ for all $j\in [k]$. 
Now, for any $l\geq 2$, if $\scl(s_n(u)) =l$ occurs, then $s_n$ must have made an upcrossing of the intervals $((2^{j-1}d_i,2^jd_i])_{1\leq j\leq l}$.
Using \cite[Lemma 3.1]{A17}, it follows that there exists a constant $C>0$ such that for any $l\geq 2$,
 \begin{eq}
\PR(T_{nl}(\sigma) \neq 0)\leq C2^{-l}.
\end{eq}
Moreover, we bound $\PR(T_{nl}(\sigma)\neq 0)$ by 1 for $l\leq 1$. 
Next define $\mathrm{visit}(l,t)$ to be the number of visits to scale $l$, i.e., this is the supremum over $k\in \N$ such that one can find $(u_j,t_j)_{j=1}^k$ with $u_1<t_1< \dots <u_k<t_k$ satisfying $\scl(s_n(u_j))\neq l$ but $\scl(s_n(t_j)) = l$. 
Set $\mathrm{visit}(1,0) =1$ and $\mathrm{visit}(l,t) = 0$ if $\scl(s_n(t)) \neq l$.
Further, define $M_{nl} = \mathrm{visit}(l,\sigma)$ (total number of visits to scale $l$) and $t_{jl} = \#\{t<\sigma:\scl(s_n(t))=l, \mathrm{visit}(l,t)=j \}$  (the time spent at scale $l$ during the $j$-th visit).
Thus $T_{nl} (\sigma) = \sum_{j=1}^{M_{nl}} t_{jl}$, and for $m\geq 2$,
\begin{eq}\label{eq:Tnl-split-up}
\PR\bigg(\sum_{j=1}^{M_{nl}} t_{jl}>a r_{nl}\bigg)\leq \PR(M_{nl}>m)+\PR\bigg(\sum_{j=1}^{m} t_{jl}>ar_{nl}\bigg).
\end{eq} 
Now $s_n$ can enter scale $l$ from below, which yields an upcrossing of the interval $[2^{l-1}d_i,2^ld_i)$. 
Otherwise, $s_n$ can enter scale $l$ from above, whence it must be the case that while leaving the scale $l$ during the previous visit, the walk went from scale $l$ to $l+1$. 
The latter case yields an upcrossing of $[2^{l}d_i,2^{l+1}d_i)$.
Therefore, if $U_n(t,[a,b))$ denotes the number of upcrossings of  $[a,b)$ by $s_n$ before time $t$, then
\begin{eq}
&\PR(M_{nl}>m)\\
&\leq \PR\Big( U_n(\sigma,[2^{l-1}d_i,2^ld_i))\geq (m+1)/2\Big)\\
&\hspace{2cm}+\PR\Big(U_n(\sigma,[2^{l}d_i,2^{l+1}d_i)\geq (m+1)/2\Big) \\
&\leq \frac{1}{2^{(m-1)/2}}.
\end{eq}
On the other hand, after each time $r_{nl}$, the probability of exiting from scale $l$ is at most 1/2, by definition.
Now, $\PR(t_{jl} > k r_{nl}) \leq 2^{-k}$, which implies that $\lfloor t_{jl}/r_{nl}\rfloor$ can be stochastically dominated by Geometric$(1/2)$ random variable.
 Thus, if $(g_i)_{i\geq 1}$ denotes an i.i.d.~collection of Geometric$(1/2)$ random variables, 
\begin{eq}
\PR\bigg(\sum_{j=1}^mt_{jl}\geq (k+m)r_{nl}\bigg) &\leq \PR\bigg(\sum_{j=1}^m\Big\lfloor\frac{t_{jl}}{r_{nl}}\Big\rfloor\geq k\bigg)\leq \PR\bigg(\sum_{i=1}^m g_i>k\bigg) \\
&= \PR(\mathrm{Bin}(k,1/2)<m)\leq \e^{-(k-2m)^2/2k},
\end{eq}where the last step follows using standard concentration inequalities such as \cite[Theorem 2.1]{JLR00}.
Therefore, the proof follows by taking $k=m=a/2$.
\end{proof}
For a sequence $a = (a_l)_{l\in \Z}$, define $V_n(a) = d_i^{-1}\sum_{l\in\Z}a_lr_{nl}/2^l$ and $\delta(a) = C\sum_{l\in \Z} \min\{1,2^{-l}\}2^{-Ca_l }$.
Using Lemma~\ref{c4:lem:time-spent-l-ub} and \eqref{c4:eq:T-H-relation}, we can now conclude that 
\begin{equation}\label{c4:eq:RW-inverse-total}
\PR(H_n(\sigma)>V_n(a)) \leq \delta(a).
\end{equation}
To apply the above bound, we need a good estimate on $r_{nl}$.  To apply the above bound, we need a good estimate on $r_{nl}$. 
Let $\PR_x$ denote the law of the random walk $s_n'$, with $s_n'(0) = x$, but satisfying identical recurrence recurrence relation as \eqref{c4:eq:random-walk-tree}.
 Suppose that $\sigma_{nl} : = \min\{t\geq 1: s_n(t) \notin I_l, s_n(0) \in I_l\}$, and $r_{nl}: = \min\{t\geq 1:\sup_{x\in I_l}\PR_x(\sigma_{nl}>t)\leq 1/2\}$.
The next two lemmas allow us to deduce such a result for $l>0$ and $l<0$ respectively:
\begin{lemma}\label{c4:lem:rw-exit}
Fix any $l\geq 0$ and let $\sigma_{nl}:= \inf\{t:s_n(t)\notin I_l\}$. 
Then, for all $i\in [n]$, 
\begin{equation}
\limsup_{n\to\infty}\sup_{x\in I_l}\PR_x \bigg(\sigma_{nl}> C n^{\eta} \frac{d_i 2^{l/2}}{\beta_i^{1/(b-1)}}\bigg) \leq \frac{1}{2}.
\end{equation} 
\end{lemma}
\begin{proof}
Fix any $x\in I_l$.
Note that for any $t>0$,
\begin{eq}\label{c4:sigmanl-bound}
&\PR_x(\sigma_{nl}>tn^{\rho})\\
& \leq \PR_x(S_n(tn^{\rho})\in I_l) \leq \PR_x(S_{n}(tn^{\rho})>2^{l-1}d_i)  \\
&= \PR_x\big(S_{n}(tn^{\rho}) + tn^{\rho}n^{-\eta}\beta_i> 2^{l-1}d_i+tn^{\alpha}\beta_i\big) \\
&\leq \frac{tn^{\rho}c_0n^{3\alpha -1}}{(2^ld_i+tn^{\alpha}\beta_i)^2} \leq \frac{c_0tn^{2\alpha}}{2^{2l}d_i^2},
\end{eq} where the last step follows from Chebyshev's inequality and the estimates in~\eqref{c4:progeny-BP}.
For $l\geq 0$, by setting $t  = \frac{d_i 2^{l/2}}{n^{\alpha} (\beta_i)^{1/(b-1)}}$ where $b$ is given by \eqref{c4:eq:assumption-degree-tail}, \eqref{c4:sigmanl-bound} reduces to
\begin{eq}\label{c4:sigmanl-bound-2}
\PR(\sigma_{nl}>tn^{\rho})  \leq \frac{C}{2^{l} (d_i/n^{\alpha})\beta_i^{1/(b-1)}} \leq \frac{C}{2^l},
\end{eq}where we have used Assumption~\ref{c4:assumption1} in the last step.
This completes the proof of Lemma~\ref{c4:lem:rw-exit} using \eqref{c4:sigmanl-bound-2}.
%
%
\end{proof}
\begin{lemma}\label{c4:lem:RW-exit-small}
For $l>0$, and $i\in [n]$,
\begin{equation}
\limsup_{n\to\infty}\sup_{x\in (d_i/2^{l+2},d_i/2^{l-1})}\PR_x \bigg(\sigma_{nl}> C n^{\eta} \frac{d_i }{\beta_i 2^{bl}}\bigg) \leq \frac{1}{2}.
\end{equation}
\end{lemma}
\begin{proof}
Recall from the definition of $\xi_n(i)$ from \eqref{c4:progeny-BP}, and let $(\xi_{nj})_{j\geq 1}$ be an iid collection with the same distribution as $\xi_n(i)$.
Note that, for $x\in I_l$,
\begin{eq}
&\PR_x(\sigma_{nl}>tn^{\rho}) \leq \PR_x(S_n(tn^{\rho})\in I_l) \leq \PR\Big(\xi_{nj}\leq \frac{d_i}{2^{l-1}} \ \forall j\leq tn^{\rho}\Big) \\
&\hspace{1cm} \leq \bigg(1-\frac{\sum_{j:d_j>d_i/2^{l-1}}d_j}{\ell_n(1+o(1))}\bigg)^{tn^{\rho}} \leq \exp\bigg( - t n^{-\alpha}\sum_{j:d_j>d_i/2^{l-1}}d_j\bigg).
\end{eq}
Putting $t= \frac{d_i}{n^{\alpha}\beta_i2^{bl}}$ and using Assumption~\ref{c4:assumption1}, it follows that the right-hand side above is at most 
\begin{eq}
\exp\bigg(- \frac{C}{(d_i/n^{\alpha})^{b-1} n^{-2\alpha} \sum_{j\leq i}d_j^2}\bigg),
\end{eq}
and the proof follows.
\end{proof}

\begin{proof}[Proof of Proposition~\ref{c4:prop:RW-hitting-estimate}] 
We will use the estimates of the terms appearing in \eqref{c4:eq:RW-inverse-total}. Firstly, note from Lemmas~\ref{c4:lem:rw-exit},~\ref{c4:lem:RW-exit-small} that, for all sufficiently large $n$, 
\begin{equation}
r_{nl} \leq \begin{cases}
\frac{d_i 2^{l/2}}{\beta_i^{1/(b-1)}}\quad &\text{for } l>0,\\
\frac{d_i }{2^{bl}\beta_i}\quad &\text{for } l<0.
\end{cases}
\end{equation}
Moreover, take $b_l  = \varepsilon 2^{ l/4 },$ for all $l\geq 0$ and $b_l = \varepsilon 2^{-(b-1)l/2}$ for $l<0$.
In the above case, there exists an absolute constant $C>0$ such that
\begin{eq}
V_n(b)\leq \frac{ C \varepsilon n^{\eta}}{\beta_i},
\end{eq}where the constant only depends on the parameter $\tau$.
Moreover, $\delta(b) \leq C2^{-\varepsilon}.$
Therefore,~\eqref{c4:eq:RW-inverse-total} yields that
\begin{eq}
\PR\bigg(H_n>\frac{ C \varepsilon n^{\eta}}{\beta_i}\bigg)\leq C 2^{-\varepsilon}.
\end{eq} 
Taking $\varepsilon' = C\varepsilon/\beta_i$, the proof of Proposition~\ref{c4:prop:RW-hitting-estimate} follows.
\end{proof}

%

%
%
%
%
%

\subsection{Proof of Proposition~\ref{c4:prop:diamter-small-comp}} 
Let us now complete the proof of Proposition~\ref{c4:prop:diamter-small-comp} using Propositions~\ref{c4:prop:coupling-uppperbound}, and Theorem~\ref{c4:lem:boundary-small-prob}.
Define $\sC_{ \mathrm{res}} (i) $ to be the connected component containing vertex $i$ in the graph $\cG_n^{\sss > i-1}=\CM \setminus [i-1]$.
Note that if $\Delta^{\sss >K}>\varepsilon_1n^{\eta}$, then there exists
a path in $\CM$ avoiding all the vertices in $[K]$. 
Suppose that the minimum index among vertices on that path is $i_0$.
Then $\Delta (\sC_{ \mathrm{res}} (i_0) )>\varepsilon n^{\eta}$.
Therefore, $\Delta^{\sss >K}>\varepsilon_1n^{\eta}$ implies that there exists an $i>K$ satisfying $\Delta (\sC_{ \mathrm{res}} (i) )>\varepsilon n^{\eta}$.
Let $\partial_i(r)$ denotes the number of vertices at distance $r$ starting from vertex $i$ in the graph $\cG_n^{\sss > i-1}$. 
Recall the definition of $\bar{\partial}$ in Proposition~\ref{c4:prop:coupling-uppperbound}.
Thus, 
\begin{equation}
\prob{ \Delta^{\sss >K}>\varepsilon_1 n^{\eta} } \leq \sum_{i>K}\PR(\partial_i(\varepsilon_1n^{\eta}/2)\neq \varnothing)\leq C\delta\sum_{i>K} \e^{-\beta_i} 
\end{equation}which tends to zero if we first take $n\to\infty$ and then take $K\to\infty$ using Assumption~\ref{c4:assumption1}. 
Thus the proof follows. \qed

\section{Conclusion}
We prove a global lower mass-bound property for the largest components of the critical configuration model when the third moment of the degree distribution diverges to infinity. 
Together with the results in Chapter~\ref{chap:mspace}, this proves that the scaling limits in Chapter~\ref{chap:mspace} hold with respect to the Gromov-Hausdorff-Prokhorov topology, and the limiting metric space in Chapter~\ref{chap:mspace} is almost surely compact under some regularity conditions.
Also, this implies that the diameter of these components converge to some non-degenerate random variable after rescaling by $n^{(\tau-3)/(\tau-1)}$.
The main proof technique involves an exponential bound on the probability that the height of a sequence of branching processes is large, which may be of independent interest. 


%
%
%
%
%
%

%

%
\cleardoublepage

\chapter[Critical window: Infinite second moment]{Critical percolation on scale-free random graphs: Effect of the single-edge constraint}
\label{chap:infinite-second}
{\small \paragraph*{Abstract.}
In this chapter, we study the percolation critical behavior for random graphs with degree distributions having a power-law with exponent $\tau \in (2,3)$. 
In this regime, the critical behavior is observed when the percolation probability tends to zero with the network size. 
We identify the critical window for the configuration model, the erased configuration model and the generalized random graph. 
The critical window turns out to be different for the multigraph version of the configuration model, a feature that is not observed for $\tau > 3$.
We provide exact asymptotics of the rescaled component sizes, and describe many structural properties of these critical components.
We also analyze the so-called barely sub/supercritical regimes, which establishes the relevance of the critical window identified in this chapter.
}

\vspace{.3cm} 
\noindent {\footnotesize Based on the preprint: 
Souvik Dhara, Remco van der Hofstad, Johan S.H. van Leeuwaarden; 
\emph{Critical percolation on scale-free random graphs: Effect of the single-edge constraint} (2018)}
\vfill
All the results and the relevant literature discussed in the previous chapters assume a finite second-moment condition on the degree distribution, and thus do not include the $\tau\in (2,3)$ case, where the asymptotic degree distribution has infinite second moment but finite first moment.
These networks are popularly known as \emph{scale-free networks} \cite{Bar16} in the literature. 
One of the well-known features of scale-free networks is that these networks are \emph{robust} under random edge-deletion, i.e., for any sequence $(p_n)_{n\geq 1}$ with $\liminf_{n\to\infty} p_n > 0$, the graph obtained by applying percolation with probability $p_n$ is always supercritical. 
This feature has been studied experimentally in~\cite{AJB00}, using heuristic arguments in~\cite{CEbAH00,CNSW00,DGM08,CbAH02} (see also \cite{braunstein2003optimal,braunstein2007optimal,Halvin05} in the context of optimal paths in the strong disorder regime), and mathematically in \cite{BR03}. 
Thus, in order to observe the percolation critical behavior, one needs to take $p_c \to 0$ with the network size, even if the average degree of the network is finite.
However, obtaining the right scaling exponents for the critical behavior was an open question in the mathematical literature.

In this chapter, we discuss the first mathematically rigorous results in the $\tau \in (2,3)$ regime for the critical behavior of component sizes and their complexity.
As canonical random graph models on which percolation acts, we take the multigraph generated by the configuration model, and the closely associated erased configuration model, obtained by deleting self-loops and multiple-edges in the configuration model. 
The latter model is often referred to in the literature as the configuration model with single-edge constraint.
The most striking observation of this chapter is that the critical value changes depending on the \emph{single-edge constraint}, a feature that has never surfaced in the finite second-moment setting.
For the configuration model multigraph, the critical value turns out to be $p_c \sim n^{-(3-\tau)/(\tau-1)}$, whereas under the single-edge constraint $p_c \sim n^{-(3-\tau)/2}$, which is much larger $n^{-(3-\tau)/(\tau-1)}$.
The largest component sizes in both the regimes are of the order $n^\alpha p_c$, and the scaling limits are in a completely different universality class than the $\tau\in (3,4)$ and $\tau>4$ case. 
 We also study percolation on the generalized random graph, which gives uniformly chosen graph conditional on degrees. 
The contributions of this chapter can be summarized as follows:
\begin{enumerate}
\item For the configuration model multigraph, we obtain scaling limits for the largest component sizes and surplus edges under a strong topology. 
Further, the diameter of the largest components is shown to be tight random variables. 
To establish that the scaling limits correspond to the critical behavior, we further look at the near-critical behavior and derive the asymptotics for the component sizes in the so-called barely sub/supercritical regimes.

\item Under the single-edge constraint, we identify the scaling limit of the largest component sizes in the part of the critical window, where the criticality parameter is sufficiently small. 

\item This is the first work on critical percolation on random graphs in the $\tau\in (2,3)$ setting, thus the techniques are novel. 
The primary difficulty in this setting is that the exploration process approach does not work. 
For the configuration model, this difficulty is circumvented by sandwiching the percolated graphs by two configuration models, which yield the same scaling limits for the component sizes. 
The main novelty in the proof of the configuration model is the analysis of the limiting exploration process. 

\item On the other hand, in the single-edge constraint scenario, the proofs 
require a more detailed understanding of the structure of the critical components. 
It turns out that the hubs (vertices of high-degree) do not connect to each other directly, but there are some special vertices that interconnect hubs. 
This interconnected structure forms the \emph{core} of the critical components, and the 1-neighborhood of the core spans the critical components.
We primarily use path counting techniques here since the exploration process approach does not work anymore. 
For path counting, we compare the connection probabilities between the hubs with the connection probabilities in a preferential attachment model, which is interesting in its own right.
\end{enumerate}


\section{Main results}
\subsection{The configuration model}
\subsubsection{Notions of convergence and the limiting objects} 
\label{c5:sec:notation}
Recall the notations from Chapter~\ref{sec:notation-intro}.
Consider a decreasing sequence $ \boldsymbol{\theta}=(\theta_1,\theta_2,\dots)\in \ell^2_{\shortarrow}\setminus \ell^1_{\shortarrow}$. 
Denote by  $\mathcal{I}_i(s):=\ind{\xi_i\leq s }$ where $\xi_i\sim \mathrm{Exp}(\theta_i/\mu)$ independently, and $\mathrm{Exp}(r)$ denotes the exponential distribution with rate $r$.  
Consider the process \begin{equation}\label{c5:defn::limiting::process}
\iS(t) =  \lambda\sum_{i=1}^{\infty} \theta_i\mathcal{I}_i(t)- 2 t,
\end{equation}for some $\lambda\in\mathbb{R}, \mu >0$ and define the reflected version of $\iS(t)$ by
\begin{equation} \label{c5:defn::reflected-Levy}
 \refl{ \iS(t)}= \iS(t) - \min_{0 \leq u \leq t} \iS(u).
\end{equation}
For any function $f\in \mathbb{D}[0,\infty)$,  define $\ubar{f}(x)=\inf_{y\leq x}f(y)$.  $\mathbb{D}_+[0,\infty)$ is the subset of $\mathbb{D}[0,\infty)$ consisting of functions with positive jumps only. Note that $\ubar{f}$ is continuous when $f\in \mathbb{D}_+[0,\infty)$. An \emph{excursion} of a function $f\in \mathbb{D}_+[0,T]$ is an interval $(l,r)$ such that 
\begin{gather*}\min\{f(l-),f(l)\}=\ubar{f}(l)=\ubar{f}(r)=\min\{f(r-),f(r)\} \\
 \text{and}\quad f(x)>\ubar{f}(r),\ \forall x\in (l,r)\subset [0,T].
\end{gather*}
Excursions of a function $f\in \mathbb{D}_+[0,\infty)$ are defined similarly.
We will show that, for any $\lambda>0$, the excursions of the process $\biS$ can be ordered almost surely as an element of $\ell^2_{\shortarrow}$. We denote this ordered vector by $(\gamma_i(\lambda))_{i\geq 1}$.

 Also, define the counting process $\mathbf{N}^\lambda$ to be the Poisson process that has intensity $\refl{ \iS(t)}$ at time $t$ conditional on $( \iS(u) )_{u \leq t}$. Formally, $\mathbf{N}^\lambda$ is characterized as the counting process for which 
\begin{equation} \label{c5:defn::counting-process}
N(t) - \frac{\sum_{i=1}^\infty\theta_i^2}{\mu^2} \int\limits_{0}^{t}\refl{ \iS(u)}\dif u
\end{equation} is a martingale.  We use  the notation $N(\gamma)$ to denote the number of marks in the interval $\gamma$.
Let $\mathbf{Z}(\lambda)$ denote the vector $((\gamma_i(\lambda),N(\gamma_i(\lambda))))_{i\geq 1}$, ordered as an 
 element of $\Unot$.
 
\subsubsection{Results for the critical window}
Fix any $\tau \in (2,3)$. 
Throughout this chapter, we denote 
\begin{equation}\label{c5:eqn:notation-const}
 \alpha= 1/(\tau-1),\qquad \rho=(\tau-2)/(\tau-1),\qquad \eta=(3-\tau)/(\tau-1),
\end{equation}
and assume the following about the degree sequences $(\bld{d}_n)_{n\geq 1}$:
\begin{assumption}
\label{c5:assumption1}
\normalfont  
\begin{enumerate}[(i)] 
\item \label{c5:assumption1-1} (\emph{High-degree vertices}) For any  $i\geq 1$, 
$ n^{-\alpha}d_i\to \theta_i,$
where the vector $\boldsymbol{\theta}:=(\theta_1,\theta_2,\dots)\in \ell^2_{\shortarrow}\setminus \ell^1_{\shortarrow}$. 
\item \label{c5:assumption1-2} (\emph{Moment assumptions}) 
 $\lim_{n\to\infty}\frac{1}{n}\sum_{i\in [n]}d_i= \mu,$ and $$ \lim_{K\to\infty}\limsup_{n\to\infty}n^{-2\alpha} \sum_{i=K+1}^{n} d_i^2=0.$$
\end{enumerate}
\end{assumption}
\noindent 
In Section~\ref{c5:sec:discussion}, we discuss that Assumption~\ref{c5:assumption1} is satisfied for power-law degrees with exponent $\tau\in (2,3)$. 
For $\CM$, the \emph{criticality parameter} $\nu_n$ is defined as
\begin{equation}
\nu_n = \frac{\sum_{i\in [n]}d_i(d_i-1)}{\sum_{i\in [n]}d_i}.
\end{equation} 
Molloy and Reed~\cite{MR95}, and Janson and Luczak~\cite{JL09} showed that, under some regularity conditions, $\CM$ has a unique giant component (a component of size $\Theta(n)$) with high probability precisely when $\nu_n \to \nu>1$. 
Under Assumption~\ref{c5:assumption1}, $\nu_n\to\infty$, as $n\to\infty$ and $\CM$ always contains a giant component.

 Let $\rCM_n(\bld{d},p)$ denote the graph obtained from percolation with probability $p$ on the graphs $\mathrm{CM}_n(\boldsymbol{d})$. 
Now, under Assumption~\ref{c5:assumption1}  for any $p\in (0,1]$, $\rCM_n(\bld{d},p)$ retains a giant component with high probability, i.e.~$\rCM_n(\bld{d},p)$ is always supercritical;  see the remark below \cite[Theorem 4.5]{Hof17}.
Thus, in order to observe the critical behavior, one must take $p \to 0$, as $n\to\infty$. 
However, it is is not clear here how to define the critical window of phase-transition.
%
One way to do it is to say that inside the critical window, the order of the sizes of largest connected components are same, and the rescaled vector of ordered component sizes converge to some \emph{non-degenerate} random vector.
This property has been observed universally for the critical window of phase transition; see \cite{DHLS15,DHLS16} and the references therein.
In this chapter, we define the critical window to be those values of $p$ for which the re-scaled vector of component sizes converge to some non-degenerate random vector. 
It is worthwhile mentioning that there is a substantial literature on how to define the critical value, and the phase transition. 
See \cite{NP08,JW18,BCHSS05,HvdH17,Hof17} for different definitions of critical probability and related discussions.

We will show that the critical window for percolation on $\CM$ is given  by 
\begin{equation}\label{c5:eq:crit-window-CM}
p_c=p_c(\lambda):= \frac{\lambda}{\nu_n}(1+o(1)), \quad \lambda \in (0,\infty).
\end{equation}
Notice that, under Assumption~\ref{c5:assumption1}, the parameter $\nu_n$ is of the order $n^{2\alpha-1} = n^{\eta}$, where $\eta = (3-\tau)/(\tau-1)>0$.
 
To avoid complicated notation, we will always write $\sC_{\sss (i)}(p)$ to denote the $i$-th largest component in the percolated graph. The random graph on which percolation acts will always be clear from the context. 
A vertex is called isolated if it has degree zero in the graph $\CMP$. 
We define the component size corresponding to an isolated vertex to be zero (see Remark~\ref{c5:rem:isolated} below). 
For any component $\mathscr{C}\subset\CMP$, let $\SP(\mathscr{C})$ denote the number of surplus edges given by $\text{the number of edges in }\mathscr{C}-|\mathscr{C}|+1$. 
Finally, let $\mathbf{Z}_n(\lambda)$ denote the vector $(n^{-\rho}|\sCi(p_c(\lambda))|,\SP(\sCi(p_c(\lambda))))_{i\geq 1}$, ordered as an element of~$\Unot$.
The following theorem gives the asymptotics for the critical component sizes and the surplus edges of  $\CMP$: 
\begin{theorem}[Component sizes and surplus edges]\label{c5:thm:main} Under \textrm{Assumption~\ref{c5:assumption1}}, as $n\to\infty$,
\begin{equation}
\mathbf{Z}_n(\lambda) \dto \mathbf{Z}(\lambda)
\end{equation}with respect to the $\Unot$ topology, where $\mathbf{Z}(\lambda)$ is defined in Section~\ref{c5:sec:notation}.
\end{theorem}
\begin{remark}\label{c5:rem:isolated} \normalfont
Note that if $\tau\in (2,3)$, then $2\rho<1$. 
When percolation is performed with probability $p_c$, there are of the order $n$ isolated vertices and thus $n^{-2\rho}$ times the number of isolated vertices tends to infinity. 
This is the reason why we must ignore the contributions due to isolated vertices, when considering the convergence of the component sizes in the $\ell^2_{\shortarrow}$ topology.
Note that an isolated vertex with self-loops does not create an isolated component. 
\end{remark}

For a connected graph $G$, let $\diam(G)$ denote its diameter. 
Our next result shows that the diameter of the largest connected components is of constant order. 

\begin{theorem}[Diameter of largest clusters]\label{c5:thm:diameter-large-comp} Under \textrm{Assumption~\ref{c5:assumption1}}, for any $i\geq 1$, $(\diam(\sCi(p_c(\lambda))))_{n\geq1}$ is a tight sequence of random variables, .
\end{theorem}
\subsubsection{Behavior in the near-critical regimes}
We now present asymptotic results for the component sizes in the so-called barely subcritical ($p_n\ll p_c(\lambda)$) and barely supercritical regimes ($p_n\gg p_c(\lambda)$).
The following two theorems summarize the near-critical behavior:
\begin{theorem}[Barely subcritical regime]\label{c5:thm:barely-subcrit} For $\rCM_n(\bld{d},p_n)$, let us assume that $\frac{\log(n)}{\ell_n}\ll p_n\ll p_c(\lambda)$ and that \textrm{Assumption~\ref{c5:assumption1}} holds. Then, for each fixed $i\geq 1$, as $n\to\infty$,
\begin{equation}
\frac{|\sCi(p_n)|}{n^\alpha p_n} \pto \theta_i.
\end{equation}
\end{theorem}
For the result about the barely supercritical regime, we need one further mild technical assumption, which is as follows:
Let $D_n^*$ denote the degree of a vertex chosen in a size-biased manner with the sizes being $(d_i/\ell_n)_{i\in [n]}$.
Then, there exists a constant $\kappa>0$ such that 
\begin{equation}\label{c5:eq:asymp-laplace}
1 - \E[\e^{-tp_n^{1/(3-\tau)}D_n^*}] = \kappa p_n^{(\tau-2)/(3-\tau)}(t^{\tau-2}+o(1)).
\end{equation}
\begin{theorem}[Barely supercritical regime]\label{c5:thm:barely-supercrit} 
For $\rCM_n(\bld{d},p_n)$, suppose that $p_n\gg p_c(\lambda)$ and \textrm{Assumption~\ref{c5:assumption1}}, \eqref{c5:eq:asymp-laplace} holds. Then, as $n\to\infty$,
\begin{equation}
\frac{|\sC_{\sss (1)}(p_n)|}{np_n^{1/(3-\tau)}} \pto \frac{\mu\kappa^{1/(3-\tau)}}{2^{(\tau-2)/(3-\tau)}}, \quad \frac{\rE(\sC_{\sss (1)}(p_n))}{np_n^{1/(3-\tau)}} \pto \frac{\mu\kappa^{1/(3-\tau)}}{2^{(4-\tau)/(3-\tau)}},
\end{equation}and for all $i\geq 2$, $|\sCi(p_n)| = \oP(np_n^{1/(3-\tau)})$, $\rE(\sCi(p_n)) = \oP(np_n^{1/(3-\tau)})$, where $\rE(G)$ denotes the number of edges in the graph $G$.
\end{theorem}

\begin{remark} \normalfont 
The identity \eqref{c5:eq:asymp-laplace} is basically a version of the celebrated Abel-Tauberian theorem \cite[Chapter XIII.5]{Fel91} (see also \cite[Chapter 1.7]{BGT89}). 
However, since both $D_n^*$ and $p_n$ depend on $n$, the joint asymptotics needs to be stated as an assumption.
In Section~\ref{c5:sec:discussion}, we discuss how this assumption is satisfied for power-law degree distributions with $\tau\in (2,3)$. 
\end{remark}

\subsection{Effect of the single-edge constraint}
In this section, we will consider two random graph models that do not allow for self-loops or multiple edges in the graph, namely the generalized random graph and the erased configuration model. 
We will see that for random graphs that generate simple graphs, the critical window for percolation is given by 
\begin{equation}\label{c5:eq:scaling-window}
p_c =p_c(\lambda):= \lambda n^{-\frac{3-\tau}{2}}(1+o(1)), \quad \lambda \in (0,\infty).
\end{equation} 
We state the results about two different random graph models in two different subsections below, starting with the generalized random graph:

\subsubsection{Generalized random graphs}\label{c5:sec:GRG-results}
Given a set of weights $(w_i)_{i\in [n]}$ on the vertex set $[n]$, the generalized random graph model  \cite{BDM06}, denoted by $\GRG$ is generated by creating an edge between vertex $i$ and $j$ independently with probability 
\begin{equation}\label{c5:eq:p-ij-GRG-defn}
p_{ij} = \frac{w_iw_j}{\ell_n+w_iw_j},
\end{equation}
where $\ell_n = \sum_{i\in [n]}w_i$. 
This model has the property that, conditionally on the degree sequence~$\bld{d}$, the law of the obtained random graph is the same as a uniformly chosen graphs from the space of all simple graphs with degree distribution $\bld{d}$.
The graph $\rGRG(\bw,p)$ is obtained by keeping each edge of the graph independently with probability $p$. The later deletion process is also independent of the randomization of the graph.
\begin{assumption}\label{c5:assumption-GRG}
\normalfont
For some $\tau \in (2,3)$, consider the distribution function satisfying $(1-F)(x) = C x^{-(\tau-1)}$ and let $w_i = (1-F)^{-1}(i/n)$. 
\end{assumption} 
\noindent In the above case, if $W_n$ denotes the weight of a typical vertex, then 
\begin{equation}
\E[W_n] = \frac{1}{n}	\sum_{i\in [n]}w_i \to \mu = \E[W].
\end{equation}
Moreover, $$n^{-\alpha}w_i  = c_{\sss F}i^{-\alpha},$$ for some constant $c_F >0$.
Throughout $c_{\sss F}$ will denote the special constant appearing above.
Assumption~\ref{c5:assumption-GRG} is strictly stronger than Assumption~\ref{c5:assumption1} in the sense that Assumption~\ref{c5:assumption-GRG} specifies not only the high-degree vertices but all the $w_i$'s. 
This is required in the proofs as one needs precise estimates of quantities like $\E[W_n \ind{W_n\geq K_n}]$.
However, Assumption~\ref{c5:assumption-GRG} also yields that the weight sequence satisfies a power law with exponent $\tau\in (2,3)$.

Let $\sCi(p)$ denote the $i$-th largest component of $\rGRG(\bw,p)$, and define $W_{\sss (i)}(p):= \sum_{k\in \sCi(p)} w_k$. 
We will consider the scaling limits of $(W_{\sss (i)}(p_c))_{i\geq 1}$ and $(\sCi(p_c))_{i\geq 1}$. 
To describe the limiting object, consider the graph $G_\infty(\lambda)$ on the vertex set $\Z_+$, where the vertices $i$ and $j$ are joined independently by Poisson$(\lambda_{ij})$ many edges with $\lambda_{ij}$ given by 
\begin{equation}\label{c5:defn:lambda-ij}
\lambda_{ij}:=\lambda^2\int_0^\infty \theta_i(x) \theta_j(x) \dif x, \quad \theta_i(x):= \frac{c_{\sss F}^2i^{-\alpha} x^{-\alpha}}{\mu+c_{\sss F}^2i^{-\alpha}x^{-\alpha}}.
\end{equation}
Let $W_{\sss (i)}^{\infty}(\lambda)$ denote the $i$-th largest element of the set $$\bigg\{\sum_{i\in \sC}\theta_i: \sC \text{ is a connected component of }G_{\infty}(\lambda)\bigg\}.$$ The following is the main result:

\begin{theorem}[Critical regime for $\GRG$] \label{c5:thm:main-crit}
There exists an absolute constant $\lambda_0$ such that for any $\lambda\in (0,\lambda_0)$,
under \textrm{Assumption~\ref{c5:assumption-GRG}}, as $n\to \infty$,
\begin{equation}
n^{-\alpha} (W_{\sss (i)}(p_c(\lambda)))_{i\geq 1} \dto (W_{\sss (i)}^{\infty}(\lambda))_{i\geq 1},
\end{equation} and 
\begin{equation}
(n^{\alpha}p_c)^{-1} (|\sC_{\sss (i)}(p_c(\lambda))|)_{i\geq 1} \dto (W_{\sss (i)}^{\infty}(\lambda))_{i\geq 1},
\end{equation} with respect to the $\ell^2_{\shortarrow}$ topology.
\end{theorem}
In the proofs we will also need to show that $W_{\sss (i)}^{\infty}(\lambda)<\infty$ almost surely for all $i\geq 1$.
In fact we will prove the following about the limiting object:
\begin{proposition}\label{c5:prop-limit-as-finite}
There exists an absolute constant $\lambda_0$ such that for any $\lambda\in (0,\lambda_0)$, $\bld{W}^{\infty}(\lambda) : = (W_{\sss (i)}^{\infty}(\lambda))_{i\geq 1}$ is in $\ell^2_{\shortarrow}$ almost surely.
\end{proposition}
\begin{remark}
\normalfont
In the proofs under the single-edge constraint, coming up with an analyzable exploration process for the clusters seems challenging.
The only tool we have is an estimate of the connection probabilities of hubs via an intermediate vertex, which allows us to estimate expectations of several moments of component sizes and total weights of those component. 
These are often referred to as susceptibility functions. 
The susceptibility functions allow us to ignore negligible contributions on the total weights of cluster using the first-moment method. 
Unfortunately, the first-moment method does not work for high values of $\lambda$. 
Also, we do not know how to show the finiteness of the limiting object in Proposition~\ref{c5:prop-limit-as-finite} for large $\lambda$.
This is the reason for assuming $\lambda \in (0,\lambda_0)$ in Theorem~\ref{c5:thm:main-crit}. 
The proof for general $\lambda$ is an open question. 
\end{remark}

\begin{remark}\normalfont Theorem~\ref{c5:thm:main-crit} also holds under different choices of $p_{ij}$'s than given by \eqref{c5:eq:p-ij-GRG-defn}.
For example, for the Chung-Lu Model ($p_{ij}: = \min \{w_iw_j/\ell_n,1\}$) and the Norros-Reittu model (with $p_{ij}: = 1-\e^{-w_iw_j/\ell_n}$), the statement of Theorem~\ref{c5:thm:main-crit} holds, with the scaling limit obtained in an identical manner with $\theta_i(x)$ is \eqref{c5:defn:lambda-ij} replaced respectively by $\min\{c_{\sss F}^2i^{-\alpha}x^{-\alpha}/\mu,1\}$, and $1-\e^{-c_{\sss F}^2 i^{-\alpha}x^{-\alpha}/\mu}$.
\end{remark}

\subsubsection{Erased configuration model}
The erased configuration model is obtained by erasing self-loops and multiple edges of $\CM$. We denote this random graph by $\ECM$, and we denote the graph obtained after bond percolation on $\ECM$ by $\mathrm{ECM}_n(\bld{d},p)$.
We will assume that $\bld{d}$ satisfies Assumption~\ref{c5:assumption-GRG}, where $F$ is a distribution function supported on the non-negative integers.
Thus, we take $d_i= (1-F)^{-1}(i/n)$, and we add an extra half-edge to vertex 1 if $\sum_{i\in [n]}d_i$ is odd.
The limiting object for $\ECM$ is identical to $\GRG$ after replacing $\theta_i(x)$ in \eqref{c5:defn:lambda-ij} by 
\begin{equation}\label{c5:defn:lambda-ij-ECM}
\theta_i(x):=1-\e^{-c_{\sss F}^2i^{-\alpha}x^{-\alpha}/\mu}.
\end{equation}

\begin{theorem}[Critical regime for $\ECM$]\label{c5:thm:main-ECM}
There exists an absolute constant $\lambda_0$ such that for any $\lambda\in (0,\lambda_0)$, under \textrm{Assumption~\ref{c5:assumption-GRG}},
the scaling limit results in \textrm{ Theorem~\ref{c5:thm:main-crit}} hold for $\mathrm{ECM}_n(\bld{d},p_c(\lambda))$ with limit objects described by \eqref{c5:defn:lambda-ij-ECM} above.
\end{theorem}
Notice that first performing percolation and then erasing self-loops and multiple-edges gives a different random graph than first erasing the self-loops and multiple-edges and then performing percolation.
%
For the $\tau>3$ case however, the order of these operations does not matter and leads to the same scaling limits as~\cite{DHLS15,DHLS16}.
However, the operations of deletion of self-loops and multiple-edges and performing percolation are not interchangeable  in the $\tau\in (2,3)$ regime, as evidenced by Theorems~\ref{c5:thm:main}~and~\ref{c5:thm:main-ECM}.
This can be understood intuitively.
In $\CM$, vertices $i$ and $j$ share $d_id_j/(\ell_n-1)$ edges in expectation.
Thus for hubs with $d_i=  O(n^{\alpha})$ and $d_j = O(n^{\alpha})$, in expectation $O(1)$ many edges survive after percolation in the critical window \eqref{c5:eq:crit-window-CM}.
On the other hand, whenever $p \to 0$, hubs are never connected directly under the single-edge constraint.
We will see in the proofs that the value $p_c$ in \eqref{c5:eq:scaling-window} is such that the hubs are connected to each other via intermediate vertices of degree $\Theta(n^{\rho})$ (see Figure~\ref{c5:fig:core-component}). 
\begin{figure}\centering
\includegraphics[scale=.05]{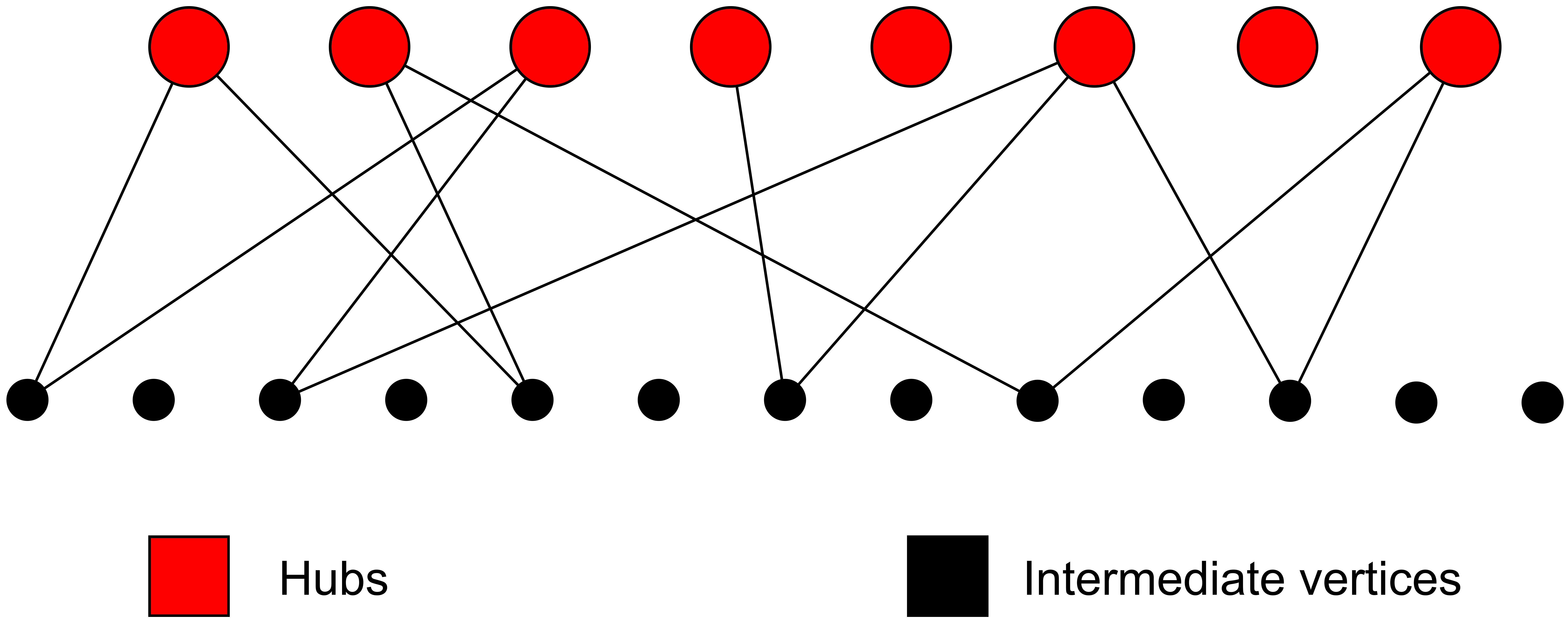}
\caption{Visualization of the connectivity between the hubs (vertices of degree $\Theta(n^{\alpha})$). 
The intermediate vertices have degree $\Theta(n^{\rho})$.}\label{c5:fig:core-component}
\end{figure}
This forms a core of the largest connected components, and the 1-neighborhood of this core spans the largest connected component asymptotically. 


\subsubsection{Behavior in the near-critical regimes} 
In this section, we state the results about the barely subcritical regime under the single-edge constraint. 
The result below holds for both percolation on $\GRG$ and $\ECM$:
\begin{theorem}\label{c5:thm:barely-subcrit-single-edge}
Suppose that \textrm{ Assumption~\ref{c5:assumption-GRG}} holds and $p_n \ll p_c(\lambda)$. 
Then, for any fixed $i\geq 1$, as $n\to\infty$, 
\begin{equation}
\frac{|\sC_{\sss (i)}(p_n)|}{n^{\alpha}p_n} \pto c_{\sss F}i^{-\alpha},\quad \text{and}\quad \frac{W_{\sss (i)}(p_n)}{n^{\alpha}} \pto c_{\sss F}i^{-\alpha}.
\end{equation}
\end{theorem}
\noindent 
Under the single-edge constraint, the exact asymptotics in the barely supercritical case is left to future work. 
In the proofs under the single-edge constraint, coming up with an analyzable exploration process for the clusters seems challenging.
The only tool we have is an estimate of the connection probabilities of hubs via an intermediate vertex, which allows us to estimate expectations of several moments of component sizes and total weights of those component. 
These are often referred to as susceptibility functions. 
The susceptibility functions allow us to ignore negligible contributions on the total weights of cluster using first-moment method. 
Unfortunately, the first-moment method does not work in the barely-supercritical regime, or even high values of $\lambda$ in Theorems~\ref{c5:thm:main-crit},~\ref{c5:thm:main-ECM}. 
This is the reason for assuming $\lambda \in (0,\lambda_0)$ in those theorems.



\subsection{Discussion}
\label{c5:sec:discussion}
\paragraph*{Assumption on the degrees.} Note that Assumption~\ref{c5:assumption1} is weaker than Assumption~\ref{c5:assumption-GRG}. 
Indeed, suppose that $d_i = (1-F)^{-1}(i/n)$, for some distribution function $F$ supported on non-negative integers, and $(1-F)(x) = C k^{-(\tau-1)},$ for  $k \leq x< k+1.$
We ignore the effect due to adding a dummy half-edge to vertex 1 if necessary to make $\sum_{i\in [n]}d_i$ even, since this does not change any asymptotics. 
Now, Assumption~\ref{c5:assumption1}~(i) is satisfied with $\theta_i = C i^{-\alpha}$. 
One can also verify Assumption~\ref{c5:assumption1}~(ii) using identical arguments as \cite[Lemma~6]{DHLS16}.
For this specific choice of $d_i$, \eqref{c5:eq:asymp-laplace} holds as well. 
To see this, write $t_n = tp_n^{1/(3-\tau)}$, and note that 
\begin{eq}
&1 - \E[\e^{-t_nD_n^*}] = \frac{1}{\ell_n} \sum_{k\in [n]} d_k \big(1-\e^{-t_nd_k}\big).
\end{eq}
Let us split the last sum in two parts on the set $\{k: d_k <1/(2t_n)\}$ and its complement, and denote them by $(I)$ and $(II)$ respectively.
We write $a_n \sim b_n$ to denote that $\lim_{n\to\infty}a_n/b_n = 1$.
Using the fact that $x-x^2/2 \leq 1-\e^{-x} \leq x$, it follows that for some constants $c_1,c_2>0$,
\begin{gather}
\frac{(I)}{t_n^{\tau-2}} \sim t_n^{3-\tau} \int_0^{1/(2t_n)} \frac{\dif x}{x^{\tau-2}} \sim c_0, \quad \text{and} \quad
\frac{(II)}{t_n^{\tau-2}} \sim t_n^{-(\tau-2)}\int_{1/(2t_n)}^\infty \frac{\dif x}{x^{\tau-1}} \sim c_1,
\end{gather}
which yields \eqref{c5:eq:asymp-laplace}.
%


\paragraph*{Critical windows: emergence of hub connectivity.}
The critical window changes due to the single-edge constraint, as noted in \eqref{c5:eq:crit-window-CM} and \eqref{c5:eq:scaling-window}.
However, there are some common features. 
Firstly, the component sizes are of the order $n^{\alpha} p_c$ in both the regimes. 
This is due to the fact that the main contribution to the component sizes comes from hubs and their finite neighborhood.
Secondly, in both cases, the critical window is the regime in which  hubs start getting connected. 
More precisely, if critical window is given by those values of $\pi$ such that for any fixed $i, j\geq 1$
\begin{eq}
\liminf_{n\to\infty} \PR(i, j\text{ are in the same component in the }\pi \text{-percolated graph} ) \in (0,1).
\end{eq}
For the configuration model, hubs are connected directly with strictly positive probability, while under the single-edge constraint, hubs are connected via intermediate vertices of degree $\Theta (n^{\rho})$.
Intuitively, in the barely subcritical regime, all the hubs are in different components.
Hubs start forming the critical components as the $p$ varies over the critical window since most paths between hubs are of length 2 and go via intermediate vertices of degree $\Theta (n^{\rho})$.
Finally in the barely super-critical regime the giant component is formed which contains all the hubs. 
This feature is also observed in the $\tau \in (3,4)$ case~\cite{BHL12}. 
However, the distinction between $\tau \in (3,4)$ and $\tau\in (2,3)$ is that for $\tau \in (3,4)$ the paths between the hubs have a length that grows as $n^{(\tau-3)/(\tau-1)}$.

\section{Configuration model: Proofs}\label{c5:sec:CM-proof}
In this section, we prove our results related to critical percolation on $\CM$. 
We start by proving some properties of the process \eqref{c5:defn::limiting::process} in Section~\ref{c5:sec:properties-exploration}. 
In Section~\ref{c5:sec:sandwich-CM}, we describe a way to approximate  percolation on a configuration model by a suitable configuration model.
In Section~\ref{c5:sec:CM-critical-window}, we analyze the latter graph by setting up an exploration process and obtaining its scaling limit.
Hence the proof of Theorem~\ref{c5:thm:main} is completed. 
In Section~\ref{c5:sec:near-critical-proofs}, we consider the near critical behavior and provide proofs of Theorems~\ref{c5:thm:barely-subcrit}~and~\ref{c5:thm:barely-supercrit}.
\subsection{Properties of the excursions of the limiting process}\label{c5:sec:properties-exploration}
In the following proposition, we summarize the properties of the limiting process $\biS$ that are required in our analysis.
\begin{proposition} \label{c5:prop:limit-properties}
\begin{itemize}
\item[\textrm{(P1)}] As $t\to\infty$, $\iS(t) \asto -\infty$. Thus, $\biS$ does not have an  excursion of infinite length almost surely.
\item[\textrm{(P2)}] For any $\delta >0$, $\biS$ has only finitely many excursions of length at least $\delta$ almost surely. 
\item[\textrm{(P3)}] Let $\mathcal{R}$ denote the set of excursion end-points of $\biS$. 
Then $\mathcal{R}$ does not have an isolated point.
\item[\textrm{(P4)}] For any $t>0$, $\PR(\iS(t)=\inf_{u\leq t}\iS(u))=0$.
\end{itemize}
\end{proposition}
The conditions in Proposition~\ref{c5:prop:limit-properties} form the bedrock of using \cite[Lemma 14]{DHLS16}, which will be crucial in the next section. 
An inquisitive reader might note that the conditions are related to \cite[Proposition 14]{AL98}. 
The proof of Proposition~\ref{c5:prop:limit-properties} requires the analysis of the martingale decomposition for the process $\biS$. 
Consider the sigma-field $\mathscr{F}_t = \sigma(\{\xi_i\leq s\}:s\leq t, i\geq 1)$, where for a collection of sets $\cA$, $\sigma (\cA)$ denotes the minimum sigma algebra containing all the sets in $\cA$. 
Then $(\mathscr{F}_t)_{t\geq 0}$ is a filtration. 
All the martingales in this section will be with respect to $(\mathscr{F}_t)_{t\geq 0}$, unless stated otherwise.
Without loss of generality we assume that $\mu=1$ in this section to simplify notation.
Below we summarize the martingale decomposition for $\biS$:
\begin{lemma}\label{c5:lem:mart-decomp-Sinfty} The process $\biS$ admits the Doob-Meyer decomposition $\iS(t) = M(t)+A(t)$ with the drift term $A(t)$, and the quadratic variation for the martingale term $\langle M\rangle (t)$  given by 
\begin{equation}
A(t) = \lambda\sum_{i=1}^\infty \theta_i^2\min\{\xi_i,t\}-2t, \qquad \langle M\rangle (t) = \lambda^2\sum_{i=1}^\infty \theta_i^3\min\{\xi_i,t\}.
\end{equation}
\end{lemma}
\begin{proof}
 Define $M_i(t) = \ind{\xi_i\leq t}-\theta_i\min\{\xi_i,t\}$. The proof follows if we show that $(M_i(t))_{t\geq 0}$ is a martingale with quadratic variation term given by 
\begin{equation}
\langle M_i\rangle(t) = \theta_i\min\{\xi_i,t\}.
\end{equation}
Denote $\mathcal{I}_i(t) = \ind{\xi_i\leq t}$. Let $(\mathcal{N}(t))_{t\geq 0}$ denote a unit jump Poisson process. 
Then $\mathcal{I}_i(t)$ can be written in terms of the random time change  of $(\mathcal{N}(t))_{t\geq 0}$ \cite{EK86,PTW07} as follows
\begin{equation}
\mathcal{I}_i(t) = \mathcal{N}\bigg(\theta_i\int_0^t (1-\mathcal{I}_i(s))\dif s\bigg).
\end{equation}
\noindent Further, $\int_0^t (1-\mathcal{I}_i(s))\dif s = \min\{\xi_i,t\}$ which completes the proof.
\end{proof} 
The rest of the section is devoted to proving the properties of $\biS$ stated in Proposition~\ref{c5:prop:limit-properties}. 
We give the proofs of different conditions separately below:

\begin{proof}[Proof of Proposition~\ref{c5:prop:limit-properties}~\textrm{(P1)}]
We use the martingale decomposition of $\biS$ from Lemma~\ref{c5:lem:mart-decomp-Sinfty}. 
Fix $K\geq 1$ such that $\lambda \sum_{i>K} \theta_i^2 <1$.
Such a choice of $K$ is always possible as $\bld{\theta}\in \ell^2_{\shortarrow}$.
Further define the stopping time $T:= \inf\{t: \xi_i\leq t,\ \forall i\in [K]\}$, and observe that $T<\infty$ almost surely.
Note that $\min\{\xi_i,t\} \leq t$ and thus, 
\begin{eq}
\frac{1}{t}\lambda\sum_{i>K}\theta_i^2 \min\{\xi_i,t\} \leq 1, \quad \text{almost surely.}
\end{eq}
Therefore, for any $t>T$,
\begin{eq}
A(t) = \lambda \sum_{i\in [K]} \theta_i^2 \xi_i  + \lambda\sum_{i>K}\theta_i^2 \min\{\xi_i,t\} -2t \leq \lambda \sum_{i\in [K]} \theta_i^2 \xi_i -t, \quad \text{almost surely.}
\end{eq}
We conclude that, for any $r\in (0,1)$,
$t^{-r}A(t) \asto -\infty.$
For the martingale part we will use the exponential concentration inequality \cite[Inequality 1, Page 899]{SW86}, which is stated below:
\begin{lemma}\label{c5:lem:concentration-inequality}
If $M$ is any continuous time local martingale such that $M(0) = 0$, and $\sup_{t\in [0,\infty)} |M(t) - M(t-)| \leq c$, almost surely, then for any $t>0$, $a>0$ and $b>0$, 
\begin{eq}
\PR\Big(\sup_{s\in [0,t]} M(s) > a, \text{ and } \langle M \rangle (t) \leq b  \Big) \leq \exp \bigg(-\frac{a^2}{2b} \psi \Big( \frac{a c}{b}\Big) \bigg), 
\end{eq}where $\psi(x) = ( (1+x)\log(1+x) - x ) / x^2$. 
\end{lemma}
\noindent In particular, $\psi(x) \geq 1/(2 (1+x/3))$ (see \cite[Page 27]{JLR00}).
Note that $ \langle M\rangle(t)\leq \lambda^2 t \sum_{i=1}^\infty \theta_i^3.$
We apply Lemma~\ref{c5:lem:concentration-inequality} with $a = \varepsilon t^{r}$, $b = \lambda^2 t \sum_{i=1}^\infty \theta_i^3$, and $c= \theta_1$. 
Now, $\psi (ac/b) \geq C/(1+t^{r-1})$, and thus for any $\varepsilon >0$, and $r\in (1/2,1)$
\begin{eq}
\PR\Big(\sup_{s\in [0,t]} |M(s)| > \varepsilon t^{r}  \Big) \leq 2 \exp (-C t^{2r-1}), 
\end{eq}for some constant $C>0$, where the bound on the absolute value of $M$ follows from the fact that $-M$ is also a martingale, so Lemma~\ref{c5:lem:concentration-inequality} applies to $-M$ as well. 
Now an application of the Borel-Cantelli lemma proves that $ t^{-r}| M(t) |\asto 0,$ for any $r\in (1/2,1)$.
This fact, together with the asymptotics of the drift term, completes the proof.
\end{proof}

\begin{proof}[Proof of Proposition~\ref{c5:prop:limit-properties}~\textrm{(P2)}]
 Let $t_k = (k-1)\delta/2$ and define the event
 \begin{equation}
  \rC_k^\delta:=\bigg\{\sup_{t\in (t_{k-1},t_k]}\iS(t_{k+1}) - \iS(t)>0\bigg\}.
 \end{equation}
Suppose that there is an excursion $(l,r)$ with $r-l > \delta$ and $l\in (t_{k-1},t_k]$ for some $k$. 
Since $r> t_{k+1}$, $\iS(t_{k+1})> \iS(l) \geq \inf_{t\in (t_{k-1},t_k]} \iS(t)$, and therefore $\rC_k^\delta$ must occur.
Therefore, if $\biS$ has infinitely many excursions of length at least $\delta$, then $\rC_k^\delta$ must occur infinitely often.
Using the Borel-Cantelli lemma, the proof follows if we can show that
\begin{equation}
\sum_{k=1}^\infty\PR(\rC_k^\delta) <\infty.
\end{equation}
As before, fix $K\geq 1$ such that $\lambda \sum_{i>K} \theta_i^2 <1$, and let $T:= \inf\{t: \xi_i\leq t,\ \forall i\in [K]\}$.
Notice that for each $K\geq 1$,
\begin{eq}
\sum_{k=1}^\infty \prob{T>t_{k-1}} &= \sum_{k=1}^\infty \prob{\exists i\in [K]: \xi_i>t_{k-1}}\leq \sum_{k=1}^\infty K \e^{-\theta_K(k-1)\delta/2}<\infty,
\end{eq} and therefore it is enough to show that 
\begin{equation}
\sum_{k=1}^\infty\PR(\rC_k^\delta\cap \{T\leq t_{k-1}\}) <\infty.
\end{equation}
Now,
\begin{eq}
&\sup_{t\in [t_{k-1},t_k]}\iS(t_{k+1})-\iS(t) \\
&\leq M(t_{k+1})+\sup_{t\in [t_{k-1},t_k]} - M(t) + \sup_{t\in [t_{k-1},t_k]} A(t_{k+1}) - A(t)\\
& \hspace{1cm}\leq M(t_{k+1}) - M(t_{k-1})+ \sup_{t\in [t_{k-1},t_k]} M(t_{k-1})- M(t) \\
&\hspace{2cm}+ \sup_{t\in [t_{k-1},t_k]}\bigg[\lambda\sum_{i=1}^\infty \theta_i^2 (\min\{\xi_i,t_{k+1}\}-\min\{\xi_i,t\}) -(t_{k+1}-t)\bigg]-\frac{\delta}{2}\\
& \hspace{1cm}\leq 2\sup_{t\in [t_{k-1},t_{k+1}]}  |M(t)-M(t_{k-1})| -\frac{\delta}{2}\\
& \hspace{2cm}+ \sup_{t\in [t_{k-1},t_k]}\bigg[\lambda\sum_{i=1}^\infty \theta_i^2 (\min\{\xi_i,t_{k+1}\}-\min\{\xi_i,t\}) -(t_{k+1}-t)\bigg].
\end{eq}
The third term is negative on the event $\{\tau\leq t_{k-1}\}$. 
Thus we only need to estimate the probability
\begin{equation}
\PR\bigg(\sup_{t\in [t_{k-1},t_{k+1}]}  |M(t)-M(t_{k-1})| >\frac{\delta}{4}\bigg).
\end{equation}
Note that $M(t)-M(t_{k-1})$ is a martingale with respect to the filtration $(\mathscr{F}_t)_{t\geq t_{k-1}}$ with quadratic variation given by 
\begin{equation}
\lambda^2\sum_{i=1}^\infty \theta_i^3\big(\min\{\xi_i,t\}-\min\{\xi_i,t_{k-1}\}\big).
\end{equation}
Further, $\E[\min\{\xi_i,t\}] = \theta_i^{-1}(1-\e^{-\theta_i t})$. 
Therefore, Doob's martingale inequality \cite[Theorem~1.9.1.3]{LS89} implies
\begin{eq}
&\sum_{k=1}^\infty\PR\bigg(\sup_{t\in [t_{k-1},t_{k+1}]}  |M(t)-M(t_{k-1})| >\frac{\delta}{4}\bigg)\\
&\leq \sum_{k=1}^\infty\frac{16}{\delta^2}\sum_{i=1}^\infty \theta_i^2 (\e^{-\theta_i t_{k-1}}-\e^{-\theta_i t_{k+1}}).
\end{eq}
By interchanging the sums, the last term is finite and the proof now follows.
\end{proof}
\begin{proof}[Proof of Proposition~\ref{c5:prop:limit-properties}~\textrm{(P3)}]
Define the process 
\begin{equation}
L(t) = \lambda\sum_{i=1}^\infty \theta_i \mathcal{N}_i(t) -2t,
\end{equation} 
where $(\mathcal{N}_i(t))_{t\geq 0}$ is a rate $\theta_i$ Poisson process, independently over $i$. 
We assume that $\biS$ and $\bld{L}$ are coupled by taking $\cI_i(s) = 
\mathbf{1}\{\mathcal{N}_i(s)\geq 1\}$, so that $\iS(t)\leq L(t)$. 
Using \cite[Chapter VII.1, Theorem 1]{Ber96} and the fact that $\sum_i\theta_i = \infty$, 
 \begin{equation}\label{c5:defn:perfect-levy-1}
  \inf \{t>0: L(t)<0\}=0, \quad \text{almost surely.}
 \end{equation}
 Moreover, for any stopping time $\mathcal{T}>0$, $(S_\infty^\lambda(\mathcal{T}+t)-S_\infty^\lambda(\mathcal{T}))_{t\geq 0}$, conditioned on the sigma-field $\sigma (S_\infty^\lambda(s):s\leq \mathcal{T})$, is distributed as a process  defined in \eqref{c5:defn::limiting::process} for some random $\boldsymbol{\theta}$. 
 Now we can take $\mathcal{T}$ to be an excursion endpoint and the proof follows.
\end{proof}

\begin{proof}[Proof of Proposition~\ref{c5:prop:limit-properties}~\textrm{(P4)}]
We leverage the proof techniques in \cite[Proposition 14 (b)]{AL98}. 
Define the process 
\begin{equation}
Q_m(t) = \lambda \sum_{i=m+1}^\infty \theta_i \1_{A_i} M_i(t) -2t,
\end{equation}
where $(A_i)_{i\geq 1}$ is a sequence of independent events with $\PR(A_i) = (1-2\theta_it_0)^+$, and $(M_i)_{i\geq 1}$ are independent Poisson processes with rates $\theta_i\e^{-\theta_it_0}$. 
Now, $\bld{Q}_m$ is a L\'evy process, and thus we can apply \eqref{c5:defn:perfect-levy-1} together with \cite[Chapter VI.1, Proposition 3]{Ber96} (see also the remark below the statement of the cited proposition)
to conclude that, for any $t_0>0$ and $m\geq 1$,
\begin{equation}\label{c5:fact1}
\PR\Big(Q_m(t_0) = \inf_{t\in [ 0, t_0]}Q_m(t)\Big)=0.
\end{equation}
Now, if one can couple $(Q_{2m}(t))_{t\in [ 0, t_0]}$ and $(S_{\infty}(t))_{t\in [ 0, t_0]}$ in such a way that
\begin{equation}\label{c5:fact2}
 \lim_{m\to\infty} \prob{Q_{2m}(t_0)-Q_{2m}(t)\leq S_{\infty}(t_0)-S_{\infty}(t),\ \forall t\in [ 0, t_0]} = 0, 
\end{equation}then \eqref{c5:fact1} and \eqref{c5:fact2} together complete the proof. 
To see \eqref{c5:fact2}, write $(\xi_{i,j},j\in J_i)$ for the set of points of $\bld{M}_i$ in  $[0,t_0]$ if $A_i$ occurs, but to be the empty set if $A_i$ does not occur.
Call a coupling \emph{successful} if $(\xi_{i,j}:i>2m, j\in J_i)$ and $(\xi_i)_{i\geq 1}$ are coupled in such a way that $\xi_{i,j} = \xi_{h(i,j)}$ for some random variable $h(i,j)\leq i$, and the values $(h(i,j):i>2m,j\in J_i)$ are distinct.
Aldous and Limic~\cite{AL98} showed the existence of a coupling such that the probability of the coupling is successful tends to 0 as $m\to\infty$. 
Although their proof is under a different setting with $\bld{\theta}\in \ell^3_{\shortarrow}\setminus \ell^2_{\shortarrow}$, the proof of the coupling holds under the assumption that $\bld{\theta}\in \ell^2_{\shortarrow}\setminus \ell^1_{\shortarrow}$. 
Moreover, \eqref{c5:fact2} holds under a successful coupling and thus the proof is complete. 
%
\end{proof}

\subsection{Sandwiching the percolated configuration model}\label{c5:sec:sandwich-CM}
A key step in all our proofs is to \emph{approximate} $\rCM_n(\bld{d},p_n)$ by a configuration model, which is given in Proposition~\ref{c5:prop:coupling-lemma} below. 
This idea has also appeared in the context of the finite third moment \cite{DHLS15} and the infinite third moment case \cite{DHLS16}.
We emphasize that Proposition~\ref{c5:prop:coupling-lemma} holds for percolation on the configuration model without any specific assumption on the degree distribution, as long as $\ell_n p_n \gg \log(n)$.
We start by describing the approximating configuration model below:
\begin{algo}\label{c5:algo-alt-cons-perc}
\begin{itemize}\normalfont
 \item[(S0)] Keep each half-edge with probability $p_n$. If the total number of retained half-edges is odd, attach a \emph{dummy} half-edge to vertex 1.
 \item[(S1)] Perform a uniform perfect matching among the retained half-edges, i.e., within the retained half-edges, pair unpaired half-edges sequentially with a uniformly chosen unpaired half-edge until all half-edges are paired. 
 The paired half-edges create edges in the graph, and we call the resulting graph $\mathcal{G}_n(p_n)$.
\end{itemize}
\end{algo}
\noindent The following proposition formally states that $\cG_n(p_n)$ approximates the percolated graph $\rCM_n(\bld{d},p_n)$:
\begin{proposition} \label{c5:prop:coupling-lemma}
Let $p_n$ be such that $\ell_np_n \gg \log(n)$. 
Then there exists a sequence $(\varepsilon_n)_{n\geq 1}$ with $\varepsilon_n\to 0 $, and a coupling such that, with high probability,
\begin{equation}
\mathcal{G}_n(p_n(1-\varepsilon_n))\subset\mathrm{CM}_n(\bld{d},p_n)\subset\mathcal{G}_n(p_n(1+\varepsilon_n)).
\end{equation}
\end{proposition}
\begin{proof} 
The proof relies on an exact construction of $\rCM_n(\bld{d},p_n)$ by Fountoulakis~\cite{F07} which goes as follows:
\begin{algo}\label{c5:algo-fountoulakis}\normalfont
\begin{itemize}
\item[(S0)] Perform a binomial trial $X\sim \mathrm{Bin}(\ell_n/2,p_n)$ and choose $2X$ half-edges uniformly at random from the set of all half-edges.
\item[(S1)] Perform a perfect matching of these $2X$ chosen half-edges. The resulting graph is distributed as $\rCM_n(\bld{d},p_n)$.
\end{itemize}
\end{algo}
Notice the similarity between Algorithm~\ref{c5:algo-alt-cons-perc}~(S1) and Algorithm~\ref{c5:algo-fountoulakis}~(S1). 
In Algorithm~\ref{c5:algo-alt-cons-perc}~(S0), given the number of retained  half-edges, the choice of the half-edges can be performed sequentially uniformly at random without replacement. 
Thus, given the number of half-edges in the two algorithms, we can couple the choice of the half-edges, and their pairing (the restriction of a uniform matching to a subset remains).
Let $\mathcal{H}_1$, $\mathcal{H}_2^-$ and $\mathcal{H}_2^+$ respectively denote the number of half-edges in $\mathrm{CM}_n(\bld{d},p_n)$, $\mathcal{G}_n(p_n(1-\varepsilon_n))$ and $\mathcal{G}_n(p_n(1+\varepsilon_n))$.
From the above discussion, the proof is complete if we can show that as $n\to\infty$,
\begin{equation}\label{c5:eq:coup-reduc}
 \PR\big(\mathcal{H}_2^-\leq \mathcal{H}_1\leq \mathcal{H}_2^+\big) \to  1.
\end{equation} 
Notice that $\mathcal{H}_1=2X$, where $X\sim\mathrm{Bin}(\ell_n/2,p_n)$, and $\mathcal{H}_2^+\sim\mathrm{Bin} (\ell_n,p_n(1+\varepsilon_n))$.
Using standard concentration inequalities \cite[Corollary 2.3]{JLR00}, it follows that
\begin{subequations}
\begin{equation}
\mathcal{H}_1 = \ell_np_n +\oP(\sqrt{\ell_np_n\log(n)}), 
\end{equation}and
\begin{equation}
\mathcal{H}_2^+ = \ell_np_n + \ell_np_n \varepsilon_n +\oP(\sqrt{\ell_np_n\log(n)}).
\end{equation}
\end{subequations}
Now, if we choose $\varepsilon_n$ such that $\varepsilon_n\gg (\log(n)/(\ell_np_n))^{1/2}$ and $\varepsilon_n\to 0$, then, with high probability, $\mathcal{H}_1\leq \mathcal{H}_2^+$.
Similarly we can conclude that $\mathcal{H}_2^-\leq \mathcal{H}_1$ with high probability.
The proof is now complete. 
\end{proof}

We conclude this subsection by stating some properties of the degree sequence of the graph $\cG_n(p_n)$ that will be crucial in the analysis below. 
Let $\Mtilde{\bld{d}}=(\tilde{d}_1,\dots,\tilde{d}_n)$ be the degree sequence induced by Algorithm~\ref{c5:algo-alt-cons-perc}~(S1).
Then the following result holds for $\Mtilde{\bld{d}}$:
\begin{lemma}\label{c5:lem:perc-degrees}
For each fixed $i\geq 1$, $\tilde{d}_i = d_ip_n(1+\oP(1))$ and  $\tilde{\ell}_n = \ell_np_n (1+\oP(1))$. Moreover, for $p_n\ll p_c$, $\sum_{i\in [n]}\tilde{d}_i^2= \tilde{\ell}_n(1+\oP(1))$, whereas for $p_n= p_c(\lambda)$ the following holds:
For any $\varepsilon >0$,
\begin{equation}\label{c5:eq:perc-tail-sum-square}
\lim_{K\to\infty}\limsup_{n\to\infty} \PR\bigg(\sum_{i>K}\tilde{d}_i^2>\varepsilon \tilde{\ell}_n\bigg) = 0.
\end{equation}
\end{lemma}  
\begin{proof}
Note that $\Mtilde{d}_i \sim \mathrm{Bin}(d_i,p_n)$, independently for $i\in [n]$. 
For each fixed $i\geq 1$, $d_i p_n = \Theta(n^\rho)$, which tends to infinity. 
Thus the first fact follows using \cite[Theorem 2.1]{JLR00}. 
Since, $\tell_n\sim \mathrm{Bin}(\ell_n,p_n)$, the second fact also follows using the same bound.
To prove \eqref{c5:eq:perc-tail-sum-square}, we first assume that $p_n = p_c(\lambda)$ given by \eqref{c5:eq:crit-window-CM}. 
Then, for any $\varepsilon>0$, the probability in \eqref{c5:eq:perc-tail-sum-square} is at most
\begin{eq}
&\PR\bigg(\sum_{i>K}\tilde{d}_i^2>\varepsilon \tilde{\ell}_n, \frac{\ell_np_n}{2}\leq \tilde{\ell}_n\leq 2\ell_np_n \bigg) +o(1)\\
&\hspace{1cm}\leq \PR\bigg(\sum_{i>K}\tilde{d}_i^2> \frac{\varepsilon\ell_np_n}{2}\bigg)\leq \frac{4p_n\sum_{i>K}d_i^2}{\ell_n},
\end{eq}where the last step follows from Markov's inequality. 
The proof now follows using Assumption~\ref{c5:assumption1} and $p_n = O(n^{2\alpha-1})$. 
The case for $p_n\ll p_c$ follows similarly.
\end{proof}

\subsection{Analysis in the critical window}\label{c5:sec:CM-critical-window}
\subsubsection{Convergence of the exploration process}\label{c5:sec:expl-process}
  Let $\Mtilde{\bld{d}}=(\tilde{d}_1,\dots,\tilde{d}_n)$ be the degree sequence induced by Algorithm~\ref{c5:algo-alt-cons-perc}~(S1) with $p_n=p_c(\lambda)$, and consider $\mathcal{G}_n(p_c(\lambda))$. 
  Note that $\mathcal{G}_n(p_c(\lambda))$ has the same distribution as $\mathrm{CM}_{n}(\Mtilde{\boldsymbol{d}})$. 
  We start by describing how the connected components in the graph can be explored while generating the random graph simultaneously:
\begin{algo}[Exploring the graph]\label{c5:algo-expl}\normalfont  The algorithm carries along vertices that can be alive, active, exploring and killed and half-edges that can be alive, active or killed. 
We sequentially explore the graph as follows:
\begin{itemize}
\item[(S0)] At stage $i=0$, all the vertices and the half-edges are \emph{alive} but none of them are \emph{active}. Also, there are no \emph{exploring} vertices. 
\item[(S1)]  At each stage $i$, if there is no active half-edge at stage $i$, choose a vertex $v$ proportional to its degree among the alive (not yet killed) vertices and declare all its half-edges to be \emph{active} and declare $v$ to be \emph{exploring}. If there is an active vertex but no exploring vertex, then declare the \emph{smallest} vertex to be exploring.
\item[(S2)] At each stage $i$, take an active half-edge $e$ of an exploring vertex $v$ and pair it uniformly to another alive half-edge $f$. Kill $e,f$. If $f$ is incident to a vertex $v'$ that has not been discovered before, then declare all the half-edges incident to $v'$ active, except $f$ (if any). 
If $\mathrm{degree}(v')=1$ (i.e. the only half-edge incident to $v'$ is $f$) then kill $v'$. Otherwise, declare $v'$ to be active and larger than all other vertices that are alive. After killing $e$, if $v$ does not have another active half-edge, then kill $v$ also.

\item[(S3)] Repeat from (S1) at stage $i+1$ if not all half-edges are already killed.
\end{itemize}
\end{algo}
Algorithm~\ref{c5:algo-expl} gives a \emph{breadth-first} exploration of the connected components of $\mathrm{CM}_n(\Mtilde{\boldsymbol{d}})$. Define the exploration process by
   \begin{equation}\label{c5:defn:exploration:process}
    S_n(0)=0,\quad
     S_n(l)=S_n(l-1)+\tilde{d}_{(l)}J_l-2,
    \end{equation} where $J_l$ is the indicator that a new vertex is discovered at time $l$ and $\tilde{d}_{(l)}$ is the degree of the new vertex chosen at time $l$ when $J_l=1$.  Suppose $\mathscr{C}_{k}$ is the $k^{th}$ connected component explored by the above exploration process and define
$\tau_{k}=\inf \big\{ i:S_{n}(i)=-2k \big\}.$
Then  $\mathscr{C}_{k}$ is discovered between the times $\tau_{k-1}+1$ and $\tau_k$, and  $\tau_{k}-\tau_{k-1}-1$ gives the total number of edges in $\mathscr{C}_k$.
 Call a vertex \emph{discovered} if it is either active or killed. Let $\mathscr{V}_l$ denote the set of vertices discovered up to time $l$ and $\mathcal{I}_i^n(l):=\ind{i\in\mathscr{V}_l}$. Note that 
   \begin{equation}\label{c5:eq:expl-process-crit}
    S_n(l)= \sum_{i\in [n]} \tilde{d}_i \mathcal{I}_i^n(l)-2l.
   \end{equation} 
   Define the re-scaled version $\bar{\mathbf{S}}_n$ of $\mathbf{S}_n$ by $\bar{S}_n(t)= n^{-\rho}S_n(\lfloor tn^{\rho} \rfloor)$. Then,
   \begin{equation} \label{c5:eqn::scaled_process}
    \bar{S}_n(t)= n^{-\rho} \sum_{i\in [n]}\tilde{d}_i \mathcal{I}_i^n(tn^{\rho})-2t + o(1).
   \end{equation}Note the similarity between the expressions in \eqref{c5:defn::limiting::process} and \eqref{c5:eqn::scaled_process}. We will prove the following:
   \begin{theorem} \label{c5:thm::convegence::exploration_process} Consider the process $\bar{\mathbf{S}}_n:= (\bar{S}_n(t))_{t\geq 0}$ defined in \eqref{c5:eqn::scaled_process} and recall the definition of  $\bar{\mathbf{S}}_\infty $ from \eqref{c5:defn::limiting::process}. Then, under \textrm{ Assumption~\ref{c5:assumption1}}, as $n\to\infty,$
 \begin{equation}
  \bar{\mathbf{S}}_n \dto \bar{\mathbf{S}}_\infty
 \end{equation} with respect to the Skorohod $J_1$ topology.
\end{theorem}
\begin{proof}
We denote $\tell_n(u) = \sum_{i\in [n]}\td_i - 2un^{\rho}$. 
Since $\tell_n = \Theta_{\sss \PR} (n^{2\rho})$, $\tell_n(u) = \tell_n(1+\oP(1))$ uniformly over $u\leq t$.
 Let $\tilde{\PR}(\cdot)$ (respectively $\tilde{\E}[\cdot]$) denote the conditional probability (respectively expectation) conditional on $(\td_i)_{i\in [n]}$.
 Note that, for any $t\geq 0$, uniformly over $l\leq tn^{\rho}$
\begin{equation}\label{c5:prob-ind-lb}
   \probc{\mathcal{I}_i^n(l)=0} \geq \bigg(1-\frac{\tilde{d}_i}{\tilde{\ell}_n (t)} \bigg)^l, \quad \text{and} \quad \exptc{\mathcal{I}_i^n(l)} \leq \frac{l\tilde{d}_i}{\tilde{\ell}_n(t)}.
  \end{equation}
Now, note that 
\begin{eq}
n^{-\rho}\tilde{\E}\bigg[\sum_{i>K}\tilde{d}_i\mathcal{I}_i^n(tn^{\rho})\bigg]\leq t \frac{\sum_{i>K}\tilde{d}_i^2}{\tilde{\ell}_n(t)}.
\end{eq}
Using \eqref{c5:eq:perc-tail-sum-square}, it is now enough to deduce the scaling limit for 
\begin{equation}
\bar{S}_n^K(t) = n^{-\rho}\sum_{i=1}^K \tilde{d}_i \mathcal{I}_i^n(tn^{\rho})- 2t
\end{equation} and then taking $K\to\infty$. 
The next lemma gives the scaling limit for $\bar{\mathbf{S}}_n^K$ and completes the proof of Theorem~\ref{c5:thm::convegence::exploration_process}.
\end{proof}
 \begin{lemma} \label{c5:lem::convergence_indicators}
   Fix any $K\geq 1$, and $\mathcal{I}_i(s):=\ind{\xi_i\leq s }$ where $\xi_i\sim \mathrm{Exp}(\theta_i/\mu)$ independently for $i\in [K]$. 
   Under \textrm{ Assumption~\ref{c5:assumption1}}, as $n\to\infty$,
   \begin{equation}
    \left( \mathcal{I}_i^n(tn^\rho) \right)_{i\in[K],t\geq 0} \dto \left( \mathcal{I}_i(t) \right)_{i\in[K],t\geq 0}.
   \end{equation}
 \end{lemma}
\begin{proof} By noting that $(\mathcal{I}_i^n(tn^\rho))_{t\geq 0}$ are indicator processes, it is enough to show that 
\begin{equation}
 \probc{\mathcal{I}_i^n(t_in^{\rho})=0,\ \forall i\in [K]} \to \probc{\mathcal{I}_i(t_i)=0,\ \forall i\in [K]} = \exp \Big( -\mu^{-1}\sum_{i=1}^{K} \theta_it_i\Big).
\end{equation} for any $t_1,\dots,t_K\in \mathbb{R}$. Now, 
\begin{equation} \label{c5:lem::eqn::expression1}
 \probc{\mathcal{I}_i^n(m_i)=0,\ \forall i\in [K]}=\prod_{l=1}^{\infty}\Big(1-\sum_{i\leq K:l\leq m_i}\frac{\td_i}{\tell_n-\Theta(l)} \Big).
\end{equation}Taking logarithms on both sides of \eqref{c5:lem::eqn::expression1} and using the fact that $l\leq \max m_i=\Theta(n^{\rho})$ we get 
\begin{equation}\label{c5:lem::eqn::ex1}
 \begin{split}
  \probc{\mathcal{I}_i^n(m_i)=0\, \forall i\in [K]}&= \exp\Big( - \sum_{l=1}^{\infty}\sum_{i\leq K:l\leq m_i} \frac{\td_i}{\tell_n}+o(1) \Big)\\
  &= \exp\Big( -\sum_{i\in [K]} \frac{\td_im_i}{\tell_n} +o(1) \Big).
 \end{split}
\end{equation} Putting $m_i=t_in^{\rho}$, Assumption~\ref{c5:assumption1}~\eqref{c5:assumption1-1},~\eqref{c5:assumption1-2} gives
\begin{equation} \label{c5:lem::eqn::expression2}
 \frac{m_i\td_i}{\tell_n}= \frac{\theta_it_i}{\mu} (1+\oP(1)).
\end{equation}
Hence \eqref{c5:lem::eqn::expression2}, and \eqref{c5:lem::eqn::ex1} complete the proof of Lemma \ref{c5:lem::convergence_indicators}.
\end{proof}

\subsubsection{Large components are explored early}
Now, we prove two key results that allow us to deduce the convergence of the component sizes. 
Firstly, we show that the rescaled vector of component sizes is tight in~$\ell^2_{\shortarrow}$ (see Proposition~\ref{c5:prop:l2-tight}). 
This result is then used to show that the largest components of~$\mathcal{G}_n(p_c(\lambda))$ are explored before time $\Theta(n^{\rho})$.
Let $\sCi$ denote the $i$-th largest component for~$\mathcal{G}_n(p_c(\lambda))$.
\begin{proposition}\label{c5:prop:l2-tight}
Under \textrm{ Assumption~\ref{c5:assumption1}}, for any $\varepsilon>0$, 
\begin{equation}
\lim_{K\to\infty}\limsup_{n\to\infty}\PR\bigg(\sum_{i>K}|\sCi|^2>\varepsilon n^{2\rho}\bigg) = 0.
\end{equation}
\end{proposition}
\noindent Denote $D_i = \sum_{k\in\sCi}\tilde{d}_k$. 
It is enough to show that
\begin{equation}\label{c5:eq:tight-D-i}
\lim_{K\to\infty}\limsup_{n\to\infty}\PR\bigg(\sum_{i>K}D_i^2>\varepsilon n^{2\rho}\bigg) = 0.
\end{equation}
In the above, we have used our convention that the component size of an isolated vertex is zero.
For a vertex $v$, let $\mathscr{C}(v)$ denote the component containing vertex $v$ in $\mathcal{G}_n(p_n)$ and $D(v) = \sum_{k\in \mathscr{C}(v)} \tilde{d}_k$.
Let $\mathcal{G}^{\sss K}$ be the random graph obtained by removing all edges attached to vertices $1,\dots,K$ and let $\boldsymbol{d}'$ be the obtained degree sequence. 
Further, let $\sC^{\sss K}(v)$ and $\sCi^{\sss K}$ denote the connected component containing $v$ and the $i$-th largest component respectively,  and $D^{\sss K}(v) = \sum_{k\in \sC^{\sss K}(v)}\td_k$, $D^{\sss K}_i = \sum_{k\in \sCi^{\sss K}}\td_k$.
Suppose $V_n^*$ is a vertex of $\mathcal{G}^{\sss K}$ chosen according to the size-biased distribution with sizes being $(\tilde{d}_k/\tell_n)_{k\geq 1}$, independently of the graph.  
Denote the criticality parameter of $\mathcal{G}^{\sss K}$ by $\nu_n^{\sss K}$.
\begin{lemma} \label{c5:lem::tail_sum_squares} Suppose that  \textrm{ Assumption~\ref{c5:assumption1}} holds.
Then, for fixed $K\geq 1$ (sufficiently large), with high probability
\begin{equation}\label{c5:expt-cvn-K-removed}
\exptc{D^{\sss K}(V_n^*)}\leq \frac{\sum_{i>K}\tilde{d}_i^2}{\tilde{\ell}_n-2\sum_{i=1}^K\tilde{d}_i}\times\frac{1}{1-\nu_n^{\sss K}}.
\end{equation}
\end{lemma}
\begin{proof}
We make use of path counting techniques \cite{J09b,BDHS17}.
Note that the criticality parameter of the graph $\mathcal{G}_n(p_c(\lambda))$ is $\tilde{\nu}_n = \lambda(1+\oP(1))$, by Lemma~\ref{c5:lem:perc-degrees}.
Now, conditional on the set of removed half-edges and $\Mtilde{\bld{d}}$, $\mathcal{G}^{\sss K}$ is still a configuration model with some degree sequence $\boldsymbol{d}'$ with $d_i'\leq \td_i$ for all $i\in [n]\setminus [K]$ and $d_i'=0$ for $i\in [K]$.
Further, the criticality parameter of $\mathcal{G}^{\sss K}$ satisfies 
 \begin{equation}\label{c5:eqn:nu-K}
  \begin{split}
   \nu^{\sss K}_n&= \frac{\sum_{i\in [n]} d_i'(d'_i-1)}{\sum_{i\in [n]} d_i'}\leq \frac{\sum_{i>K}\tilde{d}_i(\tilde{d}_i-1)}{\tilde{\ell}_n-2\sum_{i=1}^K\tilde{d}_i} \\
   &= \tilde{\nu}_n \frac{\sum_{i>K}\tilde{d}_i(\tilde{d}_i-1)}{\sum_{i\in [n]}\tilde{d}_i(\tilde{d}_i-1)}(1+\oP(1)) = \lambda\frac{\sum_{i>K}\tilde{d}_i(\tilde{d}_i-1)}{\sum_{i\in [n]}\tilde{d}_i(\tilde{d}_i-1)}(1+\oP(1))\\
   &= \lambda\frac{\sum_{i>K}d_i(d_i-1)}{\sum_{i\in [n]}d_i(d_i-1)}(1+\oP(1)).
  \end{split}
 \end{equation}
 Now, by Assumption~\ref{c5:assumption1} and Lemma~\ref{c5:lem:perc-degrees}, it is possible to choose $K_0$ large such that, 
 \begin{equation} \label{c5:eq:fact-K-large-nu}
 \text{for all }K\geq K_0,\text{ with high probability }\nu_n^{\sss K}<1,
 \end{equation} and then we can apply similar arguments as in \cite[Section 7]{BDHS17} for calculating weight-based susceptibility functions with weights being the degrees. 
The term $\sum_{i>K}\td_i^2/\sum_{i>K}\td_i$ arises in \eqref{c5:expt-cvn-K-removed} due to the fact that 
\begin{equation}
\tilde{\E}[\text{degree of }V_n^*] = \frac{\sum_{i\in [n]}d_i'^2}{\sum_{i\in [n]}d_i'}\leq 
\frac{\sum_{i>K}\td_i^2}{\tilde{\ell}_n-2\sum_{i=1}^K\tilde{d}_i}.
\end{equation}
Thus the proof of Lemma~\ref{c5:lem::tail_sum_squares} follows.
\end{proof}

\begin{proof}[Proof of Proposition~\ref{c5:prop:l2-tight}]
Denote the sum of squares of the $D$-values excluding the components containing vertices $1,2,\dots, K$ by $\mathscr{S}_K$.
 Note that 
 \begin{equation}
 \sum_{i>K}D_i^2\leq \mathscr{S}_K\leq \sum_{i\geq 1} (D_i^{\sss K})^2. 
 \end{equation}  
Now, using Lemma~\ref{c5:lem::tail_sum_squares} and \eqref{c5:eq:fact-K-large-nu}, it follows that
\begin{eq}
&\tilde{\PR}\bigg(\sum_{i\geq 1} (D_i^{\sss K})^2 >\varepsilon n^{2\rho}\bigg)\leq \frac{1}{\varepsilon n^{2\rho}}\tilde{\E}\bigg[\sum_{i\geq 1} (D_i^{\sss K})^2 \bigg] \\
&\hspace{1cm}= \frac{1}{\varepsilon n^{2\rho}}\exptc{D^{\sss K}(V_n^*)}\leq \OP(1) n^{-2\rho}\sum_{i>K}\tilde{d}_i^2.
\end{eq}
Thus, \eqref{c5:eq:tight-D-i} follows from Lemma~\ref{c5:lem:perc-degrees}.
\end{proof}
The next proposition shows that the large components are explored before time $\Theta(n^{\rho})$ by Algorithm~\ref{c5:algo-expl}. 
The proof follows using similar arguments as \cite[Lemma 13]{DHLS16} and we skip it here.
Let $\sC_{\max}^{\sss \geq T}$ denote the size of the largest component that is started exploring by Algorithm~\ref{c5:algo-expl} after time $Tn^{\rho}$. 
\begin{proposition}\label{c5:prop:large-comp-expl-early}
 Under \textrm{ Assumption~\ref{c5:assumption1}}, for any $\varepsilon>0$, 
\begin{equation}
\lim_{T\to\infty}\limsup_{n\to\infty}\PR\big(|\sC_{\max}^{\sss \geq T}|>\varepsilon n^{\rho}\big) = 0.
\end{equation}
\end{proposition}

\subsubsection{Convergence of the component sizes and the surplus edges}
We start by first showing the asymptotics of the component sizes:
\begin{lemma}\label{c5:lem:component-sizes}
Under \textrm{ Assumption~\ref{c5:assumption1}}, as $n\to\infty$, 
\begin{equation}\label{c5:eq:l2-conv}
(n^{-\rho}|\mathscr{C}_{\sss (i)}|)_{i\geq 1} \dto (\gamma_i(\lambda))_{i\geq 1},
\end{equation}with respect to the $\ell^2_{\shortarrow}$ topology.
\end{lemma}
\begin{proof}
Recall \cite[Lemma 14]{DHLS16} and notice that the process $\biS$ satisfies all the \emph{nice} properties stated therein by Proposition~\ref{c5:prop:limit-properties} (see \cite[Lemma 15]{DHLS16} for a similar application in a different context). 
This observation, together with Proposition~\ref{c5:prop:large-comp-expl-early} yields the finite-dimensional convergence in \eqref{c5:eq:l2-conv}.
Finally the proof is completed using Proposition~\ref{c5:prop:l2-tight}.
\end{proof}
\begin{lemma} \label{c5:lem:surp:poisson-conv} Let $N_n^\lambda(k)$ be the number of surplus edges discovered up to time $k$ and $\bar{N}^\lambda_n(u) = N_n^\lambda(\lfloor un^\rho \rfloor)$. Then, as $n\to\infty$,
 \begin{equation}
 (\bar{\mathbf{S}}_n,\bar{\mathbf{N}}_n^\lambda)\dto (\mathbf{S}_{\infty}^\lambda,\mathbf{N}^\lambda),
 \end{equation} where $\mathbf{N}^\lambda$ is defined in \eqref{c5:defn::counting-process}.
 \end{lemma}
 \begin{proof}
  We write $
N_n^{\lambda}(l)=\sum_{i=2}^l \xi_i$,
where $\xi_i=\ind{\mathscr{V}_i=\mathscr{V}_{i-1}}$. Let $A_i$ denote the number of active half-edges after stage $i$ while implementing Algorithm~\ref{c5:algo-expl}. Note that 
\begin{equation}
 \prob{\xi_i=1\vert \mathscr{F}_{i-1}}=\frac{A_{i-1}-1}{\tilde{\ell}_n-2i-1}= \frac{A_{i-1}}{\tilde{\ell}_n}(1+O(i/n))+O(n^{-1}),
\end{equation} uniformly for $i\leq Tn^{\rho}$ for any $T>0$. 
Therefore, the instantaneous rate of change of the re-scaled process $\bar{\mathbf{N}}^{\lambda}$ at time $t$, conditional on the past, is 
\begin{equation}\label{c5:eqn:intensity}
 n^{\rho}\frac{A_{\floor{tn^\rho}}}{n^{2\rho}\frac{\mu^2}{\sum_{i\geq 1}\theta_i^2}}\left( 1+o(1)\right) +o(1)= \frac{\sum_{i\geq 1}\theta_i^2}{\mu^2}\refl{\bar{S}_n(t)}\left( 1+o(1)\right) +o(1).
\end{equation}
Theorem~\ref{c5:thm::convegence::exploration_process} yields that $\mathrm{refl}(\bar{\mathbf{S}}_n)\dto \mathrm{refl}(\biS)$.  
Then, by the Skorohod representation theorem, we can assume that $\mathrm{refl}(\bar{\mathbf{S}}_n)\to \mathrm{refl}(\biS)$ almost surely on some probability space. Observe that $(\int_{0}^t \refl{\iS(u)}du)_{t\geq 0}$ has continuous sample paths. Therefore, the conditions of \cite[Corollary 1, Page 388]{LS89} are satisfied and the proof is complete.
 \end{proof} 

\begin{lemma}\label{c5:lem-surp-u-0} For any $\varepsilon >0$,
 \begin{equation}
 \lim_{\delta\to 0}\limsup_{n\to\infty} \PR\bigg(\sum_{i: |\mathscr{C}_{(i)}|\leq \delta n^{\rho} }|\mathscr{C}_{\sss(i)}|\times \surp{\mathscr{C}_{\sss(i)}}> \varepsilon b_n\bigg)=0.
 \end{equation}
\end{lemma}
Again, the proof of Lemma~\ref{c5:lem-surp-u-0} can be carried out in an identical manner as 
\cite[Proposition 19]{DHLS16} and therefore is skipped here. 
The only crucial thing to observe here is that $V_n$ has to be replaced by $V_n^*$ for an analogue of \cite[Lemma 20]{DHLS16} and one has to consider the cases $\lambda<1$, and $\lambda>1$ separately instead of $\lambda<0$ and $\lambda>0$ in \cite{DHLS16}.

\begin{proof}[Proof of Theorem~\ref{c5:thm:main}] Let $\mathbf{Z}_n'(\lambda)$ denote the vector $(n^{-\rho}|\sCi|,\SP(\sCi))_{i\geq 1}$, ordered as an element in~$\Unot$. Then, Lemmas~\ref{c5:lem:component-sizes},~\ref{c5:lem:surp:poisson-conv},~and~\ref{c5:lem-surp-u-0} together imply that
\begin{equation}
\mathbf{Z}_n'(\lambda) \dto \mathbf{Z}(\lambda).
\end{equation}
Finally the proof is complete using Proposition~\ref{c5:prop:coupling-lemma}.
\end{proof}
\subsubsection{Analysis of the diameter}
In this section, we deduce the asymptotics of the diameter of the  components of $\cG_n(p_c(\lambda))$, and hence complete the proof of Theorem~\ref{c5:thm:diameter-large-comp}.
\begin{proof}[Proof of Theorem~\ref{c5:thm:diameter-large-comp}] Firstly, we can leverage Janson's path counting technique again (see \cite[Lemma 5.1]{J09b}) to show that, given a vertex $v\in [n],$ the expected number of paths of length $l$ from $v$ (conditional on $\Mtilde{\bld{d}}$) is at most $\tilde{\nu}_n^l$.
Recall the notations $D(v) = \sum_{k\in \mathscr{C}(v)}\tilde{d}_k$, $D_{\sss (i)} = \sum_{k\in \sCi}\tilde{d}_k$, and $V_n^*=$ a vertex chosen according to the size-biased distribution $(\tilde{d}_i/\tilde{\ell}_n)_{i\geq 1}$.
For any fixed $\delta>0$, define
\begin{equation}
X_n(K) = \sum_{v\in [n]} \ind{D(v)> \delta n^{\rho}, |\mathscr{C}(v)|> \delta n^{\rho}, \diam(\mathscr{C}(v))>K}.
\end{equation}
Observe that, for $\lambda<1$, 
\begin{eq}
&\probc{\exists i\geq 1: D_{\sss (i)}> \delta n^{\rho}, |\sCi|>\delta n^{\rho}, \diam(\sCi)>K} \\
&\leq \sum_{i\geq 1} \probc{D_{\sss (i)}> \delta n^{\rho}, |\sCi|>\delta n^{\rho}, \diam(\sCi)>K}\\
&=\tilde{\ell}_n \sum_{v\in [n]} \frac{\tilde{d}_v}{\tilde{\ell}_n}\tilde{\E}\bigg[\frac{1}{D(v)}\ind{D(v)> \delta n^{\rho}, |\mathscr{C}(v)|> \delta n^{\rho}, \diam(\mathscr{C}(v))>K}\bigg]\\
&\leq  \frac{\tilde{\ell}_n}{\delta n^{\rho}} \probc{D(V_n^*)> \delta n^{\rho}, |\mathscr{C}(V_n^*)|> \delta n^{\rho}, \diam(\mathscr{C}(V_n^*))>K}\\
&= \frac{\tilde{\ell}_n}{\delta n^{\rho}} \probc{X_n(K)\geq \delta n^{\rho}} \leq \frac{\tilde{\ell}_n}{\delta^2 n^{2\rho}}\sum_{l>K}\tilde{\nu}_n^l =  \frac{C\tilde{\ell}_n}{\delta^2 n^{2\rho}}\frac{\lambda^K}{1-\lambda} \pto 0,
\end{eq}if we first take $\lim_{n\to\infty}$ and then $\lim_{K\to\infty}$.
The proof can be generalized naturally to the case $\lambda >1$.
In that case, we delete $R$ high-degree vertices to obtain a new graph $\mathcal{G}^{\sss >R}$, for which the above result holds (see the proof of Lemma~\ref{c5:lem::tail_sum_squares}). 
However, after putting back the $R$ deleted vertices, the diameter of $\mathcal{G}^{\sss >R}$ can change by a factor of at most $R$.
 This implies the tightness of the diameter for the largest connected components of $\mathcal{G}_n(p_c(\lambda))$ for $\lambda >1$.
 Finally the proof of Theorem~\ref{c5:thm:diameter-large-comp} follows by invoking Proposition~\ref{c5:prop:coupling-lemma} again.
\end{proof}

\subsection{Near-critical behavior} \label{c5:sec:near-critical-proofs}
Finally we consider the near-critical behavior for $\mathrm{CM}_n(\bld{d},p)$ in this section. 
The analysis for the barely subcritical and supercritical regimes are given separately below.
\subsubsection{Barely-subcritical regime}
In this section, we analyze the barely-subcritical regime ($p_n\ll p_c$) for percolation and complete the proof of Theorem~\ref{c5:thm:barely-subcrit}. 
Recall the exploration process from Algorithm~\ref{c5:algo-expl} on the graph $\cG_n(p_n)$, starting with vertex $j$.
Let $\sC(j,p_n)$ denote the connected component in~$\cG_n(p_n)$ containing vertex $j$.
 We will use the same notation for the quantities defined in Section~\ref{c5:sec:expl-process}, but the reader should keep in mind that we now deal with different $p_n$ values. 
 We avoid augmenting $p_n$ in the notation for the sake of simplicity.
 Let $\mathcal{N}_1$ denote the 1-neighborhood of $j$ in $\cG_n(p_n)$ and define $\mathcal{D}(j)$ to be the number of half-edges at vertices incident to $ \mathcal{N}_1$ which are not paired with $j$.
Define the exploration process $\mathbf{S}_n^j$ similar to \eqref{c5:eq:expl-process-crit} but starting with $\mathcal{D}(j)$ half-edges as given below: 
   \begin{equation}
   \begin{split}
    S_n^j(0) = \mathcal{D}(j), \quad S_n^j(l)= \mathcal{D}(j)+\sum_{i\notin \mathcal{N}_1} \tilde{d}_i \mathcal{I}_i^n(l)-2l.
    \end{split}
   \end{equation}         
Thus the exploration process starts from $\tilde{d}_j$ now. 
Consider the re-scaled process $\bar{\mathbf{S}}^j_n$ defined as $\bar{S}^j_n(t)= (n^{\alpha}p_n)^{-1}S_n^j(\floor{tn^{\alpha}p_n })$. Then, 
   \begin{equation} \label{c5:eqn::scaled_process_j}
    \bar{S}_n^j(t)= (n^{\alpha}p_n)^{-1}\mathcal{D}(j)+(n^{\alpha}p_n)^{-1} \sum_{i\notin \mathcal{N}_1}\tilde{d}_i \mathcal{I}_i^n(tn^{\alpha}p_n) - 2 t +\oP(1).
   \end{equation}
Recall that $\tilde{\E}$ is the conditional expectation conditionally on $(\tilde{d}_i)_{i\in [n]}$.
 Now, since the vertices are explored in a size-biased manner with the sizes being $(\tilde{d}_i/\tilde{\ell}_n)_{i\in [n]}$, for any $t\geq 0$,
\begin{eq} \label{c5:expt-barely-sub-diff}
\tilde{\E}\bigg[\frac{1}{n^{\alpha}p_n}\sum_{i\notin \mathcal{N}_1} \tilde{d_i}\cI_i^n\big(\lfloor tn^{\alpha}p_n\rfloor\big)\bigg] \leq \frac{tn^{\alpha}p_n}{n^{\alpha}p_n \tilde{\ell}_n}\sum_{i\in [n]}\tilde{d}_i^2 = \oP(1),
\end{eq}where the last step follows from Lemma~\ref{c5:lem:perc-degrees}. 
Moreover, $\mathcal{N}_1 =\Mtilde{d}_j - 2L_j$, where $L_j$ denotes  the number of self-loops associated to vertex $j$ in $\cG_n(p_n)$. 
Since $\tilde{\E}[L_j] \leq \Mtilde{d}_j^2/\tilde{\ell}_n = \oP(1)$, it follows that $\tilde{\E}[\mathcal{N}_1] = \Mtilde{d}_j(1+\oP(1))$.
Further,  $\tilde{\E}[\mathcal{D}(j)] \leq \tilde{\nu}_n \tilde{\E}[|\mathcal{N}_1|] = \oP(n^{\alpha}p_n)$. 
Consequently, $\bar{\mathbf{S}}_n^j\xrightarrow{\sss \PR} 0$, and therefore, the number of edges in $\sC(j,p_n)\setminus \mathcal{N}_1 = \oP(1)$.
Since the number of vertices in $\sC(j,p_n)\setminus \mathcal{N}_1$ is at most the number of edges +1, this yields that $|\sC(j,p_n)\setminus \mathcal{N}_1| = \oP(n^{\alpha}p_n)$. 
Also, one can use Lemma~\ref{c5:lem:perc-degrees} to show that $\mathrm{Var}(\mathcal{N}_1 \mid \Mtilde{\bld{d}}) = \oP(n^{2\alpha}p_n^2)$, which yields $\mathcal{N}_1 = \Mtilde{d}_j(1+\oP(1))$.
Thus,
\begin{equation}
(n^{\alpha}p_n)^{-1}|\sC(j,p_n)| \pto \theta_j.
\end{equation}
To conclude Theorem~\ref{c5:thm:barely-subcrit}, it remains to prove that for each fixed $j\geq 1$, 
\begin{equation}\label{c5:eq:whp-Cj-barely-sub}
|\mathscr{C}(j,p_n)| = |\sC_{\sss (j)}(p_n)|, \text{ with high probability.} 
\end{equation}
For that, we show that the rescaled vector of ordered component sizes converges in $\ell^2_{\shortarrow}$. 
It is enough to show that for any $\varepsilon >0$ 
\begin{equation}
\lim_{K\to\infty}\limsup_{n\to\infty}\PR\bigg(\sum_{i>K}|\sCi(p_n)|^2>\varepsilon n^{2\alpha}p_n^2\bigg) = 0.
\end{equation}
This can be concluded using identical arguments as Proposition~\ref{c5:prop:l2-tight} above.
Now, \eqref{c5:eq:whp-Cj-barely-sub} follows using \cite[Lemma 7.6]{BDHS17}. 
The proof of Theorem~\ref{c5:thm:barely-subcrit} is now complete.

\subsubsection{Barely-supercritical regime}
In this section, we provide the proof of Theorem~\ref{c5:thm:barely-supercrit} by leveraging techniques from~\cite{HJL16,JL09}. 
Using Proposition~\ref{c5:prop:coupling-lemma}, it is enough to prove Theorem~\ref{c5:thm:barely-supercrit} for the graph $\cG_n(p_n)$ generated by Algorithm~\ref{c5:algo-alt-cons-perc}. 
Let $\Mtilde{\bld{d}}$ denote the degree sequence obtained after performing Algorithm~\ref{c5:algo-alt-cons-perc}~(S1). 
Thus, $\cG_n(p_n)$ is distributed as $\rCM_n(\Mtilde{\bld{d}})$.
We will verify Assumptions (B1)--(B8) from \cite{HJL16} on the graph $\cG_n(p_n)$,
which allows us to conclude Theorem~\ref{c5:thm:barely-supercrit} from \cite[Theorem 5.3]{HJL16}.
Consider the following exploration process on $\Gp$ from \cite[Section 5.1]{HJL16}:

\begin{algo}\label{c5:algo:expl-janson}
\normalfont
\begin{itemize}
\item[(S0)] Associate an independent $\mathrm{Exponential}(1)$ clock $\xi_e$ to each half-edge $e$. 
Any half-edge can be in one of the \emph{states} among sleeping, active, and dead. 
Initially at time $0$, all the half-edges are sleeping. 
Whenever the set of active half-edges is empty, select a sleeping half-edge $e$ uniformly at random among all sleeping half-edges and declare it to be active. 
If $e$ is incident to $v$, declare all the other half-edges of $v$ to be active as well.
The process stops when there is no sleeping half-edge left; the
remaining sleeping vertices are all isolated and we have explored all other components.

\item[(S1)]  Pick an active half-edge (which one does not matter) and kill it, i.e., change its status to dead.
\item[(S2)] Wait until the next half-edge dies (spontaneously). This half-edge is paired to the one killed in the previous step (S1) to form an edge of the graph. 
If the vertex it belongs to is sleeping, then we declare this vertex awake and all of its other half-edges active.
Repeat from (S1) if there is any active half-edge; otherwise from (S0).
\end{itemize}
\end{algo}
Denote the number of living half-edges upto time $t$ by $L_n(t)$.
Let $V_{n,k}$ and $\tV_{n,k}(t)$ respectively denote the number of sleeping vertices of degree $k$ such that all the $k$ associated exponential clocks ring after time $t$.
Define
\begin{equation}
\tV_n(t) = \sum_{k=0}^\infty \tV_{n,k}(t), \quad \tS_n(t) = \sum_{k=0}^\infty k\tV_{n,k}(t),\quad \tA_n(t) = L_n(t)-\tS_n(t).
\end{equation}
We show that Assumptions (B1)--(B8) from \cite{HJL16} holds with 
\begin{eq}\label{c5:eq:choice-parameters-barely-supercrit}
\zeta = (\kappa/2)^{\frac{1}{3-\tau}}, \quad \gamma_n = \alpha_n &= p_n^{\frac{\tau-2}{3-\tau}}, \quad \psi(t) =  2t - \kappa t^{\tau-2}, \\
\hat{g}(t) = \mu \kappa t^{\tau-2},& \quad \hat{h}(t) = \mu t.
\end{eq}
The $\zeta$ in our notation corresponds to $\tau$ in \cite[Theorem 5.3]{HJL16}, but we have avoided that since~$\tau$ denotes the power-law exponent in our case.
\begin{remark} \normalfont
Notice that the rate at which an edge is created in the system is $\ell_np_n$, instead of rate $n$, as in the setting of \cite{HJL16}. 
Thus, in time $\alpha_n$, the total number of explored half-edges in $\ell_n p_n \alpha_n $ which is of the order $n p_n^{1/(3-\tau)}$.
For this reason, the largest component size is of the order $n p_n^{1/(3-\tau)}$ in Theorem~\ref{c5:thm:barely-supercrit}.
\end{remark}
Conditions (B1)--(B4) and (B8) in \cite{HJL16} are straightforward, and are left to the reader.
To verify Conditions (B5)--(B7), we first obtain below the asymptotics of the mean-curve and then show that the processes $\Mtilde{\bld{S}}_n$, $\tilde{\bld{V}}_n$, $\Mtilde{\bld{A}}_n$ remain uniformly close to their expected curves. 
These are summarized in the following two propositions:
\begin{proposition}\label{c5:prop:expt-supcrit} For any fixed $t\geq 0$, as $n\to\infty$, 
\begin{gather}
\sup_{u\leq t}\bigg|\frac{1}{n\alpha_np_n}\big(\E[\tS_n(0)-\E[\tS_n(\alpha_n t)]\big) - \hat{g}(t)\bigg|\to 0, \label{c5:eq:expt-S}\\
\sup_{u\leq t}\bigg|\frac{1}{n\alpha_np_n}\big(\E[\tV_n(0)]-\E[\tV_n(\alpha_n t)]\big) - \hat{h}(t)\bigg|\to 0, \label{c5:eq:expt-V}\\
\sup_{u\leq t}\bigg|\frac{1}{n\alpha_n p_n}\E[\tA_n(\alpha_n t)] - \psi(t)\bigg| \to 0.\label{c5:eq:expt-A}
\end{gather}
\end{proposition}
\begin{proposition}\label{c5:prop:supcrit-expt-concen}
For any fixed $t\geq 0$, as $n\to\infty$, all the terms \linebreak
$\sup_{u\leq t}|\tS_n(\alpha_n t)-\E[\tS_n(\alpha_n t)] |$, $\sup_{u\leq t}|\tV_n(\alpha_n t)-\E[\tV_n(\alpha_n t)]|$, and \linebreak $\sup_{u\leq t}|\tA_n(\alpha_n t) - \E[\tA_n(\alpha_n t)]| $ are $ o(n\alpha_np_n)$.
\end{proposition}
To prove Propositions~\ref{c5:prop:expt-supcrit}~and~\ref{c5:prop:supcrit-expt-concen}, we make crucial use of the following lemma:
\begin{lemma} \label{c5:lem:laplace-transform-estimate} For any $t>0$, as $n\to\infty$, 
\begin{gather}
\E\bigg[\sum_{i\in [n]} \tilde{d_i} \e^{-t\alpha_n\tilde{d}_i}\bigg] = (1+o(1)) p_n \sum_{i\in [n]} d_i \e^{-t\alpha_np_n d_i}, \\ \E\bigg[\sum_{i\in [n]} \e^{-t\alpha_n\tilde{d}_i}\bigg] = (1+o(1)) \sum_{i\in [n]}  \e^{-t\alpha_np_n d_i}. 
\end{gather}
\end{lemma}
\begin{proof}
Note that if $X\sim \mathrm{Bin}(m,p)$, then 
\begin{eq}
\E\big[X\e^{-sX}\big] = mp\e^{-s}(1-p+p\e^{-s})^{m-1}.
\end{eq}
Putting $m=d_i$, $p=p_n$, and $s = t\alpha_n$, it follows that
\begin{eq}
\E\big[\tilde{d}_i \e^{-t\alpha_n\tilde{d}_i}\big] &= d_ip_n \e^{-t\alpha_n} \Big(1-p_n\big(1-\e^{-t\alpha_n}\big)\Big)^{d_i -1} \\
&= (1+o(1)) d_ip_n (1-p_nt\alpha_n)^{d_i}\\
& = (1+o(1)) d_ip_n \e^{-t\alpha_np_n d_i}.
\end{eq}
Thus the proof follows. 
\end{proof}

\begin{proof}[Proof of Proposition~\ref{c5:prop:expt-supcrit}]
Note that by Lemma~\ref{c5:lem:laplace-transform-estimate}
\begin{eq}
\E\big[\tilde{S}_n( t)\big] &= \E\bigg[\sum_{i\in [n]} \Mtilde{d}_i\e^{-t\alpha_n \Mtilde{d}_i} \bigg]= (1+o(1))\ell_n p_n \E\big[\e^{-t\alpha_n p_n D_n^*}\big], \\ 
\E\big[\tilde{V}_n( t)\big] &= \E\bigg[\sum_{i\in [n]} \e^{-t\alpha_n \Mtilde{d}_i} \bigg] =  (1+o(1))\ell_n  \E\big[\e^{-t\alpha_n p_n D_n}\big],
\end{eq}where $D_n^*$ has a size-biased distribution with the sizes being $(d_i/\ell_n)_{i\in [n]}$, and $D_n$ is the degree of a vertex chosen uniformly at random from $[n]$. 
By the convergence of $\E[D_n]$ in Assumption~\ref{c5:assumption1}, 
\begin{eq}
\E\big[1 - \e^{-t\alpha_n p_n D_n}\big] = (1+o(1)) t\alpha_np_n \E[D_n],
\end{eq}
by noting that $1-\e^{-x} =x(1+o(1))$ as $x\to 0$.
Further, by using \eqref{c5:eq:asymp-laplace},
\begin{eq}\label{c5:laplace-exp-alpha-estimate}
\E\big[1- \e^{-t\alpha_n p_n D_n^*}\big] = \kappa \alpha_n (t^{\tau-2}+o(1)).
\end{eq}
Thus, \eqref{c5:eq:expt-S} and \eqref{c5:eq:expt-V} follow. 
Moreover, $L_n(t)$ is a pure death process, where $L(0) = \sum_{i\in [n]} \Mtilde{d}_i$,  and the jumps occur at rate $L_n(t)$, and at each jump $L_n(t)$ decreases by $2$. 
Therefore, $\E[L_n(t)] = \E[L(0)]\e^{-2t}$ and consequently, 
\begin{eq}
\E[\tA(\alpha_n t)] &= \ell_np_n\big(\e^{-2\alpha_nt} - \E\big[\e^{-tp_n\alpha_nD_n^*}\big]\big)+o(\ell_n\alpha_np_n)\\
& = \ell_n p_n \alpha_n(2t - \kappa t^{\tau-2})++o(\ell_n\alpha_np_n).
\end{eq}
Thus the proof follows.
\end{proof}

\begin{proof}[Proof of Proposition~\ref{c5:prop:supcrit-expt-concen}]
Let us consider $\tS_n$ only and the other inequalities follow using identical arguments. 
We will use precise bounds in \cite[Lemma 5.13]{HJL16}. 
In fact, using the fact that $1-\e^{-x} \geq (1\wedge x)/3$, it follows that 
\begin{eq}\label{c5:bound-expt-sup-second-moment}
\E\Big[\sup_{u\leq t\alpha_n} |\tS_n(u) - \E[\tS_n(u)]| \Big] \leq C \E \bigg[ \sum_{i\in [n]} \Mtilde{d}_i^2 \big(1-\e^{-t\alpha_n \Mtilde{d}_i} \big)\bigg]. 
\end{eq}
\noindent Now, using standard concentration inequalities for tails of binomial distributions \cite[Theorem 2.1]{JLR00}, for any $i\in [n]$,
\begin{eq}
\PR(\Mtilde{d}_i > 2d_1 p_n ) \leq C \e^{-C d_1 p_n} = C\e^{-C n^{\rho} \lambda_n},
\end{eq}and therefore $\max_{i\in [n]}\Mtilde{d}_i \leq 2d_1 p_n$, a.s.~$\PR_{p}$.
Now, using \eqref{c5:eq:choice-parameters-barely-supercrit}, the bounds\eqref{c5:laplace-exp-alpha-estimate} and \eqref{c5:bound-expt-sup-second-moment} yield
\begin{eq}
&\frac{1}{(\ell_np_n\alpha_n)^2}\E\Big[\sup_{u\leq t\alpha_n} |\tS_n(u) - \E[\tS_n(u)]| \Big] \\
&\hspace{.5cm}\leq \frac{C 2d_1p_n \ell_np_n\alpha_n}{(\ell_np_n\alpha_n)^2} = \frac{C}{\lambda_n^{(\tau-2)/(3-\tau)}} \to 0,
\end{eq}since $\lambda_n\to 0$, as $n\to\infty$. Thus the proof follows.
\end{proof}
\begin{proof}[Proof of Theorem~\ref{c5:thm:barely-supercrit}] 
The proof follows by applying \cite[Theorem 5.3]{HJL16}.
Proposition~\ref{c5:prop:expt-supcrit} verifies conditions (B5)--(B7) in \cite{HJL16}, and the rest of the conditions are straightforward to verify.
\end{proof}

\section{Generalized random graphs: Proofs}
\label{c5:sec:proof-GRG}
In this section, we prove our results related to critical percolation on $\GRG$. 
In Section~\ref{c5:sec:key-ingredients}, we set up the fundamental technical ingredients required for the proof. 
In Section~\ref{c5:sec:negligible-contributions}, we use the  first-moment method to identify the primary contributions on the total weight of the components. 
The connectivity structure between the hubs is described in detail in Section~\ref{c5:sec:hub-structure}, which allows us to deduce the component sizes of components containing hubs in Section~\ref{c5:sec:component-with-hubs}.
Finally, Theorem~\ref{c5:thm:main-crit} is proved in Section~\ref{c5:sec:l2-tightness}.

\subsection{Key Ingredients} \label{c5:sec:key-ingredients}
In this section, we provide the two key ingredients that will play a pivotal role in the proof of Theorem~\ref{c5:thm:main-crit}. 
The first one, stated in Lemma~\ref{c5:lem:order-estimates}, provides the estimates for different moments of the~$\bw$.
Next, in Lemma~\ref{c5:lem:2-nbd-prob}, we estimate the probability of two vertices being connected via another vertex. 
The later result forms the conceptual bedrock of our analysis (see Section~\ref{c5:sec:hub-structure}). 
For example, one can set up path-counting techniques in Corollary~\ref{c5:lem:geometric-decay} using Lemma~\ref{c5:lem:2-nbd-prob}.
We use a generic notation $C$ to denote a positive constant. 
Also we write $a_n \sim b_n$ to mean that $a_n/b_n\to 1$, as $n\to\infty$.
\begin{lemma}[Moment estimates]\label{c5:lem:order-estimates}
Under \textrm{ Assumption~\ref{c5:assumption-GRG}}, there exists a constant $\cf>0$ (depending only on $F$) such that, for all $i\in [n]$, 
\begin{equation}\label{c5:eq:w-i-order}
w_i = \bigg(\frac{c_{\sss F} n}{i}\bigg)^{\alpha}.
\end{equation}
For any $a>0, b\leq \alpha$, 
\begin{eq}\label{c5:eq:square-elln}
\#\{r: w_r >a\ell_n/w_j\} &\sim a^{-(\tau-1)} n\bigg(\frac{w_j}{\ell_n}\bigg)^{\tau-1} , \\
\quad
\sum_{k: w_{k}\leq a\ell_n/w_{j}}w_k^2 &\sim  C a^{3-\tau}n \bigg(\frac{\ell_n}{w_j}\bigg)^{3-\tau}, \\
\sum_{w_k>a\ell_n/w_j}w_k &\sim Ca^{-(\tau-2)}n^{3-\tau} w_j^{\tau-2}, \\ \sum_{k:w_k\leq an^{b}} w_k^{\tau-2} &\sim C n^\rho (an^{b})^\alpha,
\end{eq}
where $C>0$ is considered as a generic notation for a constant.
\end{lemma}
\begin{proof}
\eqref{c5:eq:w-i-order} follows directly from Assumption~\ref{c5:assumption-GRG}. 
Next, note that 
\begin{eq}\label{c5:eq:larger-than-elln-order}
\bigg(\frac{n}{r}\bigg)^{\alpha} \leq a \frac{\ell_n}{w_j} \iff r\geq a^{-(\tau-1)} n\bigg(\frac{w_j}{\ell_n}\bigg)^{\tau-1}.
\end{eq}
Thus, \eqref{c5:eq:square-elln} follows by noting that
\begin{eq}
\sum_{r: w_{r}\leq a\ell_n/w_{j}}w_r^2 = n^{2\alpha} \sum_{r\geq a^{-(\tau-1)} n(w_j/\ell_n)^{\tau-1}} r^{-2\alpha} \sim C a^{3-\tau}n \bigg(\frac{\ell_n}{w_j}\bigg)^{3-\tau},
\end{eq}and 
\begin{eq}
\sum_{w_k>a\ell_n/w_j}w_k &\sim Cn^\alpha \sum_{k\leq Ca^{-(\tau-1)}(w_j/\ell_n)^{\tau-1}} k^{-\alpha} \\
&\sim Cn^{\alpha} \bigg(a^{-(\tau-1)}n\Big(\frac{w_j}{\ell_n}\Big)^{\tau-1}\bigg)^{\rho}\\
& \sim Ca^{-(\tau-2)}n^{3-\tau} w_j^{\tau-2}.
\end{eq}
The last expression is also similar.
\end{proof}
\begin{lemma}[Two-hop connection probabilities]\label{c5:lem:2-nbd-prob}
There exists an absolute constant $C>0$ such that for all $n\geq 1$, 
\begin{equation}\label{c5:2hop-path-prob-pref-attachment-n}
p_{ij}(2):= p_c^2 \sum_{v\in [n]} \frac{w_iw_v^2w_j}{(\ell_n+w_iw_v)(\ell_n+w_jw_v)} \leq \frac{C\lambda^2}{(i\wedge j)^{1-\alpha} (i\vee j)^{\alpha}}, \quad \forall \ i,j\in [n].
\end{equation}
\end{lemma}
\begin{proof}
Without loss of generality, we assume that $w_i\geq w_j$, i.e., $i\leq j$. 
Let us split the sum in three parts with $\{v: w_iw_v\leq \ell_n\}$, $\{v:w_jw_v\leq \ell_n<w_iw_v\}$, and $\{v:w_jw_v> \ell_n\}$, and denote them by $(I)$, $(II)$ and $(III)$ respectively. 
Note that using Lemma~\ref{c5:lem:order-estimates},
\begin{eq}
(I)&\leq  \frac{p_c^2w_iw_j}{\ell_n} \sum_{v: w_v\leq \ell_n/w_{i}} w_v^2 = C\lambda^2 n^{-(3-\tau)} \frac{w_iw_j}{\ell_n^2} n\frac{\ell_n^{3-\tau}}{w_i^{3-\tau}} \\
&\leq C\lambda^2\frac{w_i^{\tau-2}w_j}{\ell_n}\leq \frac{C\lambda^2}{i^{1-\alpha}j^{\alpha}},
\end{eq}
\begin{eq}
(II)&\leq \frac{p_c^2w_j}{\ell_n} \sum_{v: \frac{\ell_n}{w_i}<w_v\leq \frac{\ell_n}{w_j}}w_v \leq C\lambda^2 n^{-(3-\tau)} \frac{w_j}{\ell_n} \ell_n^{3-\tau} w_i^{\tau-2}\\
&\leq C\lambda^2\frac{w_i^{\tau-2}w_j}{\ell_n}\leq \frac{C\lambda^2}{i^{1-\alpha}j^{\alpha}},
\end{eq}
\begin{eq}
(III) & \leq p_c^2 \#\{w_v>\ell_n/w_j\} = p_c^2 n\frac{w_j^{\tau-1}}{\ell_n^{\tau-1}}\leq C\lambda^2\frac{w_i^{\tau-2}w_j}{\ell_n}\leq \frac{C\lambda^2}{i^{1-\alpha}j^{\alpha}}.
\end{eq}
Thus the proof follows.
\end{proof}

\begin{corollary}[Path counting estimate]\label{c5:lem:geometric-decay}
Let $f_k(i,j)$ denote the probability that there exists a path of length $k$ in $\GRG$ from $i$ to $j$. 
For all $1-\alpha<b<\alpha$, there exists a constant $c_0>0$ such that for all $n\geq 1$,
\begin{equation}\label{c5:eq:geometric-decay}
f_{2k}(i,j) \leq \frac{(c_0\lambda^2)^{k}}{(i\wedge j)^b(i\vee j)^{1-b}}.
\end{equation}
\end{corollary}
\begin{proof}
Note that 
\begin{equation}
f_{2k}(i,j) \leq \sum_{\substack{(v_r)_{r=0}^{2k}: v_r'\text{s distinct}\\ v_0 = i, v_{2k}=j}} p_c^{2k}\prod_{r=0}^{2k} \frac{w_{v_r}w_{v_{r+1}}}{\ell_n+w_{v_r}w_{v_{r+1}}}, 
\end{equation}and the proof follows directly from \cite[Lemma 2.4]{DHH10} using Lemma~\ref{c5:lem:2-nbd-prob}.
\end{proof}

\subsection{Negligible contributions on the total weight} \label{c5:sec:negligible-contributions}
Suppose that $\sC(i)$ denotes the component in $\mathrm{GRG}_n(\bw,p_c(\lambda))$ containing vertex $i$ and $W_k(i) = \sum_{v\in \sC(i), \dst(v,i) = k}w_v$, where $\dst (\cdot, \cdot)$ is used as a notation for graph distance (the number of edges on the shortest path) throughout.
In this section, we identify the terms that have negligible contributions to $W(i)$.  
The next proposition states that the total contribution on the total weight coming from vertices in the odd neighborhood is small. 
Moreover, the total weight outside a large but finite neighborhood of $i$ is also negligible.
\begin{proposition} \label{c5:prop:tail-particular-component}Suppose that $\lambda<c_0^{-1/2}$, where $c_0$ is defined in \textrm{ Corollary~\ref{c5:lem:geometric-decay}}.
For any fixed $i\geq 1$ and $\varepsilon > 0$,
\begin{eq}\label{c5:eq:contr-negligible-both}
\lim_{K\to\infty}\limsup_{n\to\infty}\PR\bigg(\sum_{k>K}W_{2k}(i) > \varepsilon n^{\alpha}\bigg) &= 0, \\  \lim_{n\to\infty} \PR\bigg(\sum_{k=0}^\infty W_{2k+1}(i) > \varepsilon n^{\alpha} \bigg) &= 0.
\end{eq}
\end{proposition}
\begin{proof}
Recall the definition of $f_k(i,j)$ from Corollary~\ref{c5:lem:geometric-decay} and note that $c_0\lambda^2<1$. 
Therefore, using Corollary~\ref{c5:lem:geometric-decay},
\begin{eq}\label{c5:eq:path-count-geom}
&\E[W_{2k}(i)] \leq \sum_{j\in [n]}w_{j} f_{2k}(i,j) \\
&\leq n^{\alpha}(c_0\lambda^2)^k \bigg[\sum_{j<i} \frac{1}{j^{\alpha}} \frac{1}{j^{b}i^{1-b}} + \sum_{j>i} \frac{1}{j^{\alpha}} \frac{1}{i^{b}j^{1-b}}\bigg] \\
&\leq C n^\alpha (c_0\lambda^2)^k \Big[\frac{1}{i^{1-b}} + \frac{1}{i^{\alpha}}\Big] \leq C (c_0\lambda^2)^k\frac{n^{\alpha}}{i^{1-b}}\\
&\leq C (c_0\lambda^2)^k w_i i^{b- (1-\alpha)}.
\end{eq}
Now, an application of Markov's inequality proves the first part of \eqref{c5:eq:contr-negligible-both}. 
We stress that \eqref{c5:eq:path-count-geom} holds uniformly over $i\in [n]$, which we will use in the next part of the proof. 

The proof of the second part in \eqref{c5:eq:contr-negligible-both} is complete if we can show that, for any $i\in [n]$,
\begin{eq}\label{c5:eq:odd-sum-weight-negligible}
n^{-\alpha}\E\bigg[\sum_{k=0}^\infty W_{2k+1}(i)\bigg] \leq \frac{Cn^{-\epsilon}}{i^\rho},
\end{eq}for absolute constants $C>0$, $\epsilon>0$.
Firstly, note that for any vertex $i\in [n]$,  
\begin{equation}\label{c5:eq:one-hop-weight}
p_c \sum_{v\in [n]} \frac{w_{i}w_{v}^2}{\ell_n+w_{i}w_{v}} \leq C \lambda w_i^{\tau-2} n^{3-\tau}.
\end{equation}
To see this, let us split the above sum in two parts with $\{v:w_v\leq \ell_n/w_i\}$ and $\{v:w_v > \ell_n/w_i\}$, and denote them by $(I)$ and $(II)$ respectively.
Then, by Lemma~\ref{c5:lem:order-estimates},
\begin{eq}
(I) &\leq \frac{p_c}{\ell_n} w_{i} n\frac{\ell_n^{3-\tau}}{w_i^{3-\tau}} \leq C\lambda w_i^{\tau-2} n^{(3-\tau)/2},\\
 (II)&\leq p_c\sum_{v:w_v>\ell_n/w_i} w_v \leq C\lambda w_i^{\tau-2} n^{(3-\tau)/2},
\end{eq}and \eqref{c5:eq:one-hop-weight} follows.
Now, we will use the precise bound in \eqref{c5:eq:path-count-geom}. 
Choose $\epsilon = \frac{1}{2}(\frac{3-\tau}{\tau-1}-\frac{3-\tau}{2})$ and $b$ such that $b- (1-\alpha) = \epsilon$.
Thus,
\begin{eq}
n^{-\alpha}\E[W_{2k+1}(i)] &\leq n^{-\alpha} \sum_{v\in [n]} \PR(i \text{ and } v \text{ create an edge}) \E[W_{2k}(v)] \\
&\leq  n^{-\alpha}n^{\epsilon}\sum_{v\in [n]} p_c\frac{w_iw_v}{\ell_n+w_iw_v} C(c_0\lambda^2)^k w_v \\
&= \frac{C(c_0\lambda^2)^k n^{\epsilon}}{n^{\alpha}} \bigg(p_c\sum_{v\in [n]}\frac{w_iw_v^2}{\ell_n+w_iw_v}\bigg) \\
&\leq  \frac{C(c_0\lambda^2)^k}{i^{\rho}} n^{\epsilon}n^{\frac{3-\tau}{2}-\frac{3-\tau}{\tau-1}} = \frac{C(c_0\lambda^2)^k n^{-\epsilon}}{i^{\rho}},
\end{eq}where the last-but-one step follows from \eqref{c5:eq:one-hop-weight}, and in the final step we have used the choice of $\epsilon>0$. 
The proof of \eqref{c5:eq:odd-sum-weight-negligible} now follows using the fact that $c_0\lambda^2<1$, which also concludes the proof of Proposition~\ref{c5:prop:tail-particular-component}.
\end{proof}

We will be interested in obtaining the limit of $\sum_{k=1}^{\infty}W_{2k}(i)$.
Using Proposition~\ref{c5:prop:tail-particular-component}, it is enough to find the limit, as $n\to\infty$, of the quantity $\sum_{k=1}^{K}W_{2k}(i)$ for each fixed $K\geq 1$.
The next proposition states that for each fixed $k\geq 1$, the primary contribution to $W_{2k}(i)$ arises only due to the hubs. 
For $\delta>0$, define $V_{L}(\delta) := \{v:w_v>\delta n^{\alpha}\}$, and  $W_{k}(i,\delta):= \sum_{v\notin V_L(\delta), \dst(v,i)=k}w_v$. 
\begin{proposition}\label{c5:prop-non-hub-contribution-K-nbd}
For any fixed $i\geq 1$, $K\geq 1$, and $\varepsilon>0$, 
\begin{equation}
\lim_{\delta \to 0}\limsup_{n\to\infty}\PR\bigg(\sum_{k=1}^K W_{2k}(i,\delta)>\varepsilon n^{\alpha} \bigg) =0.
\end{equation}
\end{proposition}
\begin{proof}
Suppose that we choose $\delta>0$ to be so small that $i\in V_L(\delta)$. 
Therefore, $v\notin V_L$ implies that $v>i$.
Using Corollary~\ref{c5:lem:geometric-decay}, it follows that 
\begin{eq}
\E[W_{2k}(i,\delta)] &\leq \sum_{v\notin V_L(\delta)} \frac{n^\alpha}{v^\alpha} \frac{(c_0\lambda^2)^k}{i^bv^{1-b}} \\
&
\leq \frac{Cn^{\alpha}}{i^b} \sum_{v>\delta^{-(\tau-1)}} \frac{1}{v^{1+\alpha-b}} \leq \frac{Cn^{\alpha}\delta^{(\tau-1)(\alpha-b)}}{i^b}.
\end{eq}
Therefore, 
\begin{equation} \label{c5:eq:weight-nonhub-even}
n^{-\alpha}\E\bigg[\sum_{k=1}^K W_{2k}(i,\delta)\bigg] \leq CKi^{-b}\delta^{(\tau-1)(\alpha-b)}.
\end{equation}
Now, an application of Markov's inequality completes the proof.
\end{proof}

\subsection{Total weight of components containing hubs} \label{c5:sec:hub-structure}
To simplify writing we will always assume that $c_{\sss F} =1$ without loss of generality. 
Recall the definition of the graph $G_{\infty}(\lambda)$ from Section~\ref{c5:sec:GRG-results}.
In the rest of this section, we write $\theta_i = i^{-\alpha}$.
Consider the following objects defined on the graph $G_{\infty}$: $W_{k}^{\infty}(i) = \sum_{j\in C(i),\ \dst(i,j) = k} \theta_j$, where $C(i)$ is defined as the component in $G_{\infty}(\lambda)$ containing vertex $i$, and $W_{\leq K}^{\infty}(i) = \sum_{k=1}^KW_{k}^{\infty}(i)$, and $W^\infty(i) = \sum_{k=1}^\infty W_{k}^{\infty}(i)$.
The main result of this subsection is the following:
\begin{theorem}\label{c5:thm:total-weight-i}
Suppose that $\lambda<c_0^{-1/2}$, where $c_0$ is defined in \textrm{ Corollary~\ref{c5:lem:geometric-decay}}.
For each fixed $i\geq 1$, as $n\to\infty$, $n^{-\alpha}W(i) \dto W^\infty (i)$.
\end{theorem}
The key ingredient in the proof is the proposition below. 
We immediately give the proof of Theorem~\ref{c5:thm:total-weight-i} after stating the proposition, and devote the rest of this section to the proof of Proposition~\ref{c5:prop:convergence-K-nbd}:
\begin{proposition}\label{c5:prop:convergence-K-nbd}
For each fixed $i\geq 1$ and $K\geq 1$, as $n\to\infty$, $$n^{-\alpha}\sum_{k=1}^{K}W_{2k}(i) \dto W_{\leq K}^{\infty}(i).$$
\end{proposition}

\begin{proof}[Proof of Theorem~\ref{c5:thm:total-weight-i}]
Propositions~\ref{c5:prop:tail-particular-component},~\ref{c5:prop-non-hub-contribution-K-nbd},~and~\ref{c5:prop:convergence-K-nbd} together directly concludes the proof. 
\end{proof}

Throughout this subsection, we will use the notation $V_{L} := \{i:w_i>\delta n^{\alpha}\}$,  $V_{\sss SS} = \{i: w_i < \delta n^{\rho}\}$, $V_{\sss SI} = \{i: \delta n^{\rho}\leq w_i \leq \delta^{-1} n^{\rho}\}$ and $V_{\sss SL} = \{i: w_i > \delta^{-1} n^{\rho}\}$.
We have tacitly avoided augmenting $\delta > 0$ in the notation to simplify notation. 
Note that for $v\in V_{\sss SI}$ and $i\in V_L$, $w_iw_{v} = \Theta(\ell_n)$.
Consider the following multigraph $G_{n,\delta}$ on the vertex set $V_L$, where the number of edges $X_{ij}$ between $i$ and $j$ is number of distinct $v\in [n]$ such that both $(i,v)$ and $(v,j)$ are edges of $\rGRG(\bw,p_c)$.
Note that, for any $i\neq j$,
\begin{eq}\label{c5:eq:distn-Xij}
X_{ij} = \sum_{v\neq i,j} \ber\bigg(\frac{w_iw_jw_v^2p_c^2}{(\ell_n+w_iw_v)(\ell_n+w_jw_v)}\bigg),
\end{eq}with the different Bernoulli random variables in the sum \eqref{c5:eq:distn-Xij} being independent.
We can split the above sum in three parts with $v\in V_{SS}$, $v\in V_{SI}$, and $v\in V_{SL}$ and denote them by $(I)$, $(II)$, and $(III)$ respectively.
Now, using Lemma~\ref{c5:lem:order-estimates}, 
\begin{eq} \label{c5:hub-connection-term-1}
\E[(I)] &\leq \frac{w_iw_jp_c^2}{\ell_n^2} \sum_{v:w_v<\delta n^{\rho}}w_v^2 \\
&\leq C_0\theta_1^2 \delta^{3-\tau}n^{2\alpha - 2+1+(3-\tau)\rho}p_c^2 \leq C_1\delta^{3-\tau},
\end{eq}
\begin{eq} \label{c5:hub-connection-term-3}
\E[(III)] \leq p_c^2 \times \#\{v:w_v>\delta^{-1}n^{\rho}\} =C \delta^{\tau-1}.
\end{eq}
Using the above and Markov's inequality, it follows that
\begin{eq}\label{c5:eq:Xij-approximation}
X_{ij} &= \sum_{v: w_v\in [\delta n^{\rho},\delta^{-1}n^{\rho}]} \ber\bigg(\frac{w_iw_jw_v^2p_c^2}{(\ell_n+w_iw_v)(\ell_n+w_jw_v)}\bigg) + E(\delta,n) \\
&= X_{ij}(\delta) + E(\delta,n),
\end{eq} where for any $\varepsilon>0$ 
\begin{eq}\label{c5:E-delta-n}
\lim_{\delta \to 0}\limsup_{n\to\infty}\PR(E(\delta,n) >\varepsilon) = 0.
\end{eq}
Define 
\begin{eq}
\lambda_{ij}(\delta) = \sum_{v: w_v\in [\delta n^{\rho},\delta^{-1}n^{\rho}]} \frac{w_iw_jw_v^2p_c^2}{(\ell_n+w_iw_v)(\ell_n+w_jw_v)} .
\end{eq}
The proof of Proposition~\ref{c5:prop:convergence-K-nbd} is decomposed into three key lemmas below. 
After stating these lemmas, we first prove Proposition~\ref{c5:prop:convergence-K-nbd}, and subsequently prove the lemmas. 
\begin{lemma}\label{c5:lem:poisson-approximation}For any $\delta\in (0,1)$, and $i,j \geq 1$
\begin{equation}
\lim_{n\to\infty}\dTV\big(X_{ij}(\delta), \poi(\lambda_{ij}(\delta))\big) =0.
\end{equation}
\end{lemma}

\begin{lemma}\label{c5:lem:asymptotics-lambda-ijd}For any fixed $i,j \geq 1$
\begin{equation}
\lim_{\delta\to 0} \lim_{n\to\infty} \lambda_{ij}(\delta) = \lambda^2\int_{0}^{\infty} \frac{\theta_i\theta_j x^{-2\alpha}}{(\mu + \theta_i x^{-\alpha})(\mu + \theta_j x^{-\alpha})} \dif x.
\end{equation}
\end{lemma}

\begin{lemma} \label{c5:lem:symptotic-independence}For any $\delta\in (0,1)$, the collection of random variables $(X_{ij}(\delta))_{i,j\in V_L}$ is asymptotically independent. 
\end{lemma}

\begin{proof}[Proof of Proposition~\ref{c5:prop:convergence-K-nbd}]
The proof follows directly from Lemmas~\ref{c5:lem:poisson-approximation},~\ref{c5:lem:asymptotics-lambda-ijd}, and~\ref{c5:lem:symptotic-independence}.
\end{proof}

\begin{proof}[Proof of Lemma~\ref{c5:lem:poisson-approximation}]
Using standard inequalities from Stein's method \cite[Theorem 2.10]{RGCN1}, it follows that, as $n\to\infty$,
\begin{eq}
&\dTV\big(X_{ij}(\delta), \poi(\lambda_{ij}(\delta))\big) \leq \sum_{v: w_v\in [\delta n^{\rho},\delta^{-1}n^{\rho}]} \bigg(\frac{w_iw_jw_v^2p_c^2}{(\ell_n+w_iw_v)(\ell_n+w_jw_v)}\bigg)^2
\\
 &\leq Cn^{4\alpha - 4}p_c^4 \sum_{v: w_v\in [\delta n^{\rho},\delta^{-1}n^{\rho}]}w_v^4
 \leq \frac{C}{\delta^2}n^{2\alpha - 2}p_c^4\sum_{v: w_v\in [\delta n^{\rho},\delta^{-1}n^{\rho}]}w_v^2\\
 &= \frac{C}{\delta^2} n^{2\alpha-2}n^{-2(3-\tau)}n^{1+ (3-\tau) \rho} = \frac{Cn^{-(3-\tau)}}{\delta^2} \to 0,
\end{eq}and the proof follows.
\end{proof}

\begin{proof}[Proof of Lemma~\ref{c5:lem:asymptotics-lambda-ijd}]
Observe that
\begin{eq}
\lambda_{ij}(\delta) &= \sum_{v: w_v\in [\delta n^{\rho},\delta^{-1}n^{\rho}]} \frac{w_iw_jw_v^2p_c^2}{(\ell_n+w_iw_v)(\ell_n+w_jw_v)} \\
&=p_c^2 n^{4\alpha -2} \sum_{k=\delta^{\tau-1}n^{3-\tau}}^{\delta^{-(\tau-1)}n^{3-\tau}} \frac{\theta_i\theta_jk^{-2\alpha}}{(\mu + \theta_i n^{-\frac{3-\tau}{\tau-1}}k^{-\alpha})(\mu + \theta_j n^{-\frac{3-\tau}{\tau-1}}k^{-\alpha})} \\
&  = p_c^2 n^{4\alpha -2} \int_{\delta^{\tau-1}n^{3-\tau}}^{\delta^{-(\tau-1)}n^{3-\tau}} \frac{\theta_i\theta_jk^{-2\alpha}}{(\mu + \theta_i n^{-\frac{3-\tau}{\tau-1}}k^{-\alpha})(\mu + \theta_j n^{-\frac{3-\tau}{\tau-1}}k^{-\alpha})} \dif k\\
& = (1+o(1))\lambda^2\int_{\delta^{\tau-1}}^{\delta^{-(\tau-1)}} \frac{\theta_i\theta_j x^{-2\alpha}}{(\mu + \theta_i x^{-\alpha})(\mu + \theta_j x^{-\alpha})} \dif x,
\end{eq}and the proof follows.
\end{proof}

\begin{proof}[Proof of Lemma~\ref{c5:lem:symptotic-independence}]
Note that for pairs $(i,j)$, and $(k,l)$ with $\{i,j\}\cap \{k,l\} = \varnothing$, $X_{ij}$ and $X_{kl}$ are independent due to the independence of the occupancy of edges in $\rGRG(\bw,p_c)$.
The only dependence between $X_{ij}$ and $X_{ik}$ arises due to connections $(i,v)$, $(v,j)$ and $(v,k)$. 
Thus, the lemma is proved if we can show that the above does not arise with high probability. 
Note that
\begin{eq}
&\PR(i,j,k\in V_L: \exists v\in [n] \text{ such that }(i,v), (v,j),(v,k) \text{ are edges in }\rGRG(\bw,p_c)) \\
& \leq \sum_{i,j,k\in V_L}\sum_{v\in [n]}p_c^3\frac{w_iw_v}{\ell_n+w_iw_v}\frac{w_jw_v}{\ell_n+w_jw_v}\frac{w_kw_v}{\ell_n+w_kw_v}.
\end{eq}
Again, let us split the above sum in two pars with $\{v:w_v\leq n^{\rho} \}$ and $\{v:w_v> n^{\rho} \}$, and denote them by $(I)$ and $(II)$ respectively. 
Thus, using Lemma~\ref{c5:lem:order-estimates},
\begin{eq}
(I) &\leq \frac{p_c^3w_1^3|V_L|^3}{\ell_n^3} \sum_{v:w_v\leq n^{\rho}} w_v^3 \leq \frac{p_c^3w_1^3|V_L|^3 n^{\rho}}{\ell_n^3} \sum_{v:w_v\leq n^{\rho}} w_v^2 \\
&\leq C p_c^3 n^{3\alpha -3 +\rho +1+(3-\tau)\rho} = O(p_c),
\end{eq}
\begin{eq}
(II) \leq |V_L|^3p_c^3 \#\{v:w_v>n^{\rho}\} = O(p_c).
\end{eq}
This completes the proof of Lemma~\ref{c5:lem:symptotic-independence}.
\end{proof}

\subsection{Sizes of components containing hubs}\label{c5:sec:component-with-hubs}
In this section, we consider the asymptotic size of $\sC(i)$, the component containing vertex $i$. 
We will prove the following theorem: 
\begin{theorem}\label{c5:thm:comp-size-i}
For each fixed $i\geq 1$, as $n\to\infty$, $(n^{\alpha}p_c)^{-1}|\sC(i)| \dto W^\infty (i)$.
\end{theorem}
Lemmas~\ref{c5:lem:odd-even-com-size},~and~\ref{c5:lem:non-hub-comp-size} identify the primary contribution to the component sizes.
Since the proof of the lemmas are short, they are given immediately.
We conclude the section with the proof of Theorem~\ref{c5:thm:comp-size-i} using these two lemmas.
Define $\sC_k(i):= \{v\in \sC(i): \dst(v,i) = k\}$. 
Thus $\sC_k(i)$ denotes the set of vertices at exactly distance $k$ from vertex $i$.
\begin{lemma}\label{c5:lem:odd-even-com-size}
 Suppose that $\lambda<c_0^{-1/2}$, where $c_0$ is defined in \textrm{ Corollary~\ref{c5:lem:geometric-decay}}. 
 For any fixed $i\geq 1$, and $\varepsilon>0$, 
\begin{eq}\label{c5:eq:even-comp-size}
\lim_{n\to\infty} \PR( \sum_{k=0}^\infty |\sC_{2k}(i)| > \varepsilon n^{\alpha}p_c )&=  0, \\ 
\lim_{K\to\infty}\limsup_{n\to\infty}\PR\bigg( \sum_{k>K} |\sC_{2k+1}(i)| > \varepsilon n^{\alpha}& p_c\bigg) = 0.
\end{eq}
\end{lemma}
\begin{proof}
Note that 
\begin{eq}\label{c5:odd-comp-even-weight}
\E\Big[|\sC_{k+1}(i)|\Big|\bigcup_{r=1}^k \sC_{r}(i)\Big] &\leq p_c\sum_{v_1\in \sC_{k}(i)} \sum_{v_2\in [n]} \frac{w_{v_1}w_{v_2}}{\ell_n+w_{v_1}w_{v_2}} \\
&\leq p_c W_{k}(i), 
\end{eq}and therefore $\E[|\sC_{k+1}(i)|] \leq p_c \E[ W_{k}(i)]$. Now the estimates \eqref{c5:eq:path-count-geom}, \eqref{c5:eq:odd-sum-weight-negligible} conclude the proof.
\end{proof}

Let $\sC_{k}'(i)\subset \sC_{k}(i)$ denote the vertices of $\sC_{k}(i)$ that are neighbors of some vertex in $\sC_{k-1}(i)\cap V_L$, where $V_L: = \{v: w_v> \delta n^{\alpha}\}$ for some $\delta >0$. 
Then the following lemma estimates the contribution to the component size due to the non-hubs at distance $2k+1$:
\begin{lemma}\label{c5:lem:non-hub-comp-size}
For each fixed $i\geq 1$, $k\geq 1$, and $\varepsilon>0$,
\begin{equation}\label{c5:comp-nonhub-odd}
\lim_{\delta\to 0}\limsup_{n\to\infty}\PR(|\sC_{2k+1}(i)\setminus\sC_{2k+1}'(i)| >\varepsilon n^{\alpha} p_c) = 0.
\end{equation}
\end{lemma}
\begin{proof}
Using an identical argument as \eqref{c5:odd-comp-even-weight} yields $\E[|\sC_{2k+1}(i)\setminus\sC_{2k+1}'(i)|] \leq p_c \E[ W_{2k}(i,\delta)]$, where $W_{k}(i,\delta):= \sum_{v\in \sC_k(i) \cap V_L^c}w_v$. 
Now, \eqref{c5:comp-nonhub-odd} follows from \eqref{c5:eq:weight-nonhub-even}.
\end{proof}

\begin{fact} \label{c5:fact:variance}
Given a matrix $(p_{ij})_{i\in [m],j\in [n]}$, let $I_{ij}\sim \mathrm{Ber}(p_{ij})$, independently. 
For all $i\in [m]$, construct the random set $V_i:= \{j: I_{ij} = 1\}$, and let $V=\cup_{i\in [m]}V_i$. 
Then $\var{|V|} \leq \sum_{i\in [m]} \var{|V_i|}$.
\end{fact} 
\begin{proof}
Note that $\PR(v_j\in V) = 1- \prod_{i\in [m]}(1-p_{ij})$, and the events $\{u\in V\}$ and $\{w\in V\}$ are independent for all $u\neq w$. 
Further, $\var{|V_i|} = \sum_{j\in [n]} p_{ij}(1-p_{ij})$. 
Therefore, 
\begin{eq}
\var{|V|} &= \sum_{j\in [n]} \bigg(1- \prod_{i\in [m]}(1-p_{ij})\bigg) \prod_{i\in [m]}(1-p_{ij}) \\
&\leq \sum_{j\in [n]} \sum_{i\in [m]} p_{ij}\prod_{i\in [m]}(1-p_{ij}) \\
&\leq  \sum_{j\in [n]} \sum_{i\in [m]} p_{ij}(1-p_{ij}) = \sum_{i\in [m]} \var{|V_i|},
\end{eq}where the third step follows using the fact the $1-\prod (1-x_r) \leq \sum x_r$, whenever $0\leq x_{r} \leq 1$ for every $r\geq 1$.
\end{proof}

\begin{proof}[Proof of Theorem~\ref{c5:thm:comp-size-i}]
Let $F_k$ denote the minimal sigma-algebra with respect to which $\cup_{r=1}^k \sC_{r}(i)$ is measurable. 
Define $W_k'(i):= \sum_{v\in \sC_k(i)\cap V_L}w_v$. 
Using Lemmas~\ref{c5:lem:odd-even-com-size},~\ref{c5:lem:non-hub-comp-size}, it is now enough to show that 
$|\sC_{2k+1}'(i)| = p_c W_{2k}'(i) (1+E(\delta,n))$, where the random variable $E(\delta,n)$ satisfies \eqref{c5:E-delta-n}.
 This follows from Chebyshev's inequality if we can show that 
 \begin{equation} \label{c5:eq:estimate-expt-var-comp}
 \E[|\sC_{2k+1}'(i)| \vert F_{2k}] = p_c W_{2k}'(i) (1+o(1)), \quad \text{and} \quad \var{|\sC_{2k+1}'(i)| \vert F_{2k}} \leq E_{n},
 \end{equation}where $\E[E_n] = o(n^{2\alpha}p_c^2)$. 
 For $v\in [n]\setminus \sC_{2k}(i)$, let $I_{v}$ denote the indicator that there exists $u\in \sC_{2k}(i)\cap V_L$ such that $(v,u)$ creates an edge. 
 Thus, for any $v\in [n]\setminus \sC_{2k}(i)$
 \begin{equation}
 \PR(I_{v} = 1\vert F_{2k}) = 1 - \prod_{u\in \sC_{2k}(i)\cap V_L}\bigg(1-\frac{p_cw_{u}w_v}{\ell_n+w_{u}w_v}\bigg).
 \end{equation}
Using inclusion-exclusion with respect to the union of $u\in \sC_{2k}(i)\cap V_L$, it now follows that 
\begin{eq}
&\E[|\sC_{2k+1}'(i)| \vert F_{2k}] \\
&\geq \sum_{v\notin \sC_{2k}(i)} \sum_{u\in \sC_{2k}(i) \cap V_L} \frac{p_cw_{u}w_v}{\ell_n+w_{u}w_v} \\
&\hspace{2cm}-
\sum_{v\notin \sC_{2k}(i)} \sum_{\substack{u_1,u_2\in \sC_{2k}(i) \cap V_L,\\  u_1<u_2}} \frac{p_c^2w_{u_1}w_v^2 w_{u_2}}{(\ell_n+w_{u_1}w_v)(\ell_n+w_{u_2}w_v)}.
\end{eq} Let us denote the first and second terms above by $(I)$ and $(II)$ respectively. 
Note that 
\begin{eq}
(II) \leq  \sum_{u_1,u_2\in V_L} \sum_{v\in [n]} \frac{p_c^2w_{u_1}w_v^2 w_{u_2}}{(\ell_n+w_{u_1}w_v)(\ell_n+w_{u_2}w_v)} = O(1) = o(n^{\alpha}p_c), 
\end{eq}almost surely, where the third equality above follows using \eqref{c5:hub-connection-term-1}, \eqref{c5:hub-connection-term-3} and Lemma~\ref{c5:lem:asymptotics-lambda-ijd}. 
Further, by observing 
\begin{eq}
\frac{1}{n^{\alpha}p_c} &\sum_{u\in \sC_{2k}(i) \cap V_L} \sum_{v\in \sC_{2k}(i)} \frac{p_cw_{u}w_v}{\ell_n+w_{u}w_v} \\
&\leq \frac{1}{n^{\alpha}\ell_n} (W_{2k}(i))^2 = \OP(n^{\alpha-1}) = \oP(1),
\end{eq}it follows that 
\begin{eq}
\E[|\sC_{2k+1}(i)| \vert F_{2k}] =  \sum_{u\in \sC_{2k}(i) \cap V_L} \sum_{v\in [n]} \frac{p_cw_{u}w_v}{\ell_n+w_{u}w_v} +\oP(n^{\alpha}p_c).
\end{eq}
Now for $\varepsilon>0$ (sufficiently small), let us split the above term in two parts with $\{v:w_v\leq n^{\rho-\varepsilon}\}$, $\{v:w_v>n^{\rho-\varepsilon}\}$, and call them $(I)$ and $(II)$ respectively.
Now, using Lemma~\ref{c5:lem:order-estimates},
\begin{eq}
\frac{(II)}{n^{\alpha}p_c} \leq C |V_L|^2 \frac{n^{1- (\tau-1)\rho+\varepsilon (\tau-1)}}{n^{\alpha}} \leq C |V_L|^2 n^{\rho - \rho(\tau-1)+\varepsilon (\tau-1)} = o(1),
\end{eq}almost surely, and 
\begin{eq}
(I) = p_c \sum_{u\in \sC_{2k}(i) \cap V_L} \sum_{v:w_v \leq n^{\rho-\varepsilon}} \frac{w_uw_v}{\ell_n(1+o(1))} = p_c W_{2k}'(i) (1+o(1)).
\end{eq}
The estimate for the expectation term in \eqref{c5:eq:estimate-expt-var-comp} now follows. 
 For $u\in \sC_{2k}(i)$, let $N_u$ denote the number of neighbors of $u$ in $\sC_{2k+1}(i)$. For the variance term, Fact~\ref{c5:fact:variance} implies that 
\begin{eq}
&\var{|\sC_{2k+1}'(i)|\vert F_{2k}} \leq \sum_{u\in \sC_{2k}(i) \cap V_L} \var{N_u} \\
&\leq \sum_{u\in \sC_{2k}(i) \cap V_L} \sum_{v\in [n]} \frac{p_cw_{u}w_v}{\ell_n+w_{u}w_v} \leq p_c W_{2k}'(i), 
\end{eq}and the required estimate in \eqref{c5:eq:estimate-expt-var-comp} follows using Theorem~\ref{c5:thm:total-weight-i}.
\end{proof}

\subsection{Proof of Theorem~\ref{c5:thm:main-crit}} \label{c5:sec:l2-tightness}
To conclude Theorem~\ref{c5:thm:main-crit} using Theorems~\ref{c5:thm:total-weight-i},~\ref{c5:thm:comp-size-i}, it is enough to show that $(n^{-\alpha}(W_{\sss (i)})_{i\geq 1})_{n\geq 1}$ and $((n^{\alpha}p_c)^{-1}(|\sC_{\sss (i)}(p_c)|)_{i\geq 1})_{n\geq 1}$ are tight in $\ell^2_{\shortarrow}$, and the limiting object in Theorem~\ref{c5:thm:main-crit} is finite almost surely. 
We state the tightness below and defer the finiteness of the limiting object to Proposition~\ref{c5:prop:no-infinite-cluster} in the next section.
\begin{proposition} \label{c5:prop-l2-tightness-23}
 Suppose that $\lambda<c_0^{-1/2}$, where $c_0$ is defined in \textrm{ Corollary~\ref{c5:lem:geometric-decay}}. 
Then, $(n^{-\alpha}(W_{\sss (i)})_{i\geq 1})_{n\geq 1}$ and $((n^{\alpha}p_c)^{-1}(|\sC_{\sss (i)}(p_c)|)_{i\geq 1})_{n\geq 1}$ are tight in $\ell^2_{\shortarrow}$.
\end{proposition}
\begin{proof}
To show the $\ell^2_{\shortarrow}$-tightness of $(n^{-\alpha}(W_{\sss (i)})_{i\geq 1})_{n\geq 1}$, it is enough to show that for any $\varepsilon>0$
\begin{equation}
\lim_{K\to\infty} \limsup_{n\to\infty} \PR\bigg(\sum_{i>K} W_{\sss (i)}^2 > \varepsilon n^{2\alpha}\bigg) = 0.
\end{equation} 
Consider the graph $\rGRG(\bw, p_c)\setminus [K]$, and define $W_{\sss (i) K}$, $W(v;K)$ on this graph analogously as  $W_{\sss (i)}$ and $W(v)$.
It is enough to show that 
\begin{equation}\label{c5:eq:tightness-suff}
\lim_{K\to\infty} \limsup_{n\to\infty} \PR\bigg(\sum_{i\geq 1} W_{\sss (i),K}^2 > \varepsilon n^{2\alpha}\bigg) = 0.
\end{equation}
Let $V_n^*(K)$ denote a vertex chosen in a size-biased manner from $[n]\setminus [K]$ with the sizes being proportional to $(w_i)_{i>K}$, chosen independently from $\rGRG(\bw, p_c)$. 
Let $\ell_n(K):=\sum_{i>K}w_i$. Then, $\ell_n(K) = \ell_n(1+o(1))$ for each fixed $K\geq 1$.
Note that 
\begin{eq} \label{c5:ss-simple-1}
&\E\bigg[\sum_{i\geq 1} W_{\sss (i),K}^2\bigg] = \ell_n(K) \E[W(V_n^*(K);K)] \\
&\leq \ell_n(K) \bigg[\sum_{k\geq 0} \sum_{v\in[n]}w_vf_{2k} (V_n^*(K),v)+\sum_{k\geq 0} \sum_{v\in[n]}w_vf_{2k+1} (V_n^*(K),v)\bigg].
\end{eq}
Let us denote the two sums above by $(I)$ and $(II)$ respectively. 
We can now use the estimates from \eqref{c5:eq:path-count-geom} and \eqref{c5:eq:odd-sum-weight-negligible}.
Note that \eqref{c5:eq:path-count-geom} implies that, for $\lambda<c_0^{-1/2}$,
\begin{eq}
n^{-2\alpha}(I) \leq C n^{-2\alpha}\ell_n(K)\E[w_{V_n^*(K)}] = C  \sum_{i>K} i^{-2\alpha},
\end{eq}which tends to zero as $K\to\infty$.
Moreover, \eqref{c5:eq:odd-sum-weight-negligible} implies that for $\lambda<c_0^{-1/2}$,
\begin{eq}
n^{-2\alpha}(II) \leq  C n^{-\epsilon}\sum_{i>K} \frac{1}{i} = \frac{C \log(n)}{n^{\epsilon}},
\end{eq}which goes to zero as $n\to\infty$. Thus, \eqref{c5:ss-simple-1} follows, and \eqref{c5:eq:tightness-suff} follows from Markov's inequality.

For the  $\ell^2_{\shortarrow}$-tightness of $((n^{\alpha}p_c)^{-1}(|\sC_{\sss (i)}(p_c)|)_{i\geq 1})_{n\geq 1}$, note that, for any vertex $i\in [n]$, $\E[|\sC_{k+1}(i)|] \leq p_c \E[W_{k}(i)]$ for all $k\geq 1$, and therefore $\E[|\sC(i)|] \leq p_c\E[W(i)]$. 
Thus, if $V_n^*$ a vertex chosen in a size-biased manner with the sizes being $(w_i/\ell_n)_{i\in [n]}$, chosen independently of $\rGRG(\bw,p_c)$,
\begin{eq}
\E\bigg[\sum_{i\geq 1}|\sC_{\sss (i)}|^2\bigg] &= \E\bigg[\sum_{i\in [n]}|\sC(i)|\bigg] \leq p_c \E\bigg[\sum_{i\in [n]}W(i)\bigg] = p_c\E\bigg[\sum_{i\geq 1}|\sC_{\sss (i)}|W_{\sss (i)}\bigg] \\
&= \ell_n p_c \E[|\sC(V_n^*)|] \leq \ell_n p_c^2 \E[W(V_n^*)] = p_c^2 \E\bigg[\sum_{i\geq 1}W_{\sss (i)}^2\bigg].
\end{eq}We can apply this quantity to $\rGRG(\bw,p_c)\setminus [K]$ as above and the proof of Proposition~\ref{c5:prop-l2-tightness-23} is now complete.
\end{proof}


\subsection{Finiteness of the limiting object}
\label{c5:sec:limit-object-finite}
We write $\theta_i = i^{-\alpha}$. 
Recall that the graph $G_{\infty}(\lambda)$ with vertex set $\Z_+$ is created by creating $\mathrm{Poisson}(\lambda_{ij})$ many edges vertices $i$ and $j$, where 
\begin{equation}
\label{c5:eq:lambda-ij-expr}
\lambda_{ij} = \lambda^2 \int_0^\infty \frac{\theta_i\theta_jx^{-2\alpha}}{(\mu+\theta_ix^{-\alpha})(\mu+\theta_jx^{-\alpha})}\dif x.
\end{equation}
Let $C(i)$ denote the connected component containing vertex $i$ and define $W^{\infty}(i) = \sum_{j\in C(i)}\theta_j$.
We will show the following and the fact that $(W^\infty_{\sss (i)})_{i\geq 1} \in \ell^2_{\shortarrow}$ then follows from Proposition~\ref{c5:prop-l2-tightness-23} using Fatou's lemma.
\begin{proposition}\label{c5:prop:no-infinite-cluster}
Consider $\theta_{i} = i^{-\alpha}$. There exists and absolute constant $\lambda_0$ such that the following holds for any $\lambda\in (0,\lambda_0)$: 
For each $i\in \Z_+$, $W^{\infty}(i)<\infty$ almost surely.
\end{proposition}
\begin{lemma}\label{c5:lem:ub-PAM-probability} 
$\lambda_{ij} \leq \frac{C\lambda^2}{ (i\wedge j)^{1-\alpha} (i\vee j)^\alpha}$ for some absolute constant $C>0$.
\end{lemma}
\begin{proof}
Without loss of generality, let us assume that $\theta_i>\theta_j$ (i.e., $i<j$), and $\mu = 1$.
Let us split the integral \eqref{c5:eq:lambda-ij-expr} in three parts with $\{x: \theta_i x^{-\alpha}<1\}$, $\{x: 0<\theta_j x^{-\alpha}<1<\theta_i x^{-\alpha}\}$ and $\{x: \theta_j x^{-\alpha}>1\}$ and denote them by $(I)$, $(II)$ and $(III)$ respectively.
Then,
\begin{eq}
(I) \leq \lambda^2\theta_i\theta_j \int_{\theta_i^{1/\alpha}}^\infty x^{-2\alpha} \dif x \leq C\lambda^2 \theta_i \theta_j (\theta_i^{1/\alpha})^{1-2\alpha} = \frac{C \lambda^2}{  i^{1-\alpha} j^{\alpha}},
\end{eq}
\begin{eq}
(II) \leq \lambda^2\theta_j \int_{\theta_j^{1/\alpha}}^{\theta_i^{1/\alpha}} x^{-\alpha} \dif x = \frac{C\lambda^2}{j^{\alpha}} \bigg[\frac{1}{i^{1-\alpha}}-\frac{1}{j^{1-\alpha}} \bigg] \leq \frac{C \lambda^2}{  i^{1-\alpha} j^{\alpha}},
\end{eq}
\begin{eq}
(III) \leq \lambda^2\int_0^{\theta_j^{1/\alpha}} \dif x \leq \frac{C\lambda^2}{j} \leq \frac{C \lambda^2}{  i^{1-\alpha} j^{\alpha}}.
\end{eq}Thus the proof follows.
\end{proof}
\begin{proof}[Proof of Proposition~\ref{c5:prop:no-infinite-cluster}] 
We will use the path counting estimates from \cite{DHH10}. 
We estimate the probability that there exists a non self-intersecting path of length $k$ from $i$ to $j$ in $G_\infty$.
Using Lemma~\ref{c5:lem:ub-PAM-probability}, note that
\begin{eq}
&\PR(j\in C(i), \dst(i,j) = k) \\
&\leq (c_2\lambda^2)^k  \sum_{\substack{(v_{r})_{r=0}^k:v_r\text{'s distinct} \\ v_0 = i, v_k = j
}} \prod_{r=0}^{k-1} \frac{1}{(v_r\wedge v_{r+1})^{1-\alpha}(v_r\vee v_{r+1})^{\alpha}}.
\end{eq}
Using \cite[Lemma 2.4]{DHH10}, for any $b<\alpha$,  
\begin{eq}
\PR(j\in C(i), \dst(i,j) = k) \leq \frac{(c_3\lambda^2)^k }{i^{b} j^{1-b}},
\end{eq}for some absolute constant $c_3>0$.
Thus, for $\lambda<1/\sqrt{c_3}$,
\begin{eq}
\E[W^{\infty}(i)] \leq C \sum_{j\geq 1} j^{-\alpha }\frac{(c_3\lambda^2)^k }{i^{b} j^{ 1-b}} \leq \frac{C}{i^{b}} < \infty,
\end{eq}which implies $W^{\infty}(i)<\infty$ almost surely. 
\end{proof}
\subsection{Near-critical behavior}
\begin{proof}[Proof of Theorem~\ref{c5:thm:barely-subcrit-single-edge}]
The proof can be completed by modifying the arguments in Section~\ref{c5:sec:proof-GRG}. 
In fact, if $p_n = \lambda_n n^{-(3-\tau)/2}$, for some $\lambda_n\to 0$, then Lemma~\ref{c5:lem:2-nbd-prob} holds with $\lambda$, replaced by $\lambda_n$. 
One can use identical arguments as Proposition~\ref{c5:prop:tail-particular-component} to show that $W(i) = w_i (1+\oP(1))$.
Finally, one can use identical arguments as Proposition~\ref{c5:prop-l2-tightness-23} to deduce the $\ell^2_{\shortarrow}$ tightness of the vector of component sizes and weights. 
Thus, the proof of Theorem~\ref{c5:thm:barely-subcrit-single-edge} follows.
\end{proof}

\section{Erased configuration model: Proofs}
In this section, we provide the necessary adaptations required to the arguments in Section~\ref{c5:sec:proof-GRG} to complete the proof of Theorem~\ref{c5:thm:main-ECM}.
Let $e_{ij}$ denote the number of edges between vertices $i$ and $j$ in~$\CM$. 
Note that an edge $\{i,j\}$ appears in $\ECM$ if and only if $e_{ij}\geq 1$.
We start by describing two elementary properties of the occurrence of edges in Lemmas~\ref{c5:lem:edge-occupancy-negative-correlation},~\ref{c5:crude-estimate}, which will be the key to the required adaptations:
\begin{lemma}\label{c5:lem:edge-occupancy-negative-correlation}
For all distinct $i,v,j$, $\PR(e_{iv}\geq 1, e_{vj}\geq 1) \leq \PR(e_{iv}\geq 1)\PR(e_{vj}\geq 1)$ .
\end{lemma}
\begin{proof}
For any $m_1\geq 1$, 
\begin{eq}\label{c5:eq:edge-connection-prob}
\PR(e_{iv} = m_1) = \frac{{d_i \choose m_1}{d_v \choose m_1} m_1!}{(\ell_n-1)\dots (\ell_n -2m_1+1)}.
\end{eq}
Further, fix an $m_2\geq 1$.
Now conditionally on the first $m_1$ edge, the probability of $\{e_{vj}=m_2\}$ is also given by \eqref{c5:eq:edge-connection-prob}, with $d_v$ changed to $d_v-m_2$ and the product in the denominator being $(\ell_n-2m_1-1) \cdots (\ell_n -2(m_1+m_2)+1)$.
Therefore, for any $m_1,m_2\geq 1$,
\begin{eq}
\PR(e_{iv} = m_1, e_{vj} = m_2) =\frac{{d_i \choose m_1}{d_v \choose m_1} m_1!{d_v-m_1 \choose m_2}{d_j \choose m_2} m_2!}{(\ell_n-1)\cdots (\ell_n -2(m_1+m_2)+1)}.
\end{eq}
Moreover, for any $a>0$, $x/y\leq (x+a)/(y+a)$, iff $x\leq y$, 
which yields that $(d_v-m_1-k)/ (\ell_n-2m_1-k-1) \leq (d_v-k)/(\ell_n-k-1)$.
Therefore,
\begin{eq}
\frac{\PR(e_{iv} = m_1, e_{vj} = m_2)}{\PR(e_{iv} = m_1)\PR(e_{vj} = m_2)} = \frac{{d_v-m_1 \choose m_2}}{{d_v \choose m_2}} \prod_{i=1}^m \frac{\ell_n-2m_1+1-2i}{\ell_n-2i+1}\leq 1.
\end{eq}
 The proof of Lemma~\ref{c5:lem:edge-occupancy-negative-correlation} thus follows.
\end{proof}
\begin{lemma}\label{c5:crude-estimate}
For any $i\neq j$, $\PR(e_{ij}\geq 1) \leq \min\{\frac{d_id_j}{\ell_n-1},1\} $.
\end{lemma}
\begin{proof}Obviously, $\PR(e_{ij}\geq 1) \leq 1$, and by Markov's inequality, 
\begin{eq}
\PR(e_{ij}\geq 1) \leq \E[e_{ij}] = \frac{d_id_j}{\ell_n-1}, 
\end{eq}and the proof follows.
\end{proof}
Define $p_{ij}(2)$ to be the probability that there exists some intermediate vertex $v$ such that $\{i,v\}$ and $\{v,j\}$ form edges in $\ECMp$. 
We have seen that the bound in Lemma~\ref{c5:lem:2-nbd-prob} forms the bedrock for all the error estimates for $\mathrm{GRG}_n(\bw,p_c(\lambda))$, while using the first moment method. 
The next lemma provides an analogue of Lemma~\ref{c5:lem:2-nbd-prob} for $\mathrm{ECM}_n(\bld{d},p_c(\lambda))$:
\begin{lemma}\label{c5:lem:2-hop-connection-ECM}
$p_{ij}(2)\leq \frac{C\lambda^2}{(i\wedge j)^{1-\alpha}(i\vee j)^\alpha}$, $\forall i,j\in [n]$ for some constant $C>0$. 
\end{lemma}
\begin{proof}
Without loss of generality, we assume that $d_i > d_j$, i.e., $i< j$. Note that 
\begin{eq}
p_{ij}(2) \leq p_c^2 \sum_{v\in [n]} \PR(e_{iv}\geq 1, e_{iv}\geq 1) \leq p_c^2 \sum_{v\in [n]} \PR(e_{iv}\geq 1)\PR( e_{iv}\geq 1),
\end{eq}where we have used Lemma~\ref{c5:lem:edge-occupancy-negative-correlation}.
Let us split the sum in three parts with $\{v: d_id_v\leq \ell_n\}$, $\{v:d_jd_v\leq \ell_n<d_id_v\}$, and $\{v:d_jd_v> \ell_n\}$, and bound $\PR(e_{iv\geq 1})\PR(e_{vj}\geq 1)$ from above by $4d_id_v^2d_j/\ell_n^2$, $2d_vd_j/\ell_n$ and $1$ respectively on those sets. 
The rest of the argument is identical to Lemma~\ref{c5:lem:2-nbd-prob}.
\end{proof}
For the proof of the results for $\mathrm{ECM}_n(\bld{d},p_c(\lambda))$, we replace the $w_i$'s by the $d_i$'s in all the notations in Section~\ref{c5:sec:proof-GRG}.
The inequality in Lemma~\ref{c5:lem:2-hop-connection-ECM} again establishes that the two-hop connection probabilities are upper-bounded by the connection probabilities of the preferential attachment model. 
Thus, we can use Lemma~\ref{c5:lem:2-hop-connection-ECM} together with Lemma~\ref{c5:lem:edge-occupancy-negative-correlation} to get an identical estimate as \eqref{c5:eq:geometric-decay}, and therefore Proposition~\ref{c5:prop:tail-particular-component} holds.
Let us now identify the connectivity structure between the hubs to establish an analog of Proposition~\ref{c5:prop:convergence-K-nbd}.

We use the notation $V_{\sss L} := \{i:d_i>\delta n^{\alpha}\}$,  $V_{\sss SS} = \{i: d_i < \delta n^{\rho}\}$, $V_{\sss SI} = \{i: \delta n^{\rho}\leq d_i \leq \delta^{-1} n^{\rho}\}$ and $V_{\sss SL} = \{i: d_i > \delta^{-1} n^{\rho}\}$. 
Consider the following multigraph $G_{n,\delta}$ on the vertex set $V_L$, where the number of edges $X_{ij}$ between $i$ and $j$ is number of distinct $v\in [n]$ such that both $(i,v)$ and $(v,j)$ are edges of $\ECMp$.
Note that, for any $i\neq j$,
\begin{eq}\label{c5:eq:distn-Xij-ECM}
X_{ij} = \sum_{v\neq i,j} \ber\big(p_c^2 \PR(e_{iv}\geq 1,\ e_{vj}\geq 1)\big).
\end{eq}
We can split the above sum in three parts with $v\in V_{SS}$, $v\in V_{SI}$, and $v\in V_{SL}$ and denote them by $(I)$, $(II)$, and $(III)$ respectively.
Now, using Lemma~\ref{c5:lem:order-estimates} and Lemma~\ref{c5:lem:edge-occupancy-negative-correlation}, 
\begin{eq} \label{c5:hub-connection-term-1-ECM}
\E[(I)] \leq \frac{d_id_jp_c^2}{\ell_n^2} \sum_{v:d_v<\delta n^{\rho}}d_v^2 \leq C_0\theta_1^2 \delta^{3-\tau}n^{2\alpha - 2+1+(3-\tau)\rho}p_c^2 \leq C_1\delta^{3-\tau},
\end{eq}
\begin{eq} \label{c5:hub-connection-term-3-ECM}
(III) \leq p_c^2 \times \#\{v:d_v>\delta^{-1}n^{\rho}\} = \delta^{\tau-1}.
\end{eq}
Using the above calculation of $(I)$ and $(III)$ and Markov's inequality, it follows that
\begin{eq}\label{c5:eq:Xij-approximation-ECM}
X_{ij} &= \sum_{v: d_v\in [\delta n^{\rho},\delta^{-1}n^{\rho}]} \ber\big(p_c^2 \PR(e_{iv}\geq 1,\  e_{vj}\geq 1)\big) + \err_{\sss \PR}(\delta,n) \\
&= X_{ij}(\delta) + \err_{\sss \PR}(\delta,n).
\end{eq} 
Define the quantity 
\begin{eq}\label{c5:eq:lambda-ij-delta-ECM}
\lambda_{ij}(\delta) = p_c^2\sum_{v: d_v\in [\delta n^{\rho},\delta^{-1}n^{\rho}]} \Big(1-\e^{-\frac{d_id_v}{\ell_n}}\Big)\Big(1-\e^{-\frac{d_id_v}{\ell_n}}\Big).
\end{eq}

\begin{lemma}\label{c5:lem:poisson-approximation-ECM}For any $\delta\in (0,1)$, and $i,j \geq 1$
\begin{equation}
\lim_{n\to\infty}\dTV\big(X_{ij}(\delta), \poi(\lambda_{ij}(\delta))\big) =0.
\end{equation}
\end{lemma}
\begin{proof}
Let us first estimate the probabilities of the Bernoulli random variables in \eqref{c5:eq:Xij-approximation-ECM}. 
Firstly, note that using \cite[(36)]{HHLS17}, for any $v\neq i$
\begin{eq}\label{c5:ECM-connection-estimates}
\Big|\PR(e_{iv}=0) - \e^{-\frac{d_id_v}{\ell_n}} \Big| &\leq \frac{d_id_v^2}{(\ell_n-d_1)^2}, \\
\Big|\PR(e_{iv}=0,e_{vj} = 0) - \e^{-\frac{(d_i+d_j)d_v}{\ell_n}} &\Big| \leq \frac{(d_i+d_j)d_v^2}{(\ell_n-d_1)^2}.
\end{eq}
The second inequality also follows from \cite[(36)]{HHLS17}, because if we  merge two vertices $i$ and~$j$ into one single vertex, then $\PR(e_{iv}=0,e_{vj} = 0)$ becomes the probability that the merged vertex does not have an edge with $v$, and one can use the first bound to deduce the second.
Now, using Lemma~\ref{c5:lem:order-estimates},  
\begin{eq}
\frac{d_i}{(\ell_n-d_1)^2}\sum_{v\leq \delta^{-1}n^{\rho}}d_v^2 \leq Cn^{-\rho(\tau-2)} \to 0, 
\end{eq}and therefore one can conclude that
\begin{gather*}\label{c5:eq:bernoulli-prob-ECM}
E_1:= p_c^2\sum_{v: d_v\in [\delta n^{\rho},\delta^{-1}n^{\rho}]} \Big|\PR(e_{iv}\geq 1, \ e_{vj}\geq 1) -  \Big(1-\e^{-\frac{d_id_v}{\ell_n}}\Big)\Big(1-\e^{-\frac{d_id_v}{\ell_n}}\Big) \Big|
\end{gather*}which goes to zero as $n\to\infty$.
We will now use  Stein's method for convergence for sums of negatively correlated Bernoulli random variables \cite[Theorem 6.24]{JLR00}, which states that for $I_i\sim\mathrm{Bernoulli}(p_i)$ with $\mathrm{cov}(I_i,I_j)<0$, for all $i\neq j$, then 
\begin{eq}
\dTV \bigg(\sum_{i} I_i, \mathrm{Poi} \Big(\sum_{i}p_i \Big) \bigg) \leq \max_i p_i.
\end{eq} 
Note that $\ind{e_{iv} \geq 1,e_{vj} \geq 1}$ and $\ind{e_{iv'} \geq 1,e_{v'j} \geq 1}$ are negatively correlated for $v\neq v'$ which can be established using similar arguments as Lemma~\ref{c5:lem:edge-occupancy-negative-correlation}.
Therefore
\begin{eq}
&\dTV\big(X_{ij}(\delta), \poi(\lambda_{ij}(\delta))\big) \\
&\leq p_c^2E_1 + p_c^2\max_{v: w_v\in [\delta n^{\rho},\delta^{-1}n^{\rho}]} \Big(1-\e^{-\frac{d_id_v}{\ell_n}}\Big)\Big(1-\e^{-\frac{d_id_v}{\ell_n}}\Big) = O(p_c^2),
\end{eq}and the proof follows.
\end{proof}
\begin{lemma}\label{c5:lem:asymptotics-lambda-ijd-ECM}For any fixed $i,j \geq 1$
\begin{equation}
\lim_{\delta\to 0} \lim_{n\to\infty} \lambda_{ij}(\delta) = \lambda^2\int_{0}^{\infty} \Big(1-\e^{-\frac{\theta_ix^{-\alpha}}{\mu}}\Big)\Big(1-\e^{-\frac{\theta_jx^{-\alpha}}{\mu}}\Big) \dif x
\end{equation}
\end{lemma}
\begin{proof}
Using \eqref{c5:eq:lambda-ij-delta-ECM}, the proof is identical to Lemma~\ref{c5:lem:asymptotics-lambda-ijd}.
\end{proof}
\noindent For the asymptotic independence of the two-hop connections between the hubs, we need an estimate of the joint connection probabilities between distinct vertices in the configuration model, as given in the following lemma: 
\begin{lemma}\label{c5:lem:CM-joint-connection-estimate}
Consider four distinct vertices $i$, $u$, $k$ and $v$ such that $d_u,d_v\in [a n^{\rho}, b n^{\rho}]$ and $d_i,d_j \in [a n^{\alpha}, b n^{\alpha}]$, for some $a,b>0$. 
Then, as $n\to\infty$,
\begin{equation}
\Big|\PR(e_{iu}=0,e_{kv}=0) - \e^{- \frac{d_id_u}{\ell_n}}\e^{- \frac{d_kd_v}{\ell_n}} \Big| = O\bigg(\frac{d_id_u^2+d_kd_v^2}{(\ell_n-2d_1)^2}\bigg).
\end{equation}
\end{lemma}
\begin{proof}
We first sequentially pair the half-edges of $u$, and then in the next step pair the half-edges of $v$. 
Let $\cA_n$ denote the event that after the first stage of pairing $u$ does not create more than $n^{\varepsilon}$ edges with $k$ or $v$, where $0<\varepsilon<\rho$.
Firstly, note that $\PR(\cA_n^c)$ is exponentially small in $n$. 
Indeed, using \eqref{c5:eq:edge-connection-prob}, the probability that $u$ and  $k$ share at least $n^{\varepsilon}$ edges is at most 
\begin{eq}
\sum_{m>n^{\varepsilon}} \bigg(\frac{d_k d_u}{\ell_n - 2d_1}\bigg)^{m_1} \frac{1}{m_1!} \leq C \e^{-C n^{\varepsilon}},
\end{eq}where we have used the fact that $d_kd_u =O(\ell_n)$, and the last bound follows from the tail probabilities of a Poisson distribution.
A similar bound holds for the connections between $u$ and $v$ as well, and therefore, 
\begin{equation}\label{c5:eq:exponential-small-prob}
\PR(\cA_n^c) \leq C\e^{-cn^{\varepsilon}}.
\end{equation}
Now, using \eqref{c5:ECM-connection-estimates} and \eqref{c5:eq:exponential-small-prob}
\begin{eq}
\PR(\cA_n\cap \{e_{iu} = 0\})  = \e^{ - \frac{d_i d_u}{\ell_n}} + O\bigg(\frac{d_id_u^2}{(\ell_n-2d_1)^2}\bigg),
\end{eq}
and considering the second step of pairing of the remaining half-edges of $v$
\begin{eq}
\PR(e_{kv} = 0\vert \cA_n\cap \{e_{iu} = 0\} ) = \e^{ - \frac{d_k d_v}{\ell_n}} + O\bigg(\frac{d_kd_v^2}{(\ell_n-2d_1)^2}\bigg).
\end{eq}
Thus the proof follows.
\end{proof}
\begin{lemma} \label{c5:lem:symptotic-independence-ECM}For any $\delta\in (0,1)$, the collection of random variables $(X_{ij}(\delta))_{i,j\in V_L}$ is asymptotically independent. 
\end{lemma}
\begin{proof}
Consider vertices $i,j,k,l$ from $V_L$.
Let $\mathcal{E}_{iuj}$ denote the event that $\{e_{iu}\geq 1,  e_{uj}\geq 1\}$.
Suppose that we can show as $n\to\infty$
\begin{eq}\label{c5:eq:need-estimate-joint-conv}
E:=p_c^2\sum_{i,j,k.l\in V_L}\sum_{ d_u,d_v\in [\delta n^{\rho},\delta^{-1}n^{\rho}]} \Big|\PR(\mathcal{E}_{iuj} \cap\mathcal{E}_{kvl}) -  \PR(\mathcal{E}_{iuj} )\PR(\mathcal{E}_{kvl} )\Big|=o(1).
\end{eq} 
Then, with high probability, we can couple all the Bernoulli random variable in $(X_{ij}(\delta))_{i,j\in V_L}$ with an independent collection of Bernoulli random variables, and then the proof will be complete. 
Suppose that $i,j,k,l$ are distinct.
To estimate $\PR(\mathcal{E}_{iuj} \cap\mathcal{E}_{kvl})$, we use Lemma~\ref{c5:lem:CM-joint-connection-estimate}.
Indeed, if we merge vertices $i$ with $j$ and $k$ with $l$, and denote them by $v_1$ and $v_2$ respectively, then the required probability is the same as the probability that $(v_1,u)$ and $(v_2,v)$ share an edge. 
Therefore, Lemma~\ref{c5:lem:CM-joint-connection-estimate} implies that $E$ defined by \eqref{c5:eq:need-estimate-joint-conv} is of the order
\begin{eq}
p_c^2&\sum_{u,v:\ d_u,d_v\in [\delta n^{\rho},\delta^{-1}n^{\rho}]} \frac{(d_i+d_j)d_u^2+(d_k+d_l)d_v^2}{(\ell_n-2d_1)^2} \\
&\leq p_c^2\frac{C n^{\alpha}}{\ell_n^2}n^{1-(\tau-1)\rho} n^{1+\rho(3-\tau)} = Cn^{-\rho (\tau-2)}\to 0.
\end{eq} 
Since $V_L$ is a finite collection, this proves \eqref{c5:eq:need-estimate-joint-conv} on the partial sum with $i,j,k,l$ being distinct.
The case where $i = k$ can be dealt with similarly using \eqref{c5:ECM-connection-estimates}, and we do not repeat the argument again.
This completes the proof of Lemma~\ref{c5:lem:symptotic-independence-ECM}.
\end{proof}

\section{Conclusion}
In this chapter, we have provided to the best of our knowledge the first mathematically rigorous analysis for the critical window for random graphs that satisfy a power law distribution with exponent $\tau\in(2,3)$.
The main surprise is that the critical window changes depending on whether the network is constrained to have single-edges between the vertices. 
We identify the critical window and the scaling limits for the component sizes within the critical window. 
The main technical obstacle under the single-edge constraint is that the exploration process approach does not work and we had to resort to path counting techniques to identify the primary contributions on the component sizes. 
For this reason, the proof does not work for sufficiently large $\lambda$ values, which we leave as an open question.

%
%
%
%
%
 
\cleardoublepage

\chapter[Open problems]{Open problems}
\label{chap:open}
%
%
%
%
%
%

In the final chapter of this thesis, we discuss several related research questions, which are open to the best of our knowledge.

\section{Component sizes and complexity.}
\subsection{Finite and infinite third moment cases}
\paragraph*{Tail bound on the component sizes.} 
While the scaling limit results yield an exact asymptotic distribution of the largest component sizes, these results do not give any explicit probability bounds for large, but finite $n$. 
It is often interesting to obtain explicit bounds for the probabilities of the events involving ``tails'' such as the largest component is quite large ($\{|\sC_{\sss (1)}(p_c(\lambda))| > An^{\rho}\}$ for large $A>0$), or it is quite small  ($\{|\sC_{\sss (1)}(p_c(\lambda))| \leq  \delta n^{\rho}\}$ for small $\delta >0$), where $\rho$ is the scaling exponent for the component sizes.  
For critical Erd\H{o}s-R\'enyi random graphs with $p = 1/n(1+\lambda n^{-1/3})$, Nachmias and Peres \cite{NP10a} showed that 
\begin{eq}\label{eq:moderate-dev-bounds-1}
\PR(|\sC_{\sss (1)}(p_c(\lambda))| > An^{2/3}) \leq \frac{c_1}{A}\e^{-c_2 A^3}, 
\end{eq} 
\begin{eq}
\label{eq:moderate-dev-bounds-2}
\PR(|\sC_{\sss (1)}(p_c(\lambda))| \leq \delta n^{2/3}) \leq c_3 \delta^{3/5}, 
\end{eq}
for any $A>A_0$, $\delta \in (0,\delta_0)$, and $n\geq n_0$, where $c_1,c_2,c_3, A_0,\delta_0>0$ and $n_0\geq 1$ can be calculated explicitly. See also \cite{Pit01} and \cite[Corollary 19]{SS06}.
 In the context of quantum random graphs \cite{DLV14}, the bound in \eqref{eq:moderate-dev-bounds-1} was proven to be $A^{-3/2}$, with the exponential term missing, while the bound for \eqref{eq:moderate-dev-bounds-2} involves $\delta^{3/5}$.
 Even for Erd\H{o}s-R\'enyi random graphs, whether the bounds in \eqref{eq:moderate-dev-bounds-1} and \eqref{eq:moderate-dev-bounds-2} are optimal is not known to the best of our knowledge.
 Using results from \cite{Pit01} (see also \cite[(1.2)]{HKL14}), \cite[Corollary 2]{A97} and Portmanteau theorem, we can obtain a lower bound on the probability \eqref{eq:moderate-dev-bounds-1} of $c_1A^{-3/2}\e^{-c_2A^3}$, which differs from \eqref{eq:moderate-dev-bounds-1} in the polynomial term.

There are several challenges in order to derive such estimates for random graphs with general moment assumptions on the degrees such as Chapter~\ref{chap:thirdmoment}.
Specifically, the techniques involving Chernoff bound in \cite[Section 4]{NP10a} does not work for general degree distributions.
Bounds like \eqref{eq:moderate-dev-bounds-1} and \eqref{eq:moderate-dev-bounds-2} in the $\tau\in (3,4)$ case has never been studied. 
The properties of the scaling limit was studied in \cite{AHKL16,HKL14} for $\tau\in(3,4)$.
Using the results from \cite[Theorem 1.6]{HKL14}, Theorem~\ref{c2thm:spls} and Portmanteau theorem, the probability in \eqref{eq:moderate-dev-bounds-1} is bounded from below by $c_1A^{-(\tau-1)/2}\e^{-c_2A^{\tau-1}}$.
But an upper-bound with the same exponential term could also possibly be proved leading to the following conjecture:
\begin{conjecture}\label{c6:conj:tail} Under Assumption~\ref{c2assumption1}, there exists $A_0>0$, $n_{0}\geq 1$ and constant $c_1, c_2$ such that for all $A>A_0$ and $n\geq n_0$
\begin{eq}\label{eq:moderate-dev-bounds-3}
\PR(|\sC_{\sss (1)}(p_c(\lambda))| > A b_n) \leq \frac{c_1}{p(A)}\e^{-c_2 A^{\tau-1}}, 
\end{eq} 
for some polynomial $p$.
\end{conjecture}
We do not have a good guess about the bound of the form \eqref{eq:moderate-dev-bounds-2} for $\tau\in (3,4)$.
It is further interesting to learn about these tail probabilities when $A=A_n \to\infty$ or $\delta=\delta_n\to 0$. 
For $A_n \ll n^{1/12}$, suitable bounds were derived for Erd\H{o}s-R\'enyi random graphs in \cite[Proposition 3.1]{RS16}. 
In Chapter~\ref{chap:mspace-GHP}, we consider the $A_n \to\infty$ case (see Theorem~\ref{c4:lem:volume-large-deviation}), but we did not focus on obtaining the optimal bound, which is an interesting question.

%
%
%
%
%


\paragraph*{Comparison to Joseph's scaling limit.}
As observed in Section~\ref{c2sec:iid-degrees}, Assumption~\ref{c2assumption1} is satisfied almost surely when the degrees are an iid sample from a power-law distribution with exponent $\tau\in (3,4)$. 
Thus, conditionally on the observed degree sequence, the exploration process converges to the process 
\begin{eq}\label{op-scaling-limit-1}
S_{\infty,1}(t): = \sum_{i=1}^\infty \frac{C_{\sss F}}{\Gamma_i^\alpha} \bigg(\mathbf{1}_{\big\{X_i \leq \frac{s\mu \Gamma_i^\alpha}{C_{\sss F}}\big\}} - \frac{C_{\sss F}}{\mu \Gamma_i^\alpha}\int_{\Gamma_{i-1}}^{\Gamma_i} u^{-2\alpha}\dif u\bigg),
\end{eq}
where $(X_j)_{j\geq 1}$ and  $\Gamma_i = \sum_{j\leq i}E_j$ where $(E_j)_{j\geq 1}$ is an independent collections of i.i.d unit rate exponential random variables; see Theorem~\ref{c2thm::convegence::exploration_process}.
On the other hand, the iid degree setting has been studied in \cite{Jo10,CG17}, where the scaling limit of the exploration process turns out to be different. 
More precisely, the scaling limit is given by 
\begin{eq}\label{op-scaling-limit-2}
S_{\infty,2}(t) = Y(t)+A(t),
\end{eq}
where
\begin{eq}
A(t) = - \frac{C_{\sss F} \Gamma(4-\tau)}{(\tau-3)(\tau-2)\mu^{\tau-2}} t^{\tau-2},
\end{eq}
and $Y(t)$ is the unique process with independent increments such that for every $t\geq 0$ and $u\in \R$,
\begin{eq}
\E[\exp(i u Y(t))] = \exp\bigg(\int_0^t\dif s \int_0^\infty \dif x (\e^{iux} - 1 -iux)\frac{C_{\sss F}}{\mu x^{\tau-1}}\e^{-xs/\mu}\bigg).
\end{eq}
Since the second process is the limit of the same exploration process averaged out over the degrees, this indirectly implies that the law of $\mathbf{S}_{\infty,2}$ is the same as the law of $\mathbf{S}_{\infty,1}$, averaged out over the $\Gamma$-values (although we use a different exploration process from \cite{Jo10}, the
fact that the component sizes are huge compared to the number of cycles in a component, one can prove Theorem~\ref{c2thm::convegence::exploration_process} for the exploration process in \cite{Jo10} also).
This is remarkable given the vastly different descriptions of the scaling limits \eqref{op-scaling-limit-1} and \eqref{op-scaling-limit-2}. 
For example,  the martingale part of \eqref{op-scaling-limit-1} does not have independent increments due to thinning of the Poisson processes.
However, after averaging out over $\Gamma$-values, the dependence goes away.
It may be worthwhile investigating whether there is a direct approach to show that  $\mathbf{S}_{\infty,1}$, after averaged out over the $\Gamma$-values, 
yields the same law as $\mathbf{S}_{\infty,2}$.

\paragraph*{Joint convergence over the critical window.}
While studying the joint convergence over the critical window in Theorems \ref{c1:thm_multiple_convergence} and~\ref{c2thm:mul:conv}, we considered finite dimensional convergence. 
It will be interesting to show that the process $(\mathbf{Z}_n(\lambda))_{\lambda\in\R}$ converges in $\mathbb{D}(\R,\Unot)$, where we recall that $\mathbf{Z}_n(\lambda)$ is the vector of rescaled component sizes and the surplus edges and $\mathbb{D}(\R,\Unot)$ denotes the set of c\`adl\`ag functions equipped with the Skorohod $J_1$-topology.
The proof will follow if one can verify a suitable tightness criterion $\mathbb{D}(\R,\Unot)$-valued stochastic processes, but we were unable to find a suitable tightness criterion.

\paragraph*{Dynamically evolving critical random graphs.}
In \cite{RS16}, Roberts and \c{S}eng\"{u}l considered a dynamically evolving version of critical Erd\H{o}s-R\'enyi random graphs.
The dynamic graph process $(G_t)_{t\geq 0}$ starts with $G_0$ which is distributed as $\mathrm{ERRG}_n(1/n)$, and each pair $\{u,v\}$ is equipped with an independent rate-one Poisson process $\mathcal{N}_{uv}$.
At each event time of $\mathcal{N}_{uv}$, an edge is resampled according to an independent Bernoulli$(1/n)$ random variables. 
Then, for each fixed $t\geq 0$, $G_t$ is distributed as $\mathrm{ERRG}_n(1/n)$. 
Let $\sC_{\sss (1)}(t)$ denote the largest component of $G_t$.
It was shown in \cite{RS16} that $M_n:=\sup_{t\in [0,1]} |\mathscr{C}_{\sss (1)} (t)| \leq \beta n^{2/3}(\log (n))^{1/3}$ with high probability for $\beta < 2/3^{2/3}$.
There are several further interesting questions that arise for this dynamic graph process. 
\begin{enumerate}
\item Does $M_n/n^{2/3}(\log(n))^{1/3}$ converge in probability  to some $\beta_0 >0$?
\item What is the behavior of $I_n:=\inf_{t\in [0,1]} |\mathscr{C}_{\sss (1)} (t)|$? 
\item What happens in the heavy-tailed universality class of the multiplicative coalescent regime? How does the exponent of $\log(n)$ change depending on the power-law exponent $\tau$ of the degree distribution.
\end{enumerate}
Questions 1 and 2 are discussed in \cite{RS16}. 
While $\beta_0$ was conjectured to be $2/3^{1/3}$, not much was known for question 2. 
For question 3, the tail asymptotics in \eqref{eq:moderate-dev-bounds-1} is intimately related to the asymptotics of $M_n$, as noted in \cite{RS16}. 
For each $t\in [0,1]$, $\sC_{\sss (1)}(t)$ satisfies \eqref{eq:moderate-dev-bounds-1}, and therefore a simple union bound yields an upperbound of $M_n = O(n^{2/3}(\log(n))^{1/3})$ with high probability
(proving the lower bound is considerably difficult which was accomplished in \cite{RS16}).
Following the prediction of tail bounds in Conjecture~\ref{c6:conj:tail}, this leads us to the the conjecture below:
\begin{conjecture}\label{conj:dyn-ERRG-max}
Under Assumption~\ref{c2assumption1}, 
$M_n = \Theta_{\sss\PR}(n^{\rho} (\log(n))^{1/(\tau-1)})$.
\end{conjecture}
%
%

%
%
%

\paragraph*{Simulation guarantee for sample paths.}
The scaling limits for the component sizes are described by largest excursions of certain stochastic processes with negative drift. 
It is difficult to generate a sample from this distribution due to the lack of availability of the precise distribution function.
It will be interesting in the stochastic simulation literature to develop techniques for generating a sample from this distribution, and obtain exact error bounds if the simulation method is approximate. 
For the heavy-tailed scaling limits, simulating sample paths of the thinned L\'evy process is not standard.
A natural strategy could be to truncate the sum in \eqref{c2defn::limiting::process} upto first $K$ terms for a large $K$.
However, there is a more accurate approach using the techniques in \cite{AR01}.
Here, one can approximate the small jumps by a Brownian motion using the following theorem:
Recall that $\mathcal{I}_i(t)=\ind{\xi_i\leq t}$, $\xi_i\sim\mathrm{Exp}(\theta_i)$ (let $\mu=1$) and define $$E_K(t)=\sum_{i=K+1}^{\infty}\theta_i(\mathcal{I}_i(t)-\theta_it),\quad \sigma_K^2=\sum_{i=K+1}^\infty \theta_i^3.$$
\begin{theorem}[{\cite{DHM15}}]\label{th-conv-thinned}Suppose that 
$ \frac{\max_{i>K}\theta_i}{\sum_{i=K+1}^{\infty}\theta_i^3}\to 0.$
 Then, as $K\to\infty$,
 \begin{equation}
  (\sigma_K^{-1}E_K(t))_{t\geq 0}\dto \bld{W},
 \end{equation}where $\bld{W}$ is a standard Brownian motion.
\end{theorem}
\noindent See Appendix~\ref{sec:appendix-OP-1} for a proof.
Note that if we take $\theta_i=C_Fi^{-\alpha}$ for some $\alpha\in (1/3,1/2)$, we have $\sum_{i=K+1}^\infty \theta_i^3=\Theta(K^{1-3\alpha})$. Therefore,
\begin{equation}
 \frac{\max_{i>K}\theta_i}{\sum_{i=K+1}^{\infty}\theta_i^3}=\Theta(K^{-\alpha-1+3\alpha})=\Theta(K^{2\alpha-1})\to 0.
\end{equation}Thus, the assumption of Theorem~\ref{th-conv-thinned} is satisfied.
However, we did not pursue the question of simulation guarantees further.
\paragraph*{Independent proof for the scaling limit of diameter.} In Chapter~\ref{chap:mspace-GHP}, we have seen that the largest components converge as measured metric spaces under the Gromov-Hausdorff-Prokhorov topology yields the convergence of the diameters of these components.
This is yields that $\mathrm{diam}(\sC_{\sss (1)}(p_c(\lambda)))$ converges in distribution to some random variable. 
The above approach of proving the convergence of diameters is indirect and considerably difficult.
Till date there is no direct approach available to show the convergence of $\mathrm{diam}(\sC_{\sss (1)}(p_c(\lambda)))$ even for the Erd\H{o}s-R\'enyi random graphs. 
Only some bounds were derived in \cite{NP08} establishing the tightness of $\mathrm{diam}(\sC_{\sss (1)}(p_c(\lambda)))$ and $(\mathrm{diam}(\sC_{\sss (1)}(p_c(\lambda))))^{-1}$.
A direct proof is expected to yield a simpler expression for the diameter of the limiting metric spaces in \cite{ABG09,BHS15}, and it is expected to require novel techniques as well.

\paragraph*{Critical behavior on random geometric graphs.}
A random geometric graph is obtained by throwing $n$ points uniformly at random 
in the $d$-dimensional box $[0,1]^d$, and creating an edge between two points if their  euclidean distance is at most $\lambda n^{-d}$. 
Random geometric graphs are known to exhibit phase transition as $\lambda$ increases \cite{P2003}.
The phase transition has also been studied for random geometric graphs on hyperbolic spaces \cite{BFM15}.
However, analyzing the critical behavior on random geometric graphs is an open question. 
The inherent structure of these graphs are fundamentally different than the random graph models that do not depend on an underlying geometry. 
For example, the probability that a random vertex is involved in a clique of size $k$ is bounded away from zero for each fixed $k\geq 1$, showing that random geometric graphs cannot be approximated by a branching process locally. 
Thus, the critical components are not expected to have $O(1)$ many surplus edges anymore. 
Technically, it is challenging to deal with the exploration process since the drift and the quadratic variation terms depend on the area covered by the spheres centered at the active vertices, which is difficult to track. 

\paragraph*{Concentration of total size and total number of large components.} 
Let $\Xi_n$ denote the point process $\{n^{-\rho}|\sC_{\sss (i)}(p_c(\lambda))|\}_{i\geq 1}$, and let $\Xi$ be the point process $\{|\gamma_i|\}_{i\geq 1}$, where $|\gamma_i|$ denotes the scaling limit of $|\sC_{\sss (i)}(p_c(\lambda))|$ in Theorems~\ref{c1:thm_main}, or~\ref{c2thm::conv:component:size}.
Fix $\varepsilon>0$.
Then the convergence of the exploration processes imply that as $n\to\infty$
\begin{gather*}
Z_{n,\varepsilon}:=\int_{\varepsilon}^\infty x\ \Xi_n(\dif x) \dto Z_{\varepsilon}:=\int_{\varepsilon}^\infty x\ \Xi(\dif x)\\
W_{n,\varepsilon}:=  \Xi_n([\varepsilon,\infty)) \dto  W_{\varepsilon}:=  \Xi([\varepsilon,\infty));
\end{gather*} see \cite[Proposition 1.4]{JS07}.
As $\varepsilon\to 0$, $Z_{\varepsilon}$ gives the \emph{total mass} of the largest components, and $W_{\varepsilon}$ gives the number of largest components. 
Thus it is desirable to understand the asymptotics of $Z_{\varepsilon}$ and $W_{\varepsilon}$ as $\varepsilon\to 0$.
Note that $Z_{0} = \infty$ and $W_{0} =\infty$, and thus by monotone convergence theorem, $Z_{\varepsilon} \xrightarrow{\sss \PR} \infty$ and 
$W_{\varepsilon} \xrightarrow{\sss \PR} \infty$, as $\varepsilon\to 0$.
For the Erd\H{o}s-R\'enyi universality class, Janson and Spencer \cite{JS07} showed that as $\varepsilon\to 0$
\begin{equation}\label{eq:asymptotics-eps-to-zero-op}
\varepsilon^{1/2} Z_{\varepsilon} \pto \sqrt{\frac{2}{\pi}} \quad \text{and} \quad  \varepsilon^{3/2} W_{\varepsilon} \pto \sqrt{\frac{2}{9\pi}}.
\end{equation}
Thus, even if both $Z_{\varepsilon}$ and $W_{\varepsilon}$ are non-degenerate random variables, these concentrate as $\varepsilon\to 0$. 
The asymptotic normality is still an open question. 
Also, it will be interesting to derive the asymptotics \eqref{eq:asymptotics-eps-to-zero-op} for the general description of the multiplicative coalescent given in \cite{AL98}.
%


\section{Infinite second moment case} In this section, we state the open problems related to the critical behavior in the infinite third moment case.

\paragraph*{Barely super-critical regime and the large $\lambda$ case.}  
While the critical behavior was studied in detail for the configuration model in Chapter~\ref{chap:infinite-second}, the lack of an exploration process approach under the single-edge constraint limited our analysis to the barely subcritical regime and small values of $\lambda$ (i.e., $\lambda\in (0,\lambda_0)$ for some constant $\lambda_0$ which is independent of the model) within the critical window. 
In a future work, we wish to address the $\lambda>\lambda_0$ and the barely supercritical case. 
This will complete the analysis for the critical behavior of the component sizes within under the single-edge constraint.


\paragraph*{Novel evolution dynamics in the $\tau\in (2,3)$ case.}
In Chapters~\ref{chap:thirdmoment} and~\ref{chap:secondmoment}, we have seen that the evolution of the component sizes over the critical window is always guided by the multiplicative coalescent process, but apparently one would get a completely different coalescent process in the infinite second moment case, especially under the single-edge constraint.
In this case, components merge when the hubs get connected via some intermediate vertex. 
Now the evolution of the total weights of is not Markovian.
One has to keep track of all the hubs within components (which gives rise to an infinite dimensional vector) rather than some statistic of the components (like the total mass) to describe the process.
This gives rise to novel evolution dynamics in the context of critical random graphs.

\paragraph*{Bounds on the diameter in the infinite second moment case.}
Although the diameter of critical components is $O(1)$ for $\CM$, we did not derive such a result under the single-edge constraint. 
In fact it is possible that the diameters in the latter case is diverging to infinity.
We do not have a concrete intuition for this problem so far and it requires further investigation.

\paragraph*{Uniformly chosen graphs with given degrees.}
It will be interesting to study the critical behavior in the infinite second moment case for uniformly chosen graphs with given degrees (denoted by $\mathrm{UM}_n(\bld{d})$).
We expect the same scaling critical exponents as under the single edge-constraint, and the scaling limit is expected to be the same as critical percolation on generalized random graphs.  
The reason behind this is as follows:  
Suppose that $\bld{D}:=(D_i)_{i\in [n]}$ denotes the degree sequence of $\GRG$.
Conditionally on $\bld{D} = \bld{d}$, the distribution of $\GRG$ is same as that of $\mathrm{UM}_n(\bld{d})$. 
Now the ``core'' of the components consists of vertices with weight $\Theta(n^{\rho})$ and $\Theta (n^{\alpha})$. Call these special vertices.
One can probably use concentration arguments to show that $D_i \approx d_i$ for all special vertices.
Now if one can show that perturbing the degrees of the special vertices does not change the connection probabilities in $\mathrm{UM}_n(\bld{d})$ significantly, then it will be possible to show that, conditionally on $\bld{D} = \bld{d}$, the core for the connected components of $\GRG$ is the same graph as the core of the connected components of $\mathrm{UM}_n(\bld{d})$. However, formalizing this is not straightforward.

\section{Global structure.} 

\paragraph*{Joint convergence of metric spaces in $l^4$ topology.}
Let $\cM$ denote the space of measured compact metric spaces endowed with the  Gromov-Hausdorff-Prokhorov (GHP) topology.
In Chapters~\ref{chap:mspace} and~\ref{chap:mspace-GHP}, we have considered the product topology on the  $\cM^{\N}$ for the joint convergence of the components.
A stronger result was shown for Erd\H{o}s-R\'enyi random graphs in \cite{ABG09} under the metric defined below.
Define the metric $d_{l^4}$ on $\cM^{\N}$ by $$d_{l^4}(\bld{X},\bld{Y}) = \bigg(\sum_{i\geq 1}\big(d_{\sss \mathrm{GHP}}(X_i,Y_i)\big)^4\bigg)^{1/4},$$ where $d_{\sss\mathrm{GHP}}$ denotes the GHP distance.
It is desirable to extend the results in Chapters~\ref{chap:mspace} and~\ref{chap:mspace-GHP} under this stronger topology which requires suitable bounds on the diameter of small components. 


\paragraph*{Evolution as metric space-valued stochastic process.}
The augmented multiplicative coalescent process only tracks the evolution of the component sizes and the surplus edges as $\lambda$ increases over the critical window.
The metric structures of the largest connected components also evolve as $\lambda$ increases. 
It will be interesting to describe the evolution of the infinite dimensional measured metric spaces. 


\section{More challenges. }

\paragraph*{Critical behavior for general graphs.}
Studying the critical behavior for more general sequence of graphs is  an open direction.
Of course the phrase ``general graphs'' is too vague, and one must impose regularity conditions to see the critical behavior and scaling limits. 
For example, one may consider percolation on sequences of dense graphs (with $\Theta(n^{2})$ many edges) that converge in the so-called cut metric \cite{BCLSV12,BCLSV08} and impose some restrictions on the limiting graphon.
The critical value for the phase transition was identified in~\cite{BBCR10} under mild conditions, while the critical behavior is a completely open question.

\paragraph*{The minimum spanning tree problem.}
The study of critical percolation has experienced a renewed interest after a recent seminal work by Addario-Berry et al.~\cite{ABGM13}. 
They studied the limit as a metric space of the minimum spanning tree (MST) on a complete graph with iid edge weights under the GHP-topology.
Exploiting the relation between Kruskal's algorithm for generating MST, they showed that the MST can be \emph{approximated} by $\sC_{\sss (1)}(p_c(\lambda))$ after removing the cycles in a specified manner.
Using the results about the limit of $\sC_{\sss (1)}(p_c(\lambda))$ from~\cite{ABG09}, they could describe the metric structure of the MST on complete graph.
The results in~\cite{BHS15} and in Chapter~\ref{chap:mspace} forms a basis of studying MST in the heavy-tailed regime where the scaling limit is expected to be different than MST on complete graph.
However, the study of minimal spanning trees is an open question, even for the simple models like random regular graphs.

\begin{subappendices}
\section{Proof of Theorem~\ref{th-conv-thinned}} \label{sec:appendix-OP-1}
Note that $(E_K(t))_{t\geq 0}$ is a super-martingale with the Doob-Meyer decomposition (\cite[Section 2.1]{AL98})
\begin{equation}
 E_K(t)=M_K(t)+A_K(t),
\end{equation}where $A_K(t)=-\sum_{i=K+1}^{\infty}\theta_i^2(t-\xi_i)^+$ and the quadratic variation of $\bld{M}_K$ is given by 
\begin{equation}
 \langle M_K \rangle (t)=\sum_{i=K+1}^\infty \theta_i^3\min\{t,\xi_i\}.
\end{equation}
For any fixed $T>0$, we show that $\sigma_K^{-1}\sup_{t\leq T}|A_K(t)|\pto 0$ and use the martingale functional central limit theorem \cite[Theorem 2.1]{W07} to conclude the theorem. To see the first part, note that 
\begin{align*}
 \expt{(T-\xi_i)^+}\leq  T\prob{\xi_i\leq T}= T(1-\e^{-\theta_iT})\leq \theta_iT^2. 
\end{align*}
Noting that $A_K(\cdot)$ is negative and monotonically decreasing we have, for any $\varepsilon >0$,
\begin{equation}
 \begin{split}
  \PR\bigg(\sup_{t\leq T}|A_K(t)|>\varepsilon \sigma_K\bigg)&=\prob{|A_K(T)|>\varepsilon \sigma_K}\leq \frac{1}{\varepsilon \sigma_K}\expt{-A_K(T)}\\
  &\leq \frac{T^2\sigma_K^2}{\varepsilon \sigma_K}=\frac{T^2\sigma_K}{\varepsilon}\to 0.
 \end{split}
\end{equation} Let $J(M_K,T)$ denote the value of the maximum jump of $\bld{M}_K$ before time $T$. Now to show that the martingale parts converge to a Brownian motion, it is enough to show that, as $K\to\infty$,
\begin{enumerate}[(1)]
 \item \label{fclt-con-1} $\sigma_K^{-2}\langle M_K \rangle (t)\pto t$, for each fixed $t>0$,
 \item \label{fclt-con-2} $\sigma_K^{-1}\expt{|J(M_K,T)|}\to 0.$
\end{enumerate} Let us first verify \eqref{fclt-con-2}. Note that $J(M_K,T)=\theta_i^2$ if $\xi_i\leq T$ and $\xi_j>T$ for all $K<j<i$ (since $\theta_i$'s are non-increasing). Therefore,
\begin{equation}
 \sigma_K^{-1}\expt{|J(M_K,T)|}\leq \sigma_K^{-1}\sum_{i=K+1}^{\infty}\theta_i^2\prob{\xi_i\leq T}\leq T\sigma_K\to 0.
\end{equation}
To see \eqref{fclt-con-1} we will use Chebyshev's inequality. Note that 
\begin{equation}
 \expt{\min\{t,\xi_i\}}=\int_0^t\theta_ix\e^{-\theta_ix}\dif x+t\prob{\xi_i>t}=\frac{1-\e^{-\theta_it}}{\theta_i}.
\end{equation}Using the fact that $x-x^2/2\leq 1-\e^{-x}\leq x$, we compute
\begin{equation}
 t- \frac{t^2}{2}\frac{\sum_{i=K+1}^{\infty}\theta_i^4}{\sum_{i=K+1}^{\infty}\theta_i^3}\leq \sigma_K^{-2}\expt{\langle M_K \rangle (t)}= \sigma_K^{-2} \sum_{i=K+1}^{\infty}\theta_i^2(1-\e^{-\theta_it})\leq t,
\end{equation}
and therefore,
\begin{equation}
 \left|\sigma_K^{-2}\expt{\langle M_K \rangle (t)}-t\right|\leq\frac{t^2}{2}\frac{\sum_{i=K+1}^{\infty}\theta_i^4}{\sum_{i=K+1}^{\infty}\theta_i^3}\leq  \frac{t^2}{2}\max_{i> K}\theta_i\to 0.
\end{equation} 
Moreover, since $\var{\min\{t,\xi_i\}}\leq Ct^2\theta_i^{-2}$ for some uniform constant $C>0$,
\begin{equation}
\mathrm{Var}\big(\sigma_K^{-2}\langle M_K \rangle (t)\big)=\sigma_K^{-4}\sum_{i=K+1}^{\infty} \theta_i^6 \var{\min\{t,\xi_i\}}\leq Ct^2\frac{\max_{i>K}\theta_i}{\sum_{i=K+1}^{\infty}\theta_i^3}\to 0,
\end{equation}where the last part follows by our assumption. Thus, \eqref{fclt-con-1} follows and the proof is complete. \qed

\end{subappendices}

%
%

%
\cleardoublepage

 \bibliographystyle{apalike} 
\bibliography{thesis}

\chapter*{Summary}
Random graphs have played an instrumental role in modelling real-world networks arising from the internet topology, social networks, or even protein-interaction networks within cells. Percolation, on the other hand, has been the fundamental model for understanding robustness and spread of epidemics on these networks. 
From a mathematical perspective, percolation is the simplest model that exhibits phase transition, and fascinating features are observed around the critical point.
In this thesis, we prove limit theorems about structural properties of the connected components obtained from percolation on random graphs at criticality. 
The results are obtained for random graphs with general degree sequence, and we identify different universality classes for the critical behavior based on moment assumptions on the degree distribution. 

In Chapter 1, we start with an introduction to this attractive branch of probability which has spurred interest among mathematicians for several decades, with many of the interesting questions being still open. 
We briefly review the history of the percolation phase transition on finite graphs, and describe the emerging literature for the critical behavior.
Subsequently, we describe our results from a high-level, and discuss the general proof ideas. 
Three types of fundamentally different critical behaviors are observed depending on whether degree distribution satisfies (a) a finite third moment condition, (b) an infinite third moment condition, (c) and an infinite second moment condition.
In all these regimes, we ask questions about the component sizes and structures of the critical components.
The goal of this chapter is to convey the main challenges in the upcoming chapters for each of the above regimes without going into the technical framework. 

In Chapter 2, we state and prove results about the component sizes and surplus edges when the degree distribution satisfies a finite third moment condition. 
The evolution of component sizes and surplus edges over the critical window is also shown to converge to the augmented multiplicative coalescent. 
The results show that only a finite third moment condition ensures that the critical behavior lies in the same universality class as classical homogeneous random graph models like Erd\H{o}s-R\'enyi random graph or random regular graph.

In Chapter 3, we investigate the infinite third moment case. 
In this setting, the critical behavior for component sizes and surplus edges turns out to be in a completely different universality class. 
The key difference lies in the fact that the asymptotics of high degree vertices play a pivotal role in describing the scaling limits. 
For example, if the highest degree vertex is deleted, the scaling limit changes in this regime which is in sharp contrast to the finite third moment case. 
The results in Chapters 2 and 3 observe both the possible scaling limits for multiplicative coalescent processes that was predicted by Aldous and Limic. 

In Chapter 4, we describe the global structure of components in the infinite third moment case. 
More precisely, one can view any connected graph as a metric space equipped with a measure, where the elements of the metric space given by the vertices, the metric given by the graph distance, and the measure proportional to the counting measure.
With all the above ingredients, the critical components can be viewed as a random element from the space of all complete metric spaces equipped with a measure. 
In this chapter we show that after rescaling the distances suitably, the largest critical components converge with respect to Gromov-weak topology.
%
These results yield joint convergence of several functionals related to distances within these connected components. 

In Chapter 5, we establish the global lower mass bound property. 
This property ensures that the convergence results in Chapter 4 could be strengthened to hold under the Gromov-Hausdorff-Prokhorov topology.
The primary outcome of the later stronger form of convergence is that it yields convergence of global distance related functionals like the diameter.

In Chapter 6, we investigate the case where the degree distribution has infinite second moment. 
Even defining the critical window for percolation is challenging in this case and all the questions related to the critical behavior in this regime were completely open question till date.
We initiate this study by identifying critical values and scaling limits of the component sizes. 
The striking observation that we make in this regime is that the critical exponents and the scaling limits depend crucially on the so-called single-edge constraint, i.e., the critical behavior for the configuration model and the erased configuration model are fundamentally different.
We also establish the uniqueness of the critical exponents by analyzing the barely sub/super-critical regimes.



In the final Chapter, we conclude with many open problems and future directions. 

The results in this thesis are strongest in terms of the topology of convergence and the results are proved under minimal assumptions. The results are expected to have potential impact on understanding spread of epidemics, minimum spanning trees on random networks with arbitrary degree sequence. 
The proof ideas are also robust and we hope that many of the core ideas would work for more many other random graph models.

\chapter*{About the author}
Souvik Dhara was born on May 9, 1991 in Kolkata, India. 
Souvik grew up in the southern part of Kolkata, and obtained his high school degree from Harinavi D.V.A.S.~High School.
During 2009 - 2012, he obtained his bachelor's degree with Statistics  major. 
In 2012, he joined Indian Statistical Institute in the master's program, and obtained his M.Stat degree with Mathematical Statistics and Probability specialization in 2014. 
In August 2014, he joined a PhD program in the  Eindhoven University of Technology under the supervision of Remco van der Hofstad and Johan van Leeuwaarden. 
His PhD project was part of the NETWORKS program funded by the Netherlands Organisation for Scientific Research (NWO).

Souvik's research interests lie in the intersection of probability theory and combinatorics, applied probability and operations research. 
During PhD, his primary aim has been understanding the interplay between the structural properties of networks and stochastic processes on them.
Souvik has explored and applied several recent concepts in both probability theory and combinatorics related to local and global structures of random networks.
His main contribution in the field is to derive limit laws for critical percolation on graphs with arbitrary degree distribution.
His findings is the topic of this PhD thesis.
At the same time, he has made contributions to combinatorics, and applied probability by analyzing graph limits and cut properties of random graphs, and providing asymptotic analysis of stochastic process arising from modern cloud computing systems and wireless networks.

In July 2018, Souvik joined as a Schramm fellow, which is offered jointly by Microsoft Research New England and MIT Mathematics. He will be at Microsoft Research during 2018 - 2019, and at MIT during 2019 - 2021.

\chapter*{Publications and preprints}

\begin{enumerate}[{[1]}]
\item Critical behavior of percolation on random graphs with given degree: A survey (2018+);
Souvik Dhara, Remco van der Hofstad, Johan S.H. van Leeuwaarden. 
(Preprint)
\item Limits of sparse configuration models and beyond: graphexes and multi-graphexes (2018+); Christian Borgs, Jennifer T. Chayes, Souvik Dhara, Subhabrata Sen. (Preprint)
\item Critical percolation on scale-free random graphs: Effect of the single-edge constraint (2018+); Souvik Dhara, Remco van der Hofstad, Johan S.H. van Leeuwaarden. (Preprint)
\item Global lower mass-bound for critical configuration models in the heavy-tailed regime (2018+); Shankar Bhamidi, Souvik Dhara, Remco van der Hofstad, Sanchayan Sen. (Preprint)
\item Universality for critical heavy-tailed network models: Metric structure of maximal components (2017); 
Shankar Bhamidi, Souvik Dhara, Remco van der Hofstad, Sanchayan Sen. arXiv:1703.07145
\item Heavy-tailed configuration models at criticality (2016); Souvik Dhara, Remco van der Hofstad, Johan S.H. van Leeuwaarden, Sanchayan Sen; arXiv:1612.00650 
\item Corrected mean-field model for random sequential adsorption on random geometric graphs (2016); Souvik Dhara, Johan S.H. van Leeuwaarden, Debankur Mukherjee; To appear with Journal of Statistical Physics.
\item Phase transitions of extremal cuts for the configuration model (2017);  Souvik Dhara, Debankur Mukherjee, Subhabrata Sen; Electronic Journal of Probability 22, no. 86, 1–29.
\item Optimal Service Elasticity in Large-Scale Distributed Systems (2017); Debankur Mukherjee, Souvik Dhara, Sem Borst, Johan S.H. van Leeuwaarden;  
SIGMETRICS'17, Urbana-Champaign, Illinois, USA. 
Proceedings of the ACM on Measurement and Analysis of computing systems. 
\item Critical window for the configuration model: finite third moment degrees (2016); Souvik Dhara, Remco van der Hofstad, Johan S.H. van Leeuwaarden, Sanchayan Sen; Electronic Journal of Probability 22, no. 16, 1–33.
\item Generalized random sequential adsorption on Erdos-Renyi random graphs (2016); Souvik Dhara, Johan S.H. van Leeuwaarden, Debankur Mukherjee; Journal of Statistical Physics 164, 1217-1232.
\end{enumerate}

\end{document}